\documentclass[11pt]{article}
\usepackage[french, english]{babel}
\usepackage[utf8]{inputenc}

\usepackage{amsmath,amsfonts,amssymb,amsthm,mathrsfs}
\usepackage{minitoc}
\usepackage{bbm}

\renewcommand{\leq}{\leqslant}
\renewcommand{\geq}{\geqslant}
\usepackage[mono=false]{libertine}
\useosf 
\usepackage{mathpazo}

\usepackage{wasysym}

\renewcommand{\preceq}{\preccurlyeq}
\usepackage[dvipsnames]{xcolor}
\usepackage[a4paper,vmargin={3.5cm,3.5cm},hmargin={2cm,2cm}]{geometry}
\usepackage[font=sf, labelfont={sf,bf}, margin=1cm]{caption}
\usepackage{graphicx,graphics}
\usepackage{epsfig}
\usepackage{latexsym}
\usepackage{stmaryrd}
\usepackage{ae,aecompl}

 \usepackage[colorlinks=true]{hyperref}
\usepackage{pstricks}
\usepackage{enumerate}
\usepackage{tikz}
\usepackage{fontawesome}
\usetikzlibrary{arrows,automata}
\numberwithin{equation}{section}

\newcommand{\bq}{{\bf q}}

\newcommand{\bm}{{\bf m}}
\newcommand{\bu}{{\bf u}}

\newcommand{\bN}{{\bf N}}

\def\mm{{\bf m}}

\newcommand{\R}{\mathbb{R}}

\newcommand{\N}{\mathbb{N}}

\newcommand{\Z}{\mathbb{Z}}
\renewcommand{\P}{\mathbb{P}} 
 
\newcommand{\E}{\mathbb{E}}

\def\build#1_#2^#3{\mathrel{ \mathop{\kern 0pt#1}\limits_{#2}^{#3}}}

\def\d{{\rm d}}

\def\eps{\varepsilon}

\def\wh{\widehat}

\def\la{\longrightarrow}

\newcommand{\ind}{\mathbbm{1}}

\newtheorem{theo}{Theorem}[section]
\newtheorem{lem}[theo]{Lemma}
\newtheorem{rek}[theo]{Remark}
\newtheorem{prop}[theo]{Proposition}

\newtheorem{cor}[theo]{Corollary}

\newtheorem*{conjecture}{Conjecture}

\title{\textsc{The scaling limit of planar maps with large faces}}
\renewcommand{\leq}{\leqslant}
\renewcommand{\geq}{\geqslant}
\begin{document}
\author{
Nicolas \textsc{Curien}\thanks{Universit\'e Paris-Saclay\hfill  \href{mailto:nicolas.curien@gmail.com}{\texttt{nicolas.curien@gmail.com}}} \ and
Gr\'egory  \textsc{Miermont}\thanks{ENS de Lyon and Institut
  Universitaire de France \hfill \href{mailto:gregory.miermont@ens-lyon.fr}{\texttt{gregory.miermont@ens-lyon.fr}}} \ and 
Armand  \textsc{Riera}\thanks{LPSM, Sorbonne Universit\'e \hfill \href{mailto:riera@lpsm.paris}{\texttt{riera@lpsm.paris}}}}
   \date{}         
             \maketitle

\begin{abstract} We prove  that large Boltzmann stable planar maps of index $\alpha \in (1;2)$ converge  in the scaling limit towards a 
random compact metric space $ \mathcal{S}_{\alpha}$ that we construct
explicitly. They form a one-parameter family of random continuous
spaces ``with holes'' or ``faces''  different from the Brownian
sphere.  In the so-called dilute phase $ \alpha \in [3/2;2)$, the
topology of $ \mathcal{S}_{\alpha}$ is that of the Sierpinski carpet,
while in the dense phase $ \alpha \in (1;3/2)$ the ``faces'' of $
\mathcal{S}_{\alpha}$ may touch each other. En route, we prove various geometric properties of these objects concerning their faces or the behavior of geodesics.
\end{abstract}

\begin{figure}[!h]
 \begin{center}
 \includegraphics[height=4cm]{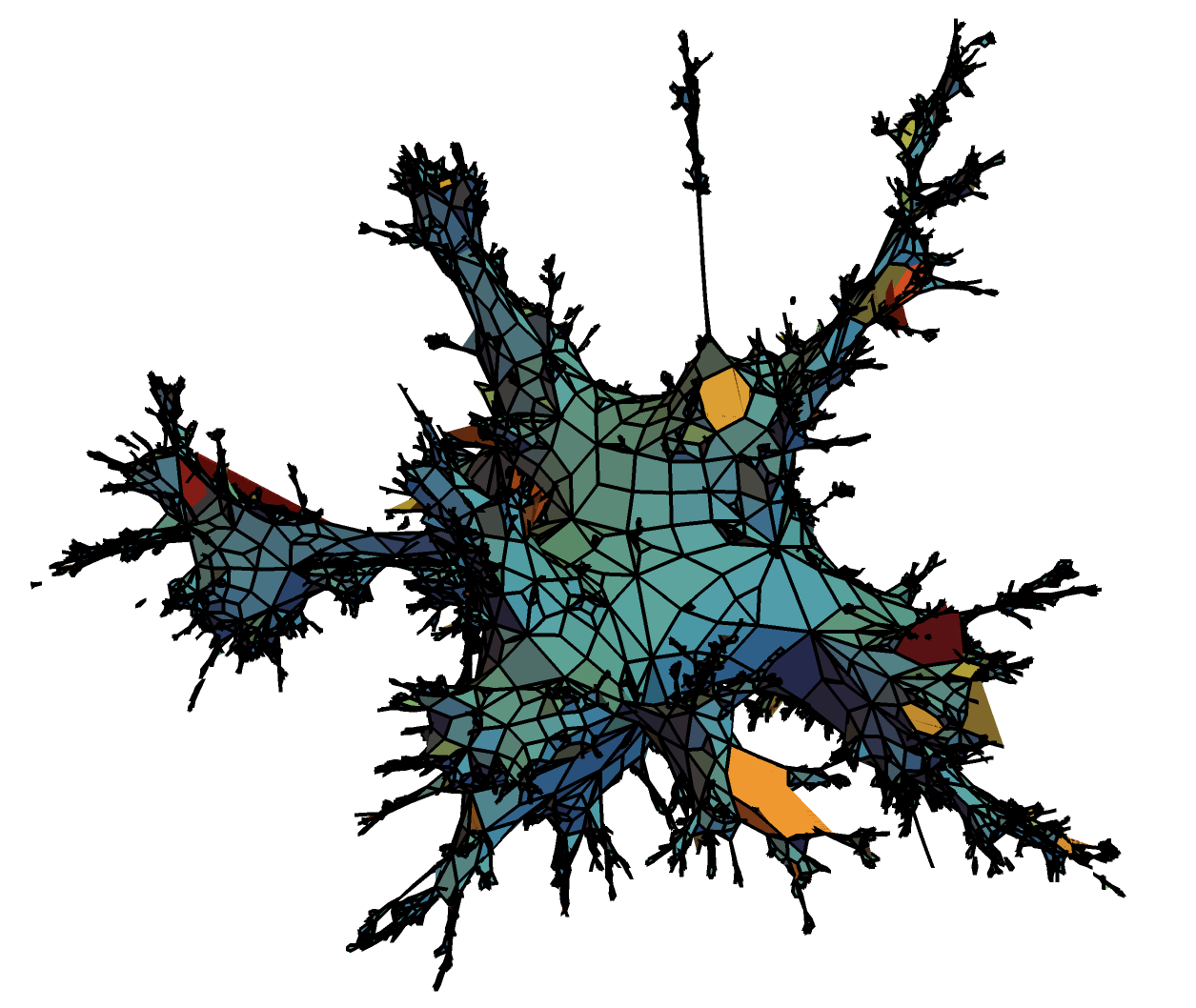}
  \includegraphics[height=4cm]{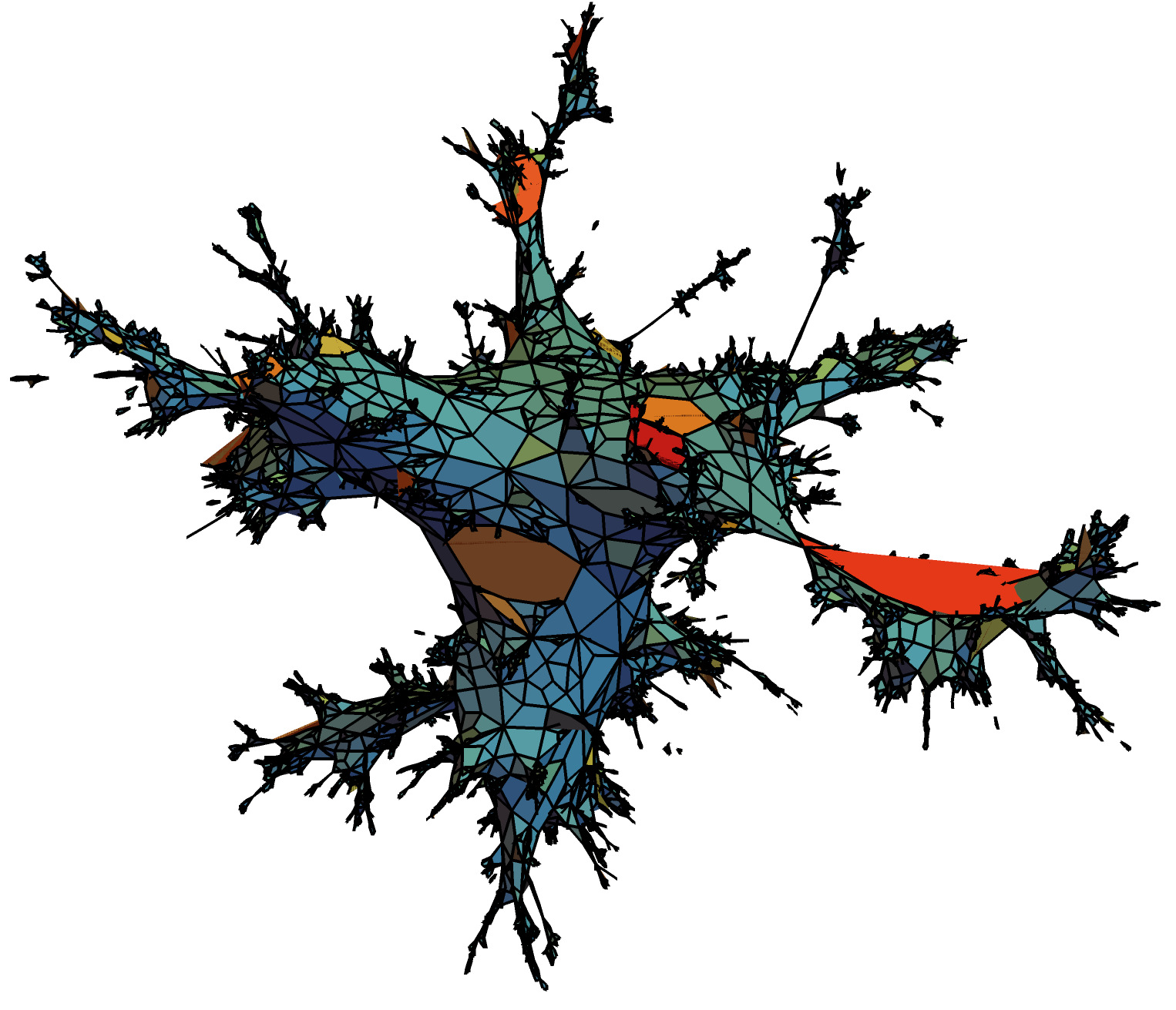}
   \includegraphics[height=4cm]{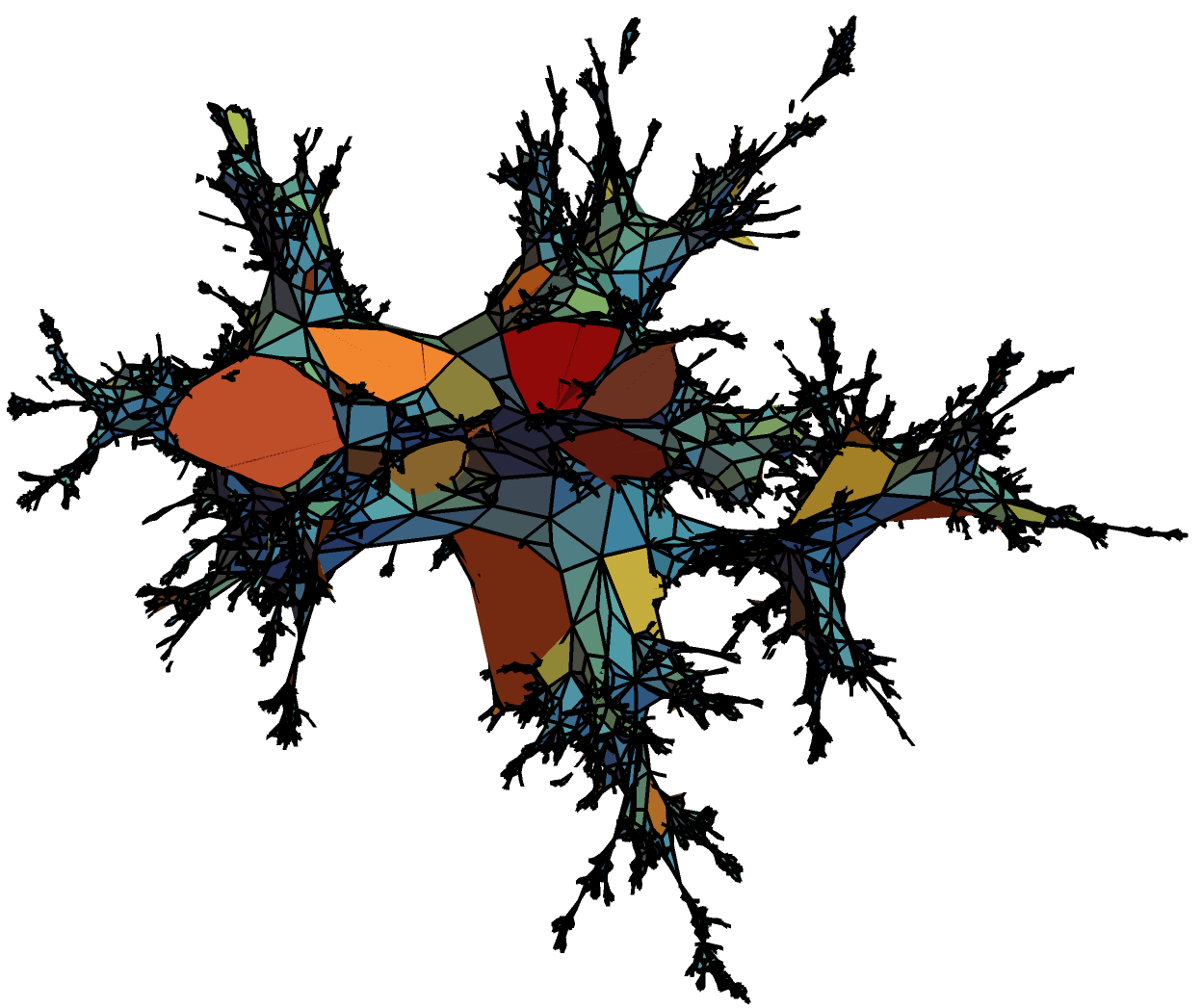}
      \includegraphics[height=4cm]{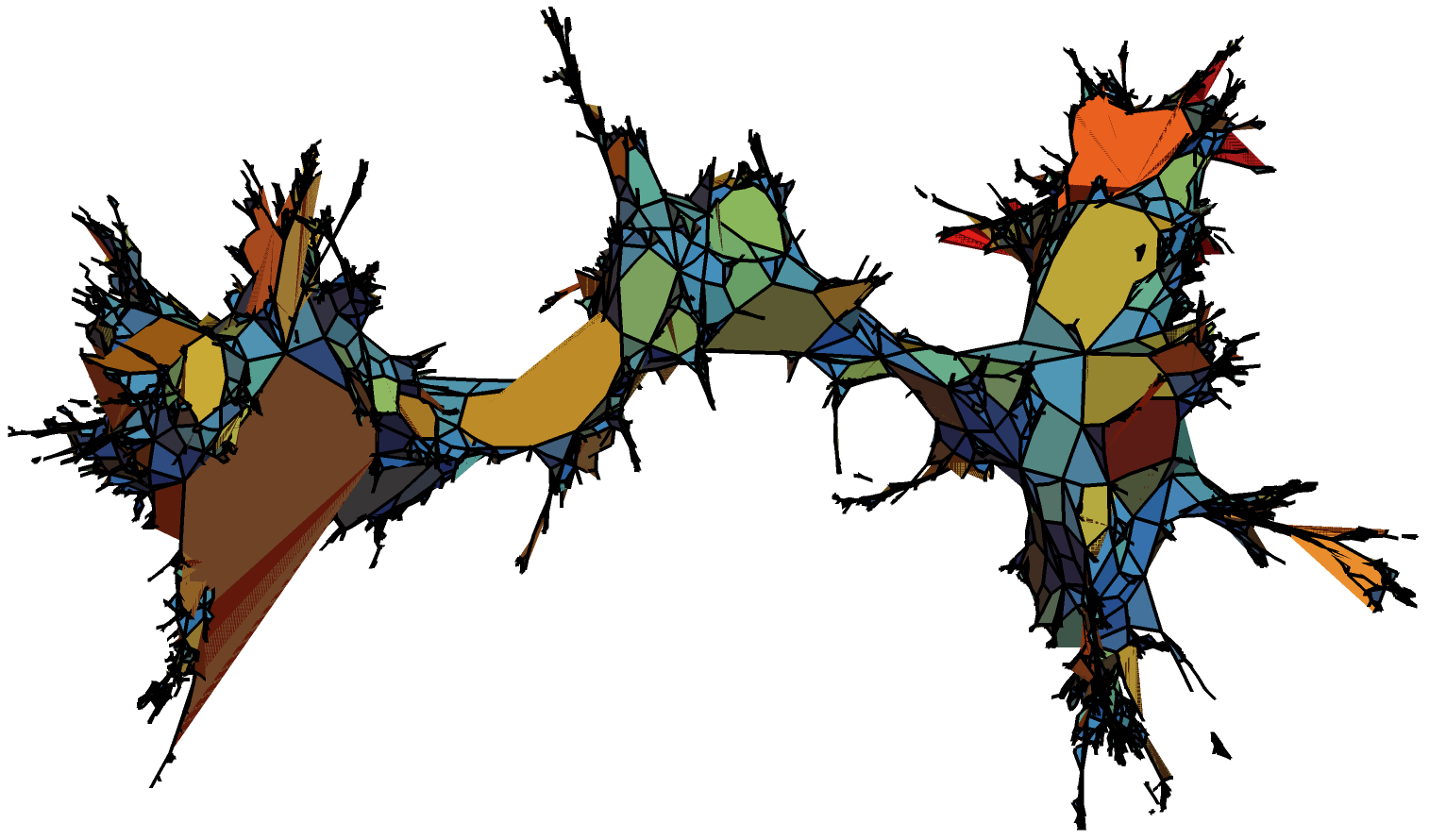}
         \includegraphics[height=4cm]{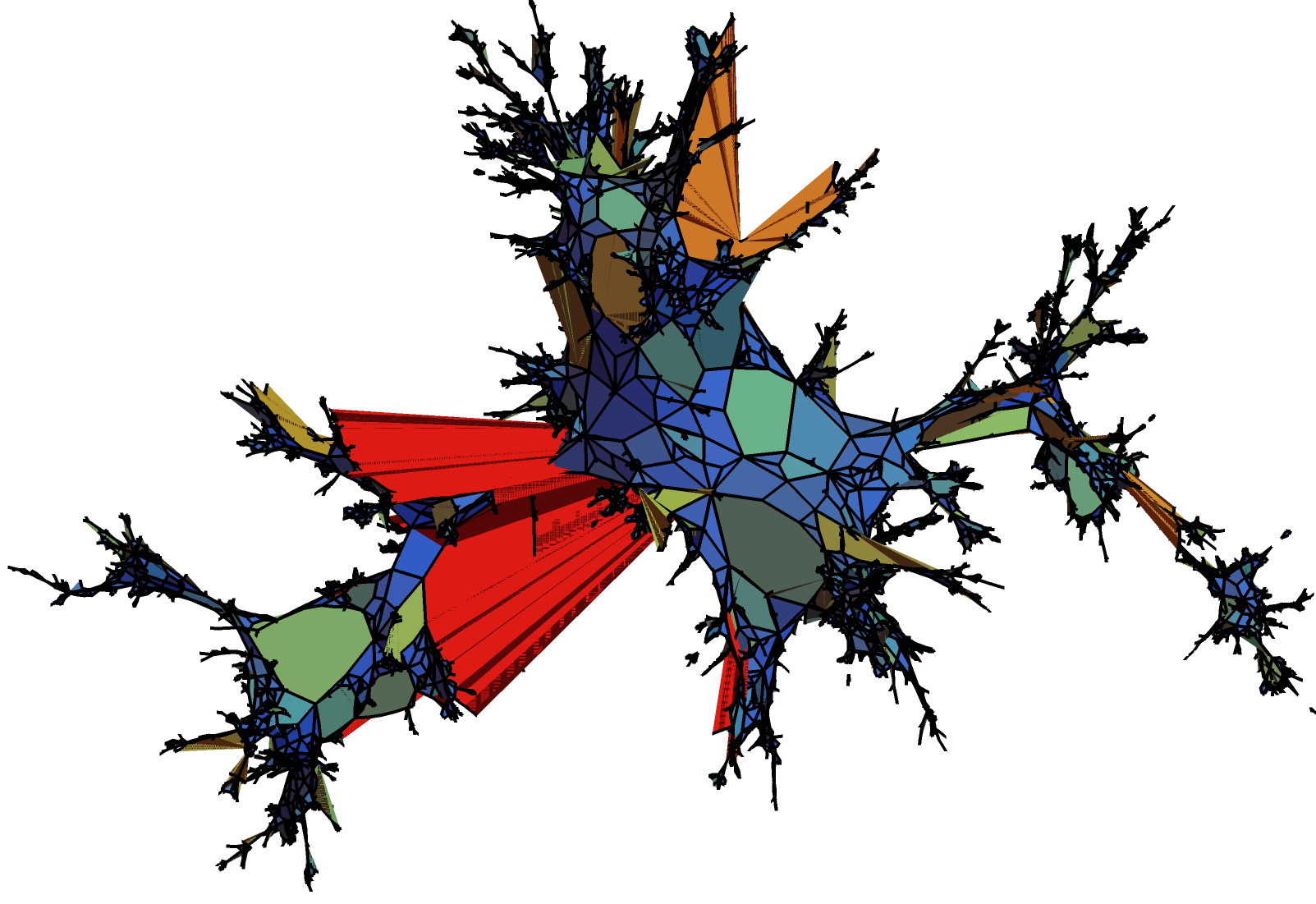}
            \includegraphics[height=4cm]{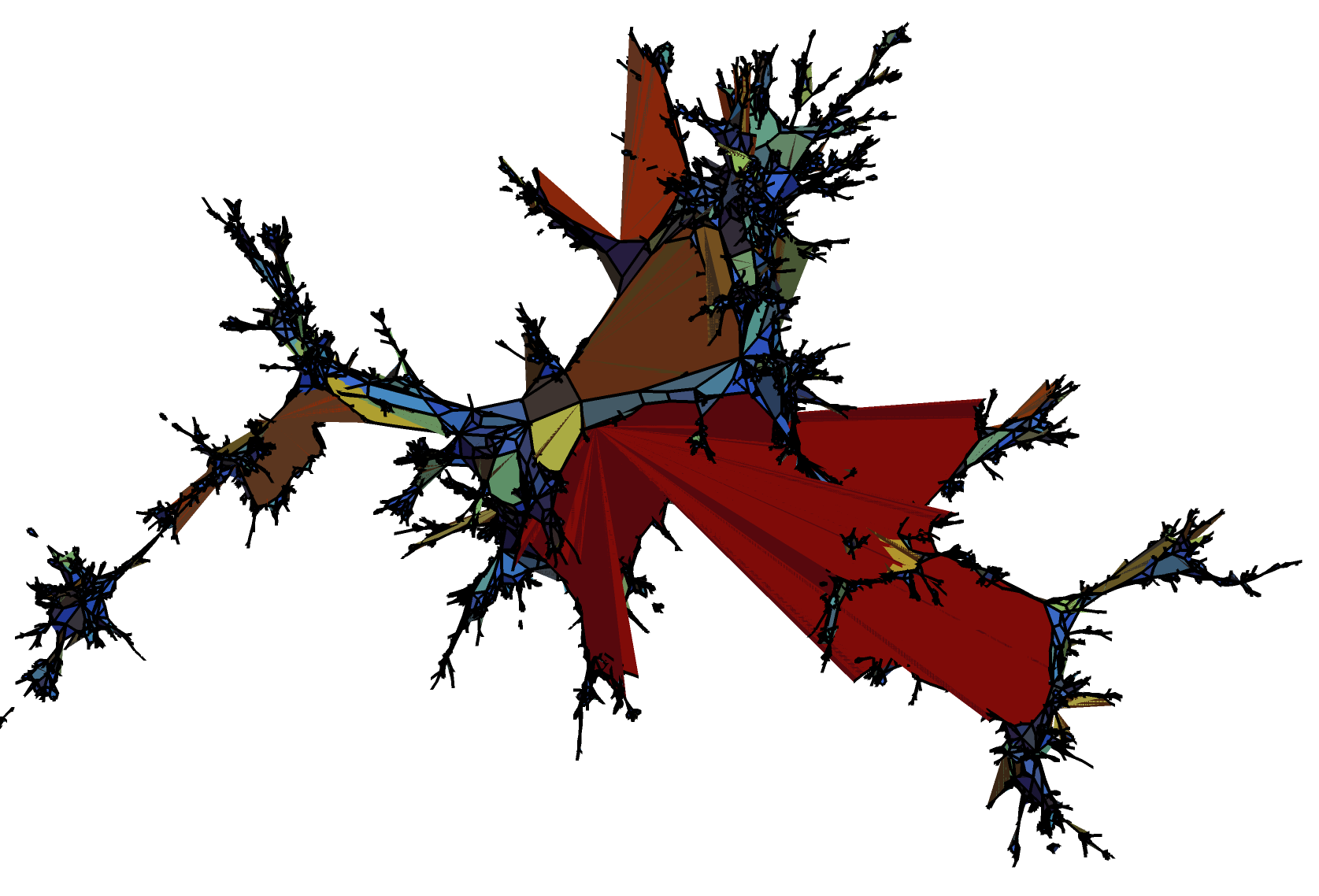}
               \includegraphics[height=4cm]{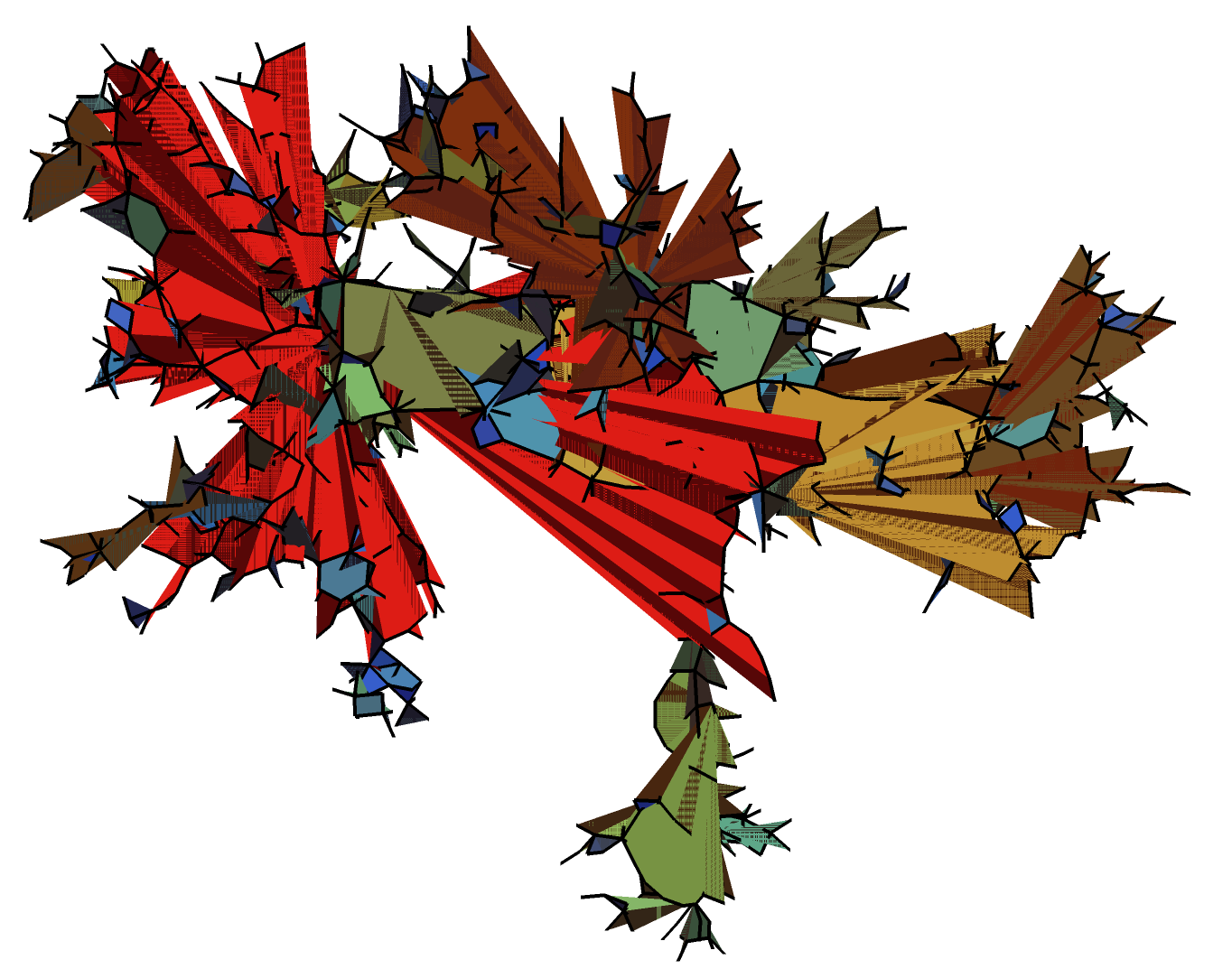}
 \caption{Simulations of large non-generic critical random Boltzmann planar maps of index $ \alpha \in \{1.9, 1.8, 1.7, 1.6, 1.5, 1.4, 1.3\}$ from top left to bottom right. }
 \end{center}
 \end{figure}

\newcommand{\scal}{\mathfrak{s}_{\mathbf{q}}}
\newcommand{\zbq}{z_{\mathbf{q}}}
\newcommand{\bT}{\boldsymbol{ \mathcal{T}}}
\newcommand{\bTn}{\boldsymbol{ \mathcal{T}}_{\!\!n}}
\newcommand{\delay}{\triangle}
\newcommand{\hdelay}{\widehat{\triangle}} 
\newcommand{\Edelay}{\triangle\!\!\!\!\triangle} 
\newcommand{\GG}{\widetilde{\mathcal{G}}}

\section{Introduction}
In recent years, the theory of random planar maps has seen many spectacular developments. The central object of this theory, the \textit{Brownian sphere}  \cite{LG11,Mie11}, is in a sense a universal model of $2$-dimensional random geometry. It is now proven to be the limit of a long list of combinatorial models of random planar maps, see e.g.~\cite{AA13,BJM13,CLGmodif,marzouk2022scaling}, but also of other $2$-dimensional geometric structures such as random hyperbolic surfaces with many cusps \cite{BC22}. The Brownian sphere and other related models of random surfaces \cite{BeMi22,CLGplane} can be constructed from canonical objects in probability theory, such as Aldous' Continuum random tree and Le Gall's Brownian snake. Other constructions can be given in terms of conformal random objects like the Gaussian Free Field and Schramm's SLE$_{6}$, which make the Brownian sphere a distinguished member of the family of Liouville quantum gravity metrics \cite{gwynne2020existence,MS15,MiSh21geod,MiSh21conf,miller2015liouville}. All these constructions allow for the study of this object from many angles and, in particular, enable a wealth of exact calculations,  see e.g.~\cite{gall2018some,Spine}.
Nevertheless, many of its properties remain unknown and are the subject of intensive ongoing research.

However, it is also known that the Brownian sphere is not the only possible limiting model for natural random maps models. In a sense,  
the paradigm of random maps converging to the Brownian sphere can be considered as the ``Gaussian'' case of the theory. In this work, we rather focus on the ``stable'' case, which despite some efforts \cite{BBCK18,BC16,LGM09,marzouk2018scalingstable} remains much less understood.

\paragraph{Non-generic Boltzmann maps.}

Let us first introduce the model we will study in this paper, starting with some basic definitions. A \textbf{planar map} is a proper embedding of a finite multigraph in the two-dimensional sphere, such that the connected components of the complement of the embedding are simply connected. These connected components are called the \textbf{faces} of the map. 
Two planar maps that can be transformed into one another by a homeomorphism of the sphere preserving the orientation are systematically identified, so that the set of planar maps (up to these identifications) is in fact countable. Alternatively, one can view a planar map as a gluing of a finite collection of polygons by identifying their edges in pairs, in such a way that the resulting space is homeomorphic to the 2-dimensional sphere. With this point of view, the polygons correspond to the faces of the map, and the number of corners of a polygon is called the degree of the corresponding face, while their edges and vertices correspond to the edges and vertices of the embedded graph. 
See \cite{LZ04} for a discussion of these and many other possible representations of maps. 

As usual, all planar maps we consider are \textbf{rooted}, i.e.~one of the corners 
incident to one of the faces of the map is distinguished and called the root corner, the incident face  is called the root face. This is equivalent to distinguishing an oriented edge, considered in clockwise order around the root face, and both notions will be used.
For technical simplicity, we will  only consider \textbf{bipartite} planar maps, meaning that all faces have even degree. 

Given a non-zero sequence $ \mathbf{q}= (q_{k})_{ k \geq 1}$ of non-negative numbers, we define the \textbf{$ \mathbf{q}$-Boltzmann measure}  $w_{\mathbf{q}}$  on the set $\mathcal{M}$ of all rooted bipartite planar maps by the formula
\begin{equation} \label{eq:defface}
w_{\mathbf{q}}(\bm) := \prod_{f\in \mathrm{Faces}(\bm)} q_{ \mathrm{deg}(f)/2}~, \qquad \bm\in \mathcal{M}.
\end{equation}
To simplify many of our formulas, we shall assume that there is a unique map in $ \mathcal{M}$ which has no face: this map, which is denoted by $\to$, is the ``edge map" made of a single oriented edge, two vertices and no face (not to be confused with the map having a single edge and a face of degree $2$), and its weight  is $w_{ \mathbf{q}}(\to)=1$. We shall also denote the set of all pointed maps by $\mathcal{M}^\bullet=\{(\bm,v):\bm\in \mathcal{M},v\in V(\bm)\}$,
where $V( \bm)$ is the vertex set of $ \bm$, and define the pointed $\bq$-Boltzmann measure $w_\bq^\bullet$ on $\mathcal{M}^\bullet$ by $w_{\bq}^\bullet((\bm,v)) := w_\bq(\bm)$.

Given a number $\alpha\in (1,2)$, we say that the sequence $\bq$ is \textbf{non-generic with exponent $\alpha$} if there exists a constant $\scal \in (0,\infty)$ such that:
\begin{equation}
  \label{eq:nongenerictail}
  w_\bq^\bullet\Big(\mathrm{deg}(\mathrm{root\ face})>k\Big)\sim
  \frac{2\scal}{
    |\Gamma(1-\alpha)|\,  k^\alpha}\, ,\qquad  \mbox{ as } k\to\infty\, .
\end{equation}
This particular expression of the prefactor will simplify the form of the scaling constants which will arise later. 
In this paper, \textbf{we will always assume that the weight sequence $\bq$ is  non-generic with some exponent $\alpha\in (1,2)$.}
This assumption implies in particular that the measure $w_\bq^\bullet$, and a fortiori $w_\bq$, is a finite measure, see \cite[Exercise 3.10]{CurStFlour}.  The normalized distributions $ w_\bq/w_\bq(\mathcal{M})$ and $w_\bq^\bullet/w_\bq^\bullet(\mathcal{M}^\bullet)$ are then called the \textbf{$\bq$-Boltzmann distributions} on $\mathcal{M}$ and $\mathcal{M}^\bullet$, respectively. A random variable with law $w_\bq/w_\bq(\mathcal{M})$ will be more colloquially called a \textbf{$\bq$-Boltzmann map}.

The non-genericity property (\ref{eq:nongenerictail}) has
many other characterizations, see \cite[Chapter 5]{CurStFlour}. In particular, it has a simple rephrasing in terms of the analytic function
  \begin{equation}
    \label{eq:genq}
    g_\bq(x):=1+\sum_{k\geq 1}\binom{2k-1}{k-1}\, q_k\, x^k\, ,\qquad x\geq 0.  
  \end{equation}
  Namely, if we define the quantity $\zbq$ by 
  \begin{equation}
    \label{eq:zbq}
    \zbq:=\inf\big\{x>0:g_\bq(x)=x\big\}\in (0,\infty]\, ,
  \end{equation}
  then it holds that $w_\bq^\bullet(\mathcal{M}^\bullet)=2\zbq$, and $\bq$ is non-generic with exponent $\alpha$ if and only if $z_\bq<\infty$ and 

  \begin{equation}
    \label{eq:nongenericgenf} 
    g_\bq(\zbq\, x)=\zbq\, x + \scal \left(\frac{1-x}{2}\right)^{\alpha}(1+o(1))\, ,\qquad  x\uparrow 1\, ,
  \end{equation}
where $\scal$ is the same constant as in (\ref{eq:nongenerictail}).

The reader should keep in mind that, as evidenced by the above discussion, the non-genericity condition (\ref{eq:nongenerictail})
is not only an asymptotic property of the weight sequence $\bq$, but depends in a fine-tuned way on all its values. However, such weight sequences are not  ``ad-hoc" since they naturally appear as the law of the gasket of critical $O(N)$-loop model -- including critical Bernoulli percolation -- on ``generic'' random planar map models\footnote{The term ``generic'' should be understood here in the loose sense that the map model converge to the Brownian sphere in the scaling limit. For a more precise discussion on this terminology, see for instance \cite{BBG11}.}, see Section \ref{sec:BDG} below or \cite{bernardi2019boltzmann,BBG11,BudOn,curien2018duality,marzouk2018scalingstable} and \cite[Section 5.3.3 and 5.3.4]{CurStFlour} for details. For convenience, in all this work, the parameter $\alpha$ will usually be implicit and often removed from the notation. 
 
For a given non-generic sequence $\mathbf{q}$, large $ \mathbf{q}$-Boltzmann maps possess ``large\footnote{This can be a bit misleading: many models of maps with large faces are known to rescale to the Brownian sphere \cite{marzouk2022scaling}, as long as, heuristically speaking, most of the faces have comparable degrees. The important fact here is that, as a consequence of (\ref{eq:nongenerictail}), the variance of the typical face degree is infinite.} faces'', and Le Gall \& Miermont \cite{LGM09} proved that after normalization of the graph distance by $ n^{-\frac{1}{2\alpha}}$, where $n$ is the number of vertices of the map, the sequence of laws of the random metric spaces that they induce is tight in the Gromov--Hausdorff topology. The main goal of this work is to prove the uniqueness of the limit, which together with \cite{LGM09} will  ensure the convergence in distribution. More precisely, for $n \geq 1$, let $ \mathfrak{M}_{n}$ be a $ \mathbf{q}$-Boltzmann map conditioned to have $n$ vertices -- which is always  possible for $n$ large enough, see \cite[Exercise 3.8]{CurStFlour}. We endow its vertex set $ V(\mathfrak{M}_{n})$ with the graph distance $\mathrm{d}^{\mathrm{gr}}_{\mathfrak{M}_{n}}$ and the uniform measure $\mathrm{vol}_{\mathfrak{M}_{n}}$.
 
 \begin{theo}[Scaling limit for non-generic Boltzmann maps]  \label{thm:main} 
Fix $\alpha\in (1,2)$. There exists a random compact metric measure space $\left( \mathcal{S}_\alpha, D^{*}_\alpha, \mathrm{Vol}_\alpha\right)$ such that, for every 
admissible, critical and non-generic weight sequence $\bq$ of exponent $\alpha$, 
we have the following convergence in distribution for the Gromov--Hausdorff--Prokhorov topology:
$$ \left( V(\mathfrak{M}_{n}) ,~  (\scal n)^{-\tfrac{1}{2 \alpha}}\cdot \mathrm{d}^{\mathrm{gr}}_{\mathfrak{M}_{n} }, ~\mathrm{vol}_{\mathfrak{M}_{n}} \right) \xrightarrow[n\to\infty]{(d)} \left( \mathcal{S}_\alpha, D^{*}_\alpha,~ \mathrm{Vol}_\alpha\right).$$
The space  $\left( \mathcal{S}_\alpha, D^{*}_\alpha, \mathrm{Vol}_\alpha\right)$ is of Hausdorff dimension $2\alpha$ almost surely. It is called the $\alpha$-stable carpet if $\alpha \in [3/2,2)$, and the $\alpha$-stable gasket if $ \alpha \in( 1,3/2)$. 
 \end{theo}

 Let us stress that the law of the limit does not depend on the choice of $\bq$, except through the value of $\alpha$, so that this result may be seen as an invariance principle, or, in more physical terms, as the identification of a ``universality class'' for random maps. One can wonder whether a similar result holds when the hypothesis (\ref{eq:nongenerictail}) is relaxed a little bit, for instance, if one assumes only that $w_\bq^\bullet(\mathrm{deg}(\mathrm{root\ face})>k)$ is regularly varying with exponent $-\alpha$, or when one drops the assumption that the maps are bipartite. We believe that extensions indeed hold, and in Section \ref{sub:reroot} below, we give a method based on a now classical re-rooting trick by Le Gall \cite{LG11} which, in principle, makes the proof of such extensions relatively easy. In fact, in this paper, we will first prove Theorem \ref{thm:main} under the more stringent assumption that $w_\bq^\bullet(\mathrm{deg}(\mathrm{root\ face})=k)$ is equivalent to a constant times $k^{-\alpha-1}$. This will then allow us to apply Le Gall's re-rooting trick to obtain the full statement with minimal effort. 

 The limiting spaces appearing in the statement of Theorem \ref{thm:main} have an explicit description in terms of certain random processes, which will be given in the next few paragraphs. As mentioned before, we consider the value of $\alpha$ as fixed, and therefore, we will usually drop it from the notation and simply denote the limit by $(\mathcal{S},D^*,\mathrm{Vol})$. 
 
  One of the main ingredient in \cite{LGM09}, as well as in our work, is the Bouttier--Di Francesco--Guitter  bijection  \cite{BDFG04} which  allows to  code the random maps  $\mathfrak{M}_{n}$ using labeled trees, see Section \ref{sec:BDG}.  This construction encapsulates (certain) graph distances on $ \mathfrak{M}_{n}$, and scaling limits for the labeled trees are known -- in fact, obtaining scaling limit for these labeled trees was the main occupation of \cite{LGM09}. Hence at a very high level, the above theorem is a kind of ``typical continuity'' of this encoding under scaling limits. 

 \begin{rek}[Dual maps] We stress that our results are only valid for the maps $ \mathfrak{M}_{n}$ and not for their duals $  \mathfrak{M}_{n}^{\dagger}$ which have vertices of large degrees. Indeed, there is no known Schaeffer-type construction of maps with large degree vertices which efficiently encodes the distances in the map. However, completely different techniques (based on Markovian explorations -- the peeling process and its link with discrete Markov branching trees \cite{bertoin2024self}) show that the diameter of $ \mathfrak{M}_{n}^{\dagger}$ scales as $n^{ 1-\frac{3}{2 \alpha}}$ for $\alpha \in (3/2, 2)$, see \cite{BBCK18,bertoin2024self,BC16}, and we expect scaling limits homeomorphic to the $2$-dimensional sphere, but of fractal dimension $ \frac{2 \alpha}{2\alpha-3}$. The diameter of those dual maps  is logarithmic in $n$ in the dense phase $\alpha \in (1,3/2)$ so that no scaling limit is expected in this case, see \cite{BBCK18,kammerer2025fpp} and \cite{BCMcauchy,kammerer2023distances,kammerer2023large} for the critical case $ \alpha = \frac{3}{2}$ which should be connected to the LQG metric on the CLE$_{4}$.
 \end{rek}

 \paragraph{Definition of the $\alpha$-stable carpet/gasket.}  Let us detail the definition of the random compact metric space $ (\mathcal{S},  D^{*},\mathrm{Vol})$. This construction  may seem ad-hoc at first glance,  but it comes from passing to the scaling limit the discrete encoding of $ \mathfrak{M}_{n}$ using labeled trees. This is similar to the definition of the Brownian sphere \cite{LG11,Mie11}, and on a high-level and for the connoisseurs, the role of Brownian tree is replaced in our context by the $\alpha$-stable looptree and Le Gall's Brownian snake becomes the Gaussian Free Field on the looptree --  or equivalently Brownian motion indexed by the stable looptree.  \medskip

 Fix $\alpha \in (1,2)$ and let  $(X_{t})_{t\geq 0}$ be  the excursion with lifetime equal to $1$ of an $\alpha$-stable L\'evy process with no negative jumps reflected above its infimum (normalized so that its Laplace exponent is $\lambda^{\alpha}$). We will use the standard notation $\Delta_{t}:=X_{t}-X_{t-}$, for every $t\in [0,1]$, and we consider $(\mathrm{t}_{i})_{i\in \mathbb{N}}$  a measurable enumeration of the set $\{t\in [0,1]:\:\:\Delta_{t}>0\}$. The construction of the $\alpha$-stable carpet/gasket relies on:
 \begin{itemize}
\item the L\'evy excursion $X$;
\item  a sequence of  independent  Brownian bridges  $(\mathrm{b}_i)_{i\in \mathbb{N}}$ (all starting and ending at $0$ with lifetime~$1$)  also independent of $X$.
\end{itemize} 
Before  giving the formal definition of $(\mathcal{S},D^{*})$ let us introduce some useful notation. First, we set
 $$I_{s,t}:=\inf\limits_{[s,t]} X,$$
 for every  $s,t\in [0,1]$ with $s\leq t$, and for convenience we let  $I_{s,t}=-\infty$ if $t<s$. We also write $s\preceq t$ and say that $s$ is an \textbf{ancestor} of $t$ 
 if $s\leq t$  and $I_{s,t}\geq X_{s-}$ 
 and  we set $x_{s,t}:=I_{s,t}-X_{s-},$ for every $s\preceq t$. We shall also \textbf{identify} the two times $s$ and $t$, and write $s \sim_{d} t$, if $ X_{t} = X_{s-}$ and $I_{s,t} = X_{s-}=X_{t}$ in the case $s \leq t$ and vice-versa if $s \geq t$, see Figure \ref{fig:identificationlooptree}. We will see in Section \ref{sec:looptree} that this equivalence relation is associated to a pseudo-metric $d$ on $[0,1]$, such that $ \mathcal{L}= ([0,1] /\sim_{d}, d)  $ is the \textbf{looptree} coded by the function $X$, as  introduced in \cite{CKlooptrees}. With this notation at hand, we can define the \textbf{continuous label function} 
 by
 \begin{equation}\label{Z_represent_Mir_intro}
Z_{t} :=    \mathop{\sum \limits_{i\in \mathbb{N},~\mathrm{t}_i\preceq t}} \Delta_{\mathrm{t}_{i}}^{\frac{1}{2}}\cdot \mathrm{b}_{i}\big(\frac{x_{\mathrm{t}_{i},t}}{\Delta_{\mathrm{t}_{i}}}\big), \quad \mbox{for } t \in [0,1].
\end{equation}
The proper definition of the  process $Z$ is given in Section \ref{sec:constr_Z}, 	where we recall the original definition of \cite{LGM09} and give an alternative point of view in terms of Gaussian processes using some of the results of Archer \cite{archer2019brownian}, concerning stable looptrees. 
This function passes to the quotient by $ \sim_{d}$ and can  be viewed as the ``Brownian motion indexed by the looptree $ \mathcal{L}$''. Section \ref{sec:constr_Z} is devoted to making this interpretation precise. See \cite{blancrenaudie2022looptree} for recent investigations of those constructions in more general contexts. \begin{figure}[!h]
 \begin{center}
 \begin{tabular}{ccc}
 & \includegraphics[width=4.4cm]{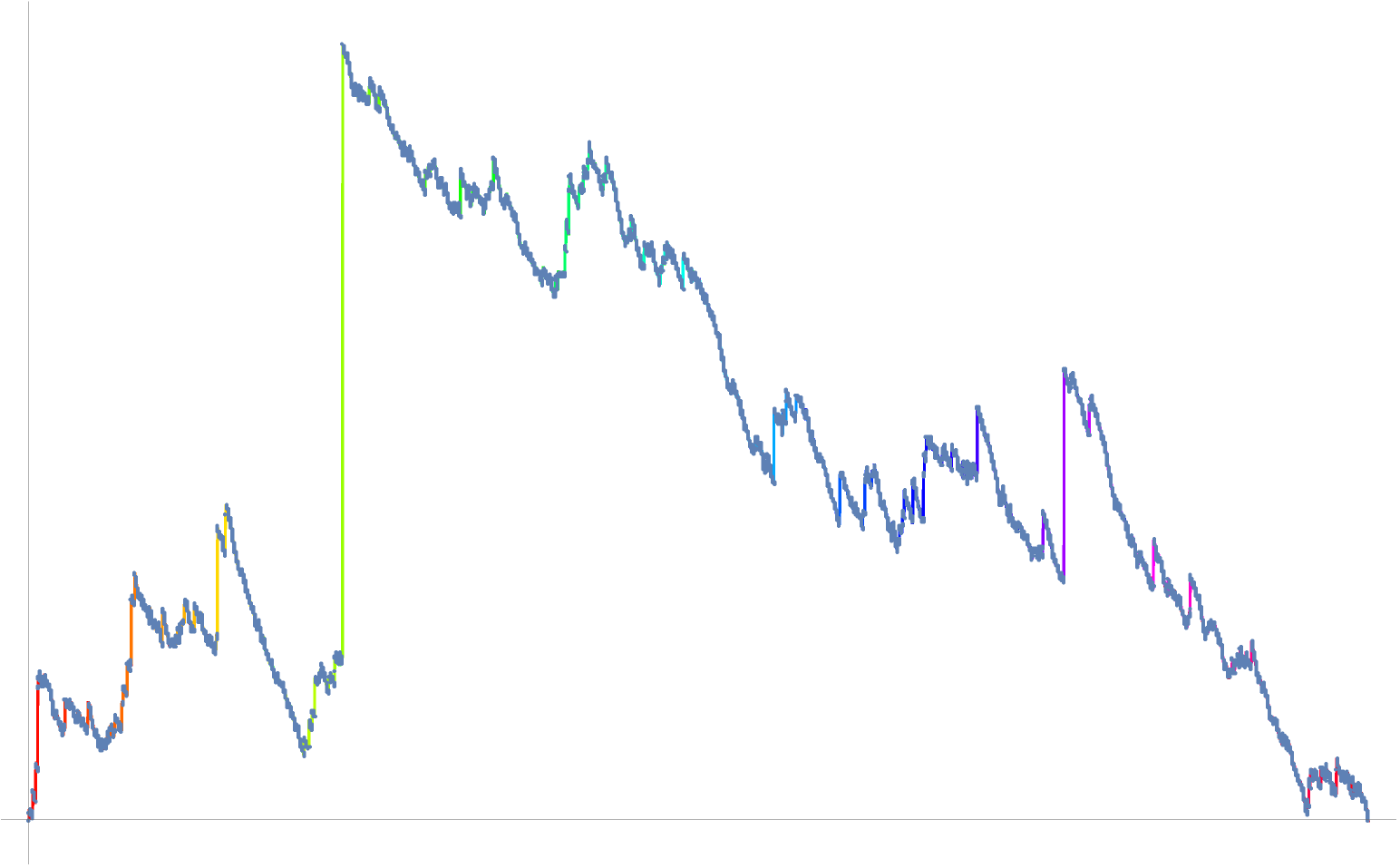} \quad &\quad \includegraphics[width=4.4cm]{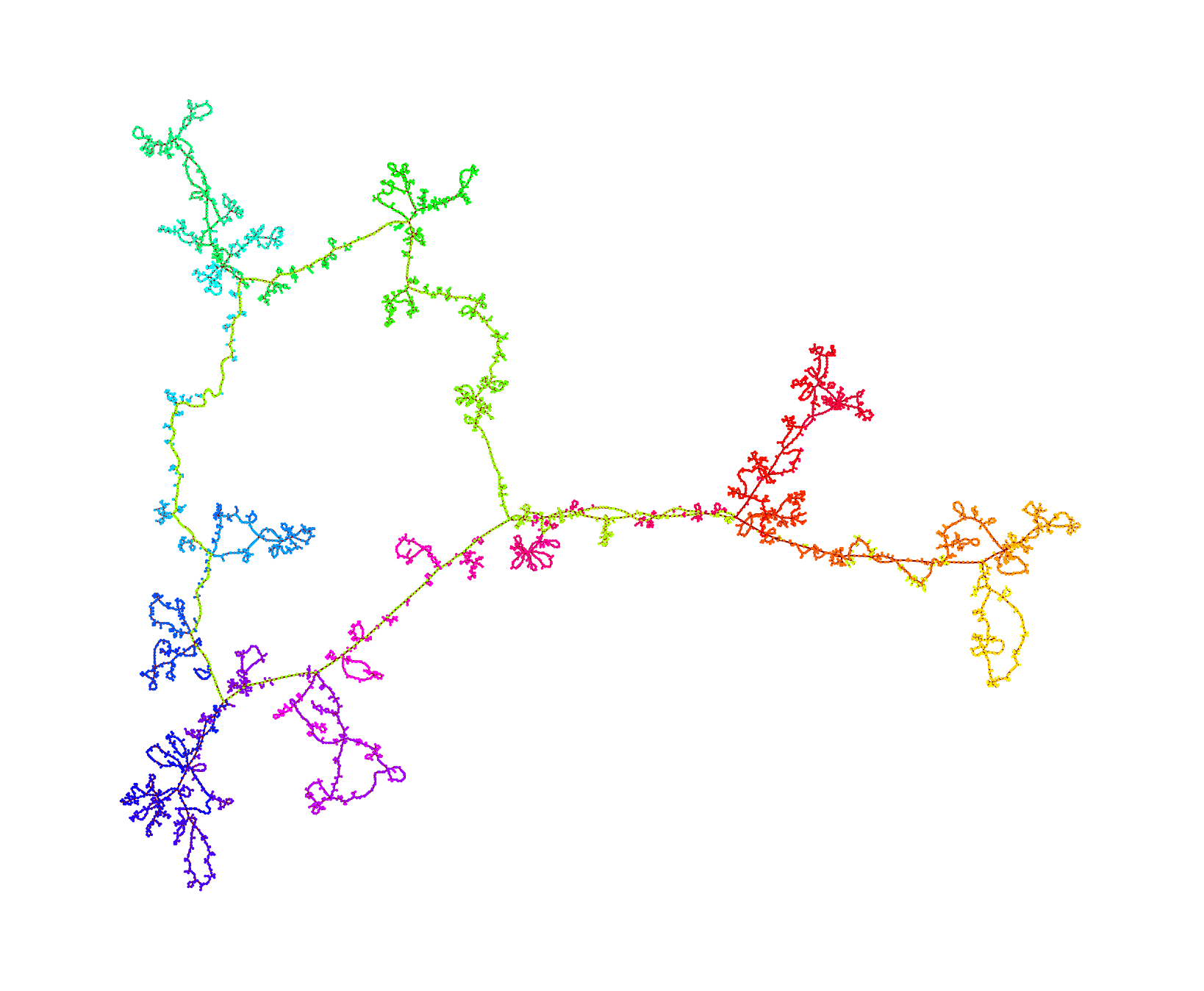}\\
&\includegraphics[width=4.4cm]{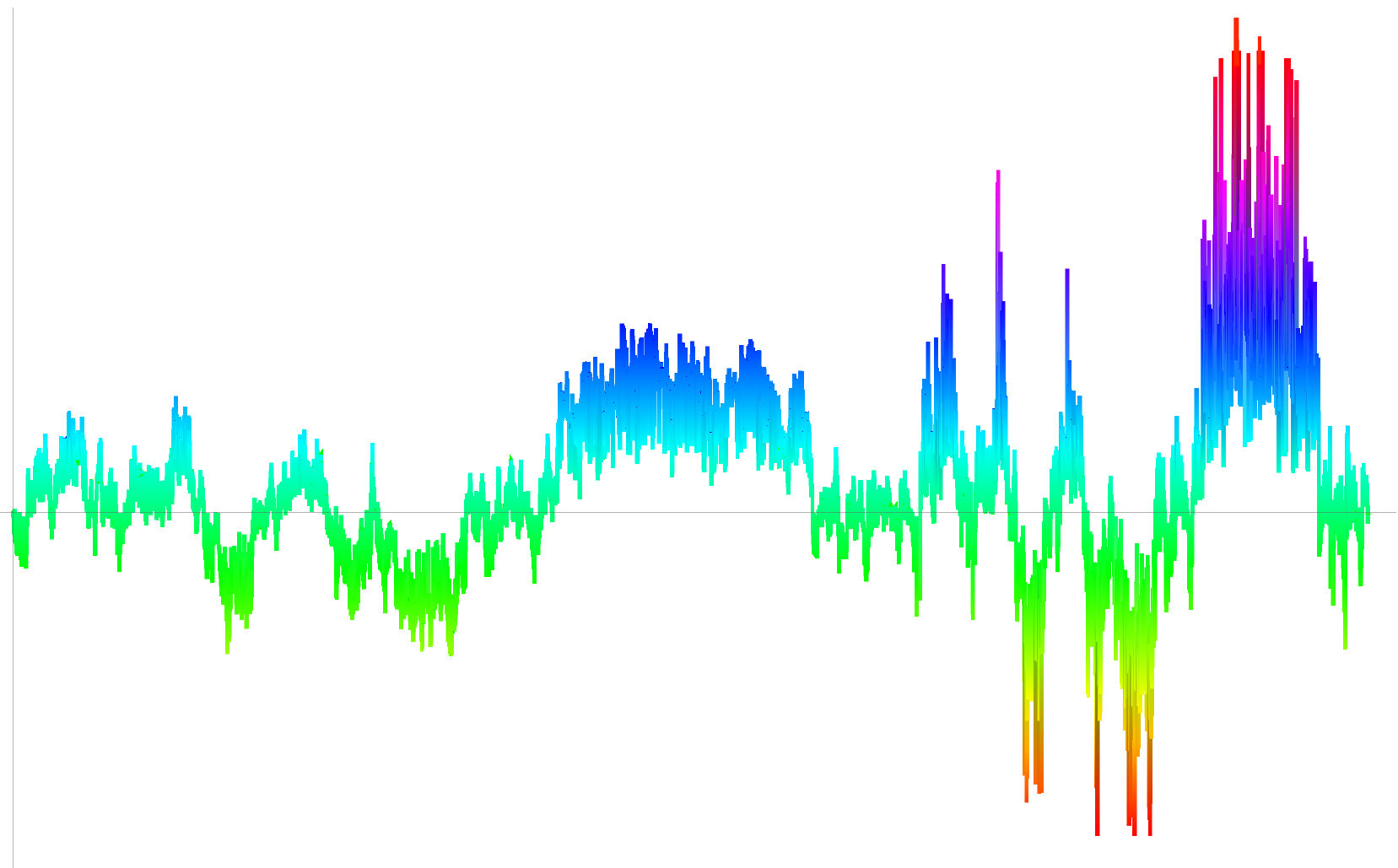}   \quad&\quad  \includegraphics[width=4.4cm]{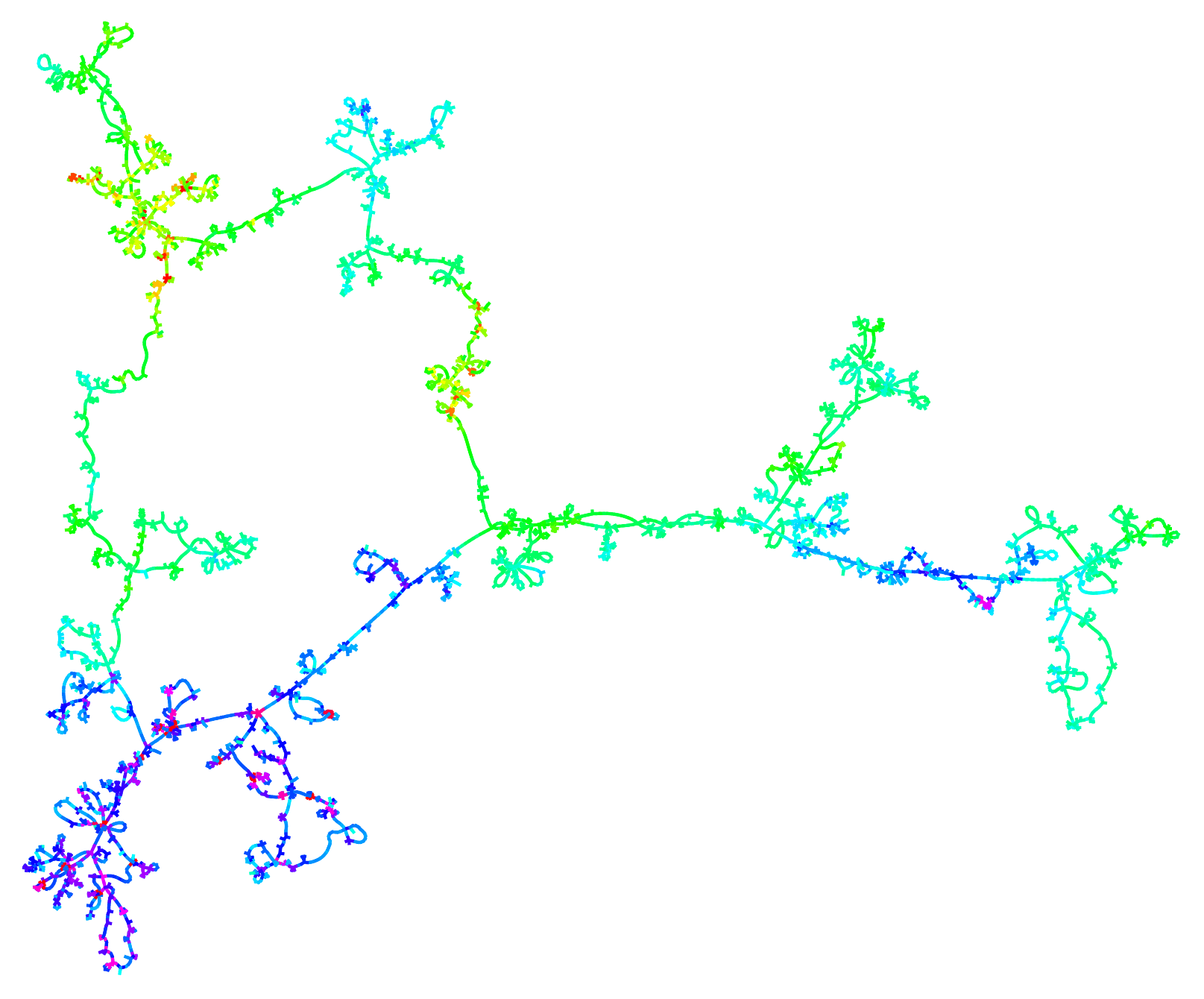}
  
  \end{tabular}

 \caption{\textbf{From top left to bottom right}: The stable excursion $X$, the looptree $ \mathcal{L}$ coded by $X$ with colors indicating the loops, the label process $Z$, and finally the same looptree with colors indicating the values of the process $Z$.}
 \end{center}
 \end{figure}
Now, for $0\leq s<t\leq 1$, set $[t,s]:=[0,s]\cup[t,1]$ and, for any $s,t \in [0,1]$, define 
  \begin{eqnarray} \label{eq:defD0}  \mathfrak{z}(s,t):=Z_{s}+Z_{t}-2 \max \left(\min_{[s,t]}Z,\min_{[t,s]} Z\right),  \end{eqnarray} and 
 \begin{eqnarray} \label{def:dstar}D^{*}(s,t):=\inf \sum \limits_{k=1}^{p} \mathfrak{z}(s_{k},t_{k})~,  \end{eqnarray}
where the infimum is over all choices of the integer $p\geq 1$ and all finite sequences $(s_{k},t_{k})_{1\leq k\leq p}$ such that $ t_{k} \sim_{d} s_{k+1}$, for every $1\leq k\leq p-1$, and $(s_1,t_p)=(s,t)$. Let us mention that the pseudo-distance $ \mathfrak{z}$ is usually denoted by $ D^{\circ}$ in the Brownian geometry literature. By construction, the function $D^{*}(\cdot,\cdot)$ is  a continuous pseudo-distance on $[0,1]$. We write $s\sim_{D^{*}}t$ if $D^{*}(s,t)=0$ and remark that $\sim_{D^{*}}$ is an equivalence relation. The $\alpha$-stable carpet/gasket is then obtained as the quotient space  $ \mathcal{S} = [0,1]/\sim_{D^{*}}$ endowed with the distance function induced by  $D^{*}$, which we still write $D^{*}$ by abuse of notation. We write $\Pi_{D^*}$ for the canonical projection $[0,1] \to [0,1]/\sim_{D^{*}}$ and $\mathrm{Vol}_{D^*}$ for the  pushforward of Lebesgue measure on $[0,1]$ under $\Pi_{D^*}$.\medskip

In the rest of this introduction, we will describe the main steps of the proof of Theorem \ref{thm:main} and compare them with the accomplished march towards the uniqueness of the Brownian sphere \cite{CS04,LG07,LG09,LG11,LGP08,MM06,Mie09,Mie11}.

\paragraph{1. Convergence of coding functions and subsequential limits.} As we said above, the first step is to encode our random maps $ \mathfrak{M}_{n}$ with random labeled trees using the Bouttier--Di Francesco--Guitter (BDG) construction, a variant of Schaeffer's bijection \cite{Sch98}. These random trees are further described by their contour and label processes, for which scaling limit results have been established \cite{LGM09}, and where the limit is given by the process $(X,Z)$ described above. From this, it is not hard to show a tightness result,  i.e. that for any given subsequence, we can further extract a subsequence $(n_{k})_{k\geq 1}$ along which 
 $$ \left( V(\mathfrak{M}_{n}) , (\scal n)^{-\tfrac{1}{2 \alpha}} \cdot \mathrm{d}^{\mathrm{gr}}_{\mathfrak{M}_{n}},\mathrm{vol}_{\mathfrak{M}_{n}}\right) \xrightarrow[n\to\infty]{(d)} \Big( [0,1] / \sim_{D}~,~  D  ~,~\mathrm{Vol}_D\Big),$$
where $D:[0,1]^{2}\to \mathbb{R}_+$ is a random pseudo-distance, $\sim_{D}$ is the equivalence relation defined by $s\sim_{D} t$ if and only if $D(s,t)=0$, and finally $\mathrm{Vol}_D$ stands for the pushforward of the Lebesgue measure on $[0,1]$ under the canonical projection associated with $\sim_D$. The previous convergence holds in the Gromov--Hausdorff--Prokhorov sense. If one forgets the measures $\mathrm{vol}_{\mathfrak{M}_{n}}$ and $\mathrm{Vol}_D$, this result already appears in the proof of Theorem 4 in \cite{LGM09}, and the argument can easily adapted to incorporate the measures, see Proposition \ref{theo:sub} for details. To prove Theorem \ref{thm:main}, it then suffices to establish uniqueness of the limit, i.e.~to show that $D=D^{*}$ regardless of the subsequence $(n_{k})_{k\geq1}$. 

As shown in Section \ref{sec:D<D*}, by coupling appropriately the pseudo-distance $D$ with the process $(X,Z)$, it follows easily from the discrete BDG construction that if   $t_{*}$ is the a.s.\ unique time    at which $Z$ realizes its minimum (Proposition \ref{distinct}), then:
\begin{eqnarray*}\label{eq:Dtetoile}
D(t_{*},s) &=& D^{*}(t_{*},s)=Z_s-Z_{t_*}, \\ 
D(s,t) &\leq& D^{*}(s,t),
\end{eqnarray*} for all $s,t \in [0,1]$. 
Although these might seem like rather weak statements, these properties are already sufficient to prove that, for \textit{any subsequential limit}, the Hausdorff dimension of $ ([0,1]/\sim_{D},D)$ is $2 \alpha$, see \cite{LG07} or  \cite{LGM09} in the case of the Brownian sphere.
Here, we see that the random time $t_*$ plays a distinguished role, and its image $\rho_*=\Pi_{D^*}(t_{*})$ in $ [0,1]/\sim_{D}$ will often be called the \textbf{root} of $ [0,1]/\sim_{D}$ (it can be seen as a random uniform point on $[0,1]/\sim_{D}$).

\paragraph{2. Topology of the stable carpet/gasket.}   Our first main contribution in proving that $D=D^{*}$ is to show that, for every $s,t\in[0,1]$, we have $$D(s,t)=0 \iff D^{*}(s,t)=0,$$ i.e.~that $D$ and $D^{*}$ identify the same points of $[0,1]$, see Theorem \ref{main_theorem_topology}. More precisely, we are able to describe (Theorem \ref{main_theorem_topology}) exactly the points which are identified by $D^{*}$ and by any sub-sequential limit metric $D$. In the case of the Brownian sphere, this was accomplished by Le Gall in the breakthrough paper \cite{LG07}. Our approach is however \textit{completely different} and relies on the presence of ``faces'' in $ [0,1]/\sim_{D}$ (there are no such notion of ``faces'' in the Brownian sphere). More precisely, given $t >0$ such that $\Delta_{t}>0$, and for $s\in [0,1]$; we let 
\begin{equation}\label{def:f:t}
 \mathrm{f}_{t}(s):=\inf \{r\geq t:X_{r}=X_{t}-s\cdot \Delta_{t}\}.
 \end{equation}
We then define the \textbf{face boundary} in $ [0,1]/\sim_{D}$ corresponding to that jump as the image  of $ \mathrm{f}_{t}( [0,1])$ under $\Pi_{D^*}$. We will also say that two points $s$ and $t$ are  \textit{trivially identified} 
$$ \text{ if }\:\:  s \sim_{d} t, \quad   \text{ or if }\:\: \mathfrak{z}(s,t)=0.$$
In the first case, the points are already identified in the looptree $ \mathcal{L}$, and in the second case, they are identified with the continuous equivalent of an edge in the BDG construction. In particular,  in all the previous cases we clearly have $s\sim_{D^{*}}t$.
 The main idea is then to use the ``faces'' to argue that as soon as the times $s$ and $t$ are not trivially identified, then they must be separated by a ``face'' which imposes $D(s,t) > 0$, see Figure \ref{fig:illusD=0}. Similar arguments have been used in \cite[Section 4.2]{bjornberg2019stable} when dealing with the so-called shredded spheres. Establishing these statements require to prove a sharper version of the ``cactus bound'' (see Lemma \ref{lem:cactusbound}) and to establish fine properties of the minima of $Z$, such as the fact that the local minima of $Z$ do not occur on the ``branches'' of the looptree (Proposition \ref{pinch_points_are_not_record}).  

\begin{figure}[!h]
 \begin{center}
 \includegraphics[width=8.5cm]{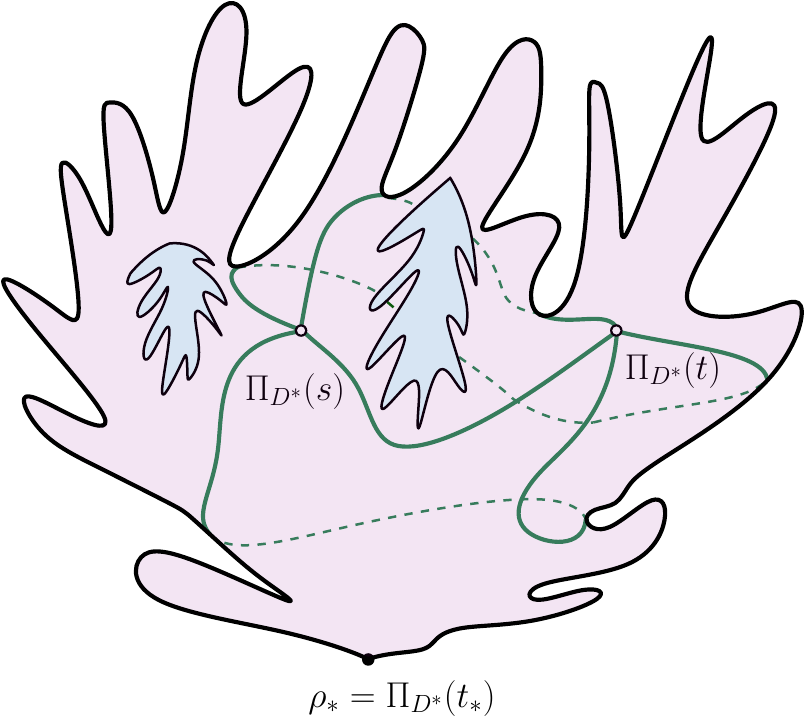}
 \caption{Illustration of the geometric underlying idea for the proof of $D=0\iff D^*=0$. The ``faces'' of $ \mathcal{S}$ are represented by the blue ``holes''.  Except for the trivial identifications, the distance $D$ cannot identify more points  since they must be separated by two faces. \label{fig:illusD=0}}
 \end{center}
 \end{figure}
 The equivalence $D=0\iff D^{*}=0$ allows us, by standard properties of compact spaces, to identify the quotient space $[0,1] / \sim_{D}$ with the space $\mathcal{S} = [0,1]/\sim_{D^{*}}$ equipped with their quotient topologies, and the volume measure $\mathrm{Vol}_D$ with $\mathrm{Vol}_{D^{*}}$.  We make this identification in the rest of the introduction, and in particular, we simply write $\mathrm{Vol}$ for the volume measure. In particular, the functions induced by $D$ and $D^{*}$ on the quotient can be seen as  two distances on $\mathcal{S}$ defining the same topology.  So one can wonder whether this (a priori random) topology can be characterized.  Here, a crucial difference happens depending on the position of $\alpha$ with respect to $3/2$. Indeed, we prove in Section \ref{sec:twopoints} that in the dilute phase $\alpha \in [3/2,2)$, the face boundaries  are simple non-intersecting curves in $( \mathcal{S},D)$,  whereas for $\alpha < 3/2$ they are self and mutually intersecting. In both cases, the Hausdorff dimension of a face boundary is $2$, see Proposition \ref{prop:HDface=2}. This dichotomy relies on an exact calculation (Theorem \ref{two_points_function} and Proposition \ref{sec:stable-map-3}) about the process $Z$, which can in some sense be seen as an extension of the connection between the Brownian snake and partial differential equations \cite{LeG99} to our case.  

In the dilute case, we establish that:
\begin{theo}[Topology in the dilute case]\label{main-topo-intro} When $\alpha \in [3/2,2)$ the topology of $( \mathcal{S},D^{*})$  (as well as that of any sub-sequential limit $( \mathcal{S},D)$) is almost surely that of the  Sierpinski carpet.
\end{theo}

The proof of the latter result  combines  Moore's theorem for quotients of the 2-dimensional sphere and a theorem of  Whyburn which establishes  that the Sierpinski carpet (on the sphere) is the unique homeomorphism type of a compact connected metric space $K$ embedded in the sphere $  \mathbb{S}_{2}$ such that its complement is made of countable many connected components $C_{1}, C_{2}, ...$ with the following properties:
\begin{itemize}
\item the diameter of $C_{i}$ goes to $0$ as $i \to \infty$;
\item $\bigcup_{i \geq 1} \partial C_{i}$ is dense in $K$;
\item the boundaries $ \partial C_{i}$ of $ C_{i}$ are simple closed curves which do not intersect each other.
\end{itemize}

The use of Moore's theorem mirrors the approach of Le Gall \&  Paulin to show that the Brownian sphere is homeomorphic to the $2$-dimensional sphere \cite{LGP08}. We also refer to \cite{Mie08} for an alternative proof in the case of the Brownian sphere.

In the dense case, we only show that almost surely, there exists a continuous injection from $\mathcal{S}$ to $\mathbb{S}_2$, see Lemma~\ref{prop:moore}. In fact, we believe that the topology of $ \mathcal{S}$ in this case is actually \textit{random}, in the strong sense that almost surely, two independent samples of $\mathcal{S}$ are not homeomorphic. We refer the reader to Figure \ref{fig:topologyrandom} in Section \ref{sect:topo:dilute} for heuristics, and  to \cite{yearwood22} for a similar behavior in the case of the SLE$_{\kappa}$ trace, with $\kappa >4$.

\paragraph{3. Two-point construction and the behavior of geodesics.} The third step in our program consists in understanding the behavior of geodesics in any subsequential limit $( \mathcal{S},D,\mathrm{Vol})$. The study of geodesics on the Brownian sphere was also instrumental in the proof of the uniqueness of the latter \cite{LG11,Mie11} and is still the subject of an intense research \cite{angel2017stability,LG11,LGstar,miller2020geodesics,mourichoux2024bigeodesic}, see also \cite{blanc2024geodesics,dauvergne202327,gwynne2021geodesic} for related contexts.  

To begin with, by passing the description of discrete geodesics in the Bouttier--Di Francesco--Guitter construction to the scaling limit, one can construct, for any $t \in [0,1]$, a  path $\gamma^{{(t)}}(\cdot)$ going from $\Pi_{D^*}(t)$ to the root $\rho_* = \Pi_{D*}(t_*)$ called a \textbf{simple geodesic}. Informally, these paths are obtained by re-rooting $Z$ at time $t$ and then following the associated running infimum process. It is easy to check that these paths are indeed geodesics for $D$ and $D^*$ towards the root $\rho_*$. In fact, the pseudo-distance $D^{*}$ can be seen as the largest pseudo-distance which passes to the quotient of the looptree and for which the simple geodesics are indeed geodesics.

We prove in Proposition \ref{thm:geodesics:rho:*} that the simple geodesics are in fact the only geodesics (for $D$ or $D^{*}$) within $ \mathcal{S}$ towards $\rho_{*}$, which is the analog of Le Gall's result \cite{LG09} in the Brownian sphere case. As a consequence, we prove that the cut locus of $ \mathcal{S}$ relative to $\rho_{*}$, defined as the set of points from which we can start two distinct geodesics to $\rho_{*}$, is a totally disconnected subset of $ \mathcal{S}$, and that the maximal number of distinct geodesics to $\rho_{*}$ that start from a given point is equal to $2$. See the discussion at the end of Section \ref{sec:classificationgeo}. This contrasts with the  Brownian sphere case, where the cut locus is a topological tree (a dendrite) with maximal degree $3$, see \cite{LG09}. One key feature of simple geodesics, derived from Proposition~\ref{sec:stable-map-3}, is that even though faces may not intersect each other in the dilute phase, the (simple) geodesics always bounce on faces even in the dilute phase, as illustrated in Figure \ref{fig:trap-2}. This property is instrumental in the surgery along geodesics discussed below.

Our method to study geodesics is to adapt the construction with two sources and  delays of \cite{Mie09} to general bipartite Boltzmann random maps. This can be seen as a variant of the Bouttier--Di Francesco--Guitter  construction, where the distances to the distinguished point of the map are replaced by the infimum of the distances to two uniform distinguished points $v_1^n, v_2^n \in \mathfrak{M}_n$, shifted by some additive delay. In return, this construction gives information about the set of points for which the difference of the distances to $v_1^n$ and $v_2^n$ takes a fixed value. We then take the scaling limits of this construction in Section~\ref{sec:scalingunicyclo} (compared to the more standard Brownian case studied in \cite{Mie09}, this step requires much more care in our ``stable'' setting). This yields Theorem \ref{alm-unique} which shows that there exists a unique  geodesic between two typical points in $ \mathcal{S}$ -- where a typical point means that it is obtained sampling a random variable with law $\mathrm{Vol}$. Adapting an argument of  Bettinelli \cite{bettinelli2016geodesics}, the essential uniqueness of  typical geodesic allows one to characterize all the geodesics towards $\rho_*$ as simple geodesic by an approximation procedure in the continuum. 

\paragraph{4. Surgery along geodesics.} We now come to the final step in proving $D=D^{*}$.
First remark that, by continuity considerations, it is enough to show that:
$$D(U_1,U_2)=D^*(U_1,U_2),$$
where $U_1$ and $U_2$ are two independent and uniform random variables on $[0,1]$. Recall that there is almost surely a unique geodesic $\gamma_{1,2}$ connecting the associated points $\rho_1=\Pi_D(U_1)$ and $\rho_2=\Pi_D(U_2)$. 
The starting idea is then similar to that of \cite{LG11,Mie11}: one needs to prove that  $\gamma_{1,2}$ can be well-approximated by pieces of simple geodesics targeting $\rho_{*}$. Both in \cite{LG11} and \cite{Mie11} this was done by estimating the dimension of ``bad'' points (namely, geodesic $3$-stars) along typical geodesics, that are points  $x$, in the range of $\gamma_{1,2}$, from which one can start a geodesic towards $\rho_{*}$ that does not coincide locally with a piece of $\gamma_{1,2}$. In \cite{LG11}, this estimate was performed using a kind of slice decomposition involving maps with a geodesic boundary, whereas in \cite{Mie11} the calculation was performed using the multipoint construction of \cite{Mie09} with three or four sources.  In our case, as for establishing Theorem~\ref{main_theorem_topology}, \textit{we exploit  the presence of faces to define our notion of bad points}. Here we only provide a heuristic description of the ideas in the continuum, and we refer to Sections \ref{sec:D=D*} and \ref{secP:uni:geo} for the precise arguments, which also rely on discrete considerations. Remember that it follows from Proposition~\ref{sec:stable-map-3} that simple geodesics bounce along faces of $ \mathcal{S}$. Roughly speaking then, a point $x$ on the geodesic $\gamma_{1,2}$ going from $\rho_1$ to  $\rho_2$ is \textbf{good} if there exist two different faces $\mathfrak{F}_1$ and $\mathfrak{F}_2$ of $\mathcal{S}$ such that (see Figure \ref{fig:trap-2} for an illustration):
\begin{itemize}
\item on its way from $x$ to $\rho_1$, the geodesic $\gamma_{1,2}$ touches both $ \mathfrak{F}_{1}$ and $ \mathfrak{F}_{2}$,
\item on its way from $x$ to $\rho_2$, the geodesic $\gamma_{1,2}$ touches both $ \mathfrak{F}_{1}$ and $ \mathfrak{F}_{2}$,
\item the pieces (in orange on Figure \ref{fig:trap-2}) of $\gamma_{1,2}$ linking $ \mathfrak{F}_{1}$ and $ \mathfrak{F}_{2}$ separate $x$ from $\rho_{*}$,
\end{itemize}
and we say that $x$ is a bad point otherwise. Notably, determining whether a point $x$  is good relies on planarity arguments. 
\begin{figure}[!h]
 \begin{center}
 \includegraphics[width=12cm]{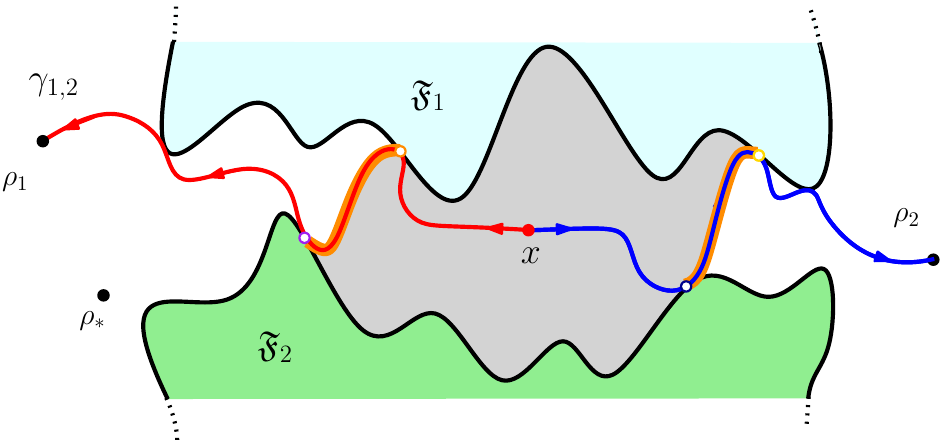}
 \caption{\label{fig:trap-2} Illustration of the neighborhood of a good point $x \in \gamma_{1,2}$.  The geodesic $ \gamma_{1,2}$ is drawn in red and blue. The faces $\mathfrak{F}_1$ and $\mathfrak{F}_2$ are drawn in light blue and green.  The points $\rho_1$ and $\rho_2$ could be in the same connected component as $x$ -- even if we will see that this is not the standard configuration.}
 \end{center}
 \end{figure}\\
The key point is that, if  $x$ is a good point, one could replace a piece of  $\gamma_{1,2}$ around $x$ with a concatenation of (at most two) geodesics directed toward $\rho_{*}$ and still obtain a geodesic going from $\rho_1$ to $\rho_2$.

  As in \cite{LG11,Mie11} the crucial step is then to prove that the dimension of the set of bad points along a typical geodesic is strictly less than $1$ (Proposition \ref{main:techni}). To prove this upper bound, we use the description of the local neighborhood around a uniform point along a typical geodesic provided by the multipoint construction with \textit{two} sources. The latter requires much of the general theory we developed on the process $(X,Z)$ and is perhaps the most technical part of the paper: compared to the Brownian sphere case, here we have an interplay between the stable jumps of $X$ and the Gaussian nature of $Z$ which complicates matters and requires very fine properties of stable L\'evy processes.  Finally, to approximate $\gamma_{1,2}$ in the neighborhood of bad points we need an \textit{a priori} bound of the type $D^{*} \leq D^{1-\delta}$ (locally) for some $\delta \in (0,1)$  as close to $0$ as wanted. In the case of the Brownian sphere, this estimate was derived from the ball volume estimates \cite[Corollary 6.2]{LG11}, see \cite{gall2021volume} for more recent works. In our case, no such precise calculation was available and we developped a robust method (Theorem \ref{technical_uniform_balls}) to prove stretched exponential tails for volume of balls in $( \mathcal{S}, D^*,\mathrm{Vol})$, which may be of independent interest.
\paragraph{Structure of the paper and tools used.}
For the reader's convenience, the paper is divided into two parts. The first one, purely in the continuum, is devoted to the study of the coding and label process $(X,Z)$ without reference to planar maps. In particular, we prove that the process $Z$ can be seen as the Gaussian free field on the stable looptree coded by $X$ and derive fine properties such as regularity properties and study of its records. We also lay the basis of an analog of the random snake theory of Le Gall, where the underlying coding object is a stable looptree instead of a random $\R$-tree. In particular, this leads to the exact calculations in Proposition \ref{sec:stable-map-3}.
The second part deals with the discrete planar maps. We recall and develop various encodings via labeled tree-like structures. Those labeled tree structures are then encoded by contour functions, which are shown to rescale towards the process $(X,Z)$, or variants thereof. We then use the information on $(X,Z)$ gathered in the first part to prove our scaling limits results following the strategy described above. To give a flavor of the broad type of tools which we will use along our journey, we provide a non-exhaustive list:
\begin{itemize}
\item Section \ref{sec:looptree}: Stable processes, stable looptree, $ \mathbb{R}$-trees coded by continuous functions
\item Section \ref{sec:constr_Z}: Regularity of Gaussian processes indexed by metric spaces, Dudley's theorem 
\item Section \ref{sec:local:minima}: Markov property and excursion theory for stable processes, regenerative sets and Hausdorff dimension of their intersections, Height process of Duquesne--Le Gall--Le Jan
\item Section \ref{sec:2pt}: Connection with integro-differential equations, Hypergeometric functions, Bessel functions, Bessel processes and their absolute continuity relations, Shepp covering theory
\item Section \ref{sec:prison}: Excursion theory, fluctuation theory of the stable L\'evy processes, conditioned stable processes, spine decomposition
\item Section \ref{sec:BDG}: Bouttier--Di Francesco--Guitter encoding, invariance principles for coding functions, Gromov--Hausdorff--Prokhorov topologies,  Jordan's theorem, Le Gall's re-rooting trick
\item Section \ref{sec:topology}: Quotient topology, (stable) laminations, Moore's theorem, Whyburn characterization of Sierpinski carpet, basic Hausdorff dimension theory (covering, Frostman lemma)
\item Section \ref{sec:D=D*}: Surgery along geodesic
\item Section \ref{sec:boltzm-stable-maps}: Encoding of  bi-marked planar maps, combinatorial decompositions
\item Section \ref{sec:scalingunicyclo}: Heavy-tailed random variables, local limit theorem, concentration,  Bretagnolle's theorem
\item Section \ref{secP:uni:geo}: Couplings, cut-locus, geodesic stars,  Jordan's theorem.
\end{itemize}

\paragraph{Brownian geometry, universality classes and open questions.} We end this introduction with a discussion on potential avenues opened by  this work. First of all, as recalled in the beginning of the paper, the developments around the Brownian sphere have led to a flourishing field now called Brownian geometry. In particular, analogs of the Brownian sphere have been defined in different topologies such as that of the plane \cite{CLGplane}, the disk \cite{BM15}, the cylinder \cite{gall2024drilling} or in higher genus  \cite{BeMi22}. A theory of ``calculs of continuous surfaces'' is currently being developed where cutting, gluing and drilling holes are the basic operations allowed \cite{bouttier2022bijective,caraceni2019self,GM16a}. Continuous Markovian explorations of those surfaces mimicking the discrete ``peeling process'' \cite{CurStFlour} are also the subject of active research \cite{le2024peeling,gall2023spatial}.
  In parallel to this, the construction of Brownian surfaces based on conformal random geometry and Liouville quantum gravity provides powerful alternative tools to study Brownian surfaces, based on the coupling between SLE and the Gaussian free field, and the mating of tree theory of quantum surfaces \cite{DMS14,gwynne2019mating,miller2016imaginary,She10}. In particular, this has led to the definition of Brownian motion on Brownian surfaces, and to the understanding of the scaling limits of percolation interfaces and of the self-avoiding walk on random maps \cite{berestycki15,GaRhVa16,GM16a,gwynne2017convergence,GwMiSh21}. 

  Developing a companion theory in the ``stable paradigm'' seems now a possible yet challenging goal. The Brownian sphere has also been shown to be the scaling limits of an increasing list of discrete map and graph models. Widening the basin of attraction of the stable carpets/gaskets, by extending Theorem \ref{thm:main} to non-bipartite maps or maps with prescribed face degrees, is a natural objective. Notice also that our work may help proving uniqueness of several variants such as the L\'evy maps recently considered by Kortchemski \& Marzouk in \cite{kortchemski2024random} or the scaling limits of quadrangulations with high degrees introduced by Archer, Carrance and M\'enard \cite{archer2024stable}.
In parallel, non-generic Boltzmann random maps and their scaling limits are also believed to be tightly connected with Liouville Quantum Gravity (LQG), see e.g. \cite{gwynne2019mating,SheHC}. One of these conjectural links is the fact that large critical loop $O(N)$-decorated random planar maps converge in the scaling limit and after uniformization on the sphere towards $\gamma$-Liouville Quantum Gravity decorated by an independent nested Conformal Loop Ensemble (CLE) with parameter ${\kappa}$ where 
$$\alpha = \frac{1}{2} + \frac{4}{\kappa} = \frac{3}{2} \pm \frac{1}{\pi} \arccos(N/2), \quad \mbox{and} \quad \kappa \in \{ \gamma^{2}, 16/\gamma^{2}\}. $$
As demonstrated in \cite{BBG12} for the case of quadrangulations, the gasket (i.e. the part of the map not disconnected by a loop from the root edge) of a critical loop $O(N)$-decorated map is a Boltzmann  planar map whose weight sequence is non-generic with exponent given by the above value of $\alpha$.  See also \cite{albenque2022geometric,chenTurunen2020critical,chen2023ising} for the case of triangulations decorated by an Ising model, or \cite{bernardi2019boltzmann,curien2018duality} for the case of Bernoulli percolation. This shows a first link between our stable carpet/gasket and (non-nested) CLE$_{\kappa}$, and in particular the dichotomy for the topology proved in Theorem \ref{main-topo-intro} parallels that of the topology of CLE$_{\kappa}$, see \cite{RohdeSchramm05,SheCLE,SheWer12}. Many recent works have studied geometric properties of CLE$_{\kappa}$ (possibly on top of $\gamma$-LQG). Of particular interest is the definition of a ``percolation exploration''  \cite{miller2017cle} inside simple CLE$_{\kappa}$ for $\kappa \in(8/3, 4)$ giving rise to a non-simple CLE with the dual parameter $\kappa'  = \frac{16}{\kappa}$, with a similar story in the discrete planar map setup \cite{curien2018duality}. Defining such continuous variant of the SLE$_{6}$ or even Brownian motion directly on the $\alpha$-stable carpet (or gasket) is a challenging open problem.  From a metric point of view, the intrinsic ``chemical'' distance inside CLE$_{\kappa}$ has recently been shown to exist in  a series of breakthroughs \cite{AMY24,miller2021tightness} first in the dense case. Based on the above uniformization conjecture,  we expect our $\alpha$‑stable carpet/gasket to coincide with the $\gamma$‑LQG‑induced chemical metric inside CLE$_\kappa$, once $\kappa$ and $\gamma$ are coupled appropriately. Establishing this link would bridge the two metric perspectives under a common paradigm of random geometry.

\bigskip

\textbf{Acknowledgments.} We thank Quentin Berger, Jean Bertoin,  Thomas Budzinski, Thomas Duquesne,  Jean-François Le Gall, Thomas Leh\'ericy, Cyril Marzouk, Mathieu Mourichoux,  Lo\"\i c Richier and especially   Alejandro Rosales-Ortiz for discussions that happened at various stages of this project. We are also grateful to Ewain Gwynne for pointing us \cite{yearwood22}. The first and last authors were supported by ERC GeoBrown (ERC 740943 GeoBrown). N.C. is supported by SuPerGRandMa, the ERC CoG 101087572.

\tableofcontents
\clearpage
\noindent \textbf{\Large Index of notation}

\medskip
To help the reader navigate these pages, we make a list of some of the most important notation that are used across several sections. We also gather here a few classical constructions that will be employed several times along the way:  
\medskip

\noindent \textbf{General notation for distances and pseudo-distances.} Consider $\mathrm{dis}:\mathtt{I}\times \mathtt{I}\to \mathbb{R}_{+}$ a pseudo-distance  on an interval $\mathtt{I}\subset \mathbb{R}_+$. The $\mathrm{dis}$-ball of radius $r$ centered at $x\in \mathtt{I}$ is the set $$ B_ \mathrm{dis}(x,r) := \{y\in \mathtt{I}:~\mathrm{dis}(x,y)\leq r\}.$$
 We write $\sim_{\mathrm{dis}}$ for the equivalence relation on $\mathtt{I}$ defined by $x\sim_{\mathrm{dis}}y$ if and only if $\mathrm{dis}(x,y)=0$. We shall still use the notation $\mathrm{dis}$ for the distance on the quotient $\mathtt{I}/\sim_{\mathrm{dis}}$ and denote  the canonical projection by $\Pi_{{\mathrm{dis}}} : \mathtt{I} \to \mathtt{I} /\sim_{\mathrm{dis}}$. We also write $ \mathrm{Vol}_{\text{dis}}$ for the pushforward of the Lebesgue measure on $\mathtt{I}$ by $\Pi_{ \mathrm{dis}}$.  
\medskip

\noindent\textbf{$ \mathbb{R}$-tree coded by continuous excursions.} Recall that for $0<s<t < 1$ we set $[t,s] = [0,s]\cup[t,1]$ as well as $(t,s) = [0,s) \cup (t,1]$. If $F : [0,1] \to \mathbb{R}$ is a continuous function satisfying $F(0)=F(1)=0$, we define a pseudo-distance (denoted with a mathfrak font) by the formula 
 \begin{eqnarray} \label{eq:defpseudodistancearbre} \mathfrak{f}(s,t) = F(s)+F(t) -2 \max \left(\min_{[s,t]} F ; \min_{[t,s] } F \right).  \end{eqnarray} In particular, if $F : [0,1] \to \mathbb{R}_+$ is a non-negative excursion, then the right-hand side of the last display simplifies to $F(s)+F(t) - 2 \min_{u \in [s \wedge t, s \vee t]}F(u)$. The quotient of $[0,1]$ by the equivalence relation $\sim_{\mathfrak{f}}$, equipped with $ \mathfrak{f}$ is an $\mathbb{R}$-tree that we denote by $ \mathcal{T}_ \mathfrak{f}$, see Section \ref{sec:codagearbre}. 

\paragraph{Usual notation}\ \\ 

\begin{tabular}{cl}
$ \mathbb{N}$& natural numbers $\{1,2,3, ...\}$ \\
$:= \mbox{ or } =:$ & definition of a mathematical object\\
$ \# E$ & cardinality of $E$\\
$\mathrm{Cl}(E)$ & closure of $E$\\
$\llbracket a,b \rrbracket$ & set of integers $\{a,a+1, ... , b-1,b\}$\\
$ \mathbb{D}( I,E)$ & space of rcll functions from $I$ to $E$, endowed with the Skorokhod J1 topology\\
$ \mathcal{C}( I,E)$ & space of continuous functions \\ & endowed with the topology of  uniform convergence on every compact\\
$ \mathbb{M}, \mathbb{M}_{\mathrm{root}}$& space of isometry classes of (rooted) weighted compact  metric spaces\\
$ \mathbb{PM}, \mathbb{PM}_{\mathrm{root}}$& space of isometry classes of of (rooted) iweighted  geodesic compact  metric spaces\\
$ V( \mathbf{g})$ & vertex set of the graph $ \mathbf{g}$\\
  $ \mathrm{d}^{ \mathrm{gr}}_{ \mathbf{g}}$ & graph distance on the vertex set $V(\mathbf{g})$ of a graph $\mathbf{g}$\\
  $\mathrm{vol}_{\mathbf{g}}$ & uniform probability measure on the set of vertices of the finite graph $\mathbf{g}$\\
  $ \lesssim$ & the LHS is bounded above by a universal constant times the RHS\\
  $ x_{n} = o(y_{n})$ & $x_{n}/y_{n} \to 0$ as $n \to \infty$\\
\end{tabular}\ \\

\begin{tabular}{cl}
 $ x_{n} = O(y_{n})$ & $x_{n}/y_{n}$ is bounded as $n \to \infty$\\
$ X_{n} = o_{a.s.}(Y_{n})$ & $X_{n}/Y_{n} \to 0$ almost surely as $n \to \infty$\\
$ X_{n} = O_{a.s.}(Y_{n})$ & $(X_{n}/Y_{n} : n \geq 0)$ is  bounded almost surely\\
$ X_{n} = o_{ \mathbb{P}}(Y_{n})$ & $X_{n}/Y_{n} \to 0$ in probability as $n \to \infty$\\
$ X_{n} = O_{ \mathbb{P}}(Y_{n})$ & $(X_{n}/Y_{n} : n \geq 0)$ is tight\\

\end{tabular}

\paragraph{Stable, Brownian and Bessel processes}\ \\ 

\begin{tabular}{cl}
$(X_t)_{ t \geq 0}$& canonical rcll process on $\mathbb{D}(\R_+,\R)$\\
$ \mathbf{Q}$ & law of a spectrally positive $\alpha$-stable L\'evy process starting from $0$\\
$ q^{[\alpha]}_{c}$ & density of $X_{c}$ under $ \mathbf{Q}$\\
$ \mathbf{N}$& Ito's excursion measure for spectrally positive $\alpha$-stable process  \eqref{eq:excursionmeasuredecomp}\\
$ \sigma$ & lifetime of $X$ under $ \mathbf{N}$\\
$ \mathbf{P}$& normalized excursion measure\\

$ \mathbf{N}^\bullet$ & biased excursion measure, see \eqref{def:N:bullet}\\
$ t_\bullet$ & $\in [0, \sigma]$ distinguished time under $ \mathbf{N}^\bullet$\\

\ \\

$(B_t)_{ t \geq 0}$& canonical continuous process on $\mathcal{C}(\R_+,\R)$\\
 $\P_x$ & law of Brownian motion started from $x$\\
  $\P_x^{{(s)}}$ & law of Brownian motion started from $x$ with life time $s>0$\\
 $\P^{\langle\nu\rangle}_x$ & law of Bessel process with index $\nu$ (i.e. dimension $2 \nu +2$) started from $x$\\
 $I_{\nu}$ & modified Bessel function of the first kind with index $\nu$\\
  $\P_{x\to y}^{(s)}$ & law of Brownian bridge with lifetime $s$ starting from $x$ and ending at $y$
  \end{tabular}

\paragraph{Looptree}\ \\ 

\begin{tabular}{cl}
$U_i : i \geq 1$ & i.i.d.\ random variables uniform on $[0,1]$ and independent of all others\\
$X$& underlying L\'evy process \\
$ \Delta_t $ & $= \Delta_t X$, i.e.\ the jump  of $X$ at time $t$\\
$ (\mathrm{b}_{i} : i\in \mathbb{N})$ & enumeration of the jump times of $X$\\
$I_{s,t}$ & $ = \inf \{ X_{u} : u \in [s,t]\}$\\
$I_t$ &$=I_{0,t} = \inf\{ X_u : u \in [0,t]\}$\\
$  s \preceq t $ & ancestor relation defined by $s\preceq t$ if $s\leq t$  and $I_{s,t}\geq X_{s-}$\\
$ s \prec t$ & strict ancestor defined by $s \preceq t$, $s<t$ and $X_{s-} < X_{t}$\\
$ s \curlywedge t $ & most recent common ancestor\\
$x_{s,t}$ & $=I_{s,t}-X_{s-}$, i.e. the position in the jump associated to $s$ of the lineage going to $t$\\
 $d$ & pseudo-distance on $[0,1]$ and distance on $\mathcal{L}$, see \eqref{def:distancelooptree}\\
\end{tabular}
\ \\

\begin{tabular}{cl}
 $\mathcal{L}$ & $ = [0,1]/ \sim_{d}$, the looptree associated with $X$\\
 $ \mathrm{f}_{t}$ & face of $ \mathcal{L}$ associated with the jump at time $t$, see \eqref{def:loops}\\
 $ \mathrm{Loops}$ & the set of all loops in $ \mathcal{L}$ (or their pre-images in $[0,1]$)\\
 $ \mathrm{Skel}$ & the set of all pinch points in $ \mathcal{L}$, called the skeleton, (or their pre-images in $[0,1]$)\\
  $ \mathrm{Leaves}$ & the set of all leaves in $ \mathcal{L}$, i.e. points of degree $1$ (or their pre-images in $[0,1]$)\\
$\mathrm{Branch}(s,t)$& a set of times whose images are the points separating $\Pi_{d}(s)$ and $\Pi_{d}(t)$\\
 
\end{tabular} 

\paragraph{Label process}\ \\ 

\begin{tabular}{cl}
 $Z$ & the continuous label process   \\
$ (\mathrm{b}_{i} : {i\in \mathbb{N}})$ & Brownian bridges of duration $1$ associated to each jump time $t_{i}$\\
 $t_*$ & the unique $t\in[0,1]$ such that $Z_{t_*}=\min Z$, see Lemma \ref{distinct}\\
 $ \mathcal{T}_{ \mathfrak{z}}$ & tree coded by the process $Z$ (re-rooted at $t_{*}$)\\
  $ \mathrm{LeftRec}$ & set of local minimal record times on the left of $Z$\\
    $ \mathrm{RightRec}$ & set of local minimal record times on the right of $Z$\\

\end{tabular}

\paragraph{Boltzmann measures}\ \\ 

\begin{tabular}{cl}
$ \mathbf{q}$ &  $= (q_{k})_{k \geq 1}$  weight sequence\\
$ \mathcal{M}$ & set of all planar  rooted bipartite maps\\
$ \mathcal{M}^{\bullet}$ & set of all planar  rooted pointed bipartite maps\\
$ w_{ \mathbf{q}}$ & Boltzmann measure on planar maps with weight sequence $ \mathbf{q}$\\
$ w_{ \mathbf{q}}^\bullet$ & Boltzmann measure on pointed planar maps i.e. $w_\bq^\bullet((\bm,v))= w_\bq(\bm)$\\
  $\zbq$ & $=w_\bq^\bullet(\mathcal{M}^\bullet)/2$\\
$ \scal$ & normalizing constant defined in \eqref{eq:nongenerictail} or \eqref{eq:nongenericgenf}\\
\end{tabular}

\paragraph{Planar maps}\ \\ 

\begin{tabular}{cl}
\hspace*{-1.5cm} 
$\mathfrak{M}_{n}$ & a random $ \mathbf{q}$-Boltzmann  conditioned to have $n$ vertices\\
\hspace*{-1.5cm} $([0,1]/\sim_D,D, \mathrm{Vol}_D)$ & an accumulation point of $\big( V(\mathfrak{M}_{n}) , (\scal n)^{-\frac{1}{2\alpha}} \cdot \mathrm{d}^{ \mathrm{gr}}_{ \mathfrak{M}_{n}}, \mathrm{vol}_{\mathfrak{M}_{n}} \big) _{n\geq 1}$, see Proposition \ref{theo:sub}\\
\hspace*{-1.5cm} $(n_k : k \geq 1)$ & the subsequence selected in Proposition \ref{theo:sub}\\
\hspace*{-1.5cm} $(m_k : k \geq 1)$ & the subsequence of $(n_{k})_{k \geq 1}$ selected at the end of Section  \ref{secP:uni:geo}\\
\hspace*{-1.5cm} $ \mathcal{S}$ &  the quotient space $[0,1]/\sim_{D^*}$ (latter identified with $ [0,1]/\sim_{D}$ by Theorem \ref{main_theorem_topology})\\
\hspace*{-1.5cm} $\mathrm{Vol}$ & the volume measure on $\mathcal{S}$ (equal to both $ \mathrm{Vol}_D$ and $ \mathrm{Vol}_{D^*}$ by Theorem \ref{main_theorem_topology}) \\
\hspace*{-1.5cm} $\rho_*$ & the root $\Pi_D(t_*)$ in $ \mathcal{S}$\\
\hspace*{-1cm} $\gamma^{(s)}$ & simple geodesic associated with $s\in[0,1]$, see Section \ref{sec:D<D*}\\
\hspace*{-1.5cm} $\gamma^{(s\to t)}$ & path from $\Pi_D(s)$ to $\Pi_D(t)$ obtained by following $\gamma^{(s)}$ and $\gamma^{(t)}$ until they merge

\end{tabular} 
\\
\\
\\
\paragraph{BDG$^\bullet$ construction}\ \\ 

\begin{tabular}{cl}
$ \bT=( \mathcal{T}, \ell)$ & labeled mobile\\
$ V_{\circ}( \mathcal{T}), V_\bullet( \mathcal{T})$ & sets of white and black vertices\\
$ \widehat{V}_\circ( \mathcal{T})$ & set of white leaves\\
$ k_u( \mathcal{T})$ & number of children of the vertex $ u$ in $ \mathcal{T}$\\
$[c,c']_{ \mathcal{T}}$ & interval of corners between $c$ and $c'$ in $ \mathcal{T}$\\
$[v,v']_{ \mathcal{T}}$ & minimal interval of vertices in the clockwise contour between $v$ and $v'$ in  $ \mathcal{T}$\\
$ \mu_{\circ}, \mu_{\bullet}$ & offspring distributions of the label mobile defined in \eqref{eq:mucirc_mubullet}\\
$ \mathrm{GW}_{ \mathbf{q}}( \mathrm{d}  \bT)$ & law of the well-labeled mobile\\
$\bTn$ & random mobile of law $ \mathrm{GW}_\mathbf{q}( \cdot \mid V_\circ = n-1)$\\
$ ( S^{ \bT}_k, L^{\bT}_k)$& Lukasiewicz and label path of $  \bT$, see Figure \ref{fig:coding}\\
$\mathrm{BDG}^\bullet( \bT, \epsilon)$ & the BDG pointed map associated with $( \bT, \epsilon)$, see Figure \ref{fig:BDGconstruction}\\
$ v_*$ & the distinguished vertex of $\mathrm{BDG}^\bullet( \bT, \epsilon)$\\
$ \gamma^{{(c)}}$ & simple geodesic starting from the corner $c$ \\
\end{tabular}

\paragraph{BDG$^{2\bullet}$ construction}\ \\ 

\begin{tabular}{cl}
$ \mathcal{U}$ & set of well-labeled unicyclomobiles $( \mathbf{u}, \ell)$\\
$ \Edelay_{ \mm, v_{1},v_{2}}$ & set of  admissible delays $ \delay$ of a bi-pointed map $(\bm,(v_1,v_2))$\\
$ \mathcal{M}^{{2 \bullet}}$& space of rooted bi-pointed maps with an admissible delay\\
$ w_ \mathbf{q}^{2 \bullet}$ & $ \mathbf{q}$-Boltzmann law defined by $w_\bq^{2\bullet}\big(\bm,(v_1,v_2),\delay\big) = w_ \mathbf{q}( \bm)$\\
$ \mathrm{BDG}^{2 \bullet}( \mathbf{u}, \epsilon)$ & the bi-pointed map associated with  the unicyclomobile $ \mathbf{u}$ and $ \epsilon$, see Figure \ref{fig:unicyclo2}\\
$ w_ \mathbf{q}^{2 \bullet}$ & $ \mathbf{q}$-Boltzmann law defined by $w_\bq^{2\bullet}\big(\bm,(v_1,v_2),\delay\big) = w_ \mathbf{q}( \bm)$\\
$ \tilde{w}_ \mathbf{q}^{2\bullet}$ & pushforward of $ w_ \mathbf{q}^{2 \bullet}$ by the inverse of $ \mathrm{BDG}^{2 \bullet}$ and forgetting the sign\\
$ \mathbf{u}_n$ & random unicyclomobile of law $\tilde{w}_ \mathbf{q}^{2\bullet}( \cdot \mid \# V_\circ = n-2)$\\
$( \mathfrak{M}_n, (\widehat{v}_1^n, \widehat{v}_2^n), \hdelay_n)$ & $= \mathrm{BDG}^{2\bullet}( \mathbf{u}_n, \epsilon)$  bi-pointed map with delay coded by $ \mathbf{u}_n$ and a sign\\
$J_n$ & a white vertex minimizing the label on the cycle of $ \mathbf{u}_n$\\
\end{tabular}

\part{Properties of the label process $\mathbf{Z}$}\label{PartI}

This part is purely ``in the continuum'' and is devoted to the study of the  label process $Z$. After a first section reviewing the construction of the stable looptree $ \mathcal{L}$ from a spectrally positive $\alpha$-stable L\'evy excursion $X$, we recapitulate the procedure to construct $Z$   once the jumps of $X$ are  decorated with independent Brownian bridges. We will show that the process $Z$ can alternatively be thought of as the Brownian motion  indexed by $ \mathcal{L}$ (or equivalently as a Gaussian Free Field on the looptree $\mathcal{L}$).  This enables us to establish quantitative  properties of $Z$ using the theory of Gaussian processes. In view of our applications to random maps, a particular attention is devoted to the study of the local minimal records of $Z$, see Propositions~\ref{pinch_points_are_not_record}, ~\ref{lem:non-icnreasealongbranches} and~\ref{prop:records-loops}. In a sense, this part can be seen as the starting point of a ``Brownian Snake theory'' for the Brownian motion indexed by $ \mathcal{L}$. In particular, Theorem \ref{two_points_function}  is  the stable counterpart of the link between the Brownian snake (on Aldous CRT) and the equation $\nabla^2 u = u^{2}$ which was central in the theory developed by Le Gall~\cite{LG99}.\bigskip

For notational convenience, we always work on the canonical space $ \mathbb{D}(  \mathbb{R}_{+}, \mathbb{R})$ of rcll functions endowed with the Skorokhod topology and we denote the canonical process by $X$. Extending the notation of the introduction, for $t\geq 0$,  we write $\Delta_t := X_t- X_{t-}$ for the associated jump and  $I_t := \inf_{s \in [0,t]} X_s$ for the running infimum. It will also be useful, for every $s,t\geq 0$ with $s\leq t$, to set
\[I_{s,t}:=\inf\limits_{[s,t]} X.\]
We also consider $(\mathrm{t}_i)_{i\in \mathbb{N}}$, a measurable indexing of the jumping times of  $X$, with the convention that $\mathrm{t}_i=\infty$ if $X$ has fewer than $i$ jumps -- in all cases we will consider,  $X$ will have infinitely many jumps.
\\
\\
 We endow $\mathbb{D}(  \mathbb{R}_{+}, \mathbb{R})$ with various measures that change the law of the canonical process $X$:
\begin{itemize}\item  Under the probability measure $ \mathbf{Q}$, the process  $(X_{t})_{t\geq 0}$ is  an $\alpha$-stable L\'evy process starting from $0$ with no negative jumps with Laplace exponent $ \lambda\mapsto \lambda^{\alpha}$, for $\lambda\geq 0$, or equivalently with  L\'evy measure given by   
 \begin{eqnarray} \label{eq:levymeasure} \  \frac{\mathrm{d}r}{\Gamma(-\alpha)r^{\alpha+1}}  \mathbf{1}_{r >0} .  \end{eqnarray}
\item Under the probability measure $ \mathbf{P}$,  the process  $(X_{t})_{t\geq 0}$   is a normalized excursion of lifetime equal to $1$ of the above stable L\'evy process. In particular, we have $X_t=0$, for every $t\geq 1$.
\item The  sigma-finite measure $ \mathbf{N}$ is the corresponding excursion measure which can be written as follows. For every $v>0$, let $\mathbf{N}^{(v)}$ be the law of the excursion of length $v$ (obtained from $ \mathbf{P}$ by scaling). The sigma-finite measure  $\mathbf{N}$ is defined by the relation:
  \begin{eqnarray} \label{eq:excursionmeasuredecomp}\mathbf{N}(A):=\int_{0}^{\infty} \mathbf{N}^{(v)}(A)\frac{\d v}{|\Gamma(-\frac{1}{\alpha})| v^{\frac{1}{\alpha}+1}},  \end{eqnarray}
for any measurable set $A$. 
\end{itemize}
The excursion measure $ \mathbf{N}$ can also be defined as follows.  Under $\mathbf{Q}$, the process $(X_{t}-I_{t})_{t\geq 0}$ is a strong Markov process and $0$ is a regular recurrent point. Moreover, $(-I_{t})_{t\geq 0}$ is a local time for $(X_{t}-I_{t})_{t\geq 0}$  at $0$. The measure $ \mathbf{N}$ is the excursion measure of $(X_{t}-I_{t})_{t\geq 0}$ away from $0$ associated with the local time  $(-I_{t})_{t\geq 0}$, see \cite[Chapters VI and VII]{Ber96} for more details.

\section{The stable Looptree}
\label{sec:looptree}
In this section we work under the probability $ \mathbf{P}$, so that $X$ is a stable L\'evy excursion with lifetime~1. We recall  the construction of the looptree $ \mathcal{L}$ associated with  $X$ given in \cite{CKlooptrees}, and we establish some elementary properties. Informally, this random compact metric space can be  obtained by gluing loops along the jumps of $X$ and following the tree structure encoded by $X$. In this direction, recall from the introduction, that for every $s,t\in [0,1]$ with $s\leq t$, we write $s\preceq t $ if $s\leq t$ and $I_{s,t}\geq X_{s-}$ and say that $s$ is an \textbf{ancestor} of $t$. 
We say that $s$ is a \textbf{strict ancestor} of $t$ and write $s\prec t$ if $s <t$ and, furthermore, $X_{s-} < X_{t}$. In particular remark that $0$ is an ancestor of every $t\in [0,1]$. The relation  $\preceq$ is a partial order on $[0,1]$ and   we can define a notion of  most recent  common ancestor of $s$ and $t$ as follows:
 \[s \curlywedge t:=\sup \big\{r\leq s\wedge t:~ r\preceq s\:\:\text{and}\:\:r\preceq t\big\}.\]
  When $s$ is an ancestor of $t$ we set $x_{s,t}:=I_{s,t}-X_{s-}$, and when it is not the case we simply  take $x_{s,t}:=0$.  Recall finally that  $(\mathrm{t}_{i})_{i\in \mathbb{N}}$ is  a measurable enumeration of the jumping times of $X$.

 \subsection{Definition of the $\alpha$-stable looptree}\label{sec:defloop}
We start by introducing the distance of $[0,1]$ seen as a loop of length $1$ after identifying $0$ and $1$, i.e. for every $0 \leq s,t \leq 1$, we set 
\begin{align*}
\delta(s,t):=\min\big(|t-s|, 1-|t-s|\big).
\end{align*}
Next, for $s\preceq t$,  we consider the quantity
\begin{equation}\label{def:d:0:1}
d_{0}(s,t):=\mathop{\sum_{s\prec r\preceq t}}_{\Delta_{r}>0} \Delta_{r} \cdot \delta\big(0,\frac{x_{r,t}}{\Delta_{r}}\big).
\end{equation}
And finally, for every $s,t\in[0,1]$, we take
  \begin{eqnarray} \label{def:distancelooptree} d(s,t):=d_{0}(s\curlywedge t,s)+d_{0}(s\curlywedge t,t)+\Delta_{s\curlywedge t} \cdot \delta\big(\frac{x_{s\curlywedge t,s}}{\Delta_{s\curlywedge t}},\frac{x_{s\curlywedge t,t}}{\Delta_{s\curlywedge t}}\big),  \end{eqnarray}
with the convention $\Delta_{s\curlywedge t}\cdot \delta\big(\frac{x_{s\curlywedge t,s}}{\Delta_{s\curlywedge t}},\frac{x_{s\curlywedge t,t}}{\Delta_{s\curlywedge t}}\big):=0$ if $\Delta_{s\curlywedge t}=0$. By  \cite[Proposition 2.2]{CKlooptrees} the function $d$  is, under $\mathbf{P}$, a continuous pseudo-distance on $[0,1]$, and in particular the equivalence relation $\sim_{d}$ defined on $[0,1]$ is $\mathbf{P}$-a.s.\ closed. The $\alpha$-stable looptree coded by $X$ is  the random compact metric space $$\mathcal{L}:=([0,1]/\sim_{d},d).$$  The point $\Pi_{d}(0)=\Pi_{d}(1)$ will be interpreted as the \textbf{root} of $\mathcal{L}$, where we recall that $\Pi_{d}:[0,1]\to [0,1]/\sim_{d}$ stands for the canonical projection. Moreover, we  interpret  $\mathrm{Vol}_{d}$, the pushforward of the Lebesgue measure on  $[0,1]$ under the associated canonical projection as the uniform measure on  the looptree $\mathcal{L}$.

 We will now   describe the equivalence classes of $\sim_d$ and classify the different types of points in $ \mathcal{L}$. In this direction, we first recall a few basics $\mathbf{P}$-almost sure properties  of $X$: 
\begin{itemize}
\item  \hypertarget{prop:A:1}{$(A_{1})$} For every $t\in(0,1)$ such that $\Delta_{t}>0$, we have:
\[\inf\limits_{[t,t+\varepsilon]}X<X_{t}\:\:\text{and}\:\:\inf\limits_{[t-\varepsilon,t]}X<X_{t-},\quad \text{ for every }\varepsilon\in\big(0,t\wedge(1-t)\big).\]
In particular the local infima of $X$ are  realized, i.e.~are local minima. 
\item \hypertarget{prop:A:2}{$(A_{2})$}  All the local minima of $X$ are distinct. Moreover if   $X$ has a local minimum at $t\in(0,1)$ (in particular, $\Delta_{t}=0$ by $(A_{1})$) and if we set $s:=\sup\{r\in[0,t]:~ X_{r}< X_{t}\}$, then $\Delta_{s}>0$ and $X_{s-}<X_{t}<X_{s}$.  
\end{itemize}
\begin{figure}[!h]
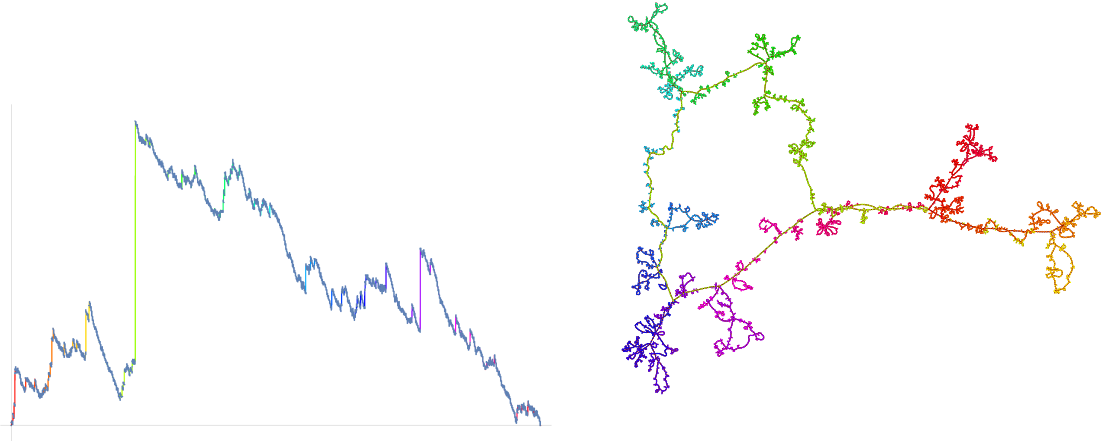

 \begin{center}
 \includegraphics[height=4.5cm]{images/exc8000}
    \includegraphics[height=6.5cm]{images/looptree8000}
    \caption{\label{fig:looptree}A simulation of 
      a $\tfrac{3}{2}$-stable L\'evy excursion and the corresponding looptree. Points belonging to the same loop (corresponding to the jumps of the excursion) are displayed with the same color.}
 \end{center}
 \end{figure}
To lighten notation, for $t \in (0,1)$, we set
\begin{equation}\label{eq:def:mathcal:A}
 \mathcal{A}_t := \{ s \in [0,1] : 0 \preceq s \preceq t \mbox{ and } \Delta_{s} >0\},
 \end{equation}
which corresponds  to the set of ancestors of $t$ which are jump times. See Figure \ref{fig:descente} for an illustration.  We stress that while the set of ancestors of $t$ is uncountable, the set $ \mathcal{A}_t$ is countable. Moreover, the following properties also hold $\mathbf{P}$-a.s.
\begin{itemize}
\item \hypertarget{prop:A:3}{$(A_{3})$}  For every $t \in (0,1)$, we have 
 \begin{eqnarray}X_{t}=\sum \limits_{\begin{subarray}{c} r \in \mathcal{A}_t \end{subarray}}x_{r,t}. \label{eq:sumdessauts}\end{eqnarray}
 \item \hypertarget{prop:A:4}{$(A_{4})$}  For every $t \in (0,1)$, no point of  $ \mathcal{A}_t$ is isolated from the left, and  we have $x_{r,t} \in (0, \Delta_r)$, for every $r \in \mathcal{A}_t \backslash \sup \mathcal{A}_t$.
\end{itemize}
Properties {\hypersetup{linkcolor=black}\hyperlink{prop:A:1}{$(A_1)$}} and {\hypersetup{linkcolor=black}\hyperlink{prop:A:2}{$(A_2)$}}  are a rewriting of properties $(H_{1})-(H_4)$ in \cite[Proposition 2.10]{Kor11}.  Property {\hypersetup{linkcolor=black}\hyperlink{prop:A:3}{$(A_3)$}} is proved in  \cite[Corollary 3.4]{CKlooptrees} and {\hypersetup{linkcolor=black}\hyperlink{prop:A:4}{$(A_4)$}} follows from \cite[Proposition 3.1]{CKlooptrees}.

\begin{prop}[Equivalence classes induced by $d$] \label{topologie_loop_tree}
The following properties hold under $\mathbf{P}$.  For every $0\leq s<t\leq 1$ we have $d(s,t)=0$ if and only if:
\begin{equation}\label{eq:equivalent:classes}
X_{t}=X_{s-}  \mbox{ and \ \ } X_{r}>X_{s-} \mbox{ \ for every } r\in (s,t).
\end{equation}
Moreover,  the equivalence classes of $\sim_d$ have at most two points.
\end{prop}

\begin{figure}[!h]
 \begin{center}
 \includegraphics[width=12cm]{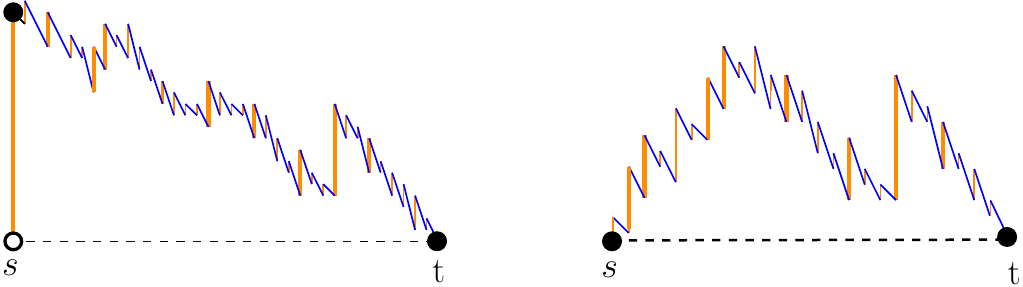}
 \caption{Illustration of times $s<t$ such that $s \sim_{d} t$. There are two possible cases: either $X$ has a jump at time $s$, as depicted on the left, or $X$ is continuous at $s$, and  $s$ is a local minimal record on the right, as depicted on the right. There are countably many pairs of identified times in the first case, and uncountably many in the second case. In both cases, $t$ is the first return time of $X$ at level $X_{s-}$ after time $s$. \label{fig:identificationlooptree}}
 \end{center}
 \end{figure}
\begin{proof}

 First of all, let us explain why the graph of $\sim_{d}$ contains all the couples of times described in the statement. Indeed, if  $0\leq s<t\leq 1$ verifies \eqref{eq:equivalent:classes}, then we have $s=s\curlywedge t$ and $\{r\in[0,1]:s\prec r\prec t\}=\varnothing$ and it follows from  \eqref{def:distancelooptree}   that:
$$ d(s,t)=\Delta_{s}\cdot \delta\big(\frac{x_{s,s}}{\Delta_{s}},\frac{x_{s,t}}{\Delta_{s}}\big). $$
Next, notice that  if $\Delta_s=0$ we directly have $d(s,t)=0$, and if $\Delta_s>0$, by \eqref{eq:equivalent:classes} we must have $x_{s,s}=\Delta_s$ and $x_{s,t}=0$, which implies $d(s,t)=\Delta_{s}\cdot \delta\big(1,0\big)=0$. Let us now focus on the other inclusion. Fix $ 0 \leq s < t \leq 1$ such that $d(s,t)=0$ and recall from \eqref{eq:def:mathcal:A} the definition of the set $ \mathcal{A}_s$ (resp.\ $ \mathcal{A}_t$) of ancestors of $s$ (resp.  $t$) which are also jump times of $X$. Notice that by   {\hypersetup{linkcolor=black}\hyperlink{prop:A:4}{$(A_4)$}},  in order to have $d_{0}(s\curlywedge t,s)=d_{0}(s\curlywedge t,t)=0$, it must hold that $\mathcal{A}_{s\curlywedge t}=\mathcal{A}_{s}=\mathcal{A}_{t}$. We  let $r = \inf\{ u>s : X_{u}=X_{s-}\}$ and we are going to show that we must have $s=s\curlywedge t$ and $t=r$, which implies that the pair $(s,t)$ verifies \eqref{eq:equivalent:classes}. In this direction, notice that by  ({\hypersetup{linkcolor=black}\hyperlink{prop:A:1}{$A_1$}},  {\hypersetup{linkcolor=black}\hyperlink{prop:A:2}{$A_2$}}), there is no $p>r$ such that  $s\curlywedge t\preceq p  $, which implies that, necessarily, 
$s\curlywedge t \leq s<t\leq r$. Property  {\hypersetup{linkcolor=black}\hyperlink{prop:A:3}{$(A_3)$}} combined with the previous discussion then entails that:
$$X_{s}= X_{(s\curlywedge t)-}+x_{s\curlywedge t,s}\quad \text{ and } X_{t}= X_{(s\curlywedge t)-}+x_{s\curlywedge t,t}.$$  Next, we argue according to whether $\Delta_{s\curlywedge t}=0$ or  $\Delta_{s\curlywedge t}>0$. If $\Delta_{s\curlywedge t}=0$, by the previous display, we must have $X_s=X_{t}=  X_{(s\curlywedge t)-}$ and, by the definition of $s\curlywedge t$ and $r$, this is possible if and only if $s=s\curlywedge t$ and $t=r$. If  $\Delta_{s\curlywedge t}>0$, by looking at the contribution of the jump time $s \curlywedge t$ in \eqref{def:distancelooptree},  we get:
\[\delta\big(\frac{x_{s\curlywedge t,s}}{\Delta_{s\curlywedge t}},\frac{x_{s\curlywedge t,t}}{\Delta_{s\curlywedge t}}\big)=0,\]
and one deduces that necessarily $x_{s\curlywedge t,s}=x_{s\curlywedge t,t}$ or $(x_{s\curlywedge t,s},x_{s\curlywedge t,t})\in \{0,\Delta_{s\curlywedge t}\}^{2}$.
But if $x_{s\curlywedge t,s}=x_{s\curlywedge t,t}\notin\{0,\Delta_{s\curlywedge t}\}$, then the real number  $p=\inf\{r\geq s\curlywedge t:\: x_{s\curlywedge t,r}=x_{s\curlywedge t,s}\}$
satisfies $p\preceq s$, $p\preceq t$ and $s\curlywedge t<p$. This is in contradiction with the definition of $s\curlywedge t$, and we deduce that  the only possibility is
$(x_{s\curlywedge t,s},x_{s\curlywedge t,t})\in \{0,\Delta_{s\curlywedge t}\}^{2}$. Notice then that by properties ({\hypersetup{linkcolor=black}\hyperlink{prop:A:1}{$A_1$}},  {\hypersetup{linkcolor=black}\hyperlink{prop:A:2}{$A_2$}}),  the latter only holds if and only if $s=s\curlywedge t $ and $t=r$. Finally, to conclude  it remains to show that, under  $\mathbf{P}$, the equivalence classes for $\sim_{d}$ contain at most two points. However, this follows directly since the previous argument also shows (by properties ({\hypersetup{linkcolor=black}\hyperlink{prop:A:1}{$A_1$}},  {\hypersetup{linkcolor=black}\hyperlink{prop:A:2}{$A_2$}})) that the point $s$ can only be identified with the point $t$.
\end{proof}

\label{Sec:equiv:d}

It follows from the above proposition that if $s \prec t$ is a strict ancestor of $t$, then $\Pi_{d}(s) \ne \Pi_{d}(t)$, thereby justifying \textit{a posteriori} the terminology. 
Furthermore, from  ({\hypersetup{linkcolor=black}\hyperlink{prop:A:1}{$A_1$}},  {\hypersetup{linkcolor=black}\hyperlink{prop:A:2}{$A_2$}}), it follows that
 if $ \Pi_{d}(s) \ne \Pi_{d}(t)$,  then $s \curlywedge t$ is necessarily a jump time of $X$.  Let us now describe the various types of points that can be encountered in the looptree $ \mathcal{L}$: 
\begin{itemize}
\item \textbf{Root}. The root of $ \mathcal{L}$ is $\Pi_d(0) = \Pi_d(1)$.
\item \textbf{Loops}. For every $t\in[0,1]$ with $\Delta_{t}>0$, we write
 \begin{eqnarray} \label{def:loops} \mathrm{f}_{t}(s):=\inf \big\{r\geq t:\:X_{r}=X_{t}-s\cdot \Delta_{t}\big\}\:\:,\:\:\text{for every}\:\: s\in[0,1].  \end{eqnarray}
 It is easy to check that the set $\Pi_{d}\circ\mathrm{f}_{t}([0,1]) \subset \mathcal{L}$, endowed with the restriction of the metric $d$, forms a loop of length $\Delta_t=X_t-X_{t-}$ equipped with the length metric. For this reason, it is called the \textbf{loop} associated with $t$ in $ \mathcal{L}$. The loops of $ \mathcal{L}$ will later become the \textbf{faces} of our limiting metric space $( \mathcal{S}, D^{*})$. By construction, there are countably many loops, and their union $\cup\{\Pi_{d}\circ\mathrm{f}_{t}([0,1]):\Delta_t>0\}$  is denoted by $ \mathrm{Loops}$. By construction,  we have $\mathcal{L}=\mathrm{Cl}( \mathrm{Loops})$, since the set of jumping times of $X$ is dense in $[0,1]$.

\item \textbf{Pinch points.} A  pinch point is a point different from the root in $ \mathcal{L}$ which has several pre-images by $\Pi_d$, hence exactly $2$ by Proposition \ref{topologie_loop_tree}. Informally, pinch points correspond  to the set of ``touching points between loops''. The set of all pinch points in $ \mathcal{L}$ is called the \textbf{skeleton} of the looptree and denoted by $ \mathrm{Skel}$.
\item \textbf{Leaves.} A point of the looptree which is not a pinch point will be called a \textbf{leaf}.  The set of all leaves is denoted by $ \mathrm{Leaves}$.
\end{itemize}
We stress that there are leaves and pinch points on loops, and that the root is a leaf, see Figure~\ref{fig:examplespoints} for an illustration.  By extension, we shall speak of leaf, pinch point or loop times for times in $[0,1]$ which project respectively on $\mathrm{Leaves}$, $\mathrm{Skel}$ or $\mathrm{Loops}$.

Let us now justify our choice of terminology. In this direction, we define the \textbf{degree} of a point $x \in \mathcal{L}$
 as the number of connected components of $ \mathcal{L} \backslash \{x\}$.

\begin{figure}[!h]
 \begin{center}
 \includegraphics[width=10cm]{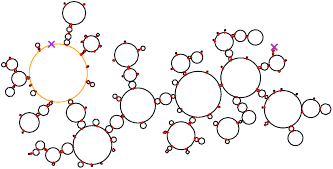}
 \caption{An example of a loop (in orange), of pinch points (in red) and of two leaves (purple crosses) on a looptree. \label{fig:examplespoints}}
 \end{center}
 \end{figure}

\begin{prop}[Degree of points] \label{prop:timeclassification} Under $\mathbf{P}$, all the pinch points of $ \mathcal{L}$ have degree $2$ and the leaves have degree $1$.
\end{prop}
\begin{proof}
 If $x \in \mathcal{L}$ is a leaf different from the root, then there exists a unique  $s\in (0,1)$ such that  $x = \Pi_d(s)$. Therefore, we have $\mathcal{L}\setminus \{x\}=\Pi_d([0,1]\setminus\{s\})$. Since $\Pi_{d}$ is continuous and identifies $0$ and $1$,  this implies that  $\mathcal{L}\setminus \{x\}$ is connected. Conversely, if $x$ is a pinch point then by Proposition \ref{topologie_loop_tree} we have $x=\Pi_d(s)=\Pi_d(t)$ with $0 < s < t < 1$ satisfying $I_{s,t} = X_{s-}=X_t$, and we introduce the sets:
  \[\mathcal{C}_{1}:=\Pi_{d}\big((s,t)\big)\mbox{ and } \mathcal{C}_{2}:=\Pi_{d}\big((t,s)\big).\]
Next, we notice that $(s,t)$ and $(t,s)$ are connected, after identifying the points $0$ and $1$. Therefore,   the continuity of $\Pi_{d}$ implies that $\mathcal{C}_{1}$ and $\mathcal{C}_{2}$ are two connected subsets of $\mathcal{L}$. Moreover we have $\mathcal{L}=\mathcal{C}_{1}\cup \mathcal{C}_{2}\cup\{\Pi_{d}(s)\}$ and $\Pi_{d}(s)\notin \mathcal{C}_{1}\cup \mathcal{C}_{2}$ since equivalence classes of $\sim_{d}$ have at most two elements.  To see that $\mathcal{C}_{1}$ and $\mathcal{C}_{2}$ are in fact the two connected components of $\mathcal{L}\setminus \{\Pi_{d}(s)\}$, just note that   for every $r_{1}\in (s,t)$ and $r_{2}\in (t,s)$ we have $I_{r_{1},r_{2}}<X_{r_{1}}$ if $r_{1}<r_{2}$ and  $I_{r_{2},r_{1}}<X_{r_{1}}$ if $r_{2}<r_{1}$ ; otherwise the equivalence class of $s$  would have  more than $2$ elements. This implies by Proposition~\ref{topologie_loop_tree} that $\Pi_{d}^{-1}\big(\mathcal{C}_{1}\big)=(s,t)$ and $\Pi_{d}^{-1}\big(\mathcal{C}_{2}\big)=(t,s)$, and consequently that $\mathcal{C}_{1}$ and $\mathcal{C}_{2}$ are  the two connected components of $\mathcal{L}\setminus \{\Pi_{d}(s)\}$. The case of the root $\Pi_{d}(0)=\Pi_{d}(1)$ is special. The preimage of the root is formed of the two times $0$ and $1$. However, in this case $\mathcal{C}_2$ is empty and we have $\mathcal{C}_1=\mathcal{L}\setminus\{\Pi_{d}(0)\}$. Consequently the root has  degree $1$. 
  \end{proof}
  
  For every $s,t \in [0,1]$, we set 
  \begin{equation} \label{eq:defpinch} \mathrm{Branch}(s,t):=\big\{r\in[0,1]:~s\curlywedge t\prec r\prec t\big\}\cup\big\{r\in[0,1]:~s\curlywedge t\prec r\prec s\big\} \cup\big\{s,t\big\}.  \end{equation}
  In the previous proof, we characterized the connected components of the complement of a point. It is easy to infer from this characterization that the image of $ \mathrm{Branch}(s, t)$ in $ \mathcal{L}$ corresponds to all the points separating $ \Pi_{d}(s)$ from $ \Pi_{d}(t)$ in $ \mathcal{L}$. Notice that, except possibly for the times $s$ and $t$, all times in $ \mathrm{Branch}(s, t)$ are pinch point times.  Also, in the case $s=0$, we simply have 
 \begin{equation}\label{eq:defpinch:0:t} 
  \mathrm{Branch}(0,t)=\big\{r\in[0,1]:~0\preceq r\prec t\big\}\cup \big\{t\big\}.
 \end{equation}
 We stress that the above set  is not necessarily equal to $\{r\in[0,1]:~0\preceq r\preceq t\}$. Actually these two sets are equal if $t$ is a leaf, but if $t$ is a pinch point time, then $\{r\in[0,1]:~0\preceq r\preceq t\}$ may contain the two representatives of $\Pi_{d}(t)$ (if $t$ is the largest representative). However their images by $\Pi_{d}$ coincide. See Figure \ref{fig:descente} for an illustration.
  \begin{figure}[!h]
 \begin{center}
 \includegraphics[width=9cm]{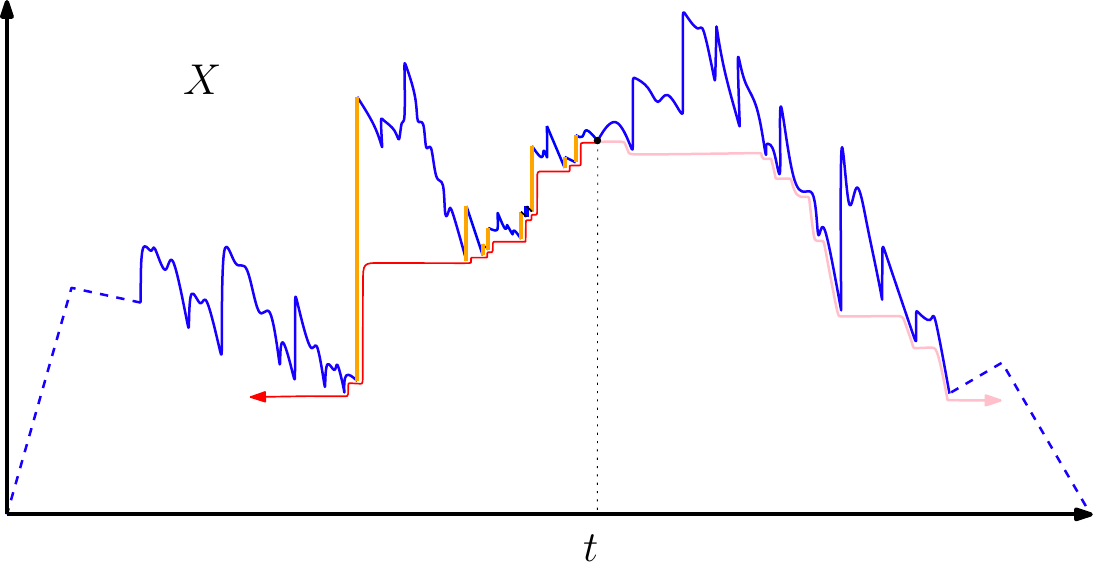}
 \caption{ For a given $t \in (0,1)$, the times $s \preceq t$ are the minimal records found by starting from $X_{t}$ and following the running infimum  of $X$ backward in time  (represented in red in the figure). These constitute the set $\text{Branch}(0,t)$. Except at $0$  and possibly at time $t$, all such times are pinch point times according to the classification above. In particular, each $s \preceq t$ is identified in the looptree to the point $ \inf \{ u \geq t : X_{u} =  X_{s-}\}$. Notice also that, although there are uncountably many $s \preceq t$, there are only countably many $ s \preceq t$ which are also jump times (represented in orange on the figure). These correspond  to the loops encountered  in the looptree $ \mathcal{L}$ when going  from $\Pi_{d}(t)$ back to the root $\Pi_{d}(0)$. \label{fig:descente}}
 \end{center}
 \end{figure} 
 \begin{rek}[Hausdorff dimensions and a heuristic]
 As established in \cite{CKlooptrees}, the (Hausdorff) dimension of $(\mathcal{L},d)$ is~$\alpha$. Additionally, the set of all loops, $\mathrm{Loops} \subset \mathcal{L}$, is a countable union of plain loops, and therefore has dimension 1. While we do not prove it here,  the reader may keep in mind that the skeleton $ \mathrm{Skel} \subset \mathcal{L}$ is of dimension $\alpha-1$. Notably, since $\alpha - 1 < 1$, there are ``fewer" points on the skeleton compared to the loops, and this asymmetry becomes more pronounced as $\alpha \downarrow 1$.
 \end{rek}

\paragraph{Resistance metric.} We shall also equip the looptree $\mathcal{L}$ with another metric, \textbf{the resistance metric}, which will enable us to consider a Gaussian process indexed by $ \mathcal{L}$. This metric has already been introduced and studied by Archer \cite{archer2019brownian} in order to construct the Brownian motion \textit{moving} over $ \mathcal{L}$ (as opposed to our forthcoming ``Brownian motion \textit{indexed} by $ \mathcal{L}$''), see also \cite{blancrenaudie2022looptree} for a systematic study of these processes on non-necessarily stable looptrees.

The construction is \textit{mutatis mutandis} the same as for the metric $d$  replacing the pseudo-distance $\delta$, turning $[0,1]$ into a loop of length $1$, by the function defined for $s,t \in [0,1]$ by
  \begin{eqnarray} \label{eq:defdtilde} \tilde{\delta}(s,t):= \left( \frac{1}{|s-t|} + \frac{1}{1-|s-t|}\right)^{-1}= |s-t|\cdot \big(1-|s-t|\big).  \end{eqnarray} The pseudo-distance $\tilde{\delta}$ is the resistance metric of $[0,1]$, considered  as a loop of length 
$1$  after identifying the points $0$ and $1$,   with unit resistance per unit length.  Replacing $\delta$ by $\tilde{\delta}$ in the definition of $d_0$ and $d$, given in \eqref{def:distancelooptree}, we obtain a new pseudo-distance $ \tilde{d}$ on $[0,1]$ (defined together with $d$) which is quasi-isometric to the original distance $d$ (see \cite[Section 4.1]{archer2019brownian}). Namely, we have  
 \begin{eqnarray}  \label{eq:quasiisod} \frac{1}{2}d(s,t) \leq \tilde{d}(s,t) \leq d(s,t), \quad \text{ for every } s,t \in [0,1],  \end{eqnarray}
since $\frac{1}{2}\delta(s,t)\leq\tilde{\delta}(s,t)\leq \delta(s,t)$. In particular, we have the identification $\mathcal{L}=[0,1]/\sim \tilde{d}$, and $\tilde{d}$ induces a distance function on $\mathcal{L}$, which, as usual, we still denote by $\tilde{d}$. 
\par Let us conclude this section with an upper bound on the minimum number of $\tilde{d}$-closed balls of radius $r$  that covers $[0,1]$. This result will play a pivotal role in the study of the variations of the Gaussian process indexed by $ \mathcal{L}$. In this direction, fix $r>0$, and  consider the finite sequence $(s_{k})_{1\leq k\leq N_{r}}$ of elements of $[0,1]$ defined as follows. First take $s_{1}=0$. For $k\geq 2$,  if $s_{k-1}=1$  then take $N_{r}:=k-1$ and stop the construction of the sequence at step $k-1$. However,  if $s_{k-1}<1$ then set:
\[s_{k}:=\inf\big\{t\in (s_{k-1},1]:~\tilde{d}(s_{k-1},t)\geq r\big\},\]
with the convention $\inf\varnothing=1$. By construction, the collection $\{B_{\tilde{d}}(s_{k},r):~1\leq k\leq N_{r}\}$
is a covering of $[0,1]$ by $\tilde{d}$-closed balls of radius $r$, and as consequence $N_r$ is an upper bound for the $r$-covering number of $ (\mathcal{L}, \tilde{d})$. Building on Archer~\cite{archer2019brownian}, we establish the following uniform control:
\begin{lem}\label{estimate:N:tilde}
There exist two constants $c,C>0$ such that:
$$\mathbf{P}\Big(\sup_{r\in (0,1]} \frac{N_{r}}{\big(r^{-1} \log(1+r^{-1})\big)^{\alpha}}\geq x\Big)\leq C\cdot  \exp\big(-c x^{\frac{1}{\alpha}}\big),\quad \text{ for } x\geq 0.$$
\end{lem}
\begin{proof}
Fix $r\in (0,1]$. By construction, we have:
$$\big\{ N_r\geq  \lceil \lambda r^{-\alpha} \rceil \big\} \subset \Big\{ \exists s, t \in [0,1] : |s-t| \leq \lambda^{-1} r^{\alpha}  \mbox{ and } \tilde{d}(s,t) \geq   r \Big\} .$$
The proof of Proposition 5.4 in \cite{archer2019brownian}  states  that  there exists $\lambda_0$ such that, for all $\lambda\geq \lambda_{0}$  and  $r\leq 1$, we have 
\begin{equation} \label{eq:archer}
 \mathbf{P}\big(N_r \geq \lceil \lambda r^{-\alpha} \rceil\big) \underset{ \tilde{d} \leq d}{\leq} 
\mathbf{P}\left( \exists s, t \in [0,1] : |s-t| \leq \lambda^{-1} r^{\alpha}  \mbox{ and } d(s,t) \geq   r \right) \underset{\cite{archer2019brownian}}{\leq} \mathrm{c}_1  \cdot \lambda r^{-\alpha}  \mathrm{e}^{- \mathrm{c}_1^\prime \lambda^{\frac{1}{\alpha}}},
\end{equation}
 for some $c_1 ,c'_1 >0$.\footnote{In order for the reader to recover the exact result presented here, we mention the following typo in \cite[ Propositions 5.3 and 5.4]{archer2019brownian}: the condition $\lambda \in(0, \frac{1}{2} r^{-\alpha})$  should be replaced by $\lambda \geq   2 r^{\alpha}$.}  Taking $\lambda = x \log(1+r^{-1})^\alpha$ and noticing that for $x$ large enough we have $x^{\frac{1}{\alpha}} \log(1+r^{-1}) \geq  \frac{1}{2} x^{\frac{1}{\alpha}} + \frac{2\alpha}{c_1^\prime} \log(1+r^{-1})$,  we deduce that 
  \begin{eqnarray*}\mathbf{P}\big(N_r \geq \lceil \lambda r^{-\alpha} \rceil\big) &\lesssim&   x\cdot ( r^{-1}\log(1+r^{-1}) )^{\alpha} \cdot \mathrm{e}^{- \mathrm{c}_1^\prime  x^{\frac{1}{\alpha}} \log(1+r^{-1})} \\
  & \lesssim &  x\cdot ( r^{-1}\log(1+r^{-1}) )^{\alpha} \cdot   \mathrm{e}^{- 2 \alpha \log (1 + r^{-1})}\mathrm{e}^{- \frac{\mathrm{c}_1^\prime}{2}  x^{\frac{1}{\alpha}}}\\
    &\lesssim &  \big(r \log(1+r^{-1}) \big)^\alpha \cdot \big( x\,  \mathrm{e}^{- \frac{\mathrm{c}_1^\prime}{2}  x^{\frac{1}{\alpha}}}\big),\end{eqnarray*}
  for every $r\in [0,1]$. The desired result now follows by   performing a union bound over $r = 2^{-k}$, for $k \geq 0$,  and using the fact that $r \mapsto N_r$ is non-decreasing.\end{proof}

\subsection{$ \mathbb{R}$-trees  and the height process } \label{sec:codagearbre}
Although this is not strictly necessary for our immediate purposes, let us compare the $\alpha$-looptree to the $\alpha$-stable tree, which is a random $\R$-tree also encoded by the process $X$. We first recall the classical construction of $\R$-trees from continuous excursions functions \cite{DLG02}. In this direction, fix  $F : [0,1] \to \mathbb{R}$  a continuous function satisfying $F(0)=F(1)=0$. Then,  we define a pseudo-distance, denoted with a lowercase mathfrak font, by the formula 
 \begin{eqnarray*}  \mathfrak{f}(s,t) := F(s)+F(t) -2 \min \left(\min_{[s \wedge t, s \vee t]} F ; \min_{[0,s \wedge t] \cup[s \vee t, 1]} F \right).  \end{eqnarray*} In particular, if $F : [0,1] \to \mathbb{R}_+$ is a non-negative excursion, then the right-hand side of the last display simplifies to $F(s)+F(t) - 2 \min_{u \in [s \wedge t, s \vee t]}F(u)$. The quotient of $[0,1]$ by the equivalence relation $\sim_{\mathfrak{f}}$, equipped with $ \mathfrak{f}$ is an $ \mathbb{R}$-tree\footnote{An $ \mathbb{R}$-tree   is a uniquely arcwise connected metric space, in which each arc is isometric to a compact interval of $\mathbb{R}$.} that we denote by $ \mathcal{T}_ \mathfrak{f}$, see \cite{DLG05} for more details. We shall usually root this tree at $\Pi_ \mathfrak{f}( \mathrm{argmin } \,F)$, and we stress that there is no ambiguity in the definition since two points realizing the minimum of $F$ are identified by $\mathfrak{f}$. We also  notice that for every $s\in [0,1]$ we have:
 $$\mathfrak{f}\big(\Pi_ \mathfrak{f}(s), \Pi_ \mathfrak{f}( \mathrm{argmin }\,F)\big)=F(s)-\min F.$$ 
One can also equip $\mathcal{T}_{\mathfrak{f}}$ with the pushforward of the Lebesgue measure on  $[0,1]$ under the associated canonical projection  which we denote by $\mathrm{Vol}_{\mathfrak{f}}$.
In this work,  we shall specialize the above construction to two random functions:
 \begin{itemize}
 \item the height process $H$ -- see below --  so that $ \mathcal{T}_{ \mathfrak{h}}$, is, under $ \mathbf{P}$, the $\alpha$-stable tree introduced by Duquesne--Le Gall--Le Jan, see \cite{DLG05},
 \item the excursion of the label process $Z$ giving the tree $ \mathcal{T}_ \mathfrak{z}$; see Section \ref{sec:local:minima}.
 \end{itemize}

Obviously, the construction of looptrees from excursions with positive jumps parallels the definition of $ \mathbb{R}$-trees coded by continuous excursions, see \cite{khanfir2022convergences} for a recent generalization which englobes both constructions. We extend the terminology introduced in the previous section to the context of $\R$-trees. More precisely, for every  $x\in \mathcal{T}_{\mathfrak{f}}$, in accordance with the notion in the looptree setting, the \textbf{degree}  of $x$ is  the number of connected components in the complement of  $\{x\}$. We say that $x$ is a \textbf{branching point} if its degree is larger than $2$ and a \textbf{leaf} if it has degree $1$. The skeleton of $\mathcal{T}_{\mathfrak{f}}$ is defined as the set of points with degree larger than $1$, that is the complement of the set of leaves. By construction, with possibly  the exception of times realizing the global minimum, the times realizing two-sided local minima of $F$ correspond, after projection, to branching points of the tree $ \mathcal{T}_ \mathfrak{f}$, whereas  the times realizing one-sided local minima correspond to points of the skeleton of $ \mathcal{T}_{ \mathfrak{f}}$.  We stress that these notions are compatible with the ones that we used for the looptree $\mathcal{L}$. However, note that in the case of $\R$-trees there are no loops, while in the case of the looptree, there are no branching points. 
\medskip

We now recall the construction of the \textbf{height process} $H$ associated with the L\'evy excursion $X$, as built in Le Gall \& Le Jan \cite{LGLJ98}, see also \cite{Du03} and \cite{LG99} for details.  By \cite[Section 1.1 and 1.2]{DLG02}, for every $t\in [0,1]$,  the quantity
\begin{equation}\label{convergence:proba:pour:H}
\varepsilon^{-1} \cdot \int_{0}^{t} \mathrm{d} s ~\mathbbm{1}_{X_s<I_{s,t}+\varepsilon} 
\end{equation}
converges in probability as $\varepsilon \downarrow 0$, under $\mathbf{P}$,  to a random variable that we denote by $H_t$. The process $ t \mapsto H_{t}$ has a continuous modification that we consider from now on, and we keep the notation $H$ for this modification, which is called the height process associated with $X$. Roughly speaking, for every fixed $t\geq 0$, the random variable $H_t$ measures the size of the set: 
\begin{equation*}
    \big\{ 0 \leq s \leq t :~ X_{s-} \leq  I_{s,t} \big\}  = \{ 0 \leq  s  \leq t : ~s \preceq t\}.
\end{equation*}

 In order to make the connection between the height process and the looptree more transparent, let us mention that  by  Equation $(5)$ in \cite{Du03}, the height process satisfies
\begin{equation}\label{def:H}
H_t=\lim_{ \varepsilon \to 0} (\alpha-1)\Gamma(-\alpha)\eps^{\alpha-1}\cdot \#\big\{s\in[0,1]:~s\preceq t\: \text{ and }\:\Delta_{s}>\eps\big\},
\end{equation}
where the convergence holds $\mathbf{P}$-a.s. for a set of values $t\in [0,1]$ of full Lebesgue measure. Since the height process is continuous, we can consider as in \eqref{eq:defpseudodistancearbre} the associated pseudo-distance $\mathfrak{h}$, so that the quotient of $[0,1]$ by the equivalence relation $\sim_{\mathfrak{h}}$, equipped with $ \mathfrak{h}$, is a random $ \mathbb{R}$-tree $ \mathcal{T}_{ \mathfrak{h}}$. The random $\R$-tree $(\mathcal{T}_{\mathfrak{h}}, \mathfrak{h})$ is   the \textbf{$\alpha$-stable tree} of Duquesne--Le Gall--Le Jan, see \cite{DLG05}.

\begin{figure}[!h]
 \begin{center}
 \includegraphics[width=13cm]{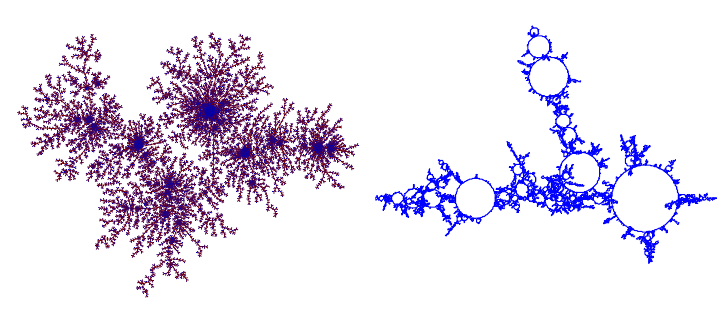}
 \caption{Simulation of an $\alpha$-stable tree and the associated looptree.}
 \end{center}
 \end{figure}
 The looptree $ \mathcal{L}$ and the stable tree $ \mathcal{T}_{\mathfrak{h}}$ have ``the same branching structure'' except that the loops in $ \mathcal{L}$ correspond to the branching points of infinite degree in $ \mathcal{T}_{\mathfrak{h}}$. More precisely, by the definition of $H$ and Proposition \ref{topologie_loop_tree} combined with \cite[Proposition 4.1]{Kor11}, we have  $s \sim_{d}t \Rightarrow s\sim_{\mathfrak{h}}t$, and the only additional identifications in $ \sim_{ \mathfrak{h}}$ are made of the identifications of times corresponding to a common loop in $ \mathcal{L}$. In particular,  the stable tree $ \mathcal{T}_{\mathfrak{h}}$ is a quotient of the looptree $ \mathcal{L}$, and  the skeleton of $ \mathcal{L}$  projects onto the skeleton of $ \mathcal{T}_{ \mathfrak{h}}$. 
 \par
 
As a consequence, the stable tree and its coding height process $H$ are especially useful for ``parametrizing''  the set of points that separate two points in the looptree. This will be play an important role in Section \ref{sec:label:pinch}. Let us make this idea precise. Recall from \eqref{eq:defpinch}  the definition of $ \mathrm{Branch}(0,t)$. For every $t\in [0,1]$ and $r\in [0,H_{t})$, set
\begin{equation}\label{xi:sup}
\xi_{t}(r):=\inf \big\{s\leq t:~H_{u}> r\text{ for every } u\in(s,t]\big\},
\end{equation}
and by convention set $\xi_{t}(r):=t$ for  $r\geq H_{t}$. Notice that the process $(\xi_{t},0\leq t\leq 1)$ is rcll. Moreover, by \eqref{eq:defpseudodistancearbre},  it holds that the image of $[0,H_t]$ by  $\Pi_{\mathfrak{h}} \circ \xi_{t}$ is 
the range of the unique geodesic connecting $\Pi_{\mathfrak{h}}(0)$ and  $\Pi_{\mathfrak{h}}(t)$ in the stable tree $\mathcal{T}_{\mathfrak{h}}$. Specifically,  the point $ \Pi_{\mathfrak{h}}(\xi_{t}(r))$ corresponds to the unique point in the geodesic  at distance $r$ from $\Pi_{ \mathfrak{h}}(0)$  in the stable tree $ \mathcal{T}_{\mathfrak{h}}$. The following lemma states that we can use $\xi_{t}$ to describe  the set $ \mathrm{Branch}(0,t)$, see also Figure \ref{fig:XetH} for an illustration.

\begin{lem}\label{Lem-cut-jump}
$\mathbf{P}$-a.s.,  for every $t\in [0,1]$, the following holds:
\\
\\
$\rm(i)$ $\big\{r\in \mathrm{Branch}(0,t)\setminus\{t\}:~\Delta_r=0\big\}=\big\{\xi_{t}(r):~r\in [0,H_t)\big\}$;
\\
\\
$\rm(ii)$ $\big\{r\in \mathrm{Branch}(0,t)\setminus\{t\}:~\Delta_r>0\big\}=\big\{\xi_{t}(r-):~r\in [0,H_t)\:\:\text{and}\:\:\xi_{t}(r-)\neq \xi_{t}(r)\big\}$.
\end{lem}

\begin{proof} The lemma will follow by combining  properties  {\hypersetup{linkcolor=black}\hyperlink{prop:A:1}{$(A_1)$}}--{\hypersetup{linkcolor=black}\hyperlink{prop:A:4}{$(A_4)$}}, stated in Section \ref{sec:defloop}, with the following extra property, which holds $\mathbf{P}$-a.s.\ as a direct consequence  of \cite[Proposition 4.1 and Remark 4.3]{Kor11}
\\
\\
$(C)$~: For every $\ell \preceq \ell^{\prime}$, we have $H_\ell\leq H_{\ell^{\prime}}$ and the condition $H_\ell=H_{\ell^{\prime}}$ is equivalent to $$\{r\in[0,1]:~\ell \prec r\prec \ell^{\prime}\}=\varnothing.$$
In the rest of the proof we work under the $\mathbf{P}$-a.s.\ event under which $(A_1)$--$(A_4)$ and $(C)$ hold simultaneously.  Moreover, since in the case $t\in\{0,1\}$ there is nothing to prove, we fix 
 $t\in(0,1)$. We start by showing that 
\begin{equation}\label{eq:Pinch:Subset}
\big\{\xi_{t}(r):~r\in [0,H_t)\big\}\subset \big\{r\in \text{Branch}(0,t)\setminus\{t\}:~\Delta_r=0\big\}.
\end{equation}
In this direction, observe that if $\ell \in[0,1]$ is a jumping time for $X$, then, for every $r>\ell$ such that $I_{\ell,r}=X_{r}$, we have $\ell\prec r$ and $\{s\in[0,1]:~\ell \prec s\prec r\}=\varnothing$. Therefore, $(C)$ ensures that $H_{\ell}=H_{r}$ for all such $\ell$ and $r$. From the definition of $\xi_t$ given in \eqref{xi:sup}, we then infer that we must have $\Delta_{\xi_{t}(r)}=0$~, for every $r\in[0,H_t)$. Let us now establish that $\xi_{t}(r)\in \text{Branch}(0,t)$, for every $r\in[0,H_t)$. To this end, fix $r\in[0,H_t)$,
and note that it suffices to show that $I_{\xi_{t}(r),t}\geq  X_{\xi_{t}(r)}$. We argue by contradiction. Namely, if $I_{\xi_{t}(r),t}< X_{\xi_{t}(r)}$ then we let $\ell^\prime>\xi_{t}(r)$ be the smallest element of $(\xi_{t}(r),t]$ such that $I_{\xi_{t}(r),t}= X_{\ell^\prime}$.
In particular, $\ell^\prime$ must be a local minimum, and then, by  {\hypersetup{linkcolor=black}\hyperlink{prop:A:2}{$(A_2)$}}, the common ancestor $\ell=\xi_{t}(r) \curlywedge \ell^\prime$ satisfies $\{s\in[0,1]:~\ell \prec r\prec \ell^{\prime}\}=\varnothing$. An application of $(C)$ then ensures that $H_{\ell}=H_{\ell^\prime}\leq H_{\xi_{t}(r)}=r$. This is in contradiction with the definition of $\xi_{t}(r)$, which completes the proof of \eqref{eq:Pinch:Subset}.

Let us now prove the reverse inclusion. To this end, take $\ell\in \text{Branch}(0,t)$, with $\Delta_\ell=0$ and $\ell\neq t$, and we want to show that $\ell\in \big\{\xi_{t}(r):~r\in [0,H_t)\big\}$. Write $\ell^{\prime}$ for the unique solution of $d(\ell, \ell^\prime)=0$ different from $\ell$. Now remark that since $\ell\prec t$, the definition of $\prec$, combined with Proposition \ref{prop:timeclassification} and property  {\hypersetup{linkcolor=black}\hyperlink{prop:A:2}{$(A_2)$}}, implies that we must have $\ell<t<\ell^{\prime}$ and 
$$X_{\ell}=X_{\ell^{\prime}}<X_{r},\quad \text{ for every } r\in(\ell, \ell^\prime).$$
Then, it follows from   {\hypersetup{linkcolor=black}\hyperlink{prop:A:3}{$(A_3)$}} that $\{r^\prime\in[0,1]:~\ell \prec r^\prime\prec r\}\neq \varnothing$
for every $r\in(\ell,t]$. Consequently, $(C)$ implies that for every $r\in(\ell,t]$ we have $H_\ell<H_{r}$ and we obtain $\ell\in \xi_{t}\big([0,H_t)\big)$, as wanted. This completes the proof of $\rm(i)$.

To obtain $\rm(ii)$  observe that $\sim_{d}$ is a closed equivalence relation, which entails that the set $ \text{Branch}(0,t)$ is closed. Since $r\mapsto \xi_{t}(r)$ is increasing and right-continuous, point $\rm(i)$ then implies the inclusion: 
$$ \big\{r\in \text{Branch}(0,t)\setminus\{t\}:~\Delta_r>0\big\}\supset\big\{\xi_{t}(r-):~r\in [0,H_t)\:\:\text{and}\:\:\xi_{t}(r-)\neq \xi_{t}(r)\big\}~.$$
Inversely, if $r\in \text{Branch}(0,t)$ with $\Delta_r>0$, a standard compactness argument, combined with   {\hypersetup{linkcolor=black}\hyperlink{prop:A:4}{$(A_4)$}}, gives that we can find a sequence $(r_{k})_{k\geq 1}$ increasing to $r$ such that $r_1\prec r_2\prec ....\prec r$ and $\Delta_{r_k}=0$ for every $k\geq 1$. So we can again apply point $\rm{(i)}$ to derive that there exists $\ell_k$ such that $\xi_t(\ell_k)=r_k$, and then we obtain point $\rm{(ii)}$ taking the limit when $k\to \infty$. 

 \end{proof}

\begin{figure}[!h]
 \begin{center}
 \includegraphics[width=13cm]{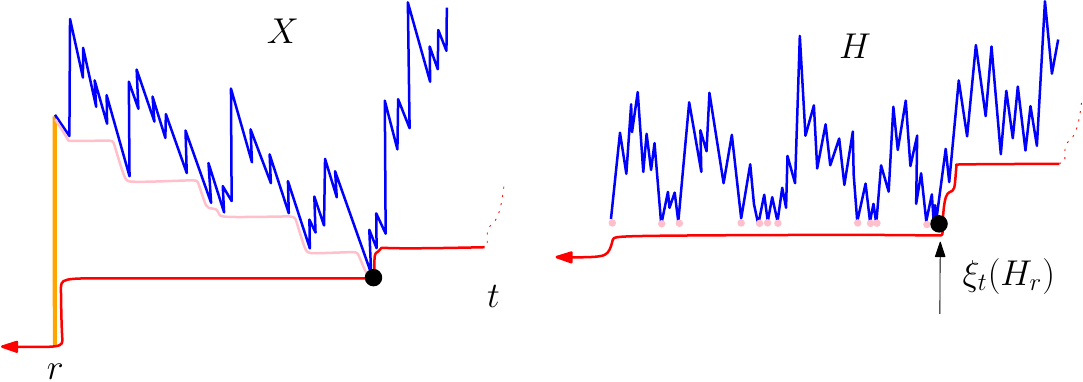}
 \caption{Illustration of a portion of the processes $X$ and $H$, 
   both represented in the vicinity of a time $t$. A jump time for $X$ (in orange) corresponds to  a ``plateau'' of $H$ on which it bounces. In particular, if $r \prec t$ is an ancestor of $t$ at which $\Delta_{r}>0$, we do not have $\xi_{t}(H_{r}) = r$. \label{fig:XetH}}
 \end{center}
 \end{figure}

 \section{Constructions of the label process $Z$ }\label{sec:constr_Z}

The first goal of this section is to show that  the label process $Z$, originally introduced in \cite{LGM09}, can be equivalently defined as  the ``Brownian motion indexed by $ \mathcal{L}$''. 
This process is also referred to as the Gaussian free field on $\mathcal{L}$ pinned at the root, see Section \ref{subsection_Z_Gaussian}, and Figure \ref{fig:sim:Z:*} for a simulation. In Section \ref{sec:equivalence:Z}, using the technology of Gaussian processes, we will be able to obtain 
sharp control on the variations of $Z$. These bounds will be useful later to study scaling limits of planar maps.

Let us now recall the construction of $Z$ given in \cite{LGM09}. To formally introduce this construction, we consider an auxiliary  probability space $(\Omega, \mathcal{G}, P)$ that supports a countable collection $(\mathrm{b}_{i})_{i\in \mathbb{N}}$ of independent Brownian bridges starting and ending at $0$ with lifetime $1$. In what follows, we argue on 
the product space $\mathbb{D}(\mathbb{R}_+,\mathbb{R})\times \Omega$, equipped with the product sigma-field. Moreover, for simplicity and with a slight abuse of notation, we will continue to write  $\mathbf{P}$ for the probability measure $\mathbf{P}\otimes P$ on this extended space. Then  \cite[Proposition 5]{LGM09}, establishes that, under $\mathbf{P}$, for every $t\in [0,1]$, the series 
\begin{equation}\label{Z_represent_Mir}
Z_t  :=    \sum \limits_{i\in \mathbb{N}} \Delta_{\mathrm{t}_{i}}^{\frac{1}{2}}\cdot \mathrm{b}_{i}\big(\frac{x_{\mathrm{t}_{i},t}}{\Delta_{\mathrm{t}_{i}}}\big),
\end{equation}
converges in $L^2$. 
Furthermore, by  \cite[Proposition 6]{LGM09} the process  $ t \mapsto Z_{t}$ has a continuous modification, which is even a.s.~$\frac{1}{2\alpha} -\varepsilon$ H\"older continuous for every $\varepsilon>0$. We consider only this modification, and for simplicity, we keep denoting it by $Z$. Finally, Section \ref{sec:Markov} is devoted to extending the construction of $(X,Z)$ under (enriched versions of) $\mathbf{N}$ and $\mathbf{Q}$, and establishing the associated  Markov properties.

\subsection{The Gaussian Free Field on $ \mathcal{L}$}\label{subsection_Z_Gaussian}
We now construct  the \textbf{Gaussian Free Field} on $( \mathcal{L}, \tilde{d})$ pinned at $\Pi_{d}(0)$.  In this direction, we introduce the function
\begin{equation}\label{def:Cov}
\Gamma(s,t):=\frac{1}{2}\tilde{d}(0,s)+\frac{1}{2}\tilde{d}(0,t)-\frac{1}{2}\tilde{d}(t,s),\qquad s,t\in[0,1]. 
\end{equation}
Our first goal is to show that $\Gamma$ can be used as a covariance function.

\begin{lem}\label{nonnegative_definite}
$\mathbf{P}$-a.s., the function $\Gamma:[0,1]^{2}\to\mathbb{R}$ is symmetric and nonnegative definite i.e:
\begin{equation}\label{equa_nonnegative_definite}
\sum \limits_{i=1}^{n}\sum \limits_{j=1}^{n} \lambda_{i}\lambda_{j}\Gamma(t_{i},t_{j})\geq 0,
\end{equation}
for every integer $n\geq 1$, every $(t_{1},...,t_{n})\in[0,1]^{n}$ and $(\lambda_{1},...,\lambda_{n})\in \mathbb{R}^{n}$.
\end{lem}
\begin{proof}
The function $\Gamma$ is clearly symmetric $\mathbf{P}$-a.s. Let us show that it is also nonnegative definite. This is a straightforward verification. One can establish directly from the definition that $\mathbf{P}$-a.s.:
\begin{equation}\label{cut:Gamma}
\sum \limits_{i=1}^{n}\sum \limits_{j=1}^{n} \lambda_{i}\lambda_{j}\Gamma(t_{i},t_{j})\geq \big(\sum \limits_{i=1}^{n}\lambda_{i}\big)^{2}\tilde{d}(0,t_{1}\curlywedge...\curlywedge t_{n}),
\end{equation}
simultaneously  for every integer $n\geq 1$, $(t_{1},...,t_{n})\in[0,1]^{n}$ and $(\lambda_{1},...,\lambda_{n})\in \mathbb{R}^{n}$. This can be proved by induction on $n$, by cutting at time $t_{1}\curlywedge...\curlywedge t_{n}$. 
   Since this verification is  a little bit tedious, we are going to deduce \eqref{cut:Gamma} from the following result due to  Archer \cite[Lemma 4.5]{archer2019brownian}.  Under $\mathbf{P}$, for every $n\geq 1$ and $(t_{1},...,t_{n})\in[0,1]^{n}$, there exists a graph $G$ with conductances such that:
\begin{itemize}
\item The set vertices of $G$ is  ${V}(G):= \{t_{i}:\:1\leq i\leq n\}\cup \{t_{i}\curlywedge t_{j}:\:1\leq i\leq j\leq  n\}$;
\item For every $s, t\in {V}(G)$:
\[\widetilde{d}(s,t)=R_{G}(s\leftrightarrow t),\]
where $R_{G}(s\leftrightarrow t)$ stands for the effective resistance in $G$ between $s$ and $t$.
\end{itemize}
In particular, notice that the point $r:=t_{1}\curlywedge...\curlywedge t_{n}$ is a vertex of $G$. Moreover, by a classical result on networks  \cite[Exercise 2.59 and the comments about it]{LP16}   the function:
\[{V}({G})\ni(s,t) \mapsto \frac{1}{2} R_{G}(r\leftrightarrow s)+\frac{1}{2} R_{G}(r\leftrightarrow t)-\frac{1}{2} R_{G}(s\leftrightarrow t),  \]
is nonnegative definite. Since for every $1\leq i\leq n$ we have $\tilde{d}(0,t_{i})=\tilde{d}(0,r)+\tilde{d}(r,t_{i})$,
we infer that for every $\lambda_{1},...,\lambda_{n}\in \mathbb{R}^{n}$:
\begin{align*}
\sum \limits_{i=1}^{n}\sum \limits_{j=1}^{n} \lambda_{i}\lambda_{j}\Gamma(t_{i},t_{j})&=\sum \limits_{i=1}^{n}\sum \limits_{j=1}^{n} \lambda_{i}\lambda_{j}\tilde{d}(0,r)
+\frac{1}{2}\sum \limits_{i=1}^{n}\sum \limits_{j=1}^{n} \lambda_{i}\lambda_{j}\big(\tilde{d}(r,t_{i})+\tilde{d}(r,t_{j})-\tilde{d}(t_{i},t_{j})\big)\\
&=\big(\sum \limits_{i=1}^{n}\lambda_{i}\big)^{2}\tilde{d}(0,r)+\frac{1}{2}\sum \limits_{i=1}^{n}\sum \limits_{j=1}^{n} \lambda_{i}\lambda_{j}\big(R_{G}(r\leftrightarrow t_{i})+R_{G}(r\leftrightarrow t_{j})-R_{G}(t_{i}\leftrightarrow t_{j})\big)\\
&\geq \big(\sum \limits_{i=1}^{n}\lambda_{i}\big)^{2}\tilde{d}(0,r),
\end{align*}
and we obtain \eqref{cut:Gamma}. Consequently, $\mathbf{P}$-a.s., the function $\Gamma:[0,1]^{2}\to\mathbb{R}$ is symmetric and nonnegative definite. 
\end{proof}

As a consequence of Lemma \ref{nonnegative_definite} we can consider, conditionally on $X$, a centered Gaussian process $(Z^*_t : t \in [0,1])$ with covariance function $\Gamma$.
In particular, since $\tilde{d}(0,0)=\tilde{d}(0,1)=0$,  we have $Z^*_0=Z^{*}_{1}=0$. To avoid confusion, we  denote the conditional distribution (resp.~expectation) of the above process  by $\mathbf{P}_{X}$ (resp.  $\mathbf{E}_{X}$). In particular, we have $ \mathbf{E}_X[Z^*_sZ^*_t] = \Gamma(s,t)$.  Although the looptree $ \mathcal{L}$ is itself a random fractal object,  the machinery of Gaussian processes is powerful enough to prove that there is a modification of $Z^*$ continuous with respect to $\tilde{d}$ with a strong control on the modulus of continuity:

\begin{prop}\label{variations_Z} There exists a modification of $Z^*$ which is continuous for the pseudo-distance $\tilde{d}$ (hence for the Euclidean norm on $[0,1]$)  for which the following hold:
\begin{itemize}
\item[\rm{(i)}]  There exist constants $c,C,\beta>0$,  such that, for every $x\geq 0$,  
\[ \mathbf{P}(\sup Z^* >x)\leq C\cdot \exp(-c x^{\beta}).\]
\item[\rm{(ii)}] There exist constants $ \mathrm{c}, \mathrm{C}>0$, $x_{0}>1$ such that,  for every  $x\geq x_0$ and every integer   $n\geq1$, 
\[\mathbf{P}\left(\exists s,t\:\text{with}\:\tilde{d}(s,t)<2^{-n}\:\text{such that}\:|Z_{s}^*-Z_{t}^*|\geq x  \sqrt{n} 2^{-n/2}\right)\leq   \mathrm{C}\cdot  \exp(- \mathrm{c} x^{2} n).\]
\end{itemize}
\end{prop}

\begin{proof}

The proof relies on classical results on Gaussian processes such as Dudley's theorem. In this direction, recall our upper bound $N_r$ on the $r$-covering number of $ ( \mathcal{L},\tilde{d})$ defined just before Lemma~\ref{estimate:N:tilde}.  We introduce the random variable 
 $$M:= \sup\limits_{r\in (0,1]}\frac{N_{r}}{\big(r^{-1} \log(1+r^{-1})\big)^{\alpha}},$$ which by Lemma \ref{estimate:N:tilde} is almost surely finite, 
 and we write   $\mathrm{diam}= \sup\{\tilde{d}(s,t):~s,t\in [0,1]\}$ for the diameter of $\mathcal{L}$ with respect to $\tilde{d}$.  Next, we observe that:
\begin{align*}
\int_{0}^{\infty}\sqrt{\log\big(N_{r^{2}}\big)}~\d r&\leq  \int_{0}^{1}\sqrt{\log\big(N_{r^{2}}\big)}~\d r+\mathrm{diam}\cdot  \sqrt{\log\big(N_{1}\big)}.
\end{align*}
Since $\mathrm{diam}\leq 2 N_1$, a direct computation using Jensen's inequality  combined with the fact that $M\geq N_1/\log(2)^\alpha>1$  yields the existence of a constant $C_{1}\in(0,\infty)$ such that:
\begin{equation}
\int_{0}^{\infty}\sqrt{\log\big(N_{r^{2}}\big)} ~\d r\leq C_1\cdot M\sqrt{\log(M)}.
\end{equation}
Since $M$ is almost surely finite we can apply   Dudley's theorem (see \cite[Theorem 6.1.2]{marcus_rosen_2006} for a reference, where  the distance $d_{X}$ used therein corresponds to  $ \sqrt{\tilde{d}}$ in our setting), which implies that the process $Z^{*}$ admits a continuous modification  for $ \tilde{d}$  satisfying 
 \begin{align}\label{eq:E:sup:Z}
  \mathbf{E}_{X}[\sup Z^{*}]&\leq  C_2\cdot M\sqrt{\log(M)},
\end{align}
for some  constant $ C_2 \in(0,\infty)$. In the remainder of the proof, we consider this continuous modification.  Dudley's theorem  (see the same reference) also ensures the existence of another constant $C_3\in(0,\infty)$ such that
$$ \mathbf{E}_{X}\Big[\sup \limits_{\tilde{d}(s,t)<2^{-n}}|Z^{*}_{s}-Z^{*}_{t}|\Big]\leq
    C_3\cdot \int_{0}^{2^{-\frac{n}{2}}}\sqrt{\log\big(N_{r^{2}}\big)}~\mathrm{d} r,$$
    for every $n\geq 0$. Therefore, another application of Jensen inequality, combined with $M\geq 1$ and the identify  $\int_{0}^{t}\log(u) \mathrm{d}u=t\log(t)-t$, ensures that there exists
    a constant $C_4\in(0,\infty)$ such that:
\begin{equation}\label{eq:espvario}
 \mathbf{E}_{X}\Big[\sup \limits_{\tilde{d}(s,t)<2^{-n}}|Z^{*}_{s}-Z^{*}_{t}|\Big]\leq   C_4 \cdot 2^{-\frac{n}{2}} \left(\sqrt{ \log (M)} + \sqrt{n} \right), 
\end{equation} 
for every $n\geq 0$. We stress that the constants $C_1,C_2,C_3$ and $C_4$ are universal and do not depend on $X$. To simplify the following expressions, we set $C:=3\vee C_1\vee C_2\vee C_3\vee C_4$.  To strengthen these   expectation estimates into tail estimates as stated in points (i) and (ii), we shall use a version of the Borell--TIS inequality,  which states that the supremum of a Gaussian process exhibits Gaussian tails when centered around its mean.
\par  Let us start proving (i). In this direction, an application of  \cite[Theorem 5.4.3]{marcus_rosen_2006} entails that,  for every $x>0$, we have
$$\mathbf{P}_{X}\Big(\sup Z^{*}-\mathbf{E}_{X}[\sup Z^{*}]>x\Big)\leq 2 \exp(-\frac{x^{2}}{ 2\text{v}_{*}}), \quad \mbox{
where } \text{v}_{*}:=\sup\limits_{s\in[0,1]} \mathbf{E}_{X}[ (Z^{*}_s)^{2}].$$ Let $x>C$, and, to simplify notation, consider the unique $a>1$ such that  $x = C  \cdot a \log (a)$. Then, by \eqref{eq:E:sup:Z}, we have \begin{align*}
    \mathbf{P}(\sup Z^{*}>2x)&=\mathbf{E}\big[\mathbf{P}_{X}(\sup Z^{*}>2x)\big]\\
    &= \mathbf{E}\Big[\mathbf{P}_{X}\big(\sup Z^{*}-\mathbf{E}_{X}[\sup Z^{*}]>2x-\mathbf{E}_{X}[\sup Z^{*}]\big)\Big]\\
    &\leq \mathbf{P}\big(M\geq a\big)+\mathbf{E}\Big[\mathbbm{1}_{M<a}\mathbf{P}_{X}(\sup Z^*-\mathbf{E}_{X}[\sup Z^*]>{x})\Big]\\
    &\leq \mathbf{P}\big(M\geq a\big)+2\mathbf{E}[\mathbbm{1}_{M<a}\exp(- \frac{x^{2}}{2\text{v}_{*}})].
\end{align*}
Point (i) now follows by using Lemma \ref{estimate:N:tilde} and noticing  that, under the event $\{M<a\}$, we have:
\begin{align*}
    \text{v}_{*}&=\sup\limits_{s\in [0,1]} \mathbf{E}_{X}[ (Z^{*}_s)^2]\leq \mathrm{diam}\leq 2 N_1\leq 2M \leq 2x.
\end{align*}
We proceed similarly to establish (ii). Recall the notation $0<s_{1} < ... < s_{N_{2^{-n}}} \leq 1$ introduced prior to Lemma \ref{estimate:N:tilde}, so that $[0,1]= \bigcup \limits_{i=1}^{N_{2^{-n}}}B_{\tilde{d}}(s_{i},2^{-n})$. In particular, for every $s,t\in [0,1]$ with $\tilde{d}(s,t)<2^{-n}$, there exists $1\leq i\leq N_{2^{-n}}$ such that  $s,t\in B_{\tilde{d}}(s_{i},2^{-n+1})$.  It follows that:
\begin{equation*}
\sup\limits_{\tilde{d}(s,t)<2^{-n}}|Z_{s}^*-Z_{t}^*|\leq 2\cdot \sup \limits_{1\leq i\leq N_{2^{-n}}} \sup\limits_{s\in B_{\tilde{d}}(s_{i},2^{-n+1})}|Z_{s}^*-Z_{s_{i}}^*|.
\end{equation*}
To control the right-hand side we note that under $\mathbf{P}_{X}$, for every  $1\leq i\leq N_{2^{-n}}$, the process $$\Big(Z_{s}^*-Z_{s_i}^*:s\in B_{\tilde{d}}(s_{i},2^{-n+1})\Big)$$ is a centered Gaussian process  with
\[\sup \limits_{s\in B_{\tilde{d}}(s_{i},2^{-n+1})}\limits\mathbf{E}_{X}\big[(Z_{s}^*-Z_{s_{i}}^*)^{2}\big]\leq 2^{-n+1}.\]
We can then apply  \cite[Theorem 5.4.3]{marcus_rosen_2006} again combined with the bound \eqref{eq:espvario},  to obtain as above that, under the event $\{\log (M) \leq x^2 n\}$,  for every $x\geq 1$ and $1\leq i\leq N_{2^{-n}}$:
 $$ \mathbf{P}_X\Big(\sup\limits_{s\in B_{\tilde{d}}(s_{i},2^{-n+1})}|Z_{s}^*-Z_{s_{i}}^*|>  5 C  x \cdot \sqrt{n} 2^{-n/2}\Big)\leq 2\exp(- C^2 x^2 n).  $$
 Next, since $\alpha\in(1,2)$,  a  direct computation shows that under $\{\log (M) \leq x^2  n\}$ we also have $N_{2^{-n}}\leq   \exp((4+x^2) n)$. Therefore, recalling that $C\geq 3$, a union bound gives
 $$ \mathbf{P}\Big( \sup \limits_{1\leq i\leq N_{2^{-n}}}  \sup\limits_{s\in B_{d}(s_{i},2^{-n+1})}|Z_{s}^*-Z_{s_{i}}^*|> 5C  x \cdot \sqrt{n} 2^{-n/2}\Big)\leq \mathbf{P}\big(\log (M) \geq x^2 n \big)+ 2 \exp\big(-4 x^2 n\big).$$
 The desired result now follows by Lemma  \ref{estimate:N:tilde}. 
 \end{proof}
 
From now on we consider the $\tilde{d}$-continuous modification of $Z^*$, that we still denote by $Z^*$. As a direct consequence  of Proposition \ref{variations_Z},  we have $\mathbf{P}\text{-a.s.}$, for all $s,t \in [0,1]$
\begin{equation}\label{label} 
d(s,t)=0\implies Z_{s}^{*}=Z_{t}^{*}.
\end{equation}
This allows to interpret the process $Z^{*}$ as a label process on $\mathcal{L}$, and, with a slight abuse of notation, we continue to denote it by $Z^{*}$. Namely,  for every $u\in \mathcal{L}$, set $Z_{u}^{*}:=Z_{t}^{*}$ where $t$ is any preimage of $u$ by $\Pi_{d}$. Furthermore, point $\mathrm{(ii)}$ and the Borel-Cantelli lemma show that $u \in \mathcal{L} \mapsto Z^{*}_{u}$ is $(\frac{1}{2}-\varepsilon)$- H\"older continuous for every $ \varepsilon>0$. By \eqref{eq:archer}, the mapping $(s,t) \mapsto d(s,t)$ is itself $ \frac{1}{\alpha}- \varepsilon$ H\"older continuous for every $ \varepsilon>0$. We deduce that for   every $ \varepsilon>0$, there exists a random variable $W \equiv W_{ \varepsilon}\in(0,\infty)$ such that    \begin{eqnarray}\label{Z:variation:L}
  \vert Z_{u}^*-Z_{v}^*\vert \leq W \cdot d(u,v)^{ \frac{1}{2}- \varepsilon}, \:\: u,v \in \mathcal{L}, \quad \mbox{ and } \quad    \vert Z_{s}^*-Z_{t}^*\vert \leq W\cdot |s-t|^{ \frac{1}{2 \alpha}- \varepsilon}, \:\: s,t \in [0,1].    \end{eqnarray}
In particular, $t \mapsto Z^{*}_{t}$ is $\frac{1}{2\alpha}-\varepsilon$ Hölder continuous, for every $\varepsilon>0$.
\begin{figure}[!h]
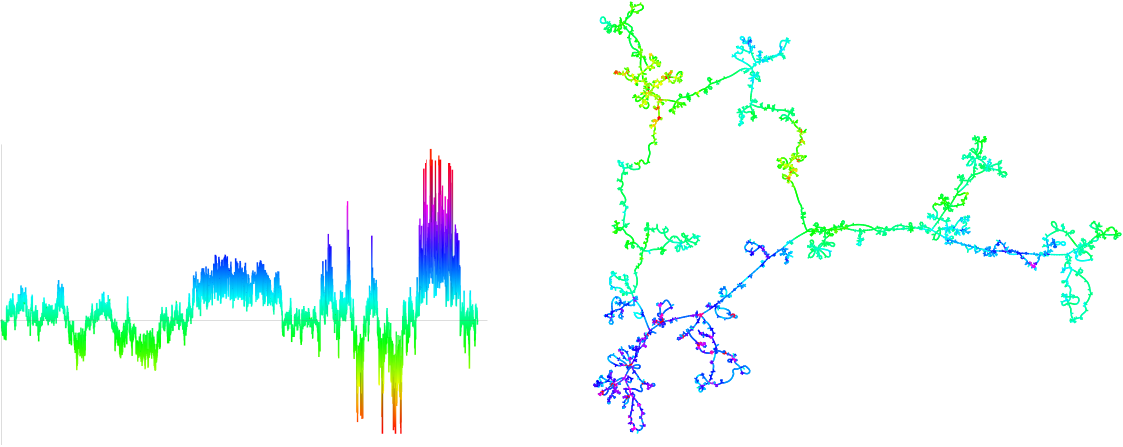

 \begin{center}
   \includegraphics[height=4cm]{images/Z8000bis} \hspace{1cm}
 \includegraphics[height=6cm]{images/Zlooptree8000} 
 \caption{Simulation of the label process $Z^*$ over $[0,1]$ (on the left) and seen on the looptree of Figure \ref{fig:looptree} on the right.}\label{fig:sim:Z:*}
 \end{center}
 \end{figure}

\subsection{Equivalence of the constructions}\label{sec:equivalence:Z}
The goal of this section is to  establish that $(X,Z)$ and $(X,Z^{*})$ have the same distribution under $\mathbf{P}$.  To this end, it will be useful to consider the trace of $ Z^{*}$ on the loops of $ \mathcal{L}$.  Recall that  $(\mathrm{t}_{i})_{i\in \mathbb{N}}$ is a measurable indexation of the jumping times of $X$, and recall the notation 
$ \mathrm{f}_{\mathrm{t}_{i}}(s)=\inf \{r\geq \mathrm{t}_{i}:~X_{r}=X_{\mathrm{t}_{i}}-s\Delta_{\mathrm{t}_{i}}\}$ for $s\in[0,1]$, which stands for the parametrization of the loop associated with $\mathrm{t}_{i}$. For every index $i\in \mathbb{N}$, we consider the function:
\begin{equation}\label{def:b:i:*}
\mathrm{b}_{i}^{*}(s):=\Delta_{\mathrm{t}_{i}}^{-\frac{1}{2}}\cdot\big(Z_{\mathrm{f}_{\mathrm{t}_{i}}(s)}^{*}-Z_{\mathrm{t}_{i}} ^{*}\big),\quad s\in [0,1],
\end{equation}
which is a continuous process. We have the following result:

\begin{prop}[Equivalence of the constructions]\label{b_Brownian_Brigde} With the above notation, under $ \mathbf{P}$ and conditionally on $X$,  the family $(\mathrm{b}_{i}^{*})_{i\in \mathbb{N}}$ is formed of i.i.d.~Brownian bridges with lifetime $1$ starting and ending at $0$. Moreover, we have $(X,Z^{*})= (X,Z)$ in distribution under $\mathbf{P}$.
\end{prop}

\begin{proof}
We begin by showing that $(\mathrm{b}_{i}^{*})_{i\in \mathbb{N}}$ is a sequence of i.i.d.~Brownian bridges with lifetime $1$, starting and ending at $0$, independent of $X$. In this direction, we work under $\mathbf{P}_{X}$. Note that by construction, conditionally on $X$, the process $(\mathrm{b}_{i}^{*}(s))_{(i,s)\in \mathbb{N}\times[0,1]}$ is a Gaussian process. Furthermore, for every $i\in \mathbb{N}$ and  $s,t\in[0,1]$, we have:
\begin{align*}
\mathbf{E}_{X}\big[\mathrm{b}_{i}^*(t)\mathrm{b}_{i}^*(s)]&=\Delta_{\mathrm{t}_{i}}^{-1}\cdot \Big(\mathbf{E}_{X}[(Z_{ \mathrm{f}_{\mathrm{t}_{i}}(s)}^{*}-Z_{\mathrm{t}_{i}}^{*})(Z_{ \mathrm{f}_{\mathrm{t}_{i}}(t)}^{*}-Z_{\mathrm{t}_{i}}^{*})]\Big)\\
&=\Delta_{\mathrm{t}_{i}}^{-1}\cdot \Big(\Gamma( \mathrm{f}_{\mathrm{t}_{i}}(s), \mathrm{f}_{\mathrm{t}_{i}}(t))+\Gamma(\mathrm{t}_{i},\mathrm{t}_{i})-\Gamma( \mathrm{f}_{\mathrm{t}_{i}}(s),\mathrm{t}_{i})-\Gamma( \mathrm{f}_{\mathrm{t}_{i}}(t),\mathrm{t}_{i})\Big).
\end{align*}
A direct computation  using \eqref{def:Cov} gives that the previous display is equal to $\frac{1}{2}\tilde{\delta}(0,t)+\frac{1}{2}\tilde{\delta}(0,s)-\frac{1}{2}\tilde{\delta}(t,s)$,
which is the covariance function of a Brownian bridge with lifetime $1$ starting and ending at $0$. In particular, it does not depend on the realization of $X$. Moreover, for $i\neq j$ in $\mathbb{N}$ and $s,t$ in $[0,1]$, a similar computation shows that
\begin{align*}
\mathbf{E}_{X}\big[\mathrm{b}^{*}_{i}(t)\mathrm{b}^{*}_{j}(s)]=\Delta_{\mathrm{t}_{i}}^{-\frac{1}{2}}\Delta_{t_{j}}^{-\frac{1}{2}}\mathbf{E}_{X}\big[(Z_{ \mathrm{f}_{\mathrm{t}_{i}}(s)}^{*}-Z_{\mathrm{t}_{i}}^{*})(Z_{ \mathrm{f}_{\mathrm{t}_{i}}(t)}^{*}-Z_{t_{j}}^{*})\big]=0.
\end{align*}
Therefore, $(\mathrm{b}_{i}^{*})_{i\in \mathbb{N}}$ has the same finite-dimensional marginals as a family of i.i.d.~Brownian bridges with lifetime $1$, starting and ending at $0$, independent of $X$, and since the processes $b^*_i$, $i\in \mathbb{N}$, have continuous sample paths, we deduce that they indeed form a family of i.i.d.\ Brownian bridges. It is then clear that we can couple the construction of $Z$ and $Z^{*}$ using the same Brownian bridges $\mathrm{b}_{i} = \mathrm{b}_{i}^{*}$.  Since both processes are continuous, to conclude it suffices to establish that, with this coupling and for every fixed $t\in[0,1]$, we have $Z_t=Z_t^*$,  $\mathbf{P}$-a.s. To this end,  note that this reduces to showing that:
$$\mathbf{E}_X\Big[\Big(Z_t^*- \sum \limits_{i\in \mathbb{N}} \Delta_{\mathrm{t}_{i}}^{\frac{1}{2}}\cdot \mathrm{b}_{i}^*\big(\frac{x_{\mathrm{t}_{i},t}}{\Delta_{\mathrm{t}_{i}}}\big)\Big)^2\Big]=0,\quad t\geq 0.$$
For the remaining of the proof we fix $t\in [0,1]$.  To simplify notation, for $i\in\mathbb{N}$, set $u_{i}=\frac{x_{\mathrm{t}_{i},t}}{\Delta_{\mathrm{t}_{i}}}$, and observe that the expectation in the previous display can be decomposed in the form:
\begin{equation}\label{eq:Z:Z:*:deco}
\mathbf{E}_{X}\big[(Z_{t}^{*})^{2}\big] + \mathbf{E}_{X}\Big[\Big(\sum \limits_{\overset{i\in \mathbb{N}}{\mathrm{t}_{i}\preceq t}} \Delta_{\mathrm{t}_{i}}^{\frac{1}{2}}\cdot \mathrm{b}_{i}^{*}\big(u_i\big)\Big)^{2}\Big]-2\sum \limits_{\mathrm{t}_{i}\preceq t} \Delta_{\mathrm{t}_{i}}^{\frac{1}{2}}\cdot\mathbf{E}_{X}\big[ Z_{t}^{*}\: \mathrm{b}_{i}^{*}\big(u_{i}\big)\big].
\end{equation}
 We are going to conclude by computing each term separately. First, it follows from  \eqref{def:Cov} that  $\mathbf{E}_{X}\big[(Z_{t}^{*})^{2}\big]= \tilde{d}(0,t)$. Moreover, using the independence of the $(b_i^*)_{i\in \mathbb{N}}$, we get:
$$\mathbf{E}_{X}\Big[\Big(\sum \limits_{\overset{i\in \mathbb{N}}{\mathrm{t}_{i}\preceq t}} \Delta_{\mathrm{t}_{i}}^{\frac{1}{2}}\cdot \mathrm{b}_{i}^{*}\big(u_i\big)\Big)^{2}\Big]= \sum \limits_{\overset{i\in \mathbb{N}}{\mathrm{t}_{i}\preceq t}} \Delta_{t_i} \cdot \widetilde{\delta}(0,u_i)= \tilde{d}(0,t).$$
 Therefore to conclude, we need to show that the remaining term in \eqref{eq:Z:Z:*:deco} equals $- 2 \tilde{d}(0,t)$. To this end, note that it follows straightforwardly from definitions  \eqref{def:Cov} and \eqref{def:b:i:*}, that under $\mathbf{P}_X$, the random variables $Z_{\mathrm{t}_{i}}^{*}$ and $Z_{t}^{*}-Z_{\mathrm{f}_{\mathrm{t}_{i}}(u_{i})}^{*}$ are independent of $\mathrm{b}_{i}^{*}(u_{i})$. We infer that:
\begin{align*}
\mathbf{E}_{X}\big[ Z_{t}^{*}\: \mathrm{b}_{i}^{*}\big(u_{i}\big)\big]&=\mathbf{E}_{X}\big[ \big(Z_{t}^{*}-Z_{\mathrm{f}_{\mathrm{t}_{i}}(u_{i})}^{*}\big)\: \mathrm{b}_{i}^{*}(u_{i})\big]+\mathbf{E}_{X}\big[ \big(Z_{\mathrm{f}_{\mathrm{t}_{i}}(u_{i})}^{*}-Z_{\mathrm{t}_{i}}^{*}\big)\: \mathrm{b}_{i}^{*}(u_{i})\big]
+\mathbf{E}_{X}\big[ Z_{\mathrm{t}_{i}}^{*}\: \mathrm{b}_{i}^{*}\big(u_{i})\big]\\
&=\Delta_{\mathrm{t}_{i}}^{\frac{1}{2}}\cdot\mathbf{E}_{X}\big[\mathrm{b}_{i}^{*}(u_{i})^{2}\big]=\Delta_{\mathrm{t}_{i}}^{\frac{1}{2}}\cdot\widetilde{\delta}(0,u_i),
\end{align*}
As a consequence, we get:
\begin{align*}
\sum \limits_{\mathrm{t}_{i}\preceq t} \Delta_{\mathrm{t}_{i}}^{\frac{1}{2}}\cdot\mathbf{E}_{X}\big[ Z_{t}^{*}\: \mathrm{b}_{i}^{*}\big(u_{i}\big)\big]=\sum \limits_{\overset{i\in \mathbb{N}}{\mathrm{t}_{i}\preceq t}} x_{\mathrm{t}_{i},t}(\Delta_{\mathrm{t}_{i}}-x_{\mathrm{t}_{i},t})= \widetilde{d}(0,t),
\end{align*}
as wanted.
 \end{proof}

We conclude this subsection with the re-rooting invariance property.   For every $s,t\in [0,1]$, set $s\oplus t=s+t$ if $s+t\leq 1$ and $s\oplus t=(s+t)-1$ otherwise. We then claim that:
\begin{lem}[Invariance by re-rooting and time reversal] \label{lem:invariancereroottime}
For every $s\in [0,1]$, under $\mathbf{P}$ we have:
\begin{equation}\label{re-rooting}
    \big(d(s\oplus t,s\oplus t^{\prime}),\tilde{d}(s\oplus t,s\oplus t^{\prime}),Z_{s\oplus u}-Z_{s}\big)_{t,t^{\prime},u\in[0,1]}\overset{(d)}{=}\big(d( t, t^{\prime}),\tilde{d}( t, t^{\prime}),Z_{ u}\big)_{t,t^{\prime},u\in[0,1]}
\end{equation}
and
\begin{equation}\label{symmetric}
\big(d(1-t,1- t^{\prime}),\tilde{d}(1- t,1- t^{\prime}),Z_{1- u}\big)_{t,t^{\prime},u\in[0,1]}\overset{(d)}{=}\big(d( t, t^{\prime}),\tilde{d}( t, t^{\prime}),Z_{u}\big)_{t,t^{\prime},u\in[0,1]}.    
\end{equation}
\end{lem}
\begin{proof}
First, remark that  for every $s\in [0,1]$, conditionally on $X$, the process $(Z_{s\oplus u})_{u\in [0,1]}$ is a Gaussian process with covariance function characterized by $\mathbf{E}_{X}[(Z_{s\oplus u}-Z_{s\oplus u^{\prime}})^{2}]=\tilde{d}(s\oplus u,s\oplus u^{\prime})$, for every $u,u^{\prime}\in [0,1]$. Similarly, the process $(Z_{1- u})_{u\in [0,1]}$ is, conditionally on $X$,  a Gaussian process with covariance function characterized by $\mathbf{E}_{X}[(Z_{1- u}-Z_{1- u^{\prime}})^{2}]=\tilde{d}( 1-u,1- u^{\prime})$, for every $u,u^\prime \in[0,1]$. Consequently, it suffices to prove that:
\begin{equation*}
    \big(d(s\oplus t,s\oplus t^{\prime}),\tilde{d}(s\oplus t,s\oplus t^{\prime})\big)_{t,t^{\prime}\in[0,1]}\overset{(d)}{=}\big(d( t, t^{\prime}),\tilde{d}( t, t^{\prime})\big)_{t,t^{\prime}\in[0,1]}
\end{equation*}
and
\begin{equation*}
\big(d(1-t,1- t^{\prime}),\tilde{d}(1- t,1- t^{\prime})\big)_{t,t^{\prime}\in[0,1]}\overset{(d)}{=}\big(d( t, t^{\prime}),\tilde{d}( t, t^{\prime})\big)_{t,t^{\prime}\in[0,1]}.    
\end{equation*}
In principle, this should follows from invariance properties of the excursion process $X$. However, we take a different route 
by reasoning from the discrete setting. Indeed,  the looptree $\mathcal{L}$ can be obtained as limit (in Gromov-Hausdorff sense) of discrete dissections equipped with the graph distance (we refer to \cite[Theorem 1.3]{CKlooptrees} for more details).  These discrete dissections are also symmetric with respect to the root and are invariant by rotation (a more formal statement is given in Section $4.2$ and Remark 4.6 therein).  Taking the limit, this entails the desired invariance properties for $d$. It is also possible to define a discrete analog of $\tilde{d}$ in the setting of discrete dissections, then  a minor adaptation of the proof of \cite[Proposition 4.6]{archer2019brownian} gives the desired properties for $(d,\tilde{d})$.
\end{proof}

\subsection{Markov property} \label{sec:Markov}
This section is devoted to the Markov properties of the process  $(X,Z)$, inherited from the standard Markov properties of  stable Lévy processes.  Recall from the beginning of Part \ref{PartI} that the canonical process $X$ is distributed under $ \mathbf{Q}$ as the $\alpha$-stable L\'evy process without negative jumps and  under $ \mathbf{N}^{(v)}$ as the  excursion above its running infimum with lifetime $v>0$. Recall also that  $ \mathbf{N}$ stands for  the  excursion measure and we write $\sigma:=\sup \{t\geq 0:~X_t\neq 0\}$ for the excursion lifetime.  As in the case of the measure $\mathbf{P}$, we can enrich $\mathbf{N}^{(v)}$,  $\mathbf{N}$,  and $\mathbf{Q}$ so that they support an i.i.d. sequence of Brownian bridges. More precisely, recall that at the beginning of Section~\ref{sec:constr_Z} we considered an auxiliary probability space $(\Omega, \mathcal{G}, P)$ that supports a countable collection $(\mathrm{b}_{i})_{i\in \mathbb{N}}$ of independent Brownian bridges starting and ending at $0$ with lifetime $1$. Then, as we did for $\mathbf{P}$, we  work on the product space $\mathbb{D}(\mathbb{R}_+, \mathbb{R})\times \Omega$ and consider the measures $\mathbf{N}^{(v)}\otimes P$, $\mathbf{N}\otimes P$, and $\mathbf{Q}\otimes P$. For simplicity we will continue to use the notation $\mathbf{N}^{(v)}$, $\mathbf{N}$ and $\mathbf{Q}$ to refer to these measures. In particular, the disintegration relation \eqref{eq:excursionmeasuredecomp} still holds. 
\par
Next  note that, for $v>0$,  the notation and results of the previous sections extend directly replacing $\mathbf{P}$ by $\mathbf{N}^{(v)}$ and the interval $[0,1]$ by $[0,v]$. In particular, we can consider the associated pseudo-distance $d$, canonical projection $\Pi_{d}$, height process $H$ and label process $Z$. Moreover, by the scaling property of the underlying Lévy process, we have:
\begin{equation}\label{eq:scaling:N:v}
\Big(\big(X,Z,H\big):~\text{ under } \mathbf{N}^{(v)}\Big)\overset{(d)}{=} \Big(\big(v^{\frac{1}{\alpha}}X_{t/v},v^{\frac{1}{2\alpha}}Z_{t/v}, v^{\frac{\alpha-1}{\alpha}}H_{t/v}\big)_{t\geq 0}:~~\text{ under } \mathbf{P}\Big)
\end{equation}
Using the disintegration relation \eqref{eq:excursionmeasuredecomp}, it follows that we can extend also to $\mathbf{N}$ the notation used under $\mathbf{P}$ and in particular  the process $Z$  becomes a continuous process with lifetime $\sigma$. 

We can argue similarly under $ \mathbf{Q}$, although some clarifications are needed. Specifically, we  extend the definitions of $\preceq$, $\prec$ and $x_{s,t}$ for $s,t\geq 0$ directly by replacing $[0,1]$ by $\mathbb{R}_+$ while adding the convention that $0$ is an ancestor for all $s\geq 0$. 
Then under $\mathbf{Q}$, we continue to use the same definition $d_0$, given in \eqref{def:d:0:1}, but we set:
$$d(s,t):=d_{0}(s\curlywedge t,s)+d_{0}(s\curlywedge t,t)+\Delta_{s\curlywedge t} \cdot \delta\big(\frac{x_{s\curlywedge t,s}}{\Delta_{s\curlywedge t}},\frac{x_{s\curlywedge t,t}}{\Delta_{s\curlywedge t}}\big)+ |I_{t}-I_{s}|,$$
for $s,t\geq 0$. In words, we treat time $0$  as corresponding to an ``infinite loop" consisting of the points $\{s \geq 0 : X_{s} = I_s\}$.  These definitions are consistent with the previous ones since, under $\mathbf{N}$ and $\mathbf{P}$, the process $X$ is non negative and thus $I_t=0$ for every $t\geq 0$. The map $d:\mathbb{R}_+^{2}\mapsto \mathbb{R}_+$, under $\mathbf{Q}$, is a pseudo-distance and we can then consider the associated equivalence relation $\sim_d$ and canonical projection $\Pi_d$. Lastly, we must  extend the label process. Under $ \mathbf{Q}$ the limit of the series \eqref{Z_represent_Mir} also exists for the $L^{2}$-norm (see \cite{LGM09} and note that the sum in \eqref{Z_represent_Mir} does not take into account time $0$) and has a continuous modification $ \widetilde{Z}$ as in the previous section. We then specify the labels on the “infinite loop” by further enlarging the underlying probability space and introducing an extra standard Brownian motion $(\mathrm{b}_0(t))_{t\geq 0}$ independent of $X$ and $(\mathrm{b}_i)_{i\in \mathbb{N}}$. Under $\mathbf{Q}$, we then set:
  \begin{eqnarray} \label{eq:ZetZtilde} Z_t := \mathrm{b}_0(-I_t)  + \widetilde{Z}_t, \quad t\geq 0, \end{eqnarray}
  Equivalently, the process $Z$ is obtained by concatenating the processes $Z$ constructed in each excursion of $X$ above its running infimum $t \mapsto I_t$, with each one shifted by the associated value $\mathrm{b}_0(-I_t)$, for $t \ge 0$. We emphasize again that all these definitions are compatible with the previous ones, so there should be no ambiguity.

\par
Under $\mathbf{Q}$, we extend the notation used under $\mathbf{P}$. To keep the framework consistent, we set $\mathrm{b}_0:=0$ under both $\mathbf{N}$ and $\mathbf{P}$ by convention. Moreover, it follows directly from Lemma~\ref{b_Brownian_Brigde} and excursion theory that, under  $\mathbf{N}$ and  $\mathbf{Q}$, and conditioned on $X$, the process remains a Gaussian process with the same covariance function given by \eqref{def:Cov}, except that the interval $[0,1]$ is replaced by $[0,\sigma]$ and $\mathbb{R}_+$, respectively.

\paragraph{Markov property under $\mathbf{Q}$.} For $t \geq 0$, introduce $\mathcal{F}_{t}$ the 
sigma-field generated by $(X_{s}:~0 \leq s \leq t)$, by $(\mathrm{b}_0(s):~0 \leq s \leq -I_t)$, and by the point measure
$$ \sum_{i\in \mathbb{N}, \mathrm{t}_{i}\leq t} \delta_{\mathrm{t}_i, \mathrm{b}_i}~,   $$
and completed by the collection of all $\mathbf{Q}$-negligible sets. Let $T$ be an $(\mathcal{F}_{t})_{t\geq 0}$--stopping time such that $T<\infty$ almost surely under $\mathbf{Q}$.   Since $X$ is strong Markov, $(X_{t\wedge T})_{t\geq 0}$ is $\mathcal{F}_{T}$--measurable.  Also,
it follows from \eqref{Z_represent_Mir} that $(X_{t\wedge T},Z_{t\wedge T})_{t\geq 0}$ is $\mathcal{F}_{T}$--measurable. It is then straightforward to infer from the classical strong Markov property of $X$  that conditionally on $ \mathcal{F}_T$, the shifted process $( X_{T+s}-X_T : s \geq 0)$ whose jumps are decorated with the Brownian bridges has the same law as the initial decorated process. To be more precise, note that we could have constructed the process $(X,Z)$ under $\mathbf{Q}$ as a measurable function of the Brownian motion $\mathrm{b}_0$ and of an indendent Poisson point measure
$$\mathcal{M}=\sum_{i\in \N}\delta_{(\mathrm{t}_i,\Delta_i,\mathrm{b}_i)}$$
with intensity $\mathrm{d} t\,  \Gamma(-\alpha)^{-1}x^{-\alpha-1}\mathrm{d}x\ind_{x>0}\, \mathbb{P}^{(1)}_{0\to 0}(\mathrm{d}b)$, where the atoms $\Delta_i$ are the jumps of the process $X$. Indeed, by the Lévy-Itô representation, $X$ can be reconstructed from this jump process, and this reprensentation is a convenient way to express the fact that every jump of $X$ is marked by an independent Brownian bridge. It is now straightforward that if $T$ is a finite $(\mathcal{F}_t)_{t\geq 0}$--stopping time, then conditionally on $\mathcal{F}_T$, the measure $\mathcal{M}^{(T)}=\sum_{i\in \N:t_i>T}\delta_{(\mathrm{t}_i-T,\Delta_i,\mathrm{b}_i)}$ has same distribution as $\mathcal{M}$: this is proved, as is usual, by first assuming that $T$ takes countably many values, and then by approximating a general stopping time $T$ by $2^{-n}\lceil 2^n T\rceil$, first considering expectations of functionals of $\mathcal{M}^{(T)}$ that do not depend on the values of the atoms $\Delta<\eps$ for some $\eps>0$.

Let us recast this ``enriched'' version of the strong Markov property in terms of excursion theory for $(X,Z)$ in a more practical way. Let $t>0$, under $\mathbf{Q}$, consider $(u_{i},v_{i})_{i \geq 1}$ the connected components of  the open set $\{r>t:\: X_{r}>I_{t,r}\}$ and  
introduce the excursions processes
\[X_{s}^{i}:=X_{(u_{i}+s)\wedge v_i}-X_{u_{i}}\:\:\text{and}\:\:Z_{s}^{i}:=Z_{(u_{i}+s)\wedge v_i}-Z_{u_{i}},\qquad s\geq 0,\]
 and the point measure
\begin{equation}\label{N:point:measure}
\mathcal{N}^{[t]}:=\sum \limits_{i \geq 1}\delta_{X_{u_{i}},X^{i},Z^{i}}.
 \end{equation}

\begin{lem}[Markov property via excursion theory]\label{Markov_lem}
 Let $T$ be an $(\mathcal{F}_t)_{t\geq 0}$--stopping time such that $\mathbf{Q}(T<\infty)=1$. Under $\mathbf{Q}(\cdot\, |\, \mathcal{F}_T)$,  the point measure $\mathcal{N}^{[T]}$ is a Poisson measure with intensity
\[\mathbbm{1}_{(-\infty,X_{T}]}(x) \mathrm{d}x\:\mathbf{N}( \mathrm{d}X\: \mathrm{d}Z).\]
\end{lem}
\begin{proof} By the preceding discussion, it suffices to prove this statement for $T=0$. The excursions of the Lévy process $X$ above its past minimum $I$ form a Poisson random measure by a classical result of excursion theory (see \cite[Chap. IV]{Ber96}). Here, we use the fact that the local time associated with the excursion measure $\mathbf{N}$ is  $t \mapsto -I_{t}$. Now, given the fact that the Brownian bridges decorating the jumps of $X$ are i.i.d.\ conditionally on $X$, we deduce that the same property is true within each of the excursions $X^i$ of $X$. 
We conclude since, by \eqref{Z_represent_Mir}, the process $Z^i$ is a measurable functional of the jump process of $X^i$ decorated by the Brownian bridges. 
\end{proof}
\begin{figure}[!h]
 \begin{center}
 \includegraphics[width=13cm]{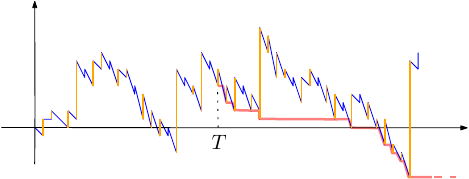}
 \caption{Illustration of the Markov property under $ \mathbf{Q}$. Conditionally on $ \mathcal{F}_{T}$, the excursions of $X$ above the running infimum starting at $T$ (in red above), together with the shifted $Z$ processes, form a Poisson process of intensity given in Lemma \ref{Markov_lem}. }
 \end{center}
 \end{figure}
Note that for a given $t>0$, the definitions of $X_{u_{i}},X^{i},Z^{i}$ and $\mathcal{N}^{[t]}$ can be directly extended under $\mathbf{N}$, where we let $\mathcal{N}^{[t]}=0$ on the event $\{\sigma\leq t\}$. By standard arguments, see for instance \cite[Theorem XII.4.1]{RY99},  the previous discussion applies under $\mathbf{N}$, for a positive stopping time, replacing the Markov property of $X$ under $\mathbf{P}$ by its analog under $\mathbf{N}$. The only difference is that the filtration $(\mathcal{F}_t)_{t\geq 0}$ should now be completed by the negligible sets under $\mathbf{N}$, but we keep the same notation for simplicity. 

\begin{cor}\label{cutting_N} 
Let $T$ be an $(\mathcal{F}_t)_{t\geq 0}$--stopping time such that $\mathbf{N}(T\in \{0,\infty\})=0$. Then, under 
$\mathbf{N}(\cdot\, |\, \mathcal{F}_{T})$, the point measure $\mathcal{N}^{[T]}$ is a Poisson point measure with intensity
\[\mathbbm{1}_{[0,X_{T}]}(x) \mathrm{d}x\:\mathbf{N}( \mathrm{d}X  \mathrm{d}Z).\]
\end{cor}

\begin{figure}[!h]
 \begin{center}
 \includegraphics[width=8cm]{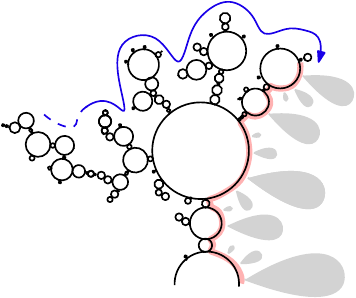}
 \caption{Illustration of the Markov property in terms of the looptree.  We explore (in blue  on the figure) in clockwise order a portion of the looptree --together with its labeling $Z$-- until a stopping time $T$. The remaining pieces (in gray) attached to the ``trunk'' made of the (closure of the) loops linking $\Pi_{d}(0)$ to $\Pi_{d}(T)$  are distributed according to $ \mathbf{N}$ \textit{provided we shift their labels by the labels of their roots}. Indeed, the $Z$-labels on the blue part in the above figure are $ \mathcal{F}_{T}$-measurable and have to be added to recover the actual $Z$-labels in the gray parts. In the proofs of Proposition \ref{pinch_points_are_not_record} and Theorem \ref{technical_uniform_balls}, we shall  use \textit{a priori} estimates given by Proposition~\ref{variations_Z} on the red part.   \label{fig:markovtrunk}}
 \end{center}
 \end{figure}
 
\section{Local minima and records of  $Z$}\label{sec:local:minima}
In this section, we begin our study of the fine properties of the process $Z$, focusing in particular on its  two-sided and one-sided minimal (local) records.  Namely,  we say that a time $t$ is a \textbf{local minimal record} of $Z$ and that  $Z_t$ is a \textbf{local minimal value} on the right, on the left, or on both sides respectively if there exists $ \varepsilon>0$ such that  we have:
 $$( \mbox{left})\quad Z_{t}  = \min_{[(t- \varepsilon)\vee 0,t)} Z, \qquad \quad  (\mbox{right}) \quad Z_t  = \min_{[t,t+\varepsilon]} Z, \qquad \quad (\mbox{two-sided})\quad  Z_{t} = \min_{[(t- \varepsilon)\vee 0, t+ \varepsilon]}Z.$$
  To simplify notation, we write  \begin{eqnarray} \mathrm{LeftRec} \quad \mbox{ and } \quad \mathrm{RightRec},  \end{eqnarray} 
 for  the set of all times $t \geq 0$ at which $Z$ attains a  local left or  right minimum, respectively. Thus, $ \mathrm{LeftRec} \cup \mathrm{RightRec}$ stands for the set of  all one-sided records, while $ \mathrm{LeftRec} \cap \mathrm{RightRec}$ is the set of all two-sided local records. Their images by $ Z$ are the corresponding (local) minimal values. 
Recall also from Section \ref{sec:codagearbre} that local minimal records of $Z$ on one side (resp.\ both sides) correspond to points of degree at least $2$ (resp.\ $3$) in the tree $ \mathcal{T}_{ \mathfrak{z}}$. In Section~\ref{sub:section:Z:0:1:minima}, we prove that, almost surely under $\mathbf{P}$, the following properties hold:
\begin{enumerate}[(i)]
\item  one-sided records do not happen on the skeleton of $\mathcal{L}$ (Proposition \ref{pinch_points_are_not_record}). That is:
$$ \Pi_d\big(\mathrm{LeftRec} \cup \mathrm{RightRec}\big) \cap \mathrm{Skel}= \varnothing.$$
\item  two-sided local minima of $Z$ \textbf{over} $\mathbf{[0,1]}$ are distinct (Proposition \ref{distinct}). In other words:
$$\mbox{$Z$ is injective on $\mathrm{LeftRec} \cap \mathrm{RightRec}$}.$$
\end{enumerate}
In particular, as a consequence of (ii),  the process $Z$ attains its global minimum at a unique time $t_*\in [0,1]$. As explained in  Section \ref{sec:codagearbre}, the image of this time $t_*$ under  $\Pi_{\mathfrak{z}}$ is the root of the tree $\mathcal{T}_{\mathfrak{z}}$.  Section \ref{sec:balls:t:z} uses this 
to  derive tail bounds for the volume of balls near the root in $ \mathcal{T}_{ \mathfrak{z}}$  (Theorem \ref{technical_uniform_balls}). Finally, in Section \ref{sec:label:pinch}, we study the behavior of  $Z$ over branches  of $\mathcal{L}$ and prove that, under $\mathbf{P}$, it holds that:
\begin{enumerate}[(iii)]
\item one-sided local minima of $Z$ \textbf{over a branch} of $ \mathcal{L}$ are surrounded by ``blocking'' loops (Proposition~\ref{lem:non-icnreasealongbranches}).
\end{enumerate}
The above results on $Z$ are  key ingredients in identifying the quotient set $ \mathcal{S}$  (see Theorem \ref{main_theorem_topology}) and in obtaining technical estimates for the volume of balls of $\mathcal{S}$ (Theorem \ref{technical_uniform_balls}).

 \subsection{Local minima of $Z$ over $[0,1]$}\label{sub:section:Z:0:1:minima}

 Recall the definition of pinch points in $\mathcal{L}$,  which, by  Proposition \ref{topologie_loop_tree}, correspond to the equivalence classes of $\sim_d$  with two elements. The goal of this section is to establish the first two items above.  Their proofs rely on the Markov property and the study of the minima of $Z$. We begin with a direct consequence of the Markov property of  Corollary \ref{cutting_N} and the scaling property:
\begin{cor}\label{Z<-1} We have $\mathbf{N}(\inf Z\leq -1)\in(0,\infty)$. Moreover, for every $r>0$:
\[\mathbf{N}(\inf Z\leq -r)=\frac{\mathbf{N}(\inf Z\leq -1)}{r^{2}}~.\]
\end{cor}
\noindent In Proposition \ref{sec:stable-map-3} we will complete this picture by giving the exact value of $\mathbf{N}(\inf Z\leq -1)$, which will require a  detailed analysis. 
\begin{proof} Recall the notation  $\sigma$  for  the lifetime of $X$ under $ \mathbf{N}$. Under $\mathbf{N}(\cdot\:|\: \sigma>1)$ and conditionally on $X_{1}$, Lemma \ref{Markov_lem} entails that the number of atoms of $ \mathcal{N}^{[1]}:=\sum_{i\geq 1}\delta_{X_{u_{i}},X^{i},Z^{i}}$ for which $\inf Z^{i}\leq -1$ 
is a Poisson random variable with parameter $X_{1}\cdot \mathbf{N}(\inf Z\leq -1)$. This random variable is non-trivial and the continuity of $Z$ entails that it  is finite $\mathbf{N}(\cdot\:|\: \sigma>1)$~--~almost surely. Therefore, we have $\mathbf{N}(\inf Z\leq -1) \in (0, \infty)$. The second point follows by disintegration with respect to $\sigma$ and the scaling property. Specifically, we have:
\begin{eqnarray*}
\mathbf{N}\big(\inf Z\leq -r\big)& \underset{\eqref{eq:excursionmeasuredecomp}}{=}&\frac{1}{|\Gamma(-\frac{1}{\alpha})|}\int_{0}^{\infty}\frac{\d v}{v^{\frac{1}{\alpha}+1}} \mathbf{N}^{(v)}\big(\inf Z\leq -r\big)\\
&\underset{ \mathrm{scaling}}{=}&\frac{1}{|\Gamma(-\frac{1}{\alpha})|}\int_{0}^{\infty}\frac{\d v}{v^{\frac{1}{\alpha}+1}} \mathbf{N}^{(v /r^{2\alpha})}(\inf Z\leq -1)\\
&=&\frac{1}{|\Gamma(-\frac{1}{\alpha})|}r^{-2}\int_{0}^{\infty}\frac{\d u}{u^{\frac{1}{\alpha}+1}} \mathbf{N}^{(u)}\big(\inf Z\leq -1\big) = \ \frac{\mathbf{N}\big(\inf Z\leq -1\big)}{r^{2}}~.\end{eqnarray*}
\end{proof}
First we show (i), i.e.\ that  pinch points of $ \mathcal{L}$ cannot be minimal records (from one side) of the process $Z$, see \cite[Proposition 4.2]{LG07} for the similar statement in the case of Brownian motion on the Brownian tree. 
\begin{prop}\label{pinch_points_are_not_record}
$\mathbf{P}$-a.s., we have $\Pi_d(\mathrm{LeftRec} \cup \mathrm{RightRec}) \cap \mathrm{Skel}= \varnothing$.
\end{prop}

\begin{proof} We start by remarking that, by the re-rooting property \eqref{re-rooting} and the invariance by time-reversal \eqref{symmetric}, it is enough to show that 
$\mathbf{P}$-a.s., we have: 
\begin{equation}\label{*:recod}
Z_{s}>\inf \limits_{[0,s]} Z, \quad \mbox{ for every $0<s<t<1$ such that $d(s,t)=0$.}
\end{equation}
Recall that by Proposition \ref{topologie_loop_tree},  the condition $d(s,t)=0$, also written $s\sim_d t$, is equivalent to
$$ X_{t}=X_{s-}  \mbox{ and \ \ } X_{r}>X_{s-} \mbox{ \ for every } r\in (s,t).$$
The strategy of the proof involves discretizing the times at which $Z$ reaches a new minimal record. We then apply the Markov property  at these times  to establish that it is very unlikely to have times nearby corresponding  to pinch points of $\mathcal{L}$ supporting a large dangling looptree. First remark that  by a scaling argument and 
\eqref{eq:ZetZtilde}, it suffices  to prove the lemma for the excursion above the running infimum straddling time $1$ under $\mathbf{Q}$. More precisely,  under $\mathbf{Q}$, set:
 \[\tau_{*}:=\sup\big\{t\in [0,1]: X_{t}= \inf \limits_{[0,1]} X \big\}\]
 the starting time of this excursion and $\tau_{q} := \inf\{t \geq 1 : Z_t = \inf_{[\tau_{*},1]}Z -q\},$ for $q>0$. It will also be useful to introduce, for every $\delta,q\in (0,\infty)$, the event $\mathcal{A}_{\delta,q}$ defined by:
 \\
 \\
 $\bullet$ There exist $1\leq s<s+\delta<t$, with $s< \tau_{q}$ and $s\sim_d t$, such that $Z_{s}=\inf_{[\tau_{*},s]}Z$.
 \\
 \\
 By the density of $\mathbb{Q}_+$ in $\mathbb{R}_+$ and since  $\mathbf{Q}(\tau_{*}\in[1-\eps,1])>0$ for every $\eps>0$, to obtain \eqref{*:recod} it is enough to show that $\mathbf{Q}(\mathcal{A}_{\delta,q})=0$, for every $\delta,q\in (0,\infty)\cap \mathbb{Q}$. In this direction, we fix two such $\delta,q$ and $n \geq 1$. We now discretize time and introduce the stopping times $(S^{n}_i,T^{n}_i)_{i\geq 1}$ defined by induction as follows. First take $$S^{n}_{1}:=\inf\big\{t\geq 1:\:\: Z_{t}=\inf\limits_{[\tau_{*},1]} Z\big\}\quad\text{and}\quad T^{n}_{1}:=\inf\big\{t\geq S^{n}_{1}:\:\: X_{t}=X_{S_1^n}-2^{-n}\big\},$$ and then, for every $i\geq 1$, take
\[S^{n}_{i+1}:=\inf\big\{t\geq T^{n}_{i}:\:\: Z_{t}=\inf\limits_{[\tau_{*},T^{n}_{i}]} Z\big\}\quad \text{and}\quad T^{n}_{i+1}:=\inf\big\{t\geq S^{n}_{i+1}:\:\: X_{t}=X_{S^{n}_{i+1}}-2^{-n}\big\}.\]
Remark that for every $i\geq 1$, the random variables $S^{n}_{i}$ and $T^{n}_{i}$ are $(\mathcal{F}_t)_{t\geq 0}$--stopping times which are finite $\mathbf{Q}$~--~a.s. For every $i\geq 1$, consider $(s^{n}_{i,k},t^{n}_{i,k})_{k\geq 1}$ the connected components of the open set $\{s\in[S^{n}_{i},T^{n}_{i}]:\: X_{s}>\inf\limits_{[S^{n}_{i},s]}X \}$ and introduce the random variables:
\[R^{n}_{i}:=\inf \limits_{k\in \mathbb{N}} \inf \limits_{r\in[s^{n}_{i,k},t^{n}_{i,k}]} \big(Z_{r}-Z_{s^{n}_{i,k}}\big) \quad;\quad \widetilde{R}^{n}_{i}:= \inf \limits_{r\in[S^{n}_{i},T^{n}_{i}]} \big(Z_{r}-Z_{S^{n}_{i}}\big) \quad \mbox{ and } \quad V^{n}_{i}:=\sup \limits_{k\geq 1} \big(t^{n}_{i,k}-s^{n}_{i,k}\big).\]
In words, the random variable $V^{n}_{i}$ is the size of the largest excursion coding for a looptree grafted on the first $2^{-n}$ unit of length on the trunk after time $S_{i}^{n}$, the random variable $\widetilde{R}^{n}_{i}$ is the smallest displacement of $Z$ on these excursions, whereas $R^{n}_{i}$ is the smallest displacement of the process $Z$ shifted by the label of the root in each of these excursions -- see Figure \ref{fig:markovpinch}.
\begin{figure}[!h]
 \begin{center}
 \includegraphics[width=10cm]{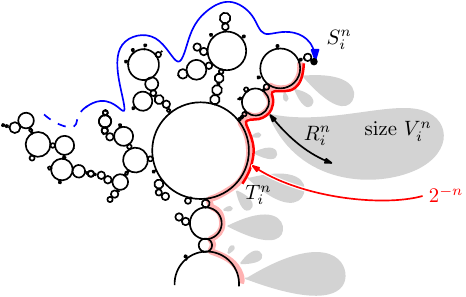}
 \caption{Illustration of the Markov property applied at time $S_{i}^{n}$: among the gray looptrees grafted on the $2^{-n}$ unit of length on the right of the trunk at time $S_{i}^{n}$ (recall the caption of Figure \ref{fig:markovtrunk}), the random variable $V_{i}^{n}$ is the maximal size of such a gray looptree, whereas $R_{i}^{n}$ is the minimal  $Z$ displacement within those looptrees. Note that on this picture, both happen to belong to the same looptrees, which need not be the case in general. The minimal displacement $\widetilde{R}_{i}^{n}$ is obtained by further shifting the labels of the gray looptrees by the values of $Z$ on the light red part. \label{fig:markovpinch}}
 \end{center}
 \end{figure}
 \par
Note now that if there exist $s \sim_d t$ as in the definition of  $\mathcal{A}_{\delta,q}$,  it follows that when $S_i^n \leq s$ and $V^n_i \leq \delta$, we must have $T_i^n \leq s$, since $t-s\geq \delta$. Furthermore, when  $T_i^n \leq s$ we also have $S^n_{i+1} \leq s$, because $Z_s = \inf_{[\tau_*,s]}Z$.  We infer that  $\mathcal{A}_{\delta,q}$ is contained in the event:
\begin{equation}\label{**:recod}
\big\{\exists i \geq 1:~S_{i}^{n}<\tau_{q}\:\:\text{and}\:\:V_{i}^{n}\geq \delta\big\}.
\end{equation}
Moreover, an application of  the Markov property, combined with Lemma \ref{Markov_lem} and  Corollary \ref{Z<-1}, gives that the sequence $(R^{n}_{i},V^{n}_{i})_{i\geq 1}$ is i.i.d., and we have
  \begin{eqnarray} \label{eq:estim1}\mathbf{Q}(R^{n}_{i}<-r)&=&1-\exp\big(-\mathbf{N}(\inf Z<-1) 2^{-n} r^{-2}\big)\\
\mathbf{Q}(V^{n}_{i}>r)&=&1-\exp\big(-\mathbf{N}(\sigma>1) 2^{-n}r^{-\frac{1}{\alpha}}\big)~,     \label{eq:estim2}\end{eqnarray}
for every $r>0$.
Let us now explain the intuition of the rest of the proof. If we had the analog of \eqref{eq:estim1} for $ \widetilde{R}$ instead of $R$, then each $ \widetilde{R}_{i}^{n}$ would be typically of order \textit{at least} $2^{-n/2}$ so that we would need fewer than $ \approx q\cdot 2^{n/2}$  steps of the above process to reach $\tau_{q}$. In the meantime, at each step of the discrete exploration process, the probability that $V_i^n > \delta$, i.e.\ that $S_{i}^{n}$ is an approximate pinch time ``with a mass at least $\delta$ above it'', is of order $2^{-n}$. Since $ 2^{-n} \cdot 2^{n/2} \ll1$  it is very unlikely that we encounter such times. The problem is that the variation $ \widetilde{R}$ is governed by $R$ and the labels of the process $Z_{s_{i,k}^n}$; the labels  on the light blue part of Figure \ref{fig:markovpinch}. To circumvent this difficulty, we will use Lemma \ref{variations_Z} (or more prosaically \eqref{Z:variation:L})  as an  a priori control on the variations of $Z$ to transfer our estimates on $R$ to $ \widetilde{R}$.   
\noindent To this end, fix
\[0<\eta<\frac{1}{2}<\eta^{\prime}<\beta<1\]
with $1-2\eta +\eta^{\prime}<\beta$ and set
\[M_{n}:=\#\{i\leq 2^{\beta n}:\:V^{n}_{i}>\delta \} \:\:;\:\:N_{n}:=\#\{i\leq 2^{\beta n}:\:R^{n}_{i}<-2^{-\eta n} \}.\]
By \eqref{eq:estim1}, we have $\mathbf{Q}(V^{n}_{i}>\delta) \leq C\cdot 2^{-n}$ and $ \mathbf{Q}(R^{n}_{i}<-2^{-\eta n}) \geq c\cdot 2^{-n(1-2 \eta)}$, for some constants $c,C>0$. Since $\beta <1$ and $1-2\eta +\eta^{\prime}<\beta$ it follows by crude  bounds and the Borel--Cantelli lemma that, $ \mathbf{Q}$ -- a.s.~, we  have 
$$ M_{n}  =0 \quad \mbox{ and } \quad N_{n} \geq 2^{\eta^{\prime}n},$$
for  $n\geq 1$ large enough. Fix $\tilde{\eta}\in(\eta,\frac{1}{2})$ and for every $n>1$, let $\mathcal{B}_{n}$ be the event defined as follows:
\\
\\
$\bullet$ For every $m\geq n$, $M_{m}=0$ and $N_{m}\geq 2^{\eta^{\prime}m}$~;
\\
\\
$\bullet$ For every $s,t\in[0,\tau_{q}]$, we have $|Z_{s}-Z_{t}|\leq n\cdot d(s,t)^{\tilde{\eta}}.$
\\
\\
Since $ \tau_{q} < \infty$ almost surely,  the above considerations and  excursion theory, combined with Equation \eqref{Z:variation:L} and the fact that the Brownian motion $\mathrm{b}_0$ on the infinite loop is $\tilde{\eta}$-Hölder continuous, imply  that
$\mathbf{Q}(\mathcal{B}_{n}) \to 1$ as $n \to \infty$. Let us now conclude that $\mathbf{Q}(\mathcal{A}_{\delta,q})=0$. In this direction,  remark that we can write
\begin{align*}
 \mathbf{Q}(\mathcal{A}_{\delta,q})&=  \lim_{n \to \infty}\mathbf{Q}(\mathcal{A}_{\delta,q}\cap \mathcal{B}_{n})\underset{\eqref{**:recod}}{\leq }\lim_{n \to \infty}\mathbf{Q}(\{T_{2^{n\beta}}^{n}<\tau_{q}\}\cap \mathcal{B}_{n}),
\end{align*}
where in the last inequality we use that, under $\mathcal{B}_{n}$, we have $M_{m}=0$ for every $m\geq n$. Consider $n_{0}$, the smallest integer such that for every $n\geq n_{0}$:
$$2^{(\eta^{\prime}-\eta)n} \big(-1+n 2^{(\tilde{\eta}-\eta)n}\big)< -q$$
We are going to show that for every $n\geq n_{0}$, we have $\mathbf{Q}\big( \{T^{n}_{2^{n\beta}}<\tau_{q}\}\cap \mathcal{B}_{n}\big)=0$, which will complete the proof of the proposition. We argue by contradiction. First remark that, by definition, we have:
\[\sup\limits_{k\in \mathbb{N}} d(S_{i}^{n}, s^{n}_{i,k})\leq 2^{-n},\quad \text{ for } n\geq 1.\]
So if $T^{n}_{2^{n\beta}}< \tau_{q}$ under $\mathcal{B}_{n}$, we will have:
\begin{align*}
  Z_{T^{n}_{2^{\beta n}}}&\leq \sum \limits_{1\leq i\leq 2^{\beta n}}\widetilde{R}^{n}_{i}
  \leq \sum \limits_{1\leq i\leq 2^{\beta n}}\min \left\{(R^{n}_{i}+n 2^{-\tilde{\eta}n}) ; 0 \right\} \leq N_{n}\cdot \big(-2^{-\eta n}+n 2^{-\tilde{\eta}n}\big)
\end{align*}
which, for $n\geq n_0$, is smaller that   $2^{\eta^{\prime}n} \big(-2^{-\eta n}+n 2^{-\tilde{\eta}n}\big)< -q$ and we obtain a contradiction.
\end{proof}
We conclude this section deducing item (ii) from Proposition \ref{pinch_points_are_not_record}.

\begin{prop}\label{distinct} $\mathbf{P}$-a.s.~the (two-sided) local minima of $Z$ are distinct. 
\end{prop}
\begin{proof} It follows straightforwardly  from the invariance under re-rooting and continuity of $Z$ that for every fixed time $t$, with $ \mathbf{P}$-probability one, $t$ is not a local minimal record, so we exclude  times $0$ and $1$ in the following. Suppose that $0  < s < t <1$ are two times such that $Z$ reaches the same local minima at $s$ and at $t$. First,  by Proposition \ref{pinch_points_are_not_record},  we notice that $s$ and $t$ cannot be pinch point times.
 Moreover, if $\Pi_d(s)$ and $\Pi_d(t)$ are leaves belonging to a common loop, then $Z$ cannot realize the same minima at  $s$ and $t$ since  the labels along this loop evolve as a Brownian bridge -- and the minima of a Brownian bridge are distinct. Consequently, by Property {\hypersetup{linkcolor=black}\hyperlink{prop:A:4}{$(A_4)$}}, the points  $\Pi_d(s)$ and $\Pi_d(t)$ are separated in $ \mathcal{L}$ by a countable collection of loops. Let us now show that this is also impossible. To study this case  we work conditionally on $X$, that is under $\mathbf{P}_{X}$.  By definition, one can find rationals $p_{1}<s<p_{2}<q_{1}<t<q_{2}$, with  $Z_{s}= \inf_{[p_{1},p_{2}]} Z$ and $ Z_{t}= \inf_{[q_{1},q_{2}]} Z$, such that $ \Pi_d([p_{1},p_{2}])$ is separated from $\Pi_d([q_{1},q_{2}])$ in the looptree $ \mathcal{L}$ by a non-trivial loop associated with a time $r \in \mathrm{Branch}(s,t) \setminus \{s,t\}$. Recall the notation $\mathrm{f}_r$ for the parametrization of the loop associated with $r$, and  let $u\in \mathrm{f}_{r}([0,1])$ (resp.\ $v\in \mathrm{f}_{r}([0,1])$) be the  closest point of the loop from  $\Pi_d\big([p_{1},p_{2}]\big)$ (resp.\ $\Pi_d\big([q_{1},q_{2}]\big)$).  In particular, $u$ and $v$ are  pinch points of $\mathcal{L}$, and we write $s_{1}<s_{2}$ and $t_{1}<t_{2}$ for the elements of $\Pi_d^{-1}(u)$ and $\Pi_d^{-1}(v)$ respectively. We refer to Figure \ref{fig:minima} for an illustration.
\begin{figure}[!h]
 \begin{center}
 \includegraphics[width=11cm]{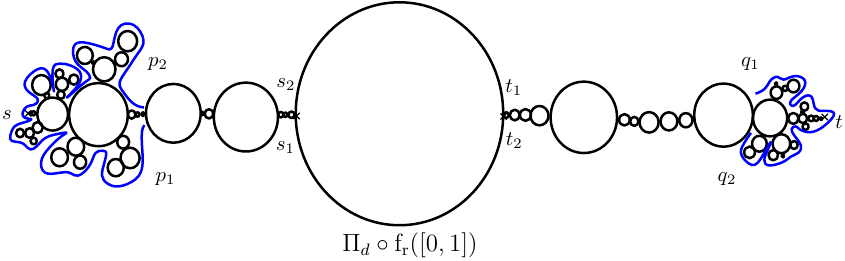}
 \caption{\label{fig:minima} The two times $s$ and $t$ are separated by the loop $\Pi_d\circ  \mathrm{f}_{r}([0,1])$. In blue are represented the set $\Pi_d([p_{1},p_{2}])$ and $\Pi_d([q_{1},q_{2}])$.}
 \end{center}
 \end{figure}
 \noindent Next, using \eqref{Z_represent_Mir} we notice that we have $ \inf_{[p_{1},p_{2}]} Z= \inf_{[q_{1},q_{2}]} Z$ if and only if 
 \begin{equation}\label{nes:disjoint}
 Z_{s_{1}}-Z_{t_{1}}= \inf_{[q_{1},q_{2}]} \big(Z-Z_{s_1})-\inf_{[p_{1},p_{2}]} (Z-Z_{t_1}).
 \end{equation}
Furthermore, by Propositon \ref{b_Brownian_Brigde}, the processes 
$\big(Z_{\ell}-Z_{s}\big)_{\ell\in[s_{1},s_{2}]}~;~\big(Z_{\mathrm{f}_{r}(\ell)}\big)_{\ell\in [0,1]}$ and $\big(Z_{\ell}-Z_{s}\big)_{\ell\in[t_{1},t_{2}]}$ 
are independent under $\mathbf{P}_{X}$. We infer that the left and right sides of \eqref{nes:disjoint} are independent. Finally since $Z_{s_{1}}-Z_{t_{1}}$ is a non-trivial Gaussian random variable, we deduce that, $\mathbf{P}_X$-a.s., identity \eqref{nes:disjoint}   cannot hold. The desired result follows
since $p_1 < p_2 < q_1 < q_2$ are arbitrary rational numbers.
\end{proof}

\subsection{The mass  of balls  centered at the root of $\mathcal{T}_{\mathfrak{z}}$}\label{sec:balls:t:z}
The continuity of $Z$,  together with Proposition \ref{distinct},  implies the existence, $\mathbf{P}$-almost surely, of a unique $t_*\in [0,1]$ realizing the minimum of $Z$, i.e. $Z_{t_*}=\inf_{[0,1]} Z$. We shall consider the tree $ \mathcal{T}_{ \mathfrak{z}}$ coded by $Z$ as in \eqref{eq:defpseudodistancearbre}. We recall that   $\Pi_{\mathfrak{z}}(t_*)$ is the root of $\mathcal{T}_{\mathfrak{z}}$ and that $\mathrm{Vol}_\mathfrak{z}$ is the pushforward of the Lebesgue measure by $\Pi_\mathfrak{z}$. The goal of this section is to control the mass --  for $\mathrm{Vol}_\mathfrak{z}$ -- of closed balls in $\mathcal{T}_{\mathfrak{z}}$ centered at  $\Pi_{\mathfrak{z}}(t_*)$. Namely, for every $r\geq 0$, we set
$$B_{\mathfrak{z}}(\Pi_{\mathfrak{z}}(t_*),r):=\{u\in\mathcal{T}_{\mathfrak{z}}:~\mathfrak{z}(u,\Pi_{\mathfrak{z}}(t_*))\leq r\}.$$
The goal of this section is to prove that:
\begin{theo}\label{technical_uniform_balls}
Fix $\delta >0$. There exist $c,C\in(0,\infty)$ such that 
\begin{equation}\label{vol-t_*}
\mathbf{P}\Big(\mathrm{Vol}_\mathfrak{z}\big(B_{\mathfrak{z}}(\Pi_{\mathfrak{z}}(t_*),r)\big)\geq r^{2\alpha-\delta}\Big)\leq  C\cdot \exp(-r^{-c}),
\end{equation}
for every $r>0$.
\end{theo}
In Part \ref{part:random_maps} of this paper, this estimate will be used to bound the mass of balls centered at the root in the subsequential scaling limits $(\mathcal{S}, D)$ of $\bq$-Boltzmann maps that were discussed in the Introduction.  In turn, this will allow us to deduce an \textit{a priori} bound of the form $ D^{*} \leq  D^{1- \delta}$ (locally), that is required in the final steps of the proof of our main result stating that $D=D^{*}$, see Proposition \ref{preliminary_control} in Section \ref{Sec:A:priori:bound}. 

\begin{proof} In order to prove Theorem \ref{technical_uniform_balls}, we are going to write $\mathrm{Vol}_\mathfrak{z}\big(B_{\mathfrak{z}}(\Pi_{\mathfrak{z}}(t_*),r)\big)$ in terms of the time that $Z$ spends near its infimum. Since $t_*$ is the argmin of $Z$, and $\mathrm{Vol}_{\mathfrak{z}}$ is the pushforward of the Lebesgue measure on  $[0,1]$ under the projection $\Pi_{\mathfrak{z}}$, we have
$$\mathrm{Vol}_\mathfrak{z}\big(B_{\mathfrak{z}}(\Pi_{\mathfrak{z}}(t_*),r)\big)\underset{ \eqref{eq:defpseudodistancearbre}}{=} \int_0^1 \d s ~\mathbbm{1}_{Z_{s}\leq Z_{t_*}+r}~. $$ 
Consequently, Theorem \ref{technical_uniform_balls} can be translated in terms of $Z$ in the following form:
Fix $\delta>0$, then there exist $c,C\in(0,\infty)$ such that 
\begin{equation*}
\mathbf{P}\Big(\int_{0}^{1}\d s~\mathbbm{1}_{Z_{s}\leq Z_{t_*}+r}\geq r^{2\alpha-\delta}\Big)\leq C\cdot \exp(-r^{-c}),
\end{equation*}
for every $r>0$.
We stress that the constants $C,c$ might  depend on $\delta$. Let us start by presenting the idea of the proof. We will first show using a re-rooting argument due to Chaumont and Uribe Bravo that 
 \begin{eqnarray} \label{eq:resteaudessus}\mathbf{P}\Big(\int_{0}^{1}\d s~\mathbbm{1}_{Z_{s}\leq Z_{t_*}+ r}\geq r^{2\alpha-\delta}\Big)\quad \leq \quad \mathbf{P}\Big(\int_{0}^{1}\d s\mathbbm{1}_{Z_{s}\leq r}\geq r^{2\alpha-\delta}\:\:\Big\vert \inf Z\geq -r\Big),\end{eqnarray}
 for every $r,\delta >0$.
Then we will argue that the process $Z$ conditioned on the event $\{\inf Z \geq - r\}$ cannot spend too much time near $0$, since informally each ``time'' the process $Z$ returns close to  $0$, it has a positive chance to drop below $- r$ in the next $r^{{2\alpha}}$ units of time. \par
To prove \eqref{eq:resteaudessus}, let $r,\delta>0$ and introduce $U$ a uniform random variable on $[0,1]$ -- independent of $(X,Z)$. Recall the notation $s\oplus t=s+t$ if $s+t\leq1$ and $s\oplus t=s+t-1$ otherwise. We also introduce $(A^{(r)}_{t})_{t\in[0,1]}$ the occupation time process associated with $Z$, i.e.
\[A_{t}^{(r)}:=\int_{0}^{t}\d s~\mathbbm{1}_{ Z_{s} \leq   \inf Z+r}~,\quad t\in[0,1],\]
and we consider the quantity $t^{(r)}:=\inf\big\{t\in[0,1]:\:A_{t}^{(r)}=U A_{1}^{(r)} \big\}$.
Now remark that by construction we have:
\begin{align*}
\mathbf{P}\Big(\int_{0}^{1}\d s ~\mathbbm{1}_{Z_{s}\leq  Z_{t_*}+r}\geq r^{2\alpha-\delta}\Big)
&\leq\mathbf{P}\Big(\int_{0}^{1}\d s ~\mathbbm{1}_{Z_{s}\leq Z_{t^{(r)}}+r}\geq r^{2\alpha-\delta}\Big)\\
&=\mathbf{P}\Big(\int_{0}^{1}\d s~\mathbbm{1}_{Z_{t^{(r)}\oplus s}-Z_{t^{(r)}}\leq r}\geq r^{2\alpha-\delta}\Big).
\end{align*}
Since   by \eqref{re-rooting} the process $Z$ has cyclically exchangeable increments,  we can  apply \cite[Theorem 2.2]{Cyclical} to deduce that the process $\big(Z_{t^{(r)}\oplus s}-Z_{t^{(r)}}\big)_{s\in [0,1]}$, under $\mathbf{P}$, is distributed as the process $Z$ under the probability measure $\mathbf{P}(\cdot\:|\:\inf Z>-r)$. In particular, we get \eqref{eq:resteaudessus}. 

Hence, to obtain the desired result, it suffices to show that, for every $\delta \in \mathbb{R}_{+}^{*}$, there exist $c_1, C_1\in(0,\infty)$ such that for every $r>0$:
\begin{equation}\label{eq:prop:volm:*}
\mathbf{P}\Big(\int_{0}^{1}\d s\mathbbm{1}_{Z_{s}\leq r}\geq r^{2\alpha-\delta}\:\:\Big\vert\:\: \inf Z\geq -r\Big)\leq C_1\cdot\exp(-r^{-c_1}).
\end{equation}
We start by showing that for every $\delta>0$, there exists $C_{1}^\prime\in\mathbb{R}_{+}^{*}$ such that:
\begin{equation*}
\mathbf{P}(\inf Z\geq -r) \geq  C_{1}^\prime \cdot
\big(1\wedge r^{2\alpha+\delta}\big),
\end{equation*}
for every $r>0$.
We mention that the previous inequality is not sharp, as we expect  $\mathbf{P}(\inf Z\geq -r)$  to be of order $r^{-2\alpha}$. Remark that by \eqref{Z:variation:L}, the random variable   \begin{eqnarray} \label{eq:variation2}W:=\sup \limits_{s\neq t} |Z_{t}-Z_{s}|\cdot |t-s|^{-\frac{1}{2\alpha+\delta}}  \end{eqnarray}
is finite $\mathbf{P}$-a.s so that $|Z_{t}-Z_{s}| \leq r$ if $|t-s| \leq (r/W)^{2 \alpha+\delta}$. We thus have:
\begin{eqnarray*}
\mathbf{P}\big(\inf Z \geq -r\big) &\underset{ \mathrm{re-rooting\ } \eqref{re-rooting}}{=}& \int_{0}^{1}\d s~\mathbf{E}\big[\mathbbm{1}_{\inf Z_{s\oplus \cdot}-Z_{s}\geq -r}\big]\\
& \underset{ \mathrm{Fubini}}{=} &\mathbf{E}\big[\int_{0}^{1}\d s\mathbbm{1}_{\inf Z_{s\oplus \cdot}-Z_{s}\geq -r}\big]\\
& \underset{}{=}&\mathbf{E}\big[\int_{0}^{1}\d s\mathbbm{1}_{Z_{s}\leq \inf Z+r}\big]~,\\
& \underset{ \eqref{eq:variation2}}{\geq}& \mathbf{E}[(r/W)^{2 \alpha + \delta}],
\end{eqnarray*}
which gives the desired inequality. Hence, to derive \eqref{eq:prop:volm:*} it suffices to show that for every $\delta \in \mathbb{R}_{+}^{*}$ there exist $c_2,C_2\in(0,\infty)$ such that:
\[ \mathbf{P}\Big(\{\inf Z >-r \}\cap \big\{\int_{0}^{1}\d s~\mathbbm{1}_{Z_{s}\leq r}\geq r^{2\alpha-\delta}\big\}\Big)\leq C_2\cdot \exp(- r^{-c_2}),\quad r>0.\]
The idea now is to argue similarly as in the proof of  Proposition \ref{pinch_points_are_not_record}.  We are going to translate the problem under the probability measure $\mathbf{Q}$ and discretize time in order to apply the Markov property.  First, under $\mathbf{Q}$, set:
 \[\tau_{*}:=\sup\big\{t\in [0,1]:~X_{t}=\inf\limits_{[0,1]} X\big\}\:\:\text{and}\:\:\tau^{*}:=\inf\big\{t\geq 1:~X_{t}=\inf\limits_{[0,1]} X\big\}.\] Under $\mathbf{Q}$, the excursion $(X_{t}-\inf_{[0,1]}X)_{t\in [\tau_{*},\tau^{*}]}$ is the excursion above the running infimum straddling time $1$.
To be able to relate it to  the label process under $ \mathbf{P}$, we shall, as in \eqref{eq:ZetZtilde}, consider the process $\widetilde{Z}_{t} := Z_{t }-\mathrm{b}_0(-I_t)$.  We then introduce $\widetilde{\tau}_{r}:=\inf\{t\geq 1:~\widetilde{Z}_{t}=-r \}$.
 Since $\mathbf{Q}(1-r<\tau_{*}< 2-r<\tau^{*}<2)$ decreases polynomially to $0$ as $r\to 0$, an application of the scaling property  shows that to obtain \eqref{eq:prop:volm:*} it is enough to prove that for every $\delta>0$ we can find some constants $c_3,C_3>0$ such that:  
\begin{equation}\label{eq:prop:volm:**}
 \mathbf{Q}\Big(\int_{1}^{2\wedge \widetilde{\tau}_{r}}\d s~\mathbbm{1}_{\widetilde{Z}_{s}\leq r}\geq r^{2\alpha-\delta}\Big)\leq C_3\cdot \exp(- r^{-c_3}),\quad r>0.
\end{equation}
 From now on we fix $\delta>0$ and we consider $r\in (0,1)$. Let us now discretize time to give an upper bound of the left-hand side of the display above.  We introduce the stopping times $(S_{i},T_{i})_{i\in \mathbb{N}}$ defined by induction as follows. First take $S_{1}:=1\:\:\text{;}\:\:T_{1}:=S_{1}+ r^{2\alpha}$, 
and  for every integer $i\geq 0$ set
\[S_{i+1}:=\inf\big\{s\geq T_{i}~:~ \widetilde{Z}_{s}\in [- r, r]\big\}\quad\text{;}\quad T_{i+1}:=S_{i+1}+ r^{2\alpha}.\]
We have the following trivial inequality:
  \begin{eqnarray}\mathbf{Q}\Big(\int_{1}^{2\wedge \widetilde{\tau}_{ r}}\d s~\mathbbm{1}_{\widetilde{Z}_{s}\leq  r}\geq  r^{2\alpha-\delta}\Big)\leq \mathbf{Q}\Big(S_{ \lfloor r^{-\delta}\rfloor }\leq 2\wedge \widetilde{\tau}_{r}\Big).  \label{eq:tropdetemps} \end{eqnarray}
Here is the rough idea to obtain an upper bound of the right-side hand of \eqref{eq:tropdetemps}. First, by scaling, the variations of the process $\widetilde{Z}$ over the time interval $[S_{i},T_{i}]$ of length $  r^{2 \alpha}$ should be of order $ r$. Then, since $\widetilde{Z}_{S_{i}} \leq r$, there should be a probability bounded away from $0$ that $ \widetilde{Z}$ touches $-r$ within $[S_i,T_i]$.  As a consequence, the probability of $\{S_n\leq \widetilde{\tau}_{ r}\}$ should decrease exponentially fast in $n$. As in the proof of Proposition \ref{pinch_points_are_not_record}, the problem may come from the fact that in the Markov property, the $\widetilde{Z}$-labels after time $S_{i}$ are shifted by the labels   on the loops connecting the root $\Pi_d(\tau_*)$ to $\Pi_d(S_i)$ -- see  Figure \ref{fig:markov-volume} for an illustration. We shall again bound them using \textit{a priori} estimates resulting from Proposition \ref{variations_Z} $\rm(ii)$. Of course, this discussion is informal, and we now make this picture precise.

\begin{figure}[!h]
 \begin{center}
 \includegraphics[width=10cm]{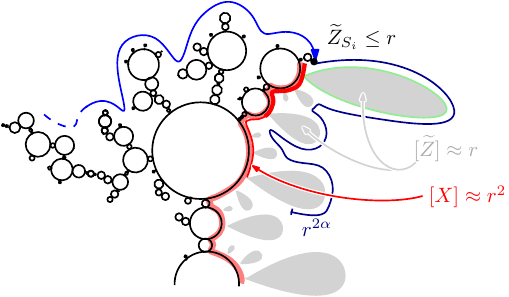}
 \caption{Illustration of the proof. At time $S_{i}$ the $\widetilde{Z}$-value is less than $ r$. By scaling, in the forthcoming $ r^{{2\alpha}}$ units of time, the variation of the infimum process $I$ of $X$ is of order $r^2$, while one expects the variation of process $\widetilde{Z}$ to be of order $r$. One way to ensure this is to ask for an excursion (in green above) to have a $\widetilde{Z}$-variation less than $-2 r$, and to happen ``early'' on the trunk, i.e.\ within $ r^{2+\eta}$ units of length after $S_i$ (in bold red above). Here we write $[\cdot ]$ to denote a variation.\label{fig:markov-volume}}
 \end{center}
 \end{figure}
\noindent Similarly to  the proof of Proposition \ref{pinch_points_are_not_record},  for every $i\geq 1$, let  $(s_{i,k},t_{i,k})_{k\in\mathbb{N}}$ be the connected components of the open set $$\{s\geq S_{i}:\: X_{s}>\inf\limits_{[S_{i},s]}X \}$$  and consider the associated excursions $ X^{i,k}_{s} := X_{(s_{i,k}+s)\wedge t_{i,k}}-X_{s_{i,k}}$ and $\widetilde{Z}^{i,k}_{s} := \widetilde{Z}_{(s_{i,k}+s)\wedge t_{i,k}}-\widetilde{Z}_{s_{i,k}}$, for $s \geq 0$. We also write $\sigma_{i,k}:=t_{i,k}-s_{i,k}$ and $\ell_{i,k}:=X_{S_i}-\inf_{[S_i,s_{i,k}]} X$. Next, we fix $\eta\in(0,1/2)$ and, for every $i\in \mathbb{N}$, we introduce the event:
\begin{eqnarray} \label{eq:graftedtotetgrand} \mathcal{B}_{i} := \Big\{ \exists k\in \mathbb{N} :~ \ell_{i,k}\leq r^{2+\eta} \mbox{ and } \min \widetilde{Z}^{i,k} \leq -4 r \mbox{ and } \sum_{\ell_{i,j}<\ell_{i,k}}\sigma_{i,j}<r^{2\alpha} \Big\}.
\end{eqnarray}
 In words, under $\mathcal{B}_{i}$, at step $i$ a looptree with a (shifted) minimal label below $- 4r$ is  branched within less than $ r^{2+\eta}$ unit of length along the ``trunk", see Figure \ref{fig:markov-volume}. By construction the $\mathcal{B}_{i}$'s are independent, and we are going to show using excursion theory that we have $\mathbf{Q}(\mathcal{B}_{i}) \geq a\cdot r^{\eta}$, for some constant $a>0$ (not depending on $r\in(0,1)$). The theorem will then follow easily from this estimate. Fix $i\in \mathbb{N}$,  and remark that by excursion theory the measure
 $$\sum \limits_{k\in \mathbb{N}} \delta_{\ell_{i,k}, \sigma_{i,k}, X^{i,k}, \widetilde{Z}^{i,k}} $$
 is a Poisson point measure with intensity $\mathbbm{1}_{[0,\infty)}(\ell)\mathrm{d}\ell\big( \mathrm{d} v/ (|\Gamma(-\frac{1}{\alpha})| v^{\frac{1}{\alpha}+1}) \big) \mathbf{N}^{(v)}(\d X, \d Z)$. In particular, the process $R^{i}_{t}:= \sum_{\ell_{i,k}\leq t} \sigma_{i,k},$  $t\geq 0, $ 
 is an $1/\alpha$-stable subordinator. By definition we have
 \begin{align*}
 \mathbf{Q}(\mathcal{B}_{i})\geq \mathbf{Q}\Big(\exists k\in\mathbb{N}: ~ \ell_{i,k}\leq r^{2+\eta} ~, ~ R^{i}_{\ell_{i,k}-}\in[0,r^{2\alpha}/{2}) ~,~ R^{i}_{\ell_{i,k}}\in[2r^{2\alpha}/3,r^{2\alpha})~,~ \min \widetilde{Z}^{i,k}\leq -4 r\Big). 
 \end{align*}
 Since  $R^{i}$ is non-decreasing,  we infer from the discussion above that the right-hand side of the previous display equals 
\begin{align*}
\mathbf{Q}\Big(\sum \limits_{k\in \mathbb{N}}& \mathbbm{1}_{\ell_{i,k}\leq r^{2+\eta}, R^{i}_{\ell_{i,k}-}\in[0,r^{2\alpha}/2), R^{i}_{\ell_{i,k}}\in[2r^{2\alpha}/3,r^{2\alpha})}\cdot \mathbf{N}^{(\Delta R^{i}_{\ell_{i,k}})}(\min Z<-4 r)  \Big)\\
&=\frac{1}{|\Gamma(-\frac{1}{\alpha})|}\int_0^{r^{2+\eta}} \d \ell~ \mathbf{Q}\Big( \mathbbm{1}_{ R^{i}_{\ell-}\in[0,r^{2\alpha}/2)}\int_{2 r^{2\alpha}/3-R_{\ell-}^i}^{ r^{2\alpha}-R_{\ell-}^i} \d v ~v^{-\frac{1}{\alpha}-1}  \cdot \mathbf{N}^{(v)}(\min Z<-4 r)\Big),
\end{align*}
where  $ \Delta R^{i}_{\ell_{i,k}} $ stands for the jump of $R^{i}$ at $\ell_{i,k}$, and where the second line follows by an application of the compensation formula. Using again scaling invariance, we derive that $\mathbf{N}^{(v)}(\min Z<-4 r)\geq \mathbf{P}(\min Z<-24)$, for every $v\geq r^{2\alpha}/6$. Therefore,  a straightforward computation, using again that $R^i$ is non-decreasing, entails that the previous display is bounded below by:
$$\frac{ \mathbf{P}(\min Z<-24)}{2 |\Gamma(-\frac{1}{\alpha})| r^{2}}\int_0^{r^{2+\eta}} \d \ell~ \mathbf{Q}\Big( \mathbbm{1}_{ R^{i}_{\ell-}\in[0,r^{2\alpha}/2)} \Big) \geq \frac{ \mathbf{P}(\min Z<- 24) \mathbf{Q}(R^{i}_{r^{2+\eta}}<r^{2\alpha})}{2 |\Gamma(-\frac{1}{\alpha})| }\cdot r^{\eta}.$$
Since $R^i$ is an $1/\alpha$-stable subordinator, we have $\mathbf{Q}(R^{i}_{r^{2+\eta}}<r^{2\alpha})=\mathbf{Q}(R^{i}_{1}<r^{-\alpha \eta})>\mathbf{Q}(R^{i}_{1}<1)$ which is a positive quantity. Furthermore, it follows from the Gaussian character of $Z$ that  $\mathbf{P}(\min Z<-24) >0$. We derive  that there exists a constant $a>0$ (not depending on $r\in(0,1)$) such that $\mathbf{Q}(\mathcal{B}_{i})\geq a \cdot r^{\eta}$. Finally, to conclude the proof of the theorem, we  consider the event
$$ \mathcal{A} := \left\{|\widetilde{Z}_{t}-\widetilde{Z}_{s}|\leq d(s,t)^{\frac{1}{2}- \frac{\eta}{10}}: ~~ \mbox{ for every }s,t\in[0,2] \mbox{ with }|s-t|\leq  r\right\}.$$
Proposition \ref{variations_Z} $\rm(ii)$ and excursion theory imply that  $\mathbf{Q}( \mathcal{A}^{\text{c}}) \leq C_4\cdot \exp(- r^{-c_4})$ for some $c_4, C_4>0$ (which only depend on $\eta$). Notice now that by construction, for every $i,k\in \mathbb{N}$ such that $\ell_{i,k}\leq r^{2+\eta}$, we must have $d(S_i,s_{i,k})\leq  r^{2+\eta}$. Thus, on the event $ \mathcal{A}  \cap \mathcal{B}_{i}$,  the minimal $\widetilde{Z}$ label over the time interval $[S_i,T_i]$ is smaller than $ \widetilde{Z}_{S_i} - 4r +   \left(r^{2+\eta}\right)^{ \frac{1}{2}-\frac{\eta}{10}} \leq - 2 r$, which implies  that $\widetilde{\tau}_r \leq T_i$.  We infer that:
\begin{align*}
\mathbf{Q}\big( S_{\lfloor r^{-\delta}\rfloor} \leq 2 \wedge \widetilde{\tau}_{ - r}\big) &\leq \mathbf{Q}\big(  \mathcal{A}^{\mathrm{c}})+ \mathbf{Q}\big(\cap_{i\leq\lfloor r^{-\delta}\rfloor } \mathcal{B}_{i}^{\mathrm{c}}\big)\leq \exp(- r^{-c_4}) + (1- a \cdot r^{\eta})^{\lfloor r^{-\delta}\rfloor} 
\end{align*}
where to obtain the second inequality we used the independence of the $\mathcal{B}_{i}$'s and \eqref{eq:graftedtotetgrand}.  Putting everything together, assuming $\eta\in (0,\delta)$ and using \eqref{eq:tropdetemps},  we deduce that we can find two constants $c,C>0$ (only depending on $\eta$), such that we have
$$\mathbf{Q}\Big(\int_{0}^{2\wedge \widetilde{\tau}_{ r}}\d s~\mathbbm{1}_{\widetilde{Z}_{s}\leq  r}\geq  r^{(2\alpha-\delta)}\Big) \leq C_4\cdot \exp(- r^{-c_4}) + (1- a \cdot r^{\eta})^{\lfloor r^{-\delta}\rfloor} \leq C\cdot \exp(-  r^{-c}).$$  Since the previous display holds for every $r\in (0,1)$, this concludes  the proof of the theorem. 
\end{proof}

\subsection{Local minima of $Z$ along branches of $ \mathcal{L}$}\label{sec:label:pinch}
In this section, we refine  Proposition \ref{pinch_points_are_not_record} by studying the local minima of  $Z$ along ``branches'' of the looptree $ \mathcal{L}$. This analysis  culminates in Proposition~\ref{lem:non-icnreasealongbranches} below (see Figure \ref{fig:non-icnreasealongbranches} for an illustration). In this direction, recall the definition of $\mathrm{Branch}(s,t)$, for $s,t\in [0,1]$, given in \eqref{eq:defpinch}   as well as its simpler form \eqref{eq:defpinch:0:t}  when either $s$ or $t$ equals $0$.

\begin{prop}[Local minimum of labels along branches] \label{lem:non-icnreasealongbranches} The following property holds $ \mathbf{P}$-almost surely:
 For any $0<t_{1}<t_{2}<1$ with $d(t_{1},t_{2})=0$  such that 
 \begin{equation}\label{prop:minimum:labels:branches}
Z_{t_{1}}=Z_{t_{2}} < Z_{r}\:\:\text{for every}\:\:r\in \mathrm{Branch}(0,t_{1})\setminus\{t_{1}\},
\end{equation}
and  every $\varepsilon>0$, there exist $t^{\prime}_{1}\in \big((t_{1}-\varepsilon)\vee 0, t_{1}\big)$ and $t^{\prime}_{2}\in \big(t_{2},(t_{2}+ \varepsilon)\wedge 1\big)$  pinch point times such that:
\[Z_{t^{\prime}_{1}}<Z_{t_{1}}=Z_{t_{2}}< Z_{r}\:\:\text{for every}\:\:r \in \mathrm{Branch}(t_{1},t^{\prime}_{1})\setminus \{t_{1},t^{\prime}_{1}\}\]
and
\[Z_{t^{\prime}_{2}}<Z_{t_{1}}=Z_{t_{2}}<Z_{r}\:\:\text{for every}\:\:r\in \mathrm{Branch}(t_{2},t^{\prime}_{2})\setminus \{t_{2},t^{\prime}_{2}\}.\]
\end{prop}
\begin{figure}[!h]
 \begin{center}
 \includegraphics[width=10cm]{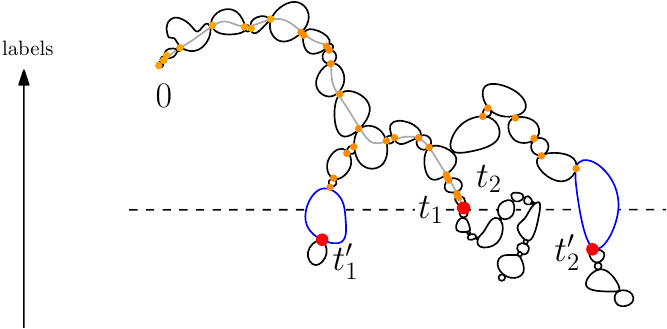}
 \caption{Illustration of the result of Proposition~\ref{lem:non-icnreasealongbranches}: Consider a branch in $ \mathcal{L}$ and stop it at a pinch point $\Pi_{d}(t_{1})$ when reaching a new minimum for the label along that branch. Then we can find two blue loops on both sides ``blocking'' the level $Z_{t_{1}}$. \label{fig:non-icnreasealongbranches}}
 \end{center}
 \end{figure}
 Let us describe first informally the main observations of this section before setting some notation which is necessary for the proofs. For this informal discussion, we argue under $\mathbf{P}$. If $s,t\in [0,1]$, by analogy with the stable tree, we refer to $\Pi_d( \mathrm{Branch}(s,t))$ as the branch between $x=\Pi_d(s)$ and $y=\Pi_d(t)$ in $ \mathcal{L}$. Let us stress that this notion is consistent with the definition of $\mathcal{L}$ since as explained in Section \ref{sec:defloop}, the set $\Pi_d( \mathrm{Branch}(s,t))$ corresponds to the points in $\mathcal{L}$ separating $x$ and $y$, and thus does not depend on the  choice of representatives $s$ and $t$. It was already observed in \cite[Lemma 16 (ii)]{LGM09} that the process $Z$ over a typical branch of $ \mathcal{L}$ evolves as a symmetric L\'evy process with index $ 2 \alpha-2$, see \eqref{one_point} below. In particular, the range\footnote{ In this work we use the standard convention that the range of a function is the closure of its image.} of $Z$ along a typical branch has Hausdorff dimension $1 \wedge (2\alpha-2)$. Notice that when $\alpha \in (1, 3/2)$,  this dimension is strictly less than $1$ and so the intersection  of $\lceil \frac{1}{1- (2\alpha -2)}\rceil +1 $ ``independent'' such closed  subsets is reduced to $\varnothing$. This heuristic suggests that: \begin{center}
the process $Z$ cannot take more than $\lceil \frac{1}{1- (2\alpha -2)}\rceil +1$ times the same value on the skeleton of $ \mathcal{L}$.\end{center}
Although we do not prove it,  this fact together with Proposition \ref{pinch_points_are_not_record} actually yield the above Proposition~\ref{lem:non-icnreasealongbranches}. In general, when $\alpha \in (1,2)$ we rather look more precisely at the  set of record values along branches. Indeed, when going along a branch in a given direction, since the process of labels evolves as a symmetric $2\alpha-2$ stable process, its running infimum evolves as an $\alpha-1$ stable subordinator, so that its range is of Hausdorff dimension $\alpha-1$. The same argument as above shows that the intersection of more than $\lceil \frac{1}{2- \alpha}\rceil +1$ such sets is reduced to $\varnothing$. This suggests as above that: 
\begin{center}
the running infimum process of $Z$ along branches cannot take more than $\lceil \frac{1}{2- \alpha}\rceil +1$ times\\ the same value on the skeleton of $ \mathcal{L}$.\end{center}
In order to make this statement precise, which is the key ingredient in the proof of Proposition \ref{lem:non-icnreasealongbranches}, let us introduce some notation and intermediate results. In this direction, recall the construction of the height process $H$ given in Section \ref{sec:codagearbre} and in particular the notation $\xi_{t}(r)$ which gives a pre-image of the ancestor of $\Pi_{ \mathfrak{h}}(t)$ at height $r \leq H_{t}$ in the stable tree $ \mathcal{T}_{ \mathfrak{h}}$. Now remark that thanks to the continuity of $Z$ and Lemma  \ref{Lem-cut-jump}, we get:
\begin{cor} \label{cor:labelbranch}
The following holds $\mathbf{P}$-a.s., for every $t\in [0,1]$ and $0\leq r_1\leq r_2\leq H_t$, we have:
\begin{equation}\label{iden:inf} 
\big\{Z_s:~\xi_{t}(r_1)\preceq s\preceq \xi_{t}(r_2)\big\}=\big\{Z_{\xi_{t}(r)}:~r\in[r_1,r_2]\big\} \cup \big\{Z_{\xi_{t}(r-)}:~r\in[r_1,r_2]\big\}.
\end{equation}
\end{cor}
In particular,  even though $\xi_t$ does not parametrize $\mathrm{Branch}(0,t)$, the (closed) range of the process $r\mapsto Z_{\xi_t(r)}$ covers all the labels along  $\mathrm{Branch}(0,t)$.  To simplify some arguments, recall that $H$ and $\xi$ can also be defined under the measure $\mathbf{N}$ by scaling. Then \cite[Lemma 16 (ii)]{LGM09} states that for every  bounded continuous function, $F:\mathbb{D}( \mathbb{R}_{+},\mathbb{R})\mapsto\mathbb{R}$, we have
\begin{equation}\label{one_point}
    \mathbf{N}\Big(\int_{0}^{\sigma} \d t F\big((Z_{\xi_{t}(r)})_{r\geq 0}\big)\Big)=\int_{0}^{\infty} \d h\: \mathbf{E}\big[F((Y_{r\wedge h})_{r\geq 0})\big],
\end{equation}
where $Y$ is a symmetric stable processes with L\'evy-Khintchine function\footnote{Lemma 16 (ii) in \cite{LGM09} does not provide the explicit normalization constant $\alpha 2^{1-\alpha}\frac{\Gamma(\alpha)^{2}}{\Gamma(2\alpha)}$. The exact value is not essential for our study. However, we include it since it can be obtained by a straightforward integral computation.} $\lambda\mapsto \alpha 2^{1-\alpha}\frac{\Gamma(\alpha)^{2}}{\Gamma(2\alpha)}|\lambda|^{2(\alpha-1)}$.  Equation \eqref{one_point} has the following direct consequence for labels along branches:

\begin{lem}\label{lem:onePinch} Under $\mathbf{P}$, let $U_1 \in [0,1]$ be uniform and independent of $(X,Z)$. Almost surely, for any $[a,b] \subset [0,H_{U_1}]$, the infimum of $ s \mapsto Z_{\xi_{U_{1}}(s)}$ over $[a,b]$ is attained. Furthermore, its two-sided minima are distinct, so that the process attains its overall minimum only once.
\end{lem}

\begin{proof} By a standard result for symmetric stable Lévy processes, it is well known  that, for any interval $[a,b]\subset \mathbb{R}_+$, the process $Y|_{[a,b]}$ attains its infimum. Moreover, all the local minima of $Y$ are distinct. The statement of the lemma is then a direct consequence of \eqref{one_point} and the scaling property.
\end{proof}
Let us also mention that the above statement also holds true simultaneously for all pinch point times $t$ instead of the uniform time $U_{1}$, simply because for every pinch points time $t$, there is a positive probability that $t\prec U_1$, and this implies that $(\xi_{t}(r))_{r\geq 0}=(\xi_{U_1}(r\wedge H_{t}))_{r\geq 0}$.
  
We can now introduce, under $\mathbf{P}$, the necessary notation to study the intersections of the range of the running infimum of $Z$ along branches. We refer to Figure \ref{fig:branches} for an illustration. Let $m\geq 1$ and  $0<t_{1}<t_{2}<\cdots<t_{m}<1$. For $1 \leq i \leq m$, we set
$$r_{i} := \sup \limits_{j\neq i} (t_{i}\curlywedge t_{j}),$$ the largest ancestor of $t_{i}$ with another $t_{j}$. Remark that for every $1\leq i\leq m$, the set   $\mathrm{Branch}\big(r_{i},t_{i}\big)\setminus \{r_{i}\}$ is the set of $t\in \mathrm{Branch}(0,t_{i})$ such that, for every $j\neq i$, $t$ is not an ancestor of $t_{j}$. In particular, these sets (and their projections by $\Pi_{d}$) are disjoint.  For every $1\leq i\leq m$, introduce the rcll processes $R_{t_{1},..., t_{m}}^{(i)}$ defined by:
\[ R_{t_{1},..., t_{m}}^{(i)}(r):=Z_{\xi_{t_{i}}(H_{r_{i}}+r)}~~;\quad \text{for } r\geq 0.\]
 In words, by \eqref{iden:inf},   the range of the  process  $R_{t_{1},..., t_{m}}^{(i)}$ is encoding the  labels $Z_{r}$ for $r\in \mathrm{Branch}\big(r_{i},t_{i}\big)\setminus \{r_{i}\}$~, when going from  $r_{i}$ to $t_{i}$. We denote  the minimal record values of $R_{t_{1},..., t_{m}}^{(i)}$ by $$\mathcal{R}_{t_{1},..., t_{m}}^{(i)}:=\Big\{\inf\limits_{ [0,r]} R_{t_{1},..., t_{m}}^{(i)}(r):~r\geq 0\Big\},$$ where we recall from Lemma \ref{lem:onePinch} that these infima are  attained, and consequently $\mathcal{R}_{t_{1},..., t_{m}}^{(i)}$ is  a closed set.  
 \begin{figure}[!h]
\begin{center}
\includegraphics[width=13cm]{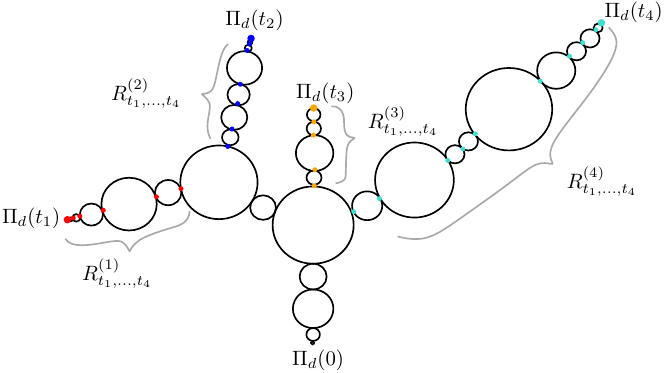}
\caption{Illustration of $\big(R_{t_{1},..., t_{m}}^{(i)}\big)_{1\leq i\leq m}$  in the case $m=4$. \label{fig:branches}}
\end{center}
\end{figure}
 
 We are going to deduce Proposition \ref{lem:non-icnreasealongbranches}  from the following result, which formalizes the above discussion:
 
\begin{prop}\label{intersection}
Fix $m \geq  \lceil \frac{1}{2-\alpha}\rceil+1$. Under $ \mathbf{P}$, let $U_{1}, ... , U_{m}$ be independent uniform random variables on $[0,1]$, also independent of $(X,Z)$. Then we have:
\[\bigcap\limits_{i=1}^{m} \mathcal{R}_{U_{1},...,U_{m}}^{(i)}=\varnothing, \quad \mathbf{P}-\mbox{a.s.}.\]
\end{prop}
Notice that the proposition does not hold simultaneously for all times $t_{1}, ... , t_{m}$ because $Z$ does attain some common values at infinitely many (leaf) times. However, by the same argument as the one appearing after Lemma \ref{lem:onePinch}, it holds that $\mathbf{P}$ almost surely, for all pinch points  times $t_{1}, ... , t_{m} \in \mathcal{L}$ such that  $t_i\npreceq t_j$ for every $i\neq j$, we have $\bigcap\limits_{1\leq i\leq m} \mathcal{R}_{t_{1},...,t_{m}}^{(i)}=\varnothing$. Indeed, as already mentioned above,  for every pinch point times  $t_{1}, ... , t_{m}$, there is a positive probability that $t_i\prec U_i$, for every $1\leq i\leq m$.
\begin{proof} 
We work again under $\mathbf{N}$, since by a scaling argument it is equivalent to prove the proposition under $\mathbf{P}$ or $\mathbf{N}$, and we fix  $m_{0}:=\lceil \frac{1}{2-\alpha}\rceil+1$. For technical reasons it will be useful to consider a  homogeneous  Poisson point process on $\mathbb{R}_+$ with intensity $1$ defined under the measure $\mathbf{N}$, independently of $X$ and $Z$. We write $0<\tau_{1}<\tau_{2}< \cdots$ for the jumping times of this Poisson point process and set $M:=\#\{i:~\tau_{i}\leq \sigma\}$. In particular, remark that:
\[\mathbf{N}\big(M\geq 1\big)=\mathbf{N}\big(1-\exp(-\sigma)\big) \underset{ \eqref{eq:excursionmeasuredecomp}}{=}1,\]
and we get that the measure $\overline{\mathbf{N}}:=\mathbf{N}(\cdot\:|\:M\geq 1)$ is a probability measure. By comparing $ \mathbf{N}$ and $\overline{\mathbf{N}}$,  to prove the proposition it suffices to show that we have 

\begin{equation}\label{eq:ancienne:(*):4}
M<m_{0}\:\:\:\: \text{ or }\:\:\:\:\bigcap\limits_{i=1}^{M} \mathcal{R}_{\tau_{1},...,\tau_{M}}^{(i)}=\varnothing~~,\quad\overline{\mathbf{N}}\text{-a.s.}
\end{equation}
Let us prove \eqref{eq:ancienne:(*):4}, reasoning on the event $M\geq m_0$, so in particular $m\geq 2$. As we noticed after Proposition \ref{topologie_loop_tree}, for every $i\leq M$, the ancestor time $$r_{i}:=\sup\limits_{j\neq i} (\tau_{i}\curlywedge \tau_{j})$$ is a.s.~a jumping time for $X$, corresponding to a loop in $\mathcal{L}$. Let us introduce the looptree dangling from this loop and containing $\Pi_d(\tau_i)$. In this direction, we set $s_{i}:=\inf \{r\geq r_{i}:~X_{r}\leq I_{r,\tau_{i}}\}\:\:\text{and}\:\:s_{i}^{\prime}:=\inf \{r> s_{i}:~X_{r}\leq X_{s_{i}}\}.$ Remark now that by {\hypersetup{linkcolor=black}\hyperlink{prop:A:2}{$(A_2)$}}   we have $s_{i}<\tau_{i}<s_{i}^{\prime}$ and  $\Delta_{s_{i}}=0$ and, to simplify notation, we write $X^{(i)}$ for the excursion
\[X^{(i)}_{r}:=X_{s_{i}+r}-X_{s_{i}}~~;\quad r\in [0,s^{\prime}_{i}-s_{i}].\] 
We can use the excursion $X^{(i)}$ to define a looptree exactly in the same way as in Section \ref{sec:defloop}, which is in fact the same as the set $\{u\in \mathcal{L}:~\Pi_{d}(s_{i})\preceq u\}$ equipped with the restriction of the distance $d$, and rooted at $\Pi_{d}(s_{i})=\Pi_{d}(s_i^\prime)$. It is then easy to see that the process $R_{\tau_{1},...,\tau_{M}}^{(i)}-R_{\tau_{1},...,\tau_{M}}^{(i)}(0)$ only depends on $X^{(i)}$ and on the Brownian bridges attached to its jumps.  
The idea now is to prove that under the event $M\geq m_{0}$~, the sequence $(X^{(i)},\tau_{i}-s_{i})_{1\leq i\leq m_{0}}$ is absolutely continuous  with respect to a sequence of $m_{0}$ independent copies of $(X,\tau_{1})$ under $\mathbf{N}(\cdot\:|\:M=1)$. Let us assume this temporarily, and explain why it implies the statement of the  proposition. By independence of the Brownian bridges conditionally on the looptree, we can then apply 
equation \eqref{one_point} to see that  under the event $M\geq m_{0}$, the law of the sequence $(\mathcal{R}_{\tau_{1},...\tau_{M}}^{(i)}-R_{\tau_{1},...\tau_{M}}^{(i)}(0))_{1\leq i\leq m_{0}}$ is absolutely continuous  with respect to that of a sequence of $m_{0}$ independent copies of the range of an $(\alpha-1)$-stable subordinator (each of these copies being stopped at a random time). 
Moreover, conditionally on $X$, on $M\geq m_{0}$ and on the times $\tau_{1},...,\tau_{M}$, the random variables $(R_{\tau_{1},...\tau_{M}}^{(i)}-R_{\tau_{1},...\tau_{M}}^{(i)}(0))_{1\leq i\leq m_{0}}$ and $(R_{\tau_{1},...\tau_{M}}^{(i)}(0))_{1\leq i\leq m_{0}}$ are independent. In order to see this, just remark that under this conditioning, all the random variables  $(R_{\tau_{1},...\tau_{M}}^{(i)}-R_{\tau_{1},...\tau_{M}}^{(i)}(0))_{1\leq i\leq m_{0}}$ and $(R_{\tau_{1},...\tau_{M}}^{(i)}(0))_{1\leq i\leq m_{0}}$ are  Gaussian, and a direct computation of the covariance function gives the desired independence. 
Now remark that by the Gaussian character of $Z$ conditionally on $X$,  the quantities $R_{\tau_{1},...\tau_{M}}^{(i)}(0)$,  for $1\leq i\leq m_{0}$ are a.s.\ all different. It follows that, when we condition on $M\geq m_{0}$, the law of the intersection in \eqref{eq:ancienne:(*):4} is absolutely continuous with respect to the law of the intersection of pieces of $m_{0}$ independent ranges of $(\alpha-1)$-stable subordinators, started from independent random points with different locations. It is then a classical result of regenerative sets theory, see \cite[Example 1]{hawkes1977intersections}, that since $m_{0} = \lceil \frac{1}{2-\alpha}\rceil+1$  this intersection is almost surely empty.

Hence, to conclude, it remains to  show that, on the event $M\geq m_{0}$, the sequence $(X^{(i)},\tau_{i}-s_{i})_{1\leq i\leq m_0}$ has a law that is absolutely continuous  with respect to that of a sequence of $m_{0}$ independent copies of $(X,\tau_{1})$ under $\mathbf{N}(\cdot\:|\:M=1)$. To this end, under $\mathbf{N}$ and independently of $M$, we introduce $(\widetilde{X}^{(i)},\widetilde{T}^{(i)})_{i\geq 1}$ an i.i.d. sequence distributed as $(X,\tau_{1})$ under $\mathbf{N}(\cdot\:|\:M=1)$~, and we set:
\[
(\overline{X}^{(i)},\overline{T}^{(i)}):=\begin{cases}
(X^{(i)},\tau_{i}-s_{i} )~& \text{if}~~ 1\leq i\leq M \\
(\widetilde{X}^{(i)},\widetilde{T}^{(i)})~& \text{if}~~ i>M.
  \end{cases}
\]
Remark now that it suffices  to show that:
\begin{equation}\label{eq:**:section4.3}
(\overline{X}^{(i)},\overline{T}^{(i)})_{i\geq 1}\overset{(d)}{=}(\widetilde{X}^{(i)},\widetilde{T}^{(i)})_{i\geq 1}~.
\end{equation}
On the event  $M=1$, the identity \eqref{eq:**:section4.3} holds directly from the definition. On the event $M\geq 2$, define $\ell=\tau_{1}\curlywedge \tau_{2} \curlywedge \dots \curlywedge \tau_{M}$, and denote by $(u_{j},v_{j})_{j\geq 1}$ the connected components of the open set $\{s\geq \ell:~I_{\ell,s}<X_{s}\}$. Next, introduce  the excursions:
\[X^{j}_{r}:=X_{u_{j}+r}-X_{u_{j}}~;\quad r\in [0,v_{j}-u_{j}],\] 
for every $j\geq 1$. In words, $(X^{j})_{j\geq 1}$ are the excursions of $X$ above the minimum after time $\ell$. We also introduce $\zeta:=\#\{j\in \mathbb{N}:\: \exists i\in[\![1,M]\!],\: \tau_{i}\in(u_{j},v_{j})\}$ the number of such excursions having at least one Poissonnian mark  $\tau_{1},..., \tau_{M}$. Write $j_{1},...,j_{\zeta}$ for the elements of $\{j\in  \mathbb{N}:\: \exists i\in[\![1,M]\!],\: \tau_{i}\in(u_{j},v_{j})\}$ in increasing order. Then classical properties of excursion theory and Poisson measures give that, conditionally on $\zeta$, the random variables $(X^{j_{k}})_{1\leq k\leq \zeta}$ with their Poissonnian marks (translated by $(-u_{j_{k}})_{1\leq k\leq \zeta}$)  are i.i.d. copies of  the process $X$ marked at $\tau_{1},...,\tau_{M}$ under $\overline{\mathbf{N}}$. Remark that $\ell$ is not a stopping time, so to obtain this property an approximation procedure  is needed. Since this a standard verification, for example cutting at dyadic times and using that $X$ is rcll, we leave the details to the reader. Finally since   $M<\infty$,  $\overline{\mathbf{N}}$-a.s~.,  the identity \eqref{eq:**:section4.3} follows by iterating the previous argument.
\end{proof}

We can now use   Propositions \ref{pinch_points_are_not_record} and \ref{intersection} to derive our Proposition~\ref{lem:non-icnreasealongbranches}:

\begin{proof}[Proof of Proposition~\ref{lem:non-icnreasealongbranches}]  We argue under $\mathbf{P}$. Fix $0<t_1<t_2<1$ such that $t_1\sim_d t_2$ and verifying \eqref{prop:minimum:labels:branches}. To simplify notation, we set  $z= Z_{t_{1}}= Z_{t_{2}}$. Note that by time-reversal \eqref{symmetric},  it suffices to show 
 that for every $\varepsilon\in(0, t_1)$, we can find $t_{1}^{\prime}\in[t_1-\varepsilon,t_1]$ such that
 \[Z_{t^{\prime}_{1}}<z< Z_{r}\:\:\text{for every}\:\:r \in \mathrm{Branch}(t_{1},t^{\prime}_{1})\setminus \{t_{1},t^{\prime}_{1}\}.\]  In this direction,  fix $\varepsilon\in(0,t_1)$, and  let 
$$  \mathrm{Trunk} := \bigcup_{ \begin{subarray}{c}  s \preceq t_{1}\\
\Delta_{s} >0\end{subarray}} \left\{ \mathrm{f}_{s}(u) : 0 \leq u \leq 1, \mbox{ such that } \mathrm{f}_{s}(u) \leq t_{1}\right\},$$
which informally corresponds (after taking a closure and projecting it to $ \mathcal{L}$) to the left side of the trunk linking $\Pi_{d}(0)$ to $\Pi_{d}(t_1)$. For $\ell \in [0,t_{1}]$, we set $ \mathrm{r}_{\ell}:=\sup\{s\in \mathrm{Trunk}:~s\preceq \ell\}$. By our assumption  \eqref{prop:minimum:labels:branches}, we know that $Z_{s} > z$ as soon as $s \prec t_{1}$. If  there exists $s\in \mathrm{Trunk} \cap [t_{1}- \varepsilon, t_{1}]$ such that $Z_s<z$ then, since  for every jump time $t$ of $X$ the pinch point times of $\mathrm{f}_t([0,1])$ are dense in  $\mathrm{f}_t([0,1])$,  without loss of generality we can assume that $s$ is a pinch point time and we infer from the definition of $\mathrm{Trunk}$ that we can directly take $t_1^{\prime}=s$.  It remains to treat the case $Z \geq z$ on $\mathrm{Trunk} \cap [t_{1}- \varepsilon, t_{1}]$, and we make this assumption for the rest of the proof. 

\begin{figure}[!h]
 \begin{center}
 \includegraphics[width=16cm]{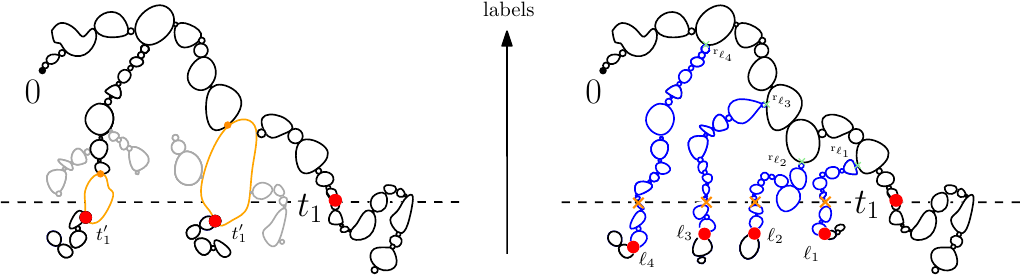}
 \caption{ \label{fig:illustrationintersection} Illustration of the proof of Proposition~\ref{lem:non-icnreasealongbranches}. Left: finding a time $t_{1}'$ corresponds to finding a loop (in orange on the figure) blocking level $z = Z_{t_{1}}$ i.e.\ such that the label process along this loop takes values that are greater than $z$, and values that are less than $z$. Right: if such blocking loops cannot be found on the trunk from $t_{1}$ to $0$, then by Proposition \ref{pinch_points_are_not_record} we can find $m_{0}$ distinct sub-looptrees (in blue in the figure) going from a level strictly above $z$ to strictly below $z$. By Proposition \ref{intersection}, one of them must possess such a blocking loop. }
 \end{center}
 \end{figure}
 
 Now, recall from Proposition \ref{pinch_points_are_not_record} that the time $t_{1}$ cannot be a left minimal record of the process $Z$. This combined with the fact that pinch point times are dense,    implies that for $m_{0} = \lceil \frac{1}{2-\alpha}\rceil +1$, we can find $t_{1}-\eps<\ell_{1}< \cdots <\ell_{m_0}<t_{1}$ such that:
\begin{itemize}
\item $\ell_{1}, ... , \ell_{m_{0}}$ are pinch point times,
\item $ Z_{\ell_{i}} < z$,
\item the values $ \mathrm{r}_{\ell_{1}}, \mathrm{r}_{\ell_{2}}, \dots , \mathrm{r}_{\ell_{m_{0}}}$ are all distinct and are elements of $\mathrm{Trunk} \cap [t_{1}- \varepsilon, t_{1}]$.
 \end{itemize}
One can then consider the processes  $R_{ \ell_{1},... , \ell_{m_{0}}}^{(i)}$, $1 \leq i \leq m_{0}$, as above  and the associated   minimal records $ \mathcal{R}^{{(i)}}_{ \ell_{1},... , \ell_{m_{0}}}$, $1 \leq i \leq m_{0}$, must all cross the value $z$ by construction. If such a crossing is made by a jump, we have found a $t_{1}'$ as desired. Otherwise, the value $z$ is common to the $m_{0}$ closed sets $\mathcal{R}^{{(i)}}_{ \ell_{1},... , \ell_{m_{0}}}$,  $1 \leq i \leq m_{0}$, which is impossible by  Proposition  \ref{intersection} and the remark just after it.\end{proof}

\section{The two-point function and local minima of $Z$ on loops  } \label{sec:2pt}
We continue our investigation of local minima of $Z$ and prove in this section that, under $ \mathbf{P}$, we have:

\begin{enumerate}[(i)]
\item[-] the sets $\Pi_d( \mathrm{LeftRec}) \cap \mathrm{Loops}$ and $\Pi_d( \mathrm{RightRec}) \cap \mathrm{Loops}$ are dense in $ \mathrm{Loops}$,
\item[-] if $\alpha \in [ \frac{3}{2},2)$ the set of local right minimal values of $Z$ that happen on a loop and the set of local left minimal values of $Z$ that happen on a loop are disjoint, whereas if $ \alpha \in (1, \frac{3}{2})$ they almost surely intersect.
\end{enumerate}
See Proposition \ref{prop:records-loops} for a precise statement. These properties will play a crucial role when identifying the topology of $ \mathcal{S}$ as being that of the Sierpinski carpet in the dilute phase and when studying the behavior of simple geodesics (namely the fact that they bounce on the faces of $ \mathcal{S}$). The proof is based on a covering argument due to Shepp \cite{FiFrSh85} which has already been used in the study of increase points of L\'evy processes by Bertoin \cite{bertoin1994increase} and later in the study of random maps \cite[Lemma 29]{Bet11} or \cite[Lemma 14]{CMMinfini}. It crucially relies on the computation of the constant $\mathbf{N}(\inf Z\leq -1)$ which appeared in Corollary~\ref{Z<-1}, specifically we prove  that $
\bN(\sup Z> 1)=\frac{\alpha(\alpha-1)}{2}$ in Proposition \ref{sec:stable-map-3}. This computation is based on an integral equation satisfied by the function defined for every  $(\lambda,x)\in \mathbb{R}_{+}\times \mathbb{R}$ by
\begin{equation}\label{eq:N:two:points}
w_{\lambda}(x)=\mathbf{N}\big(1-\exp(-\lambda \sigma)\mathbbm{1}_{\sup Z<x}\big).
\end{equation}

This analysis is of independent interest and in particular,   \eqref{eq:N:two:points} characterizes the joint law of $(\sup Z,\sigma)$ under $\mathbf{N}$, which can be related to the distance in $\mathcal{S}$ between two typical points (see Proposition \ref{theo:sub} and equation \eqref{Distance*:rho_*} below). For this reason we refer to  \eqref{eq:N:two:points} as the \textbf{two-point function}.

\subsection{An equation for the two-point function}
\label{sec:twopoints}
We begin by studying  the function \eqref{eq:N:two:points}. We stress that when $x\leq 0$, we simply have $w_{\lambda}(x)=\infty$. Also note that,  when $\lambda=0$ , we have $w_0(x)=\mathbf{N}(\sup Z>x)$ and, by dominated convergence and \eqref{eq:excursionmeasuredecomp}, we obtain
$$\lim \limits_{x\to \infty}w_{\lambda}(x)=\mathbf{N}\big(1-\exp(-\lambda \sigma)\big)=\lambda^{1/\alpha}.$$ The goal of this section is to prove that $w_{\lambda}(x)$ satisfies an integral equation that allows us to compute  $w_{0}(x)$.  Since we will be dealing with  several variants of Brownian motions and Bessel processes, in the next two sections, it will be convenient  to denote  the canonical process on $\mathcal{C}(\R_+,\R)$ by $\big(B_{t}:~t\geq 0\big)$,  and we will endow this  space with various probability measures. Namely, for every $x\in \R$ (resp.\ $x\in \mathbb{R}_+^*$), the probability measure $\P_x$ (resp.\ $\P^{\langle\nu\rangle}_x$) is the law of a Brownian motion (resp.\ Bessel process with index\footnote{There are several equivalent ways to parametrize Bessel processes, according to their parameter $a$, their index $\nu$, or their dimension $d$. These are related via the formulas $d=2a+1$ and $\nu = a- \frac{1}{2}$ and they are (locally) solutions to the stochastic differential equation 
$$ \mathrm{d}X_{t} = \frac{a}{ X_{t}} \mathrm{d}t +  \mathrm{d}W_{t},$$ where $W$ is a standard Brownian motion. See \cite[Chapter XI]{RY99}.} $\nu$) started from $x$. We also write $\P_{x}^{(s)}$ for the law of a Brownian motion with lifetime $s>0$ starting from $x$, as well as  $\P_{x\to y}^{(s)}$ for the law of a Brownian bridge with lifetime $s>0$ starting from $x$ and ending at $y$; where  by convention under $\P_{x\to y}^{(s)}$ we take $B_t=y$ for every $t\geq s$. Specifically, this section is devoted to the proof of the following result:
\begin{theo}[The two-point function]\label{two_points_function}
For every $(\lambda,x)\in \mathbb{R}_{+}\times \mathbb{R}_+^*$ we have
\begin{equation}
\label{eq:inteqw}
\int_{0}^{\infty}\frac{ \mathrm{d}s}{s^{\alpha+1}}\mathbb{E}_{0\rightarrow
  0}^{(s)}\Big[1-s w_{\lambda}(x)-\exp\Big(-\int_{0}^{s}\d
u\:w_{\lambda}(x+B_{u})\Big)\Big]=-\Gamma(-\alpha)\lambda,
\end{equation}
where the integral is well-defined. 
\end{theo}
This result is later used in the case $\lambda=0$ to prove Proposition \ref{sec:stable-map-3}. The previous theorem is reminiscent of the connection between the Brownian motion indexed by the Brownian tree -- or Brownian snake -- and differential equations, see \cite{LeG99}. In this case, the analog of $w_{\lambda}(x)$ was computed exactly in \cite[Lemma 6 \& 7]{delmas2003computation} in the Brownian case (which morally corresponds to the limit $\alpha \to 2$) by showing that  $x \mapsto \mathbf{N}\big(1-\exp(- \lambda \sigma) \mathbf{1}_{ \sup Z < x}\big)$ solves the differential equation:
$$ \frac{1}{2} w''(x) = 2 w^{2}(x)-\lambda \quad \mbox{ in }(0, \infty) \quad \mbox{ with } w(0) = \infty \mbox{ and }w(\infty) = \sqrt{\lambda/2}. $$
 A similar study of Brownian motion indexed by the stable tree (as opposed to the stable looptree here) has recently been done \cite{archer2024some}. The equation \eqref{eq:inteqw} can be understood informally by exploring the underlying looptree infinitesimally: under $ \mathbf{N}$ ``the first loop" encountered has a length $s$ ``distributed" according to the L\'evy measure \eqref{eq:levymeasure} given by $  \frac{ \mathrm{d}s}{\Gamma(-\alpha) s^{\alpha+1}}$. Then the Markov property (Corollary \ref{cutting_N}) states that the looptrees (with labels shifted by the values at their root) attached to this loop form a Poisson process with intensity $  \mathbbm{1}_{x \in [0,s]} \d x \mathbf{N}$. Conditionally on $s$, and neglecting the rest of the looptree, since the label of the loop is given by a  Brownian bridge of length $s$, an application of the exponential formula for Poisson process shows that conditionally on $s$ we have:
$$ \exp(-\lambda \sigma)\mathbbm{1}_{\sup Z<x} =   \mathbb{E}_{0\rightarrow
  0}^{(s)}\Big[\exp\Big(-\int_{0}^{s}\d
u\:w_{\lambda}(x+B_{u})\Big)\Big].$$  This explains the appearance of the corresponding term in \eqref{eq:inteqw}. The proof is however more subtle since delicate compensations are involved. 
Let us first gather some useful properties of $w_\lambda$. 
\begin{lem}\label{lem:basicpropw} Fix $\lambda\geq 0$. The following properties hold. 
\begin{itemize}
\item[$\mathrm{(i)}$] The function $x\mapsto w_{\lambda}(x)$ is in $\mathcal{C}^{\infty}((0,\infty),\R)$.
\item[$\mathrm{(ii)}$]  The function
\[x\longmapsto \int_{0}^{\infty}\frac{ \mathrm{d} s}{s^{\alpha+1}}\mathbb{E}_{0\rightarrow 0}^{(s)}\Big[1-s w_{\lambda}(x)-\exp\big(-\int_{0}^{s}\d u\:w_{\lambda}(x+B_{u})\big)\Big]\]
is well-defined on $(0,\infty)$, in the sense that the integral is
absolutely convergent, and continuous. 
\end{itemize}
\end{lem}
\begin{proof}
(i) The case $\lambda=0$ follows  directly from scaling, which gives $w_{0}(x)=x^{-2}\mathbf{N}(\sup Z>1)$, and the fact that $\mathbf{N}(\sup Z>1)\in (0,\infty)$ by Corollary \ref{Z<-1}. Next, for a fixed positive value of $x$, note that $\lambda\mapsto w_\lambda(x)$ is of class $\mathcal{C}^\infty$ on $(0,\infty)$, since by dominated convergence, we have:
\[\frac{\partial^{n} w_{\lambda}(x)}{\partial \lambda^{n}}=(-1)^{n+1}\mathbf{N}\big(\sigma^{n}\exp(-\lambda \sigma)\mathbbm{1}_{\sup Z <x}\big).\] 
Scaling properties under the measure $\mathbf{N}$ yield $w_{\lambda}(x)=x^{-2} w_{\lambda x^{2\alpha}}(1)$. As a consequence, we  conclude that $w_\lambda(x)$ is also of class $\mathcal{C}^\infty$ in the variable $x\in (0,\infty)$.

(ii) Fix $\lambda\geq 0$ and $x>0$, and observe that the integral expression
\[\int_{0}^{\infty}\frac{ \mathrm{d}s}{s^{1+\alpha}}\big(1-s w_{\lambda}(x)-\exp(-s w_{\lambda}(x))\big)\]
is absolutely convergent, and a direct computation gives that it equals $-\Gamma(-\alpha)w_{\lambda}(x)^{\alpha}$, so that it is
continuous in $x$. Therefore, to obtain (ii), it is enough to show that the integral 
\[\int_{0}^{\infty}\frac{\mathrm{d}s}{s^{1+\alpha}} \Big( \exp\big(-s w_{\lambda}(x)\big)-\mathbb{E}_{0\rightarrow 0}^{(s)} \big[\exp\big(-\int_{0}^{s}\d u\:w_{\lambda}(x+B_{u})\big)\big]\Big)\]
is also absolutely convergent and continuous in the variable $x$.  By dominated convergence, this will be entailed by proving
that, for every compact set $K\subset (0,\infty)$, we have:
\[\exp\big(-s w_{\lambda}(x)\big)- \mathbb{E}_{0\rightarrow
    0}^{(s)}\Big[\exp\big(-\int_{0}^{s}\d
  u\:w_{\lambda}(x+B_{u})\big)\Big]=O_K(1\wedge s^2)\, ,\] 
where this notation means that the left-hand side divided by $1\wedge
s^2$ is uniformly bounded in absolute value for $s>0$ and $x\in K$.
Since the left-hand side is clearly bounded by $1$ in absolute value,
it suffices to show that it is $O_K(s^2)$. 
To obtain this
bound, note that by scaling and  standard properties of Brownian
bridges, it follows that:
\[\mathbb{P}^{(s)}_{0\to 0}\Big(\sup|B|>
  \frac{x}{2}\Big)=\mathbb{P}^{(1)}_{0\to 0}\Big(\sup |B|>
  \frac{x}{2\sqrt{s}}\Big)\leq
  c_{1}\exp\Big(-c_{2}\frac{x^{2}}{s}\Big)=O_K(s^2),\]
for some universal constants $c_1,c_2>0$.
Therefore, we can write:
\begin{align*}
    &\exp\big(-s w_{\lambda}(x)\big)-\mathbb{E}_{0\rightarrow 0}^{(s)}\Big[\exp\big(-\int_{0}^{s}\d u\:w_{\lambda}(x+B_{u})\big)\Big]\\
    &= \exp\big(-s w_{\lambda}(x)\big)\:\:\mathbb{E}_{0\rightarrow 0}^{(s)}\Big[1-\exp\Big(-\int_{0}^{s}\d u\:\big(w_{\lambda}(x+B_{u})- w_{\lambda}(x)\big)\Big)\Big]\\
    &= \exp\big(-s w_{\lambda}(x)\big)\:\:\mathbb{E}_{0\rightarrow
      0}^{(s)}\Big[\Big(1-\exp\Big(-\int_{0}^{s}\d
      u\:\big(w_{\lambda}(x+B_{u})-
      w_{\lambda}(x)\big)\Big)\Big)\mathbbm{1}_{\sup
      |B|<\frac{x}{2}}\Big]+O_K(s^2)\, .
\end{align*}
By (i), we may Taylor expand $w_\lambda(x+B_u)-w_\lambda(x)=
w_\lambda'(x)B_u+S_\lambda(x,B_u)B_u^2$, where the remainder term
$S_\lambda(x,y)$ is a bounded function on $\{(x,y):x\in K,|y|\leq
x/2\}$. Using the fact that $1- \mathrm{e}^{-y}=y+O(y^2)$ for $y$ in a
compact neighborhood of $0$, we then write, on the event $\{\sup |B|<x/2\}$,  
$$1- \exp\Big(-\int_{0}^{s}\d u\:\big(w_{\lambda}(x+B_{u})- w_{\lambda}(x)\big)\Big)\Big)=\int_0^sw_\lambda'(x)B_u\d u+\int_0^sS_\lambda(x,B_u)B_u^2\d u+O_K(s^2)\, .$$ 
Taking expectations, and using the fact that $\mathbb{E}_{0\rightarrow 0}^{(s)}[B_u^2]\leq u$, we obtain:
\begin{align*}
\mathbb{E}_{0\rightarrow 0}^{(s)}\Big[\Big(1-\exp\Big(-\int_{0}^{s}\d u\:\big(w_{\lambda}(x+B_{u})- &w_{\lambda}(x)\big)\Big)\Big)\mathbbm{1}_{\sup |B|<\frac{x}{2}}\Big]\\
  &=  w_\lambda'(x)\mathbb{E}_{0\rightarrow 0}^{(s)}\Big[\int_{0}^{s} B_{u}\d u\mathbbm{1}_{\sup |B|<\frac{x}{2}}\Big]+O_K(s^{2}).
\end{align*}
By symmetry, the last expectation vanishes, and this completes the proof of the lemma.
\end{proof}
Let us continue our way towards \eqref{eq:inteqw},  and introduce some more notation. Fix $\ell>0$, and, working under the measure $\mathbf{N}$, write $s^{(\ell)}_{1},\dots, s^{(\ell)}_{M_{\ell}}$ for the elements of 
$$\Big\{t\geq 0:\: \Delta_{t}\geq \ell\text{ such that } \Delta_{s}<\ell \text{ for every } s\prec t\Big\}$$ in increasing order. Alternatively, these can be defined as the finite elements in the sequence of stopping times $(s_i^{(\ell)},i\geq 1)$ defined inductively by 
$s^{(\ell)}_1=\inf\{t\geq 0:\Delta_t\geq \ell\}$, and then, for $i\geq 1$, by setting
$$t^{(\ell)}_i:=\inf\{t\geq s^{(\ell)}_{i}:\:X_{t}=X_{s^{(\ell)}_{i}-}\}\, ,\qquad 
s^{(\ell)}_{i+1}:=\inf\{t\geq t_i^{(\ell)}:\Delta_t\geq \ell\}\, .$$  We also introduce the set 
$$\mathcal{A}^{(\ell)}:=\bigcup_{1\leq i\leq M_\ell}  [s^{(\ell)}_{i},t^{(\ell)}_{i}]. $$
In words, the image by $\Pi_{d}$ of the complement of $\mathcal{A}^{(\ell)}$ corresponds to the looptree obtained after removing from $\mathcal{L}$ all the loops with size larger than $\ell$ as well as all their descendants. In order to encode the latter and the associated labels we reparametrize  it. In this direction, we let $\Gamma^{(\ell)}$ be the right-inverse of $t\mapsto \int_{0}^{t} \mathrm{d} s \mathbbm{1}_{s\notin \mathcal{A}^{(\ell)}}$, and set
$$(X_t^{(\ell)},Z^{(\ell)}_t):= (X_{\Gamma^{(\ell)}_t},Z_{\Gamma^{(\ell)}_t}), \quad \text{ for } t\geq 0.$$
Notice that $Z^{(\ell)}$ can also be defined as a Gaussian process, by adapting the construction of $Z$ from $X$, but using the process $X^{(\ell)}$ instead. Furthermore, we claim that under $\mathbf{N}$, the process $X^{(\ell)}$ is the excursion of a L\'evy process with Laplace exponent 
\begin{eqnarray*}\phi_{\ell}(\lambda)&:=&\frac{1}{\Gamma(-\alpha)}\int_{0}^{\infty}(\exp(-\lambda r)-1+\lambda r)\frac{ \mathrm{d}r}{r^{1+\alpha}} - \frac{1}{\Gamma(-\alpha)}\int_{\ell}^{\infty}(\exp(-\lambda r)-1)\frac{ \mathrm{d}r}{r^{1+\alpha}}\\&=& \frac{\alpha}{\Gamma(2-\alpha)}\ell^{-(\alpha-1)}\lambda+\frac{1}{\Gamma(-\alpha)}\int_{0}^{\ell}(\exp(-\lambda r)-1+\lambda r)\frac{ \mathrm{d}r}{r^{1+\alpha}}.  \end{eqnarray*}
Indeed, working first under the probability distribution $\mathbf{Q}$
and applying the strong Markov property at the stopping times
$s^{(\ell)}_i,t^{(\ell)}_i$, we see that $X^{(\ell)}$ is a version of
the process $X$, where the jumps of size $\geq \ell$ have been
trimmed, and the computation of the Laplace exponent $\phi_\ell$
follows from the Lévy-Itô representation of an $\alpha$-stable L\'evy
process from a Poisson point measure, by restricting the latter to
atoms of size $<\ell$. Details are left to the reader,  a similar discussion in the context of Lévy trees can be found in \cite[Proposition 1.1]{ADN:LT}.   
In particular, if we write  $\sigma_\ell$ for the lifetime of
$X^{(\ell)}$ under $\mathbf{N}$, we have $\mathbf{N}(1-\exp(-\lambda
\sigma_\ell))=\phi^{-1}_{\ell}(\lambda)$, by \cite[Theorem 1, Chapter VII]{Ber96}. As a consequence,  note that:
\begin{equation}\mathbf{N}(\sigma_{\ell})=
(\phi_{\ell}^{-1})'(0)=(\alpha-1)\Gamma(-\alpha)\ell^{\alpha-1}\, ,\qquad 
\mathbf{N}(\sigma_\ell^2)=(\phi_{\ell}^{-1})''(0)=\frac{(\alpha-1)^3\Gamma(-\alpha)^2}{2-\alpha}\ell^{2\alpha-1}\, .
\end{equation}
Moreover, conditionally on $(X^{(\ell)},Z^{(\ell)})$, $M_\ell$ is a Poisson random variable with intensity
\[\sigma_\ell\int_\ell^\infty\frac{\d r}{\Gamma(-\alpha)r^{1+\alpha}}=\sigma_{\ell}\frac{\ell^{-\alpha}}{\alpha\Gamma(-\alpha)},\]
and therefore we have
\begin{equation}
  \label{eq:9}
  \mathbf{N}(M_{\ell})=\frac{\alpha-1}{\alpha\, \ell}\, ,\qquad 
  \mathbf{N}(M_\ell^2)= \frac{\alpha-1}{\alpha^2(2-\alpha) \ell}.
\end{equation}
The next lemma establishes that  $Z^{(\ell)}$ is in a sense small as  $\ell\to 0$. 

\begin{lem}\label{lem:boundZl} For every $\delta>0$, we have:
\begin{equation}\label{eq:integral:eq:**}
\mathbf{N}(\sup |Z^{(\ell)}|>\ell^{\frac{1}{2}-\delta})=o(\ell).
\end{equation}
\end{lem}

We stress that \eqref{eq:integral:eq:**} is not sharp, and we could use the same proof to obtain that $\mathbf{N}(\sup |Z^{(\ell)}|>\ell^{\frac{1}{2}-\delta})=o(\ell^c)$ for any exponent $c>0$. 
\begin{proof} We first notice that it suffices to prove the one-sided estimate $\mathbf{N}(\sup Z^{(\ell)}>\ell^{\frac{1}{2}-\delta})=o(\ell)$, since $Z^{(\ell)}$ and $-Z^{(\ell)}$ have the same law. 
By scaling, we have $\mathbf{N}(\sup Z^{(\ell)}>\ell^{\frac{1}{2}-\delta})=\ell^{-1}\mathbf{N}(\sup Z^{(1)}>\ell^{-\delta})$ and it suffices to establish that:
$$\mathbf{N}(\sup Z^{(1)}>\ell^{-\delta})= o(\ell^2).$$
The idea  is to control different events involving $\sup Z^{(1)}$ in terms of the excursion lengths $\sigma_{\ell}$ and $\sigma$, disintegrate with respect to the latter, and use the stretched-exponential controls on $\sup Z$ under $ \mathbf{P}$ derived in Proposition 
\ref{variations_Z}. Fix $\eta\in (0,2\alpha\delta)$. We first write 
\begin{equation}\label{eq:Z:ell}
\mathbf{N}(\sup Z^{(1)}>\ell^{-\delta})= \mathbf{N}(\sigma_1>\ell^{-\eta})+  \mathbf{N}(\sup Z^{(1)}>\ell^{-\delta}, \sigma_1\leq \ell^{-\eta}), 
\end{equation}
and study each term separately. First, since $\mathbf{N}(1-\exp(-\lambda \sigma_1))= \phi_{1}^{-1}(\lambda)$, and since $\phi_1$ is of class $\mathcal{C}^\infty$ on $\R$, it follows that $\mathbf{N}\big(\sigma_{1}^{n}\big)<\infty$ for every $n\geq 1$. Thus, by the Markov inequality, we deduce that
\begin{align*}
    \mathbf{N}(\sigma_{1}> \ell^{-\eta})=\mathbf{N}(\sigma_{1}^{n}\geq \ell^{-n\eta})\leq \ell^{n\eta}\mathbf{N}(\sigma_{1}^{n}).
\end{align*}
By choosing $n$ large enough, we obtain that $\mathbf{N}(\sigma_{1}\geq \ell^{-\eta})=o(\ell^{2})$. Let us now consider the remaining term in \eqref{eq:Z:ell}.  
  Since, conditionally on $(X^{(1)},Z^{(1)})$, the random variable $M_1$ is Poisson  with mean $\frac{\sigma_1}{\alpha\Gamma(-\alpha)}$ we have:
\begin{align*}
 \mathbf{N}(\sup Z^{(1)}>\ell^{-\delta}, \sigma_1\leq \ell^{-\eta})&= \mathbf{N}\big(\mathbbm{1}_{\sup Z^{(1)}>\ell^{-\delta}}\mathbbm{1}_{\sigma_1\leq \ell^{-\eta}}\exp\big(\frac{\sigma_1}{\alpha\Gamma(-\alpha)} \big)\mathbbm{1}_{M_{1}=0}\big)\\
 &\leq \exp\big(\frac{\ell^{-\eta}}{\alpha\Gamma(-\alpha)} \big)  \mathbf{N}\big(\mathbbm{1}_{\sup Z^{(1)}>\ell^{-\delta}}\mathbbm{1}_{\sigma_1\leq \ell^{-\eta}}\mathbbm{1}_{M_{1}=0}\big).
\end{align*}
Now, since $(Z^{(1)}, \sigma_1)=(Z,\sigma)$ on the event $\{M_1=0\}$, we infer that the previous display is bounded above by:
$$ \exp\big(\frac{\ell^{-\eta}}{\alpha\Gamma(-\alpha)} \big)  \mathbf{N}\big(\sup Z>\ell^{-\delta}, \sigma\leq \ell^{-\eta}\big). $$
Moreover,  by scaling and disintegration with respect to $\sigma$ we have:
\begin{align*}
    \mathbf{N}(\sup Z>\ell^{-\delta}\:;\: \sigma <\ell^{-\eta})&=\frac{1}{|\Gamma\big(-\frac{1}{\alpha}\big)|}\int_{0}^{\ell^{-\eta}}\frac{ \mathrm{d}v}{v^{1+\frac{1}{\alpha}}}\mathbf{N}^{(v)}\big(\sup Z>\ell^{-\delta}\big)\\
 &=   \frac{1}{|\Gamma\big(-\frac{1}{\alpha}\big)|}\int_{0}^{\ell^{-\eta}}\frac{ \mathrm{d}v}{v^{1+\frac{1}{\alpha}}}\mathbf{P}\big(\sup Z>v^{-\frac{1}{2\alpha}}\ell^{-\delta}\big)\\
    &=\frac{2}{\Gamma\big(1-\frac{1}{\alpha}\big)} \ell^{2\delta}\int_{\ell^{\frac{\eta}{2\alpha}-\delta}}^{\infty}\, \d u\,  u \cdot \mathbf{P}(\sup Z > u).
\end{align*}
Finally, by (i) in Proposition \ref{variations_Z}, we can find positive, finite constants $\beta, c^\prime,C^\prime$ such that, for every $\ell>0$, the previous display is bounded above by $C^\prime \exp(-c^{\prime} \ell^{-\beta(\delta-\eta/2\alpha)})$.  Since $\eta<2\alpha\delta$, this implies $$ \mathbf{N}(\sup Z^{(1)}>\ell^{-\delta}, \sigma_1\leq \ell^{-\eta})=o(\ell^2)\, ,$$ completing the proof of the lemma.
\end{proof}

We finally derive the integral equation \eqref{eq:inteqw}. 
\begin{proof}[ Proof of Theorem \ref{two_points_function}]
We fix $\lambda\geq 0$. For every $\ell, y>0$, consider the quantity
\[J_{\lambda}^{(\ell)}(y):=\alpha
  \ell^{\alpha}\int_{\ell}^{\infty}\frac{\mathrm{d}s}{s^{\alpha+1}}\mathbb{E}_{0\rightarrow
    0}^{(s)}\Big[1-\exp\big(-\int_{0}^{s} \d u\:w_{\lambda}(y+B_{u})\big)\Big]\quad
  \in \quad [0,1]\, .\]
Notice that we can rewrite $J_{\lambda}^{(\ell)}(y)$ as follows:
\begin{align*}
     J_{\lambda}^{(\ell)}(y)=  \frac{\alpha}{\alpha-1}\ell
                              w_{\lambda}(y)+\alpha
                              \ell^{\alpha}\int_{\ell}^{\infty}\frac{\mathrm{d}s}{s^{\alpha+1}}\mathbb{E}_{0\rightarrow
                              0}^{(s)}\Big[1-s
                              w_{\lambda}(y)-\exp\big(-\int_{0}^{s} \d u\:w_{\lambda}(y+B_{u})\big)\Big]\, .
\end{align*}
Therefore, by  Lemma \ref{lem:basicpropw} (ii), if $y$ is
restricted to a compact subset $K$ of $(0,\infty)$, it holds that:
\begin{align}
  J_\lambda^{(\ell)}(y)&= \frac{\alpha}{\alpha-1}\ell
                              w_{\lambda}(y)+\alpha
                              \ell^{\alpha}\int_{0}^{\infty}\frac{\mathrm{d}s}{s^{\alpha+1}}\mathbb{E}_{0\rightarrow
                              0}^{(s)}\Big[1-s
  w_{\lambda}(y)-\exp\big(-\int_{0}^{s} \mathrm{d} u\:w_{\lambda}(y+B_{u})\big)\Big]+o_K(\ell^\alpha)\nonumber\\
\label{eq:8}
                       &=O_K(\ell)\, .
\end{align}
As before, the subscript $K$ in the $ o,O$ notation above means that it is uniform in $y\in K$.

We now fix $x>0$ in the rest of this proof, and observe that the desired  integral equation \eqref{eq:inteqw} is equivalent to
\begin{equation}\label{eq:integral:eq:(*)}
\lim \limits_{\ell\to 0} \ell^{-\alpha+1}\big(w_{\lambda}(x)-\frac{\alpha-1}{\alpha}\ell^{-1} J_{\lambda}^{(\ell)}(x)\big)=\frac{\Gamma(2-\alpha)}{\alpha-1}\lambda.
\end{equation}
Recall now the notation $s^{(\ell)}_1,\ldots,s^{(\ell)}_{M_{\ell}}$, and for every $i\leq M_{\ell}$, introduce the process 
$$B_{i}(t):=Z_{\mathrm{f}_{s^{(\ell)}_{i}}(t/\Delta_{s^{(\ell)}_i})}-Z_{s^{(\ell)}_i},\qquad
0\leq t\leq \Delta_{s^{(\ell)}_i}\, , $$ describing the labels $Z$ as
one circles around the loop encoded by the jump time $s^{(\ell)}_i$,
shifted by their value at $s^{(\ell)}_i$. Note that conditionally
on $\Delta_{s^{(\ell)}_i}$, this process is a Brownian bridge with
duration $\Delta_{s^{(\ell)}_i}$. 
We also let  $(u_{i,j},v_{i,j})_{j\in \mathbb{N}}$ be the connected components of the open set $\{t\in[s^{(\ell)}_{i},t^{(\ell)}_{i}] :\:X_{t}>I_{s^{(\ell)}_{i},t}\}$: these intervals describe the looptrees grafted on the latter loop. Precisely, we encode the positions, encodings and labelings of these looptrees by introducing the quantities
$x^{i,j}:=X_{s^{(\ell)}_{i}}-X_{u_{i,j}}$ and 
\[X^{i,j}_{t}:=X_{(u_{i,j}+t)\wedge v_{i,j}}-X_{u_{i,j}}\quad \text{ and }\quad Z^{i,j}_{t}:=Z_{(u_{i,j}+t)\wedge v_{i,j}}-Z_{u_{i,j}}\:\:,\quad\text{for}\:\:t\geq 0, .\]
We also let $\sigma^{i,j}=v_{i,j}-u_{i,j}$. 
Now notice that $w_{\lambda}(x)$ equals
\begin{eqnarray}\label{eq:winterm}
&=&\mathbf{N}\big(1-\exp(-\lambda
\sigma)\mathbbm{1}_{\sup Z <x}\big)\nonumber\\
&=&
  \mathbf{N}\Big(1-\exp\Big(-\lambda
\sigma_{\ell}-\sum \limits_{i=1}^{M_{\ell}}\sum \limits_{j\in
  \mathbb{N}}\big(\lambda\sigma^{i,j}-\log(\mathbbm{1}_{\sup_{[u^{i,j},v^{i,j}]}
  Z <x})\big)\Big)\Big)\nonumber\\
&\underset{Z_{u^{i,j}}=Z_{s^{(\ell)}_i}+B_i(x^{i,j})}{=}& \mathbf{N}\Big(1-\exp\Big(-\lambda \sigma_{\ell}-\sum \limits_{i=1}^{M_{\ell}}\sum \limits_{j\in \mathbb{N}}\big(\lambda\sigma^{i,j}-\log\big(\mathbbm{1}_{\sup Z^{i,j} <x-B_i(x^{i,j})-Z_{s^{(\ell)}_i}}\big)\big)\Big)\Big).
\end{eqnarray}
It is now a consequence of the strong Markov property at times
$s^{(\ell)}_i$ and $t^{(\ell)}_i$ that conditionally on $\sigma_\ell$, $M_\ell$, and
$(Z_{s^{(\ell)}_i},\Delta_{s^{(\ell)}_{i}},B_i)$ for $1\leq i\leq M_\ell$,  the point measures
\[\mathcal{N}^{i}:=\sum \limits_{j\in J}\delta_{x^{i,j}, X^{i,j},
    Z^{i,j}},\quad 1\leq i\leq M_{\ell}, \]
are independent Poisson point measures with intensities 
$$\mathbbm{1}_{x\in[0,\Delta_{s^{(\ell)}_{i}}]} \mathrm{d}x \mathbf{N}( \mathrm{d}(X,Z))\,
,\qquad  1\leq i\leq M_{\ell}\, .$$ 
An application of the Laplace functional formula for Poisson measures shows that \eqref{eq:winterm} can be written in the form: 
$$
  \mathbf{N}\Big(1-\exp\Big(-\lambda \sigma_{\ell}-\sum
  \limits_{i=1}^{M_{\ell}}\int_0^{\Delta_{s^{(\ell)}_i}}\d
  u\,  w_\lambda(x-B_i(u)-Z_{s^{(\ell)}_i})\Big)\Big)\, .
  $$
Now note that conditionally given $\sigma_\ell,M_\ell,Z_{s^{(\ell)}_i}$ for $1\leq i\leq M_\ell$, the random variables  $(\Delta_{s^{(\ell)}_i},B_i)$ 
 are i.i.d.\ with common distribution $\alpha\ell^\alpha
s^{-\alpha-1}\d s\ind_{s\geq \ell}\mathbb{P}^{(s)}_{0\rightarrow 0}(\d b)$. Hence, we
may integrate with respect to these random variables and rewrite the
preceding expression as: 
\begin{align} 
\lefteqn{\mathbf{N}\Big(1-\exp(-\lambda
\sigma_{\ell})\prod_{i=1}^{M_\ell}\Big(\alpha\ell^\alpha\int_\ell^\infty\frac{\d s}{s^{\alpha+1}}\E^{(s)}_{0\rightarrow0}\Big[\exp\big(-\int_0^s\d u\,  w_\lambda(x-B_u-Z_{s_i^{(\ell)}})\big)\Big]\Big)\Big)}\nonumber\\
&=\mathbf{N}\Big(1-\exp(-\lambda
\sigma_{\ell})\prod_{i=1}^{M_\ell}\big(1-J_\lambda^{(\ell)}(x-Z_{s^{(\ell)}_i})\big)\Big)\nonumber
\\
&=\mathbf{N}\Big(\ind_{\sup |Z^{(\ell)}|<\ell^{\frac{1}{2}-\delta}}\big(1-\exp(-\lambda
\sigma_{\ell})\prod_{i=1}^{M_\ell}\big(1-J_\lambda^{(\ell)}(x-Z_{s^{(\ell)}_i})\big)\big)\Big)+o(\ell)
\end{align}
where we used Lemma \ref{lem:boundZl} in the second equality, for some $\delta$ that we choose in $(0,1-\alpha/2)$, and
where the $o(\ell)$ term is uniform in $x\in (0,\infty)$. Using
$0\leq  \mathrm{e}^{-y}-(1-y)\leq y^2$ for $y\geq 0$, we obtain that this is equal to 
\begin{equation}
  \label{eq:2}
  \mathbf{N}\Big(\ind_{\sup |Z^{(\ell)}|<\ell^{\frac{1}{2}-\delta}}\big(1-\exp(-\lambda
\sigma_{\ell}-\sum_{i=1}^{M_\ell}J_\lambda^{(\ell)}(x-Z_{s^{(\ell)}_i})\big)\big)\Big)+R(\ell,x)+
o(\ell)
\end{equation}
where the remainder satisfies, for fixed $x>0$, 
$$0\leq R(\ell,x)\leq \mathbf{N}\Big(
  \ind_{\sup
    |Z^{(\ell)}|<\ell^{\frac{1}{2}-\delta}}\sum_{i=1}^{M_\ell}J^{(\ell)}_\lambda(x-Z_{s^{(\ell)}_i})^2\Big)=
  O(\ell^2)\mathbf{N}(
  M_\ell)=O(\ell)\, ,
  $$
  due to 
\eqref{eq:8} and \eqref{eq:9}. 

Now note that, by (ii)
in Lemma \ref{lem:basicpropw} and \eqref{eq:8}, we have 
\[J_{\lambda}^{(\ell)}(x+\eta)-J_{\lambda}^{(\ell)}(x)=\frac{\alpha}{\alpha-1}\ell
  w'_{\lambda}(x)\eta+o(\ell^\alpha)\, ,\]
as $\ell\to 0$, where the remainder term is uniform over the
choice of $\eta\in (-\ell^{1/2-\delta},\ell^{1/2-\delta})$. Hence, by developing the exponential in
\eqref{eq:2}, we obtain:
\begin{align*}
     w_\lambda(x)&=\mathbf{N}\Big(\mathbbm{1}_{\sup |Z^{(\ell)}|<\ell^{\frac{1}{2}-\delta}}\Big(1-\exp\big(-\lambda \sigma_{\ell}-\sum \limits_{i=1}^{M_{\ell}}J_{\lambda}^{(\ell)}\big(x-Z_{s^{(\ell)}_i}\big)\big)\Big)\Big)+o(\ell)\\
     &= \mathbf{N}\Big(1-\exp\big(-\lambda \sigma_{\ell}-\sum
       \limits_{i=1}^{M_{\ell}}J_{\lambda}^{(\ell)}(x)\big)\Big)+\frac{\alpha\,
       \ell}{\alpha-1}w_\lambda'(x)\mathbf{N}\Big(\mathrm{e}^{-\lambda \sigma_\ell-\sum_{i=1}^{M_\ell}J^{(\ell)}_\lambda(x)}\sum_{i=1}^{M_\ell} Z_{s^{(\ell)}_i}\Big)\\
       &+o\Big(\mathbf{N}(M_\ell)\ell^\alpha\Big)
       +O(\mathbf{N}(M_\ell^2)\ell^{3-2\delta})\, .
\end{align*}
Using the fact that
$\mathbf{N}\big(\sum \limits_{i=1}^{M_{\ell}} Z_{s^{(\ell)}_i}\, |\, \sigma_\ell,M_\ell\big)=0$ by
symmetry of the law of $Z_{s^{(\ell)}_i}$,
and that the remainder terms are $o(\ell^{\alpha-1})$ by \eqref{eq:9}, we finally obtain:      
     \begin{align*}
     w_\lambda(x)
     &=\mathbf{N}\Big(1-\exp\Big(-(\lambda+\frac{\alpha-1}{\Gamma(2-\alpha)} \ell^{-\alpha}J_{\lambda}^{(\ell)}(x)) \sigma_{\ell}\Big)\Big)+o(\ell^{\alpha-1})\\
     &=\phi_{\ell}^{-1}\Big(\lambda+\frac{\alpha-1}{\Gamma(2-\alpha)} \ell^{-\alpha}J_{\lambda}^{(\ell)}(x)\Big) +o(\ell^{\alpha-1})\\
     &= \frac{\Gamma(2-\alpha)}{\alpha}\ell^{\alpha-1}
       \Big(\lambda+\frac{\alpha-1}{\Gamma(2-\alpha)}
       \ell^{-\alpha}J_{\lambda}^{(\ell)}(x)
       \Big)+o(\ell^{\alpha-1})
\end{align*}
and \eqref{eq:integral:eq:(*)} follows.
\end{proof}

\subsection{Computation of $\mathbf{N(\sup Z>1)}$}
We now use Theorem \ref{two_points_function} in the case $\lambda=0$ to deduce that:

\begin{prop}\label{sec:stable-map-3} For $x>0$, we have
\begin{equation}w_{0}(x) = \bN(\sup Z\geq x) =  \frac{\bN(\sup Z\geq 1)}{x^{2}}=\frac{\alpha(\alpha-1)}{2x^{2}}.\label{exp:inf:Z}
\end{equation}
\end{prop}
The proof of the above proposition heavily relies  on exact computations involving Bessel and hypergeometric functions.  We start with a  lemma which  can be understood as a variation on the well-known absolute continuity relations between Brownian motion and Bessel processes (see \cite[Exercise
XI.1.22]{RY99}).

\begin{lem}\label{lem:E:a:to:b}
Let $a<b<x$ and $\ell>0$. Then, for every $c>0$ we have:
\begin{equation*}
\mathbb{E}_{a\to b}^{(\ell)}\left[\exp\left(-\int_{0}^{\ell}\frac{c~\d u}{(B_u+x)^{2}}\right)\right]=\sqrt{\frac{2\pi}{\ell}}\sqrt{(x+a) (x+b)}\exp\Big(-\frac{(x+a)(x+b)}{\ell}\Big
   )I_{\nu}\Big(\frac{(x+a)(x+b)}{\ell}\Big),
\end{equation*}
where $\nu=\frac{\sqrt{8c+1}}{2}$ and  $I_{\nu}$ stands for the modified Bessel function of the first kind with index $\nu$: 
$$ I_{\nu}(z) := \left(  \frac{z}{2}\right)^{\nu} \sum_{k =0}^{\infty} \frac{(z^{2}/4)^{k}}{k! \, \Gamma(\nu+k+1)}, \quad z\in \mathbb{R}.$$
\end{lem}
\begin{proof} 
By translation invariance of Brownian motion, it is enough to establish the lemma in the case $0=a <b<x$. Now, note that  by monotone convergence, we have
\begin{eqnarray*}
\mathbb{E}_{0\to b}^{(\ell)}\Big[\exp\big(-c\int_{0}^{\ell}\frac{ \d u}{(B_u+x)^{2}}\big)\Big]&=&\lim \limits_{\varepsilon\to 0} \mathbb{E}_{0\to
  b}^{(\ell)}\Big[\exp\Big(-c\int_0^{\ell-\eps}\frac{ \d
  u}{(B_u+x)^2}\Big)\Big]\\
&=&\lim \limits_{\varepsilon\to 0}\mathbb{E}_{0}^{(\ell)}\Big[\exp\Big(-c\int_0^{\ell-\eps}\frac{\, \d
  u}{(B_u+x)^2}\Big)\frac{p_\eps(b-B_{\ell-\eps})}{p_\ell(b)}\Big]
\end{eqnarray*}
where  $p_t(x)=(2\pi t)^{-1/2}\exp(-x^2/2t)$ stands for  the
one-dimensional Gaussian density function. Let us perform the computation of the previous display.
Using again the invariance by translation of Brownian motion and the absolute
continuity relations between Brownian motion and the Bessel processes (see \cite[Exercise
XI.1.22]{RY99}), we infer that:
$$\mathbb{E}_{0}^{\ell}\Big[\exp\Big(-c\int_0^{\ell-\eps}\frac{\, \d
  u}{(B_u+x)^2}\Big)\frac{p_\eps(b-B_{\ell-\eps})}{p_\ell(b)}\mathbbm{1}_{ T_{-x} > \ell- \varepsilon}\Big]=\E_x^{\langle \nu \rangle}\Big[\Big(\frac{x}{B_{t-\eps}}\Big)^{\nu + \frac{1}{2}} \frac{p_\eps(b+x-B_{\ell-\eps})}{p_\ell(b)}  \mathbbm{1}_{ T_{0}> \ell - \varepsilon}\Big]$$
where $\nu=\frac{\sqrt{8c+1}}{2}$ is the index of the Bessel process on the right-hand side (equivalently, dimension $d= 2 \nu +2$) starting from $x>0$, and as usual $T_{z} = \inf\{ t \geq 0 : B_{t} =z\}$. Notice that since $c>0$ we have $\nu > 1/2$ so under $ \mathbb{P}^{\langle \nu \rangle }_{x}$ the process $B$ never touches $0$ and the indicator function in the right-hand side of the previous display is superfluous (it is also superfluous in the left-hand side since the integral a.s. diverges on the event when $B$ touches $-x$). If we write 
$$p^{\langle \nu \rangle}_t(x,y)=\frac{\P_x^{\langle \nu \rangle}(B_t\in \d
  y)}{\d y},$$
for the density function of the one-dimensional marginals of the
canonical process under $ \mathbb{P}_x^{\langle \nu \rangle}$, then  we obtain
$$\E_{0\to
  0}^{(\ell)}\Big[\exp\Big(-\int_0^{\ell-\eps}\frac{c\, \d
  u}{(x+B_u)^2}\Big)\Big]= \lim_{ \varepsilon \to 0}\int_0^\infty\d y\, p_{\ell-\eps}^{\langle
   \nu\rangle}(x,y) \left(\frac{x}{y}\right)^{\nu+ \frac{1}{2}} \frac{ p_\eps(b+x-y)}{p_{\ell}(b)}\, ,$$
and by an easy dominated convergence argument, the latter converges as
$\eps\to 0$ to
$$\frac{p_\ell^{\langle
   \nu\rangle}(x,x+b)}{p_\ell(b)}\, .$$
    The desired result now follows using the explicit expression of Bessel densities that can be found
in \cite[Chapter XI.1 p.446]{RY99}. 
\end{proof}
\noindent We can now proceed with the proof of \eqref{sec:stable-map-3}.
\begin{proof}[Proof of Proposition \ref{sec:stable-map-3}.]
For simplicity, we let $c:=c(\alpha)=\bN(\sup Z> 1)$.  We first use Theorem \ref{two_points_function} with $\lambda=0$ together with Proposition \ref{Z<-1} to deduce that, for every 
$x>0$, the function $w_{0}(x) = \frac{c}{x^{2}}$ satisfies: 
\begin{equation}\label{eq:3}
\int_0^\infty\frac{\d t}{t^{1+\alpha}}\E^{(t)}_{0\to
  0}\bigg[1-\frac{ct}{x^2}-\exp\Big(-\int_0^t\frac{c\, \d
  u}{(x+B_u)^2}\Big)\bigg]=0\, .
\end{equation}
We fix $x=1$ henceforth. By Lemma \ref{lem:E:a:to:b},  we have
$$\int_0^\infty\frac{\d t}{t^{1+\alpha}}\Big(1-ct- \sqrt{\frac{2\pi}{t}} \mathrm{e}^{-1/t}I_\nu\Big(\frac{1}{t}\Big)\Big)=0,$$
where $\nu=\frac{\sqrt{8c+1}}{2}$ and  $I_{\nu}$ is  the modified Bessel function with index $\nu$. We now perform 
the change of  variables $s=1/t$ to get that: 
\begin{equation}\label{eq:5}
\int_0^\infty\d s\,
\mathrm{e}^{-s}s^{\alpha-1/2}\bigg(I_\nu(s)-\frac{\mathrm{e}^s}{\sqrt{2\pi
    s}}\Big(1-\frac{c}{s}\Big)\bigg)=0.
\end{equation}
It is not obvious at first that the integral on the left hand-side is
well-defined. However, since $\alpha\in (1,2)$ and $\nu \geq 1/2$, the function $I_\nu$ is
continuous with $I_\nu(0)=0$ and the integral is well-defined in the
vicinity of $0$. On the other hand, the asymptotic properties of Bessel
functions are that, when $s\to\infty$,
$$I_\nu(s)=\frac{\mathrm{e}^s}{\sqrt{2\pi
    s}}\Big(1-\frac{4\nu^2-1}{8s}+O(\frac{1}{s^2})\Big)\, ,$$
and $(4\nu^2-1)/8$ is precisely $c$. Therefore,
as $s\to \infty$, 
$$\mathrm{e}^{-s}s^{\alpha-1/2}\bigg(I_\nu(s)-\frac{\mathrm{e}^s}{\sqrt{2\pi
    s}}\Big(1-\frac{c}{s}\Big)\bigg)=O\Big(\frac{1}{s^{3-\alpha}}\Big)\,
,$$
which is integrable near $\infty$. To evaluate the integral in
\eqref{eq:5}, we observe that it is equal to 
$$\lim_{u\downarrow 1}\int_0^\infty\d s\,
\mathrm{e}^{-us}s^{\alpha-1/2}\bigg(I_\nu(s)-\frac{\mathrm{e}^s}{\sqrt{2\pi
    s}}\Big(1-\frac{c}{s}\Big)\bigg)\, ,$$ which comes from a
dominated convergence argument, that is easily justified using the above  asymptotic properties of $I_\nu$. Now, for every fixed
$u>1$, we can split the above integral into
\begin{eqnarray}
  \lefteqn{\int_0^\infty\d s\,
    \mathrm{e}^{-us}s^{\alpha-1/2}I_\nu(s)-\int_0^\infty\d s\frac{s^{\alpha-1}\mathrm{e}^{-(u-1)s}}{\sqrt{2\pi
      }}+c\int_0^\infty\d s\frac{s^{\alpha-2}\mathrm{e}^{-(u-1)s}}{\sqrt{2\pi
      }}}\nonumber\\
  &=&\int_0^\infty\d s\,
  \mathrm{e}^{-us}s^{\alpha-1/2}I_\nu(s) -
  \frac{\Gamma(\alpha)}{(u-1)^\alpha\sqrt{2\pi}}+\frac{c\,
    \Gamma(\alpha-1)}{(u-1)^{\alpha-1}\sqrt{2\pi}}\, .\label{eq:6}
\end{eqnarray}
On the other hand, the value of the above Laplace-type transform involving
$I_\nu$ is known and equals
\begin{equation}\label{eq:prem:I:nu:desp}
\frac{\Gamma(\nu+\alpha+1/2)}{2^{\nu}
  u^{\nu+\alpha+1/2}\Gamma(\nu+1)}\, \cdot\, {} _2F_1\bigg(\begin{array}{cc}
\frac{\nu+\alpha+1/2}{2},\frac{\nu+\alpha+3/2}{2}\\ 
\nu+1
\end{array}
;\frac{1}{u^2}\bigg)
\end{equation} where $_2F_1(a,b;c;z)$ is a hypergeometric
function (we drop the indices in the sequel and simply write
$F={}_2F_1$), as defined in \cite{AnAsRo99}. The problem in letting
$u\downarrow 1$ in the latter expression is that (for  real
parameters $a,b,c$) the analytic function $F(a,b;c;z)$ diverges at its
radius of convergence $z=1$ whenever $c-a-b\leq 0$, as it is the case
here. By \cite[Corollary 2.3.3]{AnAsRo99} we can rewrite (we let
$z=1/u^2$):
\begin{eqnarray*}
\lefteqn{F\bigg(\begin{array}{cc}
\frac{\nu+\alpha+1/2}{2},\frac{\nu+\alpha+3/2}{2}\\ 
\nu+1
\end{array}
;z\bigg)}\\
&=&\frac{\Gamma(\nu+1)\Gamma(-\alpha)}{\Gamma\Big(\frac{\nu-\alpha+1/2}{2}\Big) \Gamma\Big(\frac{\nu-\alpha+3/2}{2}\Big)}F\bigg(\begin{array}{cc}
\frac{\nu+\alpha+1/2}{2},\frac{\nu+\alpha+3/2}{2}\\ 
\alpha+1
\end{array}
;1-z\bigg)\\
& &+\frac{\Gamma(\nu+1)\Gamma(\alpha)}{\Gamma\Big(\frac{\nu+\alpha+1/2}{2}\Big) \Gamma\Big(\frac{\nu+\alpha+3/2}{2}\Big)}(1-z)^{-\alpha}F\bigg(\begin{array}{cc}
\frac{\nu-\alpha+1/2}{2},\frac{\nu-\alpha+3/2}{2}\\ 
1-\alpha
\end{array}
;1-z\bigg)\, .
\end{eqnarray*}
Since the function $z\mapsto F(a,b;c;z)$ is analytic in $\{|z|<1\}$, and $F(a,b;c;0)=1$,
this form allows to deduce that, as $z\uparrow 1$, 
\begin{eqnarray*}
\lefteqn{F\bigg(\begin{array}{cc}
\frac{\nu+\alpha+1/2}{2},\frac{\nu+\alpha+3/2}{2}\\ 
\nu+1
\end{array}
;z\bigg)}\\
&=&\frac{A}{(1-z)^\alpha}+\frac{B}{(1-z)^{\alpha-1}}+
\frac{\Gamma(\nu+1)\Gamma(-\alpha)}{\Gamma\Big(\frac{\nu-\alpha+1/2}{2}\Big)
  \Gamma\Big(\frac{\nu-\alpha+3/2}{2}\Big)}+ O((1-z)^{2-\alpha})\\
&=&\frac{A}{(1-z)^\alpha}+\frac{B}{(1-z)^{\alpha-1}}+
\frac{2^{\nu-\alpha-1/2}\Gamma(\nu+1)\Gamma(-\alpha)}{\sqrt{\pi}\, \Gamma(\nu-\alpha+1/2)}+ O((1-z)^{2-\alpha})\,
,
\end{eqnarray*}
for some real constants $A,B$ that can be made explicit. In the second
line, we rewrote the constant term using the Gauss duplication formula
for the Gamma function.  Otherwise
said, using \eqref{eq:6} and the following displayed expression, the
integral in \eqref{eq:5} is equal to the limit as $z\uparrow 1$ of an
expression of the form
$$\frac{A'}{(1-z)^\alpha}+\frac{B'}{(1-z)^{\alpha-1}}+\frac{\Gamma(\nu+\alpha+1/2)\Gamma(-\alpha)}{2^{\alpha+1/2}\sqrt{\pi}\,
  \Gamma(\nu-\alpha+1/2)}+O((1-z)^{2-\alpha})\, ,$$
for some constants $A',B'$. A (tedious) computation shows that $A'=B'=0$, but we
can also argue that if this were not the case, then the above
expression would not have a finite limit as $z\uparrow 1$, and this
would be a contradiction. Putting things together, we finally obtain
that 
$$\int_0^\infty\d s\,
\mathrm{e}^{-s}s^{\alpha-1/2}\bigg(I_\nu(s)-\frac{\mathrm{e}^s}{\sqrt{2\pi
    s}}\Big(1-\frac{c}{s}\Big)\bigg)=\frac{\Gamma(\nu+\alpha+1/2)\Gamma(-\alpha)}{2^{\alpha+1/2}\sqrt{\pi}\,
  \Gamma(\nu-\alpha+1/2)}=0\, .$$ Since $\alpha\in (1,2)$ and $\nu\geq
1/2>0$, this can only happen if the denominator explodes, namely, if
$\nu-\alpha+1/2\in \{0,-1,-2,\ldots\}$, and the
only possible value is in fact $0$, giving
$$\nu=\alpha-\frac{1}{2} \qquad \mbox{ and
}\qquad c=\frac{4\nu^2-1}{8}=\frac{\alpha(\alpha-1)}{2}\, ,$$ as
wanted. 
\end{proof}

\subsection{Records of $Z$ on loops} \label{sec:recordsloops}
Our goal now is to use \eqref{sec:stable-map-3} to study the local minima of $Z$ on loops. More precisely, we aim to  establish the properties  stated  at the beginning  of the section. In this direction, we start with the ideal model made of the process $(X,Z)$ under the law $ \mathbf{Q}$ as described in Section \ref{sec:Markov}. Next, 
we consider the associated excursions of $(X,Z)$ above the running infimum  $X$. Namely, we let
 $(a_i,b_i)_{i\in \mathbb{N}}$ be the connected components of $\{t\geq 0:~X_t>I_t\}$, and
$$\ell_i:= -I_{a_i}\quad; \quad X^{i}_t:=X_{(a_i+t)\wedge b_i}-X_{a_i}\quad \text{ and }\quad  Z^{i}_t:=Z_{(a_i+t)\wedge b_i}-Z_{a_i},\quad t\geq 0,$$
for $i\in \mathbb{N}$. By excursion theory, the point measure 
  \begin{eqnarray}\mathcal{P}:=\sum \limits_{i\in \mathbb{N}} \delta_{\ell_{i},X^i,Z^i},   \label{eq:mesdecoration}\end{eqnarray}
is a Poisson measure with intensity  $\mathbbm{1}_{\ell\geq 0} \d \ell \otimes \mathbf{N}(\d X \d Z)$, independent of the Brownian motion $\mathrm{b}_0$ of the ``infinite loop''. To keep a picture in mind, the reader can informally think of  the looptree coded by $(X^i,Z^i)$ as glued at position $\ell_{i}$ on the ``infinite loop'' coded by $ \mathbb{R}_{+}$. We shall be interested in times $t$ which are left minimal record times of $Z$ and which happen on the infinite loop attached to time $0$. 
More precisely, we consider the set
\begin{equation}\label{def:mathscr:I}
 \mathscr{B}:=\big\{\mathrm{b}_0(t):~ \text{ such that }\mathrm{b}_0(\ell_i)+\inf Z^{i}> \mathrm{b}_{0}(t) \text{ for every } i\in \mathbb{N} \text{ with } 0 \leq \ell_i<t \big\},
\end{equation}
and we stress that the times $t \geq 0$ here only parametrize the infinite loop and not the whole process $(X,Z)$, and we refer to Figure \ref{fig:shepp} for an illustration.
\begin{figure}[!h]
 \begin{center}
 \includegraphics[width=13cm]{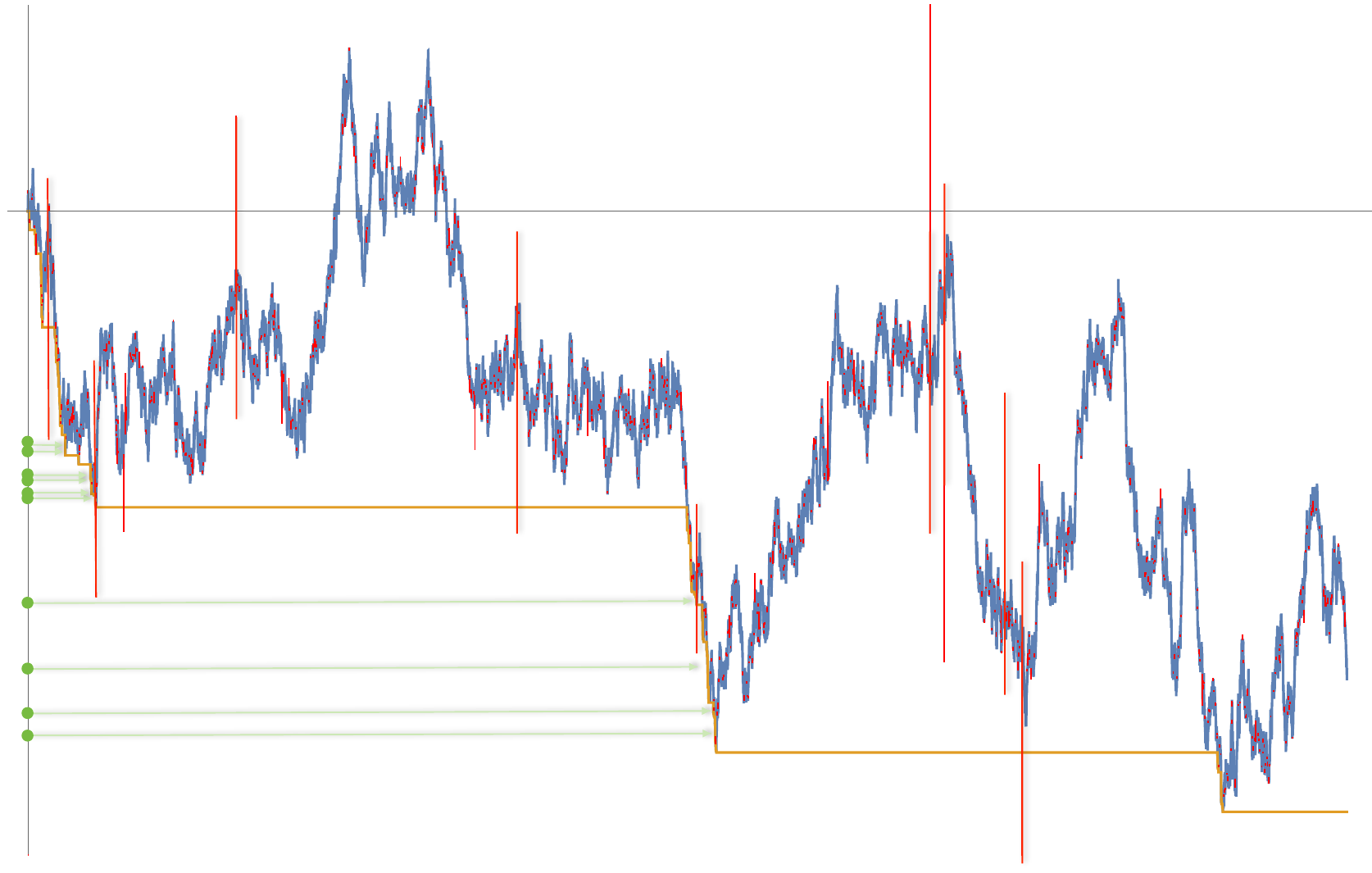}
 \caption{A simulation of a Brownian motion (in blue) and its running infimum process (in orange). The trace is decorated with red slits which happen at time $t_{i}$ and extend over $[\inf Z^{i} + \mathrm{b_0}(t_{i}) , \mathrm{b}_0(t_{i}) + \sup Z^{i}]$. We are interested in the set of values $  \mathscr{B}$ (in green on the left) which correspond to the coordinates of the points $(t,\mathrm{b}_{0}(t))$ which are visible from the left, i.e. not blocked by a red slit: a few examples are shown in light green. \label{fig:shepp}}
 \end{center}
 \end{figure}

\begin{prop}\label{I_infty}
Under $\mathbf{Q}$, the set $-\mathscr{B}$ is distributed as the range of a stable subordinator of index $2-\alpha$.
\end{prop}
\begin{proof} Since we are looking at values which are in particular running minimal values for $\mathrm{b}_0$,  we can first gather the contribution by excursions of $\mathrm{b}_0$ above its running infimum. To this end, we consider $(s_k,t_k)_{k\geq 1}$ an indexation of the connected components of the open set $\{t\geq 0:~\mathrm{b}_0(t)>\inf\limits_{[0,t]} \mathrm{b}_0\}$. To fix notation, write $\mathrm{b}_0^{k}(t):=\mathrm{b}_0((s_k+t)\wedge t_k) - \mathrm{b}_0(s_{k})$, for $t\geq 0$, and introduce the quantity
$$M_{k}:=\inf\big\{\mathrm{b}_0(\ell_i)+\inf Z^{i}:~\ell_{i}\in [s_k,t_k]\big\} -\mathrm{b}_0(s_{k}).$$ 
The reason to define these random variables is that the set 
$\mathscr{B}$ can be rewritten in the form:
\begin{equation}\label{eq:mathscr:I}
 \mathscr{B}= \mathbb{R}_- \setminus \bigcup\limits_{k=1}^{\infty}\big(\mathrm{b}_0(s_k)~,~\mathrm{b}_0(s_k)+M_k\big).
\end{equation}
We will obtain the desired result by examining the point measure:
\begin{equation}\label{eq:mathsc:I}
\sum \limits_{k\geq 1}\delta_{\mathrm{b}_{0}(s_k), M_k}.
\end{equation}
In this direction remark that,  for $z >0$, by properties of Poisson random measures we have:
  \begin{eqnarray} \mathbf{P}(M_k<-z\:|\:\mathrm{b}_0)&=&1-\exp\Big(-\int_{s_{k}}^{t_{k}} \d y \,\mathbf{N}\Big(\inf Z\leq -\big(z +\mathrm{b}_0(y) - \mathrm{b}_0(s_{k})\big)\Big) \Big)\nonumber \\
 & \underset{ \mathrm{Cor.\ } \ref{Z<-1} \  \& \  \mathrm{Prop.\ } \ref{sec:stable-map-3}}{=} &1-\exp\Big(-\int_{0}^{t_{k}-s_k} \d y \frac{\alpha (\alpha-1)}{2(z +\mathrm{b}^{k}_0(y))^2}  \Big).   \label{eq:decopoisson}\end{eqnarray}

Recall now that the excursion measure of $\mathrm{b}_0$ reflected above its infimum is $2\mathbf{n}(\d e)$, where $\mathbf{n}$ is the Itô measure of positive excursions of Brownian motion. 
It follows  by excursion theory that the measure \eqref{eq:mathsc:I}  
is a Poisson point measure with intensity $\mathbbm{1}_{t\geq 0}\d t \mu(\d z),$ 
where $\mu$ is the measure with support on $\mathbb{R}_+$ defined by:
\begin{equation}\label{eq:mu:shepp}
\mu((z,\infty))=\int 2\mathbf{n}(\d e)\Big( 1-\exp\Big(-\int_{0}^{\sigma(e)} \d y \frac{\alpha (\alpha-1)}{2(e_y+z)^2}  \Big)\Big)\, 
\end{equation}
here $\sigma(e)$ stands for the lifetime of the excursion $e$.  Let us compute the latter expression. To simplify notation, we write $c(\alpha):=\alpha(\alpha-1)/2$. First, we note that:
\begin{eqnarray*}
1-\exp\Big(-\int_0^{\sigma(e)}\frac{c(\alpha)\, \d
  y}{(e_y+z)^2}\Big)
&=&\int_0^{\sigma(e)}\d s\,
\frac{c(\alpha)}{(e_s+z)^2}\exp\Big(-\int_s^{\sigma(e)}\frac{c(\alpha)\d
y}{(e_y+z)^2}\Big)
\, .
\end{eqnarray*}
Now we use Bismut's
decomposition of the Brownian excursion  \cite[Theorem 4.5, Chap. XII, p502]{RY99} to obtain:
\begin{eqnarray*}
  \mu((z,\infty))&=&2\int \mathbf{n}(\d e) \int_0^{\sigma(e)}\d s\,
  \frac{c(\alpha)}{(e_s+z)^2}\exp\Big(-\int_s^{\sigma(e)}\frac{c(\alpha)\d
    y}{(e_y+z)^2}\Big)\\
  &=&2c(\alpha)\int_0^\infty \frac{\d a}{(a+z)^2}\E_a\Big[\exp\Big(-
  \int_0^{T_0}\frac{c(\alpha)\d
    y}{(B_y+z)^2}\Big)\Big]\\
&=&2c(\alpha)\int_0^\infty \frac{\d a}{(a+z)^2}\E_{a+z}\Big[\exp\Big(-
  \int_0^{T_z}\frac{c(\alpha)\d
    y}{B_y^2}\Big)\Big]\, ,
\end{eqnarray*}
where $T_z=\inf\{t\geq 0:~B_t=z\}$. The computation of the expectation inside the above integral is standard. Specifically, an  application of Itô's formula 
entails that the process 
$$  M_t:=\left(\frac{a+z}{B_{t \wedge T_z}}\right)^{\alpha-1} \exp\left( - \frac{\alpha (\alpha-1)}{2} \int_0^{t \wedge T_z} \frac{ \mathrm{d}s}{ B_s^2}\right)$$ 
is a local martingale (for the filtration generated by $B$). Since, it is bounded above by $ (\frac{a+z}{z})^{\alpha-1}$, we infer from  the optional sampling theorem and letting $t \to \infty$  that 
$$\E_{a+z}\Big[\exp\Big(-
  \int_0^{T_z}\frac{c(\alpha)\d
    y}{B_y^2}\Big)\Big] =  \left( \frac{z}{a+z}\right)^{\alpha-1}.$$
     We conclude  that: 
$$\mu((z,\infty))=\frac{2c(\alpha)}{\alpha}\cdot\frac{1}{z}=\frac{\alpha-1}{z}\, .$$
Recalling \eqref{eq:mathscr:I}, our proposition is now  a direct consequence of   Shepp's covering theorem, as it appears in \cite[Corollary 1]{FiFrSh85} or more explicitly in \cite[Theorem 7.2]{Ber99}, which states that the set \eqref{eq:mathscr:I} is the range of a stable subordinator of index $2-\alpha$.
\end{proof}
Since the range of a $(2-\alpha)$-stable subordinator contains arbitrarily small values, by invariance by translation and an obvious symmetry, the previous result shows that {local} left (or right) minimal record of $Z$ are actually dense within loop times under $ \mathbf{Q}$.  We complete this picture by showing that there are no two-sided records on loops:
\begin{lem} \label{lem:twosidedloops} $ \mathbf{Q}$-a.s.,  for every  $ \varepsilon>0$ we have $\{t > 0 : \mathrm{b}_{0}(\ell_i)+\inf Z^{i} \geq  \mathrm{b}_{0}(t), \forall \ell_i \in [t- \varepsilon, t+ \varepsilon]\} = \varnothing.$
\end{lem} 
\begin{proof}
First remark that by Brownian motion's invariance by translation, combined with the fact that $\mathcal{P}$ is a Poisson measure with intensity $\mathbbm{1}_{t\geq 0} \d t \otimes \mathbf{N}(\d X \d Z)$ and the density of rational numbers, it suffices to establish that, for every $\varepsilon>0$,  the set 
\begin{equation}\label{eq:B:Z:i:B:t:i:t:varepsilon}
\big\{t > 0 : \mathrm{b}_{0}(\ell_i)+\inf Z^{i} \geq  \mathrm{b}_0(t), \forall \ell_i \in [0, t+ \varepsilon]\big\}
\end{equation} is empty $\mathbf{Q}$-a.s. To this end, fix $\varepsilon>0$ and, as in the previous proof, write $(s_k,t_k)_{k\geq 1}$ for an indexation of the connected components of the open set $\{t\geq 0:~\mathrm{b}_0(t)>\inf_{[0,t]} \mathrm{b}_0\}$. Since the point $0$ is regular recurrent for $\mathrm{b}_0(t)-\inf_{[0,t]} \mathrm{b}_0$, $t\geq 0$, we infer that any time in  \eqref{eq:B:Z:i:B:t:i:t:varepsilon} must be of the form $s_k$, for some $k\geq 1$. Now remark  that a point $s_k$, for $k\geq 1$, belongs to \eqref{eq:B:Z:i:B:t:i:t:varepsilon} if and only if the quantity
$$M_{k}(\varepsilon):=\inf\big\{\mathrm{b}_0(\ell_i)+\inf Z^{i}:~\ell_{i}\in [s_k,(s_k+\varepsilon)\wedge t_k]\big\} - \mathrm{b}_0(s_k)$$
is non negative. Lastly, with the notation of the previous lemma, an application of excursion theory gives:
$$\mathbf{Q}\Big[\sum_{k\geq 1} \mathbbm{1}_{M_k(\varepsilon)>0}\Big]=2\int_0^{\infty} \mathrm{d}t ~\int \mathbf{n}(\d e) \exp\big(-\int_0^{\sigma(e)\wedge \varepsilon} \d y \frac{\alpha (\alpha-1)}{2e_y^2} \big) .$$
The previous display is null since $\int_0^\cdot \d y ~e_y^{-2}=\infty$, $\mathbf{n}(\mathrm{d} e)$-a.e. This implies that the set \eqref{eq:B:Z:i:B:t:i:t:varepsilon} is empty and concludes the proof.
\end{proof}

We now use the ideal model under $ \mathbf{Q}$ to deduce the analog result under $ \mathbf{P}$:
\begin{prop}[Records on loops] \label{prop:records-loops}
$\mathbf{P}$-almost surely, the following properties hold:
\begin{itemize}
\item[$\mathrm{(i)}$] the sets $ \mathrm{LeftRec} \cap \Pi_d^{-1}(\mathrm{Loops})$ and $ \mathrm{RightRec} \cap  \Pi_d^{-1}(\mathrm{Loops})$ are dense in $ \Pi_d^{-1}(\mathrm{Loops})$;
\item[$\mathrm{(ii)}$]  $ \mathrm{LeftRec} \cap \mathrm{RightRec} \cap  \Pi_d^{-1}(\mathrm{Loops})$ is empty;
\item[$\mathrm{(iii)}$]  if $\alpha \in [3/2,2)$ then $Z$ is injective on $(\mathrm{LeftRec} \cup \mathrm{RightRec}) \cap  \Pi_d^{-1}(\mathrm{Loops})$, whereas if $\alpha < 3/2$, one can  find times $0<t<t'<1$ belonging to $ \Pi_d^{-1}(\mathrm{Loops})$ which are not identified in the looptree, i.e.\ $d(t,t')>0$, but for which 
$$ Z_{t}=Z_{t'} = \inf_{[t,t']}Z.$$
\end{itemize}
\end{prop}
\begin{proof}
As usual, we are going to establish the statement of the proposition under the excursion measure $\mathbf{N}$, the desired result under $\mathbf{P}$ then follows directly by scaling.   Let $\eps>0$ and $T$ a $(\mathcal{F}_t)_{t\geq 0}$--stopping time taking values in $\{s\geq 0:~\Delta_s\geq \eps\}$ -- if this set is empty we take $T=\infty$ by convention. Under  $\mathbf{N}(\cdot | T<\infty,\Delta_T)$, recall the definition of  $\mathcal{N}^{[T]}:=\sum \limits_{i\in \mathbb{N}}\delta_{X_{u_{i}},X^{i},Z^{i}}$
given in \eqref{N:point:measure}, and set:
$$\mathcal{P}^{[T]}:=\mathop{\sum_{i\in \mathbb{N},~X_{u_i}>X_T-\Delta_T}}\delta_{X_T-X_{u_{i}},X^{i},Z^{i}}$$
and  $B_T(t):=Z_{\mathrm{f}_{T}(t/\Delta_T)}-Z_{\rm{f}_{T}(0)},$ for $t\in[0,\Delta_T]$.   We can now apply the Markov property, as stated in Corollary \ref{cutting_N},  and Proposition \ref{b_Brownian_Brigde}, to deduce that under  $\mathbf{N}(\cdot | T<\infty,\Delta_T)$, the measure  $\mathcal{P}^{[T]}$ is a Poisson point measure with intensity $\mathbbm{1}_{[0,\Delta_T]}(t)\d t \mathbf{N}(\d X \d Z)$, and that the process $B_{T}$ haw law $\P_{0\to 0}^{(\Delta_T)}(\mathrm{d} B)$. Moreover, $\mathcal{P}^{[T]}$ and $B_T$ are independent since we have conditioned on $\Delta_T$. Now remark that the set 
$$\Big\{Z_{\rm{f}_T(r)}:~ r\in(0,1/2]\text{ such that } Z_{\rm{f}_T(r)}=\inf_{s\in [\mathrm{f}_T(0),\mathrm{f}_T(r)]} Z_s \Big\}$$
is exactly
$$Z_{\rm{f}_{T}(0)}+\Big\{B_{T}(t):~t\in (0,\Delta_T/2] \text{ such that }B_{T}(\ell_i)+\inf Z^{i}> B_{T}(t) \text{ for every } i\in \mathbb{N} \text{ with } \ell_i<t \Big\},$$
where, by analogy with the notation under $\mathbf{Q}$, we write $\ell_i:=X_T-X_{u_{i}}$. By absolute continuity of the law of the Brownian bridge over $[0, \Delta_{T}/2]$ with respect to that of Brownian motion, and the fact that the range of the $2-\alpha$ stable subordinator contains arbitrarily small values, we deduce from Proposition~\ref{I_infty} that there exist left-local minima of $Z$ belonging to  $ \mathrm{f}_{T}([0,1])$ that are arbitrarily close to $ \mathrm{f}_{T}(0)$. By translation invariance of Brownian motion and Brownian bridge, the argument extends to any point  $ \mathrm{f}_{T}(r)$, for fixed  $r \in [0,1]$ and proves that $ \mathrm{LeftRec}$ is dense within $ \mathrm{f}_{T}([0,1])$.   The same absolute continuity argument together with Lemma~\ref{lem:twosidedloops} shows that there are no two-sided records of $Z$ in $ \mathrm{f}_{T}([0,1])$. Since $\varepsilon$ can be taken as close to $0$ as wanted, this gives points (i) for  $\mathrm{LeftRec}$ and (ii). But point (i)  for $ \mathrm{RightRec}$ follows directly from the analog result for $\mathrm{LeftRec}$ and the invariance by time-reversal \eqref{symmetric}.
\par
Let us move on to the third point. In this direction, under  $\mathbf{N}(\cdot | T<\infty,\Delta_T)$, remark that
$$\Big\{z:~\exists (r,r^\prime)\in [0,1/4)\times (1/4,1/2] \text{ such that }z=Z_{\rm{f}_T(r)}=Z_{\rm{f}_{T}(r^{\prime})}=\inf\limits_{[\rm{f}_T(r),\rm{f}_{T}(r^\prime)]}Z\Big\}$$
can be written as the intersection of:
$$\Big\{B_{T}(t):~t\in [0,\Delta_T/4) \text{ such that }B_{T}(\ell_i)+\inf Z^{i}> B_{T}(t) \text{ for every } i\in \mathbb{N} \text{ with }  t<\ell_i\leq \Delta_T/4 \Big\} $$
and
$$\Big\{B_{T}(t):~t\in (\Delta_T/4,\Delta_T/2] \text{ such that }B_{T}(\ell_i)+\inf Z^{i}> B_{T}(t) \text{ for every } i\in \mathbb{N} \text{ with } \Delta_{T}/4<\ell_i<t \Big\}, $$
shifted again by $Z_{\rm{f}_{T}(0)}$. Since, under $\mathbf{Q}$ and by Proposition \ref{I_infty} combined with classical results from regenerative set theory (see \cite[Example 1]{hawkes1977intersections}), the intersection of two independent copies of $-\mathscr{B}$ is the range of a $(3-2\alpha)$-stable subordinator when $\alpha < 3/2$ and empty otherwise, we infer from the translation invariance of Brownian motion that the intersection is non-empty if and only if $\alpha < 3/2$. Here, we again use the fact that the range of a stable subordinator contains arbitrarily small values. Since this holds for every $\varepsilon > 0$, we deduce statement (iii) when $\alpha \in (1, 3/2)$. For the remainder of the proof, we assume that $\alpha \in [3/2, 2)$. As the argument parallels the previous cases, we omit some details. First, note that by re-rooting invariance \eqref{re-rooting}, it suffices to show that for every set of rational numbers $p_1, p_2, p_1^\prime, p_2^\prime$ in $(0,1)$ with $p_1 < p_2$ and $p_1^\prime < p_2^\prime$, and  jumping times  $t, t^\prime$ of $X$ such that $\mathrm{f}_t(p_2)<\mathrm{f}_{t^\prime}(p_1^\prime)$, the process $Z$ does not take the same value on 
$$\mathrm{f}_t((p_1,p_2))\cap (\mathrm{LeftRec}\cup \mathrm{RightRec})  \quad \text{ and }\quad \mathrm{f}_{t^\prime}((p_1^\prime,p_2^\prime))\cap  (\mathrm{LeftRec}\cup \mathrm{RightRec}).$$
To this end, consider a second $(\mathcal{F}_t)_{t \geq 0}$-stopping time $T^\prime$ taking values in the set $\{ s \geq 0 : \Delta_s \geq  \varepsilon \}$ such that either $T^\prime = T$ (to handle the case $t = t^\prime$) or $T^\prime$ is measurable with respect to $X_{t + \mathrm{f}_T(p_2)} - X_{\mathrm{f}_{T}(p_2)}$ for $t \geq 0$ (to apply the Markov property and treat the case $t\neq t^\prime$). Using the above arguments, we infer that the set
$$\big\{Z_{u} :  \ u \in \mathrm{f}_{T}((p_1,p_2))\cap \mathcal{R}\big\} \cap \big\{Z_{u} : \ u \in \mathrm{f}_{T'}((p_1^\prime,p_2^\prime))\cap \mathcal{R}^\prime\big\},$$
where $\mathcal{R}, \mathcal{R}^\prime \in \{ \mathrm{LeftRec}, \mathrm{RightRec} \}$, has an absolutely continuous law with respect to that of two (pieces of) ranges of independent $(2 - \alpha)$-stable subordinators (excluding the origin), each shifted by a random variable independent of the ranges. Hence, since $2 \cdot (2 - \alpha) \leq 1$ when $\alpha \in [3/2, 2)$, this intersection is empty almost surely by \cite[Example 1]{hawkes1977intersections}. Point (iii) for $\alpha\in[3/2,2)$   now follows  letting $\varepsilon \to 0$.
\end{proof}

\section{Neighborhood of the point with minimal label along the spine} \label{sec:prison}
The goal of this section is to derive a technical estimate for the process $(X,Z)$, which will later be translated into a geometric estimate for typical points  on geodesics in the scaling limits of random planar maps. The results of this section will be used exclusively in Section \ref{sub:goodpointestimate}, and the remainder of this work can be read independently.

In order to state the main result of this section, let us introduce some notation. In this direction, recall the definition of the height process $H$ and  of $\xi_t$ given in \eqref{xi:sup} and extended by scaling in Section~\ref{sec:Markov}, in particular recall from Corollary~\ref{cor:labelbranch} that the process $s \mapsto Z_{\xi_{t}(s)}$  encodes the labels along the pinch points of $ \mathrm{Branch}(0,t)$ (a closure  is required to include the labels at the origin of loops).  
Introduce the measure $ \mathbf{N}^{\bullet}$ defined by: 
\begin{equation}\label{def:N:bullet}
\mathbf{N}^{\bullet}\big(F\big(X,Z, t_\bullet\big)\big)=\mathbf{N}\Big(\int_0^{\sigma}\d t~F\big(X,Z,t\big)\Big),
\end{equation} 
 which can be thought of as the measure $ \mathbf{N}$ together with a uniform distinguished time $t_{\bullet}$.
In this section we shall work under $\mathbf{N}^\bullet$. We also introduce  $Y^\bullet_s:=Z_{\xi_{t_\bullet}(s)}$, $s\in [0,H_{t_\bullet}]$,  the process of labels  along the branch connecting $0$ and $t_\bullet$. Recall from  \eqref{one_point} that, under $\mathbf{N}^\bullet$, we have the following properties:
\begin{itemize}
\item The ``distribution'' of $H_{t_{\bullet}}$ is  $\mathbbm{1}_{h\geq 0}\d h$;
\item Conditionally on  $H_{t_{\bullet}}=h$, the process $Y^\bullet$ is  a $2(\alpha-1)$- symmetric stable L\'evy process stopped at time $h$.
\end{itemize}
By standard properties of stable processes (see Lemma \ref{lem:onePinch}), there exists a.s.~a unique time $\varpi \in (0,H_{t_\bullet})$ such that $Y^{\bullet}_{\varpi}=\min Y^{\bullet}$ and $\varpi$ is not a jumping time of $Y^\bullet$. By the classification of points in the looptree (Section \ref{sec:defloop}), we infer  that ${\varpi}_1=\xi_{t_\bullet}(\varpi)$ is a pinch point time and let us consider ${\varpi}_2$ the unique element strictly larger than ${\varpi}_1$ such that $d({\varpi}_1,{\varpi}_2)=0$. Now, let us introduce the event that we aim to study in this section, see also Figure \ref{fig:trapped} for an illustration. For $ \varepsilon>0$, we say that the point $ \mathfrak{X}=\Pi_{d}(\varpi_{1}) = \Pi_{d}(\varpi_{2})$ is \textbf{$ \varepsilon$-trapped}  if there exist $0 < s_1<{\varpi}_1 <  s_2 <  t_\bullet <  s_3 < {\varpi}_2 < s_4 < \sigma$  
 such that:

\begin{figure}[!h]
 \begin{center}
 \includegraphics[width=16cm]{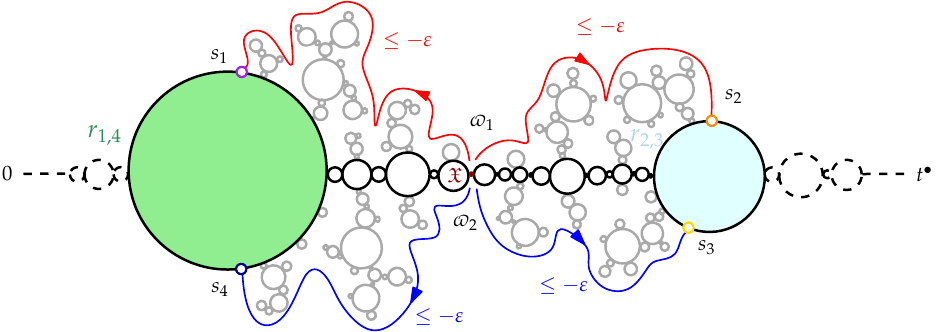}
 \caption{ Illustration of an $ \varepsilon$-trapped point $\mathfrak{X}$ (in dark red on the figure). The spine between $0$ and $t_\bullet$ is made of the black loops. When following the looptree from the point $\mathfrak{X}$  in one of the four directions, one encounters a (one-sided) minimal record time of $Z$ corresponding to a one-side minima below $Z_{\varpi_{1}}-  \varepsilon$ on two loops of the spine. \label{fig:trapped}}
 \end{center}
 \end{figure}
\noindent There exist $0 < s_1<{\varpi}_1 <  s_2 <  t_\bullet <  s_3 < {\varpi}_2 < s_4 < \sigma$  
 such that:
\begin{itemize}
\item $\max(Z_{s_1},Z_{s_2},Z_{s_3},Z_{s_4})<Z_{{\varpi}_1}-\eps.$
\item  $Z_{s_1}<Z_{t}$  for every $t\in(s_1,{\varpi}_1)$, \\
 $Z_{s_2}<Z_{t}$  for every $t\in({\varpi}_1,s_2)$, \\
  $Z_{s_3}<Z_{t}$ for every $t\in(s_3,{\varpi}_2)$\\ $Z_{s_4}<Z_{t}$ for every $t\in({\varpi}_2,s_4)$.
\item  The images of $s_1,s_4$ and $s_2,s_3$ belong to a same loop of the spine, i.e. there exists $r_{1,4}\prec  {\varpi}_1$  verifying that 
$$X_{s_1}<X_t \ \ \mbox{ for every } t \in(r_{1,4},s_{1}), \quad \text{ and } \quad X_{s_4}<X_t \ \ \mbox{ for every } t \in(r_{1,4},s_{4}),$$
and there exists $ {\varpi}_1\prec r_{2,3}\prec t_\bullet$  verifying that 
$$X_{s_2}<X_t \ \ \mbox{ for every } t \in ( r_{2,3},s_2), \quad \text{ and } \quad X_{s_3}<X_t \ \ \mbox{ for every } t \in (r_{2,3},s_3).$$
\end{itemize}

This section is devoted to the proof of the following uniform bound:
\begin{theo}[Polynomial tail for not being trapped]\label{prop:malo}
There exists $C,c>0$, such that:
$$ \mathbf{N}^\bullet\Big(\mathfrak{X}  \textnormal{ is } \varepsilon-\textnormal{trapped}~\Big|~ H_{t_\bullet}=h\Big)\geq 1-C\cdot  \Big(\frac{\eps}{h^{2(\alpha-1)}}\Big)^c, \quad \text{ for } h,\eps>0.$$
\end{theo}
We make a couple of remarks. First, the conditional version of $\mathbf{N}^\bullet$ with respect to $H_{t_\bullet}$ is well-defined by scaling and disintegration. Second, again by scaling, it suffices to establish the estimate for $ h=1$ (or $ \varepsilon=1$), and third, although our application will require the existence of the four times $s_1,s_2,s_3,s_4$, it is sufficient by Lemma \ref{lem:invariancereroottime} to prove the analogous  estimate for the one-sided event where we only require the existence of the times $s_2,s_3$ in the definition above, in which case we say that \textbf{$ \mathfrak{X}$ is $ \varepsilon$-trapped from the right}. Before embarking on the proof, let us sketch the underlying idea: we already proved in Proposition \ref{prop:records-loops} that one-sided minimal record times of the process $Z$ may happen on loops so that the reader should be convinced that, at least, the event in the above theorem has a positive probability to happen. Actually for each scale $ \varepsilon$, the point $\mathfrak{X}$  has a probability bounded away from $0$ to be $ \varepsilon$-trapped at this scale, i.e., with times $s_1,s_2, s_3 , s_4$ satisfying $Z_{ \varpi_1} - Z_{s_i}  \in [\varepsilon, 2 \varepsilon]$. If scales were independent, we would naturally get that the probability that $ \mathfrak{X}$ is not trapped at all scales between  $\varepsilon$ and $1$ should decay at least as $ \varepsilon^c$ for some $c>0$, which is the statement of the theorem. The rest of this section will implement this heuristic rigorously since the behaviors at different scales are  not independent.

We divide the proof of Theorem \ref{prop:malo} in three parts. In Section \ref{sec:spinal}, we study the law of the spine  towards $t_\bullet$, that is informally the labeled-loops containing points of $\Pi_d(\mathrm{Branch}(0,t_\bullet))$, and of the labeled looptrees that are attached on the left and right sides of it, this will lead to a spinal decomposition -- see \cite{Spine} for similar results in Brownian geometry and \cite[Section 6]{RRO:2024} for an analog decomposition in the setting of Markov processes indexed by Lévy trees. In Section \ref{sec:comparison:ideal}, we compare the distribution of the spine with an ideal model in which the analog of  Theorem \ref{prop:malo} is easier to establish.  Finally in Section \ref{sec:technical}, we prove the technical lemmas used along the way.

\subsection{Spinal decomposition}\label{sec:spinal}
In this section we work under $\mathbf{N}^\bullet$. Our first goal is to evaluate the probability that  $\mathfrak{X}$  is $ \varepsilon$-trapped on the right \textit{after conditioning on the process} $( Y^\bullet_s:=Z_{\xi_{t_\bullet}(s)}, s \in [0, H_{t_{\bullet}}] )$ of the labels along the spine. Informally, we first describe  the law of the spine (i.e. of the loops and their labels connecting $0$ and $t_\bullet$ in the looptree) conditionally on $Y^\bullet$ and then argue that the labeled looptrees attached to this spine are, after subtracting the label of their root, independent and distributed according to  the standard excursion measure $ \mathbf{N}$. To proceed formally, we need some notation which we illustrate in Figure \ref{fig:spinou}.\\

By Lemma~\ref{Lem-cut-jump}, if $\Delta Y^\bullet_r\neq 0$ then $\xi_{t_\bullet}(r-)$ is a jumping time for $X$ and this defines a one-to-one correspondence between $\{r\geq 0:~\Delta Y^\bullet_r\neq 0\}$ and 
$\{r\in [0,\sigma]:~r\preceq t_\bullet\text{ and } \Delta X_r>0\} $; here we use that $\Delta X_{t_\bullet}=0$. Now we write:
$$\check{ \mathfrak{S}}_{r}:=X_{\xi_{t_\bullet}(r-)}-X_{\xi_{t_\bullet}(r)}\text{ and } \mathfrak{S}_{r}:= X_{\xi_{t_\bullet}(r)}-X_{\xi_{t_\bullet}(r-)-},$$
and remark that $\Delta X_{\xi_{t_\bullet}(r-)}=\check{\mathfrak{S}}_r+\mathfrak{S}_r$. We also introduce the associated labels along the corresponding loop:
$$ \check{B}_{s}^{(r)}:=Z_{\mathrm{f}_{\xi_{t_\bullet}(r-)}(\frac{s}{\check{\mathfrak{S}}_r+\mathfrak{S}_r})},\quad s\in[0,\check{\mathfrak{S}}_r],~~~ \text{ and }~~~ B_{s}^{(r)}:=Z_{\mathrm{f}_{\xi_{t_\bullet}(r-)}(\frac{\check{\mathfrak{S}}_r+\mathfrak{S}_r-s}{\check{\mathfrak{S}}_r+\mathfrak{S}_r})},\quad s\in[0,\mathfrak{S}_r]. $$
By convention, we also let $\check{B}_{s}^{(r)}$ (resp. $B_{s}^{(r)}$) equal to $Z_{\mathrm{f}_{\xi_{t_\bullet}(r-)}(\frac{\check{\mathfrak{S}}_{r}}{\check{\mathfrak{S}}_r+\mathfrak{S}_{r}})}$, for $s\geq \check{\mathfrak{S}}_r$ (resp. $s\geq \mathfrak{S}_r$). These describe what is happening on the loops connecting $0$ and $t_\bullet$. We also need to keep track of the labeled looptrees glued along this spine. In order to encapsulate this information,  for every $r$ with $\Delta Y^\bullet_r\neq 0$,  consider $(u_{i},v_{i})_{i\geq 1}$ the connected components of  the open set $$\Big\{t\in[\mathrm{f}_{\xi_{t_\bullet}(r-)}(0), \mathrm{f}_{\xi_{t_\bullet}(r-)}(1) ]:\: X_{t}>I_{\xi_{t_\bullet}(r-),t}\Big\}$$ and  
introduce the excursion processes
\[X_{s}^{i}:=X_{(u_{i}+s)\wedge v_i}-X_{u_{i}}\:\:\:\:\:\text{ and }\:\:\:\:\:Z_{s}^{i}:=Z_{(u_{i}+s)\wedge v_i}-Z_{u_{i}},\qquad s\geq 0.\]
Finally, we consider the point measures
\begin{equation}\label{eq:P_r}
\check{\mathcal{P}}_r:=\sum \limits_{i\geq 1:~u_{i}<\xi_{t_\bullet}(r)} \delta_{X_{\xi_{t_\bullet}(r-)}-X_{u_i},X^{i},Z^{i}}\:\:\:\:\text{ and }\:\:\:\:\mathcal{P}_r:=\sum \limits_{i\geq 1:~u_{i}>\xi_{t_\bullet}(r)} \delta_{X_{u_i}-X_{\xi_{t_\bullet}(r-)-},X^{i},Z^{i}}.
\end{equation}
We stress that the excursion $(X^{i},Z^{i})$ associated with the interval $[u_i,v_i]$ containing $t_\bullet$ is not taken into account in $\check{\mathcal{P}}_r$ nor in $\mathcal{P}_r$. For definiteness, if $r$ is not a jumping time we simply take $\check{\mathcal{P}}_r=\mathcal{P}_r=0$. Recall from Section \ref{sec:twopoints} the notation $ \mathbb{P}^{(r)}_{a \to b}$ for the law of a Brownian bridge going from $a$ to $b$ with lifetime $r$.

\begin{figure}[!h]
 \begin{center}
 \includegraphics[width=13cm]{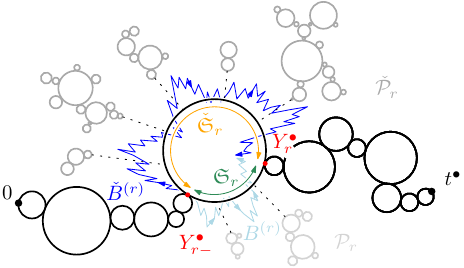}
 \caption{Illustration of Proposition \ref{prop:spine}: construction of the spine (and the dangling looptrees) from the process $Y^{\bullet}$. \label{fig:spinou}}
 \end{center}
 \end{figure}
\begin{prop}[Spinal decomposition]\label{prop:spine}
Under $\mathbf{N}^{\bullet}$ and conditionally on $r\mapsto Y_r^\bullet$, the collection of random variables:
$$\Big(\check{\mathfrak{S}}_r, \mathfrak{S}_r, \check{B}^{(r)}, B^{(r)}, \check{\mathcal{P}}_{r}, \mathcal{P}_r \Big), \quad \text{ for every } r \text{ with } \Delta Y_r^\bullet\neq 0,$$
 are independent and their conditional distribution can be determined  as follows. First, the law of $\big(\check{\mathfrak{S}}_{r},\mathfrak{S}_{r}\big)$ is:
$$\mathbbm{1}_{\ell_1,\ell_2>0}C(\alpha) \cdot |\Delta Y_r^\bullet|^{2\alpha-1} \frac{(\ell_1+\ell_2)^{-\alpha-\frac{1}{2}}}{\sqrt{2\pi\ell_1 \ell_2} }\exp\left(-(\Delta Y_r^\bullet)^{2}\frac{\ell_1+\ell_2}{2\ell_1\ell_2}\right) \d \ell_1 \d \ell_2,$$
where $C(\alpha):=2^{\alpha-1}(2\alpha-1)/\Gamma(\alpha)$ is a normalization constant.   Then conditionally on the family $\big((Y^\bullet_r,\check{\mathfrak{S}}_r,\mathfrak{S}_r):~r\geq 0\big)$, all the random variables $\big((\check{B}^{(r)}, B^{(r)}, \check{\mathcal{P}}_{r}, \mathcal{P}_r):~r\geq 0\big)$ are independent  and for every $r$ jumping time of $Y^\bullet$ we have:
\\
\\
$\bullet$ The law of $\big( \check{B}^{(r)}, B^{(r)}\big)$ is  $\P_{Y^\bullet_{r-}\to Y^\bullet_{r}}^{(\check{\mathfrak{S}}_r)}\otimes \P_{Y^\bullet_{r-}\to Y^\bullet_{r}}^{(\mathfrak{S}_r)}\big(\d \check{B}, \d B\big);$\\
\\
$\bullet$ The measures $\big( \check{\mathcal{P}}_{r}, \mathcal{P}_r\big)$ are two independent Poisson point measures, also independent of $(\check{B}_r,B_r)$, with respective intensities:
$$\mathbbm{1}_{[0,\check{\mathfrak{S}}_r]}(t) \d t \mathbf{N}(\d X, \d Z)  \text{ and } \mathbbm{1}_{[0,\mathfrak{S}_r]}(t) \d t \mathbf{N}(\d X, \d Z).$$ 
\end{prop}

\begin{proof}
We start by introducing the measure:
\begin{equation}\label{M:bullet:mesure:spine}
\sum \limits_{r\geq 0,~\Delta Y^\bullet_r\neq 0}\delta_{r; ~\check{\mathfrak{S}}_r;~ \mathfrak{S}_r;~ \Delta Y^\bullet_r}. 
\end{equation}
In order to describe the law of \eqref{M:bullet:mesure:spine}, we let $\mathfrak{M}$ be the distribution of   a Poisson point measure, say $\mathscr{M}(\d r, \d \ell_1, \d \ell_2, \d y)$, with intensity:
$$\mathbbm{1}_{r,\ell_1,\ell_2> 0}\d r \,\d \ell_1\, \d \ell_2\, \d y \frac{1}{\Gamma(-\alpha)} \frac{(\ell_1+\ell_2)^{-\alpha-\frac{1}{2}}}{\sqrt{2\pi \ell_1 \ell_2}}\exp(-y^{2}\frac{\ell_1+\ell_2}{2\ell_1\ell_2}),$$
and for every $a>0$, we write $\mathfrak{M}_a$ for  the distribution of the restriction 
$$\mathscr{M}_a(\d r, \d \ell_1, \d \ell_2, \d y):=\mathbbm{1}_{[0,a]}(r)\mathscr{M}(\d r, \d \ell_1, \d \ell_2, \d y).$$ 
Next, we notice that by Proposition \ref{b_Brownian_Brigde}, conditionally on  the point measure:
$$\sum \limits_{r\geq 0,~\Delta Y^\bullet_r\neq 0}\delta_{r; ~\check{\mathfrak{S}}_r;~ \mathfrak{S}_r}$$
the random variables  $ (\Delta Y^\bullet_r:~s\geq 0 \text{ with }~\Delta Y^\bullet_s\neq 0)$ are independent  and the distribution of  $\Delta Y^\bullet_r$ is a centered Gaussian random variable with variance $\check{\mathfrak{S}}_r \mathfrak{S}_r/(\check{\mathfrak{S}}_r +\mathfrak{S}_r) $. Furthermore, the point measure $\sum \limits_{r\geq 0,~\Delta Y^\bullet_r\neq 0}\delta_{r; ~\check{\mathfrak{S}}_r;~ \mathfrak{S}_r}$ has already been studied in the context of Lévy trees  and is distributed as  $$\int_{0}^{\infty} \d a\int_{y\in \mathbb{R}} \mathscr{M}_a(\d r, \d \ell_1, \d \ell_2, \d y),$$
we refer to \cite[Proposition 3.1.3]{DLG02} for a proof. Putting all together, we have obtained that the distribution of  \eqref{M:bullet:mesure:spine} under $\mathbf{N}^\bullet$ is precisely $\int_{0}^{\infty}\d a ~ \mathfrak{M}_a$. Consequently, if we condition on $(r, \Delta Y^{\bullet}_r)_{r\geq 0}$, the random variables $(\check{\mathfrak{S}}_r,\mathfrak{S}_r)_{r\geq 0}$ become independent with common distribution:
$$\mathbbm{1}_{\ell_1,\ell_2>0}C(\alpha) \cdot |\Delta Y_r^\bullet|^{2\alpha-1} \frac{(\ell_1+\ell_2)^{-\alpha-\frac{1}{2}}}{\sqrt{2\pi\ell_1 \ell_2} }\exp(-(\Delta Y_r^\bullet)^{2}\frac{\ell_1+\ell_2}{2\ell_1\ell_2}) \d \ell_1 \d \ell_2,$$
where:
$$C(\alpha)^{-1}:=\int_{\ell_1,\ell_2>0} \d \ell_1 \d \ell_2 \frac{(\ell_1+\ell_2)^{-\alpha-\frac{1}{2}}}{\sqrt{2\pi\ell_1 \ell_2} }\exp(-\frac{\ell_1+\ell_2}{2\ell_1\ell_2}).$$
A direct computation then gives the explicit expression $C(\alpha):=2^{\alpha-1}(2\alpha-1)/\Gamma(\alpha)$. 
Moreover, since $Y^{\bullet}$ is a L\'evy process without Brownian part, we can recover  the entire process  $Y^{\bullet}$ from the structure of its jumps $(r, \Delta Y^{\bullet}_r)_{r\geq 0}$. This gives us the first statement of the proposition. Let us now explain how to obtain the second part.  First we remark that the collection $(\mathcal{P}_r:~r\geq 0)$ can be recovered directly from the sequence $( \mathfrak{S}_{r}:~r\geq 0)$ and the point measure $\mathcal{N}^{[t_\bullet]}$ defined in~\eqref{N:point:measure}. Now we can apply the Markov property, as stated in Corollary   \ref{cutting_N}, to see that  conditionally on $X_{t_\bullet}$, the measure $\mathcal{N}^{[t_\bullet]}$ is a Poisson point measure with intensity $\mathbbm{1}_{[0,X_{t_\bullet}]}(p) \mathrm{d}p\:\mathbf{N}( \mathrm{d}X  \mathrm{d}Z)$ independent of $\big((Y^\bullet_r,\check{\mathfrak{S}}_r, \mathfrak{S}_r, \check{B}^{(r)}, B^{(r)}, \check{\mathcal{P}}_{r}):~r\geq 0\big) $. This implies the desired result for $(\mathcal{P}_r:~r\geq 0)$, since by property ({\hypersetup{linkcolor=black}\hyperlink{prop:A:3}{$A_3$}}) above Proposition \ref{topologie_loop_tree} we have $X_{t_\bullet}=\sum_{r\geq 0}  \mathfrak{S}_{r}$. Furthermore, by the definition of $\mathbf{N}^\bullet$ given in \eqref{def:N:bullet}  and the re-rooting property \eqref{re-rooting}, we obtain the same property for $(\check{\mathcal{P}}_r:~r\geq 0)$ -- replacing $\mathfrak{S}_r$ by $\check{\mathfrak{S}}_r$. Therefore to conclude it remains to show that, conditionally on $\big((Y^{\bullet}_r,\check{\mathfrak{S}}_r,\mathfrak{S}_r):~r\geq 0\big)$, 
the processes  $\big( (\check{B}^{(r)}, B^{{(r)}}):\:r\geq 0\big)$ are independent and the distribution of  $\big( \check{B}^{(r)},B^{(r)}\big)$ is 
$$\P_{Y^\bullet_{r-}\to Y^\bullet_{r}}^{(\check{\mathfrak{S}}_r)}\otimes \P_{Y^\bullet_{r-}\to Y^\bullet_{r}}^{(\mathfrak{S}_r)}\big(\d \check{B}, \d B\big), $$
but this is a direct consequence of  Proposition \ref{b_Brownian_Brigde}. 
\end{proof}

The idea to check if $\mathfrak{X} = \Pi_d(\varpi_1)= \Pi_d(\varpi_2)$ is $ \varepsilon$-trapped from the right is to explore the spine from $\mathfrak{X}$ to $\Pi_{d}( t_{\bullet})$ and for each loop encountered along the way, to check whether we can find two times $s_2,s_3$ on that loop satisfying the desired  condition, see also  Figure \ref{fig:trapped} for an illustration. In particular, if those times exist, their labels must be smaller than the labels encountered on all the looptrees attached on the respective sides of  the loops discovered so far. As in Section \ref{sec:recordsloops} this is conveniently encoded for each jump time $r \geq \varpi$ of $Y^{\bullet}$ by recording the minimal label attained by the looptrees attached on that loop. In order to formalize this, let us rewrite the point measures of \eqref{eq:P_r} in the form $\mathcal{P}_{r}=:\sum_{i\in \mathcal{I}_r}\delta_{t_i,X^i,Z^i}$, and $\check{\mathcal{P}}_r=:\sum_{i\in \check{\mathcal{I}}_r}\delta_{\check{t}_i,\check{X}^i,\check{Z}^i}$, then we set
 \begin{equation}\label{eq:def:R:r:-}\mathscr{R}_r:=\inf \big\{B_{t_i}^{(r)}+\inf Z^{i} :~ i\in \mathcal{I}_r \big\} \wedge \inf \big\{\check{B}_{\check t_i}^{(r)}+\inf \check{Z}^{i} :~ i\in \check{\mathcal{I}}_r \big\} - Y_{r-}^\bullet,  \end{equation}
 which is the (shifted) minimal label attained by a looptree attached to that loop,  as well as
  \begin{eqnarray}\label{eq:def:I:r:-}
\mathscr{I}_r&:=& \max\Big\{ \inf\big\{B_t^{(r)} - Y^{\bullet}_{r-}:~ \text{ s.t.  }B_{t_i}^{(r)}+\inf Z^{i}> B_{t}^{(r)}, \ \forall i\in \mathcal{I}_r \text{ with } 0 \leq t_i<t \big\},\nonumber \\
&& \qquad \inf\big\{\check{B}_t^{(r)} - Y^{\bullet}_{r-}:~ \text{ s.t. }\check{B}_{\check{t}_i}^{(r)}+\inf \check{Z}^{i}> \check{B}_{t}^{(r)}, \ \forall i\in \check{\mathcal{I}}_r \text{ with } 0 \leq \check{t}_i<t \big\}\Big\}
\end{eqnarray}
which corresponds to the smallest possible minimal record on each side of the loop, only taking into account the looptrees attached to that loop. By convention, if $\Delta Y^\bullet_r=0$, we set $(\mathscr{R}_r, \mathscr{I}_r):=(0,0)$.

\begin{figure}[!h]
 \begin{center}
 \includegraphics[width=14cm]{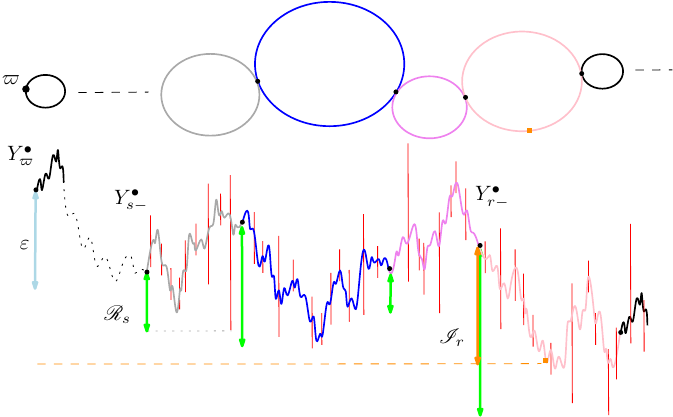}
 \caption{Illustration of the definition of the random variables $ \mathscr{R}_r$ and $ \mathscr{I}_r$. Here, only one side of the loops is displayed for visibility. The label on one side of the spine is obtained by concatenating the shifted process $B^{(r)}$ (in different colors above). The red slits correspond to the range of $Z$-values reached by the looptrees attached to the spine as in Figure \ref{fig:shepp}. The maximal negative displacements of those looptrees over a single loop is recorded in the random variables $ \mathscr{R}_\cdot$ and  are displayed by thicker green slits. If  \eqref{eq:trapped} holds, there is a point a the loop corresponding to a new minimal record of $Z$ (visible from the left, not blocked by the green slits) with label below $Y^{\bullet}_{\varpi} - \varepsilon$. A corresponding point exists on the otherside of the spine, thus trapping $\mathfrak{X}$ at level $  \varepsilon$.\label{fig:easytosee}}
 \end{center}
 \end{figure}
 
 The key  observation, being that if $ r > \varpi$ is a jump time of $ Y^{\bullet}$ satisfying
  \begin{eqnarray} \label{eq:trapped} Y_{r_{-}}^{\bullet} + \mathscr{I}_{r} \leq \inf_{ \varpi \leq s <r}\big( Y_{s_{-}}^{\bullet} + \mathscr{R}_{s} \big)\leq Y^{\bullet}_{\varpi} - \varepsilon,  \end{eqnarray}
then the loop corresponding to time $r$  supports two times $s_{2},s_{3}$ trapping the point $\mathfrak{X}$ at level $\varepsilon$. We will thus call such time $r$ an \textbf{$ \varepsilon$-trapping time}.  Notice that this is not an equivalence, since we do not split the contribution of each side in the definition of $ \mathscr{R}_{r}$.  See Figures \ref{fig:easytosee} and \ref{fig:exc-deco} for an illustration of these definitions.
 
   \begin{figure}[!h]
 \begin{center}
 \includegraphics[width=13cm]{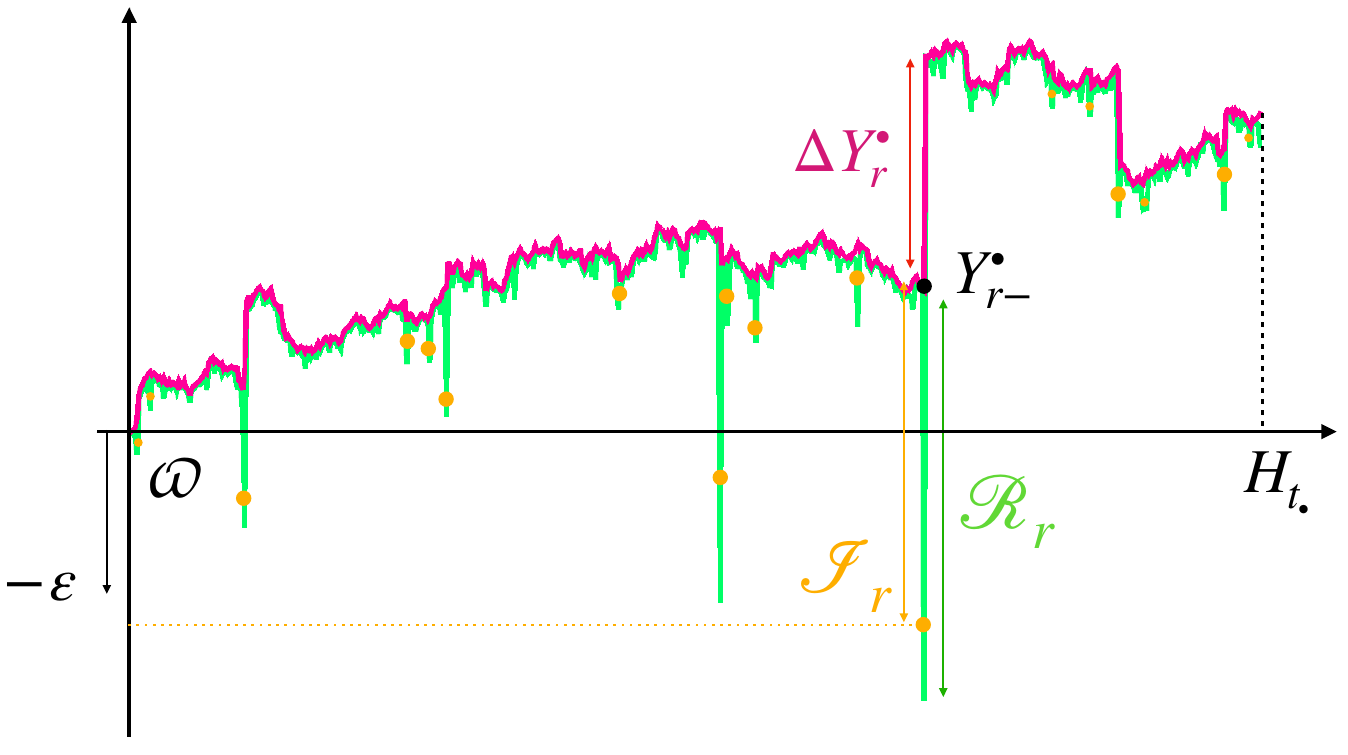}
\caption{Companion illustration to Figure \ref{fig:easytosee} where the contributions of each loop have been regrouped: The process $ Y^{\bullet}$ is in pink, the green sticks attached to each of its jumps correspond to the minimal label attained by looptrees attached to the corresponding loop of the spine, and the orange points correspond to the values $Y_{r-}^{\bullet} + \mathscr{I}_{r}$. If such an orange point is not shadowed by the previous green sticks (as illustrated by the horizontal orange dotted line above), then one can trap the point $ \mathfrak{X}$ using that loop. \label{fig:exc-deco}}
 \end{center}
 \end{figure}
 
 Thanks to Proposition \ref{prop:spine}, we see that, conditionally on $Y^{\bullet}$ the random variables $(\mathscr{R}_{r}, \mathscr{I}_{r})_{r \geq 0}$ are independent. We now formally describe the law of $( \mathscr{R}_{r}, \mathscr{I}_{r})$ given $ \Delta Y^{\bullet}_{r} $. In this direction, fix $a\in \mathbb{R}\setminus \{0\}$ and, under $\mathbb{P}_{0\to a}^{(\ell)}$ and independently of $B$,  consider a Poisson point measure $\mathcal{P}(\d t, \d X, \d Z)$ with intensity $\mathbbm{1}_{[0,\ell]}(t) \d t \mathbf{N}(\d X, \d Z).$ We also introduce the measure
\begin{align}\label{eq:Q:a:b}
\mathbb{Q}_{a}&\big(\d (\check{B}, \check{\mathcal{P}}), \d (B, \mathcal{P}) \big) \nonumber\\
&=C(\alpha) \cdot |b-a|^{2\alpha-1} \int_{0}^{\infty} \mathrm{d} \ell \int_{0}^{1} \d u ~ \frac{\ell^{-\alpha-\frac{1}{2}}}{\sqrt{2\pi u(1-u)} }\exp(-\frac{(b-a)^{2}}{u(1-u)\ell}) \P_{0\to a}^{(u\ell)}\otimes \P_{0\to a}^{((1-u)\ell)}.
\end{align}
By convention, we take $\mathbb{Q}_{0}:=\P^{(0)}_{0\to 0}\otimes \P_{0\to 0}^{(0)}$,   i.e.  $B$ and $\check{B}$ are constant processes equal to $0$ with $0$ lifetime and  $\mathcal{P}=\check{\mathcal{P}}=0$.  Under $\mathbb{Q}_a$, we define $\mathscr{R}$ and $\mathscr{I}$ replacing in \eqref{eq:def:R:r:-} and \eqref{eq:def:I:r:-} the random variables  $\big( (\check{B}^{(r)},\check{\mathcal{P}}_r) , (B^{(r)},\mathcal{P}_r)\big)$  by $\big( (\check{B},\check{\mathcal{P}}) , (B,\mathcal{P})\big)$ and $Y_r^\bullet$ by $0$.
We consider these measures because, by Proposition \ref{prop:spine}, under  $\mathbf{N}^{\bullet}$ and  conditionally on $\big(Y^\bullet_r:~r\geq 0\big)$,  for every $r\geq 0$,  the pair $\big(\mathscr{R}_r,\mathscr{I}_r\big)$ is distributed as $\big(\mathscr{R},\mathscr{I}\big)$ under $\mathbb{Q}_{\Delta Y^\bullet_{r}}$. We also  stress that that by scaling, for $a\neq 0$, the pair
\begin{equation}\label{scaling:a:Q}
\Big((|a|\cdot \mathscr{R},|a|\cdot \mathscr{I}):~\text{ under } \mathbb{Q}_{\mathrm{sign}(a)}\Big)\overset{(d)}{=} \Big((\mathscr{R}, \mathscr{I}):~\text{ under } \mathbb{Q}_{a}\Big)
\end{equation}
We conclude this section with some estimates under $\mathbb{Q}_a$, for $a\neq 0$. 

 \begin{lem}\label{lem:tech:Q}
The following properties hold:
\begin{itemize}
\item[$\mathrm{(i)}$] For every $x>0$, we have $\mathbb{Q}_{1}\big( \mathscr{I}<-x\big)>0$ and $\mathbb{Q}_{-1}\big( \mathscr{I}<-x\big)>0$.
\item[$\mathrm{(ii)}$]  There exists $C_{\mathbb{Q}}>0$, such that for every $x>0$ we have:
$$\mathbb{Q}_{\pm 1}\big( \mathscr{R}<-x\big)\leq C_{\mathbb{Q}}  \cdot q_\alpha\Big(\frac{1}{x}\Big),$$
where $q_\alpha:\mathbb{R}_+\to \mathbb{R}_+$ is the function
\begin{equation*}
 q_\alpha(y)= 
\begin{cases}
y^{2\alpha-1} &\text{if }\alpha\in(1,3/2),\\
y^{2}\cdot \log\big(2+y \big) &\text{if }\alpha=3/2,\\
y^{2}&\text{if }\alpha\in(3/2,2).
\end{cases}
\end{equation*}
\end{itemize}
\end{lem}

\begin{proof} 
Point (i) follows straightforwardly from Proposition \ref{I_infty} (where we showed that the minimal record values on an infinite ideal loop are distributed as the range of a stable subordinator) and standard absolute continuity between Brownian bridges and Brownian motion as in the proof of Proposition \ref{prop:records-loops}.
Let us proceed with the proof of  the second item which is more involved. In this direction, remark that, for $z>0$, we have:
\begin{eqnarray*}
\P_{0\to 1}^{(z)}(\mathscr{R}<-x)&=&1-
 \mathbb{E}_{0\to 1}^{(z)}\Big[\exp\Big(-\int_{0}^{z}\mathbf{N}(\sup Z\geq B_u+x) \d u\Big)\Big]\\
 &\underset{  \mathrm{Cor. \ }\  \ref{Z<-1}  \ \& \ \mathrm{Prop. \ } \ref{sec:stable-map-3}}{=} &1-\mathbb{E}_{0\to 1}^{(z)}\Big[\exp\Big(-\int_{0}^{z}\frac{\alpha(\alpha-1)~\d u}{2(B_u+x)^{2}}\Big)\Big]\\ 
 &\underset{ \mathrm{Lem. \ } \ref{lem:E:a:to:b}}{=}& g\left( \frac{x(x+1)}{z} \right),
 \end{eqnarray*}
 where $g(p) :=  1-\sqrt{2\pi  p}\cdot \exp(-p)\cdot I_{\alpha-1/2}(p)$. It then follows from the definition of $\mathbb{Q}_{1}$, given in \eqref{eq:Q:a:b}, that
  \begin{eqnarray*} 
 \mathbb{Q}_{1}( \mathscr{R}< -x)   & \leq & C(\alpha) \int_0^{\infty}\int_0^\infty \mathrm{d}\ell_{1} \mathrm{d}\ell_{2} \ \frac{(\ell_{1}+ \ell_{2})^{-\alpha - \frac{1}{2}}}{ \sqrt{2 \pi \ell_{1}\ell_{2}}} \mathrm{e}^{- \frac{\ell_{1} + \ell_{2}}{ 2 \ell_{1}\ell_{2}}} \Big(g\big( \frac{x(x+1)}{\ell_{1}}\big) + g\big( \frac{x(x+1)}{\ell_{2}}\big) \Big) \\ 
  &\leq &   2 C(\alpha) \cdot \int_{0}^{\infty} \mathrm{d} \ell ~ \frac{\ell^{-\alpha-\frac{1}{2}}}{\sqrt{2\pi} } g\big(\frac{x^2}{\ell} \big) {\int_{0}^{1} \d u ~ u^{-\frac{1}{2}} (1-u)^{-\frac{1}{2}}\exp(-\frac{1}{u(1-u)\ell})},
   \end{eqnarray*}
   where to obtain the second line we used that since $ g : \mathbb{R}_{+} \to [0,1]$ is non-increasing, we have $g( \frac{x(x+1)}{\ell_{i}}) \leq g(\frac{x^2}{\ell_1+\ell_2})$, for $i\in\{1,2\}$. Next, a  computation shows that:
\begin{equation}\label{mellin:erfc}
\int_{0}^{1} \d u ~ u^{-\frac{1}{2}} (1-u)^{-\frac{1}{2}}\exp(-\frac{1}{u(1-u)\ell})=\pi \cdot \mathrm{erfc}\big(2\cdot \ell^{-1/2}\big),
\end{equation}
where $\mathrm{erfc}$ stands for the  complementary error function. More precisely, if we consider the function  $f(\lambda)= \int_{0}^{1}\d u~ u^{-1/2}(1-u)^{-1/2}\exp\big(\frac{\lambda}{u(1-u)}\big)$, $\lambda>0$, then  for every $s>0$, we have 
\begin{align*}
\int_{0}^{\infty}\d \lambda~\lambda^{s-1}f(\lambda) &=\int_{0}^{1} \d u \frac{1}{\sqrt{u(1-u)}} \int_{0}^{\infty}\d \lambda ~\lambda^{s-1} \exp(-\frac{\lambda}{u(1-u)})\\
&=\int_{0}^{1} \d u~ u^{s-1/2}(1-u)^{s-1/2} \int_{0}^{\infty}\d \lambda ~\lambda^{s-1} \exp(-\lambda)=\text{Beta}(s+\frac{1}{2},s+\frac{1}{2}) \Gamma(s),
\end{align*}
where $\text{Beta}(\cdot,\cdot)$ stands for the Beta function. We can now rewrite the last display in the form $2^{-2s}\Gamma(1/2)\Gamma(s+1/2)s^{-1}$ 
which is the Mellin transform of the function $\lambda \mapsto \pi\cdot \text{erfc}(2\sqrt{\lambda})$; this can be established directly from the definition of $\mathrm{erfc}$ applying Fubini.   Since the Mellin transform is an injective transformation, we derive that $
f(\lambda)=\pi\cdot \text{erfc}(2\sqrt{\lambda})$ which gives \eqref{mellin:erfc}.  We have obtained that $\mathbb{Q}_{1}\big( \mathscr{R}<-x\big)$ is bounded above by 
$$\sqrt{2\pi} C(\alpha)\cdot \int_0^{\infty}\d \ell~ \ell^{-\alpha-\frac{1}{2}} g\big(\frac{x^{2}}{\ell}\big)\mathrm{erfc}\big(2\cdot \ell^{-1/2}\big),$$
and performing the change of variable $s\leftarrow x^{2}/\ell$, we can write the previous display in the form:
$$\sqrt{2 \pi} C(\alpha)x^{-(2\alpha-1)}\cdot \int_0^{\infty}\d s~ s^{\alpha-\frac{3}{2}} g\big(s\big) \mathrm{erfc}\big(2\cdot \frac{\sqrt{s}}{x} \big).$$
To conclude we need to study the behavior of: 
$$F(x):=\int_0^{\infty}\d s~ s^{\alpha-\frac{3}{2}} g\big(s\big) \mathrm{erfc}\big(2\cdot \frac{\sqrt{s}}{x} \big). $$
In this direction, notice that $\mathrm{erfc}$ and $g$ are bounded above by $1$. Moreover, by  asymptotic properties of Bessel
functions we have $g(t)={O}(t^{-1})$ as $t \to \infty$ , and an integration by parts gives:
$$\text{erfc}(t)=\frac{2}{\sqrt{\pi}}\int_{t}^{\infty}\d v~\exp(-v^{2})\leq \frac{\exp(-t^{2})}{\sqrt{\pi} t},\quad t\geq 0.$$
We then infer that there exists $C\geq1$ such that
$$g(t)\leq 1\wedge \big (C t^{-1}\big)\quad \text{and }\quad \mathrm{erfc}(t)\leq 1\wedge \big( C t^{-1}\exp(-t^{2}) \big)  ,$$
for every $t\geq 0$. Next, we distinguish two cases. First if $1<\alpha<3/2$, we simply write:
$$F(x)\leq C\cdot \Big( \int_0^{1}\d s~ s^{\alpha-\frac{3}{2}} +\int_1^{\infty}\d s~ s^{\alpha-\frac{5}{2}} \Big)<\infty.$$
This implies that, if $1<\alpha<3/2$, there exists $C_{\mathbb{Q}}$ such that $\mathbb{Q}_{1}\big( \mathscr{R}<-x\big)\leq C_{\mathbb{Q}}\cdot x^{-(2\alpha-1)}$, for every $x>0$. Finally, if $\frac{3}{2}\leq \alpha <2$, we write:
\begin{align*}
F(x)&\leq C ^2\cdot  \Big(\int_0^{1}\d s~ s^{\alpha-\frac{3}{2}} + \int_{1}^{x^{2}} \d s~ s^{\alpha-\frac{5}{2}} + \frac{x}{2}\int_{x^{2}}^{\infty} \d s~ s^{\alpha-3} \exp\big(-2 \frac{s}{x^{2}}\big)\Big).
\end{align*}
The latter is  ${O}(x^{2\alpha-3})$ as $x \to \infty$, if $\alpha>3/2$, and ${O}(\log(x))$ if $\alpha=3/2$. We derive the desired result in the regime $3/2\leq\alpha<2$. This completes the proof of the lemma in the case of $ \mathbb{Q}_1$, the case of $ \mathbb{Q}_{-1}$ follows by time reversal.
\end{proof}

\subsection{Comparison with an ideal model}\label{sec:comparison:ideal}
Our next step, in estimating the probability of trapping the point $ \mathfrak{X}$, is to compare the process $ Y^{\bullet}$ to an ideal model where, intuitively, $h=\infty$. More precisely,  under $ \mathbf{N}^{\bullet}( \cdot \mid  H_{t_{\bullet}}=h)$, let us consider the process 
$$ Y^{\bullet, \uparrow}_{r} := Y^{\bullet}_{\varpi + r} - Y^{\bullet}_{\varpi}, \quad \mbox{ for } 0 \leq r \leq h- \varpi,$$ 
where we recall that $\varpi$ is the a.s.~unique instant where $Y^{\bullet}$ reaches its overall minimum. We are going to prove Theorem \ref{prop:malo} by comparing  the  process $Y^{\bullet, \uparrow}$, under $ \mathbf{N}^{\bullet}( \cdot \mid  H_{t_{\bullet}}=h)$  when $h \to \infty$, with  a $(2\alpha-2)$-stable symmetric L\'evy process started from $0$ and conditioned to stay positive over $ (0, \infty)$. Since this is a degenerate conditioning, some care is needed and we first recall the definition of this process, see \cite[Chapter 5]{kyprianou2022stable} for details. In this direction, recall that $Y^\bullet$ under $\mathbf{N}^\bullet(\cdot | H_{t_\bullet}=h)$ is a $(2\alpha-2)$-stable symmetric Lévy process with L\'evy-Khintchine function $\lambda\mapsto c(\alpha)|\lambda|^{2(\alpha-1)}$, with $c(\alpha):=\alpha 2^{1-\alpha}\frac{\Gamma(\alpha)^{2}}{\Gamma(2\alpha)}$.  We start by constructing the  associated stable symmetric L\'evy process conditioned to  stay positive and started from $x>0$. Since this is a positive self-similar Markov process, it can be constructed by using the Lamperti transformation, and    we need to introduce some notation.

In order to have a comprehensive framework, we introduce $\mathbb{D}(\mathbb{R}, [-\infty,\infty))$ the space of rcll paths indexed by the entire real line  and taking values on $[-\infty,\infty)$. We write $(\varkappa_t)_{t\in\mathbb{R}}$ for the associated canonical process and, for $z\in \mathbb{R}$, we set 
$$\varsigma(z):=\inf\{t>  -\infty:~\varkappa_t>z\}.$$  
We then consider the Lamperti transformation of $\varkappa$, given by the process
\begin{equation}\label{def:chi}
Y^{\uparrow}_t:=\exp(\varkappa_{\kappa(t)}), \quad \mbox{ where }\quad \kappa(t):=\inf\Big\{r\in \mathbb{R}_+:~\int_{-\infty}^{r}\d u\, \mathrm{e}^{(2\alpha-2) \varkappa_u}>t\Big\},
\end{equation}
defined for every $0\leq t\leq \int_{-\infty}^{\infty}\d u\, \mathrm{e}^{(2\alpha-2) \varkappa_u}$. Here, we adopt the convention $\mathrm{e}^{-\infty}=0$. Next we consider  $(\mathtt{P}_x)_{x>0}$ a family of probability measures on $\mathbb{D}(\mathbb{R}, [-\infty,\infty))$, such that under $\mathtt{P}_x$,   $(\varkappa_{t})_{t< 0}=-\infty$ and   $(\varkappa_t)_{t\geq 0}$ is   a L\'evy process started from $\log(x) \in \mathbb{R}$  with L\'evy--Khintchine function
  \begin{eqnarray} \label{eq:LKpsi} \psi(\lambda):=c(\alpha)^{-1}\cdot\Big(\mathrm{i}a_\psi\lambda+\int_{\mathbb{R}} \left( \mathrm{e}^{\mathrm{i}\lambda y}-1-\mathrm{i}\lambda( \mathrm{e}^{y}-1)\mathbbm{1}_{| \mathrm{e}^{y}-1|\leq 1}\right ) \mathrm{e}^{\alpha y}| \mathrm{e}^{y}-1|^{-2\alpha+1}~ \mathrm{d}y  \Big) \end{eqnarray}
where $a_\psi$ is the positive constant:
$$a_\psi:=\int_{0}^{1}\frac{(1+y)^{\alpha-1}-(1-y)^{\alpha-1}}{y^{2(\alpha-1)}}~\d y. $$
We stress that, under  $\mathtt{P}_x$, the  process $(\varkappa_{t})_{t< 0}$ is superfluous and we could simply work with processes indexed by the half-line, however this will no longer be the case when studying limits in law as $x\downarrow 0$.   Under $\mathtt{P}_x$, the  process $Y^{\uparrow}$ is a self-similar Markov process of index $ 2\alpha-2$ and remark that $Y^\uparrow_t$ is well defined for every $t\geq 0$, since $\int_{-\infty}^{\infty}\d u\, \mathrm{e}^{(2\alpha-2) \varkappa_u}=\infty$ because $(\varkappa_{t})_{t\geq 0}$ drifts towards $\infty$.  Namely, it is a Markov process and  if  $\Theta$ denotes the scaling operator 
\begin{equation}\label{sca:operator}
\Theta(c,f):=\big(cf(c^{-(2\alpha-2)}t)\big)_{t\geq 0},
\end{equation}
then we have the following identity in distribution:
\begin{equation}\label{sca:dist:operator}
\big(\Theta(c,Y^\uparrow):~\mathtt{P}_x\big)\overset{(d)}{=} \big(Y^\uparrow:~\mathtt{P}_{cx}\big),\quad \text{for } (x,c)\in (0,\infty)^{2}.
\end{equation}
Under $\mathtt{P}_x$, by \cite[Theorem 5.14]{kyprianou2022stable},  the process $Y^\uparrow$ can be understood  as a Lévy process, with  L\'evy-Khintchine function $\lambda\mapsto c(\alpha)|\lambda|^{2(\alpha-1)}$, started from $x$ and conditioned  to stay positive. Moreover, the law of $Y^\uparrow$ under $\mathtt{P}_x$ converges as $ x \searrow 0$ in distribution.  Let us now give an explicit  construction of this limiting distribution and precise this convergence. From \cite[Lemma 3]{BerSav}, there exists  $\rho(\mathrm{d}r, \mathrm{d}s)$ a measure  on $ \mathbb{R}_- \times \mathbb{R}_+$, such that, for every $z\in\mathbb{R}$, we have the following weak convergence
\begin{equation}\label{eq:limit:overshoot}
\lim\limits_{x\downarrow 0}\mathtt{P}_{x}\big((\varkappa_{\varsigma(z)-} -z) \in \mathrm{d}r, (\varkappa_{\varsigma(z)} -z) \in \mathrm{d}s\big)=\rho(\mathrm{d}r, \mathrm{d}s).
\end{equation}
 Furthermore, under $\mathtt{P}_x$, since $(\varkappa_t)_{t\geq 0}$ does not creep upwards (this follows directly from Vigon's criterion \cite{Vigon}) we have  $\rho( \mathbb{R}_-\times\{0\})=0$. The $(2\alpha-2)$-stable symmetric L\'evy process conditioned to  stay positive and started from $0$ can then be constructed as follows. We consider   $\mathtt{P}_0$ a probability measure on $\mathbb{D}(\mathbb{R}, [-\infty,\infty))$, under which the law of 
$(\varkappa_{0-}, \varkappa_0)$ is $\rho$, and conditionally on $(\varkappa_{0-}, \varkappa_0)$ the processes $(\varkappa_t)_{t\geq 0}$ and $(-\varkappa_{-t-})_{t\geq 0}$ are independent and distributed according to $(\varkappa_{t})_{t\geq 0}$ under $ \mathtt{P}_{ \mathrm{e}^{\varkappa_0}}$ and $ \mathtt{P}_{ \mathrm{e}^{-\varkappa_{0-}}}( \cdot \mid \inf_{t\geq 0} \varkappa_t >0)$ respectively. This latter conditioning is well-defined because  $\rho( \mathbb{R}_-\times\{0\})=0$ and, under $ \mathtt{P}_x$, the process $(\varkappa_{t})_{t\geq 0}$ 
starts from $\log x$ and drift towards $\infty$. Under $ \mathtt{P}_0$, the process $\varkappa$ satisfies, for every $z\in \mathbb{R}$, the following stationarity property:
\begin{itemize}
\item      $(\varkappa_{\zeta(z)-}-z, \varkappa_{\zeta(z)}-z) \sim \rho$,
\item    conditionally on $(\varkappa_s : -\infty < s \leq \varsigma(z))$, the process $(\varkappa_{s + \varsigma(z)} : s\geq0)$ is distributed as $(\varkappa_{t})_{t\geq 0}$ under $ \mathtt{P}_{ \mathrm{e}^{\varkappa_{\varsigma(z)}}}$,
\end{itemize}
we refer  to   \cite[Section 2.4]{BerSav} for details.  Furthermore, the process $Y$ is rcll,  leaves $0$ continuously and  instantaneously, drifts towards $\infty$ and  is scale invariant with index $2 \alpha-2$, i.e. satisfies \eqref{sca:dist:operator} with $x=0$, see  \cite[Theorem 5.3]{kyprianou2022stable} and \cite[Corollary 4]{BerSav}.\footnote{Let us note that \cite{BerSav} deals with self-similar Markov processes with index $1$ (instead of $2(\alpha-1)$), but this is not a problem since the results apply to $(Y^{\uparrow})^{2(\alpha-1)}$ under $\mathtt{P}_x$, which is self-similar with index $1$} The same references also establish that the  law of  $Y^{\uparrow}$, under $\mathtt{P}_x$, converges in distribution towards the law $Y^{\uparrow}$, under $\mathtt{P}_0$. For this reason, this law is referred to as that of the $(2\alpha-2)$-stable symmetric L\'evy process conditioned to  stay positive and started from~$0$. For simplicity,  for $x\in [0,\infty)$,  denote  the expectation with respect to $\mathtt{P}_x$ by $\mathtt{E}_x$.

Let us now extend the notion of trapping points under $\mathtt{P}_x$, with $x\in[0,\infty)$. To this end, under  $\mathtt{P}_x$ and conditionally on $Y^\uparrow$ (up to enlarging the underlying probability space), for each jump time $r$ of $Y^\uparrow$, we introduce a pair $( \mathscr{R}_r, \mathscr{I}_r)$ of random variables with law $ \mathbb{Q}_{\Delta Y^\uparrow_r}$ (we set $(0,0)$ for the non-jump times). By analogy with \eqref{eq:trapped}, under $\mathtt{P}_x$, we say that a jump time $r$ is \textbf{an $ \varepsilon$-trapping time} if 
$$   Y_{r-}^{\uparrow} + \mathscr{I}_r  \leq  \inf_{0 \leq s < r}\big( Y_{s-}^{\uparrow} + \mathscr{R}_s\big) \leq Y_0^{\uparrow} - \varepsilon.$$
 Our first step towards Theorem \ref{prop:malo} is to prove:
\begin{prop}\label{prop:eq:P:C:H}
There exist $c,C>0$ such that 
 \begin{equation}\label{eq:P:C:H}
 \mathtt{P}_0\big( \exists t \leq h : t \textnormal{ is a }1 \textnormal{-trapping time}\big) \geq 1-C\cdot  h^{-c}, \quad \text{ for every } h>0.
 \end{equation}
\end{prop}
\noindent During the proof of Proposition \ref{prop:eq:P:C:H}, we will state, in the form of lemmas, some technical properties of $Y^{\uparrow}$ and $( \mathscr{R}_r, \mathscr{I}_r)_{r \geq 0}$, under $\mathtt{P}_x$ for $x\in [0,\infty)$. The proofs of these results will be delayed to the next section. Here, we present the main arguments and strategy of the proof.
\begin{proof} We write $ \underline{ \mathscr{R}}_r = \inf_{s \leq r} \big(Y^\uparrow_{s-} + \mathscr{R}_s\big)$ for the running infimum process and for each $a \leq x$ consider under $ \mathtt{P}_x$ the stopping times 
\begin{equation}\label{def:T:A:trap}
T_{a}: = \inf\{ t \geq 0 : \underline{ \mathscr{R}}_t  \leq a\}.
\end{equation} 
Although the process $ Y^\uparrow$ stays positive, it is easily seen by scaling arguments that $T_a$ for $a \leq x$ are all a.s. finite. Actually we even have:
\begin{lem} \label{lem:stoptrap} There exists $c_1>0$ such that for every $x \in [0,1]$ we have
$$ \mathtt{P}_x( T_{-1} \textnormal{ is a  $1$-trapping time}) > c_1 .$$
\end{lem}
Given the lemma above, which will be proved in the next section, the idea of the proof of Proposition~\ref{prop:eq:P:C:H} is then clear: at each scale $ \ell>1$, the stopping time $T_{-\ell}$ produces a $1$-trapping time with probability at least $c_1$ so that the probability that no $1$-trapping time is found among the logarithmic number of scales needed to reach time $t$ should decay polynomially fast in $t$.  
To proceed let us consider the sequence of scales $(S_k)_{k\geq 1}$  defined by induction as follows: $S_1:=-1$ and 
$$S_{k+1}:=  2  \big(\underline{ \mathscr{R}}_{T_{S_k}}-Y^\uparrow_{T_{S_k}}\big),\quad k\geq 1.$$ 
By the Markov property and scaling invariance, conditionally on the past before time $T_{S_k}$,  Lemma \ref{lem:stoptrap} entails that the next stopping time $T_{S_{k+1}}$  has a probability at least  $c_1>0$ to be a $1$-trapping time. By iterating this argument we have: 
$$ \mathtt{P}_0\big( \exists t \leq T_{S_k} : t \mbox{ is a }1 \mbox{-trapping time}\big)  \geq 1-c_1^k.$$ 
Taking $k = k(h) = 1 + \lfloor  p \log (h \vee 1) \rfloor$ for some small constant $p>0$, the previous display becomes polynomially small in $h$ as $ h \to \infty$.  Hence, to conclude it  suffices to show that for  $p$ small enough, there exists $c_2,C_2\in (0,\infty)$ 
such that:
  \begin{eqnarray} \label{eq:goalpoly}\mathtt{P}_0\big(T_{S_{k(h)}}\geq h\big) \leq   C_2\cdot h^{- c_2}, \end{eqnarray}
  for every $h\geq 1$.   To this end, for $h \geq 1$, we write 
  \begin{eqnarray} \label{eq:techtechtech1} \mathtt{P}_0\big(T_{S_{k(h)}}\geq {h}\big)\leq \mathtt{P}_0\big(S_{k(h)}\geq h^{\frac{1}{4(\alpha-1)}}\big) + \mathtt{P}_0\big(T_{-h^{1/(4(\alpha-1))}}\geq h\,\big),  \end{eqnarray} and we will show that we can bound each term  as an  $O(h^{- \mathrm{cst}})$ for some $ \mathrm{cst}>0$ using the following lemma (whose proof is postponed to the next section):
  \begin{lem}[Polynomial tail estimates at $T_{-1}$]\label{Lem:bad:uni}
There exist two constants $c_3,C_3\in(0,\infty)$ such that for every $x\in [0, 1]$ and $r>0$ we have
$$\max\Big(\mathtt{P}_x\big(Y^{\uparrow}_{T_{-1}}>r\big), \mathtt{P}_x\big(T_{-1}>r\big),  \mathtt{P}_x\big( \underline{ \mathscr{R}}_{T_{-1}}<-r\big)\Big)\leq C_3 \cdot r^{-c_3}.$$
\end{lem}
\noindent Taking these estimates as granted, we can prove that \eqref{eq:techtechtech1} yields \eqref{eq:goalpoly}. In this direction,  we start by noticing that by scaling invariance we have:
$$\mathtt{P}_0\big(T_{-h^{1/(4(\alpha-1))}}\geq h\big)=\mathtt{P}_0\big(T_{-1}\geq h^{1/2}\big) \underset{ \mathrm{Lem.\ }  \ref{Lem:bad:uni}}{\leq}  C_3\cdot  h^{-c_3/2}.  $$
It remains to control the  first term in the right-hand side  of \eqref{eq:techtechtech1}. In this direction, writing $ \mathfrak{Y}_m:= {Y^\uparrow_{T_{S_{m}} }}/({2 | \underline{ \mathscr{R}}_{T_{S_{m}}}-  Y^{\uparrow}_{T_{S_{m}}}}|)$, an application of the Markov property and scaling invariance gives:
$$\mathtt{E}_0\big[S_{m+1}^{\gamma}\big]=\mathtt{E}_0\Big[S_m^{\gamma}\cdot \mathtt{E}_{\mathfrak{Y}_m}\Big[2^{\gamma}  \big|\underline{ \mathscr{R}}_{T_{-1}}-Y^\uparrow_{T_{-1}}\big|^\gamma\Big]\Big].$$
Using Lemma \ref{Lem:bad:uni}, the constant $\gamma$ can be chosen small enough such that there exists $C_4>0$ (depending on $\gamma$) such that
$$ \sup_{x \in [0,1]}\mathtt{E}_x\left[ \big|\underline{ \mathscr{R}}_{T_{-1}}-Y^\uparrow_{T_{-1}}\big|^\gamma\right] \leq C_4.$$
Therefore, by induction, we have $\mathtt{E}_0\big[S_{m+1}^{\gamma}\big] \leq  ( 2^\gamma C_4)^m$, for every $m\geq 1$.  Finally, an application of  Markov inequality gives:
$$ \mathtt{P}_0\big(S_{k(h)}\geq h^{\frac{1}{4(\alpha-1)}}\big)=\mathtt{P}_0\big(S_{k(h)}^\gamma\geq h^{\frac{\gamma}{4(\alpha-1)}}\big)\leq ( 2^\gamma C_4)^{k(h)-1} \cdot h^{-\frac{\gamma}{4(\alpha-1)}},$$
and  choosing  $k(h):=1+\lfloor p \log(h\vee 1)\rfloor$, with $0<p< \gamma/(4(\alpha-1)\log( 2^\gamma C_4)$,  we derive that the right-side term of the previous display decreases polynomially fast in $h$ as $h\to \infty$ as desired. \end{proof}	

We are going to prove Theorem \ref{prop:malo} using Proposition \ref{prop:eq:P:C:H} and quantitative absolute continuity properties, which enable us to compare $Y^\bullet$ under $\mathbf{N}^\bullet$ with the self-similar Markov processes constructed earlier. In this context, with a slight abuse of notation and possibly enlarging the underlying probability space once again, we consider $Y$ under $\mathtt{P}_{x}$ as a Lévy process with L\'evy-Khintchine function $\lambda \mapsto c(\alpha)|\lambda|^{2(\alpha-1)}$ starting at $x$. We rely on the following uniform estimate:

\begin{lem}\label{sup:Y:local}
There exists $\gamma>0$, such that:
$$\mathtt{P}_{0}\big(\sup\limits_{[0,1]} Y\leq r \big)=c^\prime(\alpha)\cdot r^{\alpha-1}+o(r^{\alpha-1+\gamma}), \quad \mbox{ as } r \to 0,$$
where $c^\prime(\alpha):=c(\alpha)^{-1/2}/(\Gamma(\frac{1}{2})\Gamma(\alpha))$.
\end{lem}
The proof of Lemma \ref{sup:Y:local} is technical but builds on standard results for stable Lévy processes. Therefore, we defer also its proof to the next section and conclude this section by proving Theorem~\ref{prop:malo}.

\begin{proof}[Proof of Theorem \ref{prop:malo}] 
As explained right after Theorem \ref{prop:malo} and using our reduction to \eqref{eq:trapped}, it suffices to prove that, under $\mathbf{N}^\bullet$ and  conditionally on $H_{t_{\bullet}}=h$, we can find a $1$-trapping time before time $ H_{t_{\bullet}} -\varpi$ with probability at least $1-C\cdot  h^c$, for some constants $c,C\in (0,\infty)$. In this direction, for $A>0$, under $\mathbf{N}^\bullet$ (resp. $\mathtt{P}_0$), we write $\mathcal{C}(A)$ for the event where there exists a $1$-trapping time for $(Y^{\bullet}_{ \varpi + r }, \mathscr{R}_{\varpi+r}, \mathscr{I}_{\varpi+r})_{r\in [0, H_{t_\bullet}-\varpi]}$ (resp.  $( Y^{\uparrow}_r, \mathscr{R}_r, \mathscr{I}_r)_{r\in[0,\infty)}$)  before time $A$.  Here we recall that, under $\mathbf{N}^\bullet$,  $\varpi$ stands for the instant when $Y^{\bullet}$ attains its global infimum. We are going to conclude by showing that
 \begin{equation}\label{primera:sqr:h:C}
\mathbf{N}^\bullet\big({\mathcal{C}}\big(\sqrt{h}~\big)~\big|~ H_{t_\bullet}=h\big)\geq \mathtt{P}_0\big({\mathcal{C}}\big(\sqrt{h}~\big)\big)-C \cdot h^{-c},\quad h>0,
\end{equation}
for some constants $c,C\in (0,\infty)$. The latter display, combined with  Proposition \ref{prop:eq:P:C:H}, directly implies Theorem  \ref{prop:malo}. To prove \eqref{primera:sqr:h:C} we rely on standard results of excursion theory of Lévy processes.  First, recall that under $\mathtt{P}_0$, the process $Y$ is a  stable symmetric L\'evy process with exponent $\lambda\mapsto c(\alpha) |\lambda|^{2\alpha-2}$. We consider, $L$ a local time at the running infimum of $Y$ and we denote the associated excursion measure by $\mathbf{n}$. We keep the notation $Y$ for the excursion process under $\mathbf{n}$ and we denote its lifetime by $\sigma:=\inf\{t\geq 0:~Y_t\leq 0\}$. The process $L$ and $\mathbf{n}$ are defined up to a positive constant that we fix such that $\mathbf{n}(\sigma>1)=1$. In particular, by scaling we have:
\begin{equation}\label{eq:z:n:sigma}
\mathbf{n}(\sigma >z)= z^{-1/2} \cdot \mathbf{n}(\sigma>1), \quad z>0. 
\end{equation}
For later use, let us introduce the associated renewal function:
$$\mathscr{H}(x):= \texttt{E}_0\Big[\int_0^\infty \mathrm{d} L_t ~\mathbbm{1}_{\inf Y_t>-x}\Big],\quad x\geq 0. $$
It directly follows by scaling invariance that $\mathscr{H}$ is a constant times $x\mapsto x^{\alpha-1}$. This constant can be computed by standard results of Lévy processes. More precisely, by \cite[Equation (4)]{Cha96} combined with \eqref{eq:z:n:sigma} and Lemma \ref{sup:Y:local} we have:
\begin{equation}\label{const:mathscr:H:x}
\mathscr{H}(x)=c^{\prime}(\alpha) \cdot x^{\alpha-1},\quad x\geq 0,
\end{equation}
where $c^\prime(\alpha)$ is the constant appearing in Lemma \ref{sup:Y:local}. By standard results of excursion theory (see e.g. \cite[Section 4]{Getoor} or \cite{Maisonneuve:Certain}), under $\mathbf{N}^{\bullet}(\cdot | H_{t_\bullet}=h)$ and conditionally on $\varpi=a$ with $a<h$, the distribution of $(Y^{\bullet}_{ t+\varpi})_{t\leq h-a}$ is that of  $(Y_{t})_{t\leq h-a}$ under $\mathbf{n}(\cdot | \sigma> h-a)$. We can also decorate the jumps of $Y$ using the measures $(\mathbb{Q}_{a})_{a\in \mathbb{R}}$ defined in \eqref{eq:Q:a:b}  in order to define the associated processes $( \mathscr{R}_{\cdot }, \mathscr{I}_{\cdot})$ under $\mathbf{n}$ (and we keep the same notation). For simplicity, for every $A$,  under $\mathbf{n}$, we  use the notation ${\mathcal{C}}(A)$ for the event where there exists $r\in[0,A]$ such that:
$$   Y_{r-} + \mathscr{I}_r  \leq  \inf_{0 \leq s < r}\big( Y_{s-} + \mathscr{R}_s\big) \leq Y_0^{\uparrow} - 1,$$
which, in words, corresponds to  the existence of a $1$-trapping time for  $( Y_r, \mathscr{R}_r, \mathscr{I}_r)_{r\in[0,\infty)}$  before time $A$. We thus have
\begin{equation*}
\mathbf{N}^\bullet\big({\mathcal{C}}\big(\sqrt{h}~\big)~\big|~ H_{t_\bullet}=h\big)\geq 
\int_{0}^{h-h^{2/3}} \mathbf{N}^\bullet\big(\varpi\in \d a~\big|~ H_{t_\bullet}=h\big) \mathbf{n}\big({\mathcal{C}}\big(\sqrt{h}~\big)\big|\sigma>h-a\big).
\end{equation*}
We now claim that there exist  constants $c,C>0$ such that  for every $h>0$ and $z\geq h^{2/3}$ we have
\begin{equation}\label{eq:claim:C:sqrt:h}
 \mathbf{n}({\mathcal{C}}\big(\sqrt{h}~\big)|\sigma>z)\geq \mathtt{P}_0\big({\mathcal{C}}\big(\sqrt{h}~\big)\big)-C \cdot h^{-c}.
 \end{equation}
Before proving the claim \eqref{eq:claim:C:sqrt:h}, let us explain why \eqref{primera:sqr:h:C} (and then Theorem \ref{prop:malo}) follows from it. The claim \eqref{eq:claim:C:sqrt:h} and the display above it entail that there exist  constants $c,C\in(0,\infty)$ such that: 
 \begin{equation*}
\mathbf{N}^\bullet\big({\mathcal{C}}\big(\sqrt{h}~\big)~\big|~ H_{t_\bullet}=h\big)\geq \mathtt{P}_0\big({\mathcal{C}}\big(\sqrt{h}~\big)\big)-C \cdot h^{-c}-  \mathbf{N}^\bullet\big(\varpi\in [h-h^{\frac{2}{3}}, h]~\big|~ H_{t_\bullet}=h\big),
\end{equation*}
for every $h>0$. Recall now that, under $ \mathbf{N}^\bullet\big(\cdot~\big|~ H_{t_\bullet}=h\big)$, the process  $Y^\bullet$ is a $2(\alpha-1)$-stable symmetric L\'evy process stopped at time $h$ so that by  \cite[Theorem 13 p169]{Ber96} the law $ \mathbf{N}^\bullet\big(h^{-1}\varpi\in \d a~\big|~ H_{t_\bullet}=h\big)$ is a generalized  arcsine law and in particular  $\mathbf{N}^\bullet\big(\varpi\in [h-h^{\frac{2}{3}}, h]~\big|~ H_{t_\bullet}=h\big)$ decays polynomially fast as $h \to \infty$. Therefore, \eqref{primera:sqr:h:C} is a consequence of \eqref{eq:claim:C:sqrt:h}.

It remains to establish the claim \eqref{eq:claim:C:sqrt:h}. To this end, fix $h>1$ and $z\geq h^{2/3}$ and notice that by the Markov property of the measure $ \mathbf{n}$ we have
$$ \mathbf{n}\big({\mathcal{C}}\big(\sqrt{h}~\big)\big|\sigma>z\big)= \mathbf{n}\Big(\mathbbm{1}_{{\mathcal{C}}(\sqrt{h})} \mathtt{P}_{Y^{\uparrow}_{\sqrt{h}}}\big(\sigma>z-\sqrt{h}\,\big)\Big)  \frac{1}{\mathbf{n}(\sigma>z)}, $$
where under $\mathtt{P}_r$ the process $Y$ is a $2(\alpha-1)$ stable L\'evy process started from $r\in \mathbb{R}_+$ and we use the notation $\sigma:=\inf\{t>0:~Y_{t}\leq 0\}$. Since $\mathbf{n}(\sigma>z)=z^{-1/2}$ and $Y$ is a symmetric stable process the previous display equals
\begin{equation*} \mathbf{n}({\mathcal{C}}\big(\sqrt{h}~\big)|\sigma>z)= z^{1/2}\cdot \mathbf{n}\Big(\mathbbm{1}_{{\mathcal{C}}(\sqrt{h})}\mathtt{P}_{Y^{\uparrow}_{\sqrt{h}}/z_0^{1/(2\alpha-2)}}\big(\sup_{[0,1]} Y>0\big)\Big), 
 \end{equation*}
 where $z_0:= z-\sqrt{h}$. For simplicity, let us write $F(r)=c^\prime(\alpha)^{-1}r^{-\alpha+1}\mathtt{P}_r(\sup_{[0,1]}Y>0)$, where again  $c^\prime(\alpha)$ is the constant appearing in Lemma \ref{sup:Y:local}.  Now we use absolute continuity relations between $\mathbf{n}$ and $\mathtt{P}_0$. Specifically, an application of \cite[Theorem 3]{Cha96}, combined with \eqref{const:mathscr:H:x}, entails that: 
 $$ \mathbf{n}({\mathcal{C}}\big(\sqrt{h}~\big)|\sigma>z)=\sqrt{\frac{z}{z_0}} \cdot \widehat{\texttt{E}}_0\Big[ \mathbbm{1}_{{\mathcal{C}}(\sqrt{h})}F\Big(\frac{Y^{\uparrow}_{\sqrt{h}}}{z_0^{1/(2\alpha-2)}}\Big)\Big].$$
To conclude, we use the fact that the Mellin transform of $Y^{\uparrow}_{1}$ under $\mathtt{P}_0$ is explicit (see \cite[Theorems 4.13 and 5.3]{kyprianou2022stable}). In particular, there exists $\gamma > 0$ such that $\mathtt{E}_0[(Y_1^{\uparrow})^{\gamma}] < \infty$. Applying Lemma~\ref{sup:Y:local}, and potentially reducing $\gamma$ further, we deduce the existence of a constant $C_2\in(0,\infty)$ such that $F(y)\geq 1-C_2 \cdot y^{\gamma}, $ for every $y>0$. Hence, we have:
\begin{align*}
 \mathbf{n}\big({\mathcal{C}}\big(\sqrt{h}~\big)\big|\sigma>z\big)
 &\geq \sqrt{\frac{z}{z_0}} \cdot \mathtt{P}_0\Big(\mathcal{C}(\sqrt{h})\Big)- C_2 \sqrt{\frac{z}{z_0}}\cdot \mathtt{E}_0\Big[\Big(\frac{Y^{\uparrow}_{\sqrt{h}}}{z_0^{1/(2\alpha-2)}}\Big)^\gamma\Big]\\
 &=\sqrt{\frac{z}{z_0}} \cdot \mathtt{P}_0\Big(\mathcal{C}(\sqrt{h})\Big)- C_2 \sqrt{\frac{z}{z_0}}\cdot \Big(\frac{h^{\gamma/(4\alpha-4)}}{ z_0^{\gamma/(2\alpha-2)}}\Big)\cdot \mathtt{E}_0[(Y_1^\uparrow)^{\gamma}],
\end{align*}
where to obtain the second line we used the scaling invariance under $\mathtt{P}_0$. The desired result \eqref{eq:claim:C:sqrt:h} now holds, for $h>1$, recalling that $z_0=z-\sqrt{h}$ and $z>h^{2/3}$. We can then increase the constant $C$ in order to extend the \eqref{eq:claim:C:sqrt:h}  to every $h>0$.
\end{proof}

\subsection{Proof of the technical lemmas} \label{sec:technical}
In this section, we prove  the technical lemmas used in Section \ref{sec:comparison:ideal}. We begin with Lemma~\ref{sup:Y:local}, which was   used in the previous proof  to obtain a coupling with a polynomial error.

\begin{proof}[Proof of Lemma \ref{sup:Y:local}]
The content of the result is the presence of a polynomial gap between the first and second order in the above asymptotic. Our proof is built entirely upon the results and arguments of \cite{Kuzne}, see also \cite[Chapter 7]{kyprianou2022stable}. If $\alpha$ is irrational, the desired result follows directly from  \cite[Theorem 9]{Kuzne}, which gives an explicit infinite series representation for  the density of  $\sup_{[0,1]} Y$ under $\mathtt{P}_0$. Unfortunately, the statement of the lemma, although weaker, cannot be directly derived from the results in \cite{Kuzne} for general $\alpha$. However, we can deduce our lemma by adapting some of the arguments therein.  First, we introduce the Mellin transform of  $\sup_{[0,1]} Y$ under $\mathtt{P}_0$, that is the function:
$$\mathcal{M}(s):=\mathtt{E}_0\Big[(\sup_{[0,1]} Y)^{s-1}\Big],$$
which by standard result on Lévy processes  is well defined and finite if $\mathrm{Re}(s)$ is sufficiently close to $1$,  see for e.g. \cite[Proposition 7.5]{kyprianou2022stable}. 
Furthermore \cite[Theorem 7]{Kuzne}  establishes that $\mathcal{M}$ can be analytically continued to a meromorphic function in $\mathbb{C}$ and  gives a closed  formula for this extension. Namely, it shows that
$$\mathcal{M}(s)=  \Big(2(\alpha-1) c(\alpha)^{\frac{1}{2(\alpha-1)}}\Big)^{s-1} \cdot  \frac{G(\alpha-1; 2\alpha-2)}{ G(\alpha; 2\alpha-2)}\cdot \frac{G(\alpha+1-s; 2\alpha-2)}{G(2\alpha-1-s; 2\alpha-2)}\cdot \frac{G(2\alpha-3+s; 2\alpha-2)}{G(\alpha-2+s; 2\alpha-2)},$$
where $G(s; \tau)$ is Barnes' double Gamma function, we refer to  \cite{Bar99,Bar01} for background. The function $G$ can be defined as an infinite product in Weierstrass’s form as follows:
$$G(s,\tau)=\frac{s}{\tau}\exp(\frac{a(\tau) s}{\tau}+ \frac{b(\tau) s^2}{2\tau})\prod_{m\geq 0} \prod_{n\geq 0}\!'\big(1+\frac{s}{m\tau +n}\big)\exp\big(-\frac{s}{m\tau+ n}+ \frac{s^2}{2(m\tau+n)^2}\big) , \quad s\in \mathbb{C}, \tau\geq 0,$$
where $a,b$ are two fine tuned functions only depending on $\tau$ and  the prime in the second product means that the term corresponding to $m = n = 0$ is omitted;  see \cite{Bar99}.
In particular,  $s\mapsto G(s, 2\alpha-2)$ is an entire function on $\mathbb{C}$, and only vanishes on the lattice $m+2(\alpha-1)n$,  for $m,n\leq 0$. Moreover, the multiplicity of the zero at a point  $m_0+2(\alpha-1)n_0$ is the cardinal of the set $\{(m,n)\in \mathbb{Z}_{\geq 0}^{2}:~m_0+2(\alpha-1)n_0=m+2(\alpha-1)n \}$. In particular, when $\alpha$ is rational the zeros might not be simple in general. However, the point $0$ is always a simple zero of $s\mapsto G(s,2\alpha-2)$. Now, fix $v\in \mathbb{R}$, such that  $3-2\alpha<v< 2-\alpha$. Then, by the previous discussion, we infer that  $G(\alpha+1-s;2\alpha-2), G(2\alpha-3+s;2\alpha-2)$ and $G(2\alpha-1-s;2\alpha-2)$ do not vanish on $\{s\in\mathbb{C}:~\mathrm{Re}(s)\in[v,1]\}$, and $G(\alpha-2+s;2\alpha-2)$ only vanishes on the latter set at $s_0:=2-\alpha$ which is a simple zero. Hence, $s_0:=2-\alpha$ is the unique  pole of $\mathcal{M}(s)$ with $\mathrm{Re}(s)\in [v,1]$.
Moreover, using the the  infinite product in Weierstrass’s form, we infer that the residue at $s_0$ equals:
$$\widetilde{c}(\alpha):=(2\alpha-2)^{2-\alpha}\cdot c(\alpha)^{\frac{1}{2}} \cdot  \frac{G(\alpha-1; 2\alpha-2)}{ G(\alpha; 2\alpha-2)}\cdot \frac{G(2\alpha-1; 2\alpha-2) G(\alpha-1; 2\alpha-2)}{G(3\alpha-3; 2\alpha-2)}.$$
The latter expression can be simplified using the quasi-periodic properties of $G$ and the fact that $G(1,2\alpha-2)=1$, see \cite{Bar99} and (4.6) in \cite{Kuzne}, and we obtain $\widetilde{c}(\alpha)=c(\alpha)^{\frac{1}{2}}/(\Gamma(\frac{1}{2})\Gamma(\alpha-1))$.
To derive the desired result, we use that since $(\alpha-1)\in(0,1)$ and the Lévy process is symmetric,  \cite[Lemma 3]{Kuzne} entails that $\mathcal{M}(s)$ decreases exponentially fast as $|\text{Im}(s)|\to \infty$. Therefore, $\sup_{[0,1]} Y$ under $\mathbb{P}_0$ has a smooth density function $x\mapsto p(x)$, and  the inverse Mellin transform gives:
$$p(x)=\frac{1}{2\pi \mathrm{i}}\int_{1+\mathrm{i}\mathbb{R}} x^{-s}\mathcal{M}(s) \d s, $$
for every $x>0$ (for $x\leq 0$ we simply have $p(x)=0$ since $\sup_{[0,1]} Y>0$ a.s.). Finally, shifting the contour of integration from $1+\mathrm{i} \mathbb{R}$ to $v+\mathrm{i} \mathbb{R}$, we get:
$$p(x)=\tilde{c}(\alpha)\cdot x^{\alpha-2}+ \int_{v+\mathrm{i}\mathbb{R}}  x^{-s}\mathcal{M}(s)\d s.  $$
 The desired result follows since:
 \begin{align*}
\big|\int_{v+\mathrm{i}\mathbb{R}}  x^{-s}\mathcal{M}(s)\d s\big|= x^{-v} \big|\int_{\mathbb{R}}  x^{-\mathrm{i} t}\mathcal{M}(v+\mathrm{i} t)\d t\big|\leq  x^{-v} \int_{\mathbb{R}}  \big|\mathcal{M}(v+\mathrm{i} t)\big|\d t={O}(x^{-v}),
\end{align*}
where to obtain the last equality we used again that $\mathcal{M}(s)$ decreases exponentially  fast as $|\text{Im}(s)|\to \infty$, and writing $\mathtt{P}_{0}\big(\sup_{[0,1]} Y\leq r \big)=\int_0^r p(x) \mathrm{d} x$.
\end{proof}
The rest of the section is devoted to the proof of Lemmas \ref{lem:stoptrap} and \ref{Lem:bad:uni}, used to establish Proposition~\ref{prop:eq:P:C:H}. Their proof is more involved,  and we start proving two intermediate result. In this direction, we introduce the notation:
 $$\zeta(z):=\inf\{t\geq 0:~Y^{\uparrow}_t>z\},$$ for every $z\geq 0$. We also stress that, under $\mathtt{P}_x$ and recalling \eqref{def:chi} as well as the notation above it, we must have $\zeta(z)=\int_{-\infty}^{\varsigma(\log(z))}\exp((2\alpha-2)\varkappa_u) \d u $.

\begin{lem}\label{lem:exp:tau:log:z}
 The following properties hold:
\begin{itemize}
\item[$\mathrm{(i)}$] For every $\gamma> -2(\alpha-1)$, there exists $C_\gamma>0$ such that:
\begin{equation}\label{lem:eq:lim:22}
 \sup \limits_{x\in[0,z]}\mathtt{E}_x\Big[\int_{0}^{\zeta(z)}~ (Y_t^\uparrow)^\gamma~ \d t\Big]\leq C_\gamma\cdot z^{\gamma+2(\alpha-1)},\quad z>0.
\end{equation}
\item[$\mathrm{(ii)}$] For every $\gamma\leq \alpha-1$, there  there exists $\widetilde{C}_\gamma>0$ such that:
\begin{equation}\label{eq:Y:gamma:zeta}
 \sup \limits_{x\in[0,z]}\mathtt{E}_x\Big[(Y^\uparrow_{\zeta(z)})^\gamma\Big]\leq \widetilde{C}_\gamma \cdot z^{\gamma},\quad z>0.
\end{equation}
\end{itemize}
\end{lem}
\begin{proof}  Recall from Section \ref{sec:comparison:ideal} the notation $\rho$ for the distribution of the overshoot at $0$ of $\varkappa$  under $ \mathtt{P}_0$ and write $\widetilde{\rho}$ for the pushforward of $ \rho( \mathrm{d}r, \mathrm{d}s)$ by $ (r, s) \mapsto \mathrm{e}^s$. We fix $a>1$ such that $\widetilde{\rho}((1,a))>0$. To simplify  notation, for every $x\geq 0$ we consider $\mathtt{P}_x^\prime$ a copy of $\mathtt{P}_x$. Under $\mathtt{P}_x^\prime$ we use the same notation as under  $\mathtt{P}_x$ adding a prime in the expressions, and we write $\mathtt{E}_x^\prime$ for the associated mathematical expectation. In particular, $Y^{\uparrow, \prime}$ is a self-similar Markov process and $\zeta^\prime(z):=\inf\{ t\geq 0:~Y_t^{\uparrow,\prime}> z\}$. 
We start by showing that it suffices to establish \eqref{lem:eq:lim:22} and \eqref{eq:Y:gamma:zeta} with respect to $\mathtt{P}_0$ and with $z=1$ only. To this end,  let $z,x>0$, with  $0 \leq  x \leq z$, and fix  $\gamma > -2(\alpha-1)$. Next remark that  we have:
\begin{eqnarray*}
\mathtt{E}_0\Big[\int_{0}^{\zeta(az)}  ~\big(Y_t^\uparrow\big)^\gamma~\d t\Big] &\geq&\texttt{E}_0\big[\big(\int_{\zeta(x)}^{\zeta(az)} ~ \big(Y_t^\uparrow\big)^\gamma~\d t\big)\mathbbm{1}_{Y_{\zeta(x)}^\uparrow \leq a x}\big]\\
& \underset{ \mathrm{Markov}}{=} &\texttt{E}_0\Big[\texttt{E}_{Y_{\zeta(x)}}^{\prime}\Big[\int_{0}^{\zeta^\prime(az)} ~ (Y_t^{\uparrow,\prime})^\gamma~\d t\Big]\mathbbm{1}_{Y_{\zeta(x)}^\uparrow \leq a x}\Big]\\
& \underset{ \mathrm{scaling}}{=} & \texttt{E}_0\Big[\texttt{E}_{x}^\prime\Big[\int_{0}^{\zeta^\prime(az x/Y^{\uparrow}_{\zeta(x)})} ~ (Y_t^{\uparrow,\prime})^\gamma~\d t\Big]\Big(\frac{Y_{\zeta(x)}^{\uparrow}}{ x}\Big)^{\gamma+2(\alpha-1)} \mathbbm{1}_{Y_{\zeta(x)}^{\uparrow}\leq a x}\Big]\\
& \underset{\begin{subarray}{c} Y_{\zeta(x)} \geq x \\ \gamma > -2(\alpha-1) \end{subarray}}{\geq} & \widetilde{\rho}((1,a ))\cdot   \mathtt{E}_{x}^\prime\Big[\int_{0}^{\zeta^\prime(z)} ~ (Y_t^{\uparrow,\prime})^\gamma~\d t\Big].
 \end{eqnarray*}  For point (ii), similarly by the Markov property and scaling, we have:
$$
\texttt{E}_0\Big[(Y^\uparrow_{\zeta(az)})^\gamma\Big]\geq \texttt{E}_0\Big[\texttt{E}_{Y_{\zeta(x)}}^{\prime}\Big[\big(Y^{\uparrow,\prime}_{\zeta(az)}\big)^\gamma\Big]\mathbbm{1}_{Y_{\zeta(x)}\leq a x}\Big]\geq  \widetilde{\rho}((1,a)) \cdot \texttt{E}_x\Big[\big(Y_{\zeta(z)}^\uparrow\big)^\gamma\Big].
$$
We derive that, for every $z>0$, we have:
$$ \sup \limits_{x\in[0,z]}\mathtt{E}_x\big[\int_{0}^{\zeta(z)}~ \big(Y_t^\uparrow\big)^\gamma~ \d t\big]\leq 
\frac{\mathtt{E}_0\big[\int_{0}^{\zeta(az)}  ~\big(Y_t^\uparrow\big)^\gamma~\d t\big]}{ \widetilde{\rho}((1,a))}=\frac{\mathtt{E}_0\big[\int_{0}^{\zeta(1)}  ~\big(Y_t^\uparrow\big)^\gamma~\d t\big]}{ \widetilde{\rho}((1,a))}\cdot (az)^{\gamma+2(\alpha-1)},$$
and 
$$ \sup \limits_{x\in[0,z]}\mathtt{E}_x\Big[\big(Y^{\uparrow}_{\zeta(z)}\big)^\gamma\Big]\leq \frac{\mathtt{E}_0\Big[\big(Y_{\zeta(az)}^{\uparrow}\big)^\gamma\Big]}{\widetilde{\rho}((1,a))}=\frac{\mathtt{E}_0\Big[\big(Y_{\zeta(1)}^\uparrow\big)^\gamma\Big]}{\widetilde{\rho}((1,a))} \cdot (az)^\gamma $$
where to obtain the right terms  we used again the scaling invariance.  Hence, to deduce \eqref{lem:eq:lim:22} and  \eqref{eq:Y:gamma:zeta}, it suffices to show that $\mathtt{E}_0\big[\int_{0}^{\zeta(1)}  ~(Y_t^\uparrow)^\gamma~\d t\big]<\infty$, for $\gamma>-2(\alpha-1)$, and $\mathtt{E}_0\Big[(Y^\uparrow_{\zeta(1)})^\gamma\Big]<\infty$, for $\gamma\leq \alpha-1$. This follows from the Lamperti construction. Indeed, under $ \mathtt{P}_0$,  we have 
$(Y^\uparrow_{\zeta(1)})^\gamma = \exp(\gamma \varkappa_0)$ and $\varkappa_0 \sim \widetilde{\rho}(\mathrm{d}v)$. The measure $\rho$ is explicit \cite[Eq. (5)]{BerSav} and can be written in the form
  \begin{eqnarray} \label{eq:rhobertoin}\widetilde{\rho}(\d v)= c\cdot\Big(\int_1^{\infty} \texttt{P}_{u}(\inf_{t\geq 0} \varkappa_t >0) \cdot u^{\alpha-1}(u v-1)^{-2\alpha+1}  \d u\Big) v^{\alpha-1} \mathbbm{1}_{v\geq 1} \d v,  \end{eqnarray}
where $c>0$ is a normalization constant. 
This  entails that $\mathtt{E}_0\Big[\big(Y_{\zeta(1)}^\uparrow\big)^\gamma\Big]<\infty$, for $\gamma\leq \alpha-1$. To see it, just write:
$$\mathtt{E}_0\Big[\big(Y^\uparrow_{\zeta(1)}\big)^\gamma\Big]= \int_{1}^{\infty} v^\gamma~ \widetilde{\rho}(\d v)\leq \exp(2\gamma ) + c \int_{[1,\infty]\times [2,\infty]} \frac{u^{-\alpha}\cdot v^{\gamma+\alpha-1}}{(v-1)^{2\alpha-1}} \d u\d v$$
and the right-hand term is finite since $\alpha+\gamma\leq 2\alpha-1$. The argument to obtain $\mathtt{E}_0\big[\int_{0}^{\zeta(1)}  ~\big(Y_t^\uparrow\big)^\gamma~\d t\big]<\infty$ is similar. First by the Lamperti transformation \eqref{def:chi} we have: 
$$\mathtt{E}_0\Big[\int_{0}^{\zeta(1)}  ~\big(Y_t^\uparrow\big)^\gamma~\d t\Big] = \mathtt{E}_0\Big[\int_{-\infty}^{0}    \ \mathrm{e}^{(2(\alpha-1)+\gamma) \varkappa_t} ~\mathrm{d}t\Big].$$ Recalling that under $\mathtt{P}_0$ the process $(\varkappa_{-t-})_{t\geq 0}$ is conditionally on $\varkappa_{0-}$ distributed as $(\varkappa_{t})_{t\geq 0}$ under $ \mathtt{P}_{ \mathrm{e}^{- \varkappa_{0-}}}(\cdot \mid \inf_{t\geq 0} \varkappa_t>0)$, it follows from  the explicit description of (the first marginal of) $\rho$ given in \cite[Eq. (5)]{BerSav} combined with the 
the duality relation \cite[Theorem 1 (ii)]{BerSav} that 
$$ \mathtt{E}_0\Big[\int_{0}^{\zeta(1)}  ~\big(Y_t^\uparrow\big)^\gamma~\d t\Big]=\tilde{c}\int_0^{\infty} \mathtt{P}_{\exp(s)}\big(\inf\limits_{t\geq 0}\varkappa_t>0\big) \exp(-(2(\alpha-1)+\gamma)s)\d s, $$
for some fixed constant $\tilde{c}>0$. Here, we used that, since $(\varkappa_t)_{t\geq 0}$ drift towards $\infty$, the function  $s\mapsto \mathtt{P}_{\exp(s)}\big(\inf\varkappa>0\big)$ is a  renewal function of the dual of the Lévy process $(\varkappa_t)_{t\geq 0}$. 
 Finally, remark that $s\mapsto \mathtt{P}_{\exp(s)}\big(\inf_{t\geq 0}\varkappa_t>0\big)$
is sub-additive and therefore  the right-hand side  of the previous display is finite as soon as $\gamma>-2(\alpha-1)$. This completes the proof of the lemma since $s\mapsto \mathtt{P}_{\exp(s)}\big(\inf_{t\geq 0}\varkappa_t>0\big)$
is sub-additive.
\end{proof}
\noindent The second intermediate result is:
\begin{lem}\label{bound:V:zeta}
There exist $c,C\in(0,\infty)$ such that for every $r>0$ we have
$$\sup \limits_{x\in [0,r]} \mathtt{P}_x \big(\underline{ \mathscr{R}}_{\zeta(r)-}\leq -1\big)\leq C\cdot r^{c}. $$

\end{lem}

 \begin{proof} Let $r>0$ and $0 \leq x\leq r$. We start with the following trivial bound:
\begin{align*}
\texttt{P}_x \big(\underline{ \mathscr{R}}_{\zeta(r)-}\leq -1\big)\leq \texttt{E}_x\big[\sum_{t< \zeta(r)}\mathbbm{1}_{\mathscr{R}_t < -Y^\uparrow_{t-} -1} \big]= \texttt{E}_x\big[\sum_{t<\zeta(r)}\mathbb{Q}_{\Delta Y^\uparrow_{t}}\big( \mathscr{R} < -Y^\uparrow_{t-} -1\big)\big].
\end{align*}
Let $ 2\alpha-2 < \beta< (2 \alpha-1) \wedge 2 $, and notice that  Lemma \ref{lem:tech:Q} ensures the existence of a constant
 $C_\beta>0$ such that:
$$ \mathtt{P}_x \big(\underline{ \mathscr{R}}_{\zeta(r)-}\leq -1\big)\leq C_\beta \cdot  \mathtt{E}_x\big[\sum_{t<\zeta(r)}\big|\Delta Y_t^\uparrow \big|^{\beta}\big].$$
Assume for a moment that $x\neq 0$, and remark that in this case we can apply \eqref{def:chi} to translate the right term of the last display  in the following form:
$$C_\beta\cdot  \texttt{E}_x\Big[\sum_{t>0:~\Delta \varkappa_t\neq 0} \mathbbm{1}_{\sup\limits_{s\in [0,t]} \varkappa_{s}<\log(r)} \exp(\beta \varkappa_{t-})|\exp(\Delta \varkappa_t)-1|^{\beta}\Big],$$
where recall that under $ \mathtt{P}_x$ the L\'evy $ (\varkappa_t)_{t\geq 0}$ starts from $\log x$. We can now apply the compensation formula to the jump process of  $\varkappa$ to get that the expectation in the previous display equals:
\begin{align*}
& \quad \texttt{E}_x\Big[\sum\limits_{t>0:~\Delta_t \varkappa\neq 0} \mathbbm{1}_{\sup\limits_{s\in[0,t]} \varkappa_{s-}<\log(r)} \exp(\beta\varkappa_{t-})|\exp(\Delta \varkappa_t)-1|^{\beta}\mathbbm{1}_{\varkappa_{t-}+\Delta_t \varkappa<\log(r)}\Big]\\
&\underset{ \mathrm{Comp.\ form.} \& \eqref{eq:LKpsi}}{=}  \texttt{E}_x\Big[\int_{0}^{\infty} \d t~   \mathbbm{1}_{\sup\limits_{s\in[0,t]} \varkappa_{s-}<\log(r)}  \exp(\beta\varkappa_{t})\cdot \int_{-\infty}^{\log(r)-\varkappa_{t-}}\d s~\exp(\alpha s) |\exp(s)-1|^{\beta-2\alpha+1}
\Big]\\
&\underset{ \eqref{def:chi} , u = \mathrm{e}^s}{=} \texttt{E}_x\Big[\int_{0}^{\zeta(r)} \d t ~ (Y_{t}^\uparrow)^{\beta-2(\alpha-1)} 
\cdot \int_{0}^{r/Y_{t}^\uparrow}\d u~u^{\alpha-1} |u-1|^{\beta-2\alpha+1}
\Big].
\end{align*}
We have obtained that:
$$ \texttt{E}_x\Big[\sum_{t<\zeta(r)}\big|\Delta Y_t^\uparrow \big|^{\beta}\Big]=    \texttt{E}_x\Big[\int_{0}^{\zeta(r)} \d t ~ (Y_{t}^\uparrow)^{\beta-2(\alpha-1)} 
\cdot \int_{0}^{r/Y_{t}^\uparrow}\d u~u^{\alpha-1} |u-1|^{\beta-2\alpha+1}
\Big],$$
for $x\in(0,r]$, and this equality extends to the case $x=0$ by considering the process $Y_{\varepsilon+\cdot}^\uparrow$, for $\varepsilon>0$,  applying Markov property and then taking the limit $\varepsilon\downarrow 0$ thanks to monotone convergence. Next remark that, since $\beta>2\alpha- 2$, a  straightforward computation gives the existence of a constant $C_\beta^\prime>0$ such that $\int_{0}^{v}\d u~ u^{\alpha-1} |u-1|^{\beta-2\alpha+1}\leq C_\beta^\prime \cdot v^{\beta-\alpha+1} $
for every $v\geq 1$. We derive that: 
$$ \mathtt{P}_x \big( \underline{ \mathscr{R}}_{\zeta(r)-}\leq -1\big)\leq C_\beta \cdot C_\beta^\prime \cdot r^{\beta-\alpha+1}\cdot \mathtt{E}_x\Big[\int_{0}^{\zeta(r)} \d t ~ (Y_{t}^\uparrow)^{-\alpha+1} \Big],  $$
and the desired result now follows from Lemma \ref{lem:exp:tau:log:z} (i) with $\gamma=-\alpha+1$.

\end{proof}

We can now proceed with the proofs of Lemmas \ref{lem:stoptrap} and \ref{Lem:bad:uni}, and we start with the former. In this direction, recall from \eqref{def:T:A:trap} that $T_{a}$ is the first time at which the process $ \underline{ \mathscr{R}}_{t} = \inf_{s \leq t}( \mathscr{R}_{s} + Y^\uparrow _{s-})$ drops below level $a$. 

\begin{proof}[Proof of Lemma \ref{Lem:bad:uni}] 
We begin by reducing the proof of the desired result to the case where $x$ is near $0$ and then we conclude by applying Lemma \ref{lem:tech:Q}.  In this direction, we use the same notation as in the previous proof and we write   $\widetilde{\rho}$ for the pushforward of $ \rho( \mathrm{d}r, \mathrm{d}s)$ by $ (r, s) \mapsto \mathrm{e}^s$. We also  fix  $p>1$ such that $\widetilde{\rho}((p,\infty))>0$. By Lemma \ref{bound:V:zeta} and the convergence of the law of the overshoots \eqref{eq:limit:overshoot}, combined with Portemanteau theorem, we infer that we can find  $0\leq r_1\leq r_2<1$ such that: \begin{equation}\label{eq:Y:z:p:r:2} \inf_{x\in [0,r_1]}\Big( \mathtt{P}_{x}\big(Y_{\zeta(r_2)}^\uparrow>p r_2\big)- \mathtt{P}_{x}\big(\underline{ \mathscr{R}}_{\zeta(r_2)-}\leq -1\big)\Big)>0. \end{equation} 
Next, we claim that it suffices to show  that the quantity:
$$C:= \inf_{x\in [0,r_1]}  \mathtt{P}_x\big( T_{-1} \textnormal{ is a  $(1/r_1)$-trapping time}\big)  $$
is positive. Actually,  this implies the desired result since directly $C\leq  \inf_{x\in [0,r_1]}  \mathtt{P}_x\big( T_{-1} \textnormal{ is a  $1$-trapping time}\big)$, and by scaling, for every $x\in[r_1,1]$, we have:
\begin{align*}
\mathtt{P}_{x}\big(T_{-1} \textnormal{ is a  $1$-trapping time}\big)&= \mathtt{P}_{r_1}\big(T_{-x/r_1} \textnormal{ is a  $(x/r_1)$-trapping time}\big)\\
&\geq \mathtt{P}_{r_1}\big(T_{-1} \textnormal{ is a  $(1/r_1)$-trapping time}\big)\geq C, 
\end{align*}
where to obtain the first inequality we used that $x\geq r_1$ and the definition of trapping times.  Let us now establish that $C>0$, see Figure \ref{fig:sketch} for an illustration of the argument.
\begin{figure}[!h]
 \begin{center}
 \includegraphics[width=10cm]{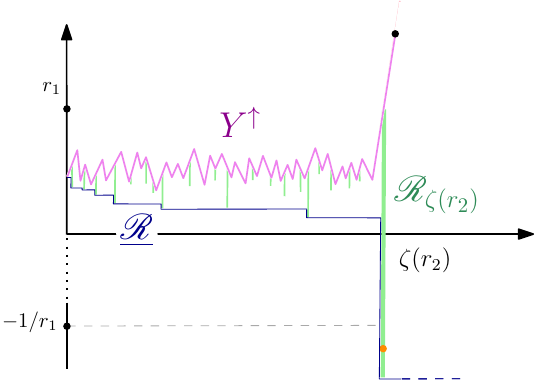}
 \caption{Illustration of the proof of $C>0$ in Lemma \ref{lem:stoptrap}: with a positive probability, the processes $ Y^\uparrow$ (in pink) and $ \mathscr{R}$ (in green) may barely move over $[0,\zeta(r_2))$ and the first large jump at time $\theta$ produces a $(1/r_1)$-trapping time. \label{fig:sketch}}
 \end{center}
 \end{figure}
To this end consider  $x\leq r_1$, and note that the Markov property entails that: $$ \texttt{P}_x\Big(\{T_{-1}=\zeta(r_2)\}\cap\{\mathscr{I}_{\zeta(r_2)} \leq -1/r_1 \}\Big)= \texttt{E}_x\Big[\mathbbm{1}_{\underline{ \mathscr{R}}_{\zeta(r_2)-}>-1}\mathbb{Q}_{\Delta Y_{\zeta(r_2)}^{\uparrow }}\big(\mathscr{I} \leq -1/r_1-Y^{\uparrow}_{\zeta(r_2)-}\big)\Big]. $$ Next, by scaling \eqref{scaling:a:Q} and taking $M:=\frac{2}{(p-1)r_1 r_2}$,  under the event $\{Y_\zeta(r_2)>p r_2\}$, we have: $$\mathbb{Q}_{\Delta Y_{\zeta(r_2)}^{\uparrow }}\big(\mathscr{I} \leq -1/r_1-Y^{\uparrow}_{\zeta(r_2)-}\big)\geq\mathbb{Q}_{1}\big(\mathscr{I} \leq -M\big)>0,$$ where the second inequality comes from Lemma \ref{lem:tech:Q} (i).  Hence, we obtain:
\begin{align*}
\texttt{P}_x\Big(\{T_{-1}=\zeta(r_2)\}\cap\{\mathscr{I}_{\zeta(r_2)} \leq -1 \}\Big)&\geq\mathbb{Q}_{1}\big(\mathscr{I} \leq -M\big)\cdot \mathtt{P}_{x}\big(\{Y_{\zeta(r_2)}^\uparrow>p r_2\}\cap\{ \underline{ \mathscr{R}}_{\zeta(r_2)-}> -1\}\big)\\
&\geq \mathbb{Q}_{1}\big(\mathscr{I} \leq -M\big)\cdot \Big(\mathtt{P}_{x}\big(Y_{\zeta(r_2)}>p r_2\big)- \mathtt{P}_{x}\big( \underline{ \mathscr{R}}_{\zeta(r_2)-}\leq -1\big)\Big).
\end{align*}
Therefore, it follows  from \eqref{eq:Y:z:p:r:2} that $C>0$.
\end{proof}
\noindent We conclude  Section \ref{sec:prison} and Part \ref{PartI} by proving Lemma \ref{Lem:bad:uni}.

\begin{proof}[Proof of Lemma \ref{Lem:bad:uni}]
As in the proof of Lemma \ref{lem:stoptrap}, the main idea is to argue that large jumps of the process $Y^{\uparrow}$ are likely to produce large negative $ \mathscr{R}$-values so that $\underline{ \mathscr{R}}$ drops below $-1$. More precisely, recalling for $r \geq 0$ that $ \zeta(r) = \inf \{ t \geq 0 : Y^{\uparrow}_{t} \geq r\}$, we start by claiming that there exist $c,C>0$ satisfying
\begin{equation}\label{eq:Y:T:1:r:f}\sup_{x\in [0,1]} \mathtt{P}_x(T_{-1}>\zeta(r))\leq C\cdot r^{-c}, \end{equation}
for every $r\geq 0$. Before proving the claim, let us explain why the lemma follows from it. On one hand, notice that $ \mathtt{P}_x(Y^\uparrow_{T_{-1}}>r)\leq  \mathtt{P}_x(T_{-1}>\zeta(r))$, and thus \eqref{eq:Y:T:1:r:f} entails
\begin{equation*}
\sup_{x\in [0,1]}  \mathtt{P}_x(Y^\uparrow_{T_{-1}}>r)\leq C\cdot r^{-c}, 
\end{equation*}
for every $r>0$. On the other hand, we have:
  \begin{eqnarray*} \mathtt{P}_x(\underline{ \mathscr{R}}_{T_{-1}}<-r) &\leq&  \mathtt{P}_x(T_{-1}\geq \zeta(\sqrt{r}))+ \mathtt{P}_x(\underline{ \mathscr{R}}_{\zeta(\sqrt{r})-}<-r) \\ & \underset{ \mathrm{scaling}}{=}&  \mathtt{P}_x(T_{-1}\geq \zeta(\sqrt{r}))+ \mathtt{P}_{x/r}(\underline{ \mathscr{R}}_{\zeta(1/ \sqrt{r})-}<-1).   \end{eqnarray*}
Combining Lemma \ref{bound:V:zeta} with \eqref{eq:Y:T:1:r:f}, we infer that there exist $c^\prime,C^\prime>0$ such that 
$$\sup_{x\in[0,1]}  \mathtt{P}_x\big(\underline{ \mathscr{R}}_{T_{-1}}<-r\big)\leq   C^\prime\cdot  r^{-c^\prime},$$
for every $r>1$. We can then increase the constant $C^\prime$ in order to extend the previous inequality to every $r>0$. Finally, to control $ \mathtt{P}_x(T_{-1}>r)$, we similarly  write:
 \begin{eqnarray*} \mathtt{P}_x(T_{-1}>r) &\leq&  \mathtt{P}_x(T_{-1}\geq \zeta(\sqrt{r}))+ \mathtt{P}_x(\zeta(\sqrt{r})\geq r)\\
 & \underset{ \mathrm{Markov}}{\leq}& \mathtt{P}_x(T_{-1}\geq \zeta(\sqrt{r}))+ r^{-1} \mathtt{E}_{x}[\zeta(\sqrt{r})].  \end{eqnarray*}
The desired result now follows by \eqref{lem:eq:lim:22} with $\gamma=0$,  and \eqref{eq:Y:T:1:r:f}. Therefore, to conclude it  remains to prove  \eqref{eq:Y:T:1:r:f}. The argument is similar to the one used in the proof of Proposition \ref{prop:eq:P:C:H}. In this direction,  let $x\in[0,1]$ and $r>1$ and introduce a sequence of (time) scales defined by $S_{1}:=\zeta(1)$ and  $S_{i+1}:=\zeta(2Y^\uparrow_{S_i})$ for every $i\geq 1$. We also set $R_i:=Y^\uparrow_{S_{i+1}}/Y^\uparrow_{S_i}$ and we remark that by self-similarity the random variables $R_i$, $i\geq 1$, are i.i.d. with common distribution $2Y^\uparrow_{\zeta(1)}$, under $\mathtt{P}_{1/2}$. 
Since $S_{i+1}$ is the first time when $Y^\uparrow$ cross level $2Y_{S_i}^\uparrow$, we have:
$$\sup \limits_{s\leq S_{m+1}}Y^\uparrow_s=Y^\uparrow_{\zeta(1)}\cdot \prod_{1\leq i\leq m} R_i~,\quad m\geq 2. $$
Hence, for every $m\geq 2$, we get:
  \begin{eqnarray} \label{eq:techtechtech2} \mathtt{P}_x\big(T_{-1}> \zeta(r)\big)\leq  \mathtt{P}_x\Big(Y^\uparrow_{\zeta(1)}\cdot \prod_{1\leq i\leq m-1} R_i> r\Big)+ \mathtt{P}_x\Big(T_{-1} \geq S_m\Big).  \end{eqnarray}
 We are going to control each term separately in terms of $m$ and then choose $m$ in a convenient regime. To this end, note that it follows from Lemma \ref{lem:exp:tau:log:z} (ii) that  we can find $\gamma>0$, such that:
$$\sup_{x\in [0,1]} \mathtt{E}_x\big[(Y^\uparrow_{\zeta(1)})^{\gamma}\big]<\infty.$$
Therefore,  an application of Markov inequality shows
  \begin{eqnarray*}  \mathtt{P}_x\Big(Y^\uparrow_{\zeta(1)}\cdot \prod_{1\leq i\leq m} R_i> r\Big) &\underset{ \mathrm{Markov}}{\leq}& r^{-\gamma}  \cdot \mathtt{E}_x\big[(Y^\uparrow_{\zeta(1)})^\gamma\cdot \prod_{1\leq i\leq m} R_i^\gamma\big] \\ 
  &=&r^{-\gamma}  \cdot \mathtt{E}_x\big[(Y^\uparrow_{\zeta(1)})^\gamma\big]\cdot \mathtt{E}_{1/2}\big[(Y^\uparrow_{\zeta(1)})^\gamma\big]^{m} \\
  & \leq & A^{m}\cdot  r^{-\gamma},  \end{eqnarray*} for some $A >0$. Let us now give an upper bound of the remaining term. To this end, we set: $$a:= \mathtt{E}_1\Big[\mathbb{Q}_{\Delta Y^\uparrow_{\zeta(2)}}\big(\mathscr{I} >-Y^\uparrow_{\zeta(2)-}-1\big)\Big],$$ and remark that by Lemma \ref{lem:tech:Q} (i) we have $a\in[0,1)$ since $\zeta(2)$ is a jumping time for $Y^\uparrow$, under $\mathtt{P}_1$. The reason to introduce the constant $a$ is that by scaling and,  for every $y\geq 1$ we have: $$L(y):=\mathtt{E}_y\Big[\mathbb{Q}_{\Delta Y^\uparrow_{\zeta(2y)}}\big(\mathscr{I} >-Y^\uparrow_{\zeta(2y)-}-1\big)\Big]=\mathtt{E}_1\Big[\mathbb{Q}_{\Delta Y^\uparrow_{\zeta(2)}}\Big(\mathscr{I} >- Y^{\uparrow}_{\zeta(2)-} -1/y\Big)\Big]\leq a. $$  The latter combined with the Markov property entails: $$\mathtt{P}_x\big(T_{-1} \geq S_m\big)\leq \mathtt{P}_x\big(\bigcap \limits_{2\leq i\leq m}\{\mathscr{I}_{S_i}>-1\}\big)=\mathtt{E}_{x}\Big[\big(\prod_{i=2}^{m-1} \mathbbm{1}_{\mathscr{I}_{S_i}>-1} \big)\cdot L(Y^\uparrow_{S_{m-1}})\Big]\leq a\cdot \mathtt{P}_x\big(\bigcap \limits_{2\leq i\leq m-1}\{\mathscr{I}_{S_i}>-1\}\big). $$ Iterating this argument we get $\mathtt{P}_x\big(\bigcap \limits_{2\leq i\leq m}\{\mathscr{I}_{S_i}>-1\}\big)\leq a^{m-1}.$ 
 Putting all together, we have obtained that 
$$ \mathtt{P}_x\big(T_{-1}> \zeta(r)\big)\leq A^{m}\cdot r^{-\gamma}+a^{m-1}; $$
for every $m\geq 2$. So taking $m=1+ \lfloor(\gamma \log(r))/(2\log(A))\rfloor$ we derive that $ \mathtt{P}_x(T_{-1}> \zeta(r))\leq C \cdot r^{-c}$, for some constant $c,C\in(0,\infty)$.
\end{proof}

\part{Random maps}\label{part:random_maps}

In this part, we prove the main results on scaling limits of discrete stable Boltzmann planar maps. We start by recalling the classical Bouttier--Di Francesco--Guitter (BDG) construction of those maps from bicolored labeled trees also called well labeled mobiles. Passing to the scaling limit of this construction leads us directly to the process $(X,Z)$ analyzed in detail in the previous section. As presented in the introduction, the proof of our main result is divided into several steps:
\begin{itemize}
\item In Section \ref{sec:BDG}, we recall the framework of \cite{LGM09} and construct the candidate $ (\mathcal{S}, D^{*})$ for the scaling limit using the (already established) scaling limits of the BDG construction. We also prove useful bounds on  distances.
\item In Section \ref{sec:topology}, we prove that the distance $D^{*}$ constructed in \eqref{def:dstar} induces the same topology as any sub sequential limits $D$, see Theorem  \ref{main_theorem_topology}.
 We identify exactly this topology in the dilute phase (Theorem \ref{main-topo}) using Proposition \ref{prop:records-loops} and Moore's theorem.
\item In Section \ref{sec:D=D*} we present the proof of our main result, admitting two results on the geometry of geodesics. This serve as a motivation for the last three sections which are devoted to the study of geodesics.
\item To do this, Section \ref{sec:boltzm-stable-maps} introduces a variant of the BDG construction of Section \ref{sec:BDG} which gives more information on discrete geodesics. This construction is then passed to the scaling limit in Section \ref{sec:scalingunicyclo}.
\item The final Section \ref{secP:uni:geo} builds upon the previous two in order derive properties of (typical) geodesic in the scaling limit of our stable Boltzmann maps and in particular prove the results used in Section \ref{sec:D=D*}.
\end{itemize}

\section{BDG$^\bullet$-bijection and subsequential scaling limits} \label{sec:BDG}

After recalling the classical BDG bijection, we present here the results of \cite{LGM09} and introduce the candidate for the scaling limits of Boltzmann maps with large faces. 
\subsection{BDG$^\bullet$ construction}\label{sec:bdg-construction}

We recall here the construction of planar maps from labeled
mobiles. Details can be found in \cite{MM07}.  A \textbf{mobile} is a plane tree $ \mathcal{T}$ with black and white vertices such that the root of the tree is white and all the neighbors of a black (resp.\ white) vertex are white (resp.\ black). Equivalently, a vertex of a {mobile} is white if its distance to the root is even, and black otherwise. We write $ {V}_\circ( \mathcal{T})$ and $ {V}_\bullet( \mathcal{T})$ the set of white and black vertices of $\mathcal{T}$. We
denote  the number of children of a vertex 
$v\in\mathcal{T}$ by $k_v(\mathcal{T})$, and we let $\widehat{V}_\circ(\mathcal{T}):=\{v\in
V_\circ(\mathcal{T}):k_v(\mathcal{T})=0\}$ be the set of \textbf{white leaves}
of $\mathcal{T}$.   A \textbf{label function} on $\mathcal{T}$ is a
function $\ell:V_\circ(\mathcal{T})\to \mathbb{Z}$ and we  say that
$\bT:=(\mathcal{T},\ell)$ is a labeled mobile. Moreover,  a labeled mobile $(\mathcal{T},\ell)$ is  called {\bf well-labeled} if the function $\ell$ satisfies the following properties:
\begin{itemize}
\item The label of the root is $0$;
\item  For every black vertex, $v\in V_\bullet( \mathcal{T})$, and  two neighbors of $v$ consecutive in the clockwise order $u_{1},u_{2}\in V_\circ(\mathcal{T})$, we have:
$$ \ell(u_{2})\geq \ell(u_{1})-1.$$
\end{itemize}

 We now follow \cite{BDFG04} and present a bijection between the set of pairs $(\bT,\epsilon)=(\mathcal{T},\ell,\epsilon)$, where 
 $\bT$ is  a well-labeled mobile and $\epsilon \in\{-1,1\}$ is a sign, and the set $\mathcal{M}^\bullet$ of pointed  bipartite planar maps. See Figure \ref{fig:BDGconstruction} for an illustration. Recall that $\mathcal{T}$ is a plane tree, which we view as a planar map with one face. The contour of this unique face, that is, the cyclic sequence of the corners of the face appearing in counterclockwise order, defines what we call the \textbf{contour order} around $\mathcal{T}$.
 We extend the definition of the labeling function $\ell$ to the set of corners of $\mathcal{T}$ incident to a white vertex (we naturally call them ``white corners''), by letting $\ell(c)=\ell(v)$ if the corner $c$ is incident to the white vertex $v$.

 The bijection now goes as follows. First, we introduce a new vertex $v_*$ belonging to the face of $\mathcal{T}$.  Second, we draw an arc going from every white corner $c$ to its successor corner, which is the first white corner with label $\ell(c)-1$ appearing after $c$ in the contour of $\mathcal{T}$. If $c$ has no successor, that is, if $\ell(c)$ equals the minimum value of $\ell$, then we link it to the vertex $ v_{*}$. These arcs can be drawn in a non-crossing manner following the orientation of the plane, and the embedded graph, whose vertex set equals $V_\circ(\mathcal{T})\cup\{v_*\}$, and whose edges are given by the arcs (hence discarding the edges of $\mathcal{T}$ and its black vertices), is a bipartite planar map that we denote by $ \mathrm{BDG}^\bullet( \bT, \epsilon)$. This map is pointed at $v_{*}$, and is rooted at the arc from the root corner of $\mathcal{T}$ to its successor, with this orientation if $\epsilon=1$, or with the reverse orientation if $\epsilon=-1$. In particular, when $\bT$ is made of a single white vertex, the resulting map $ \mathrm{BDG}^\bullet(\bT, \epsilon)$ is the ``edge-map'' $\to$ with has no face, one oriented edge and  two vertices, and which is pointed either at the origin or the extremity of the oriented edge depending on $ \epsilon$.

A few geometric properties of $(\bm,v_*)=\mathrm{BDG}^\bullet( \bT,\epsilon)$ can be directly deduced from this construction. To state them properly we recall that we denote the graph distance of $\bm$ by  $ \mathrm{d}^{ \mathrm{gr}}_{\bm}$. We also use the 
following standard terminology: a path of length $k$ in $\bm$ is a sequence $x_{0},e_{1},x_{1},e_{2},..., x_{k-1},e_{k},x_{k}$ where $x_{0},x_{1},...,x_{k}$ are vertices of $\bm$ and $e_{1},...,e_{k}$ are edges of $\bm$ such that $e_{i}$ connects $x_{i-1}$ and $x_{i}$ for every $i\in[\![1,k]\!]$. The path is called  a \textbf{geodesic} if its length is exactly the graph distance between $x_0$ and $x_k$. Then it follows from the definition of $\mathrm{BDG}^\bullet$ that:
\begin{itemize}
\item The faces  of degree $2k$  in $  \bm=\mathrm{BDG}^\bullet( \bT, \epsilon)$ are in correspondence with the black vertices of $\bT$ of degree $k$ in $ \mathcal{T}$. 
\item For every  $v\in  V_\circ(\mathcal{T})$,  starting from a corner adjacent to $v$ and following the arcs joining the consecutive iterated successors of $c$ until we reach $v_{*}$, we obtain  a geodesic path in $ \mathrm{BDG}^{\bullet}( \bT, \epsilon)$ connecting $v$ and $v_{*}$. In particular, we have:
\begin{equation}\label{dist:v_*:BDG}
   \mathrm{d}^{ \mathrm{gr}}_{\bm}(v,v_{*})=\ell(v)-\min\ell +1.
\end{equation}
\end{itemize}
The so-called ``\textbf{Schaeffer bound}'' is an improved version of the previous argument and states that for every vertices $v,v^{\prime}$ of $\mathrm{BDG}^{\bullet}( \bT, \epsilon)\setminus\{v_*\}$:
\begin{equation}\label{d_n^circ:BDG}
  |\ell(v)-\ell(v^{\prime})|\leq  \mathrm{d}^{  \mathrm{gr}}_{\bm}(v,v^{\prime})\leq \ell(v)+\ell(v^{\prime})-2\max\big( \min\limits_{[v,v^{\prime}]_{\mathcal{T}}}\ell;\min\limits_{[v^{\prime},v]_{\mathcal{T}}}\ell\big) +2~,
\end{equation}
where $[v,v^{\prime}]_{\mathcal{T}}$ stands for a minimal set of vertices appearing in clockwise order when going from $v$ to $v^{\prime}$, see the display after \cite[Equation (74)]{LGM09}, or  Section \ref{sec:D<D*} below for more details.

\begin{figure}[!h]
 \begin{center}
 \includegraphics[width=14.5cm]{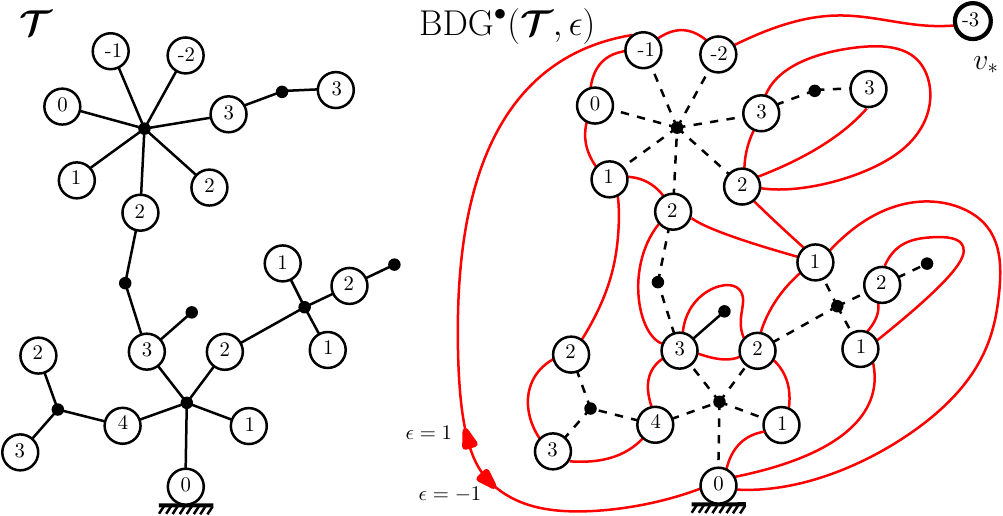}
 \caption{Illustration of the Bouttier--Di Francesco--Guitter construction of a planar pointed map (in red on the right) from a well-labeled mobile (on the left). The pointed vertex is $ v^*$ and the orientation of the root edge is given by an independent sign $ \epsilon$. \label{fig:BDGconstruction}}
 \end{center}
 \end{figure}

 \paragraph{Link with Boltzmann laws.}  \label{section:loiboltz}The connection with our model of $ \mathbf{q}$-Boltzmann planar map defined in the Introduction is as follows.  Let $ \mathbf{q}= (q_{k})_{ k \geq 1}$ be a non-zero sequence  of non-negative numbers and recall the definition of the $ \mathbf{q}$-\textbf{Boltzmann measures}  $w_{\mathbf{q}},w_\bq^\bullet$ given in and after \eqref{eq:defface}.   We assume that $w_{\mathbf{q}}$ is \textbf{admissible} in the sense that the total mass $w_{\mathbf{q}}(\mathcal{M})$ is finite.   Perhaps surprisingly, this implies, by  \cite[Corollary 3.15]{CurStFlour}, the stronger property that  $w_\bq^\bullet(\mathcal{M}^\bullet)\in (0,\infty)$. In particular, the $\bq$-Boltzmann distributions $w_\bq/w_\bq(\mathcal{M})$ and $w_\bq^\bullet/w_\bq^\bullet(\mathcal{M}^\bullet)$ are well-defined. 
We also recall that $w_\bq^\bullet(\mathcal{M}^\bullet)=2z_\bq$, where the quantity $z_\bq$ is defined at \eqref{eq:zbq}.

Then we  consider a $\mathbf{q}$-alternating two-type Bienaym\'{e}--Galton--Watson tree 
$\mathcal{T}$:  a tree with  white and black vertices at even and odd generations respectively  which reproduce independently according to the following offspring distributions
\begin{equation}
  \label{eq:mucirc_mubullet}
  \mu_\circ(k) = \zbq^{-1} (1 - \zbq^{-1})^k
\qquad\text{and}\qquad
\mu_\bullet(k) = \frac{\zbq^{k+1} \binom{2k+1}{k} q_{k+1}}{\zbq-1},
\end{equation}
for all $k \ge 0$. Note that the generating functions of these laws are respectively
\begin{equation}
  \label{eq:gencircbullet}
g_\circ(x)=\frac{1}{\zbq-x(\zbq-1)}\, ,\qquad g_\bullet(x)=\frac{g_\bq(\zbq x)-1}{(\zbq-1)x}\, ,  
\end{equation}
where we recall the definition of  $g_\bq$, given in \eqref{eq:genq},
$$g_\bq(x)=1+\sum_{k\geq 1}\binom{2k-1}{k-1}\, q_k\, x^k\, ,\qquad x\geq 0 .$$
 The law of the mobile $\mathcal{T}$ thus obtained  is denoted by $ \mathrm{GW}_{ \mathbf{q}}$.
Conditionally on $\mathcal{T}$, we consider random labels which are uniform over all well-labelings of $ \mathcal{T}$, and let $\bT$ be the resulting random well-labeled mobile. We also let $\epsilon$ be a uniform random variable on $\{-1;1\}$, independent of $ \bT$.  Proposition 7 of \cite{MM07}\footnote{Recast according to our convention of including the edge-map, instead of the vertex-map used in \cite{MM07}.} then states  that: 
$$ \mathrm{BDG}^{\bullet}( \bT,\epsilon) \mbox{ under } \mathrm{GW}_{ \mathbf{q}}( \mathrm{d} \mathcal{T}) \quad \overset{(d)}{=} \quad  \frac{w_\bq^\bullet(\cdot)}{w_\bq^\bullet(\mathcal{M}^\bullet)}.$$ For simplicity, we will say that  $\bT$  is a $\mathbf{q}$-\textbf{Boltzmann mobile}. Note in particular that if let $\bTn$ be a $\mathbf{q}$-{Boltzmann mobile} conditioned to have $n-1$ white vertices, then the map  $\mathrm{BDG}^\bullet(  \bTn,\epsilon)$ whose distinguished  vertex $v_{*}$ has been forgotten is a $ \mathbf{q}$-Boltzmann map conditioned to have $n$ vertices, meaning that it has law $w_\bq(\cdot \, |\, \#V=n)$.

 \subsection{Non-genericity assumption and scaling constants}  \label{sec:non-generic}
 Our basic assumption \eqref{eq:nongenerictail}, or equivalently \eqref{eq:nongenericgenf}, that $\bq$ is non-generic with exponent $\alpha\in (1,2)$ has simple interpretations in terms of the random tree $\bT$. First, note that \eqref{eq:nongenericgenf}  implies that $g'_{\bq}(\zbq)=1$, and therefore that 
 the two-type branching process associated with $\mathcal{T}$ is critical, in the sense that $ m_{\bullet }m_{\circ}=1$, where $m_{\bullet}=(\zbq-1)^{-1}$ and $m_{\circ}=\zbq-1$ are the means of $\mu_{\bullet}$ and $ \mu_{\circ}$. This also implies that 
 the law $\mu_{\circ\bullet}$ of the number of grandchildren of the root vertex of $\mathcal{T}$ has mean $1$, and satisfies:   \begin{eqnarray} \label{eq:tailmu}\mu_{\circ\bullet}((k,\infty))\sim \frac{\scal}{|\Gamma(1-\alpha)|(2k)^\alpha} \, ,\qquad k\to\infty\, .  \end{eqnarray}
To see this, note that $\mu_{\circ\bullet}$ has generating function $g_{\circ\bullet}(x)=g_\circ(g_\bullet(x))$, see \cite[Equation (14)]{LGM09}. 

As we pointed in the Introduction, ensuring the non-genericity assumption \eqref{eq:nongenericgenf} requires a fine-tuning of the weight sequence $(q_{k})_{k \geq 1}$. In order to convince the reader that this is still a reasonable assumption, we shall discuss several examples below. To help the reader navigate between the relevant references \cite{LGM09} and \cite{CurStFlour}, we provide a table translating the notation:

\begin{center}
\begin{tabular}{|c|c|c|}
\hline
This paper & \cite{LGM09} & \cite{CurStFlour}\\
\hline
$\alpha$ & $\alpha$ & $\alpha-\frac{1}{2}$\\
$g_{\bq}(x)$ & $x f_{\bq}(x+1)$ & $f_{\bq}(x)$ \\
$\zbq$ & $Z_{\bq}$ & $Z_{\bq}= \frac{c_\bq}{4}$\\
$\scal$ & $(2 c_0)^{\alpha}$ & $\kappa Z_{\bq}^{a-\frac{1}{2}}= \frac{p_{\bq} Z_{\bq} (2\sqrt{\pi})}{\Gamma\big(a+\frac{1}{2}\big)}$\\
\hline
\end{tabular}
\end{center}

As a first example, let us adapt the discussion around \cite[Proposition 2]{LGM09}:  Fix any weight sequence $\bq^\circ$ such that $q^\circ_k\sim k^{-\alpha-1/2}$ as $k\to\infty$.
Then, there exists a unique choice of the positive constants $C$ and $\beta$, with explicit values
$$C=\frac{1}{g_{\bq^\circ}'(1/4)}\, ,\qquad \beta=1-\frac{4(g_{\bq^\circ}(1/4)-1)}{g_{\bq^\circ}'(1/4)}$$
such that the sequence $\bq$ defined by 
\begin{equation}
  \label{eq:asymptoticqk}
  q_k=C\, \left(\frac{\beta}{4}\right)^{k-1}\, q^\circ_k\, ,\qquad k\geq 1
\end{equation}
is admissible, critical and non-generic with exponent $\alpha$. Moreover, for this particular choice, it holds that $\beta=\zbq^{-1}$.  In this case, the constant $\scal$ of (\ref{eq:nongenericgenf}) is given explicitly by:
$$\scal=\frac{2^{\alpha+1}C\Gamma(-\alpha)}{\beta\sqrt{\pi}}=\frac{2^{\alpha+1}\Gamma(-\alpha)}{(g_{\bq^\circ}'(1/4)-4g_{\bq^\circ}(1/4)+4)\sqrt{\pi}}\, .$$

\medskip 

A second example is provided by \cite{BBG11}, see also \cite{BudOn}, in connection with  random maps endowed with a statistical mechanics model. A loop-decorated rigid quadrangulation $ (\mathfrak{q},  \mathbf{l})$ is a planar map whose faces are all quadrangles, together with a family of non-intersecting loops $ \mathbf{l}=(l_{i})_{i \geq 1}$ on the dual map, in such a way that these loops can only cross quadrangles through opposite sides.  A measure on such configurations is defined by putting 
$$ W_{h,g, \mathfrak{n}}( ( \mathfrak{q}, \mathbf{l})) := g^{| \mathfrak{q}|}h^{| \mathbf{l}|}  \mathfrak{n}^{\# \mathbf{l}},$$
for $g,h>0$ and $ \mathfrak{n} \in (0,2)$ where $| \mathfrak{q}|$ is the number of edges of the quadrangulation, $| \mathbf{l}|$ is the total length of the loops and $\# \mathbf{l}$ is the number of loops. Provided that the measure $W_{h,g, \mathfrak{n}}$ has finite total mass, one can use it to define  a random loop-decorated quadrangulation denoted by $ \mathfrak{M}_{h,g, \mathfrak{n}}$. We then consider the \textbf{gasket} of $ \mathfrak{M}_{h,g, \mathfrak{n}}$ obtained by pruning off the interiors of the outer-most loops. It is easy to see that this gasket is actually a Boltzmann map for some (complicated) weight sequence depending on $h,g, \mathfrak{n}$. Fix $ \mathfrak{n} \in (0,2)$. For most of the parameters $(g,h)$ the gasket, once conditioned to be large, converges towards the Brownian sphere.  However, there is a fine tuning of $g$ and $h$ (actually a critical line) for which they are non-generic, with exponent
$$ \alpha =  \frac{3}{2} \pm \frac{1}{\pi} \arccos( \mathfrak{n}/2),$$ see \cite{BBG11,BudOn} or   \cite{kammerer2024gaskets} for the case $ \mathfrak{n}=2$.

The last example is a totally explicit family of non-generic critical weight sequences parametrized by the exponent $\alpha \in (1,2)$. This family appears in \cite{ambjorn2016generalized}, generalizing the Kazakov one-matrix model, and is used in \cite[Section 3.5.4]{CurStFlour} or  \cite[Section 6]{BC16}. It reads
$$q_k^{\mathrm{Kazakov}}:=2\left(\frac{1}{4\alpha}\right)^{k}\frac{\mathrm{Beta}(1/2,k)}{\mathrm{Beta}(-\alpha,k)}\ind_{\{k\geq
  2\}}\, ,$$
for which $\scal^{\mathrm{Kazakov}}:=2^\alpha/\alpha$, and where $\mathrm{Beta}( z_1,z_2) := \frac{\Gamma(z_1) \Gamma(z_2)}{\Gamma(z_1+z_2)}$ is Euler's beta function. \bigskip

All the above examples are non-generic in the sense of \eqref{eq:nongenerictail} and  \eqref{eq:nongenericgenf}, and actually satisfy the more stringent assumption 
  \begin{eqnarray} \label{eq:strictly-non-generic}  q_{k} \sim \mathrm{c} \cdot   \mathrm{C} ^{k} \cdot k^{-  \alpha - \frac{1}{2}}, \quad \mbox{ as }k \to \infty,  \end{eqnarray} where $c=\sqrt{\pi}\scal/(2^{\alpha+1}\Gamma(-\alpha)\zbq)$ and $C=1/( 4\zbq)$   which implies the pointwise asymptotic $\mu_{\circ\bullet}(\{k\})\sim \frac{ \scal}{\Gamma(-\alpha)2^\alpha k^{\alpha+1}}$ as $k \to \infty$, a property stronger than  \eqref{eq:tailmu}. We shall call such weight sequences \textbf{strictly non-generic}. In the opposite direction, our definition of non-genericity is stronger than that used in \cite{curien2018duality} or in Marzouk \cite{marzouk2018scalingstable} which allows for the presence of slowly varying functions: those weight sequences will be called \textbf{weakly non-generic}. We shall need to assume strict non-genericity in Section \ref{sec:scalingunicyclo} in order to apply certain local limit theorems. But our main results are indeed valid for non-generic weight sequences (and possibly for weakly non-generic ones to the cost of dealing more carefully with scaling sequences), see the discussion in Section \ref{sub:reroot}.

\subsection{Scaling limits of mobiles, and tightness of the associated maps} \label{sec:tightness}
This section presents the scaling limits results for large $\mathbf{q}$-Boltzmann mobiles, and derives a few consequences for their corresponding planar maps. To this end, we encode a labeled mobile $\bT=(\mathcal{T},\ell)$ using two processes. Namely, set $m=\#V_\circ(\mathcal{T})-1$ and consider $v_{0},...,v_{m} \in V_\circ(\mathcal{T})$ the list of the white vertices  in lexicographic (depth first) order, starting from the root corner. Then, we introduce (see Figure \ref{fig:coding}): 
\begin{itemize}
\item The \textbf{white  Lukasiewicz path} of $\bT$, which is the sequence $S^{\bT}:=(S^{\bT}_{0},S^{\bT}_{1},...)$ defined by induction as follows. First take $S^{\bT}_{0}:=0$ and, for every $i\in [\![0,m]\!]$, let   $S^{\bT}_{i+1}-S^{\bT}_{i}+1$ be the number of grandchildren of $v_{i}$.  For every $i>m$, we put $S_{i}^{\bT}:=-1$.
\item The \textbf{label  path} defined as follows:
\[ L_{i}^{\bT}:=\ell(v_{i}),\quad \text{for every } i\in[\![0,m]\!],\]
and $L_{i}^{\bT}:=0$ for all $i>m$.
\end{itemize} 
\begin{figure}[!htbp]
 \begin{center}
 \includegraphics[width=15cm]{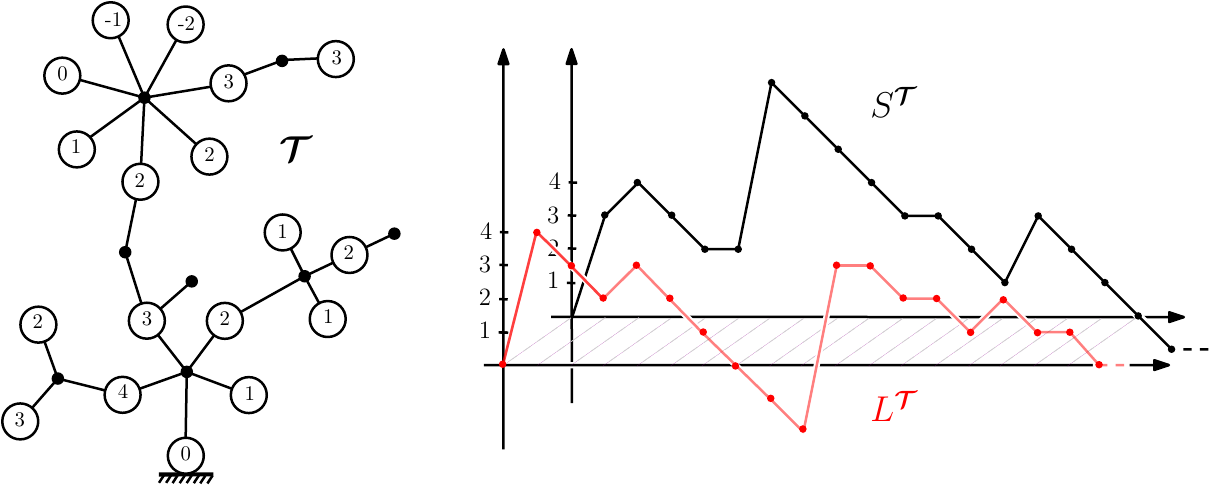}
 \caption{An example of a well-labeled mobile $ \bT = ( \mathcal{T}, \ell)$ with the associated Lukasiewicz and label paths $(S^{ \bT}, L^{ \bT})$. \label{fig:coding}}
 \end{center}
 \end{figure}
 \vspace{0.001cm}
 
\noindent Heuristically, the white Lukasiewicz path encapsulates the tree structure of $V_\circ(\mathcal{T})$, and the label process represents the label function $\ell$. These processes are introduced because their scaling limits can be studied using standard techniques from the theory of random paths.
 More precisely, let 
 $ \bTn $ be a $ \mathbf{q}$-Boltzmann mobile conditioned to have $n-1$ white vertices. Then, with our notation $(X,Z)$ of the previous part,  \cite[Lemma 12]{LGM09} states\footnote{To be more precise, \cite[Lemma 12]{LGM09} assumes that the weight sequence is strictly non-generic and of the form given in the first example presented in Section \ref{sec:non-generic}. However, a closer inspection at the proof shows that it is valid in the setting of the present paper. Indeed, the asymptotic behavior of $\bq$ is only needed to get tail bounds for the offspring distributions $\mu_\bullet, \mu_\circ$ and $\mu_{\circ\bullet}$, which are still valid under our more general situation. } that
\begin{equation}\label{first_couple}
  \left(2(\scal n)^{-\frac{1}{\alpha}}S_{\lfloor (n-1)\cdot\rfloor}^{\bTn},(\scal n)^{-\frac{1}{2\alpha}}L_{\lfloor (n-1)\cdot\rfloor}^{\bTn}\right)_{n\geq 1}
  \xrightarrow{(d)}\left(X,Z\right)~~\text{under} ~~\mathbf{P},
\end{equation}
where $\scal$ is the constant appearing in  \eqref{eq:nongenerictail}, and 
where the convergence takes place on $\mathbb{D}([0,1],\R)^2$ (here  $\mathbb{D}([0,1],\R)$ is the Skorokhod (J1) space of rcll functions from $[0,1]$ to $\R$).  Let us now draw a few consequences of the above scaling limit on the random map side, beginning with (a slight extension of) the subsequential scaling limit result of \cite{LGM09}.

\paragraph{Subsequential scaling limits.} Let $\mathfrak{M}_{n}$ be the random $ \mathbf{q}$-Boltzmann map conditioned to have $n$ vertices which is obtained as $ \mathrm{BDG}^{\bullet}( \bTn,\epsilon)$ after forgetting the distinguished point. Recall that $V(\mathfrak{M}_{n})$ is endowed with its graph distance $\mathrm{d}^{ \mathrm{gr}}_{ \mathfrak{M}_{n}}$, and the uniform probability measure $\mathrm{vol}_{ \mathfrak{M}_{n}}$. Write  $v_{0}^{n},...,v_{n-2}^{n}$ for the sequence of white vertices of $\bTn$ in lexicographical order, and for $s,t \in [0,1]$ set 
 \begin{eqnarray}\label{def:dn} d_{n}(s,t)= \mathrm{d}^{ \mathrm{gr}}_{ \mathfrak{M}_{n}}(v_{i}^{n},v_{j}^{n}), \quad \mbox{ where } \lfloor (n-1)s\rfloor =i \mbox{ and } \lfloor (n-1)t\rfloor =j,  \end{eqnarray}
with the convention $v_{i}^{n}=v_{0}^{n}$ if $i\geq n-1$.
The random metric space $\{v_{0}^n, ..., v_{n-2}^n\}$  endowed  with the graph distance  $\mathrm{d}^{ \mathrm{gr}}_{ \mathfrak{M}_{n}}$ and the uniform measure  can equivalently be seen as the quotient of $[0,1]$ by the equivalence relation $ \sim_{d_n}$, equipped with distance function induced by $d_n$, and the pushforward of the Lebesgue measure on  $[0,1]$ under the associated canonical projection.  We extend the definition of $d_{n}$ to $\mathbb{R}_{+}^{2}$ by bilinear interpolation, namely: \begin{align*} d_{n}(s,t)=&(s-\lfloor s\rfloor)(t-\lfloor t\rfloor)d(\lfloor s\rfloor,\lfloor t\rfloor)+(s-\lfloor s\rfloor)(\lceil t\rceil-t)d(\lfloor s\rfloor,\lceil t\rceil)\\ &+(\lceil s\rceil-s)(t-\lfloor t\rfloor)d(\lceil s\rceil,\lfloor t\rfloor)+(\lceil s\rceil-s)(\lceil t\rceil-t)d(\lceil s\rceil,\lceil t\rceil)\, . \end{align*}
Using  \eqref{first_couple} and the Schaeffer bound \eqref{d_n^circ:BDG}, it is easy to see that the laws of 
$$D_{n}:=\big((\scal n)^{-\frac{1}{2\alpha}}d_{n}((n-1)s,(n-1)t)\big)_{0\leq s,t \leq 1}\, ,\qquad n\geq 1, $$ 
form a tight family of probability measures on the space $\mathcal{C}([0,1]^2,\R)$ of continuous functions from $[0,1]^{2}$ to $\mathbb{R}$, see \cite[Theorem 5]{LGM09}. Therefore,  by the Skorokhod representation theorem,  we can extract a subsequence $(n_{k})_{k\geq 1}$ so that, along this subsequence, we have: 
  \begin{eqnarray}
\label{eq:second_couple}
 \Big(2(\scal n)^{-\frac{1}{\alpha}}S_{\lfloor (n-1)\cdot\rfloor}^{\bTn},(\scal n)^{-\frac{1}{2\alpha}}L_{\lfloor (n-1)\cdot\rfloor}^{\bTn},D_{n}\Big) \xrightarrow[\begin{subarray}{c}n\to\infty \\ \mathrm{along\  }(n_k)_{k\geq1} \end{subarray}]{\mathrm{a.s.}} \Big(X,Z,D\Big),~  \end{eqnarray}
where $D:[0,1]^{2}\to \mathbb{R}$ is a random continuous pseudo-distance, and the convergence takes place on the space $\mathbb{D}([0,1],\R)\times \mathbb{D}([0,1],\R)\times \mathcal{C}([0,1]^2,\R)$. 

\begin{center} \hrulefill \textit{ From now on, we fix the subsequence $(n_k)_{k\geq 1}$.} \hrulefill  \end{center}
 
\medskip 

 Let us now come back to the planar map side of the story. The space $(V(\mathfrak{M}_{n}), \mathrm{d}^{ \mathrm{gr}}_{ \mathfrak{M}_{n}}, \mathrm{vol}_{ \mathfrak{M}_{n}})$ is a random weighted compact metric space  -- a compact metric space endowed with a finite measure~-- and we let $\mathbb{M}$ be the set of all isometry classes of weighted compact metric spaces. To lighten notation, we shall often identify a compact weighted metric space with its equivalence class, and we leave the reader check that the definitions and results provided in these pages actually do not depend on the chosen representative. We equip $\mathbb{M}$ with the classical Gromov--Hausdorff--Prokhorov metric, namely for every $\textbf{M}:=(M,d_M,\mu)$ and  $\textbf{M}^{\prime}:=(M^{\prime},d_{M^{\prime}},\mu^{\prime})$ in $\mathbb{M}$:
$$d_{\mathrm{GHP}}\big(\textbf{M},\textbf{M}^{\prime}\big):=\inf\limits_{\phi,\phi^{\prime}}\Big( \delta_{\text{H}}\big(\phi(M),\phi^{\prime}(M^{\prime})\big)\vee \delta_{\text{P}}\big(\phi_* \mu,\phi^{\prime}_* \mu^{\prime}\big)\Big)~,$$
where the infimum is taken over all isometries $\phi$, $\phi^{\prime}$ from $M$, $M^{\prime}$ into a metric space $(Z, \delta)$ and $\delta_{\text{H}}$ (resp.~$\delta_{\text{P}}$) stands for the classical Hausdorff distance (resp.~the Prokhorov distance). The space $(\mathbb{M},d_{\mathrm{GHP}})$ is a Polish space, see \cite{Mie09} for more details. In order to state some re-rooting properties it will also be useful to introduce the set of   all isometry classes of \textit{marked} weighted compact metric spaces, namely $\mathbb{M}_{\mathrm{root}}:=\{(M,d_M,\mu,x): x\in M\} / \mathrm{iso}$,\footnote{Here we say that $(M,d_M,\mu,x)$ and $(M^\prime,d_{M^\prime},\mu^\prime,x^\prime)$  if there exists an isometric bijection $\varphi: M\to M'$ such that $\varphi_*\mu=\mu'$ and $\varphi(x)= x^\prime$.} and to endow it with the marked Gromov--Hausdorff--Prokhorov metric, i.e.
$$d_{\mathrm{GHP}}^{\mathrm{root}}\big((M,d_M,\mu,x),(M^\prime,d_{M^\prime},\mu^\prime,x^{\prime})\big):=\inf\limits_{\phi,\phi^{\prime}}\Big( \delta_{\text{H}}\big(\phi(M),\phi^{\prime}(M^{\prime})\big)\vee \delta_{\text{P}}\big(\phi_* \mu,\phi^{\prime}_* \mu^{\prime}\big)\vee \delta\big(\phi(x),\phi^{\prime}(x^{\prime})\big)\Big)~,$$
for every $(M,d_M,\mu),(M^\prime,d_{M^\prime},\mu^\prime) \in \mathbb{M}$ and $(x,x^{\prime})\in M\times M^{\prime}$ and where as before  the infimum is taken over all isometries $\phi$, $\phi^{\prime}$ from $M$, $M^{\prime}$ into a metric space $(Z, \delta)$. The space $(\mathbb{M}_{\mathrm{root}}, d_{\mathrm{GHP}}^{\mathrm{root}})$ is also a Polish space. As usual we equip $(\mathbb{M}, d_{\mathrm{GHP}})$ and $(\mathbb{M}_{\mathrm{root}}, d_{\mathrm{GHP}}^{\mathrm{root}})$
with the associated Borel sigma-field. We also mention  that the projection mapping a rooted  weighted compact metric space $\big(M,d_M,\mu,x\big)$ to $\big(M,d_M,\mu\big)$  plainly defines a projection from  $\mathbb{M}_{\mathrm{root}}$ onto $\mathbb{M}$.

Now recall from Proposition \ref{distinct} that the process $Z$ a.s.~attains its global minimum at a unique time $t_*\in [0,1]$. From the convergence \eqref{eq:second_couple}, along the subsequence $(n_{k})_{k\geq 1}$,  it is easy to see that:
\begin{prop}\label{theo:sub} Along the subsequence $(n_{k})_{k\geq 1}$ we have 
\begin{equation}\label{sub:conv:M}
\big( V(\mathfrak{M}_{n}) , (\scal n)^{-\frac{1}{2\alpha}} \cdot \mathrm{d}^{ \mathrm{gr}}_{ \mathfrak{M}_{n}}, \mathrm{vol}_{\mathfrak{M}_{n}} \big) \xrightarrow[n\to\infty]{a.s.} \big([0,1] / \sim_{D}, D,~\mathrm{Vol}_{D} \big)~,
\end{equation}
where the convergence holds in the Gromov--Hausdorff--Prokhorov sense. Moreover recalling that $\Pi_D : [0,1] \to [0,1]/\sim_D$ is the canonical projection, we have the following re-rooting property:
\begin{equation}\label{eq:re-rooting-S}
  \Big ([0,1] / \sim_{D}, D,~\mathrm{Vol}_{D},~\Pi_{D}(0) \Big)\overset{(d)}{=} 
  \Big ([0,1] / \sim_{D}, D,~\mathrm{Vol}_{D},~\Pi_{D}(t_*) \Big)\overset{(d)}{=} \Big([0,1] / \sim_{D}, D,~\mathrm{Vol}_{D},~\Pi_{D}(U) \Big)~,
\end{equation}
where $U$ is a uniform random variable on $[0,1]$ independent of $(X,Z,D)$. 
\end{prop}
It is  a classical technique to derive a convergence of the type \eqref{sub:conv:M} from one relating coding processes as  \eqref{eq:second_couple}. Actually, if we forget about the measures $ \mathrm{vol}_{\mathfrak{M}_{n}}$ and $\mathrm{Vol}_{D} $, the convergence \eqref{sub:conv:M} already appears  in the proof of \cite[Theorem 4]{LGM09} and it is easy to derive our case adapting this proof. For completeness, we  give the argument. 
\begin{proof}
Recall that $(v_0^{n},..., v_{n-2}^{n})$ is the sequence of white vertices of $\bTn$ and  that along the subsequence $(n_k)_{k\geq 1}$  we have \eqref{eq:second_couple}. 
We also write $v_*^{n}$ for the unique vertex of $\mathfrak{M}_n$ not appearing in  $(v_0^{n},..., v_{n-2}^{n})$. Now remark that the Gromov--Hausdorff--Prokhorov distance between 
$$\left( V(\mathfrak{M}_{n}) , n^{-\frac{1}{2\alpha}} \cdot \mathrm{d}^{ \mathrm{gr}}_{ \mathfrak{M}_{n}}, \mathrm{vol}_{\mathfrak{M}_{n}} \right)\quad \text{and}\quad \left( V( \mathfrak{M}_{n})\setminus \{v_*\} , n^{-\frac{1}{2\alpha}} \cdot \mathrm{d}^{ \mathrm{gr}}_{ \mathfrak{M}_{n}}, \mathrm{vol}_{\mathfrak{M}_{n}}\big(\cdot \cap \{v\neq v_*^n\}\big) \right)$$ is smaller than $n^{- \frac{1}{2 \alpha}}$. So it is enough to show that the GHP-distance between the latter space, that we denote  by $\mathfrak{M}_n^{\prime}$,  and $[0,1]/\sim_{D}$ tends to $0$. Recall that $\mathfrak{M}_n^{\prime}$ can be seen as the quotient $[0,1]/ \sim_{d_{n}}$ for the pseudo-distance $d_{n}$ defined in \eqref{def:dn}.  In this direction, we consider the obvious correspondence $\mathfrak{R}_n$  between $\left([0,1]/\sim_{d_{n}} \right)$ and $\left([0,1]/\sim_{D}\right)$ made of all $(\Pi_{d_{n}}(s), \Pi_{D}(s))$ for some $s \in [0,1]$.   We also equip $\mathfrak{R}_n$  with the measure $\nu_n$ obtained by the pushforward of the uniform measure on $[0,1]$ by $s \mapsto (\Pi_{d_{n}}(s), \Pi_{D}(s))$. 
In particular, $\nu_n$ is a coupling supported on $\mathfrak{R}_n$ between  $ \frac{n}{n-1} \cdot \mathrm{vol}_{\mathfrak{M}_{n}}\big(\cdot \cap \{v\neq v_*\}\big)$ and $\mathrm{Vol}_{D}$. The convergence \eqref{sub:conv:M} is now a direct application of \cite[Proposition 6]{Mie09} noticing that the distortion of $\mathfrak{R}_n$ tends to $0$ since by  \eqref{eq:second_couple}, we have:
\begin{align*}
\text{dis}\big(\mathfrak{R}_n\big)&=\sup\Big\{\big| D(s, t) - D_{n}(s,t)  \big|: (s,t)\in [0,1]\Big\}  \xrightarrow[n_{k}\to\infty]{a.s.} 0.
\end{align*}
 Let us now turn to the proof of \eqref{eq:re-rooting-S}, starting with the second equality in distribution. Let $u^{n}$ be a uniformly chosen random vertex in $\{v_0^{n},...,v_{n-2}^{n},v_*^n\}$, and remark that by the very definition of the measure $w_ \mathbf{q}^\bullet$ we have:
$$\big(V(\mathfrak{M}_n), \mathrm{d}^{ \mathrm{gr}}_{ \mathfrak{M}_{n}}, \mathrm{vol}_{\mathfrak{M}_{n}}, v_*^n\big)\overset{(d)}{=}\big(V(\mathfrak{M}_n), \mathrm{d}^{ \mathrm{gr}}_{ \mathfrak{M}_{n}}, \mathrm{vol}_{\mathfrak{M}_{n}}, u^n\big).$$ 
Now \eqref{sub:conv:M} gives:
$$\big(V(\mathfrak{M}_n), (\scal n)^{-\frac{1}{2\alpha}}\mathrm{d}^{ \mathrm{gr}}_{ \mathfrak{M}_{n}}, \mathrm{vol}_{\mathfrak{M}_{n}}, u^n\big)\longrightarrow \big([0,1]/\sim_D, D,\mathrm{Vol}_D, \Pi_{D}(U)\big),  $$
where $U$ is an independent uniform random variable on $[0,1]$ -- the convergence takes place a.s.\ along the subsequence $(n_k)_{k\geq1}$. Consequently, it remains to show that $\big(V(\mathfrak{M}_n), \mathrm{d}^{ \mathrm{gr}}_{ \mathfrak{M}_{n}}, \mathrm{vol}_{\mathfrak{M}_{n}}, v_*^n\big)$ converges along   $(n_k)_{k\geq1}$ towards $\big([0,1]/\sim_D, D,\mathrm{Vol}_D, \Pi_{D}(t_*)\big)$.  Let $w^{n}$ be the first vertex  -- in lexicographical order --  with label $\ell(v_*^n)+1$ and set $t_*^n:=\inf\{t\in[0,1]:~v_{\lfloor (n-1) t\rfloor}^n=w^n\}$. In particular, $w^{n}$ is a neighbor of $v_*^{n}$ in $\mathfrak{M}_n$. Moreover, since a.s.\ the process $Z$ is continuous and has a unique minimum (Proposition \ref{distinct}),  we have $t_*^n\to t_*$ and we then infer from  \eqref{eq:second_couple} that:
$$\lim \limits_{n\to \infty}D(t_*,t^{n}_*)=0.$$
We deduce that the distance, with respect to the $d_{\mathrm{GHP}}^{\mathrm{root}}$-metric, between
$\big(\mathfrak{M}_n,  v_*^n\big)$ and $\big(\mathfrak{M}_n^\prime,w^n\big) $ (resp. $\big([0,1]/\sim_D,  \Pi_D(t_*)\big)$ and $\big([0,1]/\sim_D,  \Pi_D(t_*^n)\big)$)
tends to $0$, as $n\to \infty$. Therefore, to conclude, it remains to show that  the $d_{\mathrm{GHP}}^{\mathrm{root}}$-distance between $\big(\mathfrak{M}_n^\prime,w^n\big) $ and $([0,1]/\sim_D,\Pi_D(t_*^n))$ tends to $0$, along $(n_k)_{k\geq 1}$. This is again a direct consequence, using \cite[Theorem 3.6]{GHPm},  of the fact that $(w^n,\Pi_D(t_*^n))$ belongs to $\mathfrak{R}_n$ and that the distortion of $\mathfrak{R}_n$ tends to $0$, along the subsequence $(n_k)_{k\geq 1}$.

Finally, let us show that the first and last terms of (\ref{eq:re-rooting-S}) are equal in distribution. To start with, note that the same arguments as the above entail that $([0,1]/\sim_D,\mathrm{Vol}_D,\Pi_D(0))$ is the a.s.\ limit in $\mathbb{M}_{\mathrm{root}}$ of $(V(\mathfrak{M}_n),  (\scal n)^{-\frac{1}{2\alpha}} \cdot\mathrm{d}^{ \mathrm{gr}}_{ \mathfrak{M}_{n}}, \mathrm{vol}_{\mathfrak{M}_{n}}, v_0^n)$ as $n\to\infty$ along the subsequence $(n_k)_{k\geq 1}$. Next, observe that $v_0^n$ has same distribution as a uniformly chosen vertex incident to the root edge of $\mathfrak{M}_n$ (depending on the value of the Rademacher random variable $\epsilon$). Moreover, if $e'_n$ is a uniformly random oriented edge of $\mathfrak{M}_n$, then $\mathfrak{M}_n$ re-rooted at $e'_n$ has same distribution as $\mathfrak{M}_n$. Therefore, to conclude, it suffices to show that for some appropriate choice of $e'_n$, it holds that $(V(\mathfrak{M}_n),  (\scal n)^{-\frac{1}{2\alpha}} \cdot\mathrm{d}^{ \mathrm{gr}}_{ \mathfrak{M}_{n}}, \mathrm{vol}_{ \mathfrak{M}_{n}}, v'_n)$ converges to $([0,1]/\sim_D,\mathrm{Vol}_D,\Pi_D(U))$ in $\mathbb{M}_{\mathrm{root}}$, where $v'_n$ is chosen at random uniformly among the two vertices incident to $e'_n$. 

In order to provide such a coupling of $e'_n$ with $\mathfrak{M}_n$, we proceed as follows. Recall that the contour order of white vertices around the mobile $\mathcal{T}_n$ is the sequence of white corners $c^n_0,c^n_1,\ldots,c^n_{C_\circ(\mathcal{T}_n)-1}$ that appear when going around the unique face of $\mathcal{T}_n$ in counterclockwise order,  starting from the root corner. Here, $C_\circ(\mathcal{T}_n)$ denotes the total number of white corners of $\mathcal{T}_n$.
We choose $e'_n$ to be the arc going from the corner $c^n_{\lfloor C_\circ(\mathcal{T}_n)U \rfloor}$ to its successor, with an orientation chosen uniformly at random among the two possible ones. In this way, $e'_n$ is indeed a uniformly chosen oriented edge of $\mathfrak{M}_n$. 

 To show that $e'_n$ has the required property, we let $\phi_n(i)$ be the unique index such that $c^n_i$ is incident to the vertex $v_{\phi_n(i)}^n$. Then, \cite[Lemma 13]{LGM09} implies that $\phi_n(\lfloor C_\circ(\mathcal{T}_n)\cdot \rfloor)/n$ converges uniformly to the identity map of $[0,1]$.
  This implies in particular, by the same arguments as above, that if
  $v''_n=v_{\phi_n(\lfloor C_\circ(\mathcal{T}_n)U \rfloor)}^n$, then it holds that
  $(V(\mathfrak{M}_n),  (\scal n)^{-\frac{1}{2\alpha}} \cdot\mathrm{d}^{ \mathrm{gr}}_{ \mathfrak{M}_{n}}, \mathrm{vol}_{\mathfrak{M}_{n}}, v_n'')$ converges in $\mathbb{M}_{\mathrm{root}}$ to $([0,1]/\sim_D,D,\mathrm{Vol}_D,\Pi_D(U))$ along the subsequence $(n_k)_{k\geq 1}$. The same holds with $v''_n$ replaced by a uniformly random extremity $v'_n$ of $e'_n$, since $v'_n$ is then at graph distance at most $1$ from $v''_n$. This concludes the proof.   
\end{proof}
Let us stress that, given the previous tightness result,  to prove  our main result, Theorem \ref{thm:main}, it suffices to establish that, under $\mathbf{P}$, we have
$$D(s,t)=D^{*}(s,t),\quad s,t\in[0,1], $$
 where $D^{*}$ is the pseudo metric constructed from $(X,Z)$ in the Introduction. The rest of this work is devoted to the proof of the previous display which in particular relies on the study of geodesics. Let us conclude this section by introducing some standard terminology. Similarly to planar maps, a metric space $(M,d_M)$ is called a \textbf{geodesic} metric space if for every $x,y\in M$ there exists an isometry $\gamma:[0,d(x,y)]\to M$ such that $(\gamma(0), \gamma(d_M(x,y)))=(x,y)$. In this case, we say that $\gamma$ is a \textbf{geodesic} going from $x$ to $y$. For  compact metric spaces, being a geodesic metric  space is equivalent to being a path metric space, i.e.\ an arcwise connected metric space $(M,d_M)$ such that the distance $d_M(x,y)$ is equal to 
  $$d_M(x,y)= \inf_{c : [0,T] \to M}\,  \sup_{0=t_1\leq t_2\leq  \dots \leq t_n=T}\sum_{i=1}^{n-1} d_M\big(c(t_i),c(t_{i+1})\big),$$
where the infimum is over all  continuous paths $c:[0,T]\to M$ with $(c(0),c(T))=(x,y)$ and where the supremum is over all choices of the integer $n\geq 1$ and of finite sequences $t_1\leq t_2 \leq \dots \leq t_n$ satisfying $t_0=0$ and $t_n=T$. In particular, it is a property invariant by isometry. Let us mention that formally, $(V(\mathfrak{M}_n),  (\scal n)^{-\frac{1}{2\alpha}} \cdot\mathrm{d}^{ \mathrm{gr}}_{ \mathfrak{M}_{n}})$ is not a geodesic compact metric space, but this problem can be circumvent  by  replacing $V(\mathfrak{M}_n)$ with the union of all its edges, each edge being represented by a copy of the interval $[0, (\scal n)^{-\frac{1}{2\alpha}}]$. We denote the space of isometry classes  of  (resp. rooted) weighted  geodesic compact metric spaces by $\mathbb{PM}$ (resp. $\mathbb{PM}_{\mathrm{root}}$). By \cite[Theorem 7.5.1]{BBI01},  $\mathbb{PM}$ and $\mathbb{PM}_{\mathrm{root}}$ are closed subsets of $\mathbb{M}$ and $\mathbb{M}_{\mathrm{root}}$ respectively. In particular, since the convergences of Proposition~\ref{theo:sub} plainly hold if we add the edges of the map (as copies of $[0, (\scal n)^{-\frac{1}{2\alpha}}]$ for $n\geq 1$), we infer that $\mathbf{P}$-a.s. we can view
$([0,1]/\sim_D,D, \mathrm{Vol}_D)$  and  $([0,1]/\sim_D,D, \mathrm{Vol}_D,\rho_*), ([0,1]/\sim_D,D, \mathrm{Vol}_D,\Pi_D(0)) $ as elements of
   $\mathbb{PM}$ and $\mathbb{PM}_{\mathrm{root}}$, respectively. In particular,  $([0,1]/\sim_D,D)$ is a geodesic metric space.

\subsection{Bounds on $D$} \label{sec:boundsonD}

In this section, we derive a few bounds on the pseudo-metrics $D$ and $D^{*}$ obtained by passing simple discrete estimates to the scaling limit. These bounds will be especially useful when proving  that $\sim_{D}$ and $\sim_{D^{*}}$ are $\mathbf{P}$-a.s.~the same equivalence relations (Theorem \ref{main_theorem_topology}). 
\subsubsection{Simple geodesics and $D \leq D^*$}
\label{sec:D<D*}
 First, recall from Proposition \ref{distinct} that the minimum of $Z$ is realized at a unique random time $t_*$.  Combining \eqref{dist:v_*:BDG}, the convergence \eqref{eq:second_couple} and the continuity of $Z$,
it follows that:
\begin{align}\label{Distance:rho_*}
 D( t_{*},{s}) = Z_s- Z_{t_*},\quad s \in [0,1].
\end{align} 
Similarly, it is possible to take limits along  the subsequence $(n_{k})_{k\geq 1}$ in \eqref{d_n^circ:BDG} to obtain the well-known continuous version of the Schaeffer bounds: Namely, for every $s,t \in [0,1]$, we have 
\begin{equation}\label{trivial_bounds}
|Z_{t}-Z_{s}|\leq D(t,s)\leq Z_{s}+Z_{t}-2\max\left( \min_{[s \wedge t, s \vee t]} Z , \min_{[s \vee t, s\wedge t]}Z \right) \underset{ \mathrm{Def. \ } \eqref{eq:defD0} }{=:}  \mathfrak{z}(s,t).
\end{equation}
This has been proved in  \cite[display after $(74)$]{LGM09} using variants of the coding functions $S^{\bT},L^{\bT}$ for the contour sequence rather than the lexicographical sequence, see \cite[Section 6]{LGM09}.

Let us introduce the notion of discrete and continuous \textbf{simple geodesics} which will give a geometric interpretation to  \eqref{trivial_bounds}. In the discrete BDG$^\bullet$ construction, if we pick a corner white $c$ of $ \bT$, we can draw (in red on Figure \ref{fig:cactus}) the associated \textbf{simple geodesics} targeting $v_{*}$ obtained by starting from $c$ and following the arcs to its consecutive successors. As we observed already just before \eqref{dist:v_*:BDG}, the path produced this way is a (discrete) geodesic which we denote by $\gamma^{{(c)}}$. By construction, the two simple geodesics started from different corners $c$ and $c'$ will merge at a corner of label   \begin{eqnarray} \label{eq:defintervallecorner}\min\left(  \min\limits_{[c,c']_{\mathcal{T}}}\ell;\min\limits_{[c',c]_{\mathcal{T}}}\ell \right) -1,   \end{eqnarray} where $[c_1,c_2]_{ \mathcal{T}}$ is the set of white corners going from $c_1$ to $c_2$ in clockwise order around $ \mathcal{T}$. If \eqref{eq:defintervallecorner} is smaller than the minimal label on $\bT$, then $\gamma^{(c)}$ and $\gamma^{(c')}$ merge at the vertex $v_{*}$ directly.

Let us now pass this construction to the scaling limit. For every $s\in[0,1]$, we set 
$$\eta^{(s)}_r:=\left\{\begin{array}{ll}
\inf\{t\geq s: Z_{t}=Z_s-r\}\quad&\hbox{if }\ \inf\{Z_{t}:t\geq s\}\leq Z_s-r,\\
\inf\{t\leq s: Z_t=Z_s-r\}\quad&\hbox{if }\ \inf\{Z_{t}:t\geq s\}> Z_s-r,
\end{array}
\right.  $$
for $r\in [0, Z_s -\inf Z]$. 
Remark that, by the very definition of $\eta^{(s)}$ and Proposition \ref{distinct}, we have $\eta^{(s)}_{Z_s-\inf Z}=t_*$ and that by \eqref{Distance:rho_*} and \eqref{trivial_bounds} the path $(\Pi_D \circ \eta^{(s)}(r),0\leq r \leq Z_{s}-\inf Z)$ is a geodesic in $([0,1]\sim_{D},D)$ which we call the \textbf{simple geodesic} started from  $s$ (it is a geodesic from $\Pi_D(s)$ to $\Pi_D(t_{*})$), and we denote it by $\gamma^{(s)}$. The reader has noticed that we use a similar notation as in the discrete setting, however the context should avoid any confusion. As in the discrete setting,   the two simple geodesics $\gamma^{{(s)}}$ and $\gamma^{{(t)}}$ in $([0,1]/\sim_{D},D)$ merge at the point $$\gamma^{(s)}\left( \frac{Z_{s}-Z_{t} + \mathfrak{z}(s,t)}{2} \right) = \gamma^{(t)}\left( \frac{Z_{t}-Z_{s} + \mathfrak{z}(s,t)}{2} \right)$$ of label equal to $ \max (\min_{[s\wedge t, s \vee t]} Z, \min_{[0, s \wedge t] \cup [s \vee t, 1]}Z)$, and coincide afterwards. In words, the bound $D \leq \mathfrak{z}$ just says that to go from $\Pi_D(s)$ to $\Pi_D(t)$ we can go towards the root using $\gamma^{(s)}$ until we merge with $\gamma^{(t)}$, and then ``climb-up'' $\gamma^{(t)}$ to $\Pi_D(t)$.  Formally, we denote by $\gamma^{(s\to t)}$ the path obtained by
$$
\gamma^{(s\to t)}(r):=\left\{\begin{array}{ll}
\gamma^{(s)}(r)\quad&\hbox{if }\ 0 \leq r\leq \frac{Z_{s}-Z_{t} + \mathfrak{z}(s,t)}{2} ,\\
\gamma^{(t)}( \mathfrak{z}(s,t) -r)\quad&\hbox{if }\  \frac{Z_{s}-Z_{t} + \mathfrak{z}(s,t)}{2}  \leq r \leq  \mathfrak{z}(s,t).
\end{array}\right.$$
In particular $\gamma^{(s\to t)}$ is a path connecting $\Pi_D(s)$ and $\Pi_D(t)$  and  its $D$-length is equal to $ \mathfrak{z}(s,t)$. For further use, notice from the above discussion that: \begin{eqnarray} \label{eq:geocoincide} \gamma^{(s)} \mbox{ and }\gamma^{(t)} \mbox{ coincide outside of } B_{D}( \Pi_{D}(s), \mathfrak{z}(s,t)).  \end{eqnarray}

\bigskip

It is also not hard to see that $D$ passes to the quotient of the looptree, which means with the notation of Section \ref{sec:looptree} that we have:

\begin{prop}\label{D_sur_le_loop_tree}
$\mathbf{P}$-a.s., for every $(s,t)\in[0,1]$, if $s\sim_{d}t$ then we have $s\sim_{D}t$. 
\end{prop}
\begin{proof} The proof necessarily goes back to the discrete setting. Fix $0 \leq s < t \leq 1$ that are identified in the looptree, i.e.~from  Proposition \ref{topologie_loop_tree} such that: 
\[X_{s-}=X_{t}\:\:\text{and}\:\:X_{r}>X_{s-}\:\:\text{for every}\:\:r\in(s,t).\]
Using \eqref{eq:second_couple} we can find integers $i_{n}$,$j_{n}$ such that 
$\frac{i_{n}}{n}\to s \:\:,\:\:\frac{j_{n}}{n}\to t$
and
\[S_{i}^{\bTn}>S_{i_{n}}^{\bTn}=S_{j_{n}}^{\bTn},\quad\text{for every}\:\:i\in ]\!]i_{n},j_{n}[\![.\]
This means that the white vertex $v_{j_{n}}^{n}$ is the last grandchildren of $v_{i_{n}}^{n}$ (in clockwise order). In particular,  by the Schaeffer bound \eqref{d_n^circ:BDG} we see that $d_{n}(i_{n},j_{n})$ is bounded above by $2$ plus the maximum label displacement between two consecutive white vertices around a black vertex of $ \bTn$. After conditioning on the plane tree $\mathcal{T}_n$, it is easy to see that this maximal displacement is $o_{a.s.}(n^{ \varepsilon})$ for every $ \varepsilon>0$ using \cite[Lemma 1]{LGM09}. As a consequence, we have $d_{n}(i_{n},j_{n}) = o_{a.s.}(n^{ \frac{1}{2 \alpha}})$, where the upper bound is uniform over the choice of the identified points $s,t$.  Passing to the limit along the subsequence $(n_k)_{k\geq 1}$ with \eqref{eq:second_couple} and using the continuity of the limiting pseudo-distance $D$, it follows that $D(s,t)=0$ as  desired.  \end{proof}
\noindent We now introduce the pseudo-distance $D^{*}$, which is the largest pseudo-distance (i.e.~a symmetric, non-negative function obeying the triangle inequality) which is bounded above by $ \mathfrak{z}$ and passes to the quotient of the looptree. Namely, it is given by the  expression   \eqref{def:dstar} which we recall here:
\begin{eqnarray} \label{def:dstar:2}D^{*}(s,t):=\inf \sum \limits_{k=1}^{p}\mathfrak{z}(s_{k},t_{k})~,  \end{eqnarray}
where the infimum is over all choices of the integer $p\geq 1$ and all finite sequences $(s_{k},t_{k})_{1\leq k\leq p}$ with $s=s_1$ and $t_p=t$ such that $ t_{k} \sim_{d} s_{k+1}$ for every $1\leq k\leq p-1$. Notice that  \eqref{Distance:rho_*} and \eqref{trivial_bounds} also hold for $D^*$ instead of $D$, and as a consequence, the simple geodesics defined above are also geodesics for the pseudo-distance $D^*$. We record the equivalent of \eqref{Distance:rho_*} for later uses: \begin{align}\label{Distance*:rho_*}
 D^*( t_{*},{s}) = Z_s- Z_{t_*},\quad s \in [0,1].
\end{align} 
In particular, by \eqref{trivial_bounds} and Proposition \ref{D_sur_le_loop_tree}, we have: \begin{equation}\label{triangle_inequality}
D\leq D^{*}.
\end{equation}
The previous display ensures that the pseudo-distance $D$ factorizes through the projection $\Pi_{D^{*}}$ -- the canonical projection associated with $\sim_{D^*}$.  If we keep the notation $D$ for this projection, we obtain that $D$ can be seen as a pseudo-distance on $$\mathcal{S} := [0,1]/\sim_{D^*}$$ and the goal of Section \ref{sec:topology} is to show that $D$  actually defines a distance on $\mathcal{S}$ and to understand the induced topology. Before doing that, let us deduce a few easy estimates on the distance $D$ from our study of the process $Z$: 
\begin{lem}\label{D:small} \label{D:expectation}
  $\mathrm{(i)}$ The $D$-diameter has stretch exponential tails, i.e. there exist $c,C, \beta >0$ such that 
$$   \mathbf{P}\Big(\sup_{s,t \in [0,1]}D(s,t) > x\Big) \leq C \mathrm{e}^{- c x^{\beta}}.$$

$\mathrm{(ii)}$ For any $\delta >0$, there exists $C_\delta>0$ so that if $U_{1},U_{2}$ are two independent uniform random variables on $[0,1]$, also independent of  $(X,Z,D)$, then
$$\mathbf{P}(D(U_{1},U_{2})\leq \eps)\leq C_{\delta}\cdot \eps^{2\alpha-\delta},\quad \text{ for every } \eps\geq 0.$$
\end{lem}
\begin{proof} The first point is a consequence of \eqref{trivial_bounds} and point (i) in Proposition \ref{variations_Z}. Namely,  \eqref{trivial_bounds} entails that $\sup_{s,t \in [0,1]}D(s,t)  \leq 2( \sup Z-\inf Z)$. We can then use the fact that the random variables $\sup Z$ and $-\inf Z$ have same distribution to derive the first point from Propositions \ref{variations_Z} and \eqref{b_Brownian_Brigde}. For the second point, write $\rho_{1} = \Pi_{D}(U_{1})$ and $\rho_{2} = \Pi_{D}(U_{2})$ and notice that for any $\varepsilon, \delta >0$ we have
\begin{align*}
\mathbf{P}\big(D(\rho_{1},\rho_{2})\leq \varepsilon \big)&\leq \mathbf{P}\Big(\mathrm{Vol}_{D}\big(B_{D}(\rho_{1},\varepsilon)\big)\geq \varepsilon^{2\alpha-\delta}\Big)+ \mathbf{P}\Big( \rho_{2}\in B_{D}(\rho_{1},\varepsilon), \mathrm{Vol}_{D}\big(B_{D}(\rho_{1},\varepsilon)\big)< \varepsilon^{2\alpha-\delta}\Big).
\end{align*}
Conditionally on $(X,Z,D)$, since $U_2$ is uniform on $[0,1]$ the point $\rho_2$ has law $ \mathrm{Vol}_D$, so the second term in the right-side hand is by definition  smaller than $ \varepsilon^{2 \alpha -\delta}$. For the first term, recall that by \eqref{Distance:rho_*} and re-rooting invariance (Theorem \ref{theo:sub})  we have
$$ \Big( \mathrm{Vol}\big(B_{D}(\rho_1,\varepsilon)\big),  \text{ under } \mathbf{P} \Big)\overset{(d)}{=}\Big(  \mathrm{Vol}_\mathfrak{z}\big(B_\mathfrak{z}(\Pi_\mathfrak{z}(t_*),\varepsilon)\big), \text{ under } \mathbf{P} \Big), $$
so that the probability $ \mathbf{P}\Big(\mathrm{Vol}_{D}\big(B_{D}(\rho_{1},\varepsilon)\big)\geq \varepsilon^{2\alpha-\delta}\Big)$ decays stretched exponentially fast as $ \varepsilon \to 0$ for fixed $\delta >0$ by Theorem \ref{technical_uniform_balls}.
\end{proof}

\subsubsection{Le Gall's re-rooting trick} \label{sub:reroot}
Imagine for an instant that our main result, Theorem \ref{thm:main}, was proved. Then, combining the equality $D=D^*$ with the re-rooting property of $D$  inherited from the discrete setting \eqref{eq:re-rooting-S}, we would deduce the following straightforward corollary: For $U,V$ independent uniform random variables on $[0,1]$, independent from $(X,Z,D)$, it holds that
\begin{equation}
  \label{eq:rerootDstar}
  D^*(U,V)\overset{(d)}=D^*(U,t_*)\, .  
\end{equation}
This corollary could seem much less powerful than the equality $D=D^*$. However, Le Gall's argument \cite[Section 8.3]{LG11}, which we reproduce below, shows that the above display, together with the rough inequalities deduced in the previous section is actually sufficient to prove $D=D^*$:

\begin{lem}[Le Gall's re-rooting trick]  \label{lem:LGtrick}Suppose that the two continuous pseudo-distances $D,D^*$ satisfy $D \leq D^*$, together with \eqref{Distance:rho_*}, \eqref{Distance*:rho_*} and are both invariant by re-rooting  \eqref{eq:re-rooting-S}, \eqref{eq:rerootDstar}. Then $D=D^*$ a.s.
\end{lem}
\proof Independently of $D,D^*$, let $U,V$ be two independent uniform random variables on $[0,1]$. 
It follows from \eqref{eq:re-rooting-S}, \eqref{eq:rerootDstar} that $D(U,V)$ and $D^*(U,V)$ have respectively the same law as $D(U,t_*)$ and $D^*(U,t_*)$, which by \eqref{Distance:rho_*}, \eqref{Distance*:rho_*} are equal. Using  $0\leq D(U,V)\leq D^*(U,V)$ we deduce that $D(U,V)=D^*(U,V)$ almost-surely.  By continuity we have $D=D^*$, as wanted.
\endproof 

The key observation is that the invariance under re‑rooting of \eqref{eq:rerootDstar} is a (rather intricate) measurable property that depends only on the processes $(X,Z)$. However, once Theorem \ref{thm:main} is established for one particular model of random map, \eqref{eq:rerootDstar} automatically follows. In turn, Lemma \ref{lem:LGtrick} then shows—a posteriori—that Theorem \ref{thm:main} holds for any random‑map model to which the results of the preceding section apply, and, in particular, for every $\mathbf{q}$‑Boltzmann map whose weight sequence $\bq$ is non‑generic with exponent $\alpha$. Given this, the remainder of our route to Theorem \ref{thm:main} can focus on a single non‑generic weight sequence with exponent $\alpha$. From now on we therefore assume that $\mathbf{q}$ is strictly non‑generic in the sense of Section \ref{sec:non-generic}, principally so that the local‑limit theorems of Section~\ref{sec:scalingunicyclo} are available.

\begin{center} \hrulefill \textit{ From now on, we suppose that $ \mathbf{q}$ is strictly non-generic with exponent $\alpha$.} \hrulefill  \end{center}

\subsubsection{Cactus-Bound}
While the definition of $D^*$ given in Section \ref{sec:D<D*} is based on an upper bound on $D$, we now prove a lower bound on $D$ improving upon the trivial inequality $ D(s,t) \geq |Z_s-Z_t|$ of \eqref{trivial_bounds}. In this direction, we elaborate on the so-called ``\textbf{Cactus bound}", see  \cite{CLGMcactus}. Recall the notation 
  \begin{eqnarray}\text{Branch}(r_{1},r_{2}):=\big\{r\in [0,1]:\: r_{1}\curlywedge r_{2}\prec r\prec r_{1}\:\:\text{or}\:\:r_{1}\curlywedge r_{2}\prec r\prec r_{2}\big\}\cup \big\{r_{1},r_{2}\big\}.   \label{eq:defpinch:recall} \end{eqnarray}
  for every $r_{1},r_{2}\in [0,1]$ and that the  set $\Pi_{d}(\mathrm{Branch}(r_{1},r_{2}))$ is precisely the set of pinch points between $\Pi_{d}(r_{1})$ to $\Pi_{d}(r_{2})$ together with $\{\Pi_{d}(r_{1}),\Pi_{d}(r_{2})\}$. Our new bound  reads as follows:
\begin{lem}[Cactus-bound] \label{lem:cactusbound}
$\mathbf{P}$-a.s., for every $0\leq s<t\leq  1$ and every $r_{1}\in [0,s]\cup [t,1]$, $r_{2}\in [s,t]$ we have:
\begin{equation}\label{cactus}
    D(s,t)\geq (Z_{s}-Z_{r_{1}})\wedge (Z_{s}-Z_{r_{2}})\wedge   \min\limits_{r\in \mathrm{Branch}(r_{1},r_{2})} |Z_{s}-Z_{r}|~.
\end{equation}
\end{lem}
The above bound is only useful when $Z_{r_1}$ and $Z_{r_2}$ are smaller than $Z_s$. The reader may replace the role of $s$ and $t$ in the second part of the proof of Lemma \ref{lem:cactusbound} to obtain \eqref{cactus} with $s$ replaced by $t$ in the right-hand  term. We break the symmetry to simplify notation and since~\eqref{cactus} will be sufficient for our applications.
\begin{figure}[!h]
 \begin{center}
 \includegraphics[height=8cm,width=10cm]{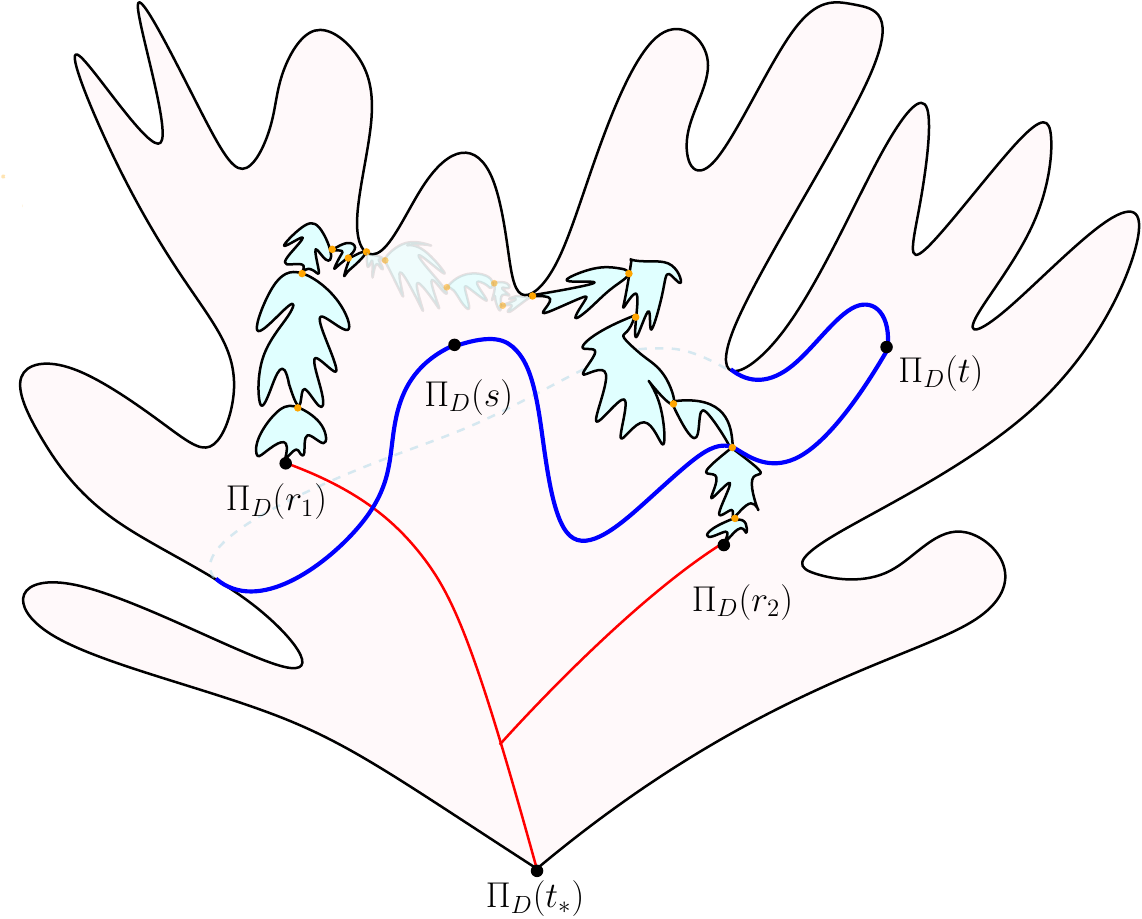}
 \caption{\label{fig:cactus:bound}Illustration of the Cactus bound in the cactus representation of $([0,1]/\sim_D,D)$. The vertical distances represent  distances to $\Pi_{D}(t_*)$. In red, we can see the two simple geodesics starting from $\Pi_{D}(r_{1})$ and $ \Pi_{D}(r_{2})$ respectively. The bound follows by arguing that a $D$-geodesic going from $\Pi_{D}(s)$ to $\Pi_{D}(t)$ must either cross the union of those two geodesics, or must pass through a pinch point on the branch in-between $\Pi_{D}(s)$ and $\Pi_{D}(t)$ (in orange on the figure above) since they are ``blocked'' by the faces of $[0,1]/\sim_D$ (in light blue in the figure).}
 \end{center}
 \end{figure}
\begin{proof} We use the notation of the statement of the lemma and we  suppose that $r_{1}\in [0,s)$, the case $r_{1}\in (t,1]$ can be obtained with the same method by symmetry.  We first establish a similar bound in the discrete setting using the BDG construction and then pass it to the limit. See Figure \ref{fig:cactus} for an illustration. Consider, in the $\mathbf{q}$-Boltzmann mobile  $\bTn= ( \mathcal{T}_{n}, \ell_{n})$, four white vertices  $v_{r_1^n}, v_{s^n}, v_{r_2^n}, v_{t^n}$ visited in the lexicographical order at times $0 \leq r_{1}^{n}< s^{n} < r_{2}^{n} < t^{n} \leq n-1$. 

\begin{figure}[!h]
 \begin{center}
 \includegraphics[width=13.5cm]{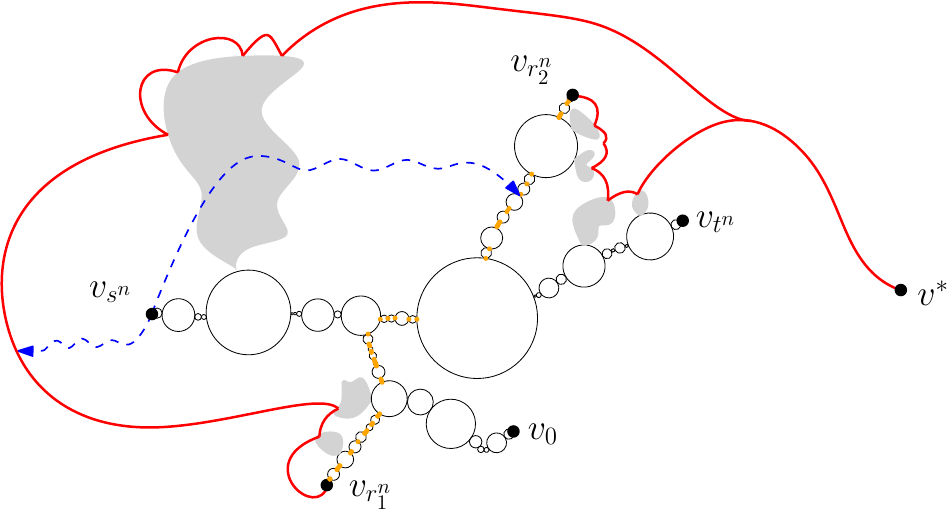}
 \caption{\label{fig:cactus} Illustration of the cactus bound in the construction of maps from well-labeled mobiles. We draw in red the geodesics going from $v_{r_{1}^{n}}$ and $v_{r_{2}^{n}}$ to $v_{*}$ obtained in the construction of BDG bijection. In orange we draw the set  $\mathrm{Branch}(v_{r_{1}^{n}},v_{r_{2}^{n}})$.}
 \end{center}
 \end{figure}
From the first corners in the contour of the vertices visited at times $r_{1}^{n}$ and $ r_{2}^{n}$ we draw (in red on Figure \ref{fig:cactus}) the {simple geodesics} targeting $v_{*}^{n}$, obtained by following the arcs to their consecutive successors in the BDG$^\bullet$ construction. Those two paths eventually merge and reach the distinguished vertex $v_{*}^{n}$. We also consider the set $\mathrm{Branch}(v_{r_{1}^{n}},v_{r_{2}^{n}})$ (in orange on Figure \ref{fig:cactus})  of all white vertices in the branch in $ \mathcal{T}_n$ between $v_{r_{1}^{n}}$ and $v_{r_{2}^{n}}$ (forgive the abuse of notation). Our goal  now is to give a lower bound of  the distance between $v_{s^{n}}$ and $v_{t^{n}}$. In this direction, consider a geodesic path $\gamma$ (in blue in Figure \ref{fig:cactus}) going from $v_{s^{n}}$ to $v_{t^{n}}$ in the map $ \mathfrak{M}_{n} = \mathrm{BDG}^\bullet( \bTn, \epsilon)$. By planarity and from the construction of the edges of the map, this path must visit (the vertices of) one of the two simple geodesics above\footnote{In particular, remark that the simple geodesics go around $ \mathcal{T}_n$ and do not cross it.}, or the set $\mathrm{Branch}(v_{r_{1}^{n}},v_{r_{2}^{n}})$. In the first case, since the labels of the vertices along these simple geodesics are all less than or equal to $ \max (\ell_n( v_{r_{1}^{n}}), \ell_n( v_{r_{2}^{n}}))$,  we obtain using the lhs of \eqref{trivial_bounds} that: 
$$ \mathrm{d}^{ \mathrm{gr}}_{ \mathfrak{M}_{n}}(  v_{s^{n}}, v_{t^{n}}) \geq \ell_n( v_{s^{n}}) + \ell_n( v_{t^{n}}) - 2 \max (\ell_n( v_{r_{1}^{n}}), \ell_n( v_{r_{2}^{n}})).$$
In the second case, we use the same bound and get: 
$$ \mathrm{d}^{ \mathrm{gr}}_{ \mathfrak{M}_{n}}(  v_{s^{n}}, v_{t^{n}}) \geq \min_{v \in \mathrm{Branch}(v_{r_{1}^{n}},v_{r_{2}^{n}})}\left( \left|\ell_n( v_{s^{n}})- \ell_n(v)\right| + \left|\ell_n( v_{t^{n}})- \ell_n(v)\right|\right).$$ In all cases we can thus write:
$$ \mathrm{d}^{ \mathrm{gr}}_{ \mathfrak{M}_{n}}(  v_{s^{n}}, v_{t^{n}}) \geq \min \left\{\min_{v \in \mathrm{Branch}(v_{r_{1}^{n}},v_{r_{2}^{n}})}\left|\ell_n( v_{s^{n}})- \ell_n(v)\right| ; \ell_n( v_{s^{n}})- \ell_n( v_{r_{1}^{n}}) ; \ell_n( v_{s^{n}}) - \ell( v_{r_{2}^{n}}) \right\}.$$
The desired estimate on $D$ is obtained by passing to the limit in the above display along the subsequence $(n_k)_{k\geq 1}$ using \eqref{eq:second_couple}. Let us proceed.
 Fix $0 \leq r_{1}  < s < r_{2} < t \leq 1$, and fix discrete times  $s^{n},t^{n}, r_{1}^{n}$ and $r_{2}^{n}$ as above  such that 
 $$ \left(\frac{s^{n}}{n},\frac{t^{n}}{n}, \frac{r_{1}^{n}}{n}, \frac{r_{2}^{n}}{n} \right) \to (s,t,r_{1},r_{2}),$$ almost surely along the subsequence $(n_k)_{k\geq1}$.
 By virtue of the convergence \eqref{eq:second_couple} along the subsequence $(n_{k})_{k\geq1}$, and  the continuity of the process $Z$, we have that the renormalized labels $(\scal n)^{-1/2\alpha}L^{\bTn}_\cdot$ at times $s^n,t^n,r_1^n$ and $r_2^n$ converge to $Z_s,Z_t,Z_{r_1}$ and $Z_{r_2}$. Therefore, to prove the proposition, it remains to establish that  along $(n_k)_{k\geq 1}$ we have:
 \begin{eqnarray} \label{eq:goalpinch}\liminf_{n \to \infty} \   (\scal n)^{-\frac{1}{2 \alpha}} \cdot \min_{v \in \mathrm{Branch}(v_{r_{1}^{n}},v_{r_{2}^{n}})}\left|\ell_{n}( v_{s^{n}})- \ell_{n}(v)\right| \geq  \min_{r \in \mathrm{Branch}(r_{1},r_{2})} | Z_{s}- Z_{r}|.  \end{eqnarray}
 In this direction, consider a visit time $h^{n}$ of a white vertex in $\mathrm{Branch}(v_{r_{1}^{n}},v_{r_{2}^{n}})$ realizing the minimum in the left-hand side of the previous display. By compactness, up to a further subsequence extraction, one can suppose that $h^{n}/n \to h \in [0,1]$, so that it suffices to prove that $h\in \mathrm{Branch}(r_1,r_2)$ to conclude. In order to avoid trivialities, we will assume that $h\notin\{r_1,r_2\}$.
 In terms of the discrete exploration $S^{\bTn}$, we must have one of the two alternatives:
\begin{equation}
  \label{eq:hnfirst}
  h^{n}\leq r^{n}_{1}\:\:;\:\:S_{h^{n}}^{\bTn}= \min \limits_{[\![h^{n}, r^{n}_{1}]\!]} S^{\bTn}\:\:\text{and }\:\:S_{h_{n}}^{\bTn}\geq  \min \limits_{[\![ r^{n}_{1}, r^{n}_{2}]\!]} S^{\bTn}
\end{equation}
or
\begin{equation}
  \label{eq:hnsecond}
  r^{n}_{1}\leq h^{n}\leq r^{n}_{2}\:\:\text{and }\:\:S_{h^{n}}^{\bTn}= \min \limits_{[\![h^{n}, r^{n}_{2}]\!]} S^{\bTn}.
\end{equation}
Let us pass these estimates to the limit using \eqref{eq:second_couple}, along the subsequence discussed above. This requires a little bit of care, since the evaluation functions $f\in \mathbb{D}([0,1],\mathbb{R})\mapsto f(u)$ are not continuous for the Skorokhod topology. However, if $(f_n)_{n\geq 1}$ is a sequence in $\mathbb{D}([0,1],\mathbb{R})$ converging to some limit $f$, and $(a_n)_{n\geq 1}$ and $(b_n)_{n\geq 1}$ are two sequences in $[0,1]$ with $a_n\leq b_n$ and with limits $a_n\to a$ and $b_n\to b$, then it holds that the limit points of $(f_n(a_n))_{n\geq 1}$ and $(\min_{[a_n,b_n]}f_n)_{n\geq 1}$ are included respectively in
$\{f(a-),f(a)\}$ and $\{\min_{[a,b)}f, \min_{[a,b]}f, f(a-)\wedge \min_{[a,b)}f\}$. If moreover $f$ has only positive jumps, then this second set is equal to $\{\min_{[a,b]}f, f(a-)\wedge \min_{[a,b]}f\}$. 
We will apply these few remarks to $f_n=2(\scal n)^{-1/\alpha}S^{\bTn}_{\lfloor (n-1)\cdot\rfloor}$ and $f=X$, always considering $n$ along appropriate subsequences.

At first, let us suppose that (\ref{eq:hnfirst}) holds
for infinitely many values of $n$. Then, passing to the limit along these values, we have $h< r_1$ (since we assumed that $h\neq r_1$), and we further distinguish two sub-cases:  

(i) If we further assume that $X_{h-}=X_h$, then passing to the limit in (\ref{eq:hnfirst}) yields
$$X_{h-}=X_h\leq \min_{[h,r_1]}X\, \quad \mbox{and}\quad X_{h-}=X_h\geq X_{r_1-}\wedge \min_{[r_1,r_2]} X\, ,$$
where we applied the above discussion with $a_n=h^n/(n-1),b_n=r_1^n/(n-1)$ to get the first inequality, and $a_n=r_1^n/(n-1),b_n=r_2^n/(n-1)$ to get the second one. But the first inequality implies in particular that $X_{h-}\leq X_{r_1-}$, which shows that the second inequality can be replaced by $X_{h-}\geq \min_{[r_1,r_2]} X$. 

(ii) If, on the other hand, it holds that $X_{h-}<X_h$, then we claim that $2(\scal n)^{-1/\alpha}S^{\bTn}_{h^n}$ converges to $X_{h-}$. Otherwise, we could find a subsequence along which it rather converges to $X_h$, and the same argument as before would imply that $X_h\leq \min_{[h,r_1]}X$, a contradiction with the property ($A_1$) in Section \ref{sec:defloop}. Then, the same limiting argument as above implies again that 
\[X_{h-}\leq  \min \limits_{[h, r_{1}]} X\:\:\text{ and }\:\:X_{h-}\geq  \min \limits_{[r_{1},r_{2}]} X\, ,\]
in both cases, this implies that $h\in \text{Branch}(r_{1},r_{2})$. 
Finally, it remains to discuss the situation where (\ref{eq:hnsecond}) holds for 
 infinitely any values of $n$, in which case a similar argument shows that $r_{1}< h< r_{2}$ and 
\[X_{h-}\leq \min \limits_{[h, r_{2}]} X.\]
Again, this implies that  $h\in \text{Branch}(r_{1},r_{2})$. Hence, in all cases, \eqref{eq:goalpinch} follows as desired.
\end{proof}

\section{The topology of  $ \mathcal{S}$} \label{sec:topology}
In this section, we identify the topology of $( [0,1]/\sim_{D},D)$ of any subsequential limit of rescaled stable discrete $ \mathbf{q}$-Boltzmann maps. We begin by describing the points identified by $\sim_{D}$ and $ \sim_{D^{*}}$ in Theorem~\ref{main_theorem_topology}. Next, we encode these equivalence classes using laminations of the disk. Following the work of Le Gall and Paulin \cite{LGP08}, we then apply Moore's theorem to embed $ \mathcal{S}$ on the sphere $ \mathbb{S}_{2}$. This allows us, in the dilute phase, to identify the topology of $ \mathcal{S}$ as that of the Sierpinski carpet; see Theorem \ref{main-topo}.

\subsection{Point identifications}\label{sect:proof:topo}
This section is devoted to the proof of the following result:
\begin{theo}[Point identification]\label{main_theorem_topology}  $ \mathbf{P}$-almost surely, for every $s,t \in [0,1]$, we have
$$ D(s,t) = 0 \quad \iff \quad D^*(s,t)=0 \quad \iff \quad \left\{ \begin{array}{c} \mathfrak{z}(s,t)=0 \\ \textnormal{ or }\\
d(s,t)=0.\end{array}\right.$$
\end{theo}

As presented in the Introduction, the main idea behind the above result is that if two times $s,t$ are not identified either by  $\mathfrak{z}$ or by $d$ (corresponding to the ``trivial identifications'' mentioned in  the Introduction), then ``there is a face''  separating the two, ensuring they are at a strictly positive distance from one another; see Figure \ref{fig:cactus:bound}. This argument is very different from the one used by Le Gall \cite{LG07} to identify the topology of the Brownian sphere, but it is reminiscent of the approach used in \cite[Theorem 4.3]{bjornberg2019stable}.

\begin{proof}
Fix $0 \leq  s < t \leq 1$ such that $D(s,t)=0$. As announced above, the informal idea is to show that except if $\mathfrak{z}(s,t)=0$ or $d(s,t)=0$ then the (projections of) $s$ and $t$ must be separated by two faces (appearing at the forthcoming times $r_{1}$ and $r_{2}$) which must force $D(s,t)$ to be positive by the Cactus bound \eqref{cactus}.

 Let us begin by a couple of simple observations. First, by \eqref{trivial_bounds} we must have $Z_{s}=Z_{t}$ and we put $z=Z_{s}=Z_{t}$ to simplify notation in the rest of the proof. We suppose next that these times are not identified in the looptree, i.e.  $d(s, t) \ne 0$, and we argue according to  the form of the set $ \mathrm{Branch}(s,t)$ -- as defined in \eqref{eq:defpinch}. Notice that $ \mathrm{Branch}(s,t)$ may be reduced to $\{s,t\}$ (e.g.~if $\Pi_{d}(s)$ and $\Pi_{d}(t)$ are on  a same loop of $ \mathcal{L}$). We first claim that for every $r \in \mathrm{Branch}(s,t)$ we must have $Z_{r} \geq z$. Because otherwise if there exists $r_{1}\in\text{Branch}(s,t)$ such that $Z_{r_{1}}<z$, we could consider $r_{2}$ the unique  element different of $r_{1}$ such that $d(r_{1},r_{2})=0$ and then we must be in the configuration $r_{1}\in  [s,t]$ and $r_{2}\in  [t,s]$ or  $r_{1}\in  [t,s]$ and $r_{2}\in  [s,t]$, and  in both cases we would get 
\[D(s,t)\geq z-Z_{r_{1}}>0,\]
by an application of  the Cactus bound as stated in Lemma \ref{lem:cactusbound}.

\begin{figure}[!h]
 \begin{center}
 \includegraphics[width=16cm]{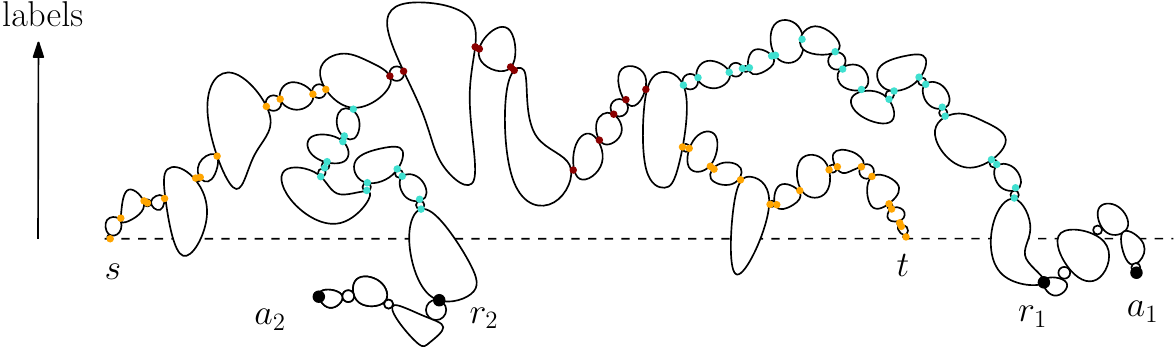}
 \caption{ \label{fig:setup1bis} Setup of the proof of Theorem \ref{main_theorem_topology} after excluding the trivial identifications $d(s,t)=0$ and $\mathfrak{z}(s,t)=0$. We can always find $r_{1}\in (s,t)$ and $r_{2}\in(s,t)$ pinch point times such that $Z_{r_{1}}<z$, $Z_{r_{2}}<z$ and such that the labels on $ \mathrm{Branch}(r_{1},r_{2}) \backslash \{r_{1},r_{2}\}$ are strictly larger than $z$. Then the Cactus bound \eqref{cactus} implies that $D(s,t)>0$.}
 \end{center}
 \end{figure}

Let us now assume that $\mathfrak{z}(s,t) >0$ since otherwise $s,t$ are trivially identified.  We start by treating the case when  $\{r\in\text{Branch}(s,t)\setminus\{s,t\}\::\:Z_{r}=z\}$ is empty. First assume that $s$ is a pinch point time and recall that then  $Z_{r} > z$  for every $r \in \mathrm{Branch}(s,t) \backslash\{s,t\}$. In this  case  the re-rooting property \eqref{re-rooting} and Proposition~\ref{lem:non-icnreasealongbranches} ensure that there are $r_{1}\in(s,t)$ and $r_{2}\in(t,s)$ such that 
\[Z_{r_{1}}<z< Z_{r},\:\:\text{for every}\:\:r \in \text{Branch}(s,r_{1})\setminus \{s,r_{1}\},\]
and
\[Z_{r_{2}}<z<Z_{r},\:\:\text{for every}\:\:r\in \text{Branch}(s,r_{2})\setminus \{s, r_{2}\},\] see Figure \ref{fig:setup1bis} for an illustration.  In particular, since $ \text{Branch}(r_{1},r_{2})\subset  \big(\text{Branch}(s,r_{1})\cup\text{Branch}(s,r_{2}))\setminus\{s\}$, we have $Z_{r}>z$ for every $r\in \text{Branch}(r_{1},r_{2})\setminus\{r_{1},r_{2}\}$. Using that $\text{Branch}(r_{1},r_{2})$ is a compact set, we get $\min \{|z-Z_r|:~r\in  \text{Branch}(r_{1},r_{2})\}>0$. 
Hence,   the Cactus bound (Lemma \ref{lem:cactusbound}), with $(r_{1},r_{2})$,  entails:
\[D(s,t)\geq (z-Z_{r_{1}})\wedge (z-Z_{r_{2}})\wedge   \min\limits_{r\in \text{Branch}(r_{1},r_{2})} |z-Z_{r}|>0,\]
and we can apply exactly the same argument if $t$ is a pinch point time. Assume now that both $s$ and $t$  are leaf times.
Since $\mathfrak{z}(s,t) >0$ there exists $a_{1}\in[0,s[ \cup ]t,1]$ and $a_{2} \in [s,t]$ such that $Z_{a_{1}}<z$ and $Z_{a_{2}}<z$. 
Next, introduce the pinch point times:
\[a_{1}^\prime:=\inf\big\{r\in \mathrm{Branch}(s,a_1)\setminus \{s\}\::\:Z_{r}<z\big\}\quad ;\quad a_{2}^\prime:=\inf\big\{r\in \mathrm{Branch}(t,a_2)\setminus \{t\}\::\:Z_{r}<z\big\}.\]
If $Z_{a_{1}}<z$ (resp. $Z_{a_{2}}<z$) take $r_{1}:=a_{1}$ (resp. $r_{2}:=a_{2}$). If it is not the case, by the discussion at the beginning of the proof  there exists $r_{1}$ (resp. $r_{2}$) with $r_{1}\in ]s,t[$ (resp. $r_{2}\in [0,t^{\prime}[\cup]s^{\prime},1[$) such that:
\[Z_{r_{1}}<z<\inf \{Z_{r}:\: r\in \mathrm{Branch}(a_{1},r_{1})\setminus \{a_{1},r_{1}\}\}\]
(resp. $Z_{r_{2}}<z=\inf \{Z_{r}:\: r\in \mathrm{Branch}(a_{2},r_{2})\setminus \{a_{2},r_{2}\}\}$). 
In every case, we can then apply again the Cactus bound  \eqref{cactus} with $(r_1,r_2)$ to deduce that in this case we also have $D(s,t) >0$, which is excluded. Finally, it remains to treat the case when the set $\{r\in\text{Branch}(s,t)\setminus\{s,t\}\::\:Z_{r}=z\}$ is not empty. Then by Lemma  \ref{lem:onePinch} and the re-rooting property \eqref{re-rooting}, the set $\{r\in\text{Branch}(s,t)\setminus\{s,t\}\::\:Z_{r}=z\}$ contains only one time $r$.   One can then apply Proposition~\ref{lem:non-icnreasealongbranches} to this time and deduce the existence of the required $r_{1},r_{2}$ as above, we refer to Figure \ref{fig:steup2bis} for an illustration.

\begin{figure}[!h]
 \begin{center}
 \includegraphics[width=15cm]{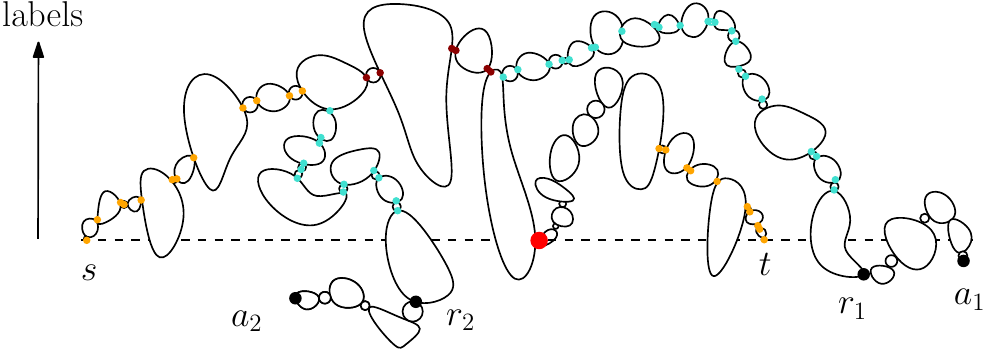}
 \caption{\label{fig:steup2bis} In the case when there exists $ r \in \mathrm{Branch}(s,t) \backslash\{s,t\}$ such that $Z_{r}=z$, then such a time is unique and we can find $r_{1}\in (t,s)$ and $r_{2}\in(s,t)$ directly in the vicinity of $r$ for the looptree distance  $d$ by Proposition~\ref{lem:non-icnreasealongbranches}.}
 \end{center}
 \end{figure}
 \end{proof}

\subsection{Lamination encoding of the equivalence classes and Moore's theorem}
Let us derive an easy consequence of Theorem \ref{main_theorem_topology}.
\begin{cor}\label{cor:topo:D:D^*}
 $\mathbf{P}$ almost surely, we have $\big([0,1]/\sim_D\!\!\big)= \big([0,1]/\sim_{D^*}\!\!\big) =  \mathcal{S}$ as point sets. Moreover the topological spaces
$$ \left( \mathcal{S},D\right)   \mbox{ and  } \left( \mathcal{S},D^*\right)  \mbox{ are a.s. homeomorphic},$$ and their topology is  the quotient topology associated with $\Pi_{D} : [0,1] \to [0,1]/ \sim_{D}$. Finally $ \mathrm{Vol}_D$ and  $\mathrm{Vol}_{D^*}$ define the same measure on the Borel sigma-field of $ \mathcal{S}$ which will from now on be denoted by $ \mathrm{Vol}$.
\end{cor}
\begin{proof} The first statement follows directly from Theorem \ref{main_theorem_topology}. Since $D:[0,1]^{2}\mapsto \mathbb{R}_{+}$ is continuous, the equivalence relation $\sim_{D}$ is closed and the quotient space $( [0,1]/\sim_{D},D)$ is a Hausdorff compact space. The same is true for $( [0,1]/\sim_{D^{*}},D^{*})$. Recalling that $D \leq D^*$, the natural projection $( [0,1]/\sim_{D^{*}},D^{*}) \to ( [0,1]/\sim_{D},D)$ is thus continuous, obviously surjective and, by Theorem \ref{main_theorem_topology}, it is also  injective.	 By a standard result in topology,  its inverse is also continuous and we infer that $( [0,1]/\sim_{D^{*}},D^{*})$ is homeomorphic to $( [0,1]/\sim_{D},D)$. The second point is similar: Let $\mathfrak{T}_{ \mathrm{quotient}}$ be the quotient topology, i.e.~the finest topology that makes $\Pi_{D} : [0,1] \to [0,1]/ \sim_{D}$  continuous. Since $ \sim_{D}$ is closed,  $\mathfrak{T}_{ \mathrm{quotient}}$ separates points. As a consequence $([0,1]/\sim_{D},\mathfrak{T}_{ \mathrm{quotient}})$ is a Hausdorff compact and  we can then repeat the argument above. Finally, since the Borel sigma fields on $( [0,1]/\sim_{D},D)$ and $( [0,1]/\sim_{D^*},D^*)$ coincide and since the projections $\Pi_D$ and $\Pi_{D^*}$ are measurable (when $[0,1]$ is equipped with the Borel sigma field), we deduce from equality of the projections the equality of measures $ \mathrm{Vol}_D = \mathrm{Vol}_{D^*}$.\end{proof}

In the sequel, we will abuse notation and write $\mathcal{S}$ for $ ([0,1]/\sim_{D}) = (0,1]/\sim_{D^{*}}$. As stated in Theorem~\ref{main_theorem_topology} the graph of equivalence classes $\{ (s,t) \in [0,1]^2 : s \sim_{D} t\}$ is the union of:
\begin{equation}\label{sets:topo}
\{(s,t)\in[0,1]^{2}:~d(s,t)=0\}\quad\text{and}\quad  \{(s,t)\in[0,1]^{2}:~\mathfrak{z}(s,t)=0\},
\end{equation}
and by Proposition \ref{pinch_points_are_not_record} the intersection of the above two sets is the diagonal $\{(s,s):~s\in[0,1]\}$.  
Furthermore, by Propositions \ref{prop:timeclassification} and \ref{distinct}, each equivalence class for $\sim_{d}$ or $\sim_{ \mathfrak{z}}$  contains at most $2$ or $3$ points respectively. The interest of writing  the two sets of \eqref{sets:topo} in terms of $X$ and $Z$ is that we can encode these sets as geodesic laminations  in the same vein as Le Gall \& Paulin  \cite{LGP08}, in order to apply Moore's theorem.  In this direction, we  work on $\mathbb{D}$,  the closed unit disk of the complex plane $\mathbb{C}$ and we write $\mathbb{S}_1$ for the boundary of $\mathbb{D}$. 
For every $a,b\in \mathbb{D}$, we  write $[a,b]_{\mathbb{D}}$ for the geodesic arc connecting $a$ and $b$ in $\mathbb{D}$ and set
$p(t):=\exp(2\pi \rm{i} t)$, for every $t\in[0,1]$. Then introduce the two compact subsets of $\mathbb{D}$:
\[L(X):=\bigcup_{s \sim_{d} t}[p(s),p(t)]_{\mathbb{D}}\quad \:\:\text{and}\:\:\quad L(Z):=\bigcup_{s \sim_{ \mathfrak{z}} t}[p(s),p(t)]_{\mathbb{D}}.\]
 By \cite[Propositions 2.9 and 2.10]{Kor11} (resp.  \cite[Proposition 2.5]{CLGrecursive}) the sets   $L(X)$ and $L(Z)$ are geodesic laminations of $\mathbb{D}$, i.e.\ they are unions of a collection of non crossing geodesic arcs $[a,b]_{\mathbb{D}}\setminus\{a,b\}$ with $a,b\in \mathbb{S}_{1}$. Moreover since the minima of $Z$ are distinct, Proposition 2.5 in \cite{CLGrecursive} also implies that $L(Z)$ is maximal -- meaning that any arc $[a,b]_{\mathbb{D}}$ with $a,b\in \mathbb{S}_{1}$ intersects $L(Z)$. It is then  straightforward to verify that  the connected components of $\mathbb{D}\setminus L(Z)$ are open triangles whose vertices are of the form $a,b,c\in\mathbb{S}_{1}$ with $a\sim_{ \mathfrak{z}}b \sim_{\mathfrak{z}} c$.  
 
On the contrary, the  connected components of $\mathbb{D}\setminus L(X)$ are not  triangles. The random set  $L(X)$ has previously been studied by Kortchemski under the name of \textbf{stable lamination} \cite{Kor11}. Let us import a few of his results. Recall that $(\mathrm{t}_i)_{i\in \mathbb{N}}$ denotes the collection of jump times of $X$ and $\mathrm{f}_{\mathrm{t}_i}$ is a parametrization of the associated faces, see \eqref{def:loops}. Proposition 3.10 in \cite{Kor11} states that for every jumping time $\mathrm{t}_i$ there exists a unique connected component $\mathcal{C}_{i}$ of $\mathbb{D}\setminus L(X)$ such that $\text{Cl}(\mathcal{C}_{i})$ is the convex envelope of
\[\big\{p(\mathrm{f}_{\mathrm{t}_i}(s)):\:s\in[0,1]\big\},\]
and $\mathcal{C}_{i}$ is its interior, it also states that the mapping $i\mapsto \mathcal{C}_{i}$  is a one-to-one correspondence between the jumps of $X$ and the connected components of $\mathbb{D}\setminus L(X)$. 

\begin{figure}[!h]
 \begin{center}
 \includegraphics[height=5.5cm]{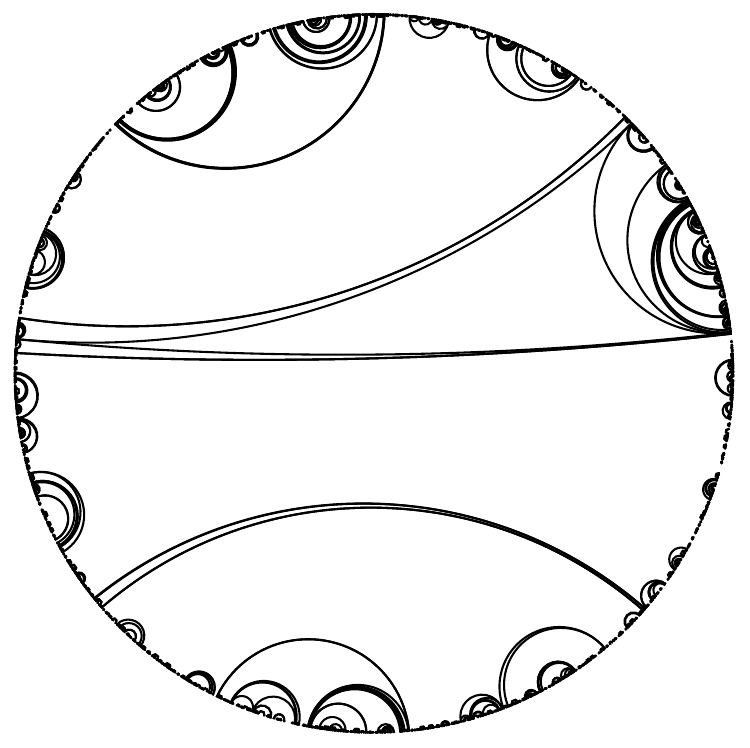}\:\:\:\:\:\:
  \includegraphics[height=5.5cm]{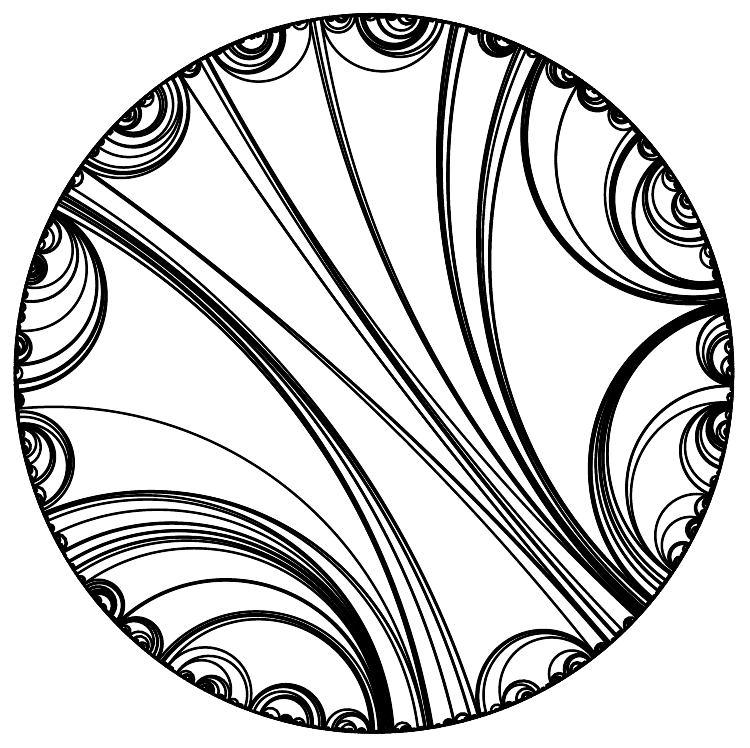}
 \caption{\label{fig:Lamination} Simulations of the geodesic laminations $L(X)$  induced by $\sim_d$ (on the left) and $L(Z)$ induced by $\sim_{ \mathfrak{z}}$ (on the right). For clearness we have changed the Euclidean arcs $[a,b]_{ \mathbb{D}}$ by hyperbolic geodesics. Notice that the connected components of $ \mathbb{D} \backslash L(Z)$ are triangles and convex polygons with infinitely many sides in the case of $L(X)$.}
 \end{center}
 \end{figure}

\paragraph{Lifting the equivalence relations on $\mathbb{S}_2$.} We shall now  ``glue'' these laminations on the sphere. More precisely, consider the $2$-dimensional sphere $ \mathbb{S}_{2}$. The closed upper and lower hemispheres $\mathbb{H}_{+}:=\{(x_{1},x_{2},x_{3})\in \mathbb{S}_{2}:\:x_{3}\geq 0\}$ and $\mathbb{H}_{-}:=\{(x_{1},x_{2},x_{3})\in \mathbb{S}_{2}:\:x_{3}\leq 0\}$  can be identified with two closed disks via the stereographic projections from the poles and this enables us to push the lamination $L(X)$ on $ \mathbb{H}_{+}$ and $L(Z)$ on $ \mathbb{H}_{-}$.

We now define a relation $\approx$ on $ \mathbb{S}_{2}$ using the two images of the laminations $L(X)$ and $L(Z)$:
\begin{itemize}
\item In the upper hemisphere $ \mathbb{H}_{+}$ we say that $a \approx b$ if 
\begin{itemize}
\item $a$ and $b$ belong to the same (image of) an arc of $L(X)$.
\end{itemize}
\item In the lower hemisphere $ \mathbb{H}_{-}$ we say that $a \approx b$ if 
\begin{itemize}
\item $a$ and $b$ belong to the same (image of) an arc of $L(Z)$,
\item or $a$ and $b$ belong to the (image of the) closure of a connected component (a triangle) of $ \mathbb{D}\backslash L(Z)$.
\end{itemize}
\end{itemize}
Notice that in the upper hemisphere we do not identify the points (of the images) of the connected components of $ \mathbb{D} \backslash L(X)$, see Figure \ref{fig:lam:sphere}. We shall abuse notation and write $\mathbb{S}_{1}$ for the equator of $\mathbb{S}_{2}$, obtained as the intersection of the upper and lower hemispheres. 

\begin{figure}[!h]
 \begin{center}
 \includegraphics[height=7cm]{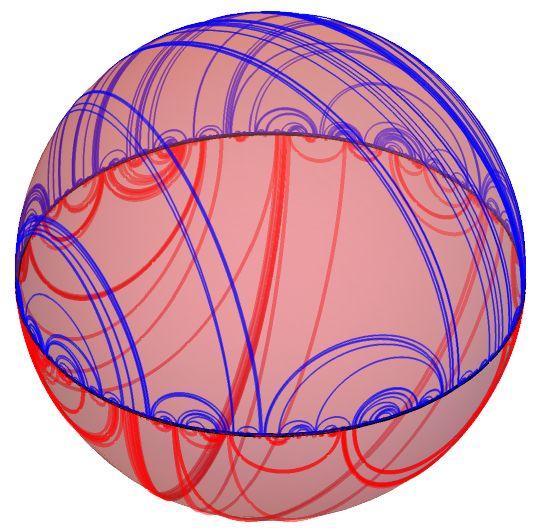}
 \caption{\label{fig:lam:sphere} An illustration of the images of $L(X)$ in blue on the top hemisphere, and $L(Z)$ in red, on the bottom hemisphere.}
 \end{center}
 \end{figure}
 
 \begin{lem}\label{Lem:injection} \label{prop:moore} Almost surely $\approx$ is a closed equivalence relation on $ \mathbb{S}_{2}$ and the quotient space $\mathbb{S}_{2}/\approx$, equipped with the quotient topology, is homeomorphic to $\mathbb{S}_{2}$. Moreover the projection of the equator $\mathbb{S}_{1}/\approx$ is homeomorphic to $(\mathcal{S},D)$. 
\end{lem}

In particular, for any $\alpha \in (1,2)$, there exists a continuous injection from $(\mathcal{S},D)$ onto $\mathbb{S}_{2}$ a.s.
\begin{proof} First, notice that Proposition \ref{pinch_points_are_not_record} ensures that  a.s.\ if two (images of the) equivalence classes of $\sim_{d}$ and $\sim_{ \mathfrak{z}}$ intersect on $ \mathbb{S}_{1}$ then these equivalence classes are reduced to a single point. 
Hence  $\approx$ defines an equivalence relation on $\mathbb{S}_{2}$ and any equivalence class  of $\approx$ is contained in the upper or lower hemisphere (it might be both if and only if the equivalence class is a singleton). In particular, we deduce that  the equivalence classes of $\approx$ are either, a point, an (image of an) arc, or (the image of) a closed triangle with extremities on $ \mathbb{S}_{1}$ in the lower hemisphere. Consequently, $\approx$ is a closed equivalence relation on $\mathbb{S}_{2}$ such that its equivalence classes are  compact path connected subsets of $\mathbb{S}_{2}$ with connected complements. In particular, we are under the assumptions of  Moore theorem \cite{Moore} which  gives that $\mathbb{S}_{2}/ \approx$ is homeomorphic to $\mathbb{S}_{2}$.  
It remains to show that  $\mathbb{S}_{1}/ \approx$ is homeomorphic to $(\mathcal{S},D^{*})$. In this direction, remark that by Theorem \ref{main_theorem_topology}, for every $s,t\in[0,1]$, we have
$p(t)\approx p(s) $ if and only if $D(s,t)=0$. This implies that the identity map from $\mathbb{S}_{1}/ \approx$ onto $(\mathcal{S},D)$ is a continuous bijection. Moreover, since $\approx$ is closed,  both spaces are (Hausdorff and) compact  ensuring that the identity map is an homeomorphism.
\end{proof}

\subsection{Identification of the topology in the dilute case}\label{sect:topo:dilute}

By Lemma \ref{prop:moore}, the topology of $(\mathcal{S},D)$ is that of the image of $ \mathbb{S}_{1}$ after taking the quotient by $\approx$. Informally, the connected components of the complement of the lamination $L(X)$ in $ \mathbb{H}_{+}$ are the ``faces'' of $ \mathcal{S}$. We will refine this image in the dilute phase since a fundamental difference arises depending on the position of $\alpha$ with respect to $3/2$. Specifically, combining Proposition \ref{prop:records-loops} with Theorem~\ref{main_theorem_topology}, we deduce  that if $\alpha \in (1,3/2)$, then $\approx$ may identify two points of the same face (a connected component of $ \mathbb{H}_{+} \backslash L(X)$) via an arc of $L(Z)$ in $ \mathbb{H}_{-}$,  whereas in the dilute case $\alpha \in [3/2,2)$ this does not happen. Our goal now is to characterize the topology of $ \mathcal{S}$, when $\alpha \in [3/2,2)$, as being  that of the Sierpinski carpet, $\mathbf{Sierp}$, which  is the unique homeomorphism type of a non-empty compact connected space $K$ in the sphere $\mathbb{S}_{2}$, such that its complement consists of countably many connected components $C_{1}, C_{2}, \dots$, satisfying:

\begin{itemize}
\item the diameter of $C_{i}$ goes to $0$ as $i \to \infty$;
\item $\bigcup_{i \geq 1} \partial C_{i}$ is dense in $K$;
\item the boundaries $ \partial C_{i}$ of $ C_{i}$ are simple closed curves which do not intersect each other.
\end{itemize}

\noindent We refer to the work of Whyburn \cite{Whyburn} for a proof of this characterization.
\begin{theo}[Identification of the topology in the dilute case]\label{main-topo}
For $ \alpha \in [3/2,2)$,  the space $(\mathcal{S},D)$ is almost surely homeomorphic to the Sierpinski carpet $\mathbf{Sierp}$.
\end{theo}
\begin{proof}
To simplify notation, write $p_+$ for the stereographic projection from the south pole  and 
introduce $\tau:\mathbb{S}_{2}\to\mathbb{S}_{2}/\approx$ the canonical projection. We also  fix $F:(\mathbb{S}_{2}/\approx) \ \to \ \mathbb{S}_{2}$ an homeomorphism which exists a.s.~by Lemma \ref{Lem:injection}. Remark that the map $F\circ \tau$ is continuous and that by the second statement of Lemma \ref{Lem:injection} the set $\mathbb{X}:=F(\tau(\mathbb{S}_{1}))$ is homeomorphic to $(\mathcal{S},D)$. So our goal is to show that $\mathbb{X}$ is homeomorphic to the Sierpinski carpet $\mathbf{Sierp}$. First note that for every $h\in \mathbb{H}_{-}$ there exists $h^{\prime}\in \mathbb{S}_{1}$ such that $h\approx h^{\prime}$ --  since $L(Z)$ is a maximal geodesic lamination and the  connected components of $\mathbb{D}\setminus L(Z)$ are open triangles whose vertices are of the form $a,b,c\in\mathbb{S}_{1}$ with $a\approx b\approx c$. This implies that $(\mathbb{H}_{-}/\approx)~=(\mathbb{S}_{1}/\approx)$ and consequently $\mathbb{X}=F(\tau(\mathbb{H}_{-}))$. 
\par We also get that  the mapping $\mathcal{C}\mapsto F(\tau(p_{+}^{-1}(\mathcal{C})))$  is a one-to-one correspondence between connected components of $\mathbb{D}\setminus L(X)$ and $\mathbb{S}_{2}\setminus \mathbb{X}$. Recall that $(\mathrm{t}_i)_{i\in \mathbb{N}}$ denote the collection of jumping times of $X$ and $\mathrm{f}_{\mathrm{t}_i}$ is a parametrization of the associated faces. We again use the notation $\mathcal{C}_i$ for the interior of the convex envelope of $\big\{p(\mathrm{f}_{\mathrm{t}_i}(s)):\:s\in[0,1]\big\}$.
As recalled in the previous section,  Proposition 3.10 in \cite{Kor11} states that the mapping $i\to \mathcal{C}_{i}$  is a one-to-one correspondence between the jumps of $X$ and the connected components of $\mathbb{D}\setminus L(X)$. Consequently, we have:
\begin{eqnarray}\mathbb{S}_2\setminus \mathbb{X}=\bigcup \limits_{i\in\mathbb{N}}C_{i},   \label{eq:facesembed}\end{eqnarray}
where $C_i=F\circ \tau \circ p_{+}^{-1}(\mathcal{C}_i)$. In particular $\mathbb{S}_{2}\setminus \mathbb{X}$ is not empty. Since by construction $\mathbb{X}$ is a non-empty connected compact metric space embedded on $\mathbb{S}_2$ (it is the image of $[0,1]$ by a continuous function), it is enough to show that $\mathbb{X}$ satisfies Whyburn's topological characterization of the Sierpinski carpet:
 \\
 \\
\textbf{The diameter of $C_i$ goes to $0$ as $i\to \infty$.} Since $[0,1]$ is locally connected,  so it is $\mathbb{X}$. A result of Sch\"onflies, see Theorem 10 page 515 \cite{Kuratowski}, then implies that the diameter of $ C_{i}$ goes to $0$ as $i \to \infty$.  
\\
\\
\textbf{$\bigcup_{i \geq 1} \partial C_{i}$ is dense in $\mathbb{X}$.}  By definition we have $\text{Cl}\big(\bigcup\limits_{i=0}^{\infty}\partial C_{i}\big)\subset F(\tau(\mathbb{S}_{1}))$. Since, for every $i\in \mathbb{N}$,  $\mathrm{t}_i\in \mathcal{C}_{i}$, the density of the jumping times of $X$  in $[0,1]$ implies that:
\[\text{Cl}\big(\bigcup_{i=0}^{\infty}\partial C_{i}\big)= F(\tau(\mathbb{S}_{1}))=\mathbb{X}.\]
 \textbf{The boundaries $ \partial C_{i}$ of $ C_{i}$ are simple closed curves that do not intersect each other.} First note that for every $i\in \mathbb{N}$, the boundary $\partial C_{i}$ is the image of $[0,1]$ by the map $\ell_i:=F\circ \tau\circ p_{+}^{-1}\circ p \circ \mathrm{f}_{\mathrm{t}_i}$. Moreover since $\alpha\in [\frac{3}{2},2)$, for $i\neq j$,  we have $\partial C_{i}\cap \partial C_{j}=\varnothing$ and $\ell_i$ is injective by Proposition~\ref{prop:records-loops} and Theorem~\ref{main_theorem_topology}. Here we also used that, by Proposition \ref{topologie_loop_tree}, for every $t,t^\prime\in\mathrm{f}_{\mathrm{t}_i}([0,1])$, with $t\neq t^\prime$, we must have $d(t,t^\prime)\neq 0$. To conclude it then remains to show that, for every  $i\in \mathbb{N}$, the function  $\ell_{i}$ is also continuous. By definition, the map $F\circ\tau\circ p_{+}^{-1}\circ p$ is continuous and since $\mathrm{f}_{\mathrm{t}_i}$ is rcll, the continuity of  $\ell_{i}$ will follow if we show that, for every $s\in[0,1]$ with $\mathrm{f}_{\mathrm{t}_i}(s)<\mathrm{f}_{\mathrm{t}_i}(s+)$, the two points $\mathrm{f}_{\mathrm{t}_i}(s)$ and $\mathrm{f}_{\mathrm{t}_i}(s+)$  have the same image by  $\tau\circ p_{+}^{-1}\circ p$.  In this direction, remark that  we have $X_{\mathrm{f}_{\mathrm{t}_i}(r)}= X_{\mathrm{t}_i}-r\Delta_{\mathrm{t}_i}$ for every $r\in [0,1]$. Now fix $s\in[0,1]$ such that $\mathrm{f}_{\mathrm{t}_i}(s)<\mathrm{f}_{\mathrm{t}_i}(s+)$.  By definition and the previous remark we get: \[X_{\mathrm{f}_{\mathrm{t}_i}(s)}=X_{\mathrm{f}_{\mathrm{t}_i}(s+)}= X_{\mathrm{t}_i}-s\Delta_{\mathrm{t}_i}\:\:\text{and}\:\: X_{r}\geq  X_{\mathrm{t}_i}-s\Delta_{\mathrm{t}_i}\:\:\text{for every}\:\:r\in(\mathrm{f}_{\mathrm{t}_i}(s),\mathrm{f}_{\mathrm{t}_i}(s+)),\] and it follows straightforwardly from the definition of $d$, given in \eqref{def:distancelooptree}, that  $d(\mathrm{f}_{\mathrm{t}_i}(s),\mathrm{f}_{\mathrm{t}_i}(s+))=0$. Consequently, we have $$\tau(p_{+}\big(p(\mathrm{f}_{\mathrm{t}_i}(s)))\big)=\tau(p_{+}\big(p(\mathrm{f}_{\mathrm{t}_i}(s+)))\big)$$ and we deduce that $\ell_{i}$ is continuous.
 \end{proof}

\noindent Let us conclude this section with a few  remarks. 

\paragraph{Separating cycles.} The embedding $  \mathbb{X} \subset \mathbb{S}^2$ on the sphere constructed in the above proof enables us to apply different versions of the Jordan theorem.  In the dilute case $\alpha \in [3/2,2)$, thanks to Theorem~\ref{main-topo}, we can give an analog of  \cite[Corollary 1.2]{LGP08}. More precisely, recall the notation $\mathfrak{M}_n$ for our Boltzmann map with $n$ vertices and that we can draw it directly on the sphere -- we make this assumption in the rest of this section. Recall that a  path of length $m$ is a sequence $x_{0},e_{1},x_{1},e_{2},..., x_{m-1},e_{m},x_{m}$ where $x_{0},x_{1},...,x_{m}$ are vertices of $\mathfrak{M}_{n}$ and $e_{1},...,e_{m}$ are edges of $\mathfrak{M}_{n}$ such that $e_{i}$ connects $x_{i-1}$ and $x_{i}$ for every $i\in[\![1,m]\!]$, and we say that it is  an injective cycle if $x_{0}=x_{m}$ and the vertices $x_{0},x_{1},...,x_{m-1}$ are all distinct. For an injective cycle $C$, we denote the union of its edges by $R(C)$ and notice that, by the Jordan theorem, $\mathbb{S}_{2}\setminus R(C)$ has two connected components.  Replacing the sphere $ \mathbb{S}_{2}$ by the Siperpinski carpet (which has not cut-point) in the proof of \cite[Corollary 1.2]{LGP08} yields:
\begin{cor}\label{inf:cycle}
Fix $\alpha\in[\frac{3}{2},2)$. Let $\delta>0$, and $\theta:\mathbb{N}\to \mathbb{R}_{+}$ a function such that $\theta(n)=o(n^{\frac{1}{2\alpha}})$ as $n\to \infty$. The probability that $\mathfrak{M}_n$ has an injective cycle $C$ with length smaller than $\theta(n)$ and such that the two connected components of $\mathbb{S}_{2}\setminus R(C)$ have diameter  larger than $\delta n^{\frac{1}{2\alpha}}$ tends to $0$ as $n$ goes to infinity.
\end{cor}
In the dense case $\alpha\in (1,\frac{3}{2})$, the analog of Corollary \ref{inf:cycle} does not hold since it is easy to see that any two loop times identified by $\Pi_{D}$ produce a cut point. Actually, the study of injective cycles in planar maps is a rich topic and very precise results can be obtained. See for example \cite{LGL} for the study of injective cycles in planar quadrangulations, \cite{Iso-Rie} for a direct study in  Brownian geometry which in particular states the  isoperimetric  profile of the Brownian plane, and the recent work \cite{bouttier2022bijective} which used a class of injective cycles to derive bijective enumerations of planar maps with three boundaries. 

\paragraph{Graph of faces.} The embedding $ \mathbb{X} \subset \mathbb{S}^2$ and \eqref{eq:facesembed} gives a precise topological meaning to the ``faces'' of $\mathcal{S}$, as the connected components of the complement of $ \mathbb{X}$ (notice that these faces require the embedding to be defined, especially in the dense phase). Some topological properties of $ \mathbb{X}$ can then be defined in terms of touching faces. In particular, we can consider the graph $ \mathscr{G}$ whose vertices are the faces of $\mathbb{X}$ and where there is an edge between two vertices if the corresponding faces touch each other. We then believe that the techniques developed in  \cite[Section 5.2]{BC16} can be used to prove that:
\begin{conjecture} The graph $ \mathscr{G}$ of the touching faces  is almost surely connected in the dense phase. \end{conjecture}
See \cite[Question 11.2]{DMS14} and \cite{gwynne2020connectivity,doherty2024connectivity} for an analogous  question in SLE random fractals and \cite[Open Problem (4)]{MP10} in the case of  planar Brownian motion.  As we said in the introduction, we believe that the topology of $(\mathcal{S},D)$, or more precisely of its embedding $ \mathbb{X} \subset \mathbb{S}^2$ constructed above is in fact \textit{random}  in the dense case. A similar situation arises for the topology of the SLE$_{\kappa}$ curve when $\kappa >4$, as was shown in \cite{yearwood22}. We take inspiration from this work and give a heuristic argument supporting our belief,  which is summarized in Figure \ref{fig:topologyrandom} and its caption. Based on this heuristic, we make the more precise conjecture:
\begin{conjecture} If $\mathbb{X}'$ is an independent copy of $\mathbb{X}$, then almost surely, no two neighborhoods of $\mathbb{X}$ and $\mathbb{X}'$ can be mapped to each other by a homeomorphism of $ \mathbb{S}^2$.
\end{conjecture}

The heuristic argument, again based on properties of the graph $ \mathscr{G}$ goes as follows. Let us consider two faces, corresponding to $\mathrm{f}_{t_i},\mathrm{f}_{t_j}$, and represented in light and dark green in Figure \ref{fig:topologyrandom}, that are mutually intersecting. We let $s_i<s'_i,s_j<s'_j$ be such that $\Pi_D(\mathrm{f}_{t_i}(s_i))=\Pi_D(\mathrm{f}_{t_j}(s_j))=:x$, $\Pi_D(\mathrm{f}_{t_i}(s'_i))=\Pi_D(\mathrm{f}_{t_j}(s'_j))=:x'$, and
$\Pi_D(\mathrm{f}_{t_i}((s_i,s'_i)))\cap \Pi_D(\mathrm{f}_{t_j}((s_j,s'_j)))=\varnothing$. We think about the points $x,x'$ as boundary points of a region delimited by the two non-touching curves  $\Gamma_i:=\Pi_D(\mathrm{f}_{t_i}((s_i,s'_i)))$ and $\Gamma_j=\Pi_D(\mathrm{f}_{t_j}((s_j,s'_j)))$. Then, we claim that 
we can find infinitely many faces (represented in yellow) that touch both $\Gamma_i$ and $\Gamma_j$. These latter faces, together with $\Gamma_i,\Gamma_j$, delimit a bi-infinite sequence of regions which may be of two possible types: {the regions depicted in red have a pinch point, i.e.~the two neighboring faces touch each other, whereas they do no touch in the regions depicted in blue.
} Hence, the bi-infinite sequence of these regions induces a bi-infinite sequence in $\{ \mathrm{Red}, \mathrm{Blue}\}^{ \mathbb{Z}}$ which should be  a topological invariant. Arguing that the local geometry around $x,x'$ is described asymptotically by a scale-invariant model, the law of this bi-infinite sequence should be ``close'' to an i.i.d.\ sample of Bernoulli random variables. Moreover, as one is allowed to change the choice of $t_i,t_j,x,x'$---note, in particular, that there are only countably many possible such choices---these sequences should be independent of one another.  Since the probability that two independent Bernoulli random sequences coincide (up to translation) is zero, one is led by this heuristic argument to conclude that two independent samples of $ \mathbb{X}$ cannot be homeomorphic as closed subset of $ \mathbb{S}^{2}$, almost surely. Moreover, since the regions as the ones considered here arise at all scales, it should hold that the same holds at the level of neighborhoods. We believe that the above claims about the existence of infinitely many traversing, yellow faces should result from a zero-one law in the vicinity of the points $x,x'$. Making this argument rigorous would require to study more cautiously such vicinities, and this will be considered elsewhere. 

  \begin{figure}[!h]
      \begin{center}
  \includegraphics[width=13cm]{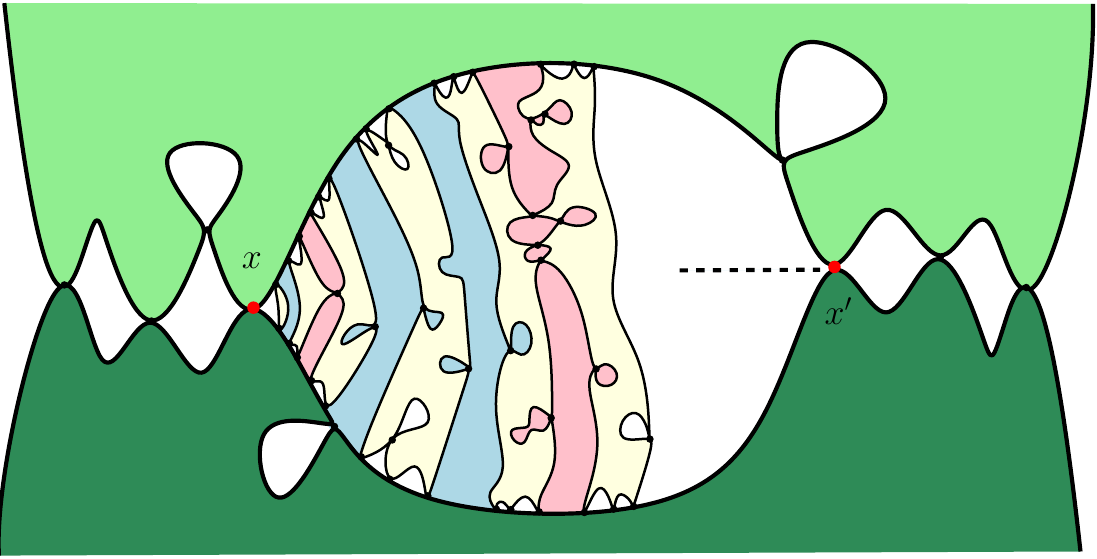}
  \caption{An illustration of the heuristic argument for the fact that the topology of $\mathcal{S}$ is sample-dependent in the dense case  $\alpha\in (1,\frac{3}{2})$. Two touching faces (light and dark green) enclose a region separated by the two red extreme points.
    In the vicinity of each of these red points, we claim that we can find infinitely many (yellow) faces touching both green faces. These yellow faces are separated by a bi-infinite sequence of regions, each of which is called ``blue'' if it {has no pinch point, 
  and ``red'' otherwise.}
  The resulting sequence in $\{ \mathrm{Red}, \mathrm{Blue}\}^{ \mathbb{Z}}$, considered up to shifts, is a topological invariant, which should be a.s.\ different for every realization.
    \label{fig:topologyrandom}}
  \end{center}
  
\end{figure}

\subsection{Fractal dimension of faces}
In this section, we compute the  Hausdorff dimension of the faces of $( \mathcal{S},D)$ and we show that this dimension is always $2$, in particular it does not depend on $\alpha$ nor on the subsequence $(n_k)_{k\geq 1}$. This result should be related with the fact that the boundary in models of Brownian geometry have also Hausdorff dimension equal to $2$. The results of this section will not be used in what follows and they can be skipped in a first reading. 
\begin{prop}[Dimension of the faces] \label{prop:HDface=2}
$\mathbf{P}$-a.s., for every $i\in \mathbb{N}$, the Hausdorff dimension of $\Pi_D\big(\mathrm{f}_{\mathrm{t}_i}([0,1])\big)$, in the space $( \mathcal{S},D)$,  is 2.
\end{prop}

\begin{proof} The proof follows the  same lines as that of \cite[Theorem 3]{Bet11}. As usual, to simplify some technicalities, we work under  $\mathbf{N}$ since the statement under $\mathbf{P}$ then follows by scaling. Fix an arbitrary $h>0$ and consider a stopping time $T$, taking values in  $\{t\geq 0:\:\Delta_{t}>h\}$. To simplify notation set: 
\[ \mathrm{f}(t):= \mathrm{f}_{T}(\frac{t}{\Delta_{T}})~,\quad t\in[0,\Delta_{T}].\] It is enough to prove that, $\mathbf{N}(\cdot\:|\:T<\infty, \Delta_{T})$-a.s.~, the Hausdorff dimension of $\Pi_D\big(\,\mathrm{f}([0,\Delta_{T}])\big)$ is 2.
\textbf{Lower bound}. For every $s\in [0,\Delta_{T}]$ and $\delta>0$, set 
\[\tau_{\delta}^{s}:=\inf\big\{t\in [s;\Delta_{T}]~:~Z_{\mathrm{f}(t)}\leq Z_{\mathrm{f}(s)}-\delta\big\}\quad\text{and}\quad\tilde{\tau}_{\delta}^{s}:=\sup\big\{t\in[0;s]~:~Z_{\mathrm{f}(t)}\leq Z_{\mathrm{f}(s)}-\delta\big\}\]
with the convention $\inf \varnothing =\infty$.
For every $x\in[0,1]/\sim_D$, we write $B_{D}(x,\delta)$ for the ball of radius $\delta$ centered at $x$ with respect to the distance $D$.   Set $\kappa$ the pushforward of the Lebesgue measure on $[0,\Delta_{T}]$ by the continuous function $\Pi_D\circ\mathrm{f}$. Provided that $\tau_{\delta}^{s}$ and $\tilde{\tau}_{\delta}^{s}$ are finite,  the cactus bound \eqref{cactus}  yields:
\[\kappa\Big(B_D\big(\Pi_D(\mathrm{f}(s)),\delta\big)\Big)\leq \tau_{\delta}^{s}-\tilde{\tau}_{\delta}^{s}.\]
We recall now that by Proposition \ref{b_Brownian_Brigde}, under $\mathbf{N}(\cdot\:|\:T<\infty, \Delta_{T})$, the process $(Z_{\mathrm{f}(t)}-Z_{\mathrm{f}(0)})_{t\in[0,\Delta_{T}]}$ is a Brownian bridge starting and ending at $0$.
By the absolute continuity of the Brownian bridge with respect to Brownian motion we deduce that, $\mathbf{N}(\cdot\:|\:T<\infty, \Delta_{T})$-a.s.,  for every $\eta>0$ and Lebesgue almost all $s\in [0,\Delta_{T}]$, we have:
\begin{align*}
\limsup \limits_{\delta \to 0}\delta^{-2+\eta}\kappa\Big(B_D\big(\Pi_D(\mathrm{f}(s)),\delta\big)\Big)&\leq\limsup \limits_{\delta \to 0}\delta^{-2+\eta}\big (\tau_{\delta}^{s}-\tilde{\tau}_{\delta}^{s}\big)=0, \quad a.s.
\end{align*}
Standard density theorems for Hausdorff measures now give that the Hausdorff dimension of $\Pi_D(\mathrm{f}([0,\Delta_{T}]))$ is bounded below by $2 - \eta$,  almost surely for $\mathbf{N}(\cdot\:|\:T<\infty, \Delta_{T})$.\\
\noindent \textbf{Upper bound}. To establish the upper bound, we construct a covering of $ \mathrm{f}([0, \Delta_{T}])$ by removing sub-looptrees corresponding to large (negative) $Z$-excursion, see Figure  \ref{fig:chopoff}. Specifically, consider $(t_{j})_{j\in J}$ the set of points $t\in[0,\Delta_{T}]$ such that $\mathrm{f}(t-)<\mathrm{f}(t)$ and $(X^{j},Z^{j})_{j\in J}$ the associated excursions defined by:
\[X^{j}_{t}=X_{\mathrm{f}(t_{j}-)+t}-X_{\mathrm{f}(t_{j}-)}, \quad Z^{j}_{t}=Z_{\mathrm{f}(t_{j}-)+t}-Z_{\mathrm{f}(t_{j}-)} \quad \mbox{ for }\quad t\in[0,\mathrm{f}(t_{j})-\mathrm{f}(t_{j}-)],\]
for every $j\in J$. By  Corollaries \ref{cutting_N} and \ref{Z<-1} imply that, conditionally on $ \mathcal{F}_{T}$, the number $ J_{ \varepsilon}$ of such excursions whose overall $Z^{j}$-infimum is below $- \varepsilon$ is a Poisson random variable with intensity $\Delta_{T} \cdot   \mathbf{N}( \inf Z \leq -1) \cdot \varepsilon^{-2}$. Using standard large deviation estimates and that $J_\varepsilon$ is decreasing in $\varepsilon$, it follows that  $J_{  \varepsilon} \leq  \varepsilon^{-2 - \eta}$ when $ \varepsilon \to 0$ for any $\eta >0$. Let $  \mathrm{f}(s_{1}^{( \varepsilon)}),... , \mathrm{f}(s_{J_{ \varepsilon}}^{( \varepsilon)}) \in  \mathrm{f}([0, \Delta_{T}])$ be the loop times associated to these large excursions, and let us consider  $ \mathrm{f}(0) = \mathrm{f}(r_{1}^{( \varepsilon)}) \leq \cdots \leq \mathrm{f}(r_{K_{ \varepsilon}}^{( \varepsilon)}) = \mathrm{f}(1)$ so that $$ \Big(\max \Big\{Z_{\mathrm{f}(u)} :~ u \in [\mathrm{f}(r_{i}^{( \varepsilon)}), \mathrm{f}(r_{i+1}^{( \varepsilon)})]\Big\} - \min \Big\{Z_{\mathrm{f}(u)} :~ u \in [\mathrm{f}(r_{i}^{( \varepsilon)}), \mathrm{f}(r_{i+1}^{( \varepsilon)})]\Big\}\Big) \leq \varepsilon, $$
for all  $1 \leq i \leq K_{ \varepsilon}$.  Since $Z_{\mathrm{f}(\cdot )}$ is a Brownian bridge of length $ \Delta_{T}$ it is in particular $  \frac{1}{2} ^{-}$-H\"older continuous. Thus, the points can be chosen such that $ K_{ \varepsilon} \leq \varepsilon^{-2 - \eta}$ as $ \varepsilon \to 0$ a.s. for any $\eta>0$. Moreover, an application of Schaeffer bound \eqref{trivial_bounds} entails that
$$ \Pi_{D}\left( \big\{\mathrm{f}(s_{i}^{( \varepsilon)}) : 1 \leq i \leq J_{ \varepsilon}\} \cup \{\mathrm{f}(r_{i}^{( \varepsilon)}) : 1 \leq i \leq K_{ \varepsilon}\big\}\right),$$ is a $2 \varepsilon$ cover of $\Pi_{D}( \mathrm{f}([0, \Delta_{T}]))$. Since $J_{ \varepsilon} + K_{ \varepsilon} \leq 2\varepsilon^{-2 - \eta}$ as $ \varepsilon\to 0$, this shows that its Hausdorff dimension is less than $2 + \eta$ for any $\eta >0$. 
\begin{figure}[!h]
 \begin{center}
 \includegraphics[width=8.5cm]{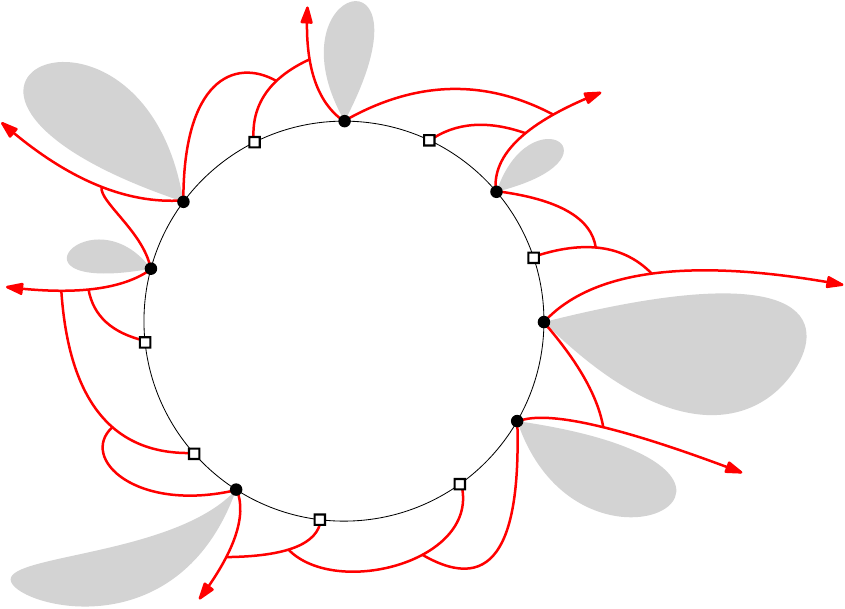}
 \caption{Illustration of the covering of a face. The blacks dots correspond to the projection of the points $ \mathrm{f}(s_{i}^{{ ( \varepsilon)}})$ and the white boxes to the projections of the points $ \mathrm{f}(r_{i}^{{ ( \varepsilon)}})$. We can start simple geodesics (in red) from the two pre-images of $\Pi_{d}( \mathrm{f}(s_{i}^{{ ( \varepsilon)}}))$ and from $\mathrm{f}(r_{i}^{{ ( \varepsilon)}})$ which enclose the face and produce a $2 \varepsilon$ covering. \label{fig:chopoff}}
 \end{center}
 \end{figure}
 \end{proof}

 \begin{rek}[Other dimensions] One might wonder about the Hausdorff dimensions of various geometric objects, such as the intersection of two touching faces, the set of contact points between a simple geodesic and a face, or the image of the skeleton of the looptree. While some of these may be accessible using our techniques, we do not address them in this work.
 \end{rek} 
 
  \subsection{The a priori local bound $D^{*}\leq D^{1-\delta}$}\label{Sec:A:priori:bound}
 This section contains an important consequence of Theorems  \ref{technical_uniform_balls} and \ref{main_theorem_topology}. It provides an \textit{a priori} local lower bound on $D$ in terms of $D^{*}$. This will be a key input in the following section when performing the surgery along geodesics:
\begin{prop}[A priori control on distances]\label{preliminary_control} For any $\delta>0$,  $\mathbf{P}$-a.s. there exists a (random) positive number $A_{\delta}$ such that
 \begin{equation}\label{ine:exp}
\:D^{*}(x,y)\leq  A_{\delta}\cdot D(x,y)^{1-\delta},\quad x,y\in \mathcal{S}.
\end{equation}
\end{prop}The proof of this proposition follows the same strategy as that presented in \cite[Proposition 6.1]{LG11} or \cite[Proposition 6]{Mie11}. It relies on the identification of the topology, together with uniform controls concerning the volume of  balls for $D$ and $D^{*}$ in $\mathcal{S}$ which are here provided by our Theorem \ref{technical_uniform_balls}. We start with the latter, which establishes the monofractality of the measure $ \mathrm{Vol}$: 
\begin{lem}\label{uniform_control_vol} For every $\delta >0$,  almost surely for $\mathbf{P}$ there exist two positive (random) numbers $0<C_{\delta}<\widetilde{C}_{\delta}$ such that:
\[C_{\delta}\cdot (\varepsilon^{2\alpha+\delta}\wedge 1)\leq \mathrm{Vol}\big(B_{D^{*}}(x,\varepsilon)\big)\leq\mathrm{Vol}\big(B_{D}(x,\varepsilon)\big)\leq \widetilde{C}_{\delta}\cdot \varepsilon^{2\alpha-\delta},\]
for every $\eps>0$ and $x\in \mathcal{S}$.
\end{lem}
\begin{proof}  The lower bound $C_{\delta}(\varepsilon^{2\alpha+\delta}\wedge 1)\leq \mathrm{Vol}\big(B_{D^{*}}(x,\varepsilon)\big)$ is a direct consequence of the fact that $t\mapsto Z_{t}$ is $\beta$-H\"older for every $\beta<(2\alpha)^{-1}$. Namely, recall from \eqref{Z:variation:L} that there exists a positive random variable $W$ such that
\[\:|Z_{t}-Z_{s}|\leq W\cdot |t-s|^{\frac{1}{2\alpha+\delta}},\quad s,t\in[0,1].\]
The lower bound follows since, for $ \varepsilon>0$ and $r\in [0,1]$, the bound \eqref{trivial_bounds}  ensures
\begin{align*}
\mathrm{Vol}\Big(B_{D^{*}}(\Pi_D(r),\varepsilon)\Big) &\geq 1\wedge \inf\big\{s\in [r,1] :\:|Z_{s}-Z_{r}|\geq \frac{\varepsilon}{2}\big\}-0\vee \sup\big\{s\in[0, r]:\:|Z_{s}-Z_{r}|\geq \frac{\varepsilon}{2}\big\}\\
&\geq 1\wedge \big((2W)^{-2\alpha-\delta} \varepsilon^{2\alpha+\delta}\big).
\end{align*}
This completes the proof of the lower bound.  Since $D\leq D^*$ it remains to establish that there exists a constant $\widetilde{C}_\delta$ such that $\mathrm{Vol}\big(B_{D}(x,\varepsilon)\big)\leq \widetilde{C}_{\delta}\cdot \varepsilon^{2\alpha-\delta}$, for every $x\in \mathcal{S}$ and $\varepsilon>0$. In this direction, notice that by \eqref{Distance:rho_*} and Proposition \ref{technical_uniform_balls} we have:
\[\mathbf{P}\big(\mathrm{Vol}(B_{D}(\rho_{*},2\varepsilon))\geq \varepsilon^{2\alpha-\delta}\big) \underset{\eqref{Distance:rho_*}}{=}\mathbf{P}\big(\int_{0}^{1}\d s ~\mathbbm{1}_{Z_{s}\leq Z_{t_*}+2\varepsilon} ~\geq \varepsilon^{2\alpha-\delta}\big) \underset{ \mathrm{Prop.}\  \ref{technical_uniform_balls}}{\leq} C \exp( - \varepsilon^{-c}), \] for some $c,C \in (0, \infty)$. The upper bound  of the proposition then follows by the re-rooting property of $( \mathcal{S},D)$. Indeed, on the event $\{ \exists x \in \mathcal{S} : \mathrm{Vol}\big(B_{D}(x,\varepsilon)\big)\geq \varepsilon^{2\alpha-\delta}\}$ the projection of an independent uniform point $U \in [0,1]$ (independent of $(X,Z)$) may fall in such a large ball with probability at least $\varepsilon^{2\alpha-\delta}$. In that case we obviously have $ \mathrm{Vol}\big(B_{D}( \Pi(U), 2\varepsilon)\big)\geq \varepsilon^{2\alpha-\delta}$. Thanks to the re-rooting property \eqref{eq:re-rooting-S} we deduce:
  \begin{eqnarray*}C\cdot \exp( - \varepsilon^{c}) &\geq& \mathbf{P}(\mathrm{Vol}\big(B_{D}( \rho_{*}, 2\varepsilon)\big)\geq \varepsilon^{2\alpha-\delta})\\ & \underset{ \mathrm{re-rooting}}{=} &\mathbf{P}(\mathrm{Vol}\big(B_{D}( \Pi(U), 2\varepsilon)\big)\geq \varepsilon^{2\alpha-\delta})\\
  &\geq&  \varepsilon^{2\alpha-\delta} \cdot \mathbf{P}(\exists x \in \mathcal{S} : \mathrm{Vol}\big(B_{D}(x,\varepsilon)\big)\geq \varepsilon^{2\alpha-\delta}).  \end{eqnarray*} it follows that $\mathbf{P}(\exists x \in \mathcal{S} : \mathrm{Vol}\big(B_{D}(x,\varepsilon)\big)\geq \varepsilon^{2\alpha-\delta})$ has a  stretched-exponential decay to $0$. Using Borel--Cantelli along the sequence $ \varepsilon = 2^{-n}$ we get that eventually as $n \to \infty$, there is no point $x$ in $ \mathcal{S}$ such that $\mathrm{Vol}\big(B_{D}(x, 2^{-n})\big)\geq 2^{-n(2\alpha-\delta)}$. This suffices to imply the lemma by interpolation.\end{proof}
The proof of Proposition \ref{preliminary_control} is now mutatis mutandis the same as in \cite[Proposition 6]{Mie11} or \cite[Proposition 6.1]{LG11}, but we present and illustrate the idea for completeness:

\begin{proof}[Proof of Proposition \ref{preliminary_control}]
Fix $\delta>0$ and consider $C_{\delta/2},\widetilde{C}_{\delta/2}$ as in Proposition \ref{uniform_control_vol}.  Let $x,y \in \mathcal{S}$ such that $D^*(x,y)\in(0,1/2)$. Since $\mathcal{S}$ is compact, Theorem  \ref{main_theorem_topology} entails that:
$$\inf \big\{D(x,y):~x,y \in \mathcal{S} \text{ with } D^*(x,y)\geq 1 \big\}>0. $$
Hence, it suffices to  establish \eqref{ine:exp} restricted to $x,y\in \mathcal{S}$ verifying $D^{*}(x,y)\in (0,1)$. So we fix $x,y\in \mathcal{S}$ such that $D^{*}(x,y)\in (0,1)$ and, to simplify notation, we set  $\varepsilon:=D(x,y)$ and $\varepsilon^*:=D^{*}(x,y)$. Remark that  we must have $\varepsilon\leq \varepsilon^*$, and our goal is to show that $ \varepsilon$ cannot be much smaller than $ \varepsilon^*$. By Theorem~\ref{main_theorem_topology}, we already know that we must have $\varepsilon>0$.  Fix a $D$-geodesic $\gamma : [0, \varepsilon] \to \mathcal{S}$ going from $x$ to $y$. As in \cite[Proposition 6.1]{LG11}, put $t_{0}:=0$ and as long as $t_{n} < \varepsilon$ define by induction 
$$ t_{n+1} := \sup\big\{  t \in [t_n, \varepsilon] : \gamma(t) \in B_{D^{*}}( \gamma(t_{n}), \varepsilon) \big\},$$ where we recall that $B_{D^{*}}(x, r)$ stands for  the closed ball of radius $r$ around $x$ for the metric $D^{*}$. Using the fact that the topologies defined by $D$ and $D^{*}$ coincides (again by Theorem \ref{main_theorem_topology}), a compactness argument shows that the construction stops after a finite number of steps $N$ and yields points $x=x_{0} = \gamma(t_{0}), ... , x_{N} = \gamma(t_{N}) = y$ such that for $0 \leq i \leq N-2$
$$ D^{*}( \gamma(t_{i}), \gamma(t_{i+1})) = \varepsilon, \quad \mbox{ and  the balls }  B_{D^{*}}(x_{i}, \varepsilon/ 3) \mbox{ are disjoint}.$$   In particular, since $D^{*}(x_{i},x_{i+1}) \leq  \varepsilon$, $0 \leq i \leq N-1$, we must have   $ N \varepsilon \geq \varepsilon^*$ and since $D \leq D^{*}$, the $D$-ball of radius $ 2 \varepsilon$ centered at $x$  at least contains the $N-1$ disjoint balls $B_{D^{*}}(x_{i}, \varepsilon/ 3)$, $0 \leq i \leq N-2$. 

\begin{figure}[!h]
 \begin{center}
 \includegraphics[width=10cm]{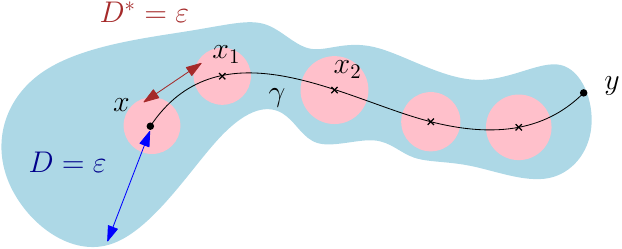}
 \caption{Illustration of the proof: If $\varepsilon$ is much smaller than $\varepsilon^*$, then we can find roughly $ N= \varepsilon^*/ \varepsilon'>> 1 $ points along a $D$-geodesic between $x$ and $y$ whose $D^{*}$-balls of radius $ \varepsilon/ 3$ are disjoint. The contradiction comes from the fact that $B_{D}( x, 2 \varepsilon)$ (in blue on the figure) contains at least $N$ balls for $D^{*}$ of radius $ \varepsilon/3$, and as a consequence its volume is too large.}
 \end{center}
 \end{figure}
 Taking volumes, an application of Lemma \ref{uniform_control_vol} gives the inequality:
 $$ (N-1) \times C_{\delta/2} \cdot (\varepsilon/3)^{ 2 \alpha + \delta/2} \leq \widetilde{C}_{\delta/2} \cdot (2 \varepsilon)^{ 2\alpha - \delta/2}.$$
 Recalling that $N \geq \frac{ \varepsilon^*}{ \varepsilon}$, we deduce the desired estimate on $ \varepsilon, \varepsilon^*$. \end{proof}

\section{A conditional proof of $D=D^*$}
\label{sec:D=D*}

{This short section presents the proof of our main theorem $D=D^{*}$ based on results we established in the previous section, and assuming certain forthcoming properties of geodesics in $ ( \mathcal{S},D)$. Proving these properties constitutes the most technical part of this work and the reader will embark a long journey using both discrete and continuous arguments.} This section can thus be seen as a resting area  offering the necessary motivation for the subsequent sections. The surgery techniques along geodesics presented here are directly inspired by the ones used in \cite{LG11} and \cite{Mie11} in the case of the Brownian sphere.  \medskip

Under $\mathbf{P}$, we  consider $U_{1},U_{2}$ two uniform random variables in $[0,1]$, independently of  $(X,Z,D)$ and we set $\rho_1:=\Pi_D(U_1)$ and $\rho_2:=\Pi_D(U_2)$ for simplicity. Since the random functions $D:[0,1]^2\to \mathbb{R}_+$ and $D^*:[0,1]^2\to \mathbb{R}_+$ are a.s.~continuous, Theorem \ref{thm:main} boils down to establishing that:
   \begin{equation} \label{eq:theoequiv}D(\rho_1,\rho_2)=D^{*}(\rho_1,\rho_2), \quad \mbox{ $\mathbf{P}$-a.s.}  \end{equation}
The starting point of the surgery along geodesics is the following result, analogous to \cite{LG09,Mie09} in the case of the Brownian sphere, which we will prove in Section \ref{secP:uni:geo}:
\begin{theo}[Essential uniqueness of geodesics] \label{alm-unique}
$\mathbf{P}$-a.s., there is a unique $D$-geodesic $$\gamma_{1,2} : [0, D(\rho_{1}, \rho_{2})] \to \mathcal{S}$$ going from $\rho_{1}$ to $\rho_{2}$.
\end{theo}
Recall from Section \ref{sec:D<D*} that $D^*$ is constructed as the largest pseudo-distance which passes to the quotient of $\sim_d$ and for which simple geodesics are geodesics. As a consequence of Theorem \ref{alm-unique} we shall prove that: \begin{prop}\label{thm:geodesics:rho:*}  \label{thm:geodesics}
$\mathbf{P}$ - a.s.,  all the geodesics towards $\rho_*$  in $(\mathcal{S},D)$ are simple geodesics.
\end{prop}
In particular, the geodesics to $\rho_{*}$ coincide for both metrics $D^*$ and $D$, and being a simple geodesic is actually a metric notion in $( \mathcal{S},D, \rho_*)$. Assuming the previous two results in this section, we shall thus drop the adjective \textit{simple} and only speak of geodesics towards $\rho_*$. The equality \eqref{eq:theoequiv} means that we can approximate the path $\gamma_{1,2}$ as closely as desired using a concatenation of pieces of geodesics towards $\rho_*$.
In this direction, for $u  \in ( \varepsilon, D( \rho_{1}, \rho_{2})- \varepsilon)$, we  say that the point $x = \gamma_{1,2}(u) \in \mathcal{S}$  is \textbf{$ \varepsilon$-good}, for $\gamma_{1,2}$ and inside $ ( \mathcal{S},D,\rho_*)$,  if $\gamma_{1,2}([u- \varepsilon, u+ \varepsilon])$ coincides {with the concatenation of  one or two  pieces of geodesics} towards $\rho_*$. It is said \textbf{$ \varepsilon$-bad} otherwise. See Figure~\ref{fig:epsilongood} for an illustration. The presence of a bad point in $\gamma_{1,2}$ is related to the concept of $3$-stars along $\gamma_{1,2}$ which are points from which we can start three locally distinct geodesics, see \cite{Mie11} and Lemma \ref{lem:stars}.

{Returning  to our random setting, $ \varepsilon$-good points along $\gamma_{1,2}$ can be used to replace parts of $\gamma_{1,2}$ by pieces of  geodesics towards $\rho_*$.} 
   \begin{figure}[!h]
    \begin{center}
    \includegraphics[width=13cm]{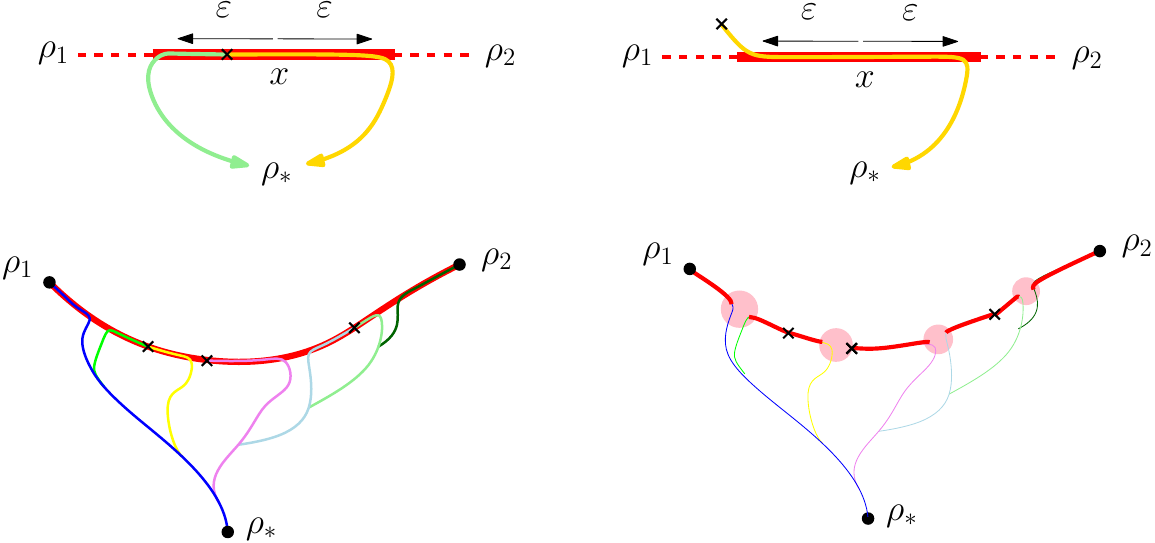}
    \caption{Top: illustration of $ \varepsilon$-good points $x$. The geodesic $\gamma_{1,2}$ is in red, while the geodesics towards $\rho_*$ are in green or yellow. Bottom: approximation of $\gamma_{1,2}$ by pieces of geodesics towards $\rho_*$ using $ \varepsilon$-good points. An a priori bound $D^{*} \leq D^{1- \delta}$ (Proposition \ref{preliminary_control}) is used in the pink regions to bridge between the parts coinciding with  geodesics towards $\rho_*$.\label{fig:epsilongood}}
    \end{center}
    \end{figure}
{However, there is an important caveat: we need to show that we can perform this surgery along a large part of $\gamma_{1,2}$. Namely, we must prove that most points along $\gamma_{1,2}$ are $ \varepsilon$-good, and then control distances in the remaining parts of the geodesic. For the latter we rely on Proposition \ref{preliminary_control} and for the former on the following estimate:}
\begin{prop}[Most points along geodesics are good]\label{main:techni} 
There exist  constants $C,c>0$ such that:
$$  \mathbf{E} \left[ \int_{0}^{D( \rho_{1}, \rho_{2})} \mathrm{d}u \  \mathbbm{1}_{\{\gamma_{1,2}(u)  \mbox{ is $\varepsilon$-bad}\,\}} \right] \leq C \cdot \eps^{c},\quad \text{for every } \eps>0.$$
\end{prop}
The proof of Proposition \ref{main:techni} is actually the main technical estimate of this paper, it is proved at the end of Section \ref{secP:uni:geo} and relies on results of Section \ref{sec:prison}. We also stress that part of the result is to establish that the random variable inside the expectation is well-defined, i.e.\ measurable with respect to the rooted Gromov--Hausdorff--Prokhorov topology. Let us now deduce Theorem \ref{thm:main} from the three results Theorem \ref{alm-unique}, Proposition \ref{thm:geodesics:rho:*} and Proposition \ref{main:techni} admitted above. See Figure~\ref{fig:epsilongood} for an illustration.
\begin{proof}[Proof of \eqref{eq:theoequiv}, hence of Theorem \ref{thm:main}.]
Using Theorem \ref{alm-unique}, let $\gamma_{1,2}$ be the a.s.~unique geodesic between $\rho_{1}$ and $\rho_{2}$. For $ \varepsilon>0$  decompose $\gamma_{1,2}$ into $  K_{ \varepsilon}= (\lfloor D(\rho_{1}, \rho_{2})/ \varepsilon \rfloor +1)$ chunks $$I_{j} = \gamma_{1,2}\Big(\Big[(j-1)\frac{D(\rho_{1}, \rho_{2})}{K_{ \varepsilon}}, j \frac{D(\rho_{1}, \rho_{2})}{K_{ \varepsilon}} \Big]\Big), \quad \mbox{ for }1 \leq j \leq K_{ \varepsilon}\, ,$$ of $D$-length less than $  \varepsilon$. We say that the chunk $I_{j}$ is  $ \varepsilon$-nice if it contains an $ \varepsilon$-good point for $\gamma_{1,2}$, in which case the entire chunk $I_{j}$ coincides with a  concatenation of at most two pieces of geodesics towards $\rho_*$. We say that the chunk  is  $ \varepsilon$-ugly otherwise. {In particular,  an $ \varepsilon$-ugly chunk must only contain $ \varepsilon$-bad points for $\gamma_{1,2}$ and thus, simultaneously  for every $\varepsilon>0$, we have:
$$  \varepsilon \cdot \# \big\{ 1 \leq i \leq K_{ \varepsilon}  : I_{j} \mbox{ is $ \varepsilon$-ugly}\big\} \leq  \int_{0}^{D( \rho_{1}, \rho_{2})} \mathrm{d}u \  \mathbbm{1}_{\{\gamma_{1,2}(u)  \mbox{ is $\varepsilon$-bad}\}}.$$ 
We now use Proposition \ref{main:techni}, which gives us that $\mathbf{E}[\int_{0}^{D( \rho_{1}, \rho_{2})} \mathrm{d}u \  \mathbbm{1}_{\{\gamma_{1,2}(u)  \mbox{ is  $\varepsilon$-bad}\}}]\leq C\cdot \varepsilon^{c}$, for some $C,c>0$ independent of $ \varepsilon>0$. }
Taking $ \varepsilon = 2^{-n}$ for $n=1,2, ...$ and using Markov's inequality together with the Borel--Cantelli lemma, we  deduce that a.s.~the number of $2^{-n}$-ugly chunks is eventually less than $2^{n(1- c/2)}$ as $n \to \infty$. 
Next, observe that by Proposition \ref{thm:geodesics}, geodesics towards $\rho_*$ are geodesics for both $D$ and $D^{*}$ and therefore  the metric $D$ and $D^{*}$ coincides on the $2^{-n}$-nice intervals of $\gamma_{1,2}$. We deduce that for any $ \delta>0$ we have
  \begin{eqnarray*} |D^{*}(\rho_{1}, \rho_{2}) - D(\rho_{1}, \rho_{2})| &\underset{D \leq D^*}{\leq}& \# \{ 1 \leq j \leq K_{ 2^{-n}}  : I_{j} \mbox{ is $2^{-n}$-ugly}\} \times \sup_{\begin{subarray}{c}x,y \in \mathcal{S}\\ D(x,y) \leq 2^{-n} \end{subarray}} D^{*}(x,y)\\ & \underset{ \mathrm{Prop.\ } \ref{preliminary_control}}{\leq}&  2^{n (1- c/2)} \times A_{\delta} \cdot 2^{-n(1- \delta)},  \end{eqnarray*} eventually as $ n \to \infty$ ; where we recall that $A_\delta$ is the positive random variable appearing in Proposition~\ref{preliminary_control}  and in particular does not depend on $n$.
Taking $\delta < c/2$ and letting $n \to \infty$ completes the  proof. \end{proof}

We now have to prove Theorem \ref{alm-unique}, Proposition \ref{thm:geodesics}, and Proposition \ref{main:techni} admitted above.  For this, we need to understand the local geometric picture around a typical point of a typical geodesic in $ \mathcal{S}$. Although this information can in principle be extracted from the encoding presented in Section~\ref{sec:BDG}, it will be more practical to introduce another discrete encoding of maps.  {This encoding is a variant of the BDG construction, originally introduced in \cite{Mie09} for quadrangulations, and offers a clearer and more transparent view of the behavior of a typical point along a typical geodesic in the discrete.}

 \section{Boltzmann stable maps with two sources}\label{sec:boltzm-stable-maps}

In this section, we introduce a tool that will be useful to study
the properties of geodesics in discrete stable maps and their scaling
limits.  It
consists of a generalization to multi-marked maps of the BDG
bijection presented in Section \ref{sec:BDG}, which is directly
inspired from the particular case of quadrangulations that is
considered in \cite{Mie09}.

\subsection{The BDG bijection with two sources and delays}\label{sec:biject-with-delays}

Consider the set $\mathcal{M}^{2\bullet}$ of
triples $(\bm,(v_1,v_2),\delay)$ where:
\begin{itemize}
\item 
$\bm$ is a planar bipartite map with one distinguished oriented
edge $e_0$;
\item
  $v_1,v_2$ are two vertices of $\bm$;
  \item $\delay$ is an integer such that
    $|\delay|<\mathrm{d}^{\mathrm{gr}}_\bm(v_1,v_2)$, and
    $\mathrm{d}^{ \mathrm{gr}}_{\bm}(v_1,v_2)+\delay$ is an even number. 
  \end{itemize}
  Note that the last condition implies that
  $\mathrm{d}^{\mathrm{gr}}_\bm(v_1,v_2)\geq 2$, and, in particular, that
  $v_{1},v_{2}$ are distinct.  
 The integer $ \delay$ is referred to as the \textbf{delay} and is said to be \textbf{admissible} when it satisfies the above conditions. For a given bi-pointed planar map $ ( \mathbf{m}, (v_{1},v_{2}))$, the set of all admissible delays will be denoted by  \begin{eqnarray} \label{def:Edelay} \Edelay_{\mathbf{m},v_1,v_2}. \end{eqnarray} 
  In order to obtain more harmonious notation, it will sometimes be useful to fix arbitrarily two integers
  $\delay_1,\delay_2$ such that $\delay=\delay_1-\delay_2$.

  We also consider the set $\mathcal{U}$ of {\bf unicyclomobiles},
  that are rooted plane maps $\bu$ having the property that:
  \begin{itemize}
  \item
    $\bu$ is a bipartite map with exactly two faces labeled as $f_1$ and $f_2$;
    \item the vertices of $\bu$ are partitioned into two ``black
      and white'' sets of vertices
      $V_\bullet(\bu),V_\circ(\bu)$ in such a way that a vertex of one
      set is only incident to vertices of the other set;
      \item $\bu$ is rooted at a corner incident to a white vertex of
        $V_\circ(\bu)$;
        \item the white vertices carry a function
          $\ell:V_\circ(\bu)\to \Z$, in such a way that $(\bu,\ell)$
          is well-labeled according to the definition given in the opening of Section
          \ref{sec:bdg-construction}. In particular, the label of the root vertex is
equal to $0$. 
  \end{itemize}

The reason for the fancy name  is that
unicyclomobiles are planar maps with exactly one \textbf{cycle}, of even
length. This comes from the usual core decomposition of planar maps
with two faces: removing inductively all edges incident to a vertex of
degree 1, one obtains a planar map with two faces and with only vertices of
degree at least $2$, which does not depend on the order of the edge
deletions.
By Euler's formula, this map must have an equal number
of vertices and edges, and the only possibility is that the resulting
map is a cycle, which must have even length by the bipartite nature
of the map we started from. Conversely, it is a consequence of the
Jordan curve theorem that a planar map with exactly one cycle has
exactly two
faces

\begin{figure}[!h]
 \begin{center}
 \includegraphics[width=10cm]{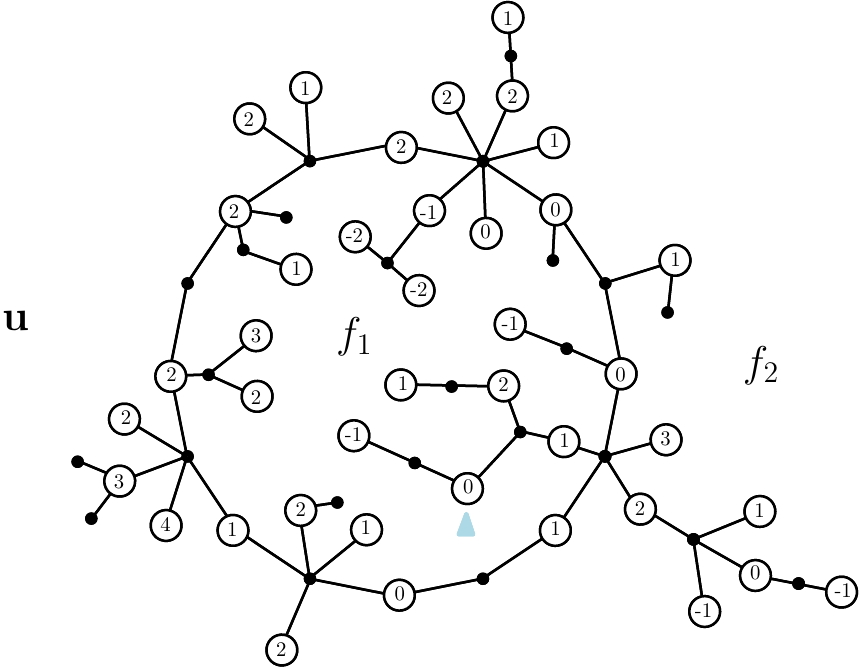}
 \caption{A unicyclomobile. \label{fig:unicyclo1}}
 \end{center}
 \end{figure}
 
The sets $\mathcal{M}^{2\bullet}$ and $\mathcal{U}\times \{-1,1\}$ are in natural
correspondence, as we now discuss.

\subsubsection{From a unicyclomobile to a map}

It is easier to describe the mapping which, to an element of
$(\bu,\epsilon)\in \mathcal{U}\times \{-1,1\}$, associates an element of
$\mathcal{M}^{2\bullet}$, as it consists in performing the BDG
construction, described in Section \ref{sec:bdg-construction},
within each face of $\bu$. We refer the reader to Figure
\ref{fig:unicyclo2} for an illustration of this mapping, which we now
define. 

For $i\in \{1,2\}$, we first add a new vertex $v_i$
inside the face $f_i$, with a label $\ell(v_i)=\min\{\ell(v):v \mbox{
  incident to }f_i\}-1$. For each white corner $c$ incident to the
face $f_i$, we define its successor to be the next white corner $c'=s(c)$ in the 
(counterclockwise) contour order around the face $f_i$ with label $\ell(c)-1$, or
$s(c)=v_i$ if $c$ has minimal label in the face $f_i$. Then, from
every white corner, we draw an arc connecting this corner to its successor,
in such a way that these arcs do not cross each other, nor an edge of
$\bu$. In particular, the arcs always connect two corners incident to
the same face $f_i$. Lastly, we remove the edges originally present in
$\bu$.

The resulting map, whose vertex set is $V_\circ(\bu)\cup \{v_1,v_2\}$ and
whose edge-set is the set of arcs from the white corners to their successors, is denoted by $\bm$. 
It is naturally $2$-marked at the vertices
$(v_1,v_2)$. We also let
 \begin{eqnarray} \label{eq:Di} \delay_i&:=&\ell(v_i), \quad \mbox{ for }~i\in \{1,2\}.  \end{eqnarray}
Finally, the map $\bm$  is rooted at the edge emanating from  the
root corner $c_0$ of
$\bu$, which we choose to orient from $c_0$ to $s(c_0)$ if
$\epsilon=+1$, and from $s(c_0)$ to $c_0$ if $\epsilon=-1$. With this
rooting choice, and letting
$\delay:=\delay_1-\delay_2$, we denote by
$\rm{BDG}^{2\bullet}(\bu)$ the tuple $(\bm,(v_1,v_2),\delay)$.

\begin{figure}[!h]
 \begin{center}
 \includegraphics[width=14cm]{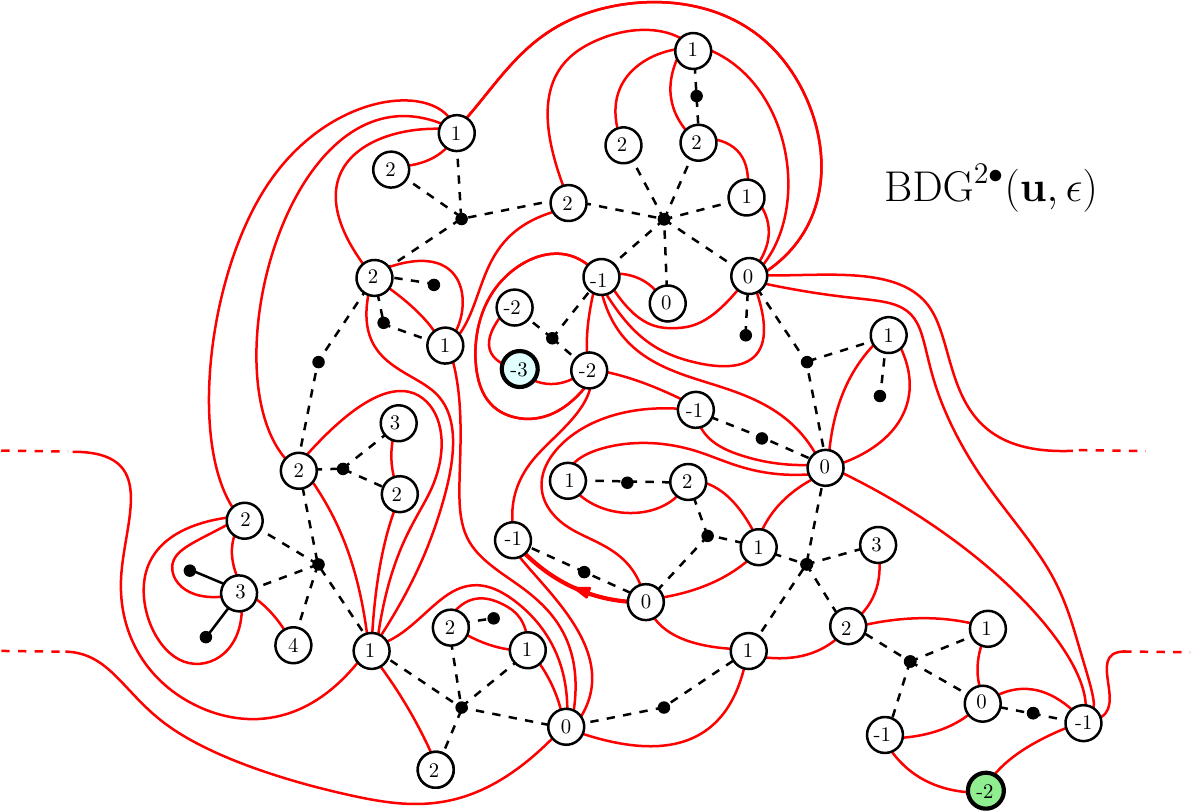}
 \caption{The bi-pointed BDG construction applied on the unicyclomobile of Figure \ref{fig:unicyclo1}. Notice the two distinguished vertices in each face of the unicyclomobile, and the root edge of the map in bold line. For better visibility two edges have been drawn on the sphere and "go around" the figure. \label{fig:unicyclo2}}
 \end{center}
 \end{figure}

 \begin{lem}\label{sec:from-unicycl-map}
   It holds that $\rm{BDG}^{2\bullet}(\bu)$ is an element of
   $\mathcal{M}^{2\bullet}$.
 \end{lem}

 This statement will be proved by deriving some important geometric
 properties of the map $ \mathrm{BDG}^{2\bullet}( \bu)$. 
First, the following can be directly deduced from the construction and
adapting classical arguments in the Bouttier--Di Francesco--Guitter construction \cite{BDFG04}. 
\begin{itemize}
\item The map  $\mathrm{BDG}^{2\bullet}( \bu)$  is a plane bipartite
  map, whose faces of degree $2k$ are
  in one-to-one correspondence with the black vertices of $\bu$ of degree~$k$. 
\item Fix $i\in \{1,2\}$, as well as a white corner $c$ incident to $f_i$,
  and denote  the vertex incident to $c$ by $v$. Then the path $\gamma^{(c)}$ of arcs from $v$ to $v_i$ following the
  consecutive 
  successors $c,s(c),s(s(c)),\ldots$ (until eventually reaching $v_{i}$) is  a geodesic path in $ \mathrm{BDG}^{2\bullet}(\bu)$, again called the \textbf{simple geodesic} starting from $c$. In particular, we have:
\begin{equation}\label{dist:v_*}
   \mathrm{d}^{\mathrm{gr}}_\bm (v,v_i)=\ell(v)-\min \limits_{v\in f_i}\ell(v) +1 = \ell(v) - \delay_{i}.
\end{equation}
\end{itemize}
The so-called ``Schaeffer bound'', analog of \eqref{d_n^circ:BDG}, is
an improved version of the last display and states that for every
$i\in \{1,2\}$ and every pair of white vertices $v,v^{\prime}$   of $\mathrm{BDG}^{2\bullet}(\bu)\setminus\{v_i\}$ incident to $f_i$, we have:
\begin{equation}\label{d_n^circ:2}
  |\ell(v)-\ell(v^{\prime})|\leq  \mathrm{d}^{\mathrm{gr}}_\bm (v,v^{\prime})\leq \ell(v)+\ell(v^{\prime})-2\max\left( \min\limits_{[v,v^{\prime}]_{\bu,f_{i}}}\ell;\min\limits_{[v^{\prime},v]_{\bu, f_{i}}}\ell\right) +2~,
\end{equation}
where $[v,v^{\prime}]_{\bu, f_{i}}$ are the white vertices appearing in a minimal interval in the clockwise
order when going from $v$ to $v^{\prime}$ inside the face $f_i$.
Another important consequence of the construction of $\mathrm{BDG}^{2\bullet}(\bu)$
is that a path $\gamma$ in the map $\bm$ between the two distinguished vertices $v_1$ and $v_2$ necessarily passes through a (white) vertex $v$ of the cycle of
$\bu$. If $c_1,c_2$ denote corners incident to $v$ that are also
respectively incident to $f_1$ and $f_2$, then the concatenation of the two simple geodesics $ \gamma^{(c_{1})}$ and $\gamma^{(c_{2})}$ yields a path of length
$2\ell(v)-\delay_1-\delay_2$. In particular, if the path
$\gamma$ as a geodesic, then it is necessary that:
$$\ell(v)=\min\big\{\ell(v'):v' \text{ white vertex belonging to the
  cycle of }\bu\big\}\, .$$
And in that case indeed,  it is easy to check that the concatenation of $ \gamma^{(c_{1})}$ and $\gamma^{(c_{2})}$ produces a geodesic between $v_{1},v_{2}$ in the map $ \mathbf{m}$. Recalling \eqref{eq:Di} and \eqref{dist:v_*}, we have in particular  for $v$ minimizing the label on the cycle: 
 \begin{eqnarray} \label{eq:distancesdelays}
  \mathrm{d}^{ \mathrm{gr}}_{ \bm}(v,v_{1}) = \frac{1}{2} \left(   \mathrm{d}^{ \mathrm{gr}}_{ \bm}(v_{1},v_{2}) - \delay\right) \quad \mbox{ and }\quad   \mathrm{d}^{ \mathrm{gr}}_{ \bm}(v,v_{2}) = \frac{1}{2} \left(   \mathrm{d}^{ \mathrm{gr}}_{ \bm}(v_{1},v_{2}) + \delay\right).  \end{eqnarray}  
 
Note that in particular, the parity of
$\delay=\delay_1-\delay_2$ is the same as that of $
\mathrm{d}^{\mathrm{gr}}_\bm(v_1,v_2)$, while 
the triangle inequality gives
$\mathrm{d}^{\mathrm{gr}}_\bm
(v_1,v_2)>|\delay_1-\delay_2|=|\delay|$, justifying
that $(\bm,(v_1,v_2),\delay)$ is indeed an element of
$\mathcal{M}^{2\bullet}$, which proves Lemma
\ref{sec:from-unicycl-map}. 

The above discussion on geodesic paths will be crucial to our
purposes. 
By establishing that the set of ``minimal'' vertices $v$ along the
cycle of $\bu$ is typically very
small and localized, this will allow us to prove, in Section \ref{sec:uni:geo},
that geodesics are typically unique in the scaling limit. This line of
reasoning is parallel to \cite[Section 7]{Mie09}, but with some
important differences due to the very different scaling limits we are
working with. This property will also allow us, in Section
\ref{sub:goodpointestimate}, to study the local structure of geodesics in
$( \mathcal{S},D)$ and prove Proposition \ref{main:techni}.  

\subsubsection{Inverse construction}\label{sec:inverse-construction}

We now argue that the mapping $\rm{BDG}^{2\bullet}$ is indeed a bijection
from $\mathcal{U}\times \{-1,1\}$ onto $\mathcal{M}^{2\bullet}$.  We will skip some details, as the discussion closely mirrors that in \cite[Section 2.4]{Mie09}, requiring only minor adaptations to fit the current extended context. Additionally, we will not employ this inverse construction in the sequel.
Consider $(\bm,(v_1,v_2),\delay)\in
\mathcal{M}^{2\bullet}$ and let $\delay_1,\delay_2$ be integers such that
$\delay_1-\delay_2=\delay$. We introduce the labeling function defined
on vertices $v\in V(\bm)$ by
$$\widetilde{\ell}(v):=\min\big( \mathrm{d}^{ \mathrm{gr}}_{\bm}(v,v_1)+\delay_1;\mathrm{d}^{ \mathrm{gr}}_{\bm}(v,v_2)+\delay_2\big)\,
.$$
From the bipartite nature of the map $\bm$, and the fact that $\delay$
has the same parity as $\mathrm{d}^{ \mathrm{gr}}_{\bm}(v_1,v_2)$, one can check that
$|\widetilde{\ell}(v)-\widetilde{\ell}(v')|=1$ for any two adjacent vertices $v,v'$ of
$\bm$. For this reason, we may canonically orient the edges $e$ of
$\bm$ in such a way that $\widetilde{\ell}(e^+)=\widetilde{\ell}(e^-)-1$. Any maximal path of edges
that is oriented in this way will necessarily end in a vertex that is
a local minimum of $\widetilde{\ell}$, in the sense that none of its neighbors has
smaller label. It is easy to see that only $v_1$ and $v_2$ are local
minima of $\widetilde{\ell}$. 

We now construct an element of $\mathcal{U}$ in the following
way. We first add a ``black'', dual vertex $v_f$  inside each face $f$
of $\bm$. Then, we let $(v^f_i : i\in \Z)$ be the sequence of vertices of
$\bm$ appearing in clockwise contour order around the face $f$,
starting from an arbitrary vertex, and extended
periodically. Note that $(\widetilde{\ell}(v^f_i) :  i\in \Z)$  is then an infinite,
$\deg(f)$-periodic path with $\pm1$ steps. We then draw $\deg(f)/2$
arcs $\gamma_i^f$ (that are non-intersecting and disjoint from the edges of $\bm$) from $v_f$ to each of the vertices $v^f_i$ such
that $\widetilde{\ell}(v^f_i)=\widetilde{\ell}(v^f_{i+1})+1$. 

Finally, we let $e_0$ be the canonical orientation of the
root edge of $\bm$, and we let $\epsilon=+1$ if this canonical
orientation coincides with the original orientation of the root, and
$\epsilon=-1$ otherwise. We define the label function
$\ell(v)=\widetilde{\ell}(v)-\widetilde{\ell}(e_0^-)$. 

\begin{lem}
  \label{sec:biject-with-delays-1}
  The graph with ``white'' vertices  $V(\bm)\setminus
  \{v_1,v_2\}$, and ``black'' vertices $\{v^f_i:  (f,i)\in  \mathrm{Faces}(\bm)\times \mathbb{Z}\}$, whose
set of edges is the set of arcs  $A=\{\gamma_{i}^f\}$  as above, and
with label function inherited from $\ell$, is a unicyclomobile $\bu$,
such that the removed vertices  $v_1,v_2$ belong to distinct faces of
$\bu$, thereby labeled $f_1$ and $f_2$. It is rooted at the corner
inherited from the root corner incident to the left of the origin of $e_0$. 
\end{lem}

The proof of this lemma is exactly parallel to Lemma 1 in
\cite{Mie09}. By construction, the graph $\bu$ comes with a bipartite
coloration of its vertices and is well-labeled by $\ell$. To prove
that $\bu$ is indeed a map with two faces, 
we consider the augmented map obtained from $\bm$ by
taking all (white) vertices of $\bm$, all (black) dual vertices, and
all edges and arcs. In this new map, we consider the edges $e^*$ dual to the
original edges $e$ of $\bm$, oriented in such a way that the vertex
incident to $e$ of
smaller label lies systematically to the right of the edge $e^*$. Then
it is easy to check that any oriented dual edge can be uniquely
extended into an infinite oriented dual path, and this path has the
property that the sequence
of vertices
lying to the right of the successive primal edges crossed by the path
have non-increasing label. This implies that the path eventually
cycles around a vertex of locally minimal label, and the only two such
vertices are $v_1$ and $v_2$. Therefore, the oriented dual edges
form a spanning graph which is a cycle-rooted forest with two
components, and whose unique cycles circle around the vertices
$v_1$ and $v_2$. The remaining edges, which are the arcs
$A=\{\gamma_{i}^f\}$, therefore form
a map with two faces.

\subsection{Combinatorial decompositions}\label{sec:comb-decomp} 

Let us discuss some decompositions of unicyclomobiles into simpler
combinatorial objects. This will eventually allow
us to study the random unicyclomobiles associated with random Boltzmann
maps via the BDG$^{2\bullet}$ construction. 

Let $\boldsymbol{\mathcal{T}}=(\mathcal{T},\ell)$ and $\boldsymbol{\mathcal{T}}^\prime=(\mathcal{T'},\ell')$
be two well-labeled mobiles, and, provided the sets
$\widehat{V}_\circ(\mathcal{T})$ and
$\widehat{V}_\circ(\mathcal{T}^\prime)$ of
 white leaves of $\mathcal{T},\mathcal{T'}$ are not empty, we let
$\hat{v}_\circ,\hat{v}'_\circ$ be two distinguished elements of these
sets respectively. 
We define a natural
concatenation operation by changing the labeling function on
$\mathcal{T}'$ to $\ell(\hat{v}_\circ)+\ell'$, and then identifying the leaf
$\hat{v}_\circ$ of $\mathcal{T}$ with the root vertex of
$\mathcal{T}'$, in such a way that the corner incident to
$\hat{v}_\circ$ is merged with the root corner of $\mathcal{T}'$. The result
is a new well-labeled mobile with a distinguished white leaf
$\hat{v}'_\circ$. The neutral element of this concatenation operation
is the trivial vertex-mobile, and the non-trivial irreducible
elements are well-labeled mobiles with a marked
white leaf at generation $2$, which will be called {\bf mobile buckle} (the name will become clear after the forthcoming Proposition \ref{sec:comb-decomp-1}). If $ (\boldsymbol{\mathcal{P}},\hat{v}'_\circ)$ is  a mobile buckle, we denote by $C_\circ(\boldsymbol{\mathcal{P}},\hat{v}'_\circ)$ the set of white corners incident to a white vertex different from $\hat{v}'_\circ$ and where the root white corner has been duplicated.

Mobile buckles can be described in the following convenient way, see Figure \ref{fig:mobile_buckle}. Start
with a  mobile ``star'', that is, a well-labeled mobile buckle
$(\boldsymbol{\mathcal{S}},\hat{v}'_\circ)$ with exactly one black vertex
$\hat{v}_\bullet$, which necessarily has to be the
parent vertex of $\hat{v}'_\circ$. The latter is the parent of the
white vertices $v_1,v_2,\ldots,v_k$ for some $k\geq 1$, arranged in
clockwise order, whose respective labels are denoted by
$\ell_1,\ldots,\ell_k$, and we let $r\in \{1,\ldots,k\}$ be the index such
that $v_r=\hat{v}'_\circ$. By convention we also let $\ell_0=\ell_{k+1}=0$. 

For every $i\in \{0,1,\ldots,k,k+1\}\setminus \{r\}$, we let 
$\boldsymbol{\mathcal{T}}^{(i)}$ be a well-labeled mobile,
whose label function has been shifted by the addition of $\ell_i$. For $1\leq i\leq k$, we graft the root of
$\boldsymbol{\mathcal{T}}^{(i)}$ (at the level of its 
root corner) to the vertex $v_i$, and we identify
the root vertices of 
$\boldsymbol{\mathcal{T}}^{(0)},\boldsymbol{\mathcal{T}}^{(k+1)}$ with
the root vertex of $\boldsymbol{\mathcal{S}}$, by grafting them to 
the two sides of the root of $\boldsymbol{\mathcal{S}}$  separated by the
edge from the root to $\hat{v}_\bullet$. See Figure
\ref{fig:mobile_buckle} for an illustration. The result is clearly a well-labeled 
mobile buckle. 

\begin{figure}
  \centering
  \includegraphics{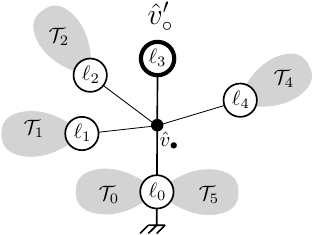}
  \caption{Decomposition of a mobile buckle into a mobile star and
    sub-mobiles, with $k=4$ and $r=3$ in the notation of the text. }
  \label{fig:mobile_buckle}
\end{figure}

\begin{prop}
  \label{sec:boltzm-meas-reduct-1}
  The previous construction is a bijection onto the set of mobile
  buckles $(\boldsymbol{\mathcal{T}},\hat{v}'_\circ)$, where
  $\hat{v}'_\circ$ has $r-1$ older siblings and $k-r$ younger
  siblings. 
\end{prop}
Next, we observe that unicyclomobiles
$\bu\in\mathcal{U}$ are in correspondence with pairs
$((\boldsymbol{\mathcal{T}},\hat{v}_\circ),(\boldsymbol{\mathcal{P}},\hat{v}'_\circ,c))$, where:
\begin{itemize}
\item $(\boldsymbol{\mathcal{T}},\hat{v}_\circ)$ is well-labeled mobile with a distinguished white leaf
  $\hat{v}_\circ$; 
  \item $(\boldsymbol{\mathcal{P}},\hat{v}'_\circ)$ is a well-labeled mobile buckle,
    and $c\in C_\circ(\mathcal{P},\hat{v}'_\circ)$ is a
    distinguished white corner distinct from the one incident to $\hat{v}_{\circ}'$ (and where the root corner has been duplicated);
    \item we have
      $\ell_{\mathcal{T}}(\hat{v}_\circ)=-\ell_{\mathcal{P}}(\hat{v}'_\circ)$. 
  \end{itemize}
The correspondence is illustrated in Figure
\ref{fig:unicyclomobile_decomp}. We concatenate
$\boldsymbol{\mathcal{T}}$ with $\boldsymbol{\mathcal{P}}$ as in the
beginning of the section, resulting in a non-trivial\footnote{Note
  that, contrary to $\mathcal{P}$, the mobile $\mathcal{T}$ may well be the
  trivial vertex-mobile.} well-labeled mobile with a distinguished
white leaf $\hat{v}'_\circ$, whose label is
$\ell_{\mathcal{T}}(\hat{v}_\circ)+\ell_{\mathcal{P}}(\hat{v}'_\circ)=0$. This
allows us to merge the corner incident to this leaf with the root corner
of $\mathcal{T}$, resulting in a unicyclomobile
$\bu=\Phi((\boldsymbol{\mathcal{T}},\hat{v}_\circ),(\boldsymbol{\mathcal{P}},\hat{v}'_\circ,c))\in  \mathcal{U}$, rooted at the white
corner $c$. By convention, we let
$f_1$ be the face that is incident to the left side of $\mathcal{P}$,
and $f_2$ be the other face. Finally, we shift all the labels by
subtracting the label of the root corner $c$ in $\mathcal{P}$.  
This construction is easily inverted, as indicated on Figure
\ref{fig:unicyclomobile_decomp}. The only delicate
point is to decide where to cut the cycle of a unicyclomobile $\bu$ to
recover the pair
$((\boldsymbol{\mathcal{T}},\hat{v}_\circ),(\boldsymbol{\mathcal{P}},\hat{v}'_\circ,c))$:
to this end, we start from the distinguished corner $c$. If it is
incident to $f_1$ (resp.\ $f_2$), we run in counterclockwise (resp.\ clockwise) contour order until we first
meet a black vertex $\hat{v}'_\bullet$ on the cycle of $\bu$, and then stop and cut the
cycle at the first
encounter of a white vertex $\hat{v}'_\circ$ on the cycle
afterwards. This is materialized by the two blue oriented curves on
the picture. The vertex $\hat{v}_\circ$ is then found by
backtraking two edges along the cycle. Summarizing, we obtain the
following statement. 

\begin{prop}[Belt-Buckle decomposition]
  \label{sec:comb-decomp-1} The mapping $\Phi$ is a bijection between
${\mathcal{U}}$ and the set of pairs
$((\boldsymbol{\mathcal{T}},\hat{v}_\circ),(\boldsymbol{\mathcal{P}},\hat{v}'_\circ,c))$
described above. We will call them respectively the {\bf belt} and the
{\bf buckle} associated with the corresponding unicyclomobile. 
\end{prop}

    \begin{figure}
      \centering
      \includegraphics[scale=.8]{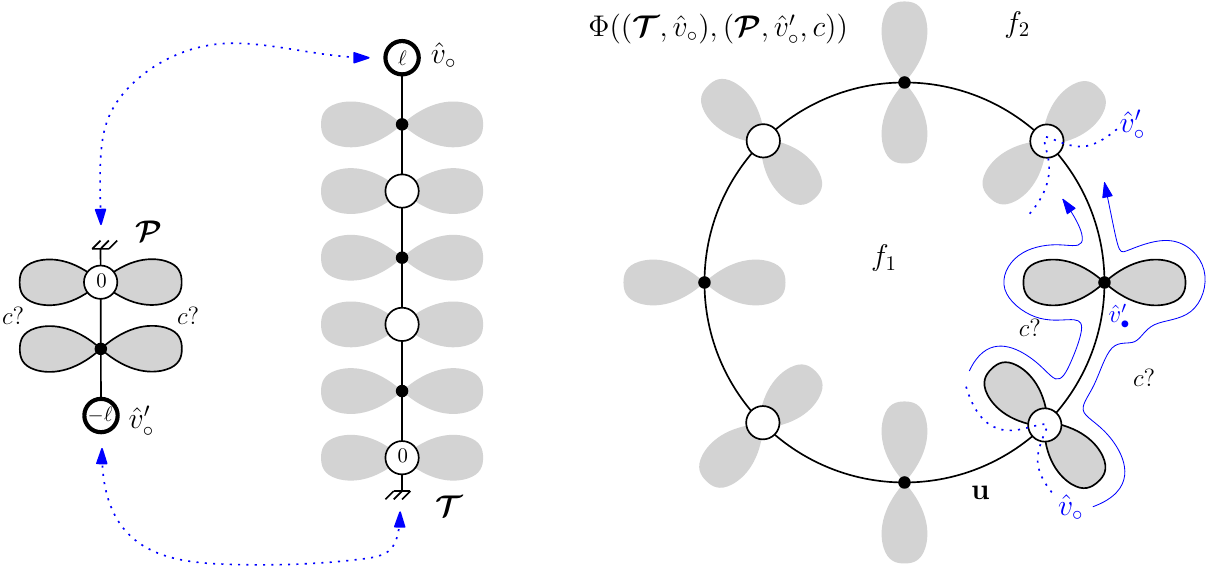}
      \caption{An element $\bu\in\mathcal{U}$ is described by a pair
        $(\boldsymbol{\mathcal{T}},\hat{v}_\circ),(\boldsymbol{\mathcal{P}},\hat{v}'_\circ,c)$, where the
        first element (the belt) is a mobile with a marked leaf, and
        the second (the buckle)
        is a mobile buckle with a marked white corner $c$. The grey
        blobs indicate sub-mobiles, and 
        circled grey blobs indicate the places where the distinguished
        white corner can be located. The oriented curves and dashed
        lines indicate how one should cut and split $\bu$ to recover
        the mobiles. }      
\label{fig:unicyclomobile_decomp}
    \end{figure}

\subsection{Boltzmann measures}\label{sec:boltzmann-measures}

Recall \eqref{eq:defface}. Let us now consider  the natural sigma-finite measure on  $\mathcal{M}^{2\bullet}$  defined by:
$$w_\bq^{2\bullet}\big(\bm,(v_1,v_2),\delay\big)=w_\bq(\bm)\, ,\qquad
(\bm,(v_1,v_2),\delay)\in \mathcal{M}^{2\bullet}\, ,$$
which we may formally rewrite as 
\begin{equation}
  \label{eq:w2bullet}
  w_{\bq}^{2\bullet}\big(\d(\bm,(v_1,v_2),\delay)\big):=
\#V(\bm)^2\, w_{\bq}(\d \bm) \mathrm{vol}_{\bm}(\d v_1)
\mathrm{vol}_{\bm}( \d v_2 )~  \#_{\Edelay_{\bm, v_1,v_2}} (\d
\delay)\, .
\end{equation}
Here, $\#_{\Edelay_{\bm, v_1,v_2}} (\d \delay)$ stands for  the counting measure on the set $\Edelay_{\mathbf{m},v_1,v_2}$ of integers $\delay$ such that   $|\delay|<\mathrm{d}^{ \mathrm{gr}}_{\bm}(v_1,v_2)$ and
    $\mathrm{d}^{ \mathrm{gr}}_{\bm}(v_1,v_2)+\delay$ is an even number. Simply, this set
    has cardinality   \begin{eqnarray} \label{eq:cardinalityD}\# \Edelay_{\mathbf{m},v_1,v_2}=(\mathrm{d}^{ \mathrm{gr}}_{\bm}(v_1,v_2)-1)_+,  \end{eqnarray}
and in particular $\# \Edelay_{\mathbf{m},v_1,v_2}\leq \# V(\bm)$. Consequently,  for every $n\geq 1$, the event $\{\#V(\bm)=n\}$ has finite measure under $w_{\bq}^{2\bullet}$. 
 Notice that the pushforward of $w_\bq^{2\bullet}$ by the bijection $\rm{BDG}^{2\bullet}:\mathcal{M}^{2\bullet}\to
\mathcal{U}\times \{-1,1\}$ described in the preceding section, and
after forgetting the sign variable $\epsilon\in \{-1,1\}$, is given by:
  \begin{eqnarray} \label{eq:defwtilde2pt}\widetilde{w}_\bq^{2\bullet}(\bu)=2\prod_{v\in
  V_\bullet(\bu)}q_{\deg(v)}\, .  \end{eqnarray}
In particular, if $(\mathfrak{M}_{n}, v_1^{n},v_2^n,\delay_n)$
has  distribution $w_{\bq}^{2\bullet}(\d(\bm,(v_1,v_2),\delay)| \#V(\bm)=n)$, then
the associated unicyclomobile $\bu_n$ has distribution
$\widetilde{w}^{2\bullet}_\bq(\d \bu\, |\, \#V_\circ(\bu)=n-2)$. 
Our primary aim now is to investigate the behavior of $\bu_n$ as $n\to
\infty$.  This analysis will enable us to deduce certain geometric
properties of $(\mathfrak{M}_n,v_1^n,v_2^n,\delay_n)$ and deliver the proofs of the results announced  in Section \ref{sec:D=D*}. 
To implement this program, we will describe the unicyclomobile $\bu$ under the
sigma-finite ``law'' $\widetilde{w}^{2\bullet}_\bq$ by using simpler
probabilistic objects, with the help of the decompositions of Section \ref{sec:comb-decomp}. \medskip

Recall the notation $\mu_\circ,\mu_\bullet$ of \eqref{eq:mucirc_mubullet}, and recall that 
$\rm{GW}_\bq(\d \boldsymbol{\mathcal{T}})$ denotes the law of a
well-labeled mobile described by an alternating 
2-type branching process, starting with a white vertex, with offspring distributions
$\mu_\circ,\mu_\bullet$, and where the labels are uniformly distributed
among possible well labelings conditionally given the tree structure.

For $h\in\Z_+$, we define the measure
  \begin{eqnarray} \label{def:gwh} \widehat{\mathrm{GW}}^{(h)}_\bq(\d(\boldsymbol{\mathcal{T}},\hat{v}))=\frac{\mathrm{GW}_\bq(\d
\boldsymbol{\mathcal{T}})}{\mu_\circ(0)}\sum_{v\in \widehat{V}_\circ(\mathcal{T})}\ind_{\{|\hat{v}|=2h\}}\delta_v(\d \hat{v})\,
,  \end{eqnarray}
which consist in marking one distinguished white leaf at height $2h$ according to
the counting measure.

In particular, note that
$\widehat{\mathrm{GW}}_\bq^{(1)}$ is supported on the set of mobile buckles. As
an easy exercise, we invite the reader to check that the image measure
of $\widehat{\mathrm{GW}}_\bq^{(h)}\otimes \widehat{\mathrm{GW}}_\bq^{(h')}$ under the
concatenation operation is
$\widehat{\mathrm{GW}}_\bq^{(h+h')}$ (notice the factor $1/{\mu_\circ(0)}$ has conveniently be included in the definition).
In order to give a description of random mobile buckles under
$\widehat{\mathrm{GW}}_\bq^{(1)}$, we invite the reader to recall the
discussion around Proposition \ref{sec:boltzm-meas-reduct-1}, and define the following random
variables associated with a mobile buckle
$(\boldsymbol{\mathcal{P}},\hat{v}'_\circ)$: 
\begin{itemize}
\item $K$ is the number of children of $\hat{v}_\bullet$, the parent vertex
  of $\hat{v}'_\circ$;
\item $R$ is the rank in clockwise order around $\hat{v}_\bullet$ of
  the vertex $\hat{v}'_\circ$ among its siblings. 
\end{itemize}

\begin{lem}
  \label{sec:boltzmann-measures-1}
Under
the law $\widehat{\mathrm{GW}}_\bq^{(1)}$, 
\begin{itemize}
\item the random variable $K$ has distribution
  $\hat{\mu}_\bullet(k)=k\mu_\bullet(k)/m_\bullet$, the size-biased
  distribution associated with $\mu_\bullet$;
  \item conditionally given $K$, the random variable $R$ is uniformly
    distributed in $\{1,\ldots,K\}$. 
  \end{itemize}
  Moreover, conditionally given $(K,R)=(k,r)$,
  \begin{itemize}
  \item   the labels
$(\ell_0,\ell_1,\ldots,\ell_k,\ell_{k+1})$ of vertices around
$\hat{v}_\bullet$ form a random walk with i.i.d.~steps distributed according to a shifted geometric law $\sum_{n\geq -1}2^{-n-2}\delta_n$,
and conditioned on $\ell_0=\ell_{k+1}=0$;
\item the random mobiles $\boldsymbol{\mathcal{T}}^{(i)}$ for $i\in
  \{0,\ldots,k+1\}\setminus \{r\}$ are i.i.d.~with distribution
  $\mathrm{GW}_\bq$. 
\end{itemize}
\end{lem}

\begin{proof}
  This is a direct consequence of Proposition
  \ref{sec:boltzm-meas-reduct-1}. For a given mobile buckle
  $(\boldsymbol{\mathcal{P}},\hat{v}'_\circ)$ obtained by appending the
  given well-labeled mobiles $\boldsymbol{\mathcal{T}}^{(i)}$, for $i\in
  \{0,\ldots,k+1\}\setminus \{r\}$, to a mobile ``star'' with $k+1$
  branches, we have, 
  \begin{align*}
    \widehat{\mathrm{GW}}_\bq^{(1)}(\{(\boldsymbol{\mathcal{P}},\hat{v}'_\circ)\})&=\frac{\mathrm{GW}_\bq(\{\boldsymbol{\mathcal{P}}\})}{\mu_\circ(0)}\\
    &=\frac{1}{\binom{2k+1}{k}}\mu_\bullet(k)\prod_{v\in
      V_\circ(\mathcal{P})\setminus
      \{\hat{v}'_\circ\}}\mu_\circ(k_v(\mathcal{P}))\prod_{v\in
      V_\bullet(\mathcal{P})\setminus \{\hat{v}_\bullet\}}\frac{\mu_\bullet(k_v(\mathcal{P}))}{\binom{2k_v(\mathcal{P})+1}{k_v(\mathcal{P})}}\, ,
  \end{align*}
where the inverse binomial factors come from the number of possible
labelings of the neighbours of the different black vertices, and the
first term is the contribution of $\hat{v}_\bullet$. Note that, from
the fact that $\mu_\circ$ is a geometric distribution and the
criticality assumption $m_\circ m_\bullet=1$, we may
rewrite
$$\mu_\circ(k_{\mathrm{root}}(\mathcal{P}))=\frac{1}{m_\bullet}\mu_\circ(k_{\mathrm{root}}(\mathcal{T}^{(0)}))\mu_\circ(k_{\mathrm{root}}(\mathcal{T}^{(k+1)}))\, , $$
so that we may re-express the above display as
$$
\widehat{\mathrm{GW}}_\bq^{(1)}(\{(\boldsymbol{\mathcal{P}},\hat{v}'_\circ)\})=\frac{1}{\binom{2k+1}{k}}\frac{\hat{\mu}_\bullet(k)}{k}\prod_{i\in
\{0,1,\ldots,k+1\}\setminus
\{r\}}\mathrm{GW}_\bq(\{\boldsymbol{\mathcal{T}}^{(i)}\})\, .$$
We conclude by the elementary observation that a uniformly chosen
possible labeling $(\ell_0,\ell_1,\ldots,\ell_k,\ell_{k+1})$ among the
$\binom{2k+1}{k}$ possible ones has the same law as the claimed
conditioned random walk.  
\end{proof}

Next, we define the two measures $\widehat{\mathrm{GW}}_\bq=\sum_{h\geq
  0}\widehat{\mathrm{GW}}^{(h)}_\bq$, and 
  \begin{eqnarray} \label{def:gwtilde}\widetilde{\mathrm{GW}}_\bq(\d
(\boldsymbol{\mathcal{P}},\hat{v}',c))=\widehat{\mathrm{GW}}_\bq^{(1)}(\d(\boldsymbol{\mathcal{P}},\hat{v}'))
\sum_{c\in C_\circ(\boldsymbol{\mathcal{P}},\hat{v}')}(\d c)\, ,  \end{eqnarray}
where we recall that $C_\circ(\boldsymbol{\mathcal{P}},\hat{v}')$ is the set of corners incident to
a white vertex different from $\hat{v}'$ and where the white root corner has been duplicated.

\begin{prop}
  \label{sec:boltzm-meas-reduct}
The measure $\widetilde{w}^{2\bullet}_\bq/2$ is the image measure under
$\Phi$ (recall Figure \ref{fig:unicyclomobile_decomp}) of 
$$\widehat{\mathrm{GW}}_\bq(\d(\boldsymbol{\mathcal{T}},\hat{v}))
\widetilde{\mathrm{GW}}_\bq(\d (\boldsymbol{\mathcal{P}},\hat{v}',c))
\ind_{\{\ell_\mathcal{T}(\hat{v})=-\ell_{\mathcal{P}}(\hat{v}')\}}\, .$$
\end{prop}

\begin{proof}
  We fix a unicyclomobile $\bu$, that we may write as
  $\Phi((\boldsymbol{\mathcal{T}},\hat{v}_\circ),(\boldsymbol{\mathcal{P}},\hat{v}'_\circ,c))$ in a unique way. In
  particular, notice that $\ell_\mathcal{T}(\hat{v}_\circ)=-\ell_{\mathcal{P}}(\hat{v}_\circ')$. Then, it
  suffices to show that
$$\mathrm{GW}_\bq(\{\boldsymbol{\mathcal{T}}\})
\mathrm{GW}_\bq(\{\boldsymbol{
\mathcal{P}}\})=\prod_{v\in V_\bullet(\bu)}q_{\deg(v)}\, .$$
Using the fact that, in a mobile, every black vertex is the child of a
white vertex, while every white vertex but the root is the child of a
black vertex, we observe that 
\begin{align*}
  \mathrm{GW}_\bq(\{\boldsymbol{\mathcal{T}}\}) &=
  \prod_{v\in
  V_\circ(\mathcal{T})\setminus\{\hat{v}_\circ\}}\mu_\circ(k_v(\mathcal{T}))\prod_{v\in
  V_{\bullet}(\mathcal{T})}\frac{\mu_\bullet(k_v(\mathcal{T}))}{\binom{2k_v(\mathcal{T})+1}{k_v(\mathcal{T})}}\\
                                     &=
  \prod_{v\in
  V_\circ(\mathcal{T}) \setminus\{\hat{v}_\circ\}}\zbq^{-1}\left(1-\zbq^{-1}\right)^{k_v(\mathcal{T})}\prod_{v\in
  V_{\bullet}(\mathcal{T})}
                                     \frac{\zbq^{k_v(\mathcal{T})}
                                     q_{k_v(\mathcal{T})+1}}{1-\zbq^{-1}}\\
                                   &=\prod_{v\in
                                     V_\bullet(\mathcal{T})}q_{k_v(\mathcal{T})+1}\, .
\end{align*}
In the first line, the product of inverses of binomial coefficients comes from the choice of the labeling function of $\mathcal{T}$. 
Similarly,
\begin{align*}
  \mathrm{GW}_\bq(\{\boldsymbol{\mathcal{P}}\}) =\prod_{v\in
  V_\circ(\mathcal{P})\setminus\{\hat{v}'_\circ\}}\zbq^{-1}\left(1-\zbq^{-1}\right)^{k_v(\mathcal{P})}\prod_{v\in
  V_{\bullet}(\mathcal{P})\setminus\{\hat{v}'_\circ\}}                         \frac{\zbq^{k_v(\mathcal{P})}q_{k_v(\mathcal{P})+1}}{1-\zbq^{-1}}
                                   =\prod_{v\in
                                     V_\bullet(\mathcal{P})}q_{k_v(\mathcal{P})+1}\, ,
\end{align*}
which gives the result. 
\end{proof}

\section{Scaling limit of Boltzmann unicyclomobiles} \label{sec:scalingunicyclo}

The purpose of this section is to obtain scaling limit results for a
random unicyclomobile under the measure $\tilde{w}^{2\bullet}_\bq(\d\bu) $
conditioned on the event that $\bu$ has $n$ white vertices, as
$n\to\infty$. In this section we will use the \textbf{assumption that $\bq$ is strictly non-generic} i.e.~satisfies
furthermore \eqref{eq:strictly-non-generic} which we supposed from
Section \ref{sub:reroot} on.

\subsection{Scaling limit for the belt of a random unicyclomobile}\label{sec:result}

Let us state the main result of this section. Let $ \mathbf{u}_{n}$ be a random
 unicyclomobile of law $\tilde{w}^{2\bullet}_\bq(\d\bu \mid \#
{V}_{\circ}( \mathbf{u})= n-2)$ and denote the belt part in the
belt-buckle decomposition of Proposition \ref{sec:comb-decomp-1}  by $
(\boldsymbol{\mathcal{B}}_{n}, \hat{v}_{\circ}^n)$. In
particular, $ \boldsymbol{ \mathcal{B}}_{n}$ is a pointed labeled tree
of random size $ \theta^{n} \leq n-1$, and we write 
$(S^{\boldsymbol{\mathcal{B}}_{n}}_{k}, L^{\boldsymbol{\mathcal{B}}_{n}}_{k})_{k \geq 0}$
for its Lukasiewicz encoding as in  Section \ref{sec:tightness}. 
We also
denote  the time of visit of $ \hat{v}_{
  \circ}$ by $a^{n}$. 
The main result of this section identifies the scaling limit of
$(S^{\boldsymbol{\mathcal{B}}_{n}}, L^{\boldsymbol{\mathcal{B}}_{n}},a^n,\theta^n)$
in terms of the law of $(X,Z, t_{\bullet}, \sigma)$ under the
measure $ \mathbf{N}^{\bullet}$ defined in Section \ref{sec:prison}. 
In order to state it, we introduce two bits of notation. We let
  \begin{eqnarray} \label{eq:defbarp}\bar{p}_t(z):=\frac{1}{t}\int_0^t\d s\, \sqrt{\frac{t}{2\pi s(t-s)}}
 \mathrm{e}^{-\frac{t\, z^2}{2s(t-s)}} \underset{ \eqref{mellin:erfc}}{=}\sqrt{\frac{\pi}{2t}}\,
\mathrm{erfc}\left( |z|\sqrt{\frac{2}{t}}\right)\, ,\qquad z\in
\R\, ,  \end{eqnarray}
be the 
density of a standard Brownian bridge of duration $t>0$ sampled at an independent, uniformly
random time in $[0,t]$. Finally, for $x>0$ and $z\in \R$, we let 
\begin{equation}
  \label{eq:GG}
   \GG(x,z) :=\int_0^\infty\frac{\d t}
   {\Gamma(-\alpha)\, 
     t^{\alpha-1}}\, q^{[\alpha]}_x(-t)\, \bar{p}_t(z)\, ,
 \end{equation}
 where, for every $\beta\in (0,1)\cup (1,2)$, and $c>0$, $q^{[\beta]}_{c}$ is the density of a stable spectrally
   positive L\'evy process with exponent $\beta$ taken at time $c$,
   defined by its Laplace transform
\begin{equation}
  \label{eq:qbetalaplace}
  \int_\R  \mathrm{e}^{-\lambda x}q^{[\beta]}_c(x)\, \d x=
    \begin{cases}
  \exp(-c\lambda^\beta) &\mbox{ if }\beta\in (0,1)\\
   \exp(c\lambda^\beta) &\mbox{ if }\beta\in (1,2)
  \end{cases}
    ,\qquad \lambda\geq 0\, ,
\end{equation}
and we let 
$\GG(0,z):=0$ for every $z\in \R$. 
 Note that $\GG$ is a
continuous function on $(0,\infty)\times (\R\setminus \{0\})$, because of the Gaussian tails of the error function for
$t\to 0$, and because of the stretched-exponential tails of
$q_x^{[\alpha]}(-t)$ as $t\to\infty$,  see \eqref{eq:tailqalphsleft} below. For  $\alpha\in (1,3/2)$, it holds that $\GG (x,0)<\infty$
and $\GG$ is in fact continuous on $(0,1)\times \R$. On the other
hand, since $q_x^{[\alpha]}(0)>0$, one should note that $\GG(x,0)=\infty$ when
$\alpha\in [3/2,2)$ and $x>0$.

\begin{prop}[Scaling limits for the belt of a Boltzmann-distributed
  unicyclomobile] \label{prop:scalingbelt} Assume that $\bq$ is strictly non-generic.
  For any bounded continuous function $ F : \mathbb{D}([0,1], \mathbb{R})^2 \times [0,1]^{2} \to \mathbb{R}$, we have 
\begin{align*}& \mathbf{E}\left[ F\left(
               2 (\scal n)^{-\frac{1}{\alpha}}S_{\lfloor (n-1)\cdot\rfloor}^{\boldsymbol{
                \mathcal{B}}_{n}},(\scal n)^{-\frac{1}{2\alpha}}L_{\lfloor (n-1)\cdot\rfloor}^{\boldsymbol{ \mathcal{B}}_{n}}, \frac{a^{n}}{n}, \frac{\theta^{n}}{n}\right)\right]\\
 \xrightarrow[n\to\infty]{}\quad & \mathrm{Cst} \cdot
                                   \mathbf{N}^{\bullet}\left( F (X,Z,
                                   t_{\bullet}, \sigma) \cdot
                                   \GG\left(1-\sigma,-Z_{t_{\bullet}}\right)\right),  \end{align*}
                               where the constant is such that the
                               right-hand side defines a probability
                               distribution. 
   \end{prop}

As a consequence of our study, we will also obtain the following
statement that will be useful 
in the next section. Recall that $\mathfrak{M}_n$ denotes a random
$\bq$-Boltzmann map conditioned to have $n$ vertices. 

\begin{prop} \label{lem:esperance} Conditionally on $
  \mathfrak{M}_{n}$, let $v_{1}^{n}$ and $v_{2}^{{n}}$
  be two independent uniform random vertices. Then
  the sequence of random variables 
  $(n^{-\frac{1}{2\alpha}}\mathrm{d}^{ \mathrm{gr}}_{ \mathfrak{M}_n}(
  v_1^n, v_2^n))_{n \geq 1}$ is uniformly
  integrable. Consequently,   along the subsequence $(n_k)_{k\geq 1}$ defined before 
  \eqref{eq:second_couple} it holds that $$ \mathbf{E}[ (\scal
  n)^{-\frac{1}{2\alpha}}\cdot \mathrm{d}^{ \mathrm{gr}}_{ \mathfrak{M}_n}(
  v_1^n, v_2^n)] \xrightarrow[n\to\infty]{} \mathbf{E}[D(
  \rho_1, \rho_2)] \, ,$$
where $\rho_{1}, \rho_{2}$ are two independent random points in $( \mathcal{S},D)$ of law $ \mathrm{Vol}$.
\end{prop}

The rest of this technical section is devoted to the proofs of Propositions
\ref{prop:scalingbelt} and \ref{lem:esperance}, which will be based on Proposition
\ref{sec:boltzm-meas-reduct}, and first requires to address a similar
question for mobiles with $\sigma$-finite ``distributions''
$\widetilde{\mathrm{GW}}_\bq$ and $\widehat{\mathrm{GW}}_\bq$  using
some slightly delicate local limit theorems.
At first, we gather some needed estimates for heavy-tailed random variables.

\subsection{Some classical estimates on heavy-tailed random variables}\label{sec:some-class-estim}

In this section, we are going to make an extensive use of classical
local limit theorems for random variables in stable domains of
attraction, so let us recall some basic facts about stable densities
and domains of attractions.

\paragraph{On the stable densities $q_c^{[\beta]}$ defined at (\ref{eq:qbetalaplace}).}
For $\beta\in (0,1)\cup
(1,2)$ and $c>0$, we note that, by 
 \cite[Theorem 1.18]{kyprianou2022stable}, it
holds that: 
\begin{equation}
  \label{eq:qcbetazero}
  c^{1/\beta}q^{[\beta]}_c(0)=
\begin{cases}
  \frac{1}{|\Gamma(-1/\beta)|} & \mbox{ if }\beta>1\\
\frac{1}{|\Gamma(-\beta)|} & \mbox{ if }\beta<1\, .
\end{cases}
\end{equation}
 We will also use stretched-exponential decay of the left tail of
 $q^{[\beta]}_{1}$: there exist $c_{1},c_{2}>0$ (depending on $\beta$) such that
  \begin{eqnarray} \label{eq:tailqalphsleft} q_{c}^{[\beta]}(-t) \leq \frac{c_{1}}{c^{1/\beta}} \exp\left(-c_{2} \left(\frac{t}{c^{1/\beta}}\right)^{ \frac{\beta}{\beta -1}}\right),   \end{eqnarray}  for $c,t>0$, see \cite[Theorem 2.5.3]{Zol86}.
 Moreover, assuming that $\beta\in (0,1)$,  the following 
identity is known as Zolotarev's duality \cite[Theorem 1.16 and (1.31)]{kyprianou2022stable}, which can also be seen as a continuous version of the cyclic lemma: 
\begin{equation}
  \label{eq:ballot}
  q^{[\beta]}_c(x)=\frac{c}{x}q^{[1/\beta]}_x(-c) \, ,\qquad c,x>0\, .
\end{equation}

\paragraph{On heavy-tailed random variables. }
Now let $\xi_1^\circ,\xi_2^\circ,\ldots$ be integer-valued, i.i.d.\ nonnegative random
variables with a non-lattice distribution, meaning that
$\mathrm{gcd}(\{k:\mathbf{P}(\xi^\circ_1=k)>0\})=1$. We assume that
\begin{equation}
  \label{eq:tailboundstable}
  \mathbf{P}(\xi^\circ_1>k)\sim \frac{c}{|\Gamma(1-\beta)|k^{\beta}}\, ,\qquad
  k\to\infty,
\end{equation}
for some $c\in
(0,\infty)$. We let $\xi_i:=\xi^\circ_i$ if $\beta\in (0,1)$, and
$\xi_i:=\xi^\circ_i-\mathbf{E}[\xi^\circ_i]$ if $\beta\in (1,2)$.
These assumptions imply, by an Abelian theorem,
that
\begin{equation}\label{eq:abelian}
   \mathbf{E}[\exp(-\lambda \xi_1)]=
\begin{cases}
  1-c\lambda^\beta(1+o(1)) & \mbox{ if }\beta\in (0,1) \\
  1+c\lambda^\beta(1+o(1))  & \mbox{ if }\beta\in (1,2)
\end{cases}\, , \quad \lambda \to 0\, ,
\end{equation}
which in turn implies that the rescaled random sum
$(\xi_1+\dots+\xi_n)/n^{1/\beta}$ converges in distribution as $n\to\infty$ to a
random variable with density $q^{[\beta]}_{c}$. 
We are going to make extensive use of the Gnedenko
local limit theorem (see \cite[Theorem 4.2.1]{IL71}), according to which 
 \begin{equation}
   \label{eq:loc_lim}
   \sup_{k\in \Z}\left|n^{1/\beta}\mathbf{P}(\xi_1+\cdots+\xi_n=k)-q^{[\beta]}_c\left(\frac{k}{n^{1/\beta}}\right)\right|\underset{n\to\infty}{\longrightarrow}0\, .
 \end{equation}
This result will be complemented 
by exponential bounds for the left-tail of $\xi_1+\dots+\xi_n$. If
$\beta\in (1,2)$, the Chernov bound, together with (\ref{eq:abelian}) applied at $\lambda = n^{-1/\beta}$, implies that
\begin{equation}
  \label{eq:lefttailboundbetabig}
   \mathbf{P}(\xi_1+\dots+\xi_n\leq -k)\leq
   C\exp\left(-\frac{k}{n^{1/\beta}}\right)\, , \quad k,n\geq 1\, ,
\end{equation}
for some constant $C\in (0,\infty)$.

Finally, in the
case where $\beta\in (0,1)$, it holds that 
\begin{equation}
  \label{eq:lefttailboundbetasmall}
  \mathbf{P}(\xi_1+\dots+\xi_n\leq k)\leq
  C\exp\left(-\frac{n}{k^{\beta}}\right)\, , \quad k,n\geq 1\, ,
\end{equation}
for some constant $C\in (0,\infty)$. This is obtained by
combining the bound
$$\mathbf{P}(\xi_1+\cdots+\xi_n\leq k)\leq
\mathbf{E}\left[\exp\left(\lambda\left(1-\frac{\xi_1+\cdots+\xi_n}{k}\right)\right)\right]= \mathrm{e}^\lambda\,
\mathbf{E}\left[\exp\left(-\frac{\lambda \xi_1}{k}\right)\right]^n\,
,$$
where $\lambda$ is chosen arbitrarily in $(0,1/c^{1/\beta})$, 
with the estimate \eqref{eq:abelian}.

\subsection{Scaling limit of the buckle}\label{sec:scaling-limit-buckle}
We first deal with the scaling limit of a marked mobile piece with
sigma-finite ``distribution'' $\widetilde{\mathrm{GW}}_\bq$ defined in \eqref{def:gwtilde}, starting with
the easier case of an unmarked mobile piece with law
$\widehat{\mathrm{GW}}_\bq^{(1)}$ see \eqref{def:gwh}.  Recall the definition of the random
variables $K,R$
associated with mobile pieces $\boldsymbol{\mathcal{P}}$  discussed around Lemma
\ref{sec:boltzmann-measures-1}, and introduce two extra random
variables, this time associated with any pointed mobile
$(\boldsymbol{\mathcal{T}},\hat{v}_\circ)$: 
\begin{itemize}
\item $M$ is the number of white vertices of $\mathcal{T}$ different from
  $\hat{v}_\circ$;
\item $L=\ell_{\boldsymbol{\mathcal{T}}}(\hat{v}_\circ)$ is the label of $\hat{v}_\circ$. 
\end{itemize}
Let $\Xi$ denote the law of $(M,L)$ under
$\widehat{\mathrm{GW}}_\bq^{(1)}$. We write
$\Xi(m,l)=\Xi(\{(m,l)\})=\widehat{\mathrm{GW}}_\bq^{(1)}(M=m,L=l)$ for
simplicity. Out first goal is
to describe the asymptotic behavior of this law.  

\begin{prop}
  \label{sec:scal-limit-adjac-1}
  Fix $x>0$ and $z\in\R$. Then, for any two sequences
  $(m_N),(l_N)$ such that $m_N\sim N^{\frac{\alpha}{\alpha-1}}\, x$ and
  $l_N\sim N^{\frac{1}{2(\alpha-1)}}\, \scal^{\frac{1}{2\alpha}} z$, it holds that
  \begin{equation}
    \label{eq:adjacent_restricted}
  2 \scal^{1-\frac{1}{2\alpha}}\,N\cdot 
      N^{\frac{\alpha}{\alpha-1}}\cdot N^{\frac{1}{2(\alpha-1)}}\cdot
      \Xi(m_N,l_N)
      \underset{N\to\infty}{\longrightarrow} \frac{\GG(x,z)}{ x}\, .
    \end{equation}
     Similarly, it holds that, for every $\eta\in (0,\infty]$,
     \begin{equation}
    \label{eq:adjacent_restricted2}
   2\scal\, N\cdot 
      N^{\frac{\alpha}{\alpha-1}}\cdot
      \Xi\big(\{m_N\}\times (-\eta
      N^{1/2(\alpha-1)},\eta N^{1/2(\alpha-1)})\big)
      \underset{N\to\infty}{\longrightarrow}
      \int_{-\eta}^\eta\frac{\GG(x,z)}{x}\,  \mathrm{d}z\, .
    \end{equation}
  \end{prop}

Here, we recall that the function $\GG$ is defined at (\ref{eq:GG}), and note that the integral
in~(\ref{eq:adjacent_restricted2}) is always finite despite the fact
that $\GG(x,z)$ may explode at $z=0$, because
$\bar{p}_t$ is an approximation of $\delta_0$ as $t\downarrow 0$. 
From this, it will be easy to deduce the following scaling limit result for the
buckle measure $\widetilde{\mathrm{GW}}_\bq$, which, we recall, is
the measure $\widehat{\mathrm{GW}}_\bq^{(1)}$ biased by the total
number of white corners $\#C_\circ$ (not incident to the pointed white leaf, and with root corner duplicated). Let us introduce 
the notation
$\widetilde{G}(m,l)=\widetilde{\mathrm{GW}}_\bq(M=m,L=l)$, which we
view as a measure on $\N\times \Z$. 
The measure $\widetilde{G}$ is closely connected to 
$\Xi$, as the following result shows. 

\begin{lem}
  \label{sec:scaling-limit-buckle-1}
For every $\epsilon>0$, there exists $c(\eps)\in (0,\infty)$ such that,
for every $A\subset \Z$ and $m\geq 1$, 
$$(\zbq-\eps)m \, \Xi(\{m\}\times A)- \mathrm{e}^{-c(\eps)m}\leq
\widetilde{G}(\{m\}\times A)\leq (\zbq+\eps)m\, \Xi(\{m\}\times A)+ \mathrm{e}^{-c(\eps)m}\, .$$
\end{lem}

Together with Proposition \ref{sec:scal-limit-adjac-1}, this lemma
will allow to show the following estimates on $\widetilde{G}$. 

\begin{prop}
  \label{sec:scal-limit-adjac-2}
  Fix $x>0$ and $z\in\R$, and let $(m_N),(l_N)$ be
  sequences such that $m_N\sim N^{\frac{\alpha}{\alpha-1}}\, x$ and $l_N\sim
  N^{\frac{1}{2(\alpha-1)}}\, \scal^{\frac{1}{2\alpha}} z$. Then it holds that 
  \begin{equation}
    \label{eq:adjacent_biased}
N\cdot
N^{\frac{1}{2(\alpha-1)}}\cdot \widetilde{G}(m_N,l_N)
\underset{N\to\infty}{\longrightarrow}\GG(x,z)\, .
\end{equation}

Moreover, it holds that
\begin{equation}
  \label{eq:limsuploc}
  \lim_{\eta\downarrow
    0}\limsup_{m\to\infty}m^{1-1/\alpha}\widetilde{G} (\{m\}\times [-\eta m^{\frac{1}{2\alpha}}, \eta m^{\frac{1}{2\alpha}}])=0\, .
\end{equation}
\end{prop}

Proving these statements requires some preliminary notation. 
For $k\geq 1, m\geq 1$, we let $Q^*_k(m)$ be the
probability that the total number of white vertices in a sequence of
$k$ independent random mobiles with law $\mathrm{GW}_\bq$ is equal to
$m$. By the cyclic lemma, we have
\begin{equation}
  \label{eq:cycliclemmaQ}
Q^*_k(m)=\frac{k}{m}Q_m(-k)\, ,  
\end{equation}
where, recalling the notation $\mu_{\circ\bullet}$ of Section
\ref{sec:non-generic} for the law of the number of grandchildren in a
$\mathrm{GW}_\bq$-distributed mobile, we let
$Q_m=\mu_{\circ\bullet}(\cdot+1)^{*m}$. The latter can also be
interpreted as the law at time $m$ of the white Lukaciewicz path in a
forest of independent $\mathrm{GW}_\bq$-distributed mobiles. Under our hypotheses \eqref{eq:tailmu}, the local
limit theorem \eqref{eq:loc_lim} and formula (\ref{eq:qcbetazero}) imply that, as $m\to\infty$, 
\begin{equation}
  \label{eq:loclimstar}
  Q^*_1(m)\sim \frac{q^{[\alpha]}_{\scal/2^{\alpha}}(0)}{
  m^{1+1/\alpha}}
=\frac{2}{\scal^{1/\alpha}|\Gamma(-1/\alpha)|m^{1+1/\alpha}}\, .
\end{equation}
For $1\leq r\leq k$ and $l\in \Z$, we also let $P_k(l)$ be the probability that a sum of $k$
independent random variables with shifted geometric law $\sum_{n\geq
  -1}2^{-n-2}\delta_{\{n\}}$ equals $l$, and
\begin{equation}
  \label{eq:prk}
  P_r^{(k)}(l):=\frac{P_r(l)P_{k+1-r}(-l)}{P_{k+1}(0)}\, ,
\end{equation}
That is, $P_r^{(k)}$ is the law of the value at time $r$ of a bridge of a random walk with step
distribution $P_1(\cdot)$ with duration $k+1$. We define the
distribution $\bar{P}_k$ by the formula
\begin{equation}
  \label{eq:barpk}
 \bar{P}_k:=\frac{1}{k}\sum_{r=1}^kP_r^{(k)}\, ,
\end{equation}
which corresponds to the law at a uniformly random time in
$\{1,2,\ldots,k\}$ of that same random walk bridge.
Note that $\bar{P}_k$ is a centered distribution, and moreover, it has
variance $(k+1)/3$, as shown for instance in \cite[Section
3.2]{MM07}.

Finally, recall that
$\widehat{\mu}_\bullet(k)=k\mu_\bullet(k)/m_\bullet$, where
$m_\bullet=(\zbq-1)^{-1}$, 
is the size-biased law associated with $\mu_\bullet$ defined in \eqref{eq:mucirc_mubullet}. 
The key formula we will need is the following.
\begin{lem}\label{lem:adjacent:eq}
  For $k\geq  1$,  $m\geq 1$ and $l\in \Z$, one has
  \begin{equation}
    \label{eq:adjacent}
   \Xi(m,l)=\sum_{k\geq
     1}\widehat{\mu}_\bullet(k)Q^*_{k+1}(m+1)\bar{P}_k(l)\, .
    \end{equation}
  \end{lem}
  
This is an immediate application of Lemma
 \ref{sec:boltzmann-measures-1}, which implies in fact the more detailed formula
$$ \widehat{\mathrm{GW}}_\bq^{(1)}(K=k,R=r,M=m,L=l)=\frac{\widehat{\mu}_\bullet(k)}{k}Q^*_{k+1}(m+1)P_r^{(k)}(l)\,
,$$
of which (\ref{eq:adjacent}) is obtained by summing over $r$ and $k$. 

Manipulating this formula will require the following estimates. 
First, it is  an
easy consequence of the hypotheses on $\bq$ that
\begin{equation}
  \label{eq:whmubulletasymp}
  \widehat{\mu}_\bullet(k)\sim \frac{\scal}{4\zbq\Gamma(-\alpha)(2k)^{\alpha}}\,,
\end{equation}
as $k\to\infty$.

Next, we state some  estimates on $\bar{P}_{k}$. 
\begin{lem}
  \label{sec:scal-limit-adjac-4}
  (i --- Local limit theorem) Recalling \eqref{eq:defbarp}, it holds that: 
  \begin{equation}
    \label{eq:loclimpbar}
      \lim_{k\to\infty}\sup_{l\in 
    \Z}\left|\sqrt{k}\bar{P}_{k}(l)-\bar{p}_2\left(\frac{l}{\sqrt{k}}\right)\right|=0\, 
  .
\end{equation}
(ii --- Local bound) For every $\beta>0$, there exists a constant $C=C_\beta\in
(0,\infty)$ such that, for every $k\geq 0$ and $l\in \Z$, 
\begin{equation}
  \label{eq:loclimimprove}
  \sqrt{k}\bar{P}_k(l)\leq
  \frac{C}{1+\left(\frac{|l|}{\sqrt{k}}\right)^\beta}\, .
\end{equation}
\noindent
(iii --- Concentration function) There exists a constant $c>0$ such
that for every integers $j\geq 1$ and $k_1,k_2,\ldots,k_j\geq 1$,
\begin{equation}
  \label{eq:concentrationpbar}
 \sup_{l\in \Z}
 \bar{P}_{k_1}*\bar{P}_{k_2}*\cdots*\bar{P}_{k_j}(l)\leq
 \frac{c}{\sqrt{k_1+k_2+\cdots+k_j}}\, .
\end{equation}
\end{lem}

\begin{proof}
  Recall that we use the notation $\lesssim$ for upper bounds that
  hold up to some universal multiplicative constant which we do not want to
  keep track of.
  
 (i) For a given $l\in \Z$, the difference in absolute values is bounded by 
  $$\int_0^1 \mathrm{d}s \, \left|\frac{\sqrt{k}P_{\lceil ks\rceil}(l)\, 
      \sqrt{k}P_{k+1-\lceil 
        ks\rceil}(-l)}{\sqrt{k}P_{k+1}(0)}-\frac{\mathrm{e}^{-\frac{l^2}{4sk}}\mathrm{e}^{-\frac{l^2}{4(1-s)k}}}{\sqrt{4\pi s(1-s)}}
\right|\, ,$$
and the integral on $[\eps,1-\eps]$ is immediately controlled by the 
local limit theorem, since the variance of a random variable with law
$P_1$ is $2$. It remains to estimate the boundary terms, that
it, to show that the integral outside $[\eps,1-\eps]$ is uniformly
small in $l$ as $n\to\infty$, provided $\eps$ has been chosen small
enough. Let us 
look at the integral on $[0,\eps]$, the other one is dealt with 
by symmetry. First, we observe that the sequence 
$\sup_{k/2\leq r\leq k+1}\sup_{l\in \Z}P_{r}(l)/P_{k+1}(0),k\geq 1$ is
bounded, as a consequence of the local limit theorem for $P_r$. Therefore,
Assuming that $\eps\in (0,1/2)$, we can bound
ratio $P_{k+1-\lceil ks\rceil}(-l)/P_{k+1}(0)$ uniformly in $s\in [0,\eps]$ and
$l\in \Z$, and it remains to check that: 
$$\lim_{\eps\downarrow 0}\limsup_{k\to\infty}\sup_{l\in 
  \Z}\frac{1}{\sqrt{k}}\sum_{r=1}^{\epsilon k}P_r(l)=0\, .$$
However, using the local limit theorem for $P_r$ again, we have 
$\sup_{l\in \Z}P_r(l)=O(1/\sqrt{r})$, so that the last sum above is 
$O(\sqrt{\eps k})$ uniformly in $l$, which implies
(\ref{eq:loclimpbar}).

(ii) We
rely on the refined bounds on the local limit theorem stated in
\cite[Theorem VII.3.16]{petrov75}, which entail that, for every
$\beta>0$, 
$$C(\beta)=\sup_{l\in \Z,r\geq
  1}\left(1+\left(\frac{|l|}{\sqrt{r}}\right)^\beta\right)\sqrt{r}P_r(l)<\infty\, .$$
Therefore, using \eqref{eq:prk} and \eqref{eq:barpk}, it holds that for every $k\geq 0,l\in
\Z$, 
\begin{align*}
  \sqrt{k}\bar{P}_k(l)&\lesssim
\sum_{r=1}^{k}\frac{C(\beta)^2}{\sqrt{r(k+1-r)}}\cdot\frac{1}{\left(1+\left(\frac{|l|}{\sqrt{r}}\right)^\beta\right)
  \left(1+\left(\frac{|l|}{\sqrt{k+1-r}}\right)^\beta\right)}\\
&\lesssim\Big( \sum_{r=1}^{k}\frac{C(\beta)^2}{\sqrt{r(k+1-r)}}\Big) \cdot \frac{1}{\left(1+\left(\frac{|l|}{\sqrt{k}}\right)^\beta\right)} \, ,
\end{align*}
which gives the result since the sum is a converging Riemann sum. 
(iii) Observe that (i)
entails that
$$\sup_{k\geq 1}\sup_{l\in \Z}\sqrt{k}\bar{P}_k(l)<\infty\, ,$$
so that the wanted result is a direct application of
\cite[ Theorem 33.1.1]{bretagnolle04}. 
\end{proof}

We can now proceed with the proofs of Propositions
\ref{sec:scal-limit-adjac-1} and \ref{sec:scal-limit-adjac-2}.

\begin{proof}[{Proof of Proposition \ref{sec:scal-limit-adjac-1}}]
  We apply (\ref{eq:adjacent}) to $m=m_N$ and $l=l_N$, to obtain that
  $$\Xi(m_N,l_N)= \sum_{k\geq 1}\widehat{\mu}_\bullet(k)Q^*_{k+1}(m_{N}+1)\bar{P}_k(l_N)$$
  Next, fix $\eps>0$.   
    The local limit theorem and dominated
      convergence, justified by the fact that
      $N^{\frac{\alpha}{\alpha-1}}Q^*_{k+1}(\cdot)$ and
      $N^{\frac{1}{2(\alpha-1)}}\bar{P}_k(\cdot)$ are uniformly bounded
      functions for $k\geq \eps N^{\frac{1}{\alpha-1}}$, entail that
  \begin{eqnarray*} &&2\scal^{1-\frac{1}{2\alpha}}N\cdot N^{\frac{\alpha}{(\alpha-1)}}\cdot N^{\frac{1}{2(\alpha-1)}}\,
      \sum_{k> \eps
        N^{1/(\alpha-1)}}\widehat{\mu}_\bullet(k)Q^*_{k+1}(m_{N}+1)\bar{P}_k(l_N)\\ &  \xrightarrow[ \eqref{eq:ballot},\eqref{eq:cycliclemmaQ},  \eqref{eq:whmubulletasymp}, \eqref{eq:loclimpbar}]{}&
      \int_\eps^\infty\frac{
        \d t}{\Gamma(-\alpha)t^\alpha}q^{[1/\alpha]}_t(x)\bar{p}_t(z)\,
      ,  \end{eqnarray*}
      which converges to $\GG(x,z)/x$ as $\eps\to 0$. 
      It remains to show the smallness of the remainder term
      \begin{multline*}
              N\cdot N^{\frac{\alpha}{\alpha-1}}\cdot N^{\frac{1}{2(\alpha-1)}}\,
      \sum_{k\leq \eps
        N^{1/(\alpha-1)}}\widehat{\mu}_\bullet(k)Q^*_{k+1}(m_{N}+1)\bar{P}_k(l_N)\\
      =N^{\frac{4\alpha-1}{2(\alpha-1)}}\sum_{k\leq \eps
        N^{1/(\alpha-1)}}\widehat{\mu}_\bullet(k)\frac{k+1}{m_N+1}Q_{m_{N}+1}(-k-1)\bar{P}_k(l_N). 
    \end{multline*}
    By the local limit theorem, it holds that
    $m_N^{1/\alpha}Q_{m_N+1}(-k-1)$ is uniformly bounded in $k$, so we
    can bound the latter term by a constant times
    \begin{equation}
      \label{eq:intermbound}
          N^{\frac{2\alpha-3}{2(\alpha-1)}}\sum_{k\leq \eps
            N^{1/(\alpha-1)}}\frac{\bar{P}_k(l_N)}{k^{\alpha-1}}\, .
        \end{equation}
If $\alpha\in (1,3/2)$, we simply bound $\sqrt{k}\bar{P}_k(l)$ by
a constant using (i) or (ii) in Lemma \ref{sec:scal-limit-adjac-4}, yielding
an upper bound of the form of a constant times
$$N^{\frac{2\alpha-3}{2(\alpha-1)}}\sum_{k\leq
  \eps N^{\frac{1}{\alpha-1}}}k^{1/2-\alpha}
\lesssim\eps^{3/2-\alpha}\, ,$$
which goes to $0$ as $\eps\to 0$, as wanted. 
Now suppose that $\alpha\in
[3/2,2)$. If $z=0$, then
the result follows directly from Fatou's lemma and the fact that
$\GG(x,0)=\infty$. So we assume that $z\neq 0$ and use the
bound (ii) in Lemma
\ref{sec:scal-limit-adjac-4} for $\beta=1$, yielding
$\bar{P}_k(l_N)\leq 
C_1(1+|l_N|)^{-1}=O(N^{-1/2(\alpha-1)})$ (this is where we use that
$z\neq 0$), and this gives 
an upper bound for (\ref{eq:intermbound}) of the form of a constant
multiple of 
$$N^{\frac{\alpha-2}{\alpha-1}}\sum_{k\leq
  \eps N^{\frac{1}{\alpha-1}}}k^{1-\alpha}
\lesssim \eps^{2-\alpha}\, .$$
Again, this goes to $0$ as $\eps\to 0$. 
The proof of \eqref{eq:adjacent_restricted2} goes along similar lines,
but is simpler: in the error term, it suffices to bound
$\bar{P}_k((-\eta N^{1/2(\alpha-1)},\eta N^{1/2(\alpha-1)}))$ by $1$. 
\end{proof}

Next, we turn to the comparison between $\Xi$ and $\widetilde{G}$.

\begin{proof}[Proof of Lemma \ref{sec:scaling-limit-buckle-1}]
  For $m\geq 1$ and $A\subset \Z$, write
  $\widetilde{G}(\{m\}\times
  A)=\widehat{\mathrm{GW}}^{(1)}_\bq(\#C_\circ\ind_{\{M=m,L\in
    A\}})$, 
  so that
  $$ (\zbq-\eps)m\, \Xi(\{m\}\times A)+ R_m(\eps)\leq
  \widetilde{G}(\{m\}\times A)\leq (\zbq+\eps)m\, \Xi(\{m\}\times A)+
  R_m(\eps)\, ,$$
  where $R_m(\eps)=\widehat{\mathrm{GW}}^{(1)}_\bq(\#C_\circ\ind_{\{M=m,|\#C_\circ-\zbq
    m|>\eps m\}})$. 
  It remains to show that $R_m(\eps)$ is exponentially small.
  To this end, let us work under the law
  $\mathrm{GW}_\bq^{\otimes \infty}$ of an infinite forest of i.i.d.\
  trees with law $\mathrm{GW}_\bq$, and let $v_1,v_2,\ldots$ be the list of white
  vertices of this infinite forest listed in lexicographical (depth-first) order. If
  $\chi_i$ is the number of children of $v_i$ plus $1$, then
  $(\chi_i : i\geq 1)$ is an i.i.d. sequence with geometric law
  $\sum_{k\geq 1}\mu_\circ(k-1)\delta_k$, which has expectation
  $\zbq$, and exponential
  moments. Hence, for a fixed $\varepsilon>0$, Cram\'er's theorem
  implies that,  for some $C(\eps)\in (0,\infty)$, 
  \begin{equation}
    \label{eq:cramergw}
      \mathrm{GW}_\bq^{\otimes \infty}\left(\left|\sum_{i=1}^m\chi_i-\zbq m\right|>\eps
        m\right)\leq  \mathrm{e}^{-C(\eps) m}
    \end{equation}
    for every $m\geq 1$. On the other
    hand, letting $\#V_\circ[k]$ be the number of white vertices of the first $k$ trees
  in the infinite forest,  Lemma
  \ref{sec:boltzmann-measures-1} implies that $R_m(\eps)$ writes
\begin{align*}
\widehat{\mathrm{GW}}^{(1)}_\bq(\#C_\circ\ind_{\{M=m,|\#C_\circ-\zbq
    m|>\eps m\}})
  &=\sum_{k\geq 1}\widehat{\mu}_\bullet(k)\mathrm{GW}_\bq^{\otimes
    \infty}\left( \left(\sum_{i=1}^{m+1}\chi_i\right)
    \ind_{\{\#V_\circ[k+1]=m+1,\left|\sum_{i=1}^{m+1}
    \chi_i-\zbq m\right|>\eps m\}}\right)\\
  &\leq \mathrm{GW}_\bq^{\otimes
    \infty}\left( \left(\sum_{i=1}^{m+1}\chi_i\right)
    \ind_{\{\left|\sum_{i=1}^{m+1}
    \chi_i-\zbq m\right|>\eps m\}}\right)\, ,
\end{align*}
  where the second inequality is obtained by
  bounding the indicator that $\#V_\circ[k+1]=m+1$ by $1$, and using
  the fact that $\widehat{\mu}_\bullet$ is a probability distribution.
  By the Cauchy-Schwarz inequality, the latter
  quantity is bounded by 
$$\mathrm{GW}_\bq^{\otimes\infty}\left(\left(\sum_{i=1}^{m+1}\chi_i\right)^2\right)^{1/2}\mathrm{GW}_\bq^{\otimes\infty}\left(\left|\sum_{i=1}^{m+1}\chi_i-\zbq
  m\right|>\eps m\right)^{1/2}\,
,$$
which, by (\ref{eq:cramergw}), is bounded by $ \mathrm{e}^{-c(\eps)m}$ for some
$c(\eps)\in (0,C(\eps))$, as wanted.   
\end{proof}

Finally, the proof of Proposition \ref{sec:scal-limit-adjac-2}
consists in using Lemma \ref{sec:scaling-limit-buckle-1} to transfer
the estimates of Proposition \ref{sec:scal-limit-adjac-1} to $\widetilde{G}$.

\begin{proof}[Proof of Proposition \ref{sec:scal-limit-adjac-2}]
  The limit (\ref{eq:adjacent_biased}) is a direct consequence of \eqref{eq:adjacent_restricted},
  with Lemma \ref{sec:scaling-limit-buckle-1}. 
To obtain (\ref{eq:limsuploc}),  we fix $\eta\in (0,\infty]$ and write: 
$$
\widetilde{G} (\{m\}\times [-  \eta m^{\frac{1}{2\alpha}},  \eta m^{\frac{1}{2\alpha}}])
  \leq (\zbq+1)m\, \Xi(\{m\}\times [-  \eta m^{\frac{1}{2\alpha}},  \eta m^{\frac{1}{2\alpha}}])+ \mathrm{e}^{-c(1)m}.
$$
Applying (\ref{eq:adjacent_restricted2}) for $x=1$
and
$m=N^{\alpha/(\alpha-1)}$, we obtain (\ref{eq:limsuploc}) by first
letting $N\to\infty$ and then $\eta\downarrow 0$. 
\end{proof}

\subsection{Joint convergence of the belt and the buckle}\label{sec:joint-convergence}

We can now give the proof of Proposition \ref{prop:scalingbelt}, which
will consist in studying a joint convergence of the pair
$((\boldsymbol{\mathcal{T}},\hat{v}_\circ),(\boldsymbol{\mathcal{P}},\hat{v}'_\circ,c))$
under the sigma-finite measure
$$\mathrm{U}_\bq:=\widehat{\mathrm{GW}}_\bq(\d 
(\boldsymbol{\mathcal{T}},\hat{v}_\circ))
\widetilde{\mathrm{GW}}_\bq(\d 
(\boldsymbol{\mathcal{P}},\hat{v}'_\circ,c))\ind_{\{\ell_{\mathcal{T}}(\hat{v}_\circ)=-\ell_{\mathcal{P}}(\hat{v}'_\circ)\}}$$
appearing in Proposition \ref{sec:boltzm-meas-reduct}, and which, we
recall, is pushed to the measure $\tilde{w}_{ \bq}^{{2 \bullet}}/2$ by the mapping
$\Phi$. As explained at
the beginning of this 
section, we will be interested in conditioning this measure on the total number
$\#V_\circ=M(\mathcal{T})+M(\mathcal{P})$ of white vertices of the
unicyclomobile $\Phi((\boldsymbol{\mathcal{T}},\hat{v}_\circ),
(\boldsymbol{\mathcal{P}},\hat{v}'_\circ,c))$ being some large integer $n$. If $F$ is some
non-negative measurable function, we have
\begin{equation}
  \label{eq:UqGWq}
  \mathrm{U}_\bq(F(\boldsymbol{\mathcal{T}},\hat{v}_\circ)\ind_{\{\#V_\circ=n\}})
=\widehat{\mathrm{GW}}_\bq\left(\widetilde{G}(n-M,-L)F(\boldsymbol{\mathcal{T}},\hat{v}_\circ)
\right)
\end{equation}
where, as before,
$\widetilde{G}(m,l)=\widetilde{\mathrm{GW}}_\bq(M=m,L=l)$. 
Now recall the notation of Section \ref{sec:tightness}, and consider the rescaled
Lukasiewicz and label process 
of $\boldsymbol{\mathcal{T}}$:
$$S^{(n)}=\frac{2 S^{\boldsymbol{\mathcal{T}}}_{\lfloor n\cdot\rfloor}}{(\scal
  n)^{\frac{1}{\alpha}}}\,
,\qquad L^{(n)}=\frac{L^{\boldsymbol{\mathcal{T}}}_{\lfloor n\cdot\rfloor}}{(\scal
  n)^{\frac{1}{2\alpha}}}\, .$$
We also let $a$ be the rank of $\hat{v}_\circ$ in
the depth-first order of white vertices of $\mathcal{T}$, that is the
integer such that $v^\circ_{a}=\hat{v}_\circ$, and let
$\theta=\#V_\circ(\mathcal{T})$. 
\begin{prop}
  \label{sec:joint-convergence-3}
  It holds that, for every bounded uniformly continuous function $F$, 
$$
      n^{\frac{1}{2\alpha}}\mathrm{U}_\bq\left(F\left(S^{(n)},L^{(n)},\frac{a}{n},\frac{\theta}{n}\right)\ind_{\{\#V_\circ=n-2\}}\right)    
 \xrightarrow[n\to\infty]{}  \mathrm{Cst} \cdot
                                   \mathbf{N}^{\bullet}\left( F (X,Z,
                                   t_{\bullet}, \sigma) \cdot
                                   \GG\left(1-\sigma,-Z_{t_{\bullet}}\right)\right)\, . $$
\end{prop}                                 
                               
By
applying this result to $F=1$, we obtain
\begin{equation}
  n^{\frac{1}{2\alpha}}\mathrm{U}_\bq\left(\#V_\circ=n-2\right) \xrightarrow[n\to\infty]{}  \mathrm{Cst} \cdot
                                   \mathbf{N}^{\bullet}\left(
                                   \GG\left(1-\sigma,-Z_{t_{\bullet}}\right)\right)\, , 
                                 \label{eq:asymptoUn}
\end{equation}
which allows to obtain the conditioned result stated in Proposition
\ref{prop:scalingbelt} by dividing. 
To prove Proposition \ref{sec:joint-convergence-3}, we rewrite the
left-hand side of the displayed expression of the statement by using
(\ref{eq:UqGWq}), as
\begin{align}
  \label{eq:beltscalingdisplay}
&n^{\frac{1}{2\alpha}} \widehat{\mathrm{GW}}_\bq\left(\widetilde{G}(n-M,-L)F\left(S^{(n)},L^{(n)},
    \frac{a}{n},\frac{M}{n}\right)\right) \\
  &  =
  n^{1+\frac{1}{2\alpha}}\int_0^1\d  x  \, \widehat{\mathrm{GW}}_\bq(M=\lceil
nx\rceil)\, \widehat{\mathrm{GW}}_\bq\left(\widetilde{G}(n-\lceil
  nx\rceil,-L^{(n)}_{a_n/n})F\left(S^{(n)},L^{(n)},
    \frac{a_n}{n},\frac{\lceil nx\rceil}{n}\right)
  \, \Big|\, M=\lceil nx\rceil\right)\, .\nonumber
\end{align}
Note that, by the definition of $\widehat{\mathrm{GW}}_\bq$ just
before \eqref{def:gwtilde}, it holds that $\widehat{\mathrm{GW}}_\bq(M=n)=\sum_{y\geq
  0}\Xi^{*y}(\{n\}\times \Z)$. Since $n^{2-1/\alpha}\, \Xi(\{n\}\times \Z)$ converges as $n\to\infty$ by \eqref{eq:adjacent_restricted2}, a standard
renewal theorem implies the existence of a
constant $c\in (0,\infty)$ such that $\widehat{\mathrm{GW}}_\bq(M=\lceil
nx\rceil)\sim c\lceil
nx\rceil^{-1/\alpha}$, with a uniformly bounded error over values of
$x$ in a compact subset of $(0,1]$.
The next lemma says that, in a sense, it is almost impossible to
distinguish a large white leaf-pointed random mobile with law
$\widehat{\mathrm{GW}}_\bq$ from a large random mobile with law
$\mathrm{GW}_\bq$ of same size, marked by a uniformly chosen white
leaf.

\begin{lem}
  \label{sec:joint-convergence-2}
  The total variation distance between the laws
  $$\widehat{\mathrm{GW}}_\bq(\d (\boldsymbol{\mathcal{T}},\hat{v}_\circ)\,
  |\, M=n)\qquad
  \mbox{ and }\qquad \mathrm{GW}_\bq(\d \boldsymbol{\mathcal{T}}\,
    |\, \#V_\circ=n+1)\frac{\sum_{v\in
        \widehat{V}_\circ(\mathcal{T})}\delta_v(\d \hat{v}_\circ)}{\#\widehat{V}_\circ(\mathcal{T})}$$
    converges to $0$ as $n\to\infty$.
    Moreover, if we let $N_\circ(k)$ be such that $v^\circ_{N_\circ(k)}$ is
    the $k$-th white leaf appearing in the list
    $v_1^\circ,v_2^\circ,\ldots,v_{\#V_\circ}^\circ$ of white vertices of $\mathcal{T}$, then, under $\mathrm{GW}_\bq(\cdot\,
    |\, \#V_\circ=n)$, we have
    $$\Big(\frac{N_{\circ}(\lfloor \#\widehat{V}_\circ t
      \rfloor)}{n} : 0\leq t\leq 1\Big)
  \underset{n\to\infty}{\longrightarrow} \big(t : 0\leq t\leq 1\big)\, ,$$
in probability for the uniform norm. 
\end{lem}

\begin{proof}
Both statements are a consequence of exponential concentration of the
number of white leaves in random mobiles, analog to that discussed in
the proof of Proposition \ref{sec:scal-limit-adjac-2}. Recall that, in
a $\mathrm{GW}_\bq^{\otimes\infty}$-distributed random forest, we let
$\chi_i-1$ be the number of children of the $i$-th white vertex
$v_i^\circ$ in depth-first order of exploration. This expresses the
fact that $v_i^\circ$ is a white leaf if and only if
$\zeta_i:=\ind_{\{\chi_i-1=0\}}=1$, the latter being i.i.d.\ Bernoulli random
variables with expectation $1/\zbq$ by
\eqref{eq:mucirc_mubullet}. Then, recalling that $\#V_\circ[1]$ is the number
of white vertices of the first tree in the above infinite forest, we have 
$$\mathrm{GW}_\bq\big(|\#\widehat{V}_\circ-n/\zbq\big|>\eps n\, |\, \#V_\circ=n\big)=\frac{
\mathrm{GW}_\bq^{\otimes
  \infty}\Big(\#V_\circ[1]=n,\Big|\sum_{i=1}^{n}\zeta_i-n/\zbq\Big|>\eps
  n\Big)}{\mathrm{GW}_\bq(\#V_\circ=n)}\, ,$$
which decays exponentially fast in $n$ for every fixed $\eps>0$ by Cram\'er's theorem and the fact
that $\mathrm{GW}_\bq(\#V_\circ=n)=Q^*_1(n)$ has a power-law decay, by
(\ref{eq:loclimstar}).
Using a simple union bound, we even obtain that the
process $(n^{-1}\sum_{i=1}^{\lfloor nt\rfloor}\zeta_i : 0\leq t\leq 1)$
under $\mathrm{GW}_\bq(\cdot\, |\, \#V_\circ=n)$ is exponentially
concentrated around $(t\#\wh{V}_\circ/n : 0\leq t\leq 1)$, which easily implies
the second statement since $N_\circ(\lfloor \cdot\rfloor)$ can be seen
as a right-continuous inverse of this process. We leave details to the reader. 

To prove the first statement, we express the total variation distance
between the two measures as
\begin{multline}
  \label{eq:totalvarlem}
\sum_{(\boldsymbol{\mathcal{T}},\hat{v}_\circ):M(\mathcal{T})=n}\left|\widehat{\mathrm{GW}}_\bq(\{(\boldsymbol{\mathcal{T}},\hat{v}_\circ)\}\,
  |\, M=n)-\frac{\mathrm{GW}_\bq(\{\boldsymbol{\mathcal{T}}\}\, |\,
    \#V_\circ=n+1)}{\#\widehat{V}_\circ(\boldsymbol{\mathcal{T}})}\right|
\\=\sum_{(\boldsymbol{\mathcal{T}},\hat{v}_\circ):M(\mathcal{T})=n}\left|\frac{\mathrm{GW}_\bq(\{\boldsymbol{\mathcal{T}}\})}{\widehat{\mathrm{GW}}_\bq( M=n)}-\frac{\mathrm{GW}_\bq(\{\boldsymbol{\mathcal{T}}\}\, |\,
  \#V_\circ=n+1)}{\#\widehat{V}_\circ(\boldsymbol{\mathcal{T}}) \mathrm{GW}_\bq( \#V_\circ=n+1)}\right|
\\ =
\mathrm{GW}_\bq\left(\frac{\ind_{\{M=n\}}}{\mathrm{GW}_\bq(\#V_\circ=n+1)}\left|\frac{\#\widehat{V}_\circ\mathrm{GW}_\bq(\#V_\circ=n+1)}{\widehat{\mathrm{GW}}_\bq(\#V_\circ=n+1)}-1\right|\right)\, .
\end{multline}
Now, since
$\widehat{\mathrm{GW}}_\bq(M=n)=\mathrm{GW}_\bq(\#\widehat{V}_\circ\ind_{\{\#V_\circ=n+1\}})$,
we have, for every $\eps>0$, 
$$\left|\widehat{\mathrm{GW}}_\bq(M=n)-\frac{n}{\zbq}\mathrm{GW}_\bq(\#V_\circ=n+1)\right|\leq
\eps
n\, \mathrm{GW}_\bq(\#V_\circ=n+1)+\mathrm{GW}_\bq\left(\#\widehat{V}_\circ\ind_{\{\#V_\circ=n+1,|\#\widehat{V}_\circ-n/\zbq|>\eps
    n\}}\right)\, ,$$
where
$$\mathrm{GW}_\bq\left(\#\widehat{V}_\circ\ind_{\{\#V_\circ=n+1,|\#\widehat{V}_\circ-n/\zbq|>\eps
    n\}}\right)\leq \mathrm{GW}_\bq^{\otimes \infty}\Big(\Big(\sum_{i=1}^{n+1}\zeta_i\Big)^2\Big)^{\frac{1}{2}} \cdot \mathrm{GW}_\bq^{\otimes
  \infty}\Big(\Big|\sum_{i=1}^n\zeta_i-n/\zbq\Big|>\eps
  n\Big)^{\frac{1}{2}}$$
has exponential decay in $n$. Therefore, using again that
$\mathrm{GW}_\bq(\#V_\circ=n+1)$ has a power-law decay, it holds that 
$\mathrm{GW}_\bq(\#V_\circ=n+1)/\widehat{\mathrm{GW}}_\bq(M=n)\sim \zbq/n$ as
$n\to\infty$. Plugging this back into \eqref{eq:totalvarlem}, and
applying a similar reasoning as above, distinguishing whether $\zbq \#\widehat{V}_\circ/n$ is at a distance greater than $\eps$ from $1$ or not, and using the Cauchy-Schwarz inequality, we obtain the first wanted result.
\end{proof}

Let us now fix $\eps,\eta\in (0,1/2)$ and consider the integral expression (\ref{eq:beltscalingdisplay})
restricted to values of $x$ in $[\eps,1-\eps]$, and to values of
$L^{(n)}_{a_n/n}$ in $[-\eta,\eta]^{\mathrm{c}}$. 
Then, by the first statement of the previous lemma, this integral
expression is equivalent to
$$cn^{1+\frac{1}{2\alpha}}\int_\eps^{1-\eps}\frac{ \mathrm{d}x}{(nx)^{\frac{1}{\alpha}}}\mathrm{GW}_\bq\left(\frac{1}{\#\wh{V}_\circ}\sum_{k=1}^{\#\wh{V}_\circ}\widetilde{G}(n-\lceil nx\rceil,-L^{(n)}_{N_\circ(k)/n})F\Big(S^{(n)},L^{(n)},\frac{N_\circ(k)}{n},\frac{\lceil nx\rceil}{n}\Big)\,
\Big|\, M=\lceil nx\rceil 
\right)\,
.$$ 
By the second statement of Lemma \ref{sec:joint-convergence-2}, together with 
\eqref{eq:second_couple} and Proposition \ref{sec:scal-limit-adjac-2} for $1-x$ instead
of $x$ and $N=n^{(\alpha-1)/\alpha}$, this expression has an equivalent of the form
$$ c \int_\eps^{1-\eps}\frac{\d  x}{x^{\frac{1}{\alpha}}}\int_{\eps}^x\frac{\d 
a}{x}\mathbf{N}^{(x)}\left(\GG\left(1-x,-Z_a\right) F\left(
  X,Z,a\right)\ind_{\{|Z_a|>\eta\}}\right)\, ,  $$
where $c\in (0,\infty)$ is some constant, and $\mathbf{N}^{(x)}$ is the probability measure introduced above \eqref{eq:excursionmeasuredecomp} under which $(X,Z)$ is the process with total
duration $x$. By \eqref{eq:excursionmeasuredecomp} and \eqref{def:N:bullet}, we may rewrite the previous expression as
$$ c \mathbf{N}^\bullet\Big(\GG(1-\sigma, -
Z_{t_\bullet}) F\big(X,Z,t_{\bullet}\big)\ind_{\{t_\bullet\in [\eps,1-\eps],|Z_{t_\bullet}|>\eta\}}\Big)\,
,$$
and as we let $\eta\downarrow 0$ and then $\eps\downarrow
0$, the above integral converges to the limiting expression appearing in Proposition
\ref{prop:scalingbelt}. In order to conclude, we need to show that the
remainder terms are asymptotically negligible, which is the object of
the next lemma.

\begin{lem}
  \label{sec:joint-convergence-1}
  For every $\eps\in (0,1/2)$,  it holds that
  \begin{equation}
    \label{eq:controllabel}
    \lim_{\eta\downarrow
  0}\limsup_{n\to\infty}n^{\frac{1}{2\alpha}}\mathrm{U}_\bq\left(M(\mathcal{T})\wedge
  M(\mathcal{P})>\eps n, \#V_\circ=n,|L|\leq \eta
  n^{\frac{1}{2\alpha}}\right)=0\, .
\end{equation}
Moreover, one has
\begin{equation}
  \label{eq:controlmass}
\lim_{\eps\downarrow
    0}\limsup_{n\to\infty}n^{\frac{1}{2\alpha}}\, \mathrm{U}_\bq(M(\mathcal{T})\wedge
  M(\mathcal{P})\leq \eps n\, ,\,
  \#V_\circ=n)=0\, .
\end{equation}
\end{lem}

\begin{proof}
In this proof, we introduce the quantity
\begin{equation}
  \label{eq:gml}
  G(m,l)=\widehat{\mathrm{GW}}_\bq(M=m,L=l)\, , 
\end{equation}
which we view as a measure on $\N\times \Z$. Note that it 
is the Green function 
of the random walk whose step distribution is $\Xi$, that is,
$G=\sum_{y\geq 0}\Xi^{*y}$. Our proofs will all rely on expansions of
expressions of the form
\begin{equation}
  \label{eq:formcano}
  n^{\frac{1}{2\alpha}}\sum_{m,l}G(m,l)\widetilde{G}(n-m,-l)\, ,
\end{equation}
where the sum over $m$ and $l$ are over certain subsets of $\N$ and
$\Z$. 
On the one hand, we can express $\widetilde{G}(n-m,-l)$ by first
expressing it in terms of $\Xi$ using Lemma
\ref{sec:scaling-limit-buckle-1}, and then expanding it by using
(\ref{eq:adjacent}). On the other hand, $G(m,l)$
can also be expanded using (\ref{eq:adjacent}), yielding
\begin{align}
 G(m,l)&=\sum_{\substack{y\geq 0\\k\geq 1\\r\geq 1}}\sum_{\substack{k_1+\cdots+k_y=k\\ m_1+\cdots+m_y=m\\l_1+\cdots+l_y=l}}
  \widehat{\mu}_\bullet(k_1)\cdots\widehat{\mu}_\bullet(k_y) 
  \bar{P}_{k_1}(l_1)\cdots\bar{P}_{k_y}(l_y)Q^*_{k_1+1}(m_1+1)\cdots Q^*_{k_y+1}(m_y+1)\nonumber\\
   &= \sum_{\substack{y\geq 0\\k\geq 1\\r\geq
  1}}Q^*_{k+y}(m+y)\sum_{k_1+\cdots+k_y=k} \widehat{\mu}_\bullet(k_1)\cdots\widehat{\mu}_\bullet(k_y) \left(\bar{P}_{k_1}*\cdots*\bar{P}_{k_y}\right)(l)
  \, .
  \label{eq:expansionG}
\end{align}
Such expansions will be especially useful in conjunction with the
concentration bound \eqref{eq:concentrationpbar}, which allows one to
get rid of the sum over $k_1,\ldots,k_y$, and of the dependence on $l$. Let us demonstrate this by proving \eqref{eq:controllabel}. 
For a given $\eps\in (0,1/2)$ and $\eta>0$, we write the quantity of
interest in the form \eqref{eq:formcano}: 
  \begin{equation}
n^{\frac{1}{2\alpha}}\mathrm{U}_\bq\left(M(\mathcal{T})\wedge
  M(\mathcal{P})>\eps n, \#V_\circ=n,|L|\leq \eta
    n^{\frac{1}{2\alpha}}\right)
    =n^{\frac{1}{2\alpha}}\sum_{\eps n\leq m\leq
      (1-\eps)n}\sum_{|l|\leq \eta
      n^{\frac{1}{2\alpha}}}G(n-m,l)\widetilde{G}(m,-l)\, .
\label{eq:firstcaseG}
  \end{equation}
Next, we expand $G$ as in \eqref{eq:expansionG}, and use the bound
\eqref{eq:concentrationpbar}, showing
\begin{equation}
  \label{eq:expansionbound}
G(m,l)\lesssim \sum_{\substack{y\geq 0\\k\geq 1}}\frac{Q^*_{k+y}(m+y) \widehat{\mu}_\bullet^{*y}(k)}{\sqrt{k}}
\, .
\end{equation}
  This yields that the quantity of
\eqref{eq:firstcaseG} is
\begin{align}
  &\lesssim n^{\frac{1}{2\alpha}}\sum_{\eps n\leq m\leq
  (1-\eps)n}\widetilde{G}(\{n-m\}\times [-\eta
n^{\frac{1}{2\alpha}},\eta n^{\frac{1}{2\alpha}}])
\sum_{\substack{y\geq 0\\k\geq 1}} \frac{Q^*_{k+y}(m+y)
  \widehat{\mu}_\bullet^{*y}(k)}{\sqrt{k}}
\label{eq:intermQstar}\\
&\lesssim \left(n^{1-\frac{1}{\alpha}}\sup_{\eps n\leq m\leq
  (1-\eps)n}\widetilde{G}(\{m\}\times [-\eta
n^{\frac{1}{2\alpha}},\eta n^{\frac{1}{2\alpha}}])\right) \cdot n^{\frac{3}{2\alpha}-1}
\sum_{\substack{y\geq 0\\k\geq 1}} \frac{Q^*_{k+y}([y+\eps
  n,y+(1-\eps)n]) \widehat{\mu}_\bullet^{*y}(k)}{\sqrt{k}}\, ,\nonumber
\end{align}
and we observe that the term in brackets involving $\widetilde{G}$ is
uniformly bounded in $n$ by a function of $\eps$ and $\eta$ that
vanishes as $\eta\to 0$ for $\eps$ fixed, as a consequence
of \eqref{eq:limsuploc}. Hence, to conclude, it suffices to show that the rest of
the expression is uniformly bounded. 
To this end, we aim at getting rid of the manifestly superfluous presence of $y$ in
the term $Q^*$. To achieve this, we use the left-tail bound
\eqref{eq:lefttailboundbetasmall}, yielding (note that the possible values of
$y$ are trivially less than or equal to $n$)
\begin{equation}\label{eq:stretchedQstar}
  Q^*_{k+y}([y+\eps
  n,y+(1-\eps)n])\lesssim
  \exp\left(-C\frac{k+y}{(n+y)^{1/\alpha}}\right)\lesssim
  \exp\left(-C\left(\frac{k}{n^{1/\alpha}}\wedge n^{1-1/\alpha}\right)\right)
    ,
\end{equation}  
  as well as
  $$\widehat{\mu}_\bullet^{*y}(k)\lesssim
  \exp\left(-C\frac{y}{k^{\alpha-1}}\right)\, , $$
for some universal $C\in (0,\infty)$. These two facts together easily
entail that the contribution of the values of $y\geq
n^{1-\frac{1}{\alpha}+\lambda}$ to the above sum are
stretched-exponentially decaying for every $\lambda\in
(0,(2/\alpha-1)\wedge (1/2\alpha))$. Since the rest
of the terms behave as powers of $n$, we can restrict without loss of
generality our attention to values $y\leq
n^{1-\frac{1}{\alpha}+\lambda}$.
If we further assume that $k\geq \delta n^{1/\alpha}$, then the local
limit theorem \eqref{eq:loc_lim} entails that
$$n^{\frac{3}{2\alpha}-1}\cdot
\sum_{\substack{y\leq n^{1-\frac{1}{\alpha}+\lambda}\\k\geq \delta n^{\frac{1}{\alpha}}}} \frac{Q^*_{k+y}([y+\eps
  n,y+(1-\eps)n])
  \widehat{\mu}_\bullet^{*y}(k)}{\sqrt{k}}\lesssim
\int_\delta^\infty t^{\alpha-\frac{5}{2}}\, \d t\, \int_\eps^{1-\eps}q_{2t/\scal^{1/\alpha}}^{[1/\alpha]}(x)\d x$$
where the dominated convergence is justified by a use of the
stretched-exponential bound \eqref{eq:stretchedQstar}, and where we
have used the fact that $\sum_{y\geq
  0}\widehat{\mu}_\bullet(k)\lesssim k^{\alpha-2}$ by virtue of the
renewal theorem. 
In turn, this integral converges to a finite value as $\delta\to
0$, due to the fact that $q^{[1/\alpha]}_{ct}(x)=(ct/x)q^{[\alpha]}_x(-t)$. 
Hence, to conclude, it remains to show that the remaining terms $k\leq
\delta n^{1/\alpha}$ have a negligible role. 
Resuming from the expression \eqref{eq:intermQstar} and re-expressing
$Q^*_{k+y}(m+y)$ by the cyclic lemma \eqref{eq:cycliclemmaQ}, we
obtain that this term, restricted to the said values of $y$ (so that
in particular, one has $m+y\leq n$ for large enough $
n$), is
\begin{equation}
  \label{eq:usefullater}
  \lesssim R(\eta)n^{\frac{3}{2\alpha}-1}\max_{\eps n\leq m\leq
  n}\sum_{\substack{y\leq n^{1-\frac{1}{\alpha}+\lambda}\\k\leq \delta
    n^{\frac{1}{\alpha}}}}\frac{(k+y)\widehat{\mu}_\bullet^{*y}(k)Q_m(-k-y)}{\sqrt{k}}\,
.
\end{equation}
Finally, an application of the local limit theorem shows that
$Q_m(-k-y)\lesssim m^{-1/\alpha}$, uniformly in $k,y$. Putting things
together, we obtain a bound of the form 
$$C(\eps) R(\eta)
n^{\frac{1}{2\alpha}-1}\sum_{k\leq
  \delta
  n^{1/\alpha}}\left(k^{\alpha-\frac{3}{2}}+n^{1-\frac{1}{\alpha}+\lambda}k^{\alpha-\frac{5}{2}}\right)
\lesssim C(\eps)
R(\eta)\left(\delta^{\alpha-1/2}+n^{-\frac{1}{2\alpha}+\lambda}\sum_{k\leq
  \delta n^{\frac{1}{\alpha}}}k^{\alpha-\frac{5}{2}}\right)\, ,  $$
for some finite constant $C(\eps)$ depending only on $\eps$. The
second term is $\lesssim n^{\lambda-1/2\alpha}$ if $\alpha<3/2$, and
is $\lesssim \delta^{\alpha-3/2}n^{1-2/\alpha+\lambda}$ if $\alpha\geq
3/2$ (with an extra logarithmic factor if $\alpha=3/2$), so that it
vanishes as $n\to\infty$. Since the first term
vanishes as $\delta\to 0$, this concludes the proof of
\eqref{eq:controllabel}.

The proof of \eqref{eq:controlmass} is separated in two parts,
depending on which of
  $M(\mathcal{T})$ or $M(\mathcal{P})$ is smaller than $\eps n$, and
  follows the same general lines as above. 

  We start with the case where $M(\mathcal{T})\leq \eps n$.  The relevant form is now 
 $$ n^{\frac{1}{2\alpha}}   \mathrm{U}_\bq(
  M(\mathcal{T})\leq \eps n\, ,\,
  \#V_\circ=n)
  =n^{\frac{1}{2\alpha}}\sum_{m=0}^{\eps n}\sum_{l\in
    \Z}G(m,l)\widetilde{G}(n-m,-l)\, .$$
  To estimate the latter, we use Lemma
  \ref{sec:scaling-limit-buckle-1}, which yields, for $0\leq m\leq n/2$, 
  \begin{equation}
    \label{eq:prelimggtilde}
      \sum_{l\in
    \Z}G(m,l)\widetilde{G}(n-m,-l)\lesssim n\sum_{l\in
    \Z}\Xi(n-m,-l)G(m,l)+n \mathrm{e}^{-c(1)n/2}G(\{m\}\times \Z)\, ,
    \end{equation}
  and the remainder term has exponential decay as $n\to\infty$
  uniformly in $0\leq m\leq n/2$ since $G(\{m\}\times
  \Z)=\widehat{\mathrm{GW}}_\bq(M=m)$ is of order $m^{-1/\alpha}$, as
  we already observed. 
Next, we express $\Xi(n-m,-l)$ by (\ref{eq:adjacent}) to get 
$$\sum_{l\in
    \Z}\Xi(n-m,-l)G(m,l)=n\sum_{l\in \Z}\sum_{r\geq
    1}\widehat{\mu}_\bullet(r)Q^*_{r+1}(n-m+1)G(m,l)\, ,$$
and then we expand the Green function $G(m,l)$ as in
\eqref{eq:expansionG}, 
showing
that the last displayed quantity equals 
\begin{align*}
 \sum_{\substack{y\geq 0\\k\geq 1\\r\geq 1}}\widehat{\mu}_\bullet(r)Q^*_{r+1}(n-m+1) Q^*_{k+y}(m+y) \sum_{\substack{k_1+\cdots+k_y=k\\ k_1,\ldots,k_y\geq 1}}
  \widehat{\mu}_\bullet(k_1)\cdots\widehat{\mu}_\bullet(k_y) 
  \left(\bar{P}_{k_1}*\cdots*\bar{P}_{k_y}*\bar{P}_r\right)(0)\\
   \lesssim \frac{1}{n}\sum_{y\geq 0}\sum_{k\geq 0,r\geq 
  1}\frac{(r+1)\widehat{\mu}_\bullet(r)\widehat{\mu}_\bullet^{*y}(k)}{\sqrt{r+k}}Q^*_{k+y}(m+y)  Q_{n-m+1}(-r-1).
\end{align*}
Here we have used (iii) in Lemma
\ref{sec:scal-limit-adjac-4} to control the convolution term, and (\ref{eq:cycliclemmaQ}) to re-express $Q^*_{r+1}(n-m+1)$.
Summing this over $0\leq
m\leq \eps n$ and using (\ref{eq:prelimggtilde}), this yields  
\begin{multline}\label{eq:smallbelt}  n^{\frac{1}{2\alpha}}\mathrm{U}_\bq(
  M(\mathcal{T})\leq \eps n\, ,\,
  \#V_\circ=n)
  \\ \lesssim  n^{\frac{1}{2\alpha}}
 \sum_{\substack{y\geq 0\\k\geq 0\\r\geq
     1}}\frac{(r+1)\widehat{\mu}_\bullet(r)\widehat{\mu}_\bullet^{*y}(k)}{\sqrt{r+k}}\sum_{0\leq
   m\leq \eps n}Q^*_{k+y}(y+m)Q_{n-m+1}(-r-1) + R_n\, ,
\end{multline}
where $R_n$ has exponential decay, and can be removed from further discussions. 
Moreover, similar argument as around (\ref{eq:stretchedQstar}) show that we can restrict our
attention to values of $y$ at most $n^{1-1/\alpha+\lambda}$, which we
now assume. Here, one has to pay attention to values of $k,r$ that
significantly deviate from the typical magnitude $n^{1/\alpha}$. To
fix the ideas, we consider $\delta>0$ and consider the situation where
$k\leq \delta n^{1/\alpha}$ and $r\geq n^{1/\alpha}$. 
To estimate this last term, we write
\begin{align}
  n^{\frac{1}{2\alpha}}\sum_{\substack{y\leq
  n^{1-\frac{1}{\alpha}+\lambda}\\k\leq  \delta n^{\frac{1}{\alpha}}\\r\geq
      n^{\frac{1}{\alpha}}}}\widehat{\mu}_\bullet^{*y}(k)\sum_{0\leq
   m\leq \eps
  n}Q^*_{k+y}(y+m)\frac{(r+1)\widehat{\mu}_\bullet(r)}{\sqrt{r+k}}Q_{n-m+1}(-r-1)\nonumber
\\
  \lesssim n^{\frac{1}{2\alpha}}\sum_{\substack{y\leq
  n^{1-\frac{1}{\alpha}+\lambda}\\k\leq  \delta n^{\frac{1}{\alpha}}}}\widehat{\mu}_\bullet^{*y}(k)\sum_{0\leq
   m\leq \eps
  n}Q^*_{k+y}(y+m)\sum_{r\geq
     n^{\frac{1}{\alpha}}}\frac{r^{1-\alpha}}{\sqrt{r+k}}Q_{n-m+1}(-r-1)\label{eq:intermscales}
\end{align}
and we smoothen the terms of this sum over $r$ by splitting it in
scales, rewriting it as
$$  \sum_{w=1}^\infty\sum_{r=2^{w-1}
  n^{\frac{1}{\alpha}}}^{2^w 
  n^{\frac{1}{\alpha}}-1}\frac{r^{1-\alpha}}{\sqrt{r+k}}
Q_{n-m+1}(-r-1)
\lesssim  \sum_{w=1}^\infty\frac{\left(2^w 
      n^{\frac{1}{\alpha}}\right)^{1-\alpha}}{\sqrt{2^w
      n^{\frac{1}{\alpha}}+k}}
Q_{n-m+1}(-1-[2^{w-1} n^{\frac{1}{\alpha}}, 2^w 
n^{\frac{1}{\alpha}}-1])\, .
$$
This allows to use the left-tail bound
\eqref{eq:lefttailboundbetabig}, uniformly over the value of $n-m\geq (1-\eps)n$, to
bound the last term by $\exp(-c2^{w-1})$, for some finite constant
$c$. Reverting the steps, we obtain a bound of the form
$$\lesssim n^{-\frac{1}{\alpha}}\sum_{r\geq 
  n^{\frac{1}{\alpha}}}\frac{r^{1-\alpha}}{\sqrt{r+k}}
\exp\left(-c\frac{r}{n^{1/\alpha}}\right)\, ,$$
possibly for some other universal constant $c$. 
This quantity does not depend on $m$ anymore, and so we may estimate
\eqref{eq:intermscales} by
$$\lesssim n^{-\frac{1}{2\alpha}}\sum_{k\leq  \delta
  n^{\frac{1}{\alpha}}}\sum_{y\geq 0}\widehat{\mu}_\bullet^{*y}(k)
\sum_{r\geq
  n^{\frac{1}{\alpha}}}\frac{\exp(-cr/n^{1/\alpha})}{r^{\alpha-1}\sqrt{r+k}}
\lesssim \int_0^{\delta}\d t\int_1^\infty \d s\frac{\exp(-cs)}{s^{\alpha-1}t^{2-\alpha}\sqrt{s+t}}\, .$$
The latter integral is finite for every $\delta>0$, as shown by
a polar change of coordinates. 
Dealing with the situation where $k\leq \delta n^{1/\alpha}$ and
$r\leq n^{1/\alpha}$ is easier, as one can simply use the bound
$Q_{n-m+1}(-r-1)\lesssim n^{-1/\alpha}$ provided by the local limit theorem,
which is uniform in $m\leq \eps n$ and $r$. 
The contribution of the
remaining terms,
when $k\geq \delta n^{1/\alpha}$, can then be bounded
similarly, using, as a final extra input, the fact that
$Q^*_{k+y}([y,y+\eps n])\lesssim Q^*_k([0,\eps n])$ for $k\geq \delta
n^{1/\alpha}$ and $y\leq n^{1-1/\alpha+\lambda}$ (choosing
$0<\lambda<2/\alpha-1$), yielding the estimate
$$
\lesssim n^{-\frac{1}{2\alpha}}\sum_{k\leq  \delta
  n^{\frac{1}{\alpha}}}
\sum_{r\geq 1}Q^*_k([0,\eps
n])\frac{\exp(-cr/n^{1/\alpha})}{k^{2-\alpha}r^{\alpha-1}\sqrt{r+k}}
\lesssim
\int_\delta^\infty \d t\int_0^\infty \d s
\frac{\exp(-c s)\int_0^\eps
  \, q^{[1/\alpha]}_{2t/\scal^{1/\alpha}}(x)\, \d x}{s^{\alpha-1}t^{2-\alpha}\sqrt{s+t}}\, .$$
This is a converging integral, as can again be checked by
changing to polar coordinates, and this goes to $0$ as
$\eps\downarrow 0$, as wanted.

The case where the buckle part $M(\mathcal{P})\leq \eps n$ is small in
\eqref{eq:controlmass} is dealt with by similar
methods. We start from
  \begin{align*}
   n^{\frac{1}{2\alpha}}   \mathrm{U}_\bq(
  M(\mathcal{P})\leq \eps n\, ,\,
    \#V_\circ=n)&\lesssim n^{\frac{1}{2\alpha}}\sum_{m=0}^{\eps
    n}\sum_{l\in\Z}G(n-m,l)\widetilde{G}(m,-l)\\
    &\lesssim n^{\frac{1}{2\alpha}}\sum_{m=0}^{\eps n}\sum_{l\in
    \Z}m\Xi(m,-l)G(n-m,l)+n^{\frac{1}{2\alpha}}\sum_{m=0}^{\eps
    n}e^{-c(1)m}\sup_{0\leq m\leq \eps n}G(\{n\}\times \Z)\, .
  \end{align*}
  Here, we have used Lemma \ref{sec:scaling-limit-buckle-1}. We
  already observed several times that $G(\{n\}\times \Z)\lesssim
  n^{-1/\alpha}$, and therefore the second term is negligible and can
  be omitted from the discussion. We then expand the main term as we
  are now used to, similarly to the previous discussion. This yields
  an estimate 
  $$n^{\frac{1}{2\alpha}}\sum_{m=0}^{\eps n}m \sum_{\substack{y\geq
      0\\k\geq 1\\r\geq
      1}}\frac{\widehat{\mu}_\bullet(r)\widehat{\mu}_\bullet^{*y}(k)}{\sqrt{k+r}}Q^*_{r+1}(m+1)Q^*_{k+y}(n-m+y)\, .$$
Now, more than before, we will make use of cyclic lemmas to adapt to
the values of $k,r$ under study. As a preliminary remark, the same
argument as around (\ref{eq:stretchedQstar}) show that we can restrict our attention to $y\leq
n^{1-\frac{1}{\alpha}+\lambda}$. We will now understand the behavior
of the above estimate depending on whether $k$ and $r$ are small (say
$\leq n^{1/\alpha}$) or large ($>n^{1/\alpha}$). The contribution of
both $k,r$ small is obtained by applying the cyclic lemma to the terms
$Q^*_{r+1}(m+1)$ and $Q^*_{k+y}(n-m+y)$, giving the
estimate
$$n^{\frac{1}{2\alpha}-1}\sum_{m=0}^{\eps n} \sum_{\substack{y\leq n^{1-1/\alpha+\lambda}
      \\k\leq n^{1/\alpha}\\r\leq
      n^{1/\alpha}}}\frac{r\widehat{\mu}_\bullet(r)(k+y)\widehat{\mu}_\bullet^{*y}(k)}{\sqrt{k+r}}Q_{m+1}(-r-1)Q_{n-m+y}(-k-y)\, .$$
In this expression, we claim that we can get rid of the additive
factor $y$ in the quantity $(k+y)$ numerator: this is proved by
a similar argument to the one following \eqref{eq:usefullater}, and is left to the reader. Then, we simply use
the local limit theorem in the form $Q_{m+1}(-r-1)\lesssim
m^{-1/\alpha}$ and $Q_{n-m+y}(-k-y)\lesssim n^{-1/\alpha}$, allowing
to sum $\widehat{\mu}_\bullet^{*y}(k)$ over values of $y$, yielding  a bound
$$ n^{-\frac{1}{2\alpha}-1}\sum_{r,k\leq
  n^{1/\alpha}}\frac{k^{\alpha-1}}{r^{\alpha-1}\sqrt{k+r}}\sum_{m=0}^{\eps
n}m^{-1/\alpha}\lesssim \eps^{1-\frac{1}{\alpha}}\int_0^1 \d
t\int_0^1 \frac{t^{\alpha-1}\d s}{s^{\alpha-1}\sqrt{s+t}}\, ,$$
a converging integral. We see that this vanishes as $\eps\to 0$.
Let us now consider the contribution of small $k$ and large $r$. This time, we
only apply the cyclic lemma on the term $Q^*_{k+y}(n-m+y)$, yielding
the estimate
$$n^{-\frac{1}{2\alpha}-1}\sum_{\substack{k\leq n^{1/\alpha}\\r\geq
    n^{1/\alpha}}}\frac{k^{\alpha-1}}{r^\alpha\sqrt{k+r}}\sum_{m=0}^{\eps
n}mQ^*_{r+1}(m+1)\, .$$
The contribution of the last sum is at most $\eps nQ^*_{r+1}([0,\eps
n])$ by using an Abel transform, and this is $\lesssim \eps
n\exp(-cr/(\eps n)^{1/\alpha})$ by the left-tail bound
\eqref{eq:lefttailboundbetasmall}. This justifies the dominated
convergence in the resulting estimate
$$\eps n^{-\frac{1}{2\alpha}}\sum_{\substack{k\leq n^{1/\alpha}\\r\geq
  n^{1/\alpha}}}\frac{k^{\alpha-1}\exp(-cr/(\eps
n)^{1/\alpha})}{r^\alpha\sqrt{k+r}}\lesssim \eps \int_{0}^1\d
t\int_1^\infty\frac{t^{\alpha-1}\exp(-cs/\eps^{1/\alpha})\d
  s}{s^\alpha\sqrt{t+s}}\, ,$$
again a converging integral, that vanishes when $\eps\to 0$. 
Finally, let us consider
the case where $k$ is large. This time, we start
from the estimate based on applying the local limit theorem to
$Q_{m+1}$: 
\begin{align}
  n^{\frac{1}{2\alpha}}\sum_{m=0}^{\eps n} \sum_{\substack{y\leq n^{1-1/\alpha+\lambda}
  \\k\geq n^{1/\alpha}}}
 \frac{r\widehat{\mu}_\bullet(r)\widehat{\mu}_\bullet^{*y}(k)}{\sqrt{k+r}}Q_{m+1}(-r-1)Q^*_{k+y}(n-m+y)\\
  \lesssim n^{\frac{1}{2\alpha}} \sum_{\substack{y\leq n^{1-1/\alpha+\lambda}
      \\k\geq n^{1/\alpha}
  }}\frac{\widehat{\mu}_\bullet^{*y}(k)}{r^{\alpha-1}\sqrt{k+r}}\sum_{m=0}^{\eps
  n}\frac{Q_{k+y}^*(n-m+y)}{m^{1/\alpha}}\, .
\end{align}
Now, we first perform an Abel transform on the last sum,
to turn it into a sum of the form
$$\sum_{m=0}^{\eps n}m^{-1-\frac{1}{\alpha}}Q^*_{k+y}([n-m+y,n+y])\,
.$$
The terms of this sum are (at least for $m$
greater than some large enough constant) respectively $\lesssim
Q^*_{k}([n-m,n])$, by the convergence in distribution entailed by the
local limit theorem. This allows to revert the Abel transformation, and to
obtain as estimate of the form 
$$n^{\frac{1}{2\alpha}-1} \sum_{\substack{y\leq n^{1-1/\alpha+\lambda}
      \\k\geq n^{1/\alpha}}}\frac{\widehat{\mu}_\bullet^{*y}(k)}{r^{\alpha-1}\sqrt{k+r}}\sum_{m=0}^{\eps
  n}Q_{n-m}(-k)\, ,$$
where $Q_{n-m}(-k)\lesssim \exp(-ck/n^{1/\alpha})$. This allows to
bound the superior limit of this expression by a constant multiple of the converging integral
$$\eps^{1-\frac{1}{\alpha}}\int_{1}^\infty\d
t\int_0^\infty\frac{\exp(-cs)\d s}{s^{\alpha-1}t^{2-\alpha}\sqrt{s+t}}\,
,$$
which vanishes as $\eps \to 0$.

This concludes the proof of Lemma \ref{sec:joint-convergence-1}, and
hence of Proposition \ref{sec:joint-convergence-3} and of Proposition
\ref{prop:scalingbelt}. 
\end{proof}

\subsection{Proof of Proposition \ref{lem:esperance}}

Fix $A>0$, and let $D_{(n)}=n^{-\frac{1}{2\alpha}}(\mathrm{d}^{ \mathrm{gr}}_{ \mathfrak{M}_n}(
v_1^n, v_2^n)-1)_+$ for simplicity.

Recall  from \eqref{eq:distancesdelays} that if
$(\bm,(v_1,v_2),\delay)=\mathrm{BDG}^{2\bullet}(\bu)$, then it
holds that $\mathrm{d}^{ \mathrm{gr}}_{ \mathfrak{M}_n}(
v_1, v_2)=2\ell-\delay_1-\delay_2$, where $\ell$ is the
minimal label along the cycle of $\bu$, and $\delay_i$ is the
minimal label along the face $f_i$, minus $1$. We denote this quantity
by $\mathrm{D}_{1,2}(\bu)$.

Then, recalling \eqref{eq:w2bullet}, we have 
  \begin{align*}
      \mathbf{E}\left[ D_{(n)}\ind_{\{D_{(n)}>A\}}\right]
  &=\int_{\mathcal{M}} \frac{w_\bq( \d \bm)}{w_\bq(\#V=n)}\mathrm{vol}_\bm(\d  v_1)\mathrm{vol}_{\bm}( \d 
    v_2)
    \frac{\#\Edelay_{\bm,v_1,v_2}}{n^{2+\frac{1}{2\alpha}}}\ind_{\{
      \#\Edelay_{\bm,v_1,v_2}>An^{\frac{1}{2\alpha}}, \#V=n\}}\\
  &=\int_{\mathcal{M}^{2\bullet}} \frac{w^{2\bullet}_\bq(\d
    (\bm,(v_1,v_2),\delay))}{n^{2+\frac{1}{2\alpha}}w_\bq(\#V=n)} 
  \ind_{\{(\mathrm{d}^{ \mathrm{gr}}_{ \bm}(
  v_1, v_2)-1)_+>An^{\frac{1}{2\alpha}}, \#V=n\}}\\
  &=2\int_{\mathcal{U}} \frac{\tilde{w}^{2\bullet}_\bq(\d \bu)}{n^{2+\frac{1}{2\alpha}}w_\bq(\#V=n)}
  \ind_{\{\mathrm{D}_{1,2}(\bu)>An^{\frac{1}{2\alpha}}, \#V_\circ(\bu)=n-2\}}\\
  &=2\frac{\tilde{w}^{2\bullet}_\bq(\#V_\circ=n-2)}{n^{2+\frac{1}{2\alpha}}w_\bq(\#V=n)}\P(\mathrm{D}_{1,2}(\bu_n)>An^{\frac{1}{2\alpha}})\, .
\end{align*}
Next, we observe that the prefactor
$2\frac{\tilde{w}^{2\bullet}(\#V_\circ=n-2)}{n^{2+\frac{1}{2\alpha}}w_\bq(\#V=n)}$
is a bounded function of $n$, by virtue of (\ref{eq:asymptoUn}) and of
the fact that $w_\bq(\#V=n)=2n^{-1}\mathrm{GW}_\bq(M=n)\sim
cn^{-2-1/\alpha}$. 
In particular, the previously
displayed quantity
is bounded by a constant multiple of 
$\P(4\omega(\bu_n)>An^{\frac{1}{2\alpha}})$,
where $\omega(\bu_n)$ is equal to one plus the maximal 
label of $\bu_n$ in absolute value. In turn, the latter quantity can be written
as the maximum of $\omega^{\mathrm{belt}}(\bu_n)$ and
$\omega^{\mathrm {buckle}}(\bu_n)$, these quantities being equal to
one plus the maximal label in absolute value along the belt (resp.\
buckle) of $\bu_n$. Proposition \ref{prop:scalingbelt} entails that
$(\omega^{\mathrm{belt}}(\bu_n)/n^{\frac{1}{2\alpha}}: n\geq 1)$ forms a tight family of
random variables, so that it
suffices to show that the same is true of
$(\omega^{\mathrm{buckle}}(\bu_n) : n\geq 1)$. This can be achieved by a re-rooting
argument. Indeed, consider a white corner $c_2$ chosen uniformly at
random in
the unicyclomobile $\bu_n$. Clearly, the unicyclomobile $\bu_n^*$
obtained by re-rooting $\bu_n$ at $c_2$ (and forgetting the first root) has same
distribution as $\bu_n$. Therefore, provided that $c_2$ belongs to
the belt of $\bu_n$ (with the initial root), we have 
$\omega^{\mathrm{buckle}}(\bu_n^*)\leq
\omega^{\mathrm{belt}}(\bu_n)$. 
Fix $\eps>0$. Let $E_\eps(n)$ be the event
$\#C_\circ(\boldsymbol{\mathcal{B}}_n)>\eps \#C_\circ(\bu_n)$.  We have by the
above observation, 
\begin{align*}
  \eps\, \mathbf{P}(\omega^{\mathrm{buckle}}(\bu_n^*)>An^{\frac{1}{2\alpha}},E_\eps(n))
&\leq \mathbf{P}(\omega^{\mathrm{buckle}}(\bu_n^*)>An^{\frac{1}{2\alpha}},E_\eps(n),c_2\in
                                                                       C_\circ(\boldsymbol{\mathcal{B}}_n))\\
  &\leq
    \mathbf{P}(\omega^{\mathrm{belt}}(\bu_n) >An^{\frac{1}{2\alpha}})\, .
\end{align*}
Therefore, from the tightness of labels in the belt entailed by
Proposition \ref{prop:scalingbelt}, we obtain that
for every $\eps>0$, 
$\limsup_{n\to\infty}\mathbf{P}(\omega^{\mathrm{buckle}}(\bu_n^*)>An^{\frac{1}{2\alpha}},E_\eps(n))=0$.
We conclude by writing
\begin{align*}
  \mathbf{P}(\omega^{\mathrm{buckle}}(\bu_n)>An^{\frac{1}{2\alpha}})&=\mathbf{P}(\omega^{\mathrm{buckle}}(\bu^*_n)>An^{\frac{1}{2\alpha}})\\
  &\leq
    \mathbf{P}(\omega^{\mathrm{buckle}}(\bu_n^*)>An^{\frac{1}{2\alpha}},E_\eps(n)) +
    \mathbf{P}(E_{\eps}(n)^c)\, ,
\end{align*}
and by observing that the superior limit of the term $\mathbf{P}(E_\eps(n)^c)$
goes to $0$ as $\eps\downarrow 0$. Indeed, since
$\#C_\circ(\boldsymbol{\mathcal{B}_n})\geq
\#V_\circ(\boldsymbol{\mathcal{B}_n})$, we deduce that the family
$(\#C_\circ(\boldsymbol{\mathcal{B}_n})/n : n\geq 1)$  is tight in $(0,\infty]$ by
Proposition~\ref{prop:scalingbelt}. Therefore, our result will follow
from the fact that $(\#C_\circ(\boldsymbol{\mathcal{B}'_n})/n : n\geq 1)$ is tight in
$[0,\infty)$, where $\boldsymbol{\mathcal{B}'_n}$ is the buckle part
of $\bu_n$. To prove this, observe that the following variant of
equation (\ref{eq:UqGWq}) holds: 
$$\mathrm{U}_\bq(\#C_\circ(\mathcal{P})>An,\#V_\circ=n)
=\wh{\mathrm{GW}}_\bq\left(\widetilde{G}(n-M,-L;An)\ind_{M\leq
  n}\right)\,
,$$
where
$\widetilde{G}(m,l;m')=\widetilde{\mathrm{GW}}_\bq(M=m,L=l,\#C_\circ>m')$.  
We conclude from the fact that, for $A>z_\bq$, it holds that
$\widetilde{G}(m,l;An)$ is exponentially small in $n$, for
reasons similar to the ones appearing in the proof of Lemma 
\ref{sec:scaling-limit-buckle-1}, while, on the other hand,
$\wh{\mathrm{GW}}_\bq(M\leq n)$ diverges as a power of $n$. This concludes
the proof of Proposition \ref{lem:esperance}. 

\section{Geodesics  between typical points}\label{secP:uni:geo}

In this final section, we establish the key  properties of geodesics needed to complete the proof of $D = D^*$ presented in  Section \ref{sec:D=D*}. Specifically, we  study  geodesics in $(\mathcal{S}, D, \mathrm{Vol},\rho_*)$ using the bi-marked construction introduced in Section \ref{sec:boltzm-stable-maps} and the scaling limit results of Section~\ref{sec:scalingunicyclo}.  More precisely, we first  prove that geodesics between typical points in $(\mathcal{S}, D, \mathrm{Vol},\rho_*)$ are almost surely unique (Theorem~\ref{alm-unique}). This result enables us to show (Proposition \ref{thm:geodesics:rho:*}) that all geodesics to the root $\rho_{*}$ are simple. Finally, we prove the last remaining estimate on good points along typical geodesics (Proposition \ref{main:techni}), thereby completing the proof of our main result.
\subsection{Coupling the bi-marked construction with  $(\mathcal{S},D,\mathrm{Vol},\rho_*)$}\label{section:coupling}

The purpose of this section is to connect  the scaling limits of the well-labeled unicyclomobiles from the previous section with the construction of $( \mathcal{S}, D,\mathrm{Vol},\rho_*)$. As in Section~\ref{sec:boundsonD}, passing from the discrete constructions to the continuum provides a new perspective on $(\mathcal{S},D,\mathrm{Vol}, \rho_*)$, enabling computations that were previously highly intricate when using the single-pointed $\text{BDG}^{\bullet}$ construction and its scaling limit. In this section, we work under the probability measure $\mathbf{P}$. We will also introduce new random variables, and when doing so, we implicitly enlarge the underlying probability space. \par
 Recall from Section \ref{sec:tightness} that $ \mathfrak{M}_n$ is a rooted $ \mathbf{q}$-Boltzmann map with $n$ vertices, that is, with law $w_\bq( \cdot  \mid \#V=n)$. Conditionally on $ \mathfrak{M}_{n}$, let $v_{1}^{n},v_{2}^{n}$ be two independent uniform vertices of $\mathfrak{M}_{n}$, and finally let $\delay_{n} \in \Edelay_{\mathfrak{M}_n,v_1^n,v_2^n}$ be a uniform integer satisfying $ |  \delay_{n}| <  \mathrm{d}^{ \mathrm{gr}}_{ \mathfrak{M}_{n}}( v_{1}^{n}, v_{2}^{n})$ and such  that $ \delay_{n}+ \mathrm{d}^{ \mathrm{gr}}_{ \mathfrak{M}_{n}}( v_{1}^{n}, v_{2}^{n})$ is even. In particular, $ \big( \mathfrak{M}_{n}, (v_{1}^{n}, v_{2}^{n}),   \delay_{n}\big)$ is a random element of  $\mathcal{M}^{2 \bullet}$, but its law  is not $w_\bq^{2\bullet}( \cdot \mid \# V =n)$, as  defined in \eqref{eq:w2bullet}.  Indeed, there is a bias between these two distributions given by $ \# \Edelay_{\mathfrak{M}_n,v_1^n,v_2^n}$, which is equal to $ (\mathrm{d}^{ \mathrm{gr}}_{ \mathfrak{M}_{n}}(v_1,v_2)-1)_+$  by \eqref{eq:cardinalityD}. More precisely, we define a random variable $\big( \widehat{\mathfrak{M}}_{n}, (\widehat{v}_{1}^{n}, \widehat{v}_{2}^{n}),  \hdelay_{n}\big) \in \mathcal{M}^{2 \bullet}$ of law $w_\bq^{2\bullet}( \cdot \mid \# V =n)$ by biasing the law of $\big( \mathfrak{M}_{n}, (v_{1}^{n}, v_{2}^{n}),   \delay_{n}\big)$ by $(\mathrm{d}^{ \mathrm{gr}}_{ \mathfrak{M}_{n}}(v_1,v_2)-1)_+$. Specifically, for every  non-negative measurable function $f:\mathcal{M}^{2\bullet}\to \mathbb{R}_+$, we have
  \begin{eqnarray} \label{eq:defcartebiaisee} \mathbf{E}\Big[f\big( \widehat{\mathfrak{M}}_{n}, (\widehat{v}_{1}^{n}, \widehat{v}_{2}^{n}),  \hdelay_{n}\big)\Big] &:= &\frac{\mathbf{E}\Big[  \displaystyle \left(\mathrm{d}^{ \mathrm{gr}}_{\mathfrak{M}_{n}}(v_1^n,v_2^n)-1\right)_+ \, f \left(   \mathfrak{M}_{n}, (v_{1}^{n}, v_{2}^{n}),    \delay_{n}\right) \Big]}{\mathbf{E}\Big[  \left(\mathrm{d}^{ \mathrm{gr}}_{\mathfrak{M}_{n}}(v_1^n,v_2^n)-1\right)_+\Big]},  \end{eqnarray} 
and  consequently it holds that:
  \begin{eqnarray*} \mathbf{E}\left[f\big( \widehat{\mathfrak{M}}_{n}, (\widehat{v}_{1}^{n}, \widehat{v}_{2}^{n}),  \hdelay_{n}\big)\right]  & \underset{\eqref{eq:w2bullet}}{=} &  w_{\bq}^{2 \bullet} \Big( f\big(\bm, (v_1,v_2), \delay\big) ~ \Big| ~\# V( \bm) = n \Big)\\
   &\underset{ \eqref{eq:defwtilde2pt}}{=}& \tilde{w}_{\bq}^{2 \bullet} \Big( f\big( \mathrm{BDG}^{2\bullet}(  \mathbf{u})\big)~ \Big|~ \# V_{\circ}( \mathbf{u}) = n-2 \Big) \\ & \underset{ \mathrm{Sec.\  }\ref{sec:result}}{=} & \mathbf{E}\left[f\big( \mathrm{BDG}^{2\bullet}(  \mathbf{u}_n, \epsilon)\big) \right].\end{eqnarray*}  In particular,  we may and will think of $\big( \widehat{\mathfrak{M}}_{n}^{2\bullet},  \hdelay_{n}\big):=\big( \widehat{\mathfrak{M}}_{n}, (\widehat{v}_{1}^{n}, \widehat{v}_{2}^{n}),  \hdelay_{n}\big)$ as being coded by the   well-labeled unicyclomobile $ \mathbf{u}_n$ together with a random sign $ \epsilon$. According to the previous section, let us write   $  (\boldsymbol{\mathcal{B}}_{n}, \hat{v}_{\circ})$ for the belt part in the belt-buckle decomposition (Proposition \ref{sec:comb-decomp-1}) of $ \mathbf{u}_n$ and recall that $(S^{\boldsymbol{ \mathcal{B}}_{n}}_{k}, L^{\boldsymbol{ \mathcal{B}}_{n}}_{k})_{k \geq 0}$ is its  Lukasiewicz encoding as in  Section \ref{sec:tightness}. We also recall the notation   $\theta^{n}$ for the number of white vertices in $\boldsymbol{ \mathcal{B}}_{n}$, $a^{n}$ for   the first time of visit of $ \hat{v}_{ \circ}$ and the constant $\scal$ appearing in the asymptotic  \eqref{eq:tailmu}. We encapsulate these random variables in the quadruplet:
 \begin{equation}\label{eq:dist2}  \mathrm{Cod}_{n} :=    \left(2 (\scal n)^{-\frac{1}{\alpha}}S_{\lfloor (n-1)\cdot\rfloor}^{\boldsymbol{
                \mathcal{B}}_{n}},(\scal n)^{-\frac{1}{2\alpha}}L_{\lfloor (n-1)\cdot\rfloor}^{\boldsymbol{ \mathcal{B}}_{n}}, \frac{a^{n}}{n}, \frac{\theta^{n}}{n}\right),
\end{equation}                  
that we interpret as a random variable in $\mathbb{D}([0,1],\R)^2\times \mathbb{R}^2$. 
 We now introduce the  continuous analog of  \eqref{eq:dist2}. Following Proposition \ref{prop:scalingbelt},  we consider a random variable $( \widehat{X}, \widehat{Z}, \widehat{t}_\bullet, \widehat{\sigma})$  characterized by 
 \begin{equation}\label{eq:Z:hat} \mathbf{E}\big[f( \widehat{X}, \widehat{Z}, \widehat{t}_\bullet, \widehat{\sigma})\big] = \mathrm{Cst} \cdot \mathbf{N}^{\bullet}\Big( f \big(X,Z, t_{\bullet}, \sigma \big) \cdot  \GG\big(1-\sigma,-Z_{t_{\bullet}}\big)\Big),
 \end{equation}
  for every non-negative measurable function $f$ on    $\mathbb{D}([0,1],\R)^2\times \mathbb{R}^2$. We interpret $( \widehat{X}, \widehat{Z}, \widehat{t}_\bullet, \widehat{\sigma})$ as the scaling limit of \eqref{eq:dist2}  since,  by the previous discussion and Proposition \ref{prop:scalingbelt}, we have
 \begin{equation}\label{Codn:convergence:distribution}
 \mathrm{Cod}_{n}\xrightarrow[n\to\infty]{(d)} ( \widehat{X}, \widehat{Z}, \widehat{t}_\bullet, \widehat{\sigma}), 
 \end{equation}
 where, of course, the convergence takes place in  $\mathbb{D}([0,1],\R)^2\times \mathbb{R}^2$. Let us now also mimic  \eqref{eq:defcartebiaisee} on the continuous side. To this end, we introduce a Polish space suitable for defining random rooted bi-marked weighted compact metric spaces, i.e.\ tuples $(M,d_M,\mu,x, (x_1,x_2))$ where $(M,d_M,\mu,x)$ is a rooted weighted compact metric space and $x_1,x_2$ are two points of $M$ called the marks. We stress that  it is also  possible to interpret the root $x$ as a mark; however, it will be useful for us to break the symmetry between the three distinguished points.  Following Section \ref{sec:tightness}, we consider the set of all isometry classes\footnote{Here we say that $(M,d_M,\mu,x, (x_1,x_2))$ and $(M^\prime,d^\prime,\mu^\prime,x^\prime, (x_1^\prime,x_2^\prime))$ are isometric if there exists an isometric bijection $\varphi: M\to M'$ such that $\varphi_*\mu=\mu'$ and $(\varphi(x),\varphi(x_1),\varphi(x_2))= (x^\prime,x_1^\prime, x_2^\prime)$.} of rooted bi-marked weighted compact metric spaces, denoted by $$\mathbb{M}^{2\bullet}_{\mathrm{root}}:=\{(M,d_M,\mu,x, (x_1,x_2)): x,x_1,x_2\in M\} / \mathrm{iso},$$  and we endow it with the rooted bi-marked Gromov--Hausdorff--Prokhorov metric, i.e.
\begin{align*}
&d_{\mathrm{GHP}}^{\mathrm{root},2\bullet}\big(\big(M,d_M,\mu,x,(x_1,x_2)\big),\big(M^\prime,d_{M^\prime},\mu^\prime,x^\prime,(x_1^\prime,x_2^\prime)\big)\big)\\
&:=\inf\limits_{\phi,\phi^{\prime}}\Big( \delta_{\text{H}}\big(\phi(M),\phi^{\prime}(M^{\prime})\big)\vee \delta_{\text{P}}\big(\phi_* \mu,\phi^{\prime}_* \mu^{\prime}\big)\vee \delta\big(\phi(x),\phi^{\prime}(x^{\prime})\big)\vee \delta\big(\phi(x_1),\phi^{\prime}(x_1^{\prime})\big)\vee \delta\big(\phi(x_2),\phi^{\prime}(x_2^{\prime})\big)\Big)~,
\end{align*}
where again the infimum is taken over all isometries $\phi$, $\phi^{\prime}$ from $M$, $M^{\prime}$ into a metric space $(Z, \delta)$.  As usual, here $\delta_{\text{H}}$ (resp.~$\delta_{\text{P}}$) stands for the classical Hausdorff distance (resp.\ the Prokhorov distance).
The space $(\mathbb{M}^{2\bullet}_\mathrm{root},d_{\mathrm{GHP}}^{\mathrm{root},2\bullet})$ is a Polish space (see e.g.\ \cite{abraham2013note,khezeli2023unified}) and we equip it with the corresponding Borel sigma-field. When no ambiguity is possible,  we identify  a rooted bi-marked weighted  compact metric space with its equivalence class. In particular, we can  see the random variables
$$\big(V( \widehat{ \mathfrak{M}}_{n}), (\scal n)^{- \frac{1}{2 \alpha}}\mathrm{d}^{ \mathrm{gr}}_{ \widehat{\mathfrak{M}}_n}, \mathrm{vol}_{ \widehat{\mathfrak{M}}_n},  \widehat{\rho}_*^n,  (\widehat{v}_1^n, \widehat{v}_2^n)\big) \quad \text{ and }\quad \big(V(  \mathfrak{M}_{n}), (\scal n)^{- \frac{1}{2 \alpha}}\mathrm{d}^{ \mathrm{gr}}_{ \mathfrak{M}_n}, \mathrm{vol}_{\mathfrak{M}_n}, \rho_*^n, (v_1^n, v_2^n)\big),  $$
for $n\geq 1$, as random variables on $\mathbb{M}^{2\bullet}_{\mathrm{root}}$, where $\widehat{\rho}_*^n$ and $\rho_*^n$ stand for the starting vertex of the distinguished oriented edge of $ \widehat{ \mathfrak{M}}_{n}$ and $\mathfrak{M}_{n}$ respectively.  When further equipped with the delays, these become random variables taking values in the space $ \mathbb{M}^{2 \bullet}_{\mathrm{root}} \times \mathbb{R}$, which is  endowed with the product topology and Borel sigma-field. We stress that the projection mapping a rooted  bi-marked weighted compact metric space $\big(M,d_M,\mu,x,(x_1,x_2)\big)$ to $\big(M,d_M,\mu,x\big)$  plainly defines a continuous projection from  $\mathbb{M}^{2 \bullet}_{\mathrm{root}}$ onto $\mathbb{M}_{\mathrm{root}}$.
\par This framework provides a natural way to translate \eqref{eq:defcartebiaisee} into the continuum and to obtain scaling limit results.  Specifically, recall the definition of $( \mathcal{S},D, \mathrm{Vol},\rho_*)$ from Proposition \ref{theo:sub} along with the subsequence $(n_k)_{k \geq 1}$ on which it may depend. Conditionally on $( \mathcal{S},D,\mathrm{Vol},\rho_*)$, let $\rho_1, \rho_2 \in \mathcal{S}$ be two independent points sampled according to $ \mathrm{Vol}$ (that is, $\Pi_D(U_1), \Pi_D(U_2)$ for two i.i.d.\ uniform random variables $U_1,U_2$ on $[0,1]$ independent of $(X,Z,D)$). The random variable $( \mathcal{S},D,\mathrm{Vol},\rho_*,(\rho_1,\rho_2))$ takes values in $\mathbb{M}^{2\bullet}_{\mathrm{root}}$, and conditionally to it, we consider $\delay$ a random variable uniform on $[-D(\rho_1,\rho_2), D(\rho_1,\rho_2)]$.  We then introduce $( \widehat{ \mathcal{S}}, \widehat{D}, \widehat{\mathrm{Vol}}, \widehat{\rho}_*, (\widehat{\rho}_1, \widehat{\rho}_2), \hdelay)$, a biased version of $( \mathcal{S},D,  \mathrm{Vol}, \rho_*,(\rho_1, \rho_2),\delay)$, whose law is  defined by 
  \begin{align} \label{eq:defShat} \mathbf{E}\Big[f\big(  \widehat{ \mathcal{S}}, \widehat{D}, \widehat{\mathrm{Vol}}, \widehat{\rho}_*, (\widehat{\rho}_1, \widehat{\rho}_2), \hdelay\big)\Big] &= \frac{\mathbf{E}\Big[ D( \rho_{1}, \rho_{2})  f\left(\mathcal{S},D,  \mathrm{Vol}, \rho_*, (\rho_{1},\rho_{2}),\delay\right)\Big]}{ \mathbf{E}[D(\rho_1, \rho_2)]}\nonumber \\
  &= \frac{\mathbf{E}\left[ \displaystyle \frac{1}{2}\int_{-D( \rho_{1}, \rho_{2})}^{D( \rho_{1}, \rho_{2})}  \mathrm{d}u f\left(\mathcal{S},D,  \mathrm{Vol}, \rho_*, (\rho_{1},\rho_{2}),u\right)\right]}{ \mathbf{E}[D(\rho_1, \rho_2)]},   \end{align} 
for every measurable function $f:\mathbb{M}^{2 \bullet}_{\mathrm{root}}\times \mathbb{R}\to \mathbb{R}_{+}$.  We stress that \eqref{eq:defShat} defines a probability measure on $\mathbb{M}^{2 \bullet}_{\mathrm{root}}\times \mathbb{R}$ since  $ \mathbf{E}[D(\rho_1,\rho_2)]< \infty$ by Lemma \ref{D:expectation}. In particular, the projection $( \widehat{ \mathcal{S}}, \widehat{D}, \widehat{\mathrm{Vol}},\widehat{\rho}_*)$  has the law of $( \mathcal{S},D, \mathrm{Vol},\rho_*)$ biased by $\int  \mathrm{Vol}( \mathrm{d}x_1) \mathrm{Vol}( \mathrm{d}x_2)\ D(x_1,x_2)$.
\begin{lem}\label{convergence:law:2:bullet}
With the notation above, we have:
\begin{equation}\label{eq:converge:in:distribution:M:S:biais}
\Big(V( \widehat{\mathfrak{M}}_{n_k}),   (\scal n_k)^{-\frac{1}{2\alpha}}\cdot \mathrm{d}^{ \mathrm{gr}}_{\widehat{\mathfrak{M}}_{n_k}},\mathrm{vol}_{ \widehat{\mathfrak{M}_{n}}},\widehat{\rho}_{*}^{n_k}, \big(\widehat{v}_{1}^{n_k}, \widehat{v}_{2}^{n_k}\big), (\scal n_k)^{-\frac{1}{2\alpha}}\cdot  {\hdelay}_{n_k}\Big) \xrightarrow[k\to\infty]{(d)}\Big( \widehat{ \mathcal{S}}, \widehat{D}, \widehat{\mathrm{Vol}},\widehat{\rho}_*, \big(\widehat{\rho}_{1}, \widehat{\rho}_{2}\big), \hdelay\Big),
\end{equation}
where the convergence takes place on the space $\mathbb{M}^{2 \bullet}_{\mathrm{root}} \times \mathbb{R}$. 
\end{lem}  
\begin{proof}
First, recall that $$\widehat{Y}_{n} := \Big(V( \widehat{ \mathfrak{M}}_{n}), (\scal n)^{- \frac{1}{2 \alpha}}\cdot \mathrm{d}^{ \mathrm{gr}}_{ \widehat{\mathfrak{M}}_n}, \mathrm{vol}_{ \widehat{\mathfrak{M}}_n}, \widehat{\rho}_*^n,   (\widehat{v}_1^n, \widehat{v}_2^n)\Big)$$ has the law of $Y_n:= ( V( \mathfrak{M}_{n}), (\scal n)^{- \frac{1}{2 \alpha}}\mathrm{d}^{ \mathrm{gr}}_{ {\mathfrak{M}}_n}, \mathrm{vol}_{ \mathfrak{M}_n}, \rho_*^n,  ({v}_1^n, {v}_2^n))$  biased by the Radon--Nikodym derivative:  $$ A_{n} := \Big(\mathrm{d}^{ \mathrm{gr}}_{{\mathfrak{M}}_n}({v}_1^n,{v}_2^n)-1\Big)_+ \Big/ \mathbf{E}\Big[\Big(\mathrm{d}^{ \mathrm{gr}}_{{\mathfrak{M}}_n}({v}_1^n,{v}_2^n)-1\Big)_+ \Big] . $$ By Proposition \ref{theo:sub}, the family $(Y_{n_{k}})_{k \geq 1}$ is tight in $ \mathbb{M}^{2 \bullet}_{\mathrm{root}}$ and converges in law towards: $$ \big( \mathcal{S},D,\mathrm{Vol},\rho_*, (\rho_{1}, \rho_{2})\big).$$  Moreover,  Corollary \ref{lem:esperance} entails that the family $(A_{n})_{n\geq 1}$  is  uniformly integrable. It then follows that $( \widehat{Y}_{n_{k}})_{k \geq 1}$  converges in law towards: $$ \big(  \widehat{\mathcal{S}},\widehat{D},\widehat{\mathrm{Vol}}, \widehat{\rho}_*,  (\widehat{\rho}_{1}, \widehat{\rho}_{2})\big).$$ To extend the result to the delays, notice that conditionally on $ \widehat{Y}_n$, the delay $ \hdelay_{n}$ is  uniform on the set $\{m\in \mathbb{Z}:~|m|< \mathrm{d}^{ \mathrm{gr}}_{ \widehat{\mathfrak{M}}_{n}}(  \widehat{v}_{1}^{n}, \widehat{v}_{2}^{n}) \text{ and } m+ \mathrm{d}^{ \mathrm{gr}}_{ \widehat{\mathfrak{M}}_{n}}( \widehat{v}_{1}^{n}, \widehat{v}_{2}^{n})  \text{ is even}\}$.
It is then straightforward  to deduce  the desired convergence \eqref{eq:converge:in:distribution:M:S:biais}   along the subsequence $(n_{k})_{k \geq 1}$.
\end{proof}
Combining \eqref{Codn:convergence:distribution} and \eqref{eq:converge:in:distribution:M:S:biais}, we derive that the family of random vectors
\begin{equation}\label{eq:V:M:n:k:prim}
\Big(\Big(V( \widehat{\mathfrak{M}}_{n_k}),   (\scal n_k)^{-\frac{1}{2\alpha}}\cdot \mathrm{d}^{ \mathrm{gr}}_{\widehat{\mathfrak{M}}_{n_k}}, \mathrm{vol}_{ \widehat{\mathfrak{M}}_{n_k}}, \widehat{\rho}_*^{n_k}, \big(\widehat{v}_{1}^{n_k}, \widehat{v}_{2}^{n_k}\big), (\scal n_k)^{-\frac{1}{2\alpha}}\cdot  {\hdelay}_{n_k}\Big) ,   \mathrm{Cod}_{n_k}\Big), \quad k\geq 1,
\end{equation}
\noindent living in the Polish space $( \mathbb{M}^{2 \bullet}_{\mathrm{root}}  \times \mathbb{R}) \times (\mathbb{D}([0,1], \mathbb{R})^2 \times \mathbb{R}^2)$, is tight, since each of its  coordinates is tight. By the Skorokhod representation theorem, we can further extract a subsequence $(m_{k})_{k \geq 1}$ from $(n_{k})_{k \geq 1}$ along which we can couple all realizations to have almost sure convergence towards some vector of  $( \mathbb{M}^{2 \bullet}_{\mathrm{root}}  \times \mathbb{R}) \times (\mathbb{D}([0,1], \mathbb{R})^2 \times \mathbb{R}^2)$ with marginals distributed as $((  \widehat{\mathcal{S}},\widehat{D},\widehat{\mathrm{Vol}}, \widehat{\rho}_*,  (\widehat{\rho}_{1}, \widehat{\rho}_{2})), \hdelay)$ and $( \widehat{X}, \widehat{Z}, \widehat{t}_\bullet, \widehat{\sigma})$ respectively. With a slight abuse of notation, we still denote this vector by: 
\begin{equation}\label{eq:S:tilde}
\Big(\big( (  \widehat{\mathcal{S}},\widehat{D},\widehat{\mathrm{Vol}}, \widehat{\rho}_*,  (\widehat{\rho}_{1}, \widehat{\rho}_{2})), \hdelay\big) , ( \widehat{X}, \widehat{Z}, \widehat{t}_\bullet, \widehat{\sigma})\Big).
\end{equation}
We stress that the coupling between the two parts of this vector is a priori unknown.  The point is that we can use this coupling to pass the discrete $ \mathrm{BDG}^{2\bullet}$ to the scaling limit. In this way, we can import some new geometric information on $\widehat{D}$ (equivalently on  $D$ after de-biasing) from $ (\widehat{X}, \widehat{Z}, \widehat{t}_{\bullet}, \widehat{\sigma})$ similarly as we imported information on $D$ from $(X,Z)$ in Section \ref{sec:boundsonD}.  We also stress that for convenience, all random variables $( ( \widehat{\mathcal{S}},\widehat{D},\widehat{\mathrm{Vol}}, \widehat{\rho}_*,  (\widehat{\rho}_{1}, \widehat{\rho}_{2})), \hdelay ,  \widehat{X}, \widehat{Z}, \widehat{t}_\bullet, \widehat{\sigma})$ as well as  $( ( \mathcal{S},D,\mathrm{Vol}, \rho_*,  (\rho_{1}, \rho_{2})), \delay ,  X,Z, t_\bullet, \sigma)$ are defined under the probability measure $  \mathbf{P}$, but there is no coupling between those two vectors; they could very well be independent.
\medskip 
\begin{center}
\begin{minipage}{0.9\linewidth}
\hrule\vspace{1ex}
\centering\textit{In the remainder of Section \ref{secP:uni:geo}, we shall always suppose that we have made the above couplings and implicitly restrict to the subsequence $(m_k)_{k\geq 1}$.}
\vspace{1ex}\hrule
\end{minipage}
\end{center}
\vspace{0.2cm}
Let us conclude this section by  addressing a technicality. Under $\mathbf{P}$, the space  $ \big( \mathcal{S},D,\mathrm{Vol},\rho_*, (\rho_{1}, \rho_{2})\big)$ is constructed from the pseudo-distance $D:[0,1]^2\to \mathbb{R}_+$ and two uniform independent  random variables on $[0,1]$.
Therefore, $ \big( \mathcal{S},D,\mathrm{Vol},\rho_*, (\rho_{1}, \rho_{2})\big)$  is a \textit{concrete} rooted bi-marked weighted compact metric space, which we then view as a random variable when  considering its equivalence class in $\mathbb{M}_{\mathrm{root}}^{2\bullet}$.
This allowed us to discuss metric  properties directly using the representative $ \big( \mathcal{S},D,\mathrm{Vol},\rho_*, (\rho_{1}, \rho_{2})\big)$. 
However, this is not longer the case for  $(  \widehat{\mathcal{S}},\widehat{D},\widehat{\mathrm{Vol}}, \widehat{\rho}_*,  (\widehat{\rho}_{1}, \widehat{\rho}_{2}))$ which is directly defined as an element of  $\mathbb{M}_{\mathrm{root}}^{2\bullet}$. To circumvent this difficulty and treat it as  a 
rooted bi-marked weighted compact metric space, we rely on general representation theorems for Gromov-Hausdorff-type topologies \cite{khezeli2023unified}. Namely, we fix $\omega$ (in the underlying probability space) such that the convergence of \eqref{eq:V:M:n:k:prim}  to  \eqref{eq:S:tilde} holds. Then, by \cite[Lemma 2.5]{khezeli2023unified}, we can
embed all the rooted bi-marked weighted compact metric spaces in \eqref{eq:V:M:n:k:prim}  and  \eqref{eq:S:tilde}  isometrically in the same compact 
metric space $(E,d_E)$, in such a way that
\begin{equation}\label{eq:V:M:k:convergence:E}
V( \widehat{\mathfrak{M}}_{m_k})\build{\la}_{k\to\infty}^{} \widehat{ \mathcal{S}}\quad \text{ and } \quad \big(\widehat{\rho}_*^{m_k}, \widehat{v}_{1}^{m_k}, \widehat{v}_{2}^{m_k}\big)\build{\la}_{k\to\infty}^{} \big(\widehat{\rho}_*, \widehat{\rho}_{1}, \widehat{\rho}_{2}\big),
\end{equation}
where the first convergence is with respect to    the Hausdorff distance in  $(E,d_E)$. We also have the weak convergence $\mathrm{vol}_{ \widehat{\mathfrak{M}}_{m_k}} \rightarrow \widehat{\mathrm{Vol}} $,  but this last convergence will not be needed in this work. We stress that the  compact metric space $(E,d_E)$  might depends on $\omega$.
Then, if we establish that some property holds, for $\mathbf{P}$ almost every  such $\omega$, then we can use \eqref{eq:defShat} to import it to $ \big( \mathcal{S},D,\mathrm{Vol},\rho_*, (\rho_{1}, \rho_{2})\big)$,  $\mathbf{P}$-a.s., provided that the related events are measurable.  We will implement this method to deduce the uniqueness of typical geodesics in Section \ref{sec:uni:geo} and to study good points in Section \ref{sub:goodpointestimate}. 
Let us list here the almost sure properties that we will use: We write $\Omega^{\prime}$ for the subset of all such $\omega$ such that \eqref{eq:V:M:k:convergence:E} hold,  such that  the process $\widehat{X}$ does not have negative jumps and verifies properties {\hypersetup{linkcolor=black}\hyperlink{prop:A:1}{$(A_1)$}}--{\hypersetup{linkcolor=black}\hyperlink{prop:A:4}{$(A_4)$}} of Section \ref{sec:defloop} with $X$ replaced by $\widehat{X}$, and  such that $\widehat{Z}$ is continuous and its  restriction  to $\widehat{\mathrm{Branch}}(0,\widehat{t}_\bullet):=\{h\leq \widehat{t}_\bullet:~\widehat{X}_{h-}\leq \min_{[h,\widehat{t}_\bullet]} \widehat{X} \}$ realizes its infimum at a unique point. We stress  that the set  $\widehat{\mathrm{Branch}}(0,\widehat{t}_\bullet)$ corresponds,  in terms of the looptree encoded by $\widehat{X}$, to  the pinch point times in the branch from $0$ to $\widehat{t}_\bullet$, in the sense of Section \ref{sec:defloop}. By the above discussion, \eqref{eq:Z:hat} and Lemmas~\ref{Lem-cut-jump}~and~\ref{lem:onePinch}, we have $ \mathbf{P}( \Omega^\prime)=1$.

\subsection{Uniqueness of geodesics between two typical points}\label{sec:uni:geo}
Let us present the first application of the above methodology to prove uniqueness of typical geodesics. Specifically, we are going to show that, $\mathbf{P}$-a.s., there is a unique geodesic on $\mathcal{S}$ going from $\rho_1$ to $\rho_2$, thereby establishing Theorem \ref{alm-unique}. Since to achieve this, we rely on the biased version $$(  \widehat{\mathcal{S}},\widehat{D},\widehat{\mathrm{Vol}}, \widehat{\rho}_*,  (\widehat{\rho}_{1}, \widehat{\rho}_{2})),$$ we need to check measurability issues. In this direction, we introduce $\mathbb{PM}_{\mathrm{root}}^{2\bullet,1}$, the subset of $\mathbb{M}_{\mathrm{root}}^{2\bullet}$ consisting of all equivalence classes of weighted geodesic compact metric spaces with exactly one geodesic between the two marked points. We emphasize that $\mathbb{PM}_{\mathrm{root}}^{2\bullet,1}$ is well-defined, as the property of having exactly one geodesic between the two marked points is invariant under isometry.

\begin{lem}
$\mathbb{PM}_{\mathrm{root}}^{2\bullet,1}$ is a measurable subset of $\mathbb{M}_{\mathrm{root}}^{2\bullet}$. 
\end{lem}
\begin{proof} First note that by (a marked extension of) \cite[Theorem 7.5.1]{BBI01}, the space $\mathbb{PM}_{\mathrm{root}}^{2\bullet}$
of equivalence classes of rooted bi-marked weighted \textit{geodesic} compact metric spaces is a closed subset of   $\mathbb{M}_{\mathrm{root}}^{2\bullet}$, and furthermore that the number of geodesics connecting $x_1$ and $x_2$ does not depend on the choice of representative  $(M, d, \mu, x, (x_1, x_2))$.
For every $\delta > 0$, consider $\mathbb{PM}_{\mathrm{root}}^{2\bullet}[\delta]$, the space of equivalence classes of spaces $(M, d, \mu, x, (x_1, x_2))$ possessing at least two geodesics, $\gamma$ and $\gamma'$, going from $x_1$ to $x_2$, such that
$$\sup\big\{d(\gamma(t),\gamma^\prime(t)):~t\in [0,d(x_1,x_2)]\big\}\geq \delta.$$ Then it is straightforward to verify that $\mathbb{PM}_{\mathrm{root}}^{2\bullet}[\delta]$ is a closed subset of $\mathbb{PM}_{\mathrm{root}}^{2\bullet}$,  and  $\mathbb{PM}_{\mathrm{root}}^{2\bullet,1}$ is a measurable set since:
$$ \mathbb{PM}_{\mathrm{root}}^{2\bullet,1}= \mathbb{PM}_{\mathrm{root}}^{2\bullet}\setminus \bigcup_{n\in \mathbb{N}}  \mathbb{PM}_{\mathrm{root}}^{2\bullet}[2^{-n}] .$$
\end{proof}
Therefore, our formal goal is to establish that $(\mathcal{S},D,\mathrm{Vol},\rho_*, (\rho_1,\rho_2))\in \mathbb{PM}_{\mathrm{root}}^{2\bullet,1}$, $\mathbf{P}$-a.s.,  which by \eqref{eq:defShat} is equivalent to:
\begin{equation}\label{eq:S:hat:geo:1}
\big(  \widehat{\mathcal{S}},\widehat{D},\widehat{\mathrm{Vol}}, \widehat{\rho}_*,  (\widehat{\rho}_{1}, \widehat{\rho}_{2})\big)\in \mathbb{PM}_{\mathrm{root}}^{2\bullet,1}, \quad \mathbf{P}\text{-a.s.}
\end{equation}
To demonstrate \eqref{eq:S:hat:geo:1}, we  exploit the connection with well-labeled unicyclomobiles, following a method similar to that in \cite[Section 7.2.3]{Mie09}. Specifically, we fix an element $\omega\in \Omega^\prime$ in the underlying probability space, as described at the end of Section \ref{eq:S:tilde}.  In particular,  $(  \widehat{\mathcal{S}},\widehat{D},\widehat{\mathrm{Vol}}, \widehat{\rho}_*,  (\widehat{\rho}_{1}, \widehat{\rho}_{2}))$ is embedded into a compact space $(E,d_E)$.  Next, for every $h\in  [- \widehat{D}( \widehat{\rho}_{1}, \widehat{\rho}_{2}), D( \widehat{\rho}_{1}, \widehat{\rho}_{2})]$, we write 
$$\widehat{\mathscr{G}}(h) :=  \Big \{ x \in \widehat{\mathcal{S}} : \widehat{D}(x,\widehat{\rho}_{1}) = \frac{\widehat{D}( \widehat{\rho}_{1}, \widehat{\rho}_{2})-h}{2} \mbox{ and } \widehat{D}(x,\widehat{\rho}_{2}) = \frac{\widehat{D}( \widehat{\rho}_{1}, \widehat{\rho}_{2})+h}{2} \Big\},$$
\noindent which is the set of all $h$-median points in a $\widehat{D}$-geodesic between $\widehat{\rho}_{1}$ and $\widehat{\rho}_{2}$. The assertion that there is a unique geodesic connecting $\widehat{\rho}_1$ to $\widehat{\rho}_2$ is equivalent to stating that $\widehat{\mathscr{G}}(h)$ is a singleton for every $h \in [-\widehat{D}(\rho_1, \rho_2), \widehat{D}(\rho_1, \rho_2)]$. Moreover by continuity, if $\#\widehat{\mathscr{G}}(h) \geq 2$ for some $h$, then $\#\widehat{\mathscr{G}}(h') \geq 2$ for all $h'$ in a small neighborhood of $h$.
Thus,  to obtain the uniqueness of geodesics going from $\widehat{\rho}_1$ to $\widehat{\rho}_2$, it suffices to show that:

\begin{lem}\label{lem:alm-unique} For every  $\omega\in \Omega^\prime$ and with the notation above,  the set $ \widehat{\mathscr{G}}( \hdelay)$ is a singleton.
\end{lem}

\begin{proof}
Let $\omega$  be fixed as explained  above. We begin by reformulating the statement of the lemma in terms of $( \widehat{\mathfrak{M}}_{m_k}, \widehat{v}_{1}^{m_k},  \widehat{v}_{2}^{m_k},  \hdelay_{m_k})$. In this direction, we introduce the set  $\mathcal{A}$ of all sequences $(x_{m_k})_{k\geq 1}$, with $x_{m_k}$ a vertex of $\widehat{\mathfrak{M}}_{m_k}$, such that
 \begin{equation} \label{eq:distanceconverge:A}(\scal m_k)^{- \frac{1}{2 \alpha}}\cdot  \mathrm{d}^{ \mathrm{gr}}_{\widehat{ \mathfrak{M}}_{m_k}}( \widehat{v}_{1}^{m_k},  x_{m_k}) \to \frac{\widehat{D}( \widehat{\rho}_{1}, \widehat{\rho}_{2})-\hdelay}{2} \mbox{ \ \ and \ \  } (\scal m_k)^{- \frac{1}{2 \alpha}}\cdot \mathrm{d}^{ \mathrm{gr}}_{\widehat{ \mathfrak{M}}_{m_k}}( \widehat{v}_{2}^{m_k},  x_{m_k}) \to   \frac{\widehat{D}( \widehat{\rho}_{1}, \widehat{\rho}_{2})+\hdelay}{2}.  \end{equation}  
Next, remark that  if there exist  $x,x^{\prime} \in \widehat{\mathscr{G}}( \hdelay)$ with $x\neq x^\prime$, then by the definition of $\widehat{\mathscr{G}}( \hdelay)$ and  \eqref{eq:V:M:k:convergence:E} we can find  two sequences $(x_{m_k})_{k\geq 1}, (x_{m_k}^\prime)_{k\geq 1}$ in $\mathcal{A}$ such that 
  \begin{equation}\label{eq:limsup:x:m}
\limsup_{k\to \infty} (\scal m_k)^{- \frac{1}{2 \alpha}}\cdot \mathrm{d}^{ \mathrm{gr}}_{\widehat{ \mathfrak{M}}_{m_k}}(x_{m_k},x_{m_k}^\prime)>0.
\end{equation}
We shall prove that this is impossible. The interest of this reformulation lies in the fact that
  \eqref{eq:distanceconverge:A} and \eqref{eq:limsup:x:m}  involve quantities that can be directly controlled with  the construction $( \widehat{\mathfrak{M}}_{m_k}, \widehat{v}_{1}^{m_k},  \widehat{v}_{2}^{m_k},  \hdelay_{m_k}) = \mathrm{BDG}^{2\bullet}( \mathbf{u}_{m_k},  \epsilon)$ of Section \ref{sec:boltzm-stable-maps}. In this direction, for  $k\geq 1$, we recall the notation $f_1^{m_k}$ and $f_2^{m_k}$ for the faces containing $\widehat{v}_1^{m_k}$ and $\widehat{v}_2^{m_k}$ respectively, and let us  start with a couple of easy remarks. First, for each $k\geq 1$, consider $J_{m_k}$  a white vertex attaining the minimal label along the cycle of $ \mathbf{u}_{m_k}$, and remark that by  \eqref{dist:v_*}, we must have:
 $$(\scal m_k)^{- \frac{1}{2 \alpha}}\cdot \Big(\mathrm{d}^{ \mathrm{gr}}_{\widehat{ \mathfrak{M}}_{m_k}}( \widehat{v}_{1}^{m_k},  J_{m_k})-\mathrm{d}^{ \mathrm{gr}}_{\widehat{ \mathfrak{M}}_{m_k}}( \widehat{v}_{2}^{m_k},  J_{m_k}) \Big)\to \hdelay.$$
Next, recall from the discussion preceding \eqref{eq:distancesdelays} that if we concatenate two simple geodesics in the map  starting from $J_{m_k}$ and obtained by iterating the successor function in the faces $f_{1}^{m_k}$ and $f_{2}^{m_k}$ respectively, until reaching $\widehat{v}_{1}^{m_k}$ and $\widehat{v}_{2}^{m_k}$, then this concatenation produces a geodesic $\widehat{\gamma}_{1,2}^{m_k}$ connecting $\widehat{v}_{1}^{m_k}$ to $\widehat{v}_{2}^{m_k}$.  Hence,  we must also have
  $$(\scal m_k)^{- \frac{1}{2 \alpha}}\cdot \Big(\mathrm{d}^{ \mathrm{gr}}_{\widehat{ \mathfrak{M}}_{m_k}}( \widehat{v}_{1}^{m_k},  J_{m_k})+\mathrm{d}^{ \mathrm{gr}}_{\widehat{ \mathfrak{M}}_{m_k}}( \widehat{v}_{2}^{m_k},  J_{m_k}) \Big)\to \widehat{D}(\widehat{\rho_1},\widehat{\rho}_2).$$
  We infer that $(J_{m_k})_{k\geq 1}$ is in $\mathcal{A}$. Now remark  that in order to have \eqref{eq:limsup:x:m} for two sequences in $\mathcal{A}$, there must exist $(x_{m_k})_{k\geq 1}\in\mathcal{A}$ such that $
\limsup_{k\to \infty} (\scal m_k)^{- \frac{1}{2 \alpha}}\cdot \mathrm{d}^{ \mathrm{gr}}_{\widehat{ \mathfrak{M}}_{m_k}}(J_{m_k},x_{m_k})>0$. We argue by contradiction and assume that we can find such a sequence $(x_{m_k})_{k\geq 1}$. For $k\geq 1$,  consider a path $\gamma_{m_k}$ going from $\widehat{v}_{1}^{m_k}$ to $\widehat{v}_{2}^{m_k}$ and passing through $x_{m_k}$ obtained by concatenating two geodesics from $x_{m_k}$ to $\widehat{v}_{1}^{m_k}$ and from $x_{m_k}$ to $\widehat{v}_{2}^{m_k}$. By the Jordan theorem, this path must go through a white vertex $J_{m_k}^\prime$ of the cycle of $ \mathbf{u}_{m_k}$, and since $(x_{m_k})_{k\geq 1}$ is in $\mathcal{A}$, the length of $\gamma_{m_k}$  is equal to $\mathrm{d}^{ \mathrm{gr}}_{\widehat{\mathfrak{M}}_{m_k}}( \widehat{v}_{1}^{m_k}, \widehat{v}_{2}^{m_k}) + o(m_k^{ \frac{1}{2 \alpha}})$,
  and as a consequence \begin{equation}  \label{eq:distancepastrop} \mathrm{d}^{ \mathrm{gr}}_{\widehat{\mathfrak{M}}_{m_k}}( \widehat{v}_{1}^{m_k}, \widehat{v}_{2}^{m_k}) + o( m_k^{ \frac{1}{2 \alpha}}) \geq  \mathrm{d}^{ \mathrm{gr}}_{\widehat{\mathfrak{M}}_{m_k}}( \widehat{v}_{1}^{m_k}, J_{m_k}^\prime) +  \mathrm{d}^{ \mathrm{gr}}_{\widehat{\mathfrak{M}}_{m_k}}( J_{m_k}^\prime, \widehat{v}_{2}^{m_k}) \quad \text{ and }\quad \mathrm{d}^{ \mathrm{gr}}_{\widehat{ \mathfrak{M}}_{m_k}}(  x_{m_k}, J_{m_k}^\prime) = o(m_k^{ \frac{1}{2\alpha}}). \end{equation}    
  We refer to Figure \ref{fig:douane} for an illustration. 
\begin{figure}[!h]
 \begin{center}
 \includegraphics[width=12cm]{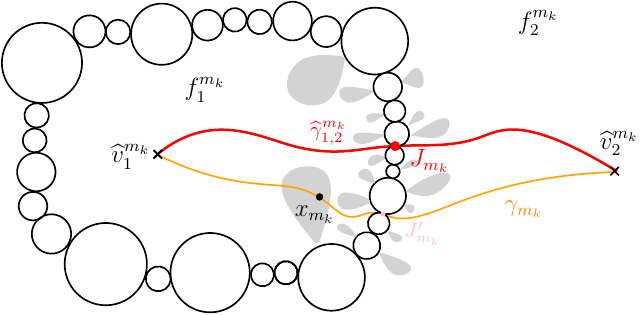}
 \caption{Illustration of the notation. The (essentially unique) vertex minimizing the label along the cycle of $ \mathbf{u}_{m_k}$ is denoted by $J_{m_k}$. The vertex $ x_{m_k}$  must be close to a vertex $J_{m_k}^\prime$ on the cycle which asymptotically minimizes the label. \label{fig:douane}}
 \end{center}
 \end{figure}
  In particular, to obtain a contradiction it suffices to show that $\mathrm{d}^{ \mathrm{gr}}_{\widehat{ \mathfrak{M}}_{m_k}}(  J_{m_k}, J_{m_k}^\prime) = o(m_k^{ \frac{1}{2\alpha}})$. 
In this direction, remark that by \eqref{dist:v_*} and \eqref{eq:distancepastrop}, combined with the fact that $J_{m_k}$ minimizes  the label along the cycle of $\mathbf{u}_{m_k}$, we have
  \begin{eqnarray} \label{eq:Knpresqmin} \ell_{m_k}(J_{m_k}^\prime)  = \ell_{m_k}(J_{m_k}) + o( m_k^{ \frac{1}{2 \alpha}}),  \end{eqnarray}
  where $\ell_{m_k}$ stands for the label function of $\mathbf{u}_{m_k}$.
We conclude using now  the convergence of $\mathrm{Cod}_{m_k}$ towards $( \widehat{X}, \widehat{Z}, \widehat{t}_\bullet, \widehat{\sigma})$ by an argument parallel to  the one of the proof of the Cactus bound of Lemma~\ref{lem:cactusbound}. In this direction, let us write  $j_{m_k}\leq  a^{m_k}$ (resp.\ $j_{m_k}^\prime \leq a^{m_k}$) for the  visit time of the vertex $J_{m_k}$ (resp.\ $J_{m_k}^\prime$) in the belt-part of $ \mathbf{u}_{m_k}$ (by definition all white vertices of the cycle are present in the belt part). Up to further extracting, we can suppose that almost surely $j_{m_k}/n \to  \mathfrak{j} \in [0, \widehat{t}_{\bullet}]$ (resp.\ $ j_{m_k}^\prime/n \to \mathfrak{j}^\prime \in [0, \widehat{t}_{\bullet}]$). 
It is straightforward to deduce from the facts that $J_{m_k}$ and $J_{m_k}^\prime$ lie in the cycle of $\mathbf{u}_{m_k}$, the assumed properties of $(\widehat{X}, \widehat{Z})$ on $\omega$, and \eqref{eq:Knpresqmin},  
 that
 $\mathfrak{j},\mathfrak{j}^\prime\in \widehat{\mathrm{Branch}}(0,\widehat{t}_\bullet)$ and $\widehat{Z}_{\mathfrak{j}}=\widehat{Z}_{\mathfrak{j}^\prime}= \min\{\widehat{Z}_t:~t\in \widehat{\mathrm{Branch}}(0,\widehat{t}_\bullet)\}$. Since the restriction of $\widehat{Z}$ at $ \widehat{\mathrm{Branch}}(0,\widehat{t}_\bullet)$ realizes its infimum at a unique point we derive that $\mathfrak{j}^\prime=\mathfrak{j}$. Lastly, by the Schaeffer bound \eqref{d_n^circ:2} and the continuity of $\widehat{Z}$, we deduce that
 $$\limsup \limits_{k\to \infty} (\scal m_k)^{- \frac{1}{2 \alpha}}\cdot \mathrm{d}^{ \mathrm{gr}}_{\widehat{ \mathfrak{M}}_{m_k}}(  J_{m_k}, J_{m_k}^\prime)\leq \widehat{Z}_{ \mathfrak{j}}+\widehat{Z}_{ \mathfrak{j}^\prime} -2\min \limits_{[ \mathfrak{j}\wedge  \mathfrak{j}^\prime, \mathfrak{j} \vee  \mathfrak{j}^\prime]} \widehat{Z}=0.$$
 This completes the proof of the lemma.
 \end{proof}

For simplicity, we let $\Omega^{\prime\prime}$ be the subset of all $\omega\in \Omega^\prime$ such that $\big(  \mathcal{S},D,\mathrm{Vol}, \rho_*,  (\rho_{1}, \rho_{2})\big)$ and  $\big(  \widehat{\mathcal{S}},\widehat{D},\widehat{\mathrm{Vol}}, \widehat{\rho}_*,  (\widehat{\rho}_{1}, \widehat{\rho}_{2})\big)$ belong to $\mathbb{PM}_{\mathrm{root}}^{2\bullet,1}$. By Lemma \ref{lem:alm-unique}, we have $\mathbf{P}(\Omega^{\prime\prime})=1$.

\subsection{Classification of the geodesics towards $\rho_{*}$} \label{sec:classificationgeo}

The aim of this section is to use Theorem \ref{alm-unique} to derive the structure of the $D$-geodesics towards $\rho_*$ in the metric space $(\mathcal{S},D,\mathrm{Vol},\rho_*)$. Specifically, recall from Section \ref{sec:D<D*} the notion of simple geodesics $\gamma^{(s)}$ and the definition of the path $\gamma^{(s \to t)}$ going from $\Pi_D(s)$ to $\Pi_D(t)$. Informally, this path is obtained by following $\gamma^{(s)}$ down to the merging point with $\gamma^{(t)}$ and then ascending to $\Pi_D(t)$. The goal of this section is to prove Proposition \ref{thm:geodesics:rho:*}, which we restate here for clarity:
\\
\\
\centerline{\textit{$\mathbf{P}$ - a.s.,  all the geodesics towards $\rho_*$  in $(\mathcal{S},D,\mathrm{Vol},\rho_*)$ are simple geodesics.}}
\\
\\
\noindent In particular, the latter  entails that  the geodesics to $\rho_{*}$ coincide for both metrics $D^*$ and $D$.  The idea of the proof, adapted from \cite[Section 6]{bettinelli2016geodesics}, is to cage general geodesics towards the root by typical geodesics towards the root.

\begin{proof}

We begin by introducing   the set  $\mathfrak{S}$ of all points $x\in \mathcal{S}$ such that there exists only one $D$-geodesic from $x$ to $\rho_{*}$. If $x\in \mathfrak{S}$, then the unique geodesic between $x$ and $\rho_{*}$ is a simple geodesic of the form $\gamma^{(s)}$, with  $s\in \Pi^{-1}_D(x)$. By Theorem \ref{alm-unique}  and the re-rooting property \eqref{eq:re-rooting-S}, if $(U_{k} : k \geq 1)$ are i.i.d.\ uniform random variables independent of $(X,Z,D)$, the event $\{\Pi_D(U_k):~k\geq 1\}\subset \mathfrak{S}$ has full $ \mathbf{P}$-probability. Moreover, since $(U_k : k \geq 1)$ is dense in $[0,1]$ and $\mathfrak{z}$ is continuous,  \eqref{eq:geocoincide} guarantees that for any $t\in [0,1]$ and $\varepsilon>0$, there exists $k\geq 0$ such that the simple geodesics $\gamma^{(t)}$ and $\gamma^{(U_k)}$ coincide outside $B_D(\Pi_D(t),\varepsilon)$. We thus deduce that the event:
$$\big\{\gamma^{(t)}\big((0,Z_t-Z_{t_*}]\big):~t\in [0,1]\big\}\subset \mathfrak{S}, $$
has full probability.  For the remainder of the proof, we work on this event and we argue by contradiction, supposing that there exists a geodesic $\gamma:[0,r_0]\to \mathcal{S}$, with  $\gamma(r_0)=\rho_{*}$, that does not coincide with a simple geodesic. By a compactness argument, there must exist $\varepsilon\in(0,r_0)$ such that:
\begin{equation}\label{pince:gamma}
\gamma([0,\eps])\cap \Big\{\gamma^{(t)}\big((0,Z_t-Z_{t_*}]\big):~t\in [0,1]\Big\}=\varnothing.
\end{equation}
For the sake of clarity, let us fix  $t\in\Pi^{-1}_D(\gamma(0))$ and note that $t\notin \{t_*\}\cup \{U_k:~k\geq 1\}$. We argue based on whether $\Pi_{d}(t)$ is a leaf or a pinch point of $\mathcal{L}$, and we refer to Figure \ref{fig:unicity-trap} for an illustration of the argument.
\par
\textsc{Case 1: The point $\Pi_d(t)$ is a leaf.}  First assume that $t\notin \{0,1\}$. Using the characterization of leaves given in Section \ref{Sec:equiv:d}, we know  that $\Pi_d(t)$ belongs to a loop or that there exists a sequence $(t_n)_{n\geq 0}$ of jumping times converging towards $t$ and with $t_n\preceq t_{n+1}$ for every $n\geq 0$ -- in particular the loop associated with $t_n$ disconnects the points $\Pi_{d}(0)$ and $\Pi_{d}(t)$.  Now recall that, conditionally on $X$, the process $Z$ along a loop of  $\mathcal{L}$ is a Brownian bridge. We deduce  that we can always find $u_m< u_m^{\prime}<t<v^{\prime}_m<v_m$ verifying the following properties (see Figure \ref{fig:unicity-trap} for an illustration):
\begin{itemize}
\item There exists $t^{\prime}\prec t$, with $\Delta_{t^{\prime}}>0$, such that $\Pi_d(u_m),\Pi_d(u_m^{\prime}),\Pi_d(v_m),\Pi_d(v^{\prime}_m)$ belong to the loop associated with $t^{\prime}$;
\item All the sequences  $(u_m)_{m>0},(u_m^{\prime})_{m>0},(v_m)_{m>0},(v^{\prime}_m)_{m>0}$ converge to $t$;
\item We have $Z_{u_m}\wedge Z_{v_m}>Z_{u_{m^{\prime}}}\vee Z_{v_{m^{\prime}}}.$
\end{itemize}
Here, we only used the fact that $\Pi_d(t)$ is a leaf in the second item.  We can even suppose that for all $m \geq 1$ we have $ t_{*} \notin [u_{m},v_{m}]$. Furthermore it is straightforward  to verify that, for every $m\geq 1$, the path $\gamma^{(u_m^{\prime}\to v_m^{\prime})}$ disconnects $\Pi_D\big((u_m^{\prime},v_m^{\prime})\big)$ from its complement. To see why, note that the discrete analogue of the previous statement follows directly from the BDG bijection and planarity. The same method used to prove the Cactus Bound (Lemma \ref{lem:cactusbound}) can then be applied to extend this result to the continuous setting.
\begin{figure}[!h]
 \begin{center}
 \includegraphics[width=14.5cm]{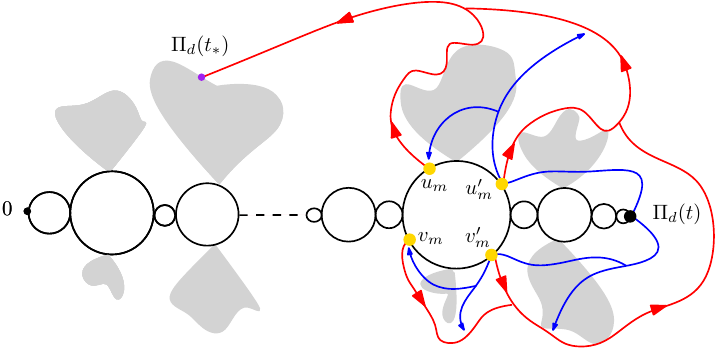}
 \caption{\label{fig:unicity-trap} Illustration of the paths $\gamma^{(u_m\to v_m)}$ and $\gamma^{(u_m^{\prime}\to v_m^\prime)}$ disconnecting respectively $\Pi_D((u_m,v_m))$ and $\Pi_D((u_m^\prime,v_m^\prime))$ from its complement. A geodesic starting from $\Pi_{D}(t)$ and targeting $\rho_{*}$ must either cross those paths or pass through  $\Pi_{D}(u_{m}')$ and $\Pi_{D}(u_{m})$, or $\Pi_{D}(v_{m}')$ and $\Pi_{D}(v_{m})$, in this order, which is excluded by \eqref{Distance:rho_*}.}
 \end{center}
 \end{figure}
Since $\rho_*$ is in the complement of $\Pi_D\big((u_m^{\prime},v_m^{\prime})\big)$, it follows that $\gamma$ must intersect the range of $\gamma^{(u_m^{\prime}\to v_m^{\prime})}$. Recalling that $(u_m^{\prime})_{m>0}$ and $(v_m^{\prime})_{m>0}$ both converge to $t$, we obtain from \eqref{pince:gamma} that $\gamma$ must hit $\gamma^{(u_m^{\prime}\to v_m^{\prime})}$ at $\Pi_D(u_m^{\prime})$ or $\Pi_D(v_m^{\prime})$ for some time $s_m^{\prime}$, for every $m$ large enough. We can then apply the same argument with $u_m^\prime$ and $v_m^\prime$ replaced by $u_m$ and $v_m$ to deduce that, for $m$ large enough, $\gamma$ has also to intersect the set $\{\Pi_D(u_m),\Pi_D(v_m)\}$ at some time $s_m>s^\prime_m$. By \eqref{Distance:rho_*}, we get:
\begin{align*}
D(\gamma(s_m),\rho_*)&\geq D\big(\rho_*,\Pi_D(u_m)\big)\wedge D\big(\rho_*,\Pi_D(v_m)\big)
\\
 &=\big(Z_{u_m}-\inf Z \big)\wedge \big(Z_{v_m}-\inf Z \big) \\
&\geq \big(Z_{u_m^{\prime}}-\inf Z \big)\vee \big(Z_{v_m^{\prime}}-\inf Z \big)\\
&=D(\rho_*,\Pi_D(u_m^{\prime}))\vee D(\rho_*,\Pi_D(v_m^{\prime}))\geq D(\gamma(s_m^\prime),\rho_*),
\end{align*}
where in the second and last line we used \eqref{Distance:rho_*}.  The previous display is in contradiction with the fact that $\gamma$ is a geodesic towards $\rho_*$ since $s_m>s_m^\prime$. This completes the proof in Case $1$ for $\Pi_D(t)$ a leaf different 
of $\Pi_D(0)=\Pi_D(1)$. The same argument as above applies to the case $\Pi_D(0) = \Pi_D(1)$ after re-rooting the looptree at time $U_1$. As before, we can disconnect $\Pi_D(0) = \Pi_D(1)$ from $\rho_*$ using simple geodesics starting from faces near $\Pi_D(0) = \Pi_D(1)$. The details are left to the reader.

\textsc{Case 2: The point $\Pi_d(t)$ is a pinch point}. We can use a similar argument to the previous case, but it requires considering 8 points to trap the geodesic $\gamma$. To avoid this technicality, we can deduce Case 2 directly from Case 1 using topological considerations. Recall that  the skeleton of the looptree is the set $\mathrm{Skel}:=\{\Pi_{d}(r) : r \in [0,1]\:\:\text{is a pinch point time}\}$. Using the same arguments as in the proof of Proposition~\ref{prop:timeclassification}, it is straightforward to verify that  $\mathrm{Skel}$ is totally disconnected -- meaning that  the connected subsets are singletons. Furthermore, by a compactness argument  combined with Theorem~\ref{main_theorem_topology} and Proposition~\ref{pinch_points_are_not_record}, it follows that the canonical projection $\texttt{p} : \mathcal{L}\to \mathcal{S}$ realizes an homeomorphism between $\mathrm{Skel}$ and its image $\texttt{p} (\mathrm{Skel})=\Pi_D(r \in [0,1]\:\:\text{is a pinch point time})$. This entails that the set $\texttt{p} (\mathrm{Skel})$ is also totally disconnected, implying that the geodesic $\gamma$ must take values outside of $\texttt{p} (\mathrm{Skel})$ at any non-empty open interval. Therefore, there exists some $\eps^\prime<\eps$ such that $\gamma(\eps^\prime)=\Pi_D(s)$ for some leaf time $s\in[0,1]$.  We can then apply Case 1 to  derive that $\gamma(\eps^\prime+\cdot)$ is a simple geodesic, which contradicts \eqref{pince:gamma}.
  \end{proof}
  We conclude this section with a discussion of some consequence of Proposition \ref{thm:geodesics}. For simplicity, as in the previous proof, let $\texttt{p} : \mathcal{L}\to \mathcal{S}$ denote the canonical projection. Proposition \ref{thm:geodesics}, combined with Proposition \ref{pinch_points_are_not_record}, shows that $\mathbf{P}$-almost surely, for any point $u \in \mathcal{L}$ in the looptree, the number of distinct geodesics from $\texttt{p} (u)$ towards $\rho_{*}$ is equal to the degree of $u$ in $ \mathcal{L}$ (see Section \ref{sec:defloop}) which belongs to $\{1,2\}$. In particular, the \textbf{cut locus}\footnote{There are different notions of cut locus depending on the regularity of the underlying space, here we follow \cite{angel2017stability}.} of $ \mathcal{S}$  relative to the root $\rho_{*}$, i.e. the set of all points that are connected to $\rho_{*}$ by at least two distinct geodesics, is exactly $\texttt{p} ( \mathrm{Skel})$, where we recall that $ \mathrm{Skel}$ is the skeleton  of $\mathcal{L}$.  Furthermore, as seen in the  proof of Proposition \ref{thm:geodesics}, the set $\texttt{p} ( \mathrm{Skel})$ is in fact totally disconnected, i.e.~the connected subsets are singletons. This situation contrasts with the Brownian sphere case where the cut locus is a tree, with uncountably many regular points of order $2$, from which we can start two distinct geodesics towards $\rho_*$, and a countable dense set of branch points of order $3$, from which  we can start three such geodesics, see \cite{LG09}.
  
More generally, this raises the question of studying \textbf{geodesic networks} in the $\alpha$-stable carpet/gasket, i.e.~the possible topology of the set of all geodesics linking two points. This was first addressed by Angel, Kolesnik \& Miermont \cite{angel2017stability}, and then considerably developed by Miller \& Qian \cite{miller2020geodesics} and Le Gall \cite{LGstar} in the Brownian sphere case. Some universality is expected  for planar random metrics and a basis of $27,28$ or $29$ networks are expected to show up in Brownian geometry, in Liouville Quantum metric \cite{gwynne2021geodesic}, in Aldous-Kendall planar metrics \cite{blanc2024geodesics}  and in the directed landscape \cite{dauvergne202327}. However, our situation differs, as the $\alpha$-stable carpets and gaskets are not homeomorphic to the sphere.   
  
  \subsection{Estimates for good points along typical geodesics} \label{sub:goodpointestimate}
We  finally come to the proof of Proposition \ref{main:techni} on the density of good points along typical geodesics, which we recall here for the reader's convenience:
There exist  constants $c,C\in(0,\infty)$ such that
$$  \mathbf{E} \left[ \int_{0}^{D( \rho_{1}, \rho_{2})} \mathrm{d}u \  \mathbbm{1}_{\{\gamma_{1,2}(u)  \mbox{ is $ \varepsilon$-bad}\,\}} \right] \leq C \cdot \eps^{c},\quad \text{for every } \eps>0.$$
As in the previous section, since the derivation of the last display will use the biased version of $( \mathcal{S},D)$, one first needs to check measurability issues and begin with some remarks on the concept of good points in a fixed geodesic metric space. \medskip 

In this direction, consider $\mathbf{M}:=(M,d_M,\mu,x,(x_1,x_2))$ a rooted bi-marked weighted geodesic compact space with a unique geodesic  $\gamma_{1,2}$  going from $x_1$ to $x_2$. Generalizing the notion introduced in Section \ref{sec:D=D*}, we say that $u\in [0,d_M(x_1,x_2)]$ is an $\varepsilon$-good time (for $\mathbf{M}$) if $\gamma_{1,2}(t)$ for $(u-\varepsilon)\vee 0\leq t\leq (u+\varepsilon)\wedge d_M(x_1,x_2)$,  coincides with the concatenation of at most  two portions of geodesics towards the root $x$. Otherwise, we say that $u$ is an $\eps$-bad time (for $\mathbf{M}$), and we use  the same terminology for the point $\gamma_{1,2}(u)$. 
We now relate this notion with that of \textbf{aligned points}. Specifically, three points ${y_1, y_2, y_3}$ are said to be aligned if they all lie on the range of the same geodesic. This condition can be expressed equivalently in terms of the metric: $d_M(y_{\tau(1)}, y_{\tau(3)}) = d_M(y_{\tau(1)}, y_{\tau(2)}) + d_M(y_{\tau(2)}, y_{\tau(3)})$ for some permutation $\tau$ of the indices ${1, 2, 3}$. The connection with good points is formalized in the following lemma for which we introduce some notation, and refer to Figure \ref{fig:aligned} for an illustration. Fix $u \in (0, d(x_1,x_2))$ as well as $ \varepsilon>0$ and decompose ${\gamma_{1,2}}$ into the three parts 
  \begin{eqnarray*}L_{ \varepsilon} &:=&   \big\{ {\gamma_{1,2}}(t) : t \in  [0, (u - \varepsilon)\vee 0]\big\},\\
  C_{ \varepsilon} &:=& \big\{ {\gamma_{1,2}}(t) : t \in ( (u-  \eps)\vee 0,(u+\eps)\wedge {d}( {x}_1, {x}_2))  \big\}, \\
  R_{ \varepsilon} &:=&   \big\{ {\gamma_{1,2}}(t) : t \in [ (u+\eps)\wedge {d}( {x}_1, {x}_2), {d}( x_{1},{x}_{2})]  \big\}  \end{eqnarray*} thus isolating the $  \varepsilon$-neighborhood $C_{ \varepsilon}$ of ${\gamma_{1,2}}(u)$ inside  ${\gamma_{1,2}}$. We write $y = \gamma_{1,2}(u)$ and $$ \qquad y_1^{ (\varepsilon)} = \gamma_{1,2} \left( (u-\varepsilon) \vee 0 \right), \qquad y_2^{ (\varepsilon)} = \gamma_{1,2} \left( (u+\varepsilon)\wedge d(x_1,x_2)\right)$$ for the extremities of $C_{ \varepsilon}$ inside $\gamma_{1,2}$. Then we have: 
\begin{lem}\label{lem:stars}
With the notation above, the point $y$ is $\eps$-good if and only if for every $z\in C_{\varepsilon}$, the points $\{z,y_1^{ (\varepsilon)},x\}$ or the points $\{z,y_2^{ (\varepsilon)},x\}$ are aligned. Moreover, if $y$ is $\eps$-bad, then there exists  $z' \in C_{ \varepsilon}$ such that the geodesic from $z'$ to $x$ is non-trivial\footnote{A trivial geodesic is a constant geodesic with length $0$.} and only intersects the range of $\gamma_{1,2}$ at its starting point~$z'$.
\end{lem}
  \begin{figure}[!h]
  \begin{center}
  \includegraphics[width=10cm]{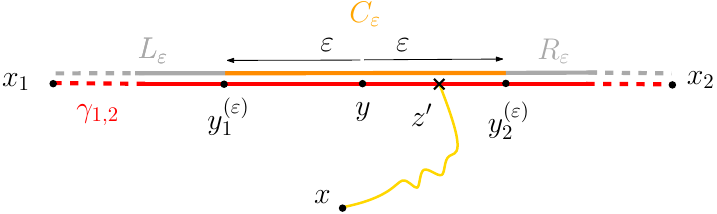}
  \caption{Illustration of the notation for Lemma \ref{lem:stars}. On this picture, $y$ is an $\eps$-bad point, because there exists a geodesic starting from the point $z'\in C_\eps$ that immediately leaves $\gamma_{1,2}$. 
    \label{fig:aligned}}
  \end{center}
  \end{figure}
  \begin{rek}[Geodesic stars] In particular, with the notation of the lemma, the point $z'$ is a $3$-star point along $\gamma_{1,2}$, where we recall that a \textbf{ $k$-geodesic star} is a point from which  which we can start $k$ disjoint geodesics (appart from their starting point).  These points have been studied in depth in the case of the Brownian geometry \cite{angel2017stability,LGstar,Mie11,miller2020geodesics} and in related contexts \cite{blanc2024geodesics,dauvergne202327}.
  \end{rek} 

 \begin{proof}  
Notice first that since there is a unique geodesic between $x_{1}$ and $x_{2}$, if a geodesic $\gamma$ intersects  ${\gamma_{1,2}}$ at two points, then it must actually coincide with it in-between. A geodesic $\gamma$ to $x$ starting from $C_{ \varepsilon}$ must then stick to a part of ${\gamma_{1,2}}$ before leaving ${\gamma_{1,2}}$ at some point $x_{\gamma}$, and then go to $x$ without crossing $ {\gamma_{1,2}}$ again. Using this remark, if there exists $z \in C_{\varepsilon}$ such that neither $\{z,y_1^{ ( \varepsilon)},x\}$ nor $\{z,y_2^{ ( \varepsilon)},x\}$ are aligned, then a geodesic $\gamma$ going from $z$ to $x$ necessary leaves $\gamma_{1,2}$ at  $x_{\gamma}\in C_{ \varepsilon}$. As a consequence, the part of $\gamma$ from $z'=x_{\gamma}$ to $x$ gives a non-trivial  geodesic only intersecting the range of $\gamma_{1,2}$ at its starting point $x_{\gamma}\in C_\varepsilon$, as claimed in the statement of the lemma. Similarly if $y$ is $ \varepsilon$-good, then $C_{ \varepsilon}$, and its closure, is covered by the range of  two geodesics towards the root $x$. Those must form two closed intervals by the above remark, and we infer that for any $z \in C_{ \varepsilon}$ either $\{z, y_1^{ (\varepsilon)},x\}$ or $\{z, y_2^{ (\varepsilon)},x\}$ belong to one of the above two geodesics and are thus aligned.

  \begin{figure}[!h]
  \begin{center}
\includegraphics[width=12cm]{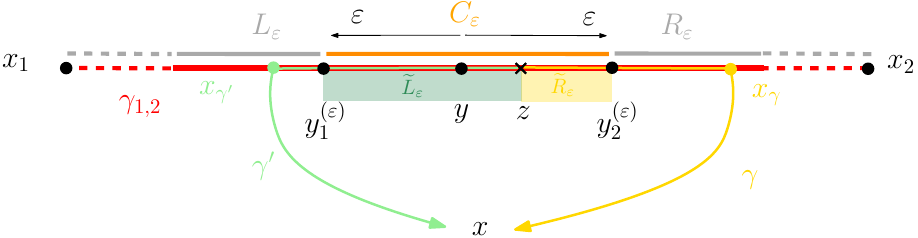}
  \caption{Illustration of the proof of Lemma \ref{lem:stars}. If any point of $C_{\varepsilon}$ is aligned with either $\{y_1^{ (\varepsilon)},x\}$ or $\{y_2^{ (\varepsilon)},x\}$, then we can cover $ C_{ \varepsilon}$ with the range of two geodesics $\gamma, \gamma'$ towards the root $x$. 
  \label{fig:3stars}
  }
  \end{center}
  \end{figure}

To conclude, it remains to show that if for every $z\in C_\varepsilon$  the points  $\{z,y_1^{ ( \varepsilon)},x\}$ or the points $\{z,y_2^{ ( \varepsilon)},x\}$ are aligned, then $y$ is necessarily  $\varepsilon$-good. 
In this direction,  consider   the set $ \widetilde{L}_{ \varepsilon}$ (resp.\ $ \widetilde{R}_{ \varepsilon}$) of all points in the closure of $C_{  \varepsilon}$ which are aligned with $\{y_1^{ (\varepsilon)},x\}$ (resp.\ $\{y_2^{ (\varepsilon)},x\}$), that is from which there is a  geodesic  $\gamma$ to $x$ that leaves ${\gamma_{1,2}}$ on the left part $L_{ \varepsilon}$, i.e. $x_{\gamma}\in L_{\varepsilon}$ (resp.~$x_\gamma\in R_{\varepsilon}$). By our assumption these two sets cover the closure of $C_{ \varepsilon}$. If  the right extremity $y_2^{( \varepsilon)}$ of $C_{ \varepsilon}$ belongs to  $ \widetilde{L}_{ \varepsilon}$ then there exists a geodesic  $\gamma$ going to $x$ which coincides with ${\gamma_{1,2}}$ on $C_{ \varepsilon}$, this  implies that $y$ is $ \varepsilon$-good. Otherwise, by the Hopf-Rinow theorem we can find a point $z \in C_{ \varepsilon}$ from which we can start two geodesics to $x$, one leaving ${\gamma_{1,2}}$ on $L_{ \varepsilon}$ and one on $R_{ \varepsilon}$. This implies in particular that $\mathrm{Cl}(C_{ \varepsilon})$ is covered by the range of these two geodesics targeting $x$, this  again gives that $y$ is $ \varepsilon$-good.
\end{proof}
Let us now use the above lemma to prove that the random variables considered in Proposition~\ref{main:techni} are indeed measurable.  To this end, for every  $\mathbf{M}:=(M,d_M,\mu,x,(x_1,x_2))$  rooted bi-marked weighted geodesic compact space  with exactly one geodesic going from $x_1$ to $x_2$ and $u\geq 0$, we  set $\mathrm{Bad}_{\varepsilon}(\mathbf{M}, u):=\mathbbm{1}_{u  \mbox{\small ~is a $ \varepsilon$-bad time for $\mathbf{M}$} }$, with the convention $\mathrm{Bad}(\mathbf{M},u):=0$ if $u>d_M(x_1,x_2)$.  This function is invariant by isometry and therefore we can view it as a maps from $\mathbb{PM}^{2\bullet, 1}_{\mathrm{root}}\times \mathbb{R}_+$ to $\mathbb{R}_+$,  that we still denote by $\mathrm{Bad}_{\varepsilon}$ with a slight abuse of notation.

\begin{lem}\label{Good-Bad-measurable}
The function   $\mathrm{Bad}_{\varepsilon}: \mathbb{PM}_{\mathrm{root}}^{2\bullet,1}\times\mathbb{R}_+\to \{0,1\}$ is measurable.
\end{lem}

\begin{proof}
It suffices to establish that the function  $\mathrm{Bad}_{\varepsilon}^\prime: \mathbb{PM}_{\mathrm{root}}^{2\bullet,1}\times[0,1]\to \{0,1\}$ defined by $\mathrm{Bad}_{\varepsilon}^\prime(\mathbf{M}, v)= \mathrm{Bad}_{\varepsilon}^\prime(\mathbf{M}, v\cdot d_M(x_1,x_2))$ is measurable; here $d_M(x_1,x_2)$ stands for the distance between the two marks which is a quantity invariant by isometry and it is thus well defined. We are going to prove that $\mathrm{Bad}_{\varepsilon}^\prime$ is lower semicontinuous, which directly implies the desired result. Let  $(\mathbf{M}, v)\in\mathbb{PM}_{\mathrm{root}}^{2\bullet,1}\times[0,1]$ and  $(\mathbf{M}_n,v_n)$, $n\geq 1$, a sequence in  $\mathbb{PM}_{\mathrm{root}}^{2\bullet,1}\times[0,1]$ converging to $(\mathbf{M}, v)$. We need to show that:
$$\liminf_{n\to \infty} \mathrm{Bad}_{\varepsilon}^\prime(\mathbf{M}_n,v_n)\geq \mathrm{Bad}_{\varepsilon}^\prime(\mathbf{M},v).$$
Moreover, since $ \mathrm{Bad}_{\varepsilon}$ takes values on $\{0,1\}$,  it suffices to consider the case when $\mathrm{Bad}_{\varepsilon}(\mathbf{M}_n,v_n)=0$, for every $n\geq 1$. In this direction, we recall that by \cite[Lemma 2.5]{khezeli2023unified}, we may assume that all the spaces $(M_n,d_{M_n},\mu_n, x^n,(x_1^n ,x_2^n))$, $n\geq 1$, and $(M,d_M,\mu, x, (x_1,x_2))$  are  isometrically embedded into the same compact metric space $(Z, \delta)$, and
\begin{equation}\label{eq:convergence:embedding:bad:good}
M_n\to M, \quad \mu_n\to \mu, \quad (x^n, x_1^n,x_2^n)\to (x,x_1,x_2),  
\end{equation}
as $n\to \infty$.
Next, in this embedding, consider $\gamma$ (resp.\ $\gamma_n$, $n\geq 1$) the unique geodesic staying in  $M$ (resp.\ $M_n$, $n\geq 1$) going from $x_1$ to $x_2$ (resp.\ $x_1^n$ to $x_2^n$, $n\geq 1$). We claim that for every $t_n\in[0,d_{M_n}(x_1^n,x_2^n)]$, $n\geq 1$, converging to some $t\in [0,d_M(x_1,x_2)]$ we must have $\gamma_n(t_n)\to \gamma(t)$. To see why the claim holds, remark that by compactness, it is equivalent to establish  that  $\gamma(t)$ is the unique limit point of   $(\gamma_n(t_n))_{n\geq 1}$.  In this direction, fix $z$ a limit point and note that necessarily  $z\in M$, by the convergence $M_n\to M$.  Next, observe  that since  $\gamma_n$, $n\geq 1$, are geodesics we must have  $\delta(x_1^n ,\gamma_n(t_n))= t_n$ and  $\delta(x_2^n,\gamma_n(t_n))=  \delta(x_1^n,x_2^n)-t_n$. Hence, passing to the limit, we obtain $\delta(x_1 ,z)= t$ and  $\delta(x_2,z)= \delta(x_1,x_2)-t$. This implies that $z$ is in a geodesic  taking values on $M$ and going from $x_1$ to $x_2$ . Since $(M,d,\mu, x, (x_1,x_2))$ only has one such geodesic, it follows  that $z=\gamma(t)$. Next set $t_n^-:=\big(v_n\cdot d_{M_n}(x_1^n,x_2^n)-\varepsilon\big)\vee 0 $ and $t_n^+:=\big(v_n\cdot  d_{n}(x_1^n,x_2^n)+\varepsilon\big)\wedge d_{M_n}(x_1^n,x_2^n)$. Similarly, let $t^{-}:=(v \cdot d_M(x_1,x_2)-\varepsilon)\vee 0 $ and $t^{+}:= (v\cdot d_M(x_1,x_2)+\varepsilon)\wedge d_M(x_1,x_2)$. Consider now $t^{-}\leq t\leq t^{+}$ and remark that, by \eqref{eq:convergence:embedding:bad:good} and since $v_n\to v$, we can always find  $t^{-}_n\leq t_n\leq t^{+}_n$, for every $n\geq 1$, such that $t_n\to t$, as $n\to \infty$. Then, by the previous discussion, we must have:
\begin{equation}\label{eq:final:Bad:Good:measurable}
\big(\gamma_n(t_n^-), \gamma_n(t_n), \gamma_n(t_n^+)\big)\to \big(\gamma(t^-), \gamma(t), \gamma(t^+)\big),\quad n\to \infty.
\end{equation}
Since $\mathrm{Bad}_{\varepsilon}^\prime(\mathbf{M}_n,v_n)=0$, for $n\geq 1$, then by definition it holds that, for every $n\geq 1$, the points $\{\gamma_n(t_n^-), \gamma_n(t_n), x^n\}$ or $\{\gamma_n(t_n^+), \gamma_n(t_n), x^n\}$ are aligned in $M_n$. Hence, by the convergences \eqref{eq:convergence:embedding:bad:good} and \eqref{eq:final:Bad:Good:measurable}, we derive that the points $\{\gamma(t^-), \gamma(t), x\}$ or $\{\gamma(t^+), \gamma(t), x\}$ are aligned in $M$. Since this holds for any arbitrary $t\in [t^-,t^+]$, an application of Lemma \ref{lem:stars} gives  $\mathrm{Bad}_{\varepsilon}^\prime(\mathbf{M},v)=0$, completing the proof.
\end{proof}

Coming back to Proposition \ref{main:techni}, recall that, under $\mathbf{P}$, the random variables:
$$\boldsymbol{\mathcal{S}}:=\big(\mathcal{S},D, \mathrm{Vol}, \rho_*,  (\rho_{1}, \rho_{2})\big)  \quad \text{and } \quad \boldsymbol{\widehat{\mathcal{S}}}:=\big(\widehat{\mathcal{S}},\widehat{D},\widehat{\mathrm{Vol}}, \widehat{\rho}_*,  (\widehat{\rho}_{1}, \widehat{\rho}_{2})\big),$$
are a.s.\ in $\mathbb{PM}^{2\bullet,1}_{\mathrm{root}}$. In particular,  Lemma \ref{Good-Bad-measurable} entails that 
$$ \int_{0}^{D( \rho_{1}, \rho_{2})} \mathrm{d}u \  \mathbbm{1}_{\{\gamma_{1,2}(u)  \mbox{ is $ \varepsilon$-bad}\}}= \int_{0}^{\infty} \mathrm{d} u ~\mathrm{Bad}_{\eps}(\boldsymbol{\mathcal{S}}, u)$$ 
is a well defined random variable. Furthermore, since conditionally on $\big(\mathcal{S},D, \mathrm{Vol}, \rho_*,  (\rho_{1}, \rho_{2})\big)$,  the random variable $\delay$ is uniformly distributed in $[-D(\rho_1,\rho_2), D(\rho_1,\rho_2)]$,  Proposition \ref{main:techni}  is equivalent to the existence of two constants $c,C\in (0,\infty)$ such that:
  \begin{equation}  \label{eq:newgoal:pre} \mathbf{E} \Big[  D(\rho_1,\rho_2)\cdot \mathbbm{1}_{\{(D(\rho_1,\rho_2)-\delay)/2 \mbox{ is $\varepsilon$-bad time for } \boldsymbol{\mathcal{S}} \}} \Big] \leq C\cdot \eps^{c},\quad \text{for every } \eps>0.  \end{equation}
By \eqref{eq:defShat}, the latter  can be finally be rewritten in the form:
  \begin{equation} \label{eq:newgoal}   \mathbf{P} \Big( \, \frac{\widehat{D}(\widehat{\rho}_1,\widehat{\rho}_2)-\hdelay}{2}  \mbox{ is an $ \varepsilon$-bad time for } \boldsymbol{\widehat{\mathcal{S}}}\,\Big) \leq C\cdot \eps^{c},\quad \text{for every } \eps>0, \end{equation}
for possibly two other constants. Here we use again the fact that the random variable and event appearing in \eqref{eq:newgoal:pre} and \eqref{eq:newgoal} are well defined and measurable by Lemma \ref{Good-Bad-measurable}. 

We aim to prove \eqref{eq:newgoal} by connecting it to the framework established in Section \ref{sec:prison} and the results presented therein. In this direction, recall that $( \widehat{X}, \widehat{Z}, \widehat{t}_\bullet, \widehat{\sigma})$ under $\mathbf{P}$ is distributed as $(X, Z, t_\bullet, \sigma)$ under $ \mathrm{cst}\cdot\GG\big(1-\sigma,-Z_{t_{\bullet}}\big)\cdot \mathbf{N}^{\bullet}$.
As a key consequence, we adopt the notation from Part \ref{PartI} for the process $( \widehat{X}, \widehat{Z}, \widehat{t}_\bullet, \widehat{\sigma})$ by adding a  ``~$\widehat{}$~'' in the notation. For instance, we write $\widehat{\varpi} \leq \widehat{t}_{\bullet}$ for the unique pinch point time that minimizes the label along the spine (which is unique by Lemma \ref{lem:onePinch}), and we denote by  $  \widehat{\mathfrak{X}}$ its image in the looptree coded by $ \widehat{X}$. To obtain \eqref{eq:newgoal} it suffices to establish that:

  \begin{eqnarray}  \label{eq:gogo1} \mathbf{P}\Big( \Big\{\widehat{\mathfrak{X}} \text{ is  }2\varepsilon\text{-trapped} \Big\} \cap \Big\{\frac{(\widehat{D}(\widehat{\rho}_1,\widehat{\rho}_2)-\hdelay)}{2}  \text{ is an $\varepsilon$-bad time for }\boldsymbol{\widehat{\mathcal{S}}}\,\Big\} \Big) = 0,  \end{eqnarray}
and 
 \begin{eqnarray} \label{eq:gogo2} \mathbf{P}(\widehat{\mathfrak{X}} \text{ is not }\varepsilon\text{-trapped}) \leq C \cdot  \varepsilon^c,  \end{eqnarray} for yet possibly other constants $c,C\in(0,\infty)$.

The proof of \eqref{eq:gogo2} is based on the estimates of Section~\ref{sec:prison}, more precisely Theorem \ref{prop:malo},  and can be found at the end of the section (Lemma \ref{lem:derdesder}). We shall thus start with \eqref{eq:gogo1}. As in Section~\ref{sec:uni:geo}, this is addressed by translating the problem in terms of planar maps, for which we use the construction from well-labeled unicyclomobiles.  In this direction, recall the notation $\Omega^{\prime\prime}$ introduced at the end of Section~\ref{sec:uni:geo}, which has full probability, and fix $\omega\in \Omega^{\prime\prime}$.  In particular,  $(  \widehat{\mathcal{S}},\widehat{D},\widehat{\mathrm{Vol}}, \widehat{\rho}_*,  (\widehat{\rho}_{1}, \widehat{\rho}_{2}))$ is embedded  on a compact space $(E,d_E)$, and has a unique geodesic $\widehat{\gamma}_{1,2}$ going from  $\widehat{\rho}_1$ to $\widehat{\rho}_2$ (staying in $\widehat{\mathcal{S}}$).  Let us now give the underlying geometric idea  which is best illustrated by putting Figures~\ref{fig:trap-2}~and~\ref{fig:trapped} side by side,  see also Figure~\ref{fig:sidebyside} and its caption. For an interval $I \subset \mathbb{R}$, we  write 
$$ \Gamma_I = \widehat{\gamma}_{1,2} \Big( \Big(I +  \frac{1}{2}\big( \widehat{D}(\widehat{\rho}_1,\widehat{\rho}_2) - \widehat{\Delta}\big)\Big) \cap [0,  \widehat{D}(\widehat{\rho}_1,\widehat{\rho}_2)]\Big),$$ which corresponds to the $I$-neighborhood of $ \mathbb{J}$ in $\widehat{\gamma}_{1,2}$ where $ \mathbb{J}$ correspond to $$ \mathbb{J}=\Gamma_{[0]} =  \widehat{\gamma}_{1,2} \Big(\frac{1}{2}\big( \widehat{D}(\widehat{\rho}_1,\widehat{\rho}_2) - \widehat{\Delta}\big)\Big).$$
For simplicity, we write $\Gamma$ for $\Gamma_{\mathbb{R}}$. Heuristically, when $ \widehat{\mathfrak{X}}$ is $ \varepsilon$-trapped in the sense of Section \ref{sec:prison}, then there exist two faces $  \mathfrak{F}_{1}, \mathfrak{F}_{2}$ of $ ( \widehat{\mathcal{S}},\widehat D)$, corresponding respectively to the times $\widehat{r}_{1,4}$ and $\widehat{r}_{2,3}$ therein,  such that the piece $ \Gamma_{(- \infty, \varepsilon]}$ of $\widehat{\gamma}_{1,2}$ going to $\widehat{\rho}_{1}$ touches both $ \mathfrak{F}_{1},\mathfrak{F}_{2}$, and similarly the piece $ \Gamma_{[ \varepsilon, \infty)}$ of $\widehat{\gamma}_{1,2}$ going to $\widehat{\rho}_{2}$  of $\widehat{\gamma}_{1,2}$ going to $\widehat{\rho}_{2}$ touches both $ \mathfrak{F}_{1},\mathfrak{F}_{2}$ (in any order), in such a way that the
 parts of these sub-geodesics in-between $ \mathfrak{F}_{1}, \mathfrak{F}_{2}$ separate the root $\widehat{\rho}_{*}$ and $\mathbb{J}$. We call  these sub-geodesics  the left and right ``doors'', see Figure~\ref{fig:sidebyside} for an illustration. In presence of such doors, it is impossible for a geodesic as in Lemma \ref{lem:stars} to exist, and a fortiori the point $ \mathbb{J}$ is $  \varepsilon$-good.\bigskip

\begin{figure}[!h]
 \begin{center}
 \includegraphics[height=9cm]{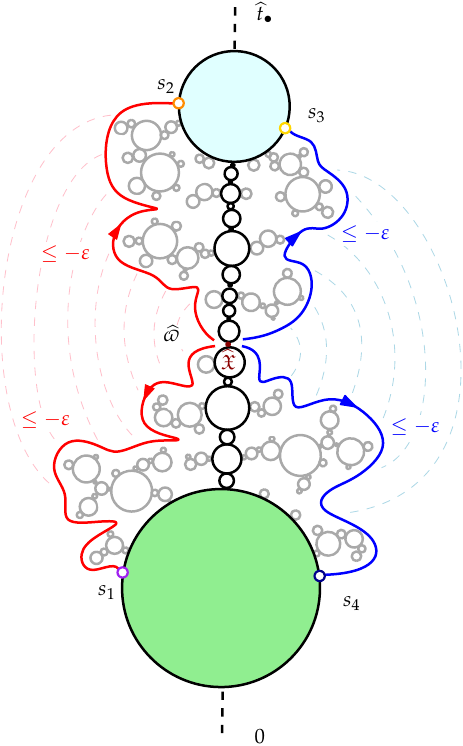} \hspace{0.5cm}  \includegraphics[height=9.5cm]{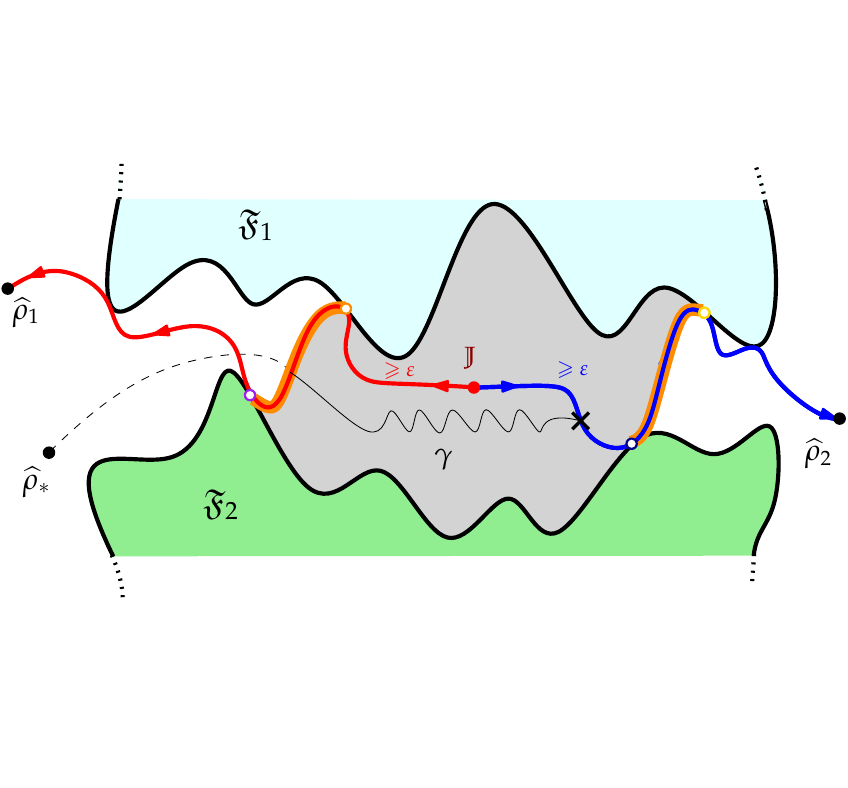}
 \caption{On the left, the $ \varepsilon$-trapped condition for the point $ \widehat{\mathfrak{X}}$ imported from Section \ref{sec:prison} and its geometric consequence on the right figure in $( \widehat{ \mathcal{S}}, \widehat{D})$. The color code is the same in the left and right pictures, in particular,  the geodesic $ \widehat{\gamma}_{1,2}$ is constructed by concatenating the simple geodesics starting from $ \widehat{\mathfrak{X}}$ in each face and the $ \varepsilon$-trap condition shows that those geodesics hit two faces separating the $ \varepsilon$-neighborhood of $  \mathbb{J}$ from $ \widehat{\rho}_{*}$. In such a situation, it is  impossible for a geodesic $\gamma$ to start  from the $ \varepsilon$-neighborhood of $   \mathbb{J}$ and targeting $ \widehat{\rho}_{*}$ without crossing $\widehat{\gamma}_{1,2}$. This contradicts Lemma \ref{lem:stars} and thus implies that $  \mathbb{J}$ is $ \varepsilon$-good. \label{fig:sidebyside}}
 \end{center}
 \end{figure}

 Let us now proceed with the formal proof and recall that we have fixed $\omega\in \Omega^{\prime \prime}$.  We start discussing the implication of $(\widehat{D}(\widehat{\rho}_1,\widehat{\rho}_2)-\hdelay)/2$ being $\varepsilon$-bad for $\boldsymbol{\widehat{\mathcal{S}}}$ in terms of the planar maps $( \widehat{\mathfrak{M}}_{m_k}, \widehat{v}_{1}^{m_k},  \widehat{v}_{2}^{m_k},  \hdelay_{m_k})$. In this direction, recall the construction from well-labeled unicyclomobiles,  $( \widehat{\mathfrak{M}}_{m_k}, \widehat{v}_{1}^{m_k},  \widehat{v}_{2}^{m_k},  \hdelay_{m_k}) = \mathrm{BDG}^{2\bullet}( \mathbf{u}_{m_k},  \epsilon)$ of Section \ref{sec:boltzm-stable-maps}. As in the proof of Lemma \ref{lem:alm-unique}, for every $k\geq 1$, we fix a white vertex  $J_{m_k}$ attaining the minimal label along the cycle of $\mathbf{u}_{m_k}$, and write $\widehat{\gamma}_{1,2}^{m_k}$ for the geodesic in the map between $\widehat{v}_1^{m_k}$ and $\widehat{v}_2^{m_k}$ obtained by iterating the successor function in each face. Specifically, we see the belt part $\boldsymbol{\mathcal{B}}_{m_k}$ (with its labels) as a submap of $\mathbf{u}_{m_k}$, and for every white corner $c$, we write $\gamma^{(c)}$ for the geodesic obtained by iterating the successor  function. Then, we let $c^{m_k}_f$ and $c^{m_k}_\ell$ be respectively the first and last corner of $J_{m_k}$ in the contour order of $\boldsymbol{\mathcal{B}}_{m_k}$. In particular, $\gamma^{(c^{m_k}_f)}$ (resp.\ $\gamma^{(c^{m_k}_\ell)}$) defines a geodesic staying in the face $f_1^{m_k}$ (resp.\ $f_2^{m_k}$), starting in $J_{m_k}$ and ending in $\widehat{v}_1^{m_k}$ (resp.\ $\widehat{v}_2^{m_k}$). The geodesic $\widehat{\gamma}_{1,2}^{m_k}$ is obtained by reversing $\gamma^{(c^{m_k}_f)}$ and concatenating  it with $\gamma^{(c^{m_k}_\ell)}$.
 \par
We begin with a preliminary lemma, translating the fact that $(\widehat{D}(\widehat{\rho}_1,\widehat{\rho}_2)-\hdelay)/2$ is $ \varepsilon$-bad  into a geometric condition in the discrete maps: 
 
 \begin{lem}\label{chemin:proches}
With the notation above, let $\eps>0$ and  assume that $(\widehat{D}(\widehat{\rho}_1,\widehat{\rho}_2)-\hdelay)/2$ is an $ \varepsilon$-bad time for $ \boldsymbol{\widehat{ \mathcal{S}}}$. Then for every $\omega\in \Omega^{\prime\prime}$, we can find $\eta\equiv \eta(\omega)>0$ and a sequence, of paths $\gamma_{m_k},k\geq 1$, respectively in $\widehat{\mathfrak{M}}_{m_k}$, from  a point of $B_{\mathrm{d}^{ \mathrm{gr}}_{\widehat{ \mathfrak{M}}_{m_k}}}(J_{m_k},\frac{3}{2}\eps (\scal m_k)^{\frac{1}{2 \alpha}})$ to $\widehat{\rho}_*^{m_k}$, such that, for $k$ large enough, the path $\gamma_{m_k}$ stays at graph  distance at least $\eta (\scal m_k)^{\frac{1}{2 \alpha}}$ from $\widehat{\gamma}_{1,2}^{m_k}$ in $\widehat{ \mathfrak{M}}_{m_k}$.
\end{lem}

\begin{proof}
Fix $\varepsilon>0$ and $\omega\in \Omega^{\prime\prime}$. Recall that the spaces in \eqref{eq:V:M:n:k:prim} and \eqref{eq:S:tilde} are embedded in $(E,d_{E})$, $\widehat{\gamma}_{1,2}$ is the unique geodesic going from $\widehat{\rho}_1$ to $\widehat{\rho}_2$, the notation $\Gamma=\widehat{\gamma}_{1,2}([0,\widehat{D}(\widehat{\rho}_1,\widehat{\rho}_2)])$,  and:
\begin{equation}\label{eq:V:M:k:convergence:E:lemma:geo:dis}
V( \widehat{\mathfrak{M}}_{m_k})\to\widehat{ \mathcal{S}}\quad \text{ and }\quad \big(\widehat{\rho}_{m_k}, \widehat{v}_{1}^{m_k}, \widehat{v}_{2}^{m_k}\big)\to\big(\widehat{\rho}_*, \widehat{\rho}_{1}, \widehat{\rho}_{2}\big), \quad \text{as} \quad k\to \infty.
\end{equation}
 Now we claim that the trace of $\gamma^{m_k}_{1,2}$ in $V( \widehat{\mathfrak{M}}_{m_k})$ is  arbitrarily close to $\Gamma$, asymptotically. Namely, let $\Gamma_k$ be the set of all points of $V( \widehat{\mathfrak{M}}_{m_k})$ corresponding to the vertices appearing in the range of  $\gamma^{m_k}_{1,2}$. Then we claim that:
\begin{equation}\label{eq:distance:range}
\lim \limits_{k\to \infty} \sup\big\{ d_{E}\big(x, \Gamma\big):~x\in \Gamma_k \big\}=0.
\end{equation}
To see this, remark that by compactness it suffices to show that any adherent point $z$ of a sequence  $x_k\in \Gamma_{k},$ $k\geq 1$,  must verify $d_{E}\big(z, \Gamma\big)=0$. But by \eqref{eq:V:M:k:convergence:E:lemma:geo:dis}, any such adherent point  belongs to $\widehat{\mathcal{S}}$ and satisfies $\widehat{D}(\widehat{\rho}_1,\widehat{\rho}_2)= \widehat{D}(\widehat{\rho}_1,z)+\widehat{D}(z,\widehat{\rho}_2)$. This implies that $z$ is on a geodesic of $\widehat{\mathcal{S}}$ going from $\widehat{\rho}_1$ to $\widehat{\rho}_2$, and thus $z\in \Gamma$, since $\widehat{\gamma}_{1,2}$ is the unique such geodesic. For the remainder of the proof we assume that $d_*:=(\widehat{D}(\widehat{\rho}_1,\widehat{\rho}_2)-\hdelay)/2$ is $ \varepsilon$-bad, and notice that by the above argument we have
\begin{equation}\label{convergence:J:m:k}
J_{m_k} \to \widehat{\gamma}_{1,2}(d_*)=\mathbb{J},  \quad \text{as} \quad k\to \infty,
\end{equation}
 in the embeddings.   Then  by  Lemma \ref{lem:stars}, we know  that we can find a point $t\in[(d_*-\eps)\vee 0,(d_*+\eps)\wedge \widehat{D}(\widehat{\rho}_1,\widehat{\rho}_2)]$ such that there exists a geodesic $\widehat{\gamma}_{\perp}:[0,\widehat{D}(\widehat{\gamma}_{1,2}(t),\widehat{\rho}_*)]\to \widehat{\mathcal{S}}$ connecting $\widehat{\gamma}_{1,2}(t)$ and $\widehat{\rho}_*$, which only intersects $\Gamma$ at its starting point.  In particular, $\widehat{\rho}_*$ is not in $\Gamma$. Next set
$$\widehat{\gamma}(r):=\widehat{\gamma}_\perp\big(r+(\eps/4)\wedge \widehat{D}(\widehat{\gamma}_{1,2}(t),\widehat{\rho}_*)\big),\quad 0\leq r\leq R:=(\widehat{D}(\widehat{\gamma}_{1,2}(t),\widehat{\rho}_*)-\eps/4)\vee 0.$$
By construction and \eqref{eq:distance:range}, we can find $\eta\in(0,\eps/4)$, such that $\widehat{\gamma}$ stays at distance at least $5\eta$ from $\Gamma_{k}$, for $k$ large.  The desired result now  follows straightforwardly by  approximating $\widehat{\gamma}$ using paths in $\widehat{\mathfrak{M}}_{m_k}$, $k\geq 1$. Specifically, pick an integer $N\geq 1$ such that $R/N< \eta$. By \eqref{eq:V:M:k:convergence:E:lemma:geo:dis}, for $k$ even larger, we can find, $0\leq i\leq N$,  
a point $x_i^{k}$ in $V( \widehat{\mathfrak{M}}_{m_k})$ such that $d_E(\widehat{\gamma}(  i R /N ), x_i^{k})< \eta$, and  without loss of generality we may assume that $x_N^{k}= \widehat{\rho}_{m_k}$ since $\widehat{\gamma}(R)=\widehat{\rho}_*$. By the triangle inequality and using the fact that $\widehat{\gamma}$ is  a geodesic, we infer that 
\begin{equation}\label{eq:12:4:approx:i}
d_{E}\big(x_{i}^k, \Gamma_{m_k}\big)>4\eta\quad \text{ and }\quad d_E\big(x_i^{k}, x_{(i+1)\wedge N}^{k}\big)< 3\eta, 
\end{equation}
for every $1\leq i\leq N$. Now,  concatenate in $\widehat{\mathfrak{M}}_{m_k}$ geodesics from $x_i^{k}$ to $x_{i+1}^{k}$, for $1\leq i\leq N-1$,  and denote the obtained path on the map 
$\widehat{\mathfrak{M}}_{m_k}$ by $\gamma_{m_k}$. Then \eqref{eq:12:4:approx:i} entails that  the trace of $\gamma_{m_k}$ in $V( \widehat{\mathfrak{M}}_{m_k})$ is at Hausdorff distance in $(E,d_E)$ from $\Gamma_{k}$ bigger than $\eta$.  Furthermore, by construction, $\gamma_{m_k}$ ends at $\widehat{\rho}_{m_k}$ and  by the triangle inequality, using   that $\widehat{\gamma}_\perp$ is a geodesic starting from the $\varepsilon$-neighborhood of $\widehat{\gamma}_{1,2}(d_*)$, it plainly follows that $m_k^{-\frac{1}{2\alpha}}\cdot\mathrm{d}^{ \mathrm{gr}}_{\widehat{ \mathfrak{M}}_{m_k}}\big(\gamma_{m_k}(0), J_{m_k}\big)<5 \varepsilon/4+\eta+ d_{E}( \widehat{\gamma}_{1,2}(d_*),J_{m_k})$. Hence, applying \eqref{convergence:J:m:k} and using that $\eta\in(0,\eps/4)$, we infer that the family of paths $\gamma_{m_k}$, for $k$ large enough, verifies the statement of the lemma.
\end{proof}

We can now proceed with the proof that $\{\widehat{\mathfrak{X}} \text{ is  }2\varepsilon\text{-trapped} \}$ and   $\{(\widehat{D}(\widehat{\rho}_1,\widehat{\rho}_2)-\hdelay)/2  \text{ is  $\varepsilon$-bad for }\boldsymbol{\widehat{\mathcal{S}}}\,\}$ are incompatible under $\mathbf{P}$.
\begin{prop} \label{lem:prison-discrete} Fix $\omega \in \Omega^{\prime\prime}$. With the notation above, it is impossible for $\widehat{\mathfrak{X}}$ to be $2\varepsilon$-trapped and for $(\widehat{D}(\widehat{\rho}_1, \widehat{\rho}_2) - \hdelay)/2$ to be $\varepsilon$-bad for $\boldsymbol{\widehat{\mathcal{S}}}$ simultaneously.
\end{prop}

\begin{proof}
Fix $\omega\in \Omega^{\prime\prime}$ and assume that $\widehat{\mathfrak{X}}$ is $2\varepsilon$-trapped.  The proof reduces to show that necessarily $\omega\notin\{(\widehat{D}(\widehat{\rho}_1,\widehat{\rho}_2)-\hdelay)/2  \text{ is  $\varepsilon$-bad for }\boldsymbol{\widehat{\mathcal{S}}}\,\}$. Our goal is to construct the discret equivalent of the doors describe in Figure \ref{fig:sidebyside}, which for maps are just two paths, and then use the Jordan theorem to deduce that these doors disconnect a small ball centered at $J_{m_k}$ from $\widehat{\rho}_*^{m_k}$, making impossible the existence of the family of paths $\gamma_{m_k}$, $k\geq 1$,   in Lemma \ref{chemin:proches}. Let us proceed, and in this direction recall the notation  $\theta_{m_k}$ and $a_{m_k}$ from \eqref{eq:dist2}.

 \noindent \textsc{Building the discrete doors.}  
 For every integer $i\leq \theta_{m_k}$ we denote by $w_i^{m_k} \in V_\circ( \boldsymbol{ \mathcal{B}}_{m_k})$ the white vertex visited at time $i$ in the Lukasiewicz exploration of the belt part of $ \mathbf{u}_{m_k}$ and write $c_i^{m_k}$ for the first corner of $v_i^{m_k}$ in the contour order. We remind the reader of the notation $c_f^{m_k}$ and $c_\ell^{m_k}$ respectively for the first and last corner of $J_{m_k}$ in the contour of $\boldsymbol{\mathcal{B}}_{m_k}$.  Next recall the definition of  $\{\widehat{\mathfrak{X}}$ is $2\varepsilon$-trapped$\}$, given in Section~\ref{sec:prison} with $(X,Z,t_\bullet,\sigma)$ replaced by $(\widehat{X},\widehat{Z},\widehat{t}_\bullet,\widehat{\sigma})$, where we use the  same notation with a ``~$\widehat{}$~''. Then by the convergence of $\mathrm{Cod}_{m_k}$ to  $(\widehat{X},\widehat{Z},\widehat{t}_\bullet,\widehat{\sigma})$, and the assumptions on $\omega$, it is straightforward to verify that:
 \begin{itemize}
 \item   We can find discrete equivalents of $\widehat{r}_{1,4}$ and $\widehat{r}_{2,3}$, that is, integers $\widehat{r}_{1,4}^{m_k}$ and $\widehat{r}_{2,3}^{m_k}$, smaller than $a_{m_k}$,  corresponding to the visit of a white vertex on the spine of $ \boldsymbol{\mathcal{B}}_{m_k}$, such that $\widehat{r}_{1,4}^{m_k}/{m_k} \to \widehat{r}_{1,4}$ and $\widehat{r}_{2,3}^{m_k}/{m_k}\to \widehat{r}_{2,3}$ and so that the renormalized number of grand-children of these vertices converge  towards $\Delta \widehat{X}_{\widehat{r}_{1,4}}$ and $\Delta \widehat{X}_{\widehat{r}_{2,3}}$ respectively.  
 \item Similarly,  we can find four visit times $\widehat{s}_1^{m_k}< \widehat{s}_2^{m_k} < a^{m_k} < \widehat{s}_3^{m_k} < \widehat{s}_4^{m_k}$ such that $w_{\widehat{s}_1^{m_k}}^{m_k}$ and $w_{\widehat{s}_4^{m_k}}^{m_k}$ are grand-children of $w_{\widehat{r}_{1,4}^{m_k}}^{m_k}$ (and similarly for the pair $2,3$), and   if $\ell_{m_k}$ is the labeling of $ \boldsymbol{ \mathcal{B}}_{m_k}$, we have: 
  \begin{align} 
  \label{eq:s1omega1} \min_{[ c_{\widehat{s}_{1}^{m_k}}^{m_k},c_{f}^{{m_k}}]_{ \boldsymbol{\mathcal{B}}_{m_k}}} \ell_{m_k} = \ell_{m_k}\big(c_{\widehat{s}_1^n}^{m_k}\big) + o(m_k^{ \frac{1}{2\alpha}}) \, \mbox{ and } \,  \ell_{m_k}(c_{\widehat{s}_1^{m_k}}^{m_k}) \leq \ell_{m_k}(J_{m_k}) - 2 \varepsilon (\scal m_k)^{\frac{1}{2 \alpha}} + o(m_k^{ \frac{1}{2\alpha}}), \\ 
\label{eq:s2omega1} \min_{[ c_{f}^{m_k},c_{\widehat{s}_2^{m_k}}^{m_k}]_{ \boldsymbol{\mathcal{B}}_{m_k}}} \ell_{m_k} = \ell_{m_k}\big(c_{\widehat{s}_2^{m_k}}^{m_k}\big) + o(m_k^{ \frac{1}{2\alpha}})\, \mbox{ and } \,    \ell_{m_k}\big(c_{\widehat{s}_2^{m_k}}^{m_k}\big) \leq \ell_{m_k}(J_{m_k}) - 2 \varepsilon (\scal m_k)^{ \frac{1}{2 \alpha}} + o(m_k^{ \frac{1}{2\alpha}}),\\
\label{eq:s3omega2} \min_{[c_{\widehat{s}_3^{m_k}}^{m_k}, c_{\ell}^{m_k}]_{ \boldsymbol{\mathcal{B}}_{m_k}}} \ell_{m_k}= \ell_{m_k}\big(c_{\widehat{s}_3^{m_k}}^{m_k}\big) + o(m_k^{ \frac{1}{2\alpha}})\, \mbox{ and } \,    \ell_{m_k}\big(c_{\widehat{s}_3^{m_k}}^{m_k}\big) \leq \ell_n(J_{m_k}) - 2 \varepsilon (\scal m_k)^{ \frac{1}{2 \alpha}} + o(m_k^{ \frac{1}{2\alpha}}), \\ 
\label{eq:s4omega2} \min_{[ c_{\ell}^{m_k},c_{\widehat{s}_4^{m_k}}^{m_k}]_{ \boldsymbol{\mathcal{B}}_{m_k}}} \ell_{m_k} = \ell_{m_k}\big(c_{\widehat{s}_4^{m_k}}^{m_k}\big) + o(m_k^{ \frac{1}{2\alpha}}) \, \mbox{ and } \,   \ell_{m_k}\big(c_{\widehat{s}_4^{m_k}}^{m_k}\big) \leq \ell_{m_k}(J_{m_k}) - 2 \varepsilon (\scal m_k)^{ \frac{1}{2 \alpha}} + o(m_k^{ \frac{1}{2\alpha}}),
  \end{align}
where we recall the notation for the corner interval $[c,c']_{ \boldsymbol{ \mathcal{T}}}$ from \eqref{eq:defintervallecorner}, defined  on a well-labeled  mobile~$\boldsymbol{ \mathcal{T}}$.
  \end{itemize}
    \begin{figure}[!h]
 \begin{center}
 \includegraphics[height=13cm]{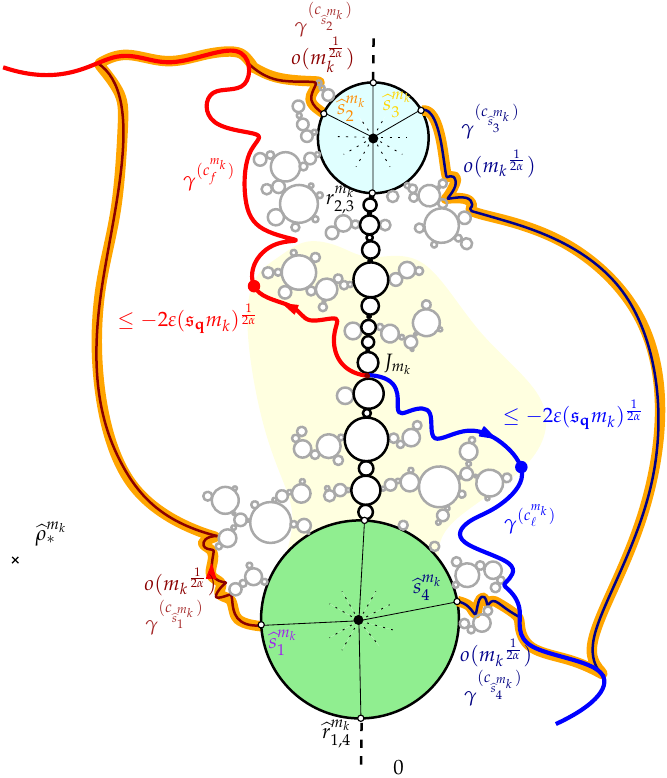} 
 \caption{Illustration of the proof of Lemma \ref{lem:prison-discrete}. On the event where $\widehat{ \mathfrak{X}}$ is $ \varepsilon$-trapped, we can build in the discrete setting the analog of the blocking doors appearing in Figure~\ref{fig:sidebyside}. In particular, any path starting from the yellow region within $\frac{3}{2}\varepsilon ( \mathfrak{s}_{\mathbf{q}} m_k)^{ \frac{1}{2 \alpha}}$ of $J_{m_k}$ and going to $\widehat{\rho}_*^{m_k}$ must cross one of the discrete doors (in orange) and come close to $\widehat{\gamma}_{1,2}^{m_k}$ (the union of the red and blue path).\label{fig:prison-discrete}}
 \end{center}
 \end{figure}
  Now, using the left hand sides of \eqref{eq:s1omega1} and \eqref{eq:s2omega1} we deduce that the simple geodesics $\gamma^{(c_{\widehat{s}_{1}^{m_k}})}$ and $\gamma^{(c_{\widehat{s}_{2}^{m_k}})}$ take only $o( m_k^{ \frac{1}{2\alpha}})$ steps before merging with $\gamma^{(c^{m_k}_f)}$. The discrete left door (in orange on Figure~\ref{fig:prison-discrete}) is the concatenation of the pieces of these geodesics going from $w_{\widehat{s}_{1}^{m_k}}^{m_k}$ to $w_{\widehat{s}_{2}^{m_k}}^{m_k}$. We do the symmetric construction in the right part of the belt for the discrete right door by starting from the last corner of $J_{m_k}$ and replacing the indexes $1,2$ by $3,4$.  \\
 \textsc{Planar separation.} An application of Jordan theorem together with the construction of the edges in $ \mathrm{BDG}^{2\bullet}( \mathbf{u}_n, \epsilon)$ shows that the union of these two discrete doors separate two regions $ \mathrm{In}^{m_k}$ and $ \mathrm{Out}^{m_k}$ in $ \widehat{ \mathfrak{M}}_{m_k}$ where the vertices of the discrete doors are declared to be in $\mathrm{Out}^{m_k}$. In particular, since $ \widehat{\rho}_*^{m_k}$ is adjacent to the root edge associated to a white corner of the buckle, we deduce from the fact that the edges in $\mathrm{BDG}^{2\bullet}( \mathbf{u}_n, \epsilon)$ do not cross the unicyclomobile and follow its contour that we always have $\widehat{\rho}_*^{m_k}\in\mathrm{Out}^{m_k}$. Notice now that the discrete labels on these doors are all smaller than $\ell_{m_k}(J_{m_k})-2 \varepsilon (\scal m_k)^{\frac{1}{2 \alpha}} + o(m_k^{ \frac{1}{2\alpha}})$. We infer  by the Schaeffer bound for unicyclomobiles, given in \eqref{d_n^circ:2},  that, for every $k\geq 1$ large enough, all vertices within graph distance less than $  \frac{3}{2}\varepsilon (\scal m_k)^{ \frac{1}{2 \alpha}}$ from $J_{m_k}$ (the region in yellow on Figure \ref{fig:prison-discrete}) belong to $ \mathrm{In}^{m_k}$. Therefore, we deduce that:
 \begin{center} For $k\geq 1$ large enough,  any path on $\widehat{ \mathfrak{M}}_{m_k}$  going from $ B_{ \mathrm{d}^{ \mathrm{gr}}_{\widehat{ \mathfrak{M}}_{m_k}} }\big(J_{m_k} , \frac{3}{2}\varepsilon (\scal m_k)^{\frac{1}{2 \alpha}}\big)$ to $\widehat{\rho}_*^{m_k}$ must intersect the union of the two doors, and in particular must come within $o( m_k^{ \frac{1}{2\alpha}})$ of $\widehat{\gamma}_{1,2}^{m_k}$.\end{center}
 We can now apply Lemma \ref{chemin:proches}, which implies that $\omega\notin\{(\widehat{D}(\widehat{\rho}_1,\widehat{\rho}_2)-\hdelay)/2  \text{ is $\varepsilon$-bad for }\boldsymbol{\widehat{\mathcal{S}}}\,\}$, as wanted.
\end{proof}

Our final task is to prove \eqref{eq:gogo2}, thus completing the upper bound  \eqref{eq:newgoal} equivalent to Proposition~\ref{main:techni}. Recalling the definition of $\GG$ given in \eqref{eq:GG}, and the results of Section~\ref{sec:prison}, Equation \eqref{eq:gogo2} is equivalent to:

 \begin{lem} \label{lem:derdesder}There exist $c,C \in(0,\infty)$ such that 
 $$ \mathbf{N}^{\bullet}\Big(  \mathbbm{1}_{\mathfrak{X}  \textnormal{\small ~is not $ \varepsilon$-trapped}} \cdot \GG\big(1-\sigma, - Z_{t_\bullet}\big) \Big) \leq C\cdot  \varepsilon^{c}.$$
 \end{lem}

  \begin{proof}
By the scaling property, we have $q_{x}^{[\alpha]}(t)=q_{1}^{[\alpha]}(t/x^{1/\alpha})/x^{1/\alpha}$, for $t\in \mathbb{R}$ and $x>0$. Therefore, using the stretched-exponential tails of $q_{1}^{[\alpha]}$, given in \eqref{eq:tailqalphsleft}, we infer the upper bound
 $$\GG(x,z)\leq  C_1 x^{-\frac{1}{\alpha}} \int_0^\infty \frac{\d t }{t^{\alpha-1/2}} \exp\big(-c_1 t^{\frac{\alpha}{\alpha-1}}/x^{\frac{1}{\alpha-1}} \big)  \mathrm{erfc}\big(- |z|  \sqrt{\frac{2}{t}} \big),\quad x>0, z\in\mathbb{R}, $$
  for some constants $c_1,C_1\in(0,\infty)$. Performing the change of variable $u=  c_1 t^{\frac{\alpha}{\alpha-1}}/x^{\frac{1}{\alpha-1}}$, we get that the right hand side of the previous display is bounded above by 
  $$C_2 x^{-1+\frac{1}{2\alpha}} \int_0^\infty \frac{\d u }{u^{(\alpha-1)(\alpha-3/2)/\alpha  +1}} \exp(-u)  \mathrm{erfc}\Big(- c_2 \frac{|z|}{x^{\frac{1}{2\alpha}}} u^{-\frac{\alpha-1}{2\alpha}} \Big),\quad x>0, z\in\mathbb{R},$$
  for other constant $c_2, C_2\in(0,\infty)$.
 Now fix $0\vee (2\alpha-3)<\beta<1\wedge (2\alpha-2)$ and remark that for some $C>0$ we have $\mathrm{erfc}(t)  \leq C t^{-\beta}$ for all $t >0$. Therefore, since $\int_0^\infty \frac{\d u}{u^{(\alpha-1)(\alpha-(3+\beta)/2)/\alpha  +1}}\exp(-u)<\infty$,  we can find $C_3\in(0,\infty)$, such that
 \begin{equation}\label{inequality:GG:beta}
\GG(x,z)\leq \frac{C_3}{x^{1-\frac{1+\beta}{2\alpha}} |z|^{\beta}},\quad x>0, z\in\mathbb{R} .
\end{equation}
Our goal now is to apply Hölder’s inequality to separate the variables $\mathbbm{1}_{\mathfrak{X}  \mbox{\small ~is not $ \varepsilon$-trapped}}$ and  $\GG\big(1-\sigma, - Z_{t_\bullet}\big) $, thereby enabling us to use the results established in Part \ref{PartI}. In this direction, note that by our choice of $\beta$, we can find $q$ such that:
\begin{equation}\label{eq:q:final}
1<q<\big(\frac{1}{\beta}\big) \wedge \big(1-\frac{1+\beta }{2\alpha}\big)^{-1} \wedge \big(\frac{2(\alpha-1)}{\beta}\big). 
\end{equation}
Next, we set $p=q/(q-1)$ and  recall the convention $\GG(x,\cdot )=0$, for $x\leq 0$.  By \eqref{inequality:GG:beta}, an application of   Hölder inequality gives 
\begin{align*}
\mathbf{N}^{\bullet}\Big( & \mathbbm{1}_{\mathfrak{X}  \mbox{\small ~is not $ \varepsilon$-trapped}} \cdot \GG\big(1-\sigma, - Z_{t_\bullet}\big) \Big)\\
&\leq C_3\cdot \mathbf{N}^{\bullet}\Big( \sigma\in(0,1), \mathfrak{X}  \mbox{\small ~is not $ \varepsilon$-trapped} \Big)^{1/p}\cdot \mathbf{N}^{\bullet}\Big( \big(1-\sigma\big)^{-(1-\frac{1+\beta}{2\alpha}) q} \cdot |Z_{t_\bullet}|^{-\beta q} \mathbbm{1}_{\sigma\in(0,1)} \Big)^{1/q}.
\end{align*}
We conclude by proving that  $\mathbf{N}^{\bullet}\big( \sigma\in(0,1), \mathfrak{X}  \mbox{\small ~is not $ \varepsilon$-trapped} \big)={O}(\varepsilon^{\ell})$, for some $\ell>0$,  and  the quantity $ \mathbf{N}^{\bullet}\Big( \big(1-\sigma\big)^{-(1-\frac{1+\beta}{2\alpha}) q} \cdot |Z_{t_\bullet}|^{-\beta q} \mathbbm{1}_{\sigma\in(0,1)} \Big)$ is  finite. Let us treat each term separately, starting with the former. In this direction, consider the constants $c,C\in(0,\infty)$ appearing in the statement of Theorem~\ref{prop:malo} and note that without loss of generality we may assume that $c<1/(2\alpha-2)$. Next, fix $r\in(0, c/(1-2(\alpha-1)c))$ and write:
\begin{align*}
 \mathbf{N}^{\bullet}\Big( \sigma\in(0,1), \mathfrak{X}  \mbox{\small ~is not $ \varepsilon$-trapped} \Big)\leq \mathbf{N}^{\bullet}\Big( \mathfrak{X}  \mbox{\small~is not $ \varepsilon$-trapped}, H_{t_\bullet}\leq \varepsilon^{-r}\Big)+\mathbf{N}^{\bullet}\Big( \sigma\in(0,1), H_{t_\bullet}\geq \varepsilon^{-r} \Big).
\end{align*}
On one hand, an application of Theorem \ref{prop:malo} gives that:
\begin{align*}
 \mathbf{N}^{\bullet}\Big(\mathfrak{X}  \mbox{\small~is not $ \varepsilon$-trapped}, H_{t_\bullet}\leq \varepsilon^{-r}\Big)= \int_{0}^{\varepsilon^{-r}} \d h ~\mathbf{N}^\bullet\big( \mathfrak{X}  \mbox{\small~is not $ \varepsilon$-trapped}~\big|~H_{t_\bullet}=h\big)\leq \frac{C\cdot \varepsilon^{c-r(1-2(\alpha-1)c)}}{1-2(\alpha-1) c},  
\end{align*}
where to obtain the first equality we used that, by \eqref{one_point},   the distribution of $H_{t_{\bullet}}$ is  $\mathbbm{1}_{h\geq 0}\d h$. On the other hand, by the scaling invariance \eqref{eq:scaling:N:v} and \eqref{eq:excursionmeasuredecomp}, we have:
\begin{align*}
\mathbf{N}^{\bullet}\Big( \sigma\in(0,1), H_{t_\bullet}\geq \varepsilon^{-r} \Big)&=\int_0^{1} \frac{\d v}{|\Gamma(-\frac{1}{\alpha})| v^{\frac{1}{\alpha}+1}}\cdot  \mathbf{N}^{(v)}\Big(\int_0^{v} \d t  ~\mathbbm{1}_{H_t\geq \varepsilon^{-r}}\Big)\\
&=\int_0^{1} \frac{\d v}{|\Gamma(-\frac{1}{\alpha})| v^{\frac{1}{\alpha}}} \cdot \mathbf{E}\Big[\int_0^{1} \d t~ \mathbbm{1}_{ v^{\frac{\alpha-1}{\alpha}}H_t\geq \varepsilon^{-r}}\Big]\\
&\leq \Big(\int_0^{1} \frac{\d v}{|\Gamma(-\frac{1}{\alpha})| v^{\frac{1}{\alpha}}} \Big)\cdot \mathbf{E}\Big[\int_0^{1} \d t~ \mathbbm{1}_{H_t\geq \varepsilon^{-r}}\Big]=\frac{\mathbf{E}\Big[\int_0^{1} \d t~ \mathbbm{1}_{H_t\geq \varepsilon^{-r}}\Big]}{(\alpha-1)\Gamma(1-\frac{1}{\alpha})} .
\end{align*}
It follows from \cite[Theorem 3.3.3]{DLG02} that $\mathbf{E}\big[\int_0^{1} \d t~ \mathbbm{1}_{H_t\geq \varepsilon^{-r}}\big]={O}(\varepsilon^{\ell})$, for some $\ell>0$. Putting all together, we derive as wanted that  $\mathbf{N}^{\bullet}\big( \sigma\in(0,1), \mathfrak{X}  \mbox{\small ~is not $ \varepsilon$-trapped} \big)={O}(\varepsilon^{\ell})$, possibly for another $\ell>0$. It remains to prove that  $\mathbf{N}^{\bullet}\big( \big(1-\sigma\big)^{-(1-\frac{1+\beta}{2\alpha}) q} \cdot |Z_{t_\bullet}|^{-\beta q}\mathbbm{1}_{\sigma\in(0,1)} \big)$ is finite. To this end notice that, again by scaling invariance and  \eqref{eq:excursionmeasuredecomp}, the latter equals
\begin{align*}
\int_0^{1} \frac{\d v}{|\Gamma(-\frac{1}{\alpha})|} v^{-\frac{1}{\alpha}-1} &\big(1-v\big)^{-(1-\frac{1+\beta}{2\alpha}) q} \cdot \mathbf{N}^{(v)}\Big( \int_0^{v} \d t~   |Z_{t}|^{-\beta q}\Big)\\
&=\Big(\int_0^{1} \frac{\d v}{|\Gamma(-\frac{1}{\alpha})|} v^{-\frac{2+\beta q}{2\alpha}} \big(1-v\big)^{-(1-\frac{1+\beta}{2\alpha}) q}\Big) \cdot \mathbf{E}\Big[\int_{0}^{1} \d t ~|Z_{t}|^{-\beta q}\Big]\\
&= \frac{\mathrm{Beta}(1-\frac{2+\beta q}{2\alpha}, 1- (1-\frac{1+\beta}{2\alpha}) q)}{|\Gamma(-\frac{1}{\alpha})|} \cdot \mathbf{E}\Big[\int_{0}^{1} \d t ~|Z_{t}|^{-\beta q}\Big],
\end{align*}
where $\mathrm{Beta}$ stands for the Beta function. We derive from \eqref{eq:q:final} that $\mathrm{Beta}(1-\frac{2+\beta q}{2\alpha}, 1- (1-\frac{1+\beta}{2\alpha}) q)$ is a finite quantity. To estimate $\mathbf{E}\Big[\int_{0}^{1} \d t ~|Z_{t}|^{-\beta q}\Big]$ recall that, for every fixed $t\in [0,1]$, under $\mathbf{P}$ and conditionally on $X$, the random variable $Z_t$ is a centered Gaussian with variance $\widetilde{d}(0,t)$, we infer that:
\begin{align*}
 \mathbf{E}\big[\int_{0}^{1} \d t ~|Z_{t}|^{-\beta q}\big]= \int_{0}^{1} \d t~ \mathbf{E}\big[ ~|Z_{t}|^{-\beta q}\big] =\frac{2^{-\beta q/2}}{\sqrt{\pi}}\Gamma\big((1-\beta q)/2\big)\int_0^{1} \d t~ \mathbf{E}\big[ ~\widetilde{d}(0,t)^{-\beta q/2}\big] .
 \end{align*}
Lastly, we observe that
\begin{align*}
\int_0^{1} \d t~ \mathbf{E}\big[ ~\widetilde{d}(0,t)^{-\beta q/2}\big]&=\int_0^\infty \d x ~ \mathbf{E}\Big[ \int_0^{1} \d t~ \mathbbm{1}_{\widetilde{d}(0,t)\leq x^{-2/(\beta q)}}\Big]\leq \int_0^\infty \d x ~ \mathbf{E}\Big[\mathrm{Vol}_{d}\big(B_d(\Pi_d(0),2 x^{-2/(\beta q)})\big)\Big], \end{align*}
where we recall that $B_d(\Pi_d(0),y)$ denotes the ball centered at $\Pi_d(0)$ with radius $y$ in $\mathcal{L}$ and  $\mathrm{Vol}_{d}$ stands for the  pushforward of the Lebesgue measure on  $[0,1]$ under the canonical projection $\Pi_d$. We also stress that to obtain the previous inequality we used that $d\leq 2 \widetilde{d}$. We deduce from \cite[Proposition 5.8]{archer2019brownian} and the re-rooting invariance of Lemma~\ref{lem:invariancereroottime} that the  previous display is finite. This completes the proof, the section and the paper.
  \end{proof}

\addcontentsline{toc}{section}{References}

\bibliographystyle{abbrv}
\bibliography{bibli}

\begin{thebibliography}{100}

\bibitem{abraham2013note}
R.~Abraham, J.-F. Delmas, and P.~Hoscheit.
\newblock A note on the {G}romov-{H}ausdorff-{P}rokhorov distance between
  (locally) compact metric measure spaces.
\newblock {\em Electron. J. Probab.}, 18, 2013.

\bibitem{ADN:LT}
R.~Abraham, J.-F. Delmas, and M.~Nassif.
\newblock Conditioning (sub)critical {L}{\'e}vy trees by their maximal degree:
  {D}ecomposition and local limit.
\newblock {\em arXiv:2211.02317}, 2022.

\bibitem{AA13}
L.~Addario-Berry and M.~Albenque.
\newblock The scaling limit of random simple triangulations and random simple
  quadrangulations.
\newblock {\em Ann. Probab.}, 45(5):2767--2825, 2017.

\bibitem{albenque2022geometric}
M.~Albenque and L.~M{\'e}nard.
\newblock Geometric properties of spin clusters in random triangulations
  coupled with an {I}sing model.
\newblock {\em arXiv:2201.11922}, 2022.

\bibitem{ambjorn2016generalized}
J.~Ambj{\o}rn, T.~Budd, and Y.~Makeenko.
\newblock Generalized multicritical one-matrix models.
\newblock {\em Nuclear Physics B}, 913:357--380, 2016.

\bibitem{AMY24}
V.~Ambrosio, J.~Miller, and Y.~Yuan.
\newblock Chemical distance metric for non-simple {CLE}.
\newblock {\em (in preparation)}, 2024.

\bibitem{AnAsRo99}
G.~E. Andrews, R.~Askey, and R.~Roy.
\newblock {\em Special Functions}, volume~71 of {\em Cambridge Studies in
  Advanced Mathematics}.
\newblock Cambridge University Press, Cambridge, 1999.

\bibitem{angel2017stability}
O.~Angel, B.~Kolesnik, and G.~Miermont.
\newblock Stability of geodesics in the {B}rownian map.
\newblock {\em Ann. Probab.}, 45(5):3451--3479, 2017.

\bibitem{archer2019brownian}
E.~Archer.
\newblock Brownian motion on stable looptrees.
\newblock In {\em Annales de l'Institut Henri Poincar{\'e}-Probabilit{\'e}s et
  Statistiques}, volume~57, pages 940--979, 2021.

\bibitem{archer2024some}
E.~Archer, A.~Carrance, and L.~M{\'e}nard.
\newblock Some properties of stable snakes.
\newblock {\em arXiv:2403.15275}, 2024.

\bibitem{archer2024stable}
E.~Archer, A.~Carrance, and L.~M{\'e}nard.
\newblock Stable quadrangulations and stable spheres.
\newblock {\em arXiv:2405.05677}, 2024.

\bibitem{Bar99}
E.~W. Barnes.
\newblock The genesis of the double {G}amma function.
\newblock {\em Proceedings of the London Mathematical Society}, 31:358--381,
  1899.

\bibitem{Bar01}
E.~W. Barnes.
\newblock The theory of the double {G}amma function.
\newblock {\em Philosophical Transactions of the Royal Society of London (A)},
  196:265--387, 1901.

\bibitem{berestycki15}
N.~Berestycki.
\newblock Diffusion in planar {L}iouville quantum gravity.
\newblock {\em Ann. Inst. Henri Poincar\'e{} Probab. Stat.}, 51(3):947--964,
  2015.

\bibitem{bernardi2019boltzmann}
O.~Bernardi, N.~Curien, and G.~Miermont.
\newblock A {B}oltzmann approach to percolation on random triangulations.
\newblock {\em Canadian Journal of Mathematics}, 71(1):1--43, 2019.

\bibitem{bertoin1994increase}
J.~Bertoin.
\newblock Increase of stable processes.
\newblock {\em J. Theoret. Probab.}, 7(3):551--563, 1994.

\bibitem{Ber96}
J.~Bertoin.
\newblock {\em L\'evy processes}, volume 121 of {\em Cambridge Tracts in
  Mathematics}.
\newblock Cambridge University Press, Cambridge, 1996.

\bibitem{Ber99}
J.~Bertoin.
\newblock Subordinators: examples and applications.
\newblock In {\em Lectures on probability theory and statistics
  ({S}aint-{F}lour, 1997)}, volume 1717 of {\em Lecture Notes in Math.}, pages
  1--91. Springer, Berlin, 1999.

\bibitem{BBCK18}
J.~Bertoin, T.~Budd, N.~Curien, and I.~Kortchemski.
\newblock Martingales in self-similar growth-fragmentations and their
  connections with random planar maps.
\newblock {\em Probab. Theory Related Fields}, 172(3):663--724, 2018.

\bibitem{bertoin2024self}
J.~Bertoin, N.~Curien, and A.~Riera.
\newblock Self-similar markov trees and scaling limits.
\newblock {\em arXiv:2407.07888}, 2024.

\bibitem{BerSav}
J.~Bertoin and M.~Savov.
\newblock {Some applications of duality for {L}{\'e}vy processes in a
  half-line}.
\newblock {\em Bulletin of the London Mathematical Society}, 43(1):97--110, 10
  2010.

\bibitem{Bet11}
J.~Bettinelli.
\newblock Scaling limit of random planar quadrangulations with a boundary.
\newblock In {\em Annales de l'IHP Probabilit{\'e}s et statistiques},
  volume~51, pages 432--477, 2015.

\bibitem{bettinelli2016geodesics}
J.~Bettinelli.
\newblock Geodesics in {B}rownian surfaces ({B}rownian maps).
\newblock {\em Ann. Inst. Henri Poincar\'e{} Probab. Stat.}, 52(2):612--646,
  2016.

\bibitem{BJM13}
J.~Bettinelli, E.~Jacob, and G.~Miermont.
\newblock The scaling limit of uniform random plane maps, via the
  {A}mbj{\o}rn-{B}udd bijection.
\newblock {\em Electronic J. Probab.}, 19, 2014.

\bibitem{BM15}
J.~Bettinelli and G.~Miermont.
\newblock Compact {B}rownian surfaces {I}: {B}rownian disks.
\newblock {\em Probab. Theory Related Fields}, 167(3-4):555--614, 2017.

\bibitem{BeMi22}
J.~Bettinelli and G.~Miermont.
\newblock Compact {B}rownian surfaces {II}. {O}rientable surfaces.
\newblock {\em arXiv:2212.12511}, 2022.

\bibitem{bjornberg2019stable}
J.~Bj{\"o}rnberg, N.~Curien, and S.~{\"O}. Stef{\'a}nsson.
\newblock Stable shredded spheres and causal random maps with large faces.
\newblock {\em Ann. of Probab.}, 50(5):2056--2084, 2022.

\bibitem{blanc2024geodesics}
G.~Blanc, N.~Curien, and J.~Kahn.
\newblock Geodesics in planar {P}oisson roads random metric.
\newblock {\em arXiv:2407.07887}, 2024.

\bibitem{blancrenaudie2022looptree}
A.~Blanc-Renaudie.
\newblock Looptree, {F}ennec, and {S}nake of {ICRT}.
\newblock {\em arXiv:2203.10891}, 2022.

\bibitem{BBG12}
G.~Borot, J.~Bouttier, and E.~Guitter.
\newblock Loop models on random maps via nested loops: case of domain symmetry
  breaking and application to the potts model.
\newblock {\em J. Phys. A: Math. Theor.}, 45(49), 2012.

\bibitem{BBG11}
G.~Borot, J.~Bouttier, and E.~Guitter.
\newblock A recursive approach to the {O}({N}) model on random maps via nested
  loops.
\newblock {\em J. Phys. A: Math. Theor.}, 45, 2012.

\bibitem{BDFG04}
J.~Bouttier, P.~Di~Francesco, and E.~Guitter.
\newblock Planar maps as labeled mobiles.
\newblock {\em Electron. J. Combin.}, 11(1):Research Paper 69, 27 pp.
  (electronic), 2004.

\bibitem{bouttier2022bijective}
J.~Bouttier, E.~Guitter, and G.~Miermont.
\newblock Bijective enumeration of planar bipartite maps with three tight
  boundaries, or how to slice pairs of pants.
\newblock {\em Annales Henri Lebesgue}, 5:1035--1110, 2022.

\bibitem{bretagnolle04}
J.~Bretagnolle.
\newblock Sur l'in\'egalit\'e{} de concentration de {D}oeblin-{L}\'evy,
  {R}ogozin-{K}esten.
\newblock In {\em Parametric and semiparametric models with applications to
  reliability, survival analysis, and quality of life}, Stat. Ind. Technol.,
  pages 533--551. Birkh\"auser Boston, Boston, MA, 2004.

\bibitem{BudOn}
T.~Budd and L.~Chen.
\newblock The peeling process on random planar maps coupled to an {O}(n) loop
  model (with an appendix by {L}inxiao {C}hen).
\newblock {\em Electron. J. Probab.}, 2019.

\bibitem{BC16}
T.~Budd and N.~Curien.
\newblock Geometry of infinite planar maps with high degrees.
\newblock {\em Electron. J. Probab.}, 22:Paper No. 35, 37, 2017.

\bibitem{BC22}
T.~Budd and N.~{Curien}.
\newblock Random punctured hyperbolic surfaces \& the {B}rownian sphere.
\newblock {\em (in preparation)}, 2024.

\bibitem{BCMcauchy}
T.~Budd, N.~Curien, and C.~Marzouk.
\newblock Infinite random planar maps related to {C}auchy processes.
\newblock {\em J. \'Ec. polytech. Math.}, 5:749--791, 2018.

\bibitem{BBI01}
D.~Burago, Y.~Burago, and S.~Ivanov.
\newblock {\em A course in metric geometry}, volume~33 of {\em Graduate Studies
  in Mathematics}.
\newblock American Mathematical Society, Providence, RI, 2001.

\bibitem{caraceni2019self}
A.~Caraceni and N.~Curien.
\newblock Self-avoiding walks on the {UIPQ}.
\newblock In {\em Sojourns in Probability Theory and Statistical Physics-III},
  pages 138--165. Springer, 2019.

\bibitem{CS04}
P.~Chassaing and G.~Schaeffer.
\newblock Random planar lattices and integrated super{B}rownian excursion.
\newblock {\em Probab. Theory Related Fields}, 128(2):161--212, 2004.

\bibitem{Cha96}
L.~Chaumont.
\newblock Conditionings and path decompositions for {L}\'evy processes.
\newblock {\em Stochastic Process. Appl.}, 64(1):39--54, 1996.

\bibitem{Cyclical}
L.~Chaumont and G.~Uribe~Bravo.
\newblock Shifting processes with cyclically exchangeable increments at random.
\newblock {\em XI Symposium on Probability and Stochastic Processes},
  (69):101--117, 2015.

\bibitem{chenTurunen2020critical}
L.~Chen and J.~Turunen.
\newblock Critical ising model on random triangulations of the disk:
  enumeration and local limits.
\newblock {\em Communications in Mathematical Physics}, 374(3):1577--1643,
  2020.

\bibitem{chen2023ising}
L.~Chen and J.~Turunen.
\newblock Ising model on random triangulations of the disk: phase transition.
\newblock {\em Communications in Mathematical Physics}, 397(2):793--873, 2023.

\bibitem{CurStFlour}
N.~Curien.
\newblock {\em Peeling Random Planar Maps: {\'E}cole d'{\'E}t{\'e} de
  Probabilit{\'e}s de Saint-Flour XLIX--2019}, volume 2335.
\newblock Springer Nature, 2023.

\bibitem{CKlooptrees}
N.~Curien and I.~Kortchemski.
\newblock Random stable looptrees.
\newblock {\em Electron. J. Probab.}, 19:no. 108, 35, 2014.

\bibitem{CLGrecursive}
N.~Curien and J.-F. Le~Gall.
\newblock Random recursive triangulations of the disk via fragmentation theory.
\newblock {\em Ann. Probab.}, 39(6):2224--2270, 2011.

\bibitem{CLGplane}
N.~Curien and J.-F. Le~Gall.
\newblock The {B}rownian plane.
\newblock {\em J. Theoret. Probab.}, 27(4):1249--1291, 2014.

\bibitem{CLGmodif}
N.~Curien and J.-F. Le~Gall.
\newblock First-passage percolation and local perturbations on random planar
  maps.
\newblock {\em Ann. Sci. \'Ec. Norm. Sup\'er}, 3(52):631--701, 2019.

\bibitem{CLGMcactus}
N.~Curien, J.-F. Le~Gall, and G.~Miermont.
\newblock The {B}rownian cactus {I}. {S}caling limits of discrete cactuses.
\newblock {\em Ann. Inst. Henri Poincar\'e Probab. Stat.}, 49(2):340--373,
  2013.

\bibitem{CMMinfini}
N.~Curien, L.~M\'enard, and G.~Miermont.
\newblock A view from infinity of the uniform infinite planar quadrangulation.
\newblock {\em ALEA Lat. Am. J. Probab. Math. Stat.}, 10(1):45--88, 2013.

\bibitem{curien2018duality}
N.~Curien and L.~Richier.
\newblock Duality of random planar maps via percolation.
\newblock In {\em Annales de l'Institut Fourier}, volume~70, pages 2425--2471,
  2020.

\bibitem{dauvergne202327}
D.~Dauvergne.
\newblock The 27 geodesic networks in the directed landscape.
\newblock {\em arXiv:2302.07802}, 2023.

\bibitem{delmas2003computation}
J.-F. Delmas.
\newblock Computation of moments for the length of the onedimensional {ISE}
  support.
\newblock {\em Electron. J. Probab.}, 8:1--15, 2003.

\bibitem{doherty2024connectivity}
C.~Doherty, K.~Kavvadias, and J.~Miller.
\newblock Connectivity of the adjacency graph of complementary components of
  the {SLE} fan.
\newblock {\em arXiv:2411.13133}, 2024.

\bibitem{DMS14}
B.~Duplantier, J.~R. Miller, and S.~Sheffield.
\newblock Liouville quantum gravity as a mating of trees.
\newblock {\em Asterisque}, 427, 2021.

\bibitem{Du03}
T.~Duquesne.
\newblock A limit theorem for the contour process of conditioned
  {G}alton-{W}atson trees.
\newblock {\em Ann. Probab.}, 31(2):996--1027, 2003.

\bibitem{DLG02}
T.~Duquesne and J.-F. Le~Gall.
\newblock Random trees, {L}\'evy processes and spatial branching processes.
\newblock {\em Ast\'erisque}, (281):vi+147, 2002.

\bibitem{DLG05}
T.~Duquesne and J.-F. Le~Gall.
\newblock Probabilistic and fractal aspects of {L}\'evy trees.
\newblock {\em Probab. Theory Related Fields}, 131(4):553--603, 2005.

\bibitem{FiFrSh85}
P.~J. Fitzsimmons, B.~Fristedt, and L.~A. Shepp.
\newblock The set of real numbers left uncovered by random covering intervals.
\newblock {\em Z. Wahrsch. Verw. Gebiete}, 70(2):175--189, 1985.

\bibitem{GaRhVa16}
C.~Garban, R.~Rhodes, and V.~Vargas.
\newblock Liouville {B}rownian motion.
\newblock {\em Ann. of Probab.}, 44(4):3076--3110, 2016.

\bibitem{Getoor}
R.~K. Getoor.
\newblock {Excursions of a Markov Process}.
\newblock {\em Ann. of Probab.}, 7(2):244 -- 266, 1979.

\bibitem{Whyburn}
G.Whyburn.
\newblock Topological characterization of the {S}ierpinski curve.
\newblock {\em Fund. Math. 45}, 1958.

\bibitem{gwynne2021geodesic}
E.~Gwynne.
\newblock Geodesic networks in {L}iouville quantum gravity surfaces.
\newblock {\em Probab. Math. Phys.}, 2(3):643--684, 2021.

\bibitem{gwynne2019mating}
E.~Gwynne, N.~Holden, and X.~Sun.
\newblock Mating of trees for random planar maps and {L}iouville quantum
  gravity: a survey.
\newblock {\em arXiv:1910.04713}, 2019.

\bibitem{gwynne2017convergence}
E.~Gwynne and J.~Miller.
\newblock Convergence of percolation on uniform quadrangulations with boundary
  to {SLE}$_6$ on $\sqrt{8/3}$-{L}iouville quantum gravity.
\newblock {\em arXiv:1701.05175}, 2017.

\bibitem{gwynne2020existence}
E.~Gwynne and J.~Miller.
\newblock Existence and uniqueness of the {L}iouville quantum gravity metric
  for $\gamma\in (0, 2)$.
\newblock {\em Invent. Math.}, pages 1--121, 2020.

\bibitem{GM16a}
E.~Gwynne and J.~Miller.
\newblock Convergence of the self-avoiding walk on random quadrangulations to
  {SLE}$_{8/3}$ on $\sqrt{8/3}$-{L}iouville quantum gravity.
\newblock {\em Ann. Sci. \'Ec. Norm. Sup\'er.}, 54:305--405, 2021.

\bibitem{GwMiSh21}
E.~Gwynne, J.~Miller, and S.~Sheffield.
\newblock The {T}utte embedding of the mated-{CRT} map converges to {L}iouville
  quantum gravity.
\newblock {\em Ann. Probab.}, 49(4):1677--1717, 2021.

\bibitem{gwynne2020connectivity}
E.~Gwynne and J.~Pfeffer.
\newblock Connectivity properties of the adjacency graph of {SLE} $\kappa$
  bubbles for $\kappa \in (4,8)$.
\newblock {\em Ann. of Probab.}, 48(3):1495--1519, 2020.

\bibitem{hawkes1977intersections}
J.~Hawkes.
\newblock Intersections of markov random sets.
\newblock {\em Zeitschrift f{\"u}r Wahrscheinlichkeitstheorie und Verwandte
  Gebiete}, 37(3):243--251, 1977.

\bibitem{IL71}
I.~A. Ibragimov and Y.~V. Linnik.
\newblock {\em Independent and stationary sequences of random variables}.
\newblock Wolters-Noordhoff Publishing, Groningen, 1971.
\newblock With a supplementary chapter by I. A. Ibragimov and V. V. Petrov,
  Translation from the Russian edited by J. F. C. Kingman.

\bibitem{kammerer2025fpp}
E.~Kammerer.
\newblock Scaling limit of first passage percolation geodesics on planar maps.
\newblock {\em arXiv:2412.02666}.

\bibitem{kammerer2023distances}
E.~Kammerer.
\newblock Distances on the $\mathrm{CLE}_4$, critical {L}iouville quantum
  gravity and $3/2$-stable maps.
\newblock {\em arXiv:2311.08571}, 2024.

\bibitem{kammerer2024gaskets}
E.~Kammerer.
\newblock Gaskets of {O}(2) loop-decorated random planar maps.
\newblock {\em arXiv:2411.05541}, 2024.

\bibitem{kammerer2023large}
E.~Kammerer.
\newblock On large $3/2$-stable maps.
\newblock {\em Annales de l'Institut Henri Poincar{\'e}}, To appear.

\bibitem{khanfir2022convergences}
R.~Khanfir.
\newblock Convergences of looptrees coded by excursions.
\newblock {\em Annales de l'Institut Henri Poincar{\'e}}, To appear.

\bibitem{GHPm}
A.~Khezeli.
\newblock Metrization of the {G}romov--{H}ausdorff (-{P}rokhorov) topology for
  boundedly-compact metric spaces.
\newblock {\em Stochastic Processes and their Applications}, 130(6):101--117,
  2020.

\bibitem{khezeli2023unified}
A.~Khezeli.
\newblock A unified framework for generalizing the {G}romov-{H}ausdorff metric.
\newblock {\em Probability Surveys}, 20:837--896, 2023.

\bibitem{Kor11}
I.~Kortchemski.
\newblock Random stable laminations of the disk.
\newblock {\em Ann. of Probab.}, 42(2):725--759, 2014.

\bibitem{kortchemski2024random}
I.~Kortchemski and C.~Marzouk.
\newblock Random {L}\'evy {L}ooptrees and {L}\'evy maps.
\newblock {\em arXiv:2402.04098}, 2024.

\bibitem{Kuratowski}
K.~Kuratowski.
\newblock {\em Topology II.}
\newblock Academic Press, New York, (1968).

\bibitem{Kuzne}
A.~Kuznetsov.
\newblock On extrema of stable processes.
\newblock {\em Ann. of Probab.}, 39:1027--1060, 2011.

\bibitem{kyprianou2022stable}
A.~E. Kyprianou and J.~C. Pardo.
\newblock {\em Stable L{\'e}vy processes via Lamperti-type representations},
  volume~7.
\newblock Cambridge University Press, 2022.

\bibitem{LZ04}
S.~K. Lando and A.~Zvonkin.
\newblock {\em Graphs on surfaces and their applications}.
\newblock Springer-Verlag, 2004.

\bibitem{LeG99}
J.-F. Le~Gall.
\newblock {\em Spatial branching processes, random snakes and partial
  differential equations}.
\newblock Lectures in Mathematics ETH Z\"urich. Birkh\"auser Verlag, Basel,
  1999.

\bibitem{LG99}
J.-F. Le~Gall.
\newblock {\em Spatial branching processes, random snakes and partial
  differential equations}.
\newblock Lectures in Mathematics ETH Z\"urich. Birkh\"auser Verlag, Basel,
  1999.

\bibitem{LG07}
J.-F. Le~Gall.
\newblock The topological structure of scaling limits of large planar maps.
\newblock {\em Invent. Math.}, 169(3):621--670, 2007.

\bibitem{LG09}
J.-F. Le~Gall.
\newblock Geodesics in large planar maps and in the {B}rownian map.
\newblock {\em Acta Math.}, 205:287--360, 2010.

\bibitem{LG11}
J.-F. Le~Gall.
\newblock Uniqueness and universality of the {B}rownian map.
\newblock {\em Ann. Probab.}, 41:2880--2960, 2013.

\bibitem{LGstar}
J.-F. Le~Gall.
\newblock {Geodesics stars in random geometry}.
\newblock {\em Ann. of Probab.}, 50:1013--1058, 2022.

\bibitem{gall2021volume}
J.-F. Le~Gall.
\newblock The volume measure of the {B}rownian sphere is a {H}ausdorff measure.
\newblock {\em Electron. J. Probab.}, 27:1--28, 2022.

\bibitem{LGLJ98}
J.-F. Le~Gall and Y.~Le~Jan.
\newblock Branching processes in {L}\'evy processes: the exploration process.
\newblock {\em Ann. Probab.}, 26(1):213--252, 1998.

\bibitem{LGL}
J.-F. Le~Gall and T.~Leh{\'e}ricy.
\newblock Separating cycles and isoperimetric inequalities in the uniform
  infinite planar quadrangulation.
\newblock {\em Ann. Probab.}, 47:1498--1540, 2019.

\bibitem{gall2024drilling}
J.-F. Le~Gall and A.~Metz-Donnadieu.
\newblock Drilling holes in the {B}rownian disk: The {B}rownian annulus.
\newblock {\em arXiv:2407.13544}, 2024.

\bibitem{LGM09}
J.-F. Le~Gall and G.~Miermont.
\newblock Scaling limits of random planar maps with large faces.
\newblock {\em Ann. Probab.}, 39(1):1--69, 2011.

\bibitem{LGP08}
J.-F. Le~Gall and F.~Paulin.
\newblock Scaling limits of bipartite planar maps are homeomorphic to the
  2-sphere.
\newblock {\em Geom. Funct. Anal.}, 18(3):893--918, 2008.

\bibitem{gall2018some}
J.-F. Le~Gall and A.~Riera.
\newblock Some explicit distributions for {B}rownian motion indexed by the
  {B}rownian tree.
\newblock {\em Markov Processes Relat. Fields}, 26:659--686, 2020.

\bibitem{Spine}
J.-F. Le~Gall and A.~Riera.
\newblock Spine representations for non-compact models of random geometry.
\newblock {\em Probab. Theory Related Fields}, 181(1):571--645, 2021.

\bibitem{le2024peeling}
J.-F. Le~Gall and A.~Riera.
\newblock Peeling the {B}rownian half-plane.
\newblock {\em arXiv:2404.18489}, 2024.

\bibitem{gall2023spatial}
J.-F. Le~Gall and A.~Riera.
\newblock Spatial {M}arkov property in {B}rownian disks.
\newblock {\em Annales de l'Institut Henri Poincar{\'e}}, To appear.

\bibitem{LP16}
R.~Lyons and Y.~Peres.
\newblock {\em Probability on Trees and Networks}, volume~42 of {\em Cambridge
  Series in Statistical and Probabilistic Mathematics}.
\newblock Cambridge University Press, New York, 2016.
\newblock Available at \url{http://pages.iu.edu/~rdlyons/}.

\bibitem{Maisonneuve:Certain}
B.~Maisonneuve.
\newblock On the structure of certain excursions of a {M}arkov process.
\newblock {\em Z. Wahrscheinlichkeitstheorie verw Gebiete}, 47:61--67, 1979.

\bibitem{MM07}
J.-F. Marckert and G.~Miermont.
\newblock Invariance principles for random bipartite planar maps.
\newblock {\em Ann. Probab.}, 35(5):1642--1705, 2007.

\bibitem{MM06}
J.-F. Marckert and A.~Mokkadem.
\newblock Limit of normalized quadrangulations: the {B}rownian map.
\newblock {\em Ann. Probab.}, 34(6):2144--2202, 2006.

\bibitem{marcus_rosen_2006}
M.~B. Marcus and J.~Rosen.
\newblock {\em Markov Processes, Gaussian Processes, and Local Times}.
\newblock Cambridge Studies in Advanced Mathematics. Cambridge University
  Press, 2006.

\bibitem{marzouk2018scalingstable}
C.~Marzouk.
\newblock On scaling limits of planar maps with stable face-degrees.
\newblock {\em ALEA}, 15:1089--1122, 2018.

\bibitem{marzouk2022scaling}
C.~Marzouk.
\newblock On scaling limits of random trees and maps with a prescribed degree
  sequence.
\newblock {\em Annales Henri Lebesgue}, 5:317--386, 2022.

\bibitem{Mie08}
G.~Miermont.
\newblock On the sphericity of scaling limits of random planar
  quadrangulations.
\newblock {\em Electron. Commun. Probab.}, 13:248--257, 2008.

\bibitem{Mie09}
G.~Miermont.
\newblock Tessellations of random maps of arbitrary genus.
\newblock {\em Ann. Sci. \'Ec. Norm. Sup\'er. (4)}, 42(5):725--781, 2009.

\bibitem{Mie11}
G.~Miermont.
\newblock The {B}rownian map is the scaling limit of uniform random plane
  quadrangulations.
\newblock {\em Acta Math.}, 210(2):319--401, 2013.

\bibitem{miller2021tightness}
J.~Miller.
\newblock Tightness of approximations to the chemical distance metric for
  simple {C}onformal {L}oop {E}nsembles.
\newblock {\em arXiv:2112.08335}, 2021.

\bibitem{miller2020geodesics}
J.~Miller and W.~Qian.
\newblock Geodesics in the {B}rownian map: Strong confluence and geometric
  structure.
\newblock {\em arXiv:2008.02242}, 2020.

\bibitem{miller2016imaginary}
J.~Miller and S.~Sheffield.
\newblock Imaginary geometry {I}: interacting {SLE}s.
\newblock {\em Probab. Theory Related Fields}, 164(3-4):553--705, 2016.

\bibitem{miller2015liouville}
J.~Miller and S.~Sheffield.
\newblock {L}iouville quantum gravity and the {B}rownian map {I}: The {QLE}
  (8/3, 0) metric.
\newblock {\em Invent. Math.}, 219(1):75--152, 2020.

\bibitem{MS15}
J.~Miller and S.~Sheffield.
\newblock An axiomatic characterization of the brownian map.
\newblock {\em Journal de l'{\'E}cole polytechnique---Math{\'e}matiques},
  8:609--731, 2021.

\bibitem{MiSh21geod}
J.~Miller and S.~Sheffield.
\newblock Liouville quantum gravity and the {B}rownian map {II}: {G}eodesics
  and continuity of the embedding.
\newblock {\em Ann. of Probab.}, 49(6):2732--2829, 2021.

\bibitem{MiSh21conf}
J.~Miller and S.~Sheffield.
\newblock Liouville quantum gravity and the {B}rownian map {III}: the conformal
  structure is determined.
\newblock {\em Probab. Theory Related Fields}, 179(3-4):1183--1211, 2021.

\bibitem{miller2017cle}
J.~Miller, S.~Sheffield, and W.~Werner.
\newblock {CLE} percolations.
\newblock In {\em Forum of Mathematics, Pi}, volume~5. Cambridge University
  Press, 2017.

\bibitem{Moore}
R.~Moore.
\newblock Concerning upper-semicontinuous collections of continua.
\newblock {\em Amer.Math.Soc}, 1925.

\bibitem{MP10}
P.~M{\"o}rters and Y.~Peres.
\newblock {\em Brownian motion}.
\newblock Cambridge Series in Statistical and Probabilistic Mathematics.
  Cambridge University Press, Cambridge, 2010.
\newblock With an appendix by Oded Schramm and Wendelin Werner.

\bibitem{mourichoux2024bigeodesic}
M.~Mourichoux.
\newblock The bigeodesic {B}rownian plane.
\newblock {\em arXiv:2410.00426}, 2024.

\bibitem{petrov75}
V.~V. Petrov.
\newblock {\em Sums of independent random variables}, volume Band 82 of {\em
  Ergebnisse der Mathematik und ihrer Grenzgebiete [Results in Mathematics and
  Related Areas]}.
\newblock Springer-Verlag, New York-Heidelberg, 1975.
\newblock Translated from the Russian by A. A. Brown.

\bibitem{RY99}
D.~Revuz and M.~Yor.
\newblock {\em Continuous martingales and {B}rownian motion}, volume 293 of
  {\em Grundlehren der Mathematischen Wissenschaften [Fundamental Principles of
  Mathematical Sciences]}.
\newblock Springer-Verlag, Berlin, third edition, 1999.

\bibitem{Iso-Rie}
A.~Riera.
\newblock Isoperimetric inequalities in the {B}rownian map and the {B}rownian
  plane.
\newblock {\em Ann. Probab.}, 50:2013--2055, 2022.

\bibitem{RRO:2024}
A.~Riera and A.~Rosales-Ortiz.
\newblock Excursion theory for {M}arkov processes indexed by {L}evy trees.
\newblock {\em arXiv:2411.12717}, 2024.

\bibitem{RohdeSchramm05}
S.~Rohde and O.~Schramm.
\newblock Basic properties of {SLE}.
\newblock {\em Ann. of Math. (2)}, 161(2):883--924, 2005.

\bibitem{Sch98}
G.~Schaeffer.
\newblock Conjugaison d'arbres et cartes combinatoires al{\'e}atoires. {PhD}
  thesis.
\newblock 1998.

\bibitem{SheCLE}
S.~Sheffield.
\newblock Exploration trees and conformal loop ensembles.
\newblock {\em Duke Math. J.}, 147(1):79--129, 2009.

\bibitem{She10}
S.~Sheffield.
\newblock Conformal weldings of random surfaces: {SLE} and the quantum gravity
  zipper.
\newblock {\em Ann. Probab.}, pages 3474--3545, 2016.

\bibitem{SheHC}
S.~Sheffield.
\newblock Quantum gravity and inventory accumulation.
\newblock {\em Ann. Probab.}, 44(6):3804--3848, 2016.

\bibitem{SheWer12}
S.~Sheffield and W.~Werner.
\newblock Conformal loop ensembles: the {M}arkovian characterization and the
  loop-soup construction.
\newblock {\em Ann. of Math.}, 176(3):1827--1917, 2012.

\bibitem{Vigon}
V.~Vigon.
\newblock Votre {L}\'evy rampe-t-il?
\newblock {\em Journal of the London Mathematical Society}, 65(1):243 -- 256,
  2002.

\bibitem{yearwood22}
S.~Yearwood.
\newblock The topology of {${\rm SLE}_\kappa$} is random for {$\kappa>4$}.
\newblock {\em Electron. J. Probab.}, 27:Paper No. 156, 14, 2022.

\bibitem{Zol86}
V.~M. Zolotarev.
\newblock {\em One-dimensional Stable Distributions}, volume~65.
\newblock American Mathematical Society, translations of mathematical
  monographs edition, 1986.

\end{thebibliography}
\end{document}